\documentclass[12pt]{amsart}
\usepackage{geometry}
\usepackage{currvita, amssymb, amsmath, amsthm, 
graphicx, subfigure, pifont, wrapfig, pinlabel, amscd,
latexsym ,exscale, enumerate, amsfonts, times, color, hyphenat, pinlabel, 
mathtools, fullpage, array, bbm, url}
\usepackage{float}
\usepackage[bookmarks=false]{hyperref}
\usepackage{stmaryrd}
\usepackage{multirow}
\usepackage[all]{xy}

\newcommand{\bn}[1]{\llbracket #1 \rrbracket}

\newcommand{\foamt}{\textbf{Foam}_{3}}

\newcommand{\llambda}{\overline{\lambda}}

\newcommand{\ii}{\underline{\textbf{\textit{i}}}}
\newcommand{\jj}{\underline{\textbf{\textit{j}}}}

\newcommand{\kom}{\textbf{Kom}}

\newcommand{\glcat}{\mathcal{U}(\mathfrak{gl}_n)}

\newcommand{\Ugl}{\dot{\textbf U}(\mathfrak{gl}_n)}

\DeclareMathOperator{\Mor}{Mor}
\DeclareMathOperator{\Mat}{Mat}
\DeclareMathOperator{\End}{End}
\DeclareMathOperator{\HomGL}{Hom_{\glcat}}
\DeclareMathOperator{\HOMGL}{HOM_{\glcat}}

\newcommand{\bN}{\mathbb{N}}
\newcommand{\bZ}{\mathbb{Z}}
\newcommand{\bQ}{\mathbb{Q}}
\newcommand{\bR}{\mathbb{R}}
\newcommand{\bC}{\mathbb{C}}

\DeclareMathOperator{\Tor}{Tor}
\DeclareMathOperator{\Kh}{Kh}

\newcommand{\sseq}{{\rm SiSeq}}

\newcommand{\onel}{{\textbf 1}_{\lambda}}
\newcommand{\onelp}{{\textbf 1}_{\lambda'}}

\newcommand{\refequal}[1]{\xy {\ar@{=}^{#1}
(-1,0)*{};(1,0)*{}};
\endxy}

\newcommand{\U}{\dot{{\bf U}}(\mathfrak{sl}_n)}
\newcommand{\UZ}{\dot{\bf U}^{\mathbb{Z}}}

\newcommand{\Ucat}{\mathcal{U}(\mathfrak{sl}_n)}
\newcommand{\Scat}{\mathcal{S}}

\newcommand{\UcatD}{\dot{\mathcal U}(\mathfrak{sl}_n)}
\newcommand{\ScatD}{\dot{\Scat}}

\newcommand{\qbin}[2]{
\left[
 \begin{array}{c}
 #1 \\
 #2 \\
 \end{array}
 \right]
}

\DeclareMathOperator{\Ob}{Ob}
\DeclareMathOperator{\Cb}{\textbf{Cb}}
\DeclareMathOperator{\mat}{\textbf{Mat}}
\DeclareMathOperator{\Arr}{Mor}
\DeclareMathOperator{\cob}{\textbf{Cob}^2}

\DeclareMathOperator{\ucob}{u\textbf{Cob}^2}
\DeclareMathOperator{\ukob}{\textbf{Kob}_b(\emptyset)}
\DeclareMathOperator{\ukobk}{\textbf{Kob}_b(k)}

\DeclareMathOperator{\RMOD}{R-\textbf{Mod}}
\DeclareMathOperator{\AMOD}{A-\textbf{Mod}}
\DeclareMathOperator{\ApMOD}{A-p\textbf{Mod}}

\DeclareMathOperator{\MOD}{\textbf{Mod}}
\DeclareMathOperator{\CAT}{\textbf{Cat}}
\DeclareMathOperator{\COB}{\textbf{Cob}}

\DeclareMathOperator{\AGR}{\textbf{Abel}}
\DeclareMathOperator{\TOSP}{\textbf{Top}}

\DeclareMathOperator{\KAR}{\textbf{Kar}}
\DeclareMathOperator{\MON}{\textbf{Mon}}
\DeclareMathOperator{\BMOD}{\textbf{BiMod}}
\DeclareMathOperator{\SK}{Sk}
\DeclareMathOperator{\Kom}{\textbf{Kom}}

\newtheorem{prop}{Proposition}[subsection]
\newtheorem{thm}[prop]{Theorem}
\newtheorem{lem}[prop]{Lemma}
\newtheorem{cor}[prop]{Corollary}

\theoremstyle{remark}
\newtheorem{rem}[prop]{Remark}

\newtheorem*{nota}{Notation}
\theoremstyle{remark}
\newtheorem{ex}[prop]{\textbf{Example}}

\newtheorem{question}[prop]{Question}

\theoremstyle{remark}

\theoremstyle{definition}
\newtheorem{defn}[prop]{Definition}

\numberwithin{equation}{subsection}

\newcommand{\figins}[3] 
{\raisebox{#1pt}{\includegraphics[height=#2 in]{res/figs/section22/#3}}}

\newcommand{\figwhins}[4] 
{\raisebox{#1pt}{\includegraphics[height=#2 in, width=#3 in]{res/figs/section22/#4}}}

\newcommand{\jpg}[2]{{\hspace{-3pt}\begin{array}{c}%
  \raisebox{-2.5pt}{\includegraphics[width=#1]{res/figs/topcom/#2.eps}}%
\end{array}\hspace{-3pt}}}

\newcommand{\Ver}{\mathrm{Vert}}
\newcommand{\F}{\mathcal{F}}
\newcommand{\backoverslash}{\xy
 (0,0)*{\includegraphics[height=.02\textheight]{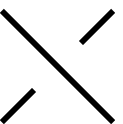}};
 \endxy}
\newcommand{\slashoverback}{\xy
 (0,0)*{\includegraphics[height=.02\textheight]{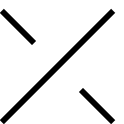}};
 \endxy}
\newcommand{\smoothing}{\xy
 (0,0)*{\includegraphics[height=.02\textheight]{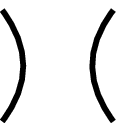}};
 \endxy}
\newcommand{\hsmoothing}{\xy
 (0,0)*{\includegraphics[height=.02\textheight]{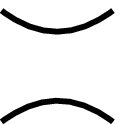}};
 \endxy}
\newcommand{\virtual}{\xy
 (0,0)*{\includegraphics[height=.02\textheight]{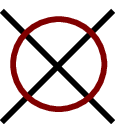}};
 \endxy}
 \newcommand{\overcrossing}{\xy
 (0,0)*{\includegraphics[height=.02\textheight]{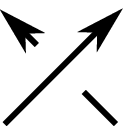}};
 \endxy}
\newcommand{\undercrossing}{\xy
 (0,0)*{\includegraphics[height=.02\textheight]{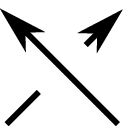}};
 \endxy}
 \newcommand{\osmoothing}{\xy
 (0,0)*{\includegraphics[height=.02\textheight]{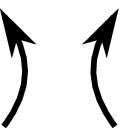}};
 \endxy}
\newcommand{\virtualoriented}{\xy
 (0,0)*{\includegraphics[height=.02\textheight]{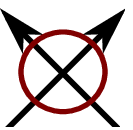}};
 \endxy}

\newcommand{\ttrivalent}{\xy
 (0,0)*{\includegraphics[height=.02\textheight]{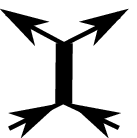}};
 \endxy}
\newcommand{\uu}{\xy
 (0,0)*{\includegraphics[height=.02\textheight]{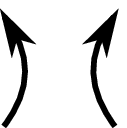}};
 \endxy}
\newcommand{\dd}{\xy
 (0,0)*{\includegraphics[height=.02\textheight]{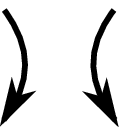}};
 \endxy}
\newcommand{\du}{\xy
 (0,0)*{\includegraphics[height=.02\textheight]{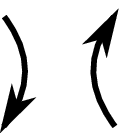}};
 \endxy}
\newcommand{\ud}{\xy
 (0,0)*{\includegraphics[height=.02\textheight]{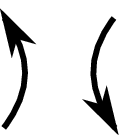}};
 \endxy}
 \newcommand{\lel}{\xy
 (0,0)*{\includegraphics[height=.02\textheight, angle=90]{res/figs/orsmoothings/uu.eps}};
 \endxy}
\newcommand{\rir}{\xy
 (0,0)*{\includegraphics[height=.02\textheight, angle=90]{res/figs/orsmoothings/dd.eps}};
 \endxy}
\newcommand{\ril}{\xy
 (0,0)*{\includegraphics[height=.02\textheight, angle=90]{res/figs/orsmoothings/du.eps}};
 \endxy}
\newcommand{\ler}{\xy
 (0,0)*{\includegraphics[height=.02\textheight, angle=90]{res/figs/orsmoothings/ud.eps}};
 \endxy}
 
\newcommand{\uup}{\xy
 (0,0)*{\includegraphics[height=.025\textheight]{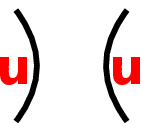}};
 \endxy}
\newcommand{\uupp}{\xy
 (0,0)*{\includegraphics[width=.025\textheight]{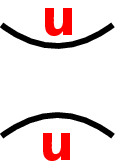}};
 \endxy}
\newcommand{\udp}{\xy
 (0,0)*{\includegraphics[height=.025\textheight]{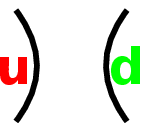}};
 \endxy}
\newcommand{\udpp}{\xy
 (0,0)*{\includegraphics[width=.025\textheight]{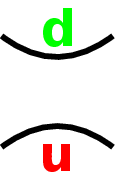}};
 \endxy}
\newcommand{\dup}{\xy
 (0,0)*{\includegraphics[height=.025\textheight]{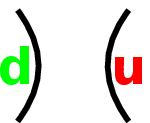}};
 \endxy}
\newcommand{\dupp}{\xy
 (0,0)*{\includegraphics[width=.025\textheight]{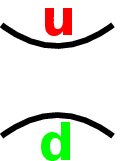}};
 \endxy}
\newcommand{\ddp}{\xy
 (0,0)*{\includegraphics[height=.025\textheight]{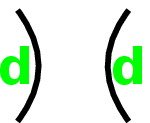}};
 \endxy}
\newcommand{\ddpp}{\xy
 (0,0)*{\includegraphics[width=.025\textheight]{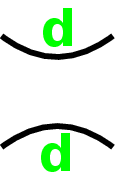}};
 \endxy}
\newcommand{\none}{\xy
 (0,0)*{\includegraphics[height=.03\textheight]{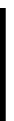}};
 \endxy}
\newcommand{\down}{\xy
 (0,0)*{\includegraphics[height=.03\textheight]{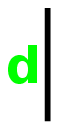}};
 \endxy}
\newcommand{\up}{\xy
 (0,0)*{\includegraphics[height=.03\textheight]{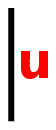}};
 \endxy}
\newcommand{\smoothinga}{\xy
 (0,0)*{\includegraphics[height=.025\textheight]{res/figs/intro/smoothing.eps}};
 \endxy}
\newcommand{\hsmoothinga}{\xy
 (0,0)*{\includegraphics[height=.025\textheight]{res/figs/intro/hsmoothing.eps}};
 \endxy} 
 
 \newcounter{In}

\def\In{\stepcounter{In}\vskip 6pt \par\noindent\scriptsize
  In[\theIn]$:=$\tt
  \catcode`\#=12 \catcode`\&=12 \catcode`\~=12 \catcode`\%=12
  \catcode`\$=12 \catcode`\^=12 \catcode`\_=12
  \obeylines\obeyspaces}
\def\Out{\vskip 6pt \par\noindent\scriptsize\parindent=12pt
  Out[\theIn]$=$\tt
  \catcode`\#=12 \catcode`\&=12 \catcode`\~=12 \catcode`\%=12
  \catcode`\$=12 \catcode`\^=12 \catcode`\_=12 \catcode`\{=12
  \catcode`\}=12 \obeyspaces
}

\title{}
\author{}
\date{Last compiled \today}

\begin{document}

\thispagestyle{empty}

\begin{center}
\vspace*{15mm}

{\Large \textbf{Categorification and applications in topology and representation theory}}
\end{center}
\vspace{10mm}

\begin{center}
{\large \textbf{Dissertation}}
\end{center}

\vspace{5mm}

\begin{center}
zur Erlangung des mathematisch-naturwissenschaftlichen Doktorgrades
\end{center}

\vspace{3mm}

\begin{center}
{\large ``\emph{Doctor rerum naturalium}''}
\end{center}

\vspace{3mm}

\begin{center}
der Georg-August-Universit\"at G\"ottingen
\end{center}

\vspace{3mm}
\begin{center}
im Promotionsprogramm der PhD School of Mathematical Science (SMS)\\
der Georg-August University School of Science (GAUSS)
\end{center}
\vspace{3mm}

\begin{center}
{\small vorgelegt von}
\end{center}

\vspace{3mm}

\begin{center}
Daniel Tubbenhauer
\end{center}

\vspace{3mm}

\begin{center}
{\small aus G\"ottingen}
\end{center}

\vspace{95mm}
\begin{center}
Georg-August-Universit\"at G\"ottingen, G\"ottingen, Germany, 2013
\end{center}

\newpage

\thispagestyle{empty}
\phantom{}
\vspace{15mm}

\begin{flushleft}
\textbf{Betreungsausschuss:}
\end{flushleft}

\vspace{0.1mm}

\begin{flushleft}
Referent 1: Prof. Dr. Thomas Schick
\end{flushleft}

\vspace{0.1mm}

\begin{flushleft}
{\small Mathematisches Institut, Georg-August-Universit\"at G\"ottingen}
\end{flushleft}

\vspace{2mm}

\begin{flushleft}
Referent 2: Dr. Marco Ari\"en Mackaaij
\end{flushleft}

\vspace{0.1mm}

\begin{flushleft}
{\small Departamento de Matem\'{a}tica, FCT, Universidade do Algarve}
\end{flushleft}

\vspace{8mm}

\begin{flushleft}
\textbf{Mitglieder der Pr\"ufungskommission:}
\end{flushleft}

\vspace{2mm}

\begin{flushleft}
Dr. Marco Ari\"en Mackaaij
\end{flushleft}

\vspace{0.1mm}
\begin{flushleft}
{\small Departamento de Matem\'{a}tica, FCT, Universidade do Algarve}
\end{flushleft}

\vspace{2mm}

\begin{flushleft}
Prof. Dr. Gerlind Plonka-Hoch
\end{flushleft}

\vspace{0.1mm}
\begin{flushleft}
{\small Institut f\"ur Num. und Angew. Mathematik, Georg-August-Universit\"at G\"ottingen}
\end{flushleft}

\vspace{2mm}

\begin{flushleft}
Prof. Dr. Karl-Henning Rehren
\end{flushleft}

\vspace{0.1mm}
\begin{flushleft}
{\small Institut f\"ur Theoretische Physik, Georg-August-Universit\"at G\"ottingen}
\end{flushleft}

\vspace{2mm}

\begin{flushleft}
Prof. Dr. Thomas Schick
\end{flushleft}

\vspace{0.1mm}
\begin{flushleft}
{\small Mathematisches Institut, Georg-August-Universit\"at G\"ottingen}
\end{flushleft}

\vspace{2mm}

\begin{flushleft}
Prof. Dr. Anja Sturm
\end{flushleft}

\vspace{0.1mm}
\begin{flushleft}
{\small Institut f\"ur Mathematische Stochastik, Georg-August-Universit\"at G\"ottingen}
\end{flushleft}

\vspace{2mm}

\begin{flushleft}
Prof. Dr. Chenchang Zhu 
\end{flushleft}

\vspace{0.1mm}
\begin{flushleft}
{\small Mathematisches Institut, Georg-August-Universit\"at G\"ottingen}
\end{flushleft}

\vspace{4mm}

\begin{flushleft}
\textbf{Tag der m\"undlichen Pr\"ufung:} 02.07.2013
\end{flushleft}

\newpage
\clearpage
\begin{abstract}
This thesis splits into two major parts. The connection between the two parts is the notion of ``categorification'' which we shortly explain/recall in the introduction.
\vskip0.5cm
In the first part of this thesis we extend Bar-Natan's cobordism based categorification of the Jones polynomial to virtual links. Our topological complex allows a direct extension of the classical Khovanov complex ($h=t=0$), the variant of Lee ($h=0,t=1$) and other classical link homologies. We show that our construction allows, over rings of characteristic $2$, extensions with no classical analogon, e.g. Bar-Natan's $\bZ/2$-link homology can be extended in two non-equivalent ways.

Our construction is computable in the sense that one can write a computer program to perform calculations, e.g. we have written a MATHEMATICA based program.

Moreover, we give a classification of all unoriented TQFTs which can be used to define virtual link homologies from our topological construction. Furthermore, we prove that our extension is combinatorial and has semi-local properties. We use the semi-local properties to prove an application, i.e. we give a discussion of Lee's degeneration of virtual homology.
\vskip0.5cm
In the second part of this thesis (which is based on joint work with Mackaay and Pan) we use Kuperberg's $\mathfrak{sl}_3$ webs and Khovanov's 
$\mathfrak{sl}_3$ foams to define a new algebra $K_S$, which we call the 
$\mathfrak{sl}_3$ web algebra. It is the $\mathfrak{sl}_3$ analogue of 
Khovanov's arc algebra $H_n$. 

We prove that $K_S$ is a graded symmetric Frobenius algebra. 
Furthermore, we categorify an instance of $q$-skew Howe duality, 
which allows us to prove that $K_S$ is Morita equivalent to 
a certain cyclotomic KLR-algebra. This allows us to determine 
the split Grothendieck group $K^{\oplus}_0(K_S)$, to show that its center is 
isomorphic to the cohomology ring of a certain Spaltenstein variety, 
and to prove that $K_S$ is a graded cellular algebra. 
\end{abstract}
\maketitle
%
%
\newpage

\paragraph*{\textbf{Acknowledgements}}
\vskip0.5cm
Many thanks to my two great advisors Thomas Schick and Marco Mackaay who read the thesis from beginning to end (although it is so long) and whose detailed(!) criticism, comments and suggestions greatly improved the presentation.
\vskip0.5cm
Moreover, I especially like to thank Weiwei Pan. Not just that she spend so much time together with me working on the ``web algebra project'', but also her style and philosophy of mathematical exposition has greatly influenced me while she stayed in G\"ottingen.   
\vskip0.5cm
For the first part of this thesis the author thanks Vassily Manturov for corrections and many helpful comments. Moreover, I wish to thank Aaron Kaestner and Louis Kauffman for suggestions and helpful comments, hopefully helping me to write a faster computer program for calculations in future work.
\vskip0.5cm
For the second part of this thesis, I wish to thank Jonathan Brundan, Joel Kamnitzer, Mikhail Khovanov and Ben Webster for helpful exchanges of emails, some of which will hopefully bear fruit in future publications on this topic.  In particular, I thank Mikhail Khovanov for spotting a crucial mistake in 
a previous version of the joint paper with Mackaay and Pan. Moreover, I thank Ben Webster for 
suggesting to use $q$-skew Howe duality in order to relate 
$K_S$ to a cyclotomic KLR-algebra.
\vskip0.5cm
I have also benefited from numerous mathematical papers and books. A lot of them have deeply influenced my personal viewpoints. Especially, I wish to thank Aaron Lauda whose papers had a great impact on my personal style.
\vskip0.5cm
Special thanks to the Courant Research Center ``Higher Order Structures'', the Mathematisches Institute der Georg-August-Universit\"at G\"ottingen and 
the Graduiertenkolleg 1493 in G\"{o}ttingen for sponsoring me in many ways over the years here in G\"ottingen. I thank the University of the Algarve and the Instituto Superior 
T\'{e}cnico for sponsoring three research visits.
\vskip0.5cm
I am indebted to a number of people who where very generous with their time and support, on a mathematical and non-mathematical level. One of the biggest pleasures of writing is getting criticism and input from other - special thanks to all of you!
\vskip0.5cm
Last but not least I would like to thank all my friends, colleagues and my family. Although I do not know how many of you will ever read this thesis but without you it would not have been worth writing it.
\vskip0.5cm
Which leaves open the question of what my personal contribution to this thesis is.
\newpage
  
\tableofcontents

\newpage
\section{Introduction}
\subsection{Categorification}\label{sec-intro}
The notion \textit{categorification} was introduced by Crane in~\cite{cr} based on an earlier work together with Frenkel in~\cite{cf}. We will start by explaining the basic idea in the present section. Forced to reduce this introduction to one sentence, the author would choose:
\[
\text{Interesting integers are shadows of richer structures in categories.}
\]
We try to give an informal introduction in this section. More details can be found in Section~\ref{sec-tech} or for the main parts of this thesis in the Sections~\ref{sec-introa},~\ref{sec-vkhsum} and Sections~\ref{sec-introb},~\ref{sec-websum} respectively. 

The basic idea can be seen as follows. Take a \textit{``set-based''} structure $S$ and try to find a \textit{``category-based''} structure $\mathcal C$ such that $S$ is just a shadow of the category $\mathcal C$. If the category $\mathcal C$ is chosen in a ``good'' way, then one has an explanation of facts about the structure $S$ in a categorical language, that is certain facts in $S$ can be explained as special instances of natural constructions.

As an example, consider the following categorification of the natural numbers $S=\bN$. We take $\mathcal C=$\textbf{FinVec}${}_K$ for a fixed field $K$, i.e. objects are finite dimensional $K$-vector spaces $V,W,\dots$ and morphisms are $K$-linear maps $f\colon V\to W$ between them. Note that the set of isomorphism classes of its objects, i.e. the \textit{skeleton} of $\mathcal C$, is isomorphic to the natural numbers $\bN$ with $0$, since finite dimensional vector spaces are isomorphic iff they have the same dimension. We call this \textit{decategorification}. To be more precise, the category $\mathcal C$ gives a functor $\mathrm{decat}\colon \mathrm{DECAT}(\mathcal C)\to \mathcal \bN$, where $\mathrm{DECAT}(\mathcal C)$ denotes the isomorphism classes of objects.

Since categorification can be seen as \textit{``remembering'' or ``inventing''} information and decategorification is more like \textit{``forgetting'' or ``identifying''} structure which is easier, it is convenient to study the latter in more detail, e.g., if we change the decategorification to be the Grothendieck group $K_0(\mathcal C)$ of the abelian category $\mathcal C$, then we can say that the category $\mathcal C$ is a \textit{categorification} of the integers since $K_0(\mathcal C)=\bZ$. Hence, we can say that the category \textbf{FinVec}${}_K$ is a \textit{categorification} of $\bN$ with \textit{decategorification=dim} or that \textbf{FinVec}${}_K$ is a \textit{categorification} of $\bZ$ with \textit{decategorification=$K_0$}.

We make the following observations. Analogous statements are also true for the Grothendieck group decategorification.
\begin{itemize}
\item Much information is lost if we only consider $\bN$, i.e. we can only say that two objects are isomorphic instead of \textit{how} they are isomorphic.
\item The extra structure of the natural numbers is decoded in the category \textbf{FinVec}${}_K$, e.g.:
\begin{itemize}
\item The product and coproduct in \textbf{FinVec}${}_K$ is the direct sum $\oplus$ and the category comes with a monoidal structure called tensor product $\otimes_K$ and they \textit{categorify} addition and multiplication, i.e. we have $\dim(V\oplus W)=\dim V+\dim W$ and we also have $\dim(V\otimes_K W)=\dim V\cdot\dim W$.
\item The category has $0$ as a zero object and $K$ as an identity for the monoidal structure and we have $V\oplus 0\simeq V$ and $V\otimes_K K\simeq V$, i.e. we can \textit{categorify} the identities.
\item We have $V\hookrightarrow W$ iff $\dim V\leq \dim W$ and $V\twoheadrightarrow W$ iff $\dim V\geq \dim W$, i.e. injections and surjections \textit{categorify} the order relation.
\end{itemize}
\end{itemize}
Moreover, one can write down the categorified statements of each of following properties and one can show that all the isomorphisms are natural.
\begin{itemize}
\item Addition and multiplication are associative and commutative.
\item Multiplication distributes over addition.
\item Addition and multiplication preserve order.
\end{itemize}
Note that categorification is \textit{not} unique, e.g. we can also go from the category $\mathcal C=$\textbf{FinVec}${}_K$ to $\Kom_b(\mathcal C)$, i.e. the \textit{category of bounded chain complexes} of finite dimensional $K$-vector spaces. The decategorification changes to $\chi$, that is taking the Euler characteristic of a complex. As we explain now, this approach leads to a construction that can \textit{also} be called a categorification of $\bZ$.

If we lift $m,n\in\bN$ to the two $K$-vector spaces $V,W$ with dimensions $\dim V=m,\dim W=n$, then the difference $m-n$ lifts to the complex
\[
\begin{xy}
  \xymatrix{
     0\ar[r] & V\ar[r]^d & W\ar[r] & 0,
  }
\end{xy}
\]
for any linear map $d$ and $V$ in even homology degree. More generally, if we lift $m,n$ to complexes $C,D$ with $\chi(C)=m,\chi(D)=n$, then we can lift $m-n$ to $\Gamma(f)$ for any map $f\colon C\to D$ between complexes, where $\Gamma$ denotes the cone. As before, some of the basic properties of the integers $\bZ$ can be lifted to the category $\Kom_b(\mathcal C)$.

This construction is not artificial, i.e. the \textit{Betti-numbers} of a reasonable topological space $X$ can be categorified using \textit{homology groups} $H_k(X,\bZ)$ and the \textit{Euler characteristic} $\chi(X)$ of a reasonable topological space can be categorified using \textit{chain complexes} $(C(X),c_*)$ - an observation which goes back to Noether and Hopf in the 1920's in G\"ottingen. Although of course they never called it categorification. We note the following observations.
\begin{itemize}
\item The space $H_k(X,\bZ)$ is a graded abelian group, while the Betti-number is just a number. More information of the space $X$ is encoded. Again, homomorphisms between the groups tell \textit{how} some groups are related.
\item Singular homology works for all topological spaces. And while the Euler characteristic is only defined (in its initially, naive formulation) for spaces with finite CW-decomposition, the homological Euler characteristic can be defined for a bigger class of spaces.
\item The homology extends to a functor and provides information about continuous maps as well.
\item More sophisticated constructions like multiplication in cohomology provide even more information.
\end{itemize}
Another example in this spirit that we consider in more detail in this thesis is the so-called \textit{categorification of the Jones (or $\mathfrak{sl}_2$) polynomial} from Khovanov~\cite{kh1}. We follow the normalisation used by Bar-Natan in~\cite{bn1}. Let $L_D$ be a diagram of an oriented link. We denote the number of positive crossings by $n_+$ and the number of negative crossings by $n_-$ as shown in the figures below respectively.
\[
n_+=\text{number of crossings }\xy
 (0,0)*{\includegraphics[height=.03\textheight]{res/figs/intro/positive.eps}};
 \endxy\;\;\;\;\;n_-=\text{number of crossings }\xy
 (0,0)*{\includegraphics[height=.03\textheight]{res/figs/intro/negative.eps}};
 \endxy
\]
The \textit{bracket polynomial} of the diagram $L_D$ (without orientations) is a polynomial $\langle L_D\rangle\in\mathbb{Z}[q,q^{-1}]$ given by the rules.
\begin{itemize}
\item $\langle \emptyset\rangle=1$ (normalisation).
\item $\langle\xy
 (0,0)*{\includegraphics[height=.03\textheight]{res/figs/intro/crossing.eps}};
 \endxy\rangle=\langle\xy
 (0,0)*{\includegraphics[height=.03\textheight]{res/figs/intro/smoothing.eps}};
 \endxy\rangle-q\langle\xy
 (0,0)*{\includegraphics[height=.03\textheight]{res/figs/intro/hsmoothing.eps}};
 \endxy\rangle$ (recursion step 1).
\item $\langle \bigcirc\amalg L_D\rangle=(q+q^{-1})\langle L_D\rangle$ (recursion step 2).
\end{itemize}
Then the \textit{Kauffman polynomial} $K(L_D)$ of the oriented diagram $L_D$ is defined by a shift and the \textit{Jones polynomial} $J(L_D)$ by a renormalisation, i.e. by
\[
K(L_D)=(-1)^{n_-}q^{n_+-2n_-}\langle L_D\rangle\text{ and }K(L_D)=(q+q^{-1})J(L_D).
\]
It is well-known that the Jones polynomial is uniquely determined by the property $J(\bigcirc)=1$, where $\bigcirc$ denotes the trivial diagram, and the so-called \textit{$\mathfrak{sl}_2$ skein relations}
\[
q^2J\left(\xy
 (0,0)*{\includegraphics[height=.03\textheight]{res/figs/intro/positive.eps}};
 \endxy\right)-q^{-2}J\left(\xy
 (0,0)*{\includegraphics[height=.03\textheight]{res/figs/intro/negative.eps}};
 \endxy\right)=(q+q^{-1})J\left(\xy
 (0,0)*{\includegraphics[height=.03\textheight]{res/figs/intro/osmoothing.eps}};
 \endxy\right).
\]
Khovanov's idea given in~\cite{kh1} or as explained by Bar-Natan in~\cite{bn1} is based on the idea from the categorification of the Euler characteristic $\chi(X)$ explained above, i.e. if one can categorify a number in $\chi(X)\in\bZ$ using chain complexes, then one can try to categorify a polynomial in $J(L_D)\in\mathbb{Z}[q,q^{-1}]$ using chain complexes of \textit{graded vector spaces} (note that it works over $\bZ$ as well - Khovanov's original work uses $\bZ[c]$ with $c$ of degree two).

In particular, if $V$ denotes a two dimensional $\bQ$-vector space with a basis element $v_+$ of degree $1$ and a basis element $v_-$ of degree $-1$ (the graded dimension is $q+q^{-1}$), then Khovanov categorifies the normalisation and the recursion-step 2 conditions from above as
\[
\llbracket \emptyset \rrbracket=0\to \bQ\to 0,\;\;\text{ and }\;\;\llbracket \bigcirc\amalg L_D \rrbracket=V\otimes_{\bQ}\llbracket L_D\rrbracket,
\]
where $\llbracket \cdot\rrbracket$ takes values in the category of chain complexes of finite dimensional, graded $\bQ$-vector spaces. Let $\Gamma(\cdot)$ again denote the cone complex. To categorify the recursion-step 1 condition Khovanov propose the rule
\[
\left\llbracket\xy
 (0,0)*{\includegraphics[height=.03\textheight]{res/figs/intro/crossing.eps}};
 \endxy\right\rrbracket=\Gamma\left(0\to\left\llbracket\xy
 (0,0)*{\includegraphics[height=.03\textheight]{res/figs/intro/smoothing.eps}};
 \endxy\right\rrbracket\stackrel{d}{\to}\left\llbracket\xy
 (0,0)*{\includegraphics[height=.03\textheight]{res/figs/intro/hsmoothing.eps}};
 \endxy\right\rrbracket\to 0\right).
\]
Of course, the differential $d$ is a main ingredient here. Details can be for example found in~\cite{bn1}. Note that the shift from~\cite{bn1} is already included in the usage of the cone. Indeed, the appearance of chain complexes and the rule above suggest an alternative construction by actions of functors on certain categories. Details can be found for example in the work of Stroppel~\cite{st1}.

It is worth noting that again the terminology is that Khovanov has given \textit{ONE} categorification of the Jones polynomial and not \textit{THE} categorification, e.g. a different categorification is the so-called \textit{odd Khovanov homology} as described by Ozsv\'{a}th, Rasmussen and Szab\'{o} in~\cite{ors}.

Notice that one can ask the following question. Given an additive category $\mathcal C$, then one can go to the category of bounded complexes over $\mathcal C$ denoted by $\Kom_b(\mathcal C)$. Now the two approaches above suggest that we have two notions of \textit{``natural'' decategorification}.
\begin{itemize}
\item One can take the \textit{Euler characteristic as decategorification}. This can be viewed as a sum of elements in the split Grothendieck group $K_0^{\oplus}(\mathcal C)$ of the additive category $\mathcal C$.
\item The category $\Kom_b(\mathcal C)$ is triangulated and one can therefore take its \textit{Grothendieck group} $K^{\Delta}_0(\Kom_b(\mathcal C))$ \textit{as decategorification}.
\end{itemize}
The obvious question is how these two approaches are \textit{related}, i.e. given this setting, then how are $K_0^{\oplus}(\mathcal C)$ and $K^{\Delta}_0(\Kom_b(\mathcal C))$ related? The answer is known: The corresponding groups are isomorphic, see Rose~\cite{rose} for example. In particular, the two examples of categorification that are discussed in Section~\ref{sec-vkh} and Section~\ref{sec-web} follow the ``same idea'' of \textit{decategorification}.

We provide a list of other interesting examples. This list is far from being complete. Much more can be found in the work of Baez and Dolan~\cite{bd1} and~\cite{bd2} for examples that are related to more combinatorial parts of categorification or Crane and Yetter~\cite{cy} and Khovanov, Mazorchuk and Stroppel~\cite{kms} for examples from algebraic categorification.
\begin{itemize}
\item Khovanov's construction can be extended to a categorification of the HOMFLY-PT polynomial, e.g. in~\cite{kr1}. Moreover, some applications of Khovanov's categorification are listed below.
\begin{itemize}
\item It is functorial and provides a strictly stronger invariant. 
\item Kronheimer and Mrowka showed in~\cite{kmr}, by comparing Khovanov homology to Knot Floer homology, that Khovanov homology detects the unknot. This is still an open question for the Jones polynomial.
\item Rasmussen obtained his famous invariant by comparing Khovanov homology to a variation of it. He used it to give a combinatorial proof of the Milnor conjecture, see~\cite{ra}. Note that he in~\cite{ra2} also gives a way to detect exotic $\bR^4$ from his approach.
\end{itemize}
\item Floer homology can be seen as a categorification of the Casson invariant of a manifold. Floer homology is again ``better'' than the Casson invariant, e.g. it is possible to construct a $3+1$ dimensional Topological Quantum Field Theory (TQFT) which for closed four dimensional manifolds gives Donaldson's invariants.
\item Knot Floer homology can be seen as a categorification of the Alexander-Conway knot invariant, see for example~\cite{os}.
\item The Grothendieck group decategorification from above provides another source of examples. Namely, the categorification of certain quantum algebras which have bases with interesting
positive integrality properties. For example, Khovanov and Lauda~\cite{kl1}, and independently Rouquier~\cite{rou}, categorified the quantum Kac-Moody algebras with their canonical bases.
\item The so-called Soergel category $\mathcal S$ can be seen in the same vein as a categorification of the Hecke algebras in the sense that the split Grothendieck group gives the Hecke algebras. We note that Soergel's construction shows that Kazhdan-Lusztig bases have positive integrality properties, see~\cite{soe} and related publications.
\item Ariki gave in~\cite{ariki} a remarkable categorification of all finite dimensional, irreducible representation of $\mathfrak{sl}_m$ for \textit{all} $m$ as well as a categorification of integrable, irreducible representations of the affine version $\widehat{\mathfrak{sl}}_n$. In short, he identified the Grothendieck group of blocks of so-called Ariki-Koike cyclotomic Hecke algebras with weight spaces of such representations in such a way that direct summands of induction and restriction functors between cyclotomic Hecke algebras for $n,n+1$ act on the $K_0$ as the $e_i,f_i$ of $\mathfrak{sl}_m$.
\item In Conformal Field Theory (CFT) researchers study fusion algebras, e.g. the Verlinde algebra. Examples of categorifications of such algebras are known, e.g. using categories connected to the representation theory of quantum groups at roots of unity~\cite{kh6}, and contain more information than these algebras, e.g. the $R$-matrix and the quantum $6j$-symbols.
\item The Witten genus of certain moduli spaces can be seen as an element of $\bZ[[q]]$. It can be realised using elliptic cohomology, see~\cite{ahs} and related papers.
\end{itemize}
This thesis deals with two different instances of categorification. The first was given by the author in~\cite{tub1},~\cite{tub2} and the second by the author in joint work with Mackaay and Pan in~\cite{mpt}.

In Section~\ref{sec-vkh}, with its own introduction in Section~\ref{sec-introa}, the author explains the first part of his thesis, i.e. a categorification of the \textit{virtual} Jones polynomial. That is an extended version of Khovanov's construction explained above that works for so-called \textit{virtual} links, i.e. links that are embedded in a thickened $\Sigma_g$ for an orientable surface of genus $g$.

In Section~\ref{sec-web}, with its own introduction in Section~\ref{sec-introb}, the author explains the second part of his thesis, i.e. the construction of a new algebra, called \textit{web algebra}, providing a connection between \textit{categorified} link invariants in the spirit of Khovanov and \textit{categorified} Reshetikhin-Turaev invariants.

Moreover, we have collected some technical (but ``well-known'') facts in Section~\ref{sec-tech}.
\newpage
\subsection{Virtual Khovanov homology}\label{sec-introa}
This part of the introduction is intended to explain the first part of the thesis, i.e. Section~\ref{sec-vkh}, which is based on two preprints~\cite{tub1} and~\cite{tub2} of the author. A summary of the construction and results of Section~\ref{sec-vkh} is given later, i.e. in Section~\ref{sec-vkhsum}. 
\vskip0.5cm
In Section~\ref{sec-vkh} we consider \textit{virtual link diagrams} $L_D$, i.e. planar graphs of valency four where every vertex is either an overcrossing $\slashoverback$, an undercrossing $\backoverslash$ or a virtual crossing $\virtual$, which is marked with a circle. We also allow circles, i.e. closed edges without any vertices.

We call the crossings $\slashoverback$ and $\backoverslash$ \textit{classical crossings} or just \textit{crossings}. For a virtual link diagram $L_D$ we define the \textit{mirror image} $\overline L_D$ of $L_D$ by switching all classical crossings from an overcrossing to an undercrossing and vice versa.

A \textit{virtual link} $L$ is an equivalence class of virtual link diagrams modulo planar isotopies and \textit{generalised Reidemeister moves}, see Figure~\ref{figureintroa-1}.
\begin{figure}[ht]
  \centering
     \includegraphics[scale=0.55]{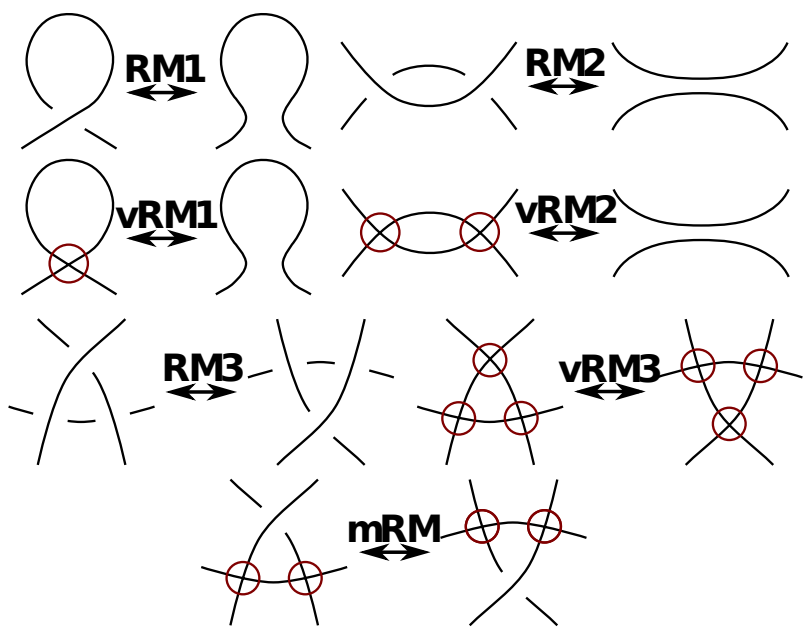}
  \caption{The generalised Reidemeister moves are the moves pictured plus mirror images.}
  \label{figureintroa-1}
\end{figure}

We call the moves RM1, RM2 and RM3 the \textit{classical Reidemeister moves}, the moves vRM1, vRM2 and vRM3 the \textit{virtual Reidemeister moves} and the move mRM the \textit{mixed Reidemeister move}. We call a virtual link diagram $L_D$ \textit{classical} if all crossings of $L_D$ are classical crossings. Furthermore, we say a that virtual link $L$ is \textit{classical}, if the set $L$ contains a classical link diagram.

The notions of an \textit{oriented} virtual link diagram and of an \textit{oriented} virtual link are defined analogously. The latter modulo isotopies and \textit{oriented} generalised Reidemeister moves. Note that an \textit{oriented virtual link diagram} is a diagram together with a choice of an orientation of the diagram such that every crossing is of the form $\overcrossing$, $\undercrossing$ or $\virtualoriented$. Furthermore, we use the short hand notations c- and v- for everything that starts with classical or virtual, e.g. c-knot means classical knot and v-crossing means virtual crossing. 
\vskip0.5cm
Virtual links are an essential part of modern knot theory and were proposed by Kauffman in~\cite{ka3}. They arise from the study of links which are embedded in a thickened $\Sigma_g$ for an orientable surface $\Sigma_g$. These links were studied by Jaeger, Kauffman and Saleur in~\cite{jks}. Note that for c-links the surface is $\Sigma_g=S^2$, i.e. v-links are a generalisation of c-links and they should for example have analogous ``applications'' in quantum physics.

From this perception v-links are a combinatorial interpretation of projections on $\Sigma_g$. It is well-known that two v-link diagrams are equivalent iff their corresponding surface embeddings are \textit{stable equivalent}, i.e. equal modulo:
\begin{itemize}
\item The Reidemeister moves RM1, RM2 and RM3 and isotopies.
\item Adding/removing handles which do not affect the link diagram.
\item Homeomorphisms of surfaces.
\end{itemize}
For a sketch of the proof see Kauffman~\cite{ka1}. For an example see Figure~\ref{figureintroa-2}.
\begin{figure}[ht]
  \centering
     \includegraphics[scale=0.5]{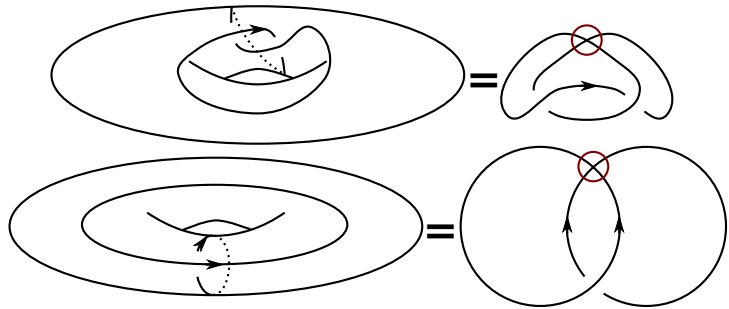}
  \caption{Two knot diagrams on a torus. The first virtual knot is called the \textit{virtual trefoil}.}
  \label{figureintroa-2}
\end{figure}

We are also interested in \textit{virtual tangle diagrams} and \textit{virtual tangles}. The first ones are graphs embedded in a disk $D^2$ such that each vertex is either one valent or four valent. The four valent vertices are, as before, labelled with an \textit{overcrossing} $\slashoverback$, an \textit{undercrossing} $\backoverslash$ or a \textit{virtual crossing} $\virtual$. The one valent vertices are part of the boundary of $D^2$ and we call them \textit{boundary points} and a virtual tangle diagram with $k$ one valent vertices a virtual tangle diagram with \textit{$k$-boundary points}.

A virtual tangle with $k$-boundary points is an equivalence class of virtual tangle diagrams with $k$-boundary points modulo the generalised Reidemeister moves and boundary preserving isotopies. We note that all of the moves in Figure~\ref{figureintroa-1} can be seen as virtual tangle diagrams. Examples are given later, e.g. in Section~\ref{sec-vkhcat}. As before, the notions of \textit{oriented} virtual tangle diagrams and \textit{oriented} virtual tangles can be defined analogously, but modulo \textit{oriented} generalised Reidemeister moves and boundary preserving isotopies.

If the reader is unfamiliar with the notion v-link or v-tangle, we refer to some introductory papers of Kauffman and Manturov, e.g.~\cite{ka2} and~\cite{kama}, and the references therein.
\vskip0.5cm
Suppose one has a crossing $c$ in a diagram of a v-link (or an oriented v-link). We call a substitution of a crossing as shown in Figure~\ref{figureintroa-3} a \textit{resolution} of the crossing $c$.
\begin{figure}[ht]
  \centering
     \includegraphics[scale=0.8]{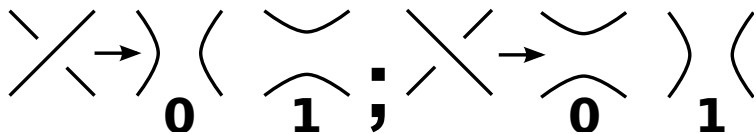}
  \caption{The two possible resolutions of a crossing called \textit{0-resolution} and \textit{1-resolution}.}
  \label{figureintroa-3}
\end{figure}

Furthermore, if we have a v-link diagram $L_D$, a \textit{resolution} of the v-link diagram $L_D$ is a diagram where all crossings of $L_D$ are replaced by one of the two resolutions from Figure~\ref{figureintroa-3}. We use the same notions for v-tangle diagrams.
\vskip0.5cm
One of the greatest developments in modern knot theory was the discovery of \textit{Khovanov homology} by Khovanov in his famous paper~\cite{kh1} (Bar-Natan gave an exposition of Khovanov's construction in~\cite{bn1}). As explained above, Khovanov homology is a categorification of the Jones polynomial in the sense that the graded Euler characteristic of the \textit{Khovanov complex}, which we call the \textit{classical Khovanov complex}, is the Jones polynomial (up to normalisation). 

Recall that the Jones polynomial is known to be related to various parts of modern mathematics and physics, e.g. it origin lies in the study of von Neumann algebras. We note that the Jones polynomial can be extended to v-links in a rather straightforward way, see e.g.~\cite{ka1}. We call this extension the \textit{virtual Jones polynomial} or \textit{virtual $\mathfrak{sl}_2$ polynomial}. 

As a categorification, Khovanov homology reflects these connections on a ``higher level''. Moreover, the Khovanov homology of c-links is strictly stronger than its decategorification, e.g. see~\cite{bn1}. Another great development was the \textit{topological interpretation} of the Khovanov complex by Bar-Natan in~\cite{bn2}. This topological interpretation is a generalisation of the classical Khovanov complex for c-links and one of its modifications has functorial properties~\cite{cmw}. He constructed a \textit{topological complex} whose chain groups are formal direct sums of c-link resolutions and whose differentials are formal matrices of cobordisms between these resolutions.

Bar-Natan's construction modulo chain homotopy and the \textit{local relations} $S,T,4Tu$, also called \textit{Bar-Natan relations}, see Figure~\ref{figureintroa-4}, is an invariant of c-links.
\begin{figure}[ht]
  \centering
     \includegraphics[scale=0.55]{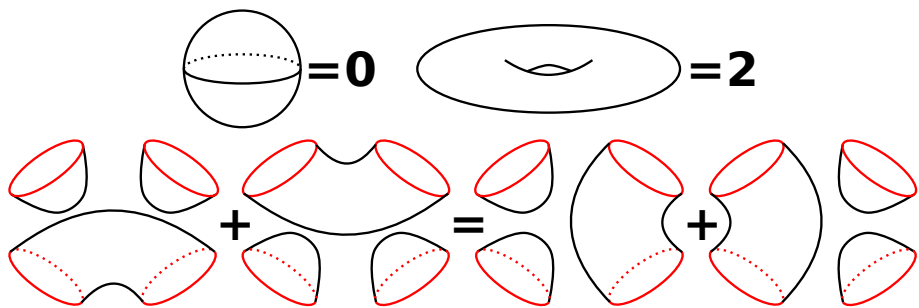}
  \caption{The local relations. A cobordism that contains a sphere $S$ should be zero, a cobordism that contains a torus $T$ should be two times the cobordism without the torus and the four tubes relation.}
  \label{figureintroa-4}
\end{figure}

It is possible with this construction to classify all TQFTs which can be used to define c-link homologies from this approach, see~\cite{kh2}. Moreover, it is algorithmic, i.e. computable in less than exponential time (depending on the number of crossings of a given diagram), see~\cite{bn3}. So it is only natural to search for such a topological categorification of the virtual Jones polynomial.

An algebraic categorification of the virtual Jones polynomial over the ring $\bZ/2$ is rather straightforward and was done by Manturov in~\cite{ma1}. Moreover, he also published a version over the integers $\bZ$ later in~\cite{ma2}. A topological categorification was done by Turaev and Turner in~\cite{tutu}, but their version does not generalise Khovanov homology, since their complex is not bi-graded. Another problem with their version is that it is not clear how to \textit{compute} the homology.

The author gave a topological categorification which generalises the version of Turaev and Turner in the sense that a restriction of the version given in~\cite{tub1} gives the topological complex of Turaev and Turner, another restriction gives a bi-graded complex that agrees with the Khovanov complex for c-links and another restriction gives the so-called \textit{Lee complex}, i.e. a variant of the Khovanov complex that can be used to define the \textit{Rasmussen invariant} of a c-knot, see~\cite{ra}, which is also not included in the version of Turaev and Turner. Moreover, the version given in~\cite{tub1} is computable and also strictly stronger than the virtual Jones polynomial.

Another restriction of the construction from~\cite{tub1} gives a different version than the one given by Manturov~\cite{ma2} in the sense that we conjecture it to be strictly stronger than his version. Moreover, in~\cite{tub2}, the author extended the construction to v-tangles in a ``good way'', something that is not known for Manturov's construction.

To be more precise, the categorification extends from c-tangles to v-tangles in a trivial way (by setting open saddles to be zero). This has an obvious disadvantage, i.e. it is neither a ``good'' invariant of v-tangles nor can it be used to calculate bigger complexes by ``tensoring'' smaller pieces. We give a local notion that is a strong invariant of v-tangles and allows ``tensoring''.

It is worth noting that the construction for v-links is more difficult (combinatorial) than the classical case. That is a reason why in~\cite{tub1} the Bar-Natan approach was not extended to v-tangles. In this thesis, i.e. in Section~\ref{sec-vkh}, we combine the preprints~\cite{tub1} and~\cite{tub2} in one text.
\vskip0.5cm
The author conjectures that the whole construction can be used as a ``blueprint'' for a \textit{categorification of the virtual $\mathfrak{sl}_3$ polynomial} (as explained in Section~\ref{sec-introb}), since Khovanov published in~\cite{kh3} a categorification of the classical $\mathfrak{sl}_3$ polynomial using foams, a special type of singular cobordisms.

Moreover, the author conjectures that it can also be used as a ``blueprint'' for a \textit{categorification of the virtual $\mathfrak{sl}_n$ polynomial}, if one can find a way to avoid the so-called Kapustin-Li formula used by Mackaay, Sto\v{s}i\'{c} and Vaz in~\cite{msv} to give a foam based categorification of the classical $\mathfrak{sl}_n$ polynomial.

Furthermore, the author wants to point out that a \textit{virtual analogue} of the constructions explained in Section~\ref{sec-introb} could be interesting and would be based on the constructions of the author given in his preprints~\cite{tub1},~\cite{tub2} or Section~\ref{sec-vkh}, but has not been done yet.
\subsection{The $\mathfrak{sl}_3$ web algebra}\label{sec-introb}
This part of the introduction is intended to explain the second part of the thesis, i.e. Section~\ref{sec-web}. Note that the results in Section~\ref{sec-web} are based on a preprint of the author together with Mackaay and Pan, see~\cite{mpt}.

We note that, because the results of the Section~\ref{sec-web} are based on joint work, the only things that we have changed is the introductory part given here, a summary of the results given in Section~\ref{sec-websum}, a part about future work in Section~\ref{sec-webend} and the appendix in~\cite{mpt} is now Section~\ref{sec-appendix}. Furthermore, we have also done some (small) notation changes to make the notation consistent with the other sections of this thesis, e.g. the thesis is in British English.

I have also added an isotopy invariant basis obtained from work together with Mackaay and Pan of which we hope that it has ``nice'' (i.e. we hope that it is cellular) properties given in Section~\ref{sec-webbase}.
\vskip0.5cm
We already mentioned the Jones polynomial in Section~\ref{sec-intro} and Section~\ref{sec-introa}. Shortly after Jones announced his discovery, several mathematicians found a generalisation of his construction, which is nowadays called \textit{HOMFLY polynomial}, named after the discoverers Hoste, Ocneanu, Millett, Freyd, Lickorish and Yetter~\cite{homfly}, or \textit{HOMFLY-PT polynomial}, recognising independent contributions of Przytycki and Traczyk.

All these polynomials can be explained using so-called \textit{Skein theory}, which has a completely combinatorial nature. Given an oriented diagram of a c-link $L_D$ (we note that this also works for v-links, e.g. see~\cite{ka1}), the HOMFLY polynomial $P(L_D)$ is a polynomial in $\bZ[a^{\pm 1},b^{\pm 1}]$ given by the following recursive rules.
\begin{itemize}
\item $P(\mathrm{unknot})=1$, where $\mathrm{unknot}$ should be any diagram of the unknot (normalisation).
\item $aP(L_+)-a^{-1}P(L_-)=bP(L_0)$, where $L_+=\overcrossing$, $L_-=\undercrossing$ and $L_0=\osmoothing$ should replace the corresponding parts of the diagram (recursion rule).
\item If $L^{\prime}_D,L^{\prime\prime}_D$ are two link diagrams, then the polynomial for the split union $L_D$ is given by $P(L_D)=\frac{-(a+a^{-1})}{b}P(L^{\prime}_D)P(L^{\prime\prime}_D)$ (union).
\end{itemize}
We note that the polynomial is uniquely determined by these rules and is an invariant of the link.

Hence, in the mid of the 1980s, new knot polynomials were discovered. They were used to solve open and old problems in knot theory in a very simple fashion. And they are related to different parts of modern mathematics, like operator algebras, Hopf algebras, Lie algebras, Chern-Simons theory, conformal field theory etc. Moreover, the Skein theory is combinatorial and makes it ``easy'' to compute these invariants.
\vskip0.5cm
If you have something obviously interesting, one wants to know how this fits into a \textit{``bigger picture''} and not just \textit{``that''} something is true.

An ``explanation'' how to obtain these invariants in terms of \textit{representation theory of quantum groups} was given around 1990 by Reshetikhin and Turaev in~\cite{rt}. To be more precise, they gave an explicit construction that works roughly in the following way.
\begin{itemize}
\item Start by colouring the strings of a tangle diagram in Morse position with irreducible representations $V_i$ of quantum groups.
\item Then at the bottom and top of the tangle diagram $T_D$ one has a tensor product of these $V_i$.
\item Then one associates certain intertwiners to cups, caps and crossings and composition gives an intertwiner $P_{T_D}$ between the bottom and top tensors. This is an invariant.
\item In the special case of link diagrams $L_D$ the intertwiner is a map $P_{L_D}\colon \bC(q)\to\bC(q)$ and can be seen as a polynomial $P_{L_D}(1)$. Note that this polynomial is in fact in $\bZ[q,q^{-1}]$, i.e. it has integer coefficients, and note that Reshetikhin and Turaev's construction restricted to the \textit{invariant tensors} gives the same link invariant.
\end{itemize}
For example, if we consider the substitution $a=q^n,b=q-q^{-1}$ with $n>1$ for the HOMFLY polynomial, then we can obtain these polynomials using the \textit{fundamental representations} of $U_q(\mathfrak{sl}_n)$ as colours. Note that the case $n=2$ gives the Jones polynomial, hence the name $\mathfrak{sl}_2$ polynomial. Much more details can be found for example in~\cite{tur}.
\vskip0.5cm
A connection between these two pictures, i.e. the combinatorial and the one from representation theory, is given by the theory of \textit{$\mathfrak{sl}_n$ webs}, where an $\mathfrak{sl}_n$ web is a graphical presentation of intertwiners between fundamental representations of the corresponding quantum groups $U_q(\mathfrak{sl}_n)$. In particular, in the case $n=2$ the calculus of these webs can be described by the \textit{Temperley-Lieb algebra} and in the case $n=3$ by a graphical calculus formulated by Kuperberg in~\cite{kup} using oriented trivalent graphs. Generalisations of these for all $n>3$ have recently been found, see Cautis, Kamnitzer and Morrison~\cite{ckm} and the references therein.

Let us assume for simplicity that $n=3$ and we restrict to the fundamental representation $V_+$ and its dual $V_-=V_+\wedge V_+$. Then Reshetikhin and Turaev's construction from above assigns to every sign sequence $S=(s_1,\dots,s_m),s_k\in\{+,-\}$ (a boundary of such webs is such a sequence) a tensor product of these representations $V_S=V_{s_1}\otimes\cdots\otimes V_{s_m}$ and to each tangle diagram an intertwiner. Hence, since intertwiners maps invariant tensors to invariant tensors, one can restrict to the space $\mathrm{Inv}_{U_q(\mathfrak{sl}_3)}(V_S)$ for a lot of purposes. Note that the general case, that is arbitrary tensors, is harder and work in progress.
\vskip0.5cm
Recall that Murakami, Ohtsuki and Yamada gave in~\cite{moy} a variant of the skein calculus known as \textit{MOY relations} or as \textit{MOY calculus}. The HOMFLY polynomial $P_n(\cdot)$ with the substitution $a=q^n$,$b=q-q^{-1}$ with $n>1$ from above can also be calculated by the following recursive rules. Notice that these rules can be drastically simplified using Kauffman's calculus if $n=2$. Recall that $[m]=\frac{q^m-q^{-m}}{q-q^{-1}}$ denotes the \textit{quantum integer}.
\begin{itemize}
\item $P_n(\overcrossing)=q^{n-1}P_n(\osmoothing)-q^{n}P_n(\ttrivalent)$ (recursion rule 1).
\item $P_n(\undercrossing)=q^{1-n}P_n(\osmoothing)-q^{-n}P_n(\ttrivalent)$ (recursion rule 2).
\item The \textit{circle removal}
\[
\xy(0,0)*{\includegraphics[width=20px]{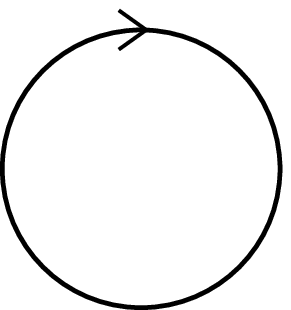}}\endxy\;\; =\;\; [n].
\]
\item The \textit{two digon removals}
\[
\xy(0,0)*{\includegraphics[width=70px, angle=-90]{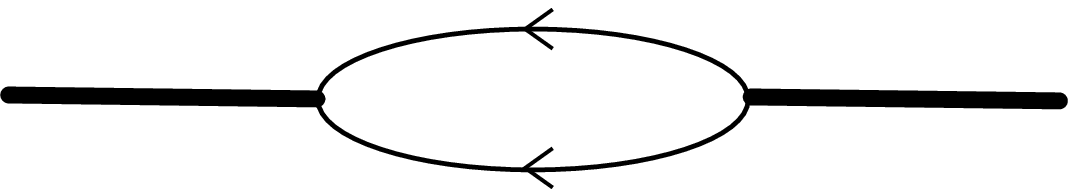}}\endxy\;\; =[2]\cdot\xy(0,0)*{\includegraphics[width=70px, angle=-90]{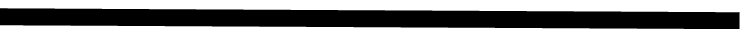}}\endxy\;\;\text{ and }\;\;\xy(0,0)*{\includegraphics[width=70px, angle=90]{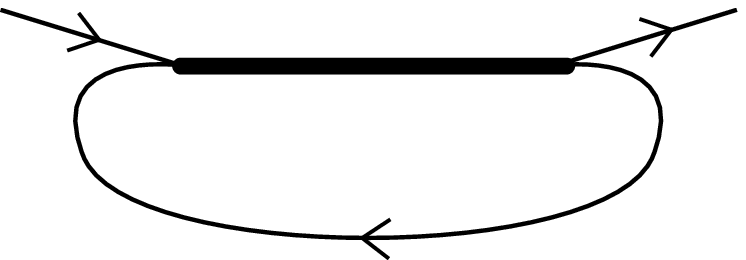}}\endxy\;\; =[n-1]\cdot\xy(0,0)*{\includegraphics[width=70px, angle=90]{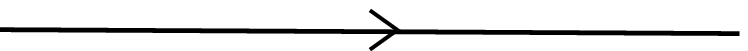}}\endxy
\]
\item The \textit{first square removal}
\[
\xy(0,0)*{\includegraphics[width=50px]{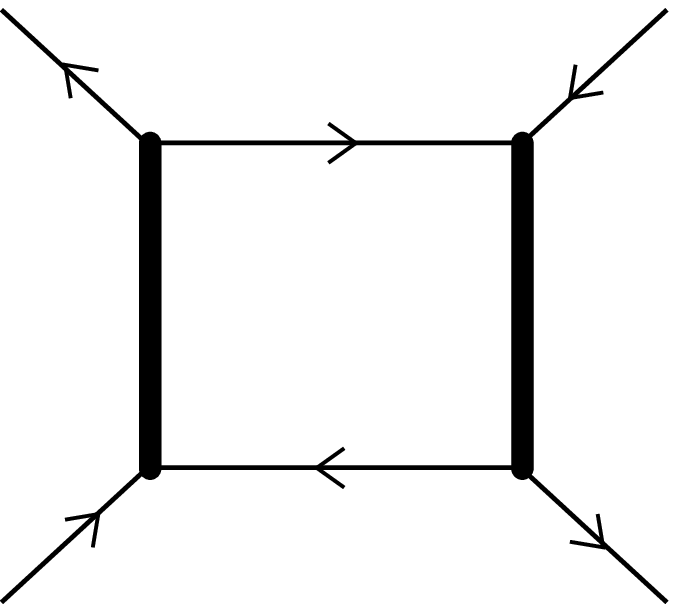}}\endxy\;\; =[n-2]\cdot\xy(0,0)*{\includegraphics[height=45px]{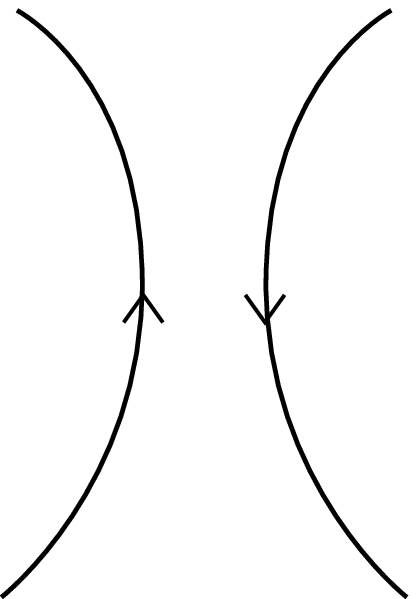}}\endxy+\xy(0,0)*{\includegraphics[width=45px]{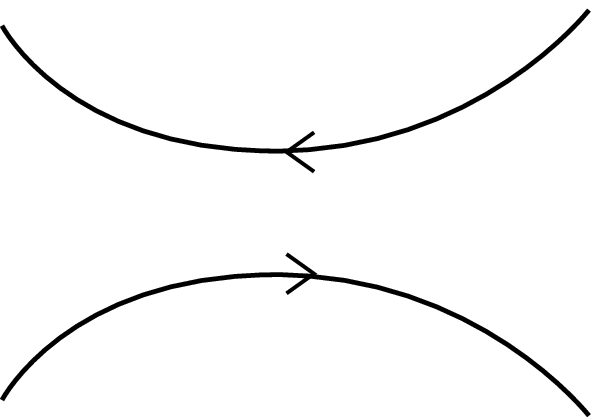}}\endxy
\]
\item The \textit{second square removal}
\[
\xy(0,0)*{\includegraphics[width=65px, angle=90]{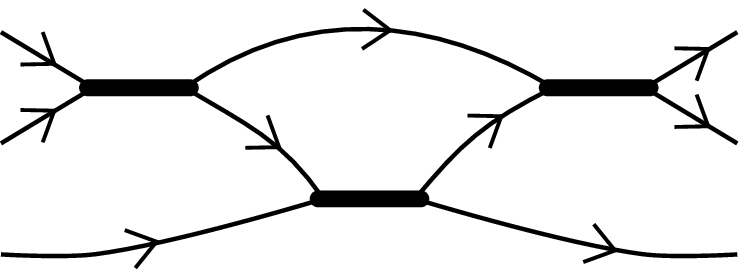}}\endxy\;\;+\;\;\xy(0,0)*{\includegraphics[width=65px, angle=90]{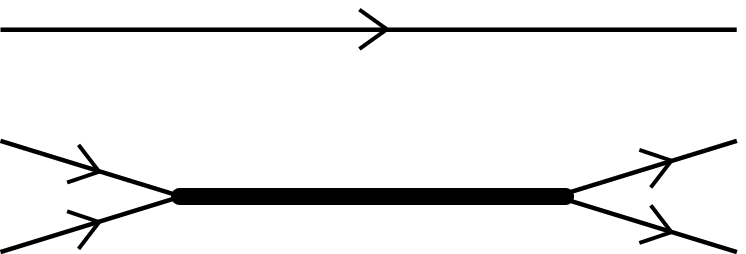}}\endxy\;\;=\;\;
\xy(0,0)*{\includegraphics[width=65px, angle=90]{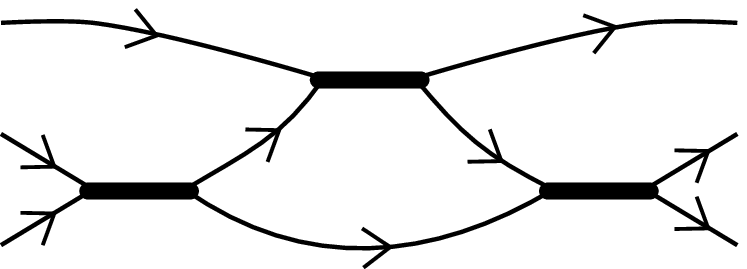}}\endxy\;\;+\;\;\xy(0,0)*{\includegraphics[width=65px, angle=90]{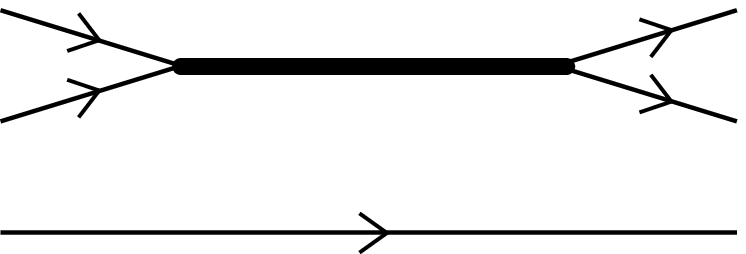}}\endxy\;\; 
\]
\end{itemize}
Note that in the case $n=3$ the two digon removals are the same, the first square removal does not contain any quantum integers any more and can replace the second square removal, i.e. they simplify to the so-called \textit{Kuperberg relations}.

To summarise, we give the following diagram, the \textit{classical, uncategorified} picture.
\begin{center}
\begin{minipage}{0.49\linewidth}
\begin{xy}
  \xymatrix{
      \mathfrak{sl}_n\text{-webs}\phantom{s} \ar@{<->}[rr]^{\text{Intertwiners}\phantom{sgs}} \ar[rd]|{Kauffman, Kuperberg, MOY}  &     &  U_q(\mathfrak{sl}_n)\text{-Tensors}\phantom{sg} \ar[dl]|{Reshetikhin, Turaev}  \\
                             &  \mathfrak{sl}_n\text{-knot polynomials}  &
  }
\end{xy}
\end{minipage}
\end{center}
\vskip0.5cm
Kuperberg showed in~\cite{kup} that the web space $W_S$ of $\mathfrak{sl}_3$ webs with boundary $S$ is isomorphic to the space of invariant tensors $\mathrm{Inv}_{U_q(\mathfrak{sl}_3)}(V_S)$ mentioned above. Without giving the details here, by so-called \textit{$q$-skew Howe duality}, which we explain in Section~\ref{sec-webhowea}, this implies that
\[
V_{(3^k)}\cong\bigoplus_SW_S,
\]
where $V_{(3^k)}$ is the irreducible $U_q(\mathfrak{gl}_{3k})$ representation of highest weight $\lambda=(3^k)$ and, by restriction, this gives rise to a $U_q(\mathfrak{sl}_{3k})$ representation.

Hence, we get the \textit{``Howe dual picture''} to the one from above (we note that the case $m>3$ or the case with arbitrary representations is work in progress).
\begin{center}
\begin{minipage}{0.49\linewidth}
\begin{xy}
  \xymatrix{
      \mathfrak{sl}_3\text{-webs}\phantom{s} \ar@{<->}[rr]^{\text{Howe duality}\phantom{sgs}} \ar[rd]|{Kauffman, Kuperberg, MOY\phantom{....}}  &     &  U_q(\mathfrak{sl}_n)\text{-Irreducibles}\phantom{sg} \ar[dl]|{\phantom{....}Lusztig, Cautis, Kamnitzer, Licata}  \\
                             &  \mathfrak{sl}_3\text{-knot polynomials}  &
  }
\end{xy}
\end{minipage}
\end{center}
\vskip0.5cm
What about the ``categorified world'' now? Recall that we explained in Section~\ref{sec-intro} how Khovanov in~\cite{kh1} gave a \textit{categorification} of the Jones polynomial. Shortly after his breakthrough, he together with Rozansky in~\cite{kr1} and~\cite{kr2}, gave a \textit{categorification of the $\mathfrak{sl}_n$ polynomial} using matrix factorisations. Others, like Khovanov in~\cite{kh3}, Mackaay, Sto\v{s}i\'{c} and Vaz, see~\cite{mv1},~\cite{mv2} and~\cite{msv} gave a \textit{categorification} based on foams in the spirit of Bar-Natan~\cite{bn2}. These homologies are highly interesting and studied from different viewpoints nowadays.

Other approaches are due to Mazorchuk and Stroppel in~\cite{mast} using sophisticated techniques and constructions in category $\mathcal O$, i.e. techniques from of representation theory, and another due to Cautis and Kamnitzer in~\cite{caka} using constructions from algebraic geometry.

As before, one wants to know how all of this fits into a \textit{``bigger picture''} and not just \textit{``that''} something is true. That is one of the reasons people started to look for \textit{categorifications of quantum groups}, i.e. if a $\mathfrak{g}$-invariant can be obtained by studying the representation category of $U_q(\mathfrak{g})$ in the spirit of Reshetikhin and Turaev, there categorifications should be obtained by studying some kind of \textit{``$2$-representation category''} of the categorification associated to $U_q(\mathfrak{g})$. In particular, the ``Khovanov like homologies'' should be obtained in this way.

Indeed, such an approach follows from Webster~\cite{web1} and~\cite{web2}. We note that his work utilises a connection to the picture like $\mathfrak{sl}_n$ categorifications indirectly using Mazorchuk and Stroppel's work.

To summarise, we give the following diagram, the \textit{categorified} picture, the picture we still want to understand.
\begin{center}
\begin{minipage}{0.49\linewidth}
\begin{xy}
  \xymatrix{
      \mathfrak{sl}_n\text{-foams} \ar@{<->}[rr]^{\textbf{???}\phantom{sgsgs}} \ar[rd]|{Khovanov, Khovanov-Rozansky}  &     &  \mathfrak{sl}_n\text{-string diagrams} \ar[dl]|{Webster}  \\
                             &  \mathfrak{sl}_n\text{-knot homologies}  &
  }
\end{xy}
\end{minipage}
\end{center}
\vskip0.5cm
Let us briefly explain instead what we can say about the ``Howe dual picture''. We defined in~\cite{mpt} the $\mathfrak{sl}_3$ analogue of Khovanov's arc algebras $H_n$, introduced in~\cite{kh4}. We call them \textit{web algebras} and denote them by $K_S$, where $S$ is a sign string, i.e. a string of $+$ and $-$ signs which correspond to the two fundamental representations of $U_q(\mathfrak{sl}_3)$. Khovanov uses in his paper so-called \textit{arc diagrams}, which give a diagrammatic presentation of the representation theory of $U_q(\mathfrak{sl}_2)$. These diagrams are related to the Kauffman calculus for the Jones polynomial mentioned above.

Since we defined an $\mathfrak{sl}_3$ analogue, we use the Kuperberg webs, introduced by Kuperberg in~\cite{kup}, mentioned above. These webs give a diagrammatic presentation of the representation theory of $U_q(\mathfrak{sl}_3)$. And of course, instead of $\mathfrak{sl}_2$ cobordisms, which Bar-Natan used in~\cite{bn2} to give his formulation of Khovanov's categorification, we use Khovanov's~\cite{kh3} $\mathfrak{sl}_3$ foams.

To be more precise, in Section~\ref{sec-webhoweb}, we show the following. Let $V_S=V_{s_1}\otimes\cdots\otimes V_{s_n}$, where $V_+$ is the basic $U_q(\mathfrak{sl}_3)$ representation and $V_-$ its dual. Kuperberg~\cite{kup} proved, as indicated above, that $W_S$, the space of $\mathfrak{sl}_3$ webs whose boundary is determined by $S$, is isomorphic to 
$\mathrm{Inv}_{U_q(\mathfrak{sl}_3)}(V_S)$, the space of invariant tensors in $V_S$. Our algebra can be seen as a \textit{categorification} of this, i.e. we show 
\[
K^{\oplus}_0\left(K_S\text{-}\mathrm{p\textbf{Mod}}_{\mathrm{gr}}\right)\cong W_S^{\mathbb{Z}},
\]
for any $S$. Here $K_0$ denotes the Grothendieck group and the superscript $\mathbb{Z}$ denotes the integral form and $K_S\text{-}\mathrm{p\textbf{Mod}}$ the category of finite dimensional, projective $K_S$-modules.

In order to obtain this result, we have categorified an instance of the $q$-skew Howe duality mentioned above, as we explain in Section~\ref{sec-webhoweb}. Without giving the details here, we get the \textit{``categorified Howe dual picture''} (the general case is again work in progress).
\begin{center}
\begin{minipage}{0.49\linewidth}
\begin{xy}
  \xymatrix{
      \mathfrak{sl}_3\text{-foams} \ar@{<->}[rr]^{\text{Howe 2-duality}\phantom{sgsgs}} \ar[rd]|{Khovanov, Khovanov-Rozansky}  &     &  \mathfrak{sl}_n-\text{cyl. KRL algebras} \ar[dl]|{Chuang, Rouquier}  \\
                             &  \mathfrak{sl}_3\text{-knot homologies}  &
  }
\end{xy}
\end{minipage}
\end{center}
\vskip0.5cm
But this is only one reason to study these \textit{web algebras}. Since the Jones polynomial and the $\mathfrak{sl}_n$ polynomial in general are known to be related to different branches of modern mathematics, the categorifications should reflect these connections on a \textit{higher level} and one has possibly more sophisticated connections. We explain some connections of our work in the following.
\vskip0.5cm
As we showed in our paper, see Section~\ref{sec-webcenter}, the center of the algebra $K_S$ is graded isomorphic to the cohomology ring of a certain \textit{Spaltenstein variety} $X^{\lambda}_{\mu}$, an interesting variety from combinatorial, algebraic geometry. To be more precise, if one has a nilpotent endomorphism $N\colon\bC^m\to\bC^m$, then the classical \textit{Springer fiber} is the variety given by the flags fixed under $N$. Generalising to partial flags gives the Spaltenstein varieties, introduced by Spaltenstein~\cite{spa}. Their geometry is still not well understood.

A related aspect is the following. In~\cite{fkk}, Fontaine, Kamnitzer and Kuperberg study spiders from the viewpoint of algebraic geometry. For $\mathfrak{sl}_3$ these 
spiders are exactly the webs that we study.

Given a sign string $S$, the so-called \textit{Satake fiber} $F(S)$, denoted 
$F(\overrightarrow{\lambda})$ in~\cite{fkk}, is isomorphic to the Spaltenstein variety $X^{\lambda}_{\mu}$ mentioned above. Moreover, given a web $w$ with boundary corresponding to $S$, Fontaine, Kamnitzer and Kuperberg also define a variety $Q(D(w))$. They call it the \textit{web variety}. A question asked by Kamnitzer is how their work is related to ours. We give a more detailed description of his question later in Section~\ref{sec-websum}.
\vskip0.5cm
Another connection is given in Section~\ref{sec-webhoweb}, i.e. we show that our algebra is \textit{Morita equivalent} (it has the ``same'' representation theory), as a certain cyclotomic Khovanov-Lauda-Rouquier algebra $R_{(3^k)}$. By Brundan and Kleshchev graded isomorphism given in \cite{bk1}, we obtain that our algebra has the ``same'' representation theory as certain so-called \textit{cyclotomic Hecke algebras}. These algebras, introduced by Ariki and Koike~\cite{ak} and independently by Brou\'{e} and Malle~\cite{bm}, are generalisations of Hecke algebras, i.e. \textit{quantised} versions of the group ring of symmetric groups $S_m$, in the sense that the Hecke algebras are cyclotomic Hecke algebras of level one. One amazing aspect about these algebras is that they contain the Hecke algebras of type $A$ and $B$ as special cases and they are therefore useful to study the \textit{modular} representation theory of finite classical groups of Lie type. They are studied by varies mathematicians nowadays. An introduction to these algebras can be found for example in a lecture notes of Ariki~\cite{ari}.
\vskip0.5cm
It is worth noting that, in the study of the representation theory of $\mathfrak{sl}_n$, the case $n=3$ can be seen as a \textit{blueprint} how to tackle the case $n>3$, while the $n=2$ case seems to be ``too special'' to generalise. Let us explain why we expect something similar in our case for the results in Section~\ref{sec-web}, although the combinatorics get quite hard for $n>3$. 

For any string $S=(s_1,\ldots,s_n)$, such that $1\leq s_i\leq n-1$, Fontaine, generalising work of Westbury~\cite{wes}, constructs in~\cite{fon} a $\mathfrak{sl}_n$ web basis $B_S^n$ by generalising Khovanov and Kuperberg's $\mathfrak{sl}_3$ growth algorithm~\cite{kk}. To any $w\in B_S^n$, one can associate the coloured Khovanov-Rozansky matrix factorization $M_w$, as defined by Wu~\cite{wu} and Yonezawa~\cite{yo}. For any $u,v\in B_S^n$, one can then define 
\[
{}_uK^n_v=\mathrm{Ext}(M_u,M_v).
\]
The multiplication in 
\[
K_S^n=\bigoplus_{u,v\in B_S^n} {}_uK^n_v
\] 
is induced by the composition of homomorphisms of matrix factorizations. Note that for $\mathfrak{sl}_3$, the definition using matrix factorizations indeed gives an algebra isomorphic to $K_S$, as follows from the equivalence between matrix factorizations and foams for $\mathfrak{sl}_3$ proved in~\cite{mv2}. While the author writes this thesis, Mackaay and Yonezawa are preparing a paper~\cite{my} on the $\mathfrak{sl}_n$ web algebra following the ideas explained above.
\vskip0.5cm
The author notes that a \textit{virtual version} of Khovanov's arc algebra, our web algebra or even versions for $n>3$ would be also interesting to study. But this is not done yet.

More details concerning our paper~\cite{mpt} are summarised in Section~\ref{sec-websum}.
\newpage
\section{Virtual Khovanov homology}\label{sec-vkh}
\subsection{A brief summary}\label{sec-vkhsum}
Let us give a brief summary of the constructions in Section~\ref{sec-vkh}. We will assume that the reader is not completely unfamiliar with the notion of the classical Khovanov complex as mentioned before in~\ref{sec-intro}, e.g. the construction of the Khovanov cube (more about cubes in Section~\ref{sec-techcube}) based on so-called \textit{resolutions of crossings} as shown in Figure~\ref{figureintroa-3}. There are many good introductions to classical Khovanov homology, e.g. a nice exposition of the classical Khovanov homology can be found in Bar-Natan's paper~\cite{bn1}. Note that this section is based on two preprints~\cite{tub1} and~\cite{tub2} of the author. The summary is informal. We hope to demonstrate that the main ideas of the construction are easy, e.g. the construction is given by an algorithm, general, e.g. it extends all the ``classical'' homologies, but if one works over a ring $R$ of characteristic $2$, then, by setting $\theta\neq 0$, one obtains ``non-classical'' homologies, and has other nice properties, e.g. it has, up to a sign, functorial properties.
\vskip0.5cm
Let $a$ be a word in the alphabet $\{0,1\}$. We denote by $\gamma_a$ the resolution of a v-link diagram $L_D$ with $|a|$ crossings, where the $i$-th crossing of $L_D$ is resolved $a_i\in\{0,1\}$ as indicated in Figure~\ref{figureintroa-3}. Beware that we \textit{only} resolve classical crossings. We denote the number of v-circles, that is closed circles with only v-crossings, in the resolution $\gamma_a$ by $|\gamma_a|$.

Moreover, suppose we have two words $a,b$ with $a_k=b_k$ for $k=1,\dots,|a|=|b|,k\neq i$ and $a_i=0,b_i=1$. Then we call $S\colon\gamma_a\to\gamma_b$ a \textit{(formal) saddle} between the resolutions.

Furthermore, suppose we have a v-link diagram $L_D$ with at least two crossings $c_1,c_2$. We call a quadruple $F=(\gamma_{00},\gamma_{01},\gamma_{10},\gamma_{11})$ of four resolutions of the v-diagram $L_D$ a \textit{face} of the diagram $L_D$, if in all four resolutions $\gamma_{00},\gamma_{01},\gamma_{10},\gamma_{11}$ all crossings of $L_D$ are resolved in the same way except that $c_1$ in resolved $i$ and $c_2$ is resolved $j$ in $\gamma_{ij}$ (with $i,j\in\{0,1\}$). Furthermore, there should be an oriented arrow from $\gamma_{ij}$ to $\gamma_{kl}$ if $i=j=0$ and $k=0,l=1$ or $k=1,l=0$ or if $i=0,j=1$ and $k=l=1$ or if $i=1,j=0$ and $k=l=1$. That is faces look like
\[
\xymatrix{
 & \gamma_{01}\ar[rd]^{S_{*1}} &\\
 \gamma_{00}\ar[rd]_{S_{*0}}\ar[ru]^{S_{0*}} &  & \gamma_{11},\\
 & \gamma_{10}\ar[ru]_{S_{1*}} &}
\]
where the * for the saddles should indicate the change $0\to 1$. 

We also consider \textit{algebraic faces} of a resolution. That is the same as above, but we replace $\gamma_a$ with $\bigotimes_n A$, if $\gamma_a$ has $n$ components. Here $A$ is an $R$-module and $R$ is a commutative, unital ring.

Moreover, recall that the differential in the classical Khovanov complex consists of a multiplication $m\colon A\otimes A\to A$ and a comultiplication $\Delta\colon A\to A\otimes A$ for the $R$-algebra $A=R[X]/(X^2)$ with gradings $\deg 1=1,\deg X=-1$. The comultiplication $\Delta$ is given by
\[
\Delta\colon A\to A\otimes A;\begin{cases}

  1\mapsto 1\otimes X+X\otimes 1,\\
  X\mapsto X\otimes X.
\end{cases}
\]
The problem in the case of v-links is the emergence of a new map. This happens, because it is possible for v-links that a saddle $S\colon\gamma_a\to\gamma_b$ between two resolutions does not change the number of v-circles, i.e. $|\gamma_a|=|\gamma_b|$. This is a difference between c-links and v-links, i.e. in the first case one always has $|\gamma_a|=|\gamma_b|+1$ or $|\gamma_b|+1=|\gamma_a|$.

So in the algebraic complex we need a new map $\cdot\theta\colon A\rightarrow A$ together with the classical multiplication and comultiplication $m\colon A\otimes A\rightarrow A$ and $\Delta\colon A\rightarrow A\otimes A$. As we will see later the only possible way to extend the classical Khovanov complex to v-links is to set $\theta=0$ (for $R=\bZ$). But then a face could look like (maybe with extra signs).
\begin{equation}\label{probcube}
\begin{gathered}
\begin{xy}
  \xymatrix{
 & A\otimes A\ar[rd]^{m} &\\
 A\ar[ru]^{\Delta}\ar[rd]_{\cdot\theta} &  & A.\\
 & A\ar[ru]_{\cdot\theta} &}
\end{xy}
\end{gathered}
\end{equation}
We call such a face a \textit{problematic face}. With $\theta=0$ and the classical $\Delta,m$, this face does not commute (for $R=\bZ$). Therefore, there is no straightforward extension of the Khovanov complex to v-links. Moreover, in the cobordism based construction of the classical Khovanov complex, there is no corresponding cobordism for $\theta$.
\vskip0.5cm
To solve these problems we consider a certain category called $\ucob_R(\emptyset)$, i.e. a category of (possible non-orientable) cobordisms with boundary decorations $\{+,-\}$. Roughly, a punctured M\"obius strip plays the role of $\theta$ and the decorations keep track of how (orientation preserving or reversing) the surfaces are glued together. Hence, in our category we have different (co)multiplications, depending on the different decorations. Furthermore, in order to get the right signs, one has to use constructions related to $\wedge$-products (sometimes called skew-products). Note that this is rather surprising, since such constructions are not needed for Khovanov homology  in the c-case. And furthermore, such constructions are in the c-case related to so-called \textit{odd} Khovanov homology. But we show that in fact our construction agrees for c-links with the (even) Khovanov homology (see Theorem~\ref{thm-classic}).

The following table summarises the connection between the classical and the virtual case.
\begin{center}
\begin{tabular}{|c|c|c|}
\hline $\phantom{.}$ & \textbf{Classical} & \textbf{Virtual} \\ 
\hline \textbf{Objects} & c-link resolutions & v-link resolutions \\ 
\hline \textbf{Morphisms} & Orientable cobordisms & Possible non-orientable cobordisms \\ 
\hline \textbf{Cobordisms} & Embedded & Immersed \\ 
\hline \textbf{Decorations} & None & $+,-$ at the boundary \\ 
\hline \textbf{Signs} & Usual & Related to $\wedge$-products \\ 
\hline 
\end{tabular} 
\end{center}
Hence, a main point in the construction of the virtual Khovanov complex is to say which saddles, i.e. morphisms, are orientable and which are non-orientable, how to place the decorations and how to place the signs. This is roughly done in the following way.
\begin{itemize}
\item Every saddle either splits one circle (orientable, called \textit{comultiplication}, denoted $\Delta$. See Figure~\ref{figure1-1} - fourth column), glues two circles (orientable, called \textit{multiplication}, denoted $m$. See Figure~\ref{figure1-1} - fifth column) or does not change the number of circles at all (non-orientable, called \textit{M\"obius cobordism}, denoted $\theta$. See Figure~\ref{figure1-1} - last morphism).
\item Every saddle $S$ can be locally denoted (up to a rotation) by a formal symbol $S\colon\smoothing\rightarrow\hsmoothing$ (both smoothings are neighbourhoods of the crossing). The \textit{glueing numbers}, i.e. the decorations, are now spread by \textit{choosing a formal orientation} for the resolution. We note that the construction will not depend on the choice.
\item After all resolutions have an orientation, a saddle $S$ could for example be of the form $S\colon\du\rightarrow\ler$. This is the \textit{standard form}, i.e. in this case all glueing number will be $+$.
\item Now spread the decorations as follows. Every boundary component gets a $+$ iff the orientation is as in the standard case and a $-$ otherwise. The \textit{degenerated} cases (everything non-alternating), e.g. $S\colon\dd\rightarrow\rir$, are the non-orientable surfaces and do not get any decorations. Compare to Table~\ref{tab-deco} in Definition~\ref{defn-deco}.
\item The signs are spread based on a numbering of the v-circles in the resolutions and on a special x-marker for the crossings. Note that without the x-marker one main lemma, i.e. Lemma~\ref{lem-virtualisation}, would not work.
\end{itemize}
Or summarised in Figure~\ref{figure0-big}. The complex below is the complex of a trivial v-link diagram.
\begin{figure}[ht]
  \centering
     \includegraphics[scale=0.525]{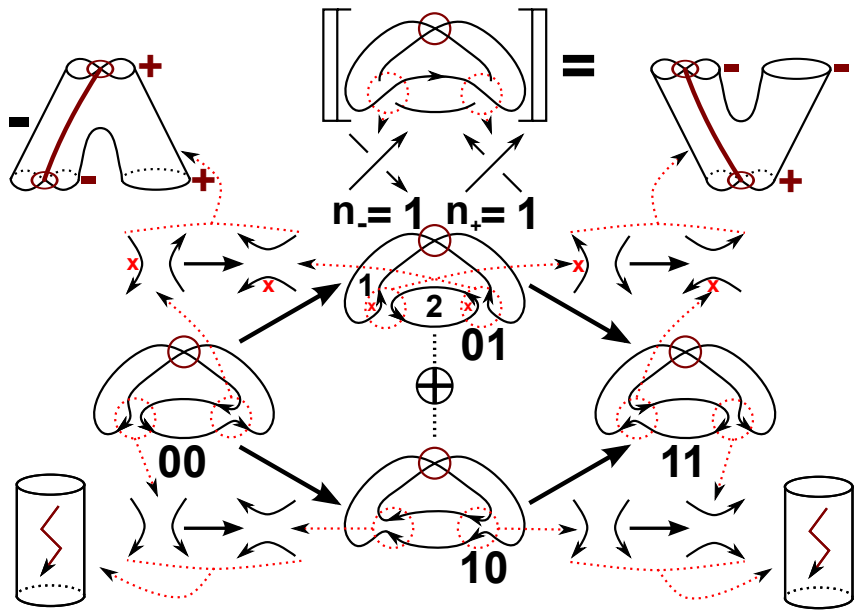}
  \caption{The virtual Khovanov complex of the unknot.}\label{figure0-big}
\end{figure}

To construct the virtual Khovanov complex for v-tangles we need to extend these notions in such a way that they still work for ``open'' cobordisms. A first generalisation is easy, i.e. we will still use immersed, possible non-orientable surfaces with decorations, but we allow vertical boundary components, e.g. the three \textit{v-Reidemeister cobordisms vRM1, vRM2 and vRM3} in Figure~\ref{figure0-reide}.
\begin{figure}[ht]
  \centering
     \includegraphics[scale=0.7]{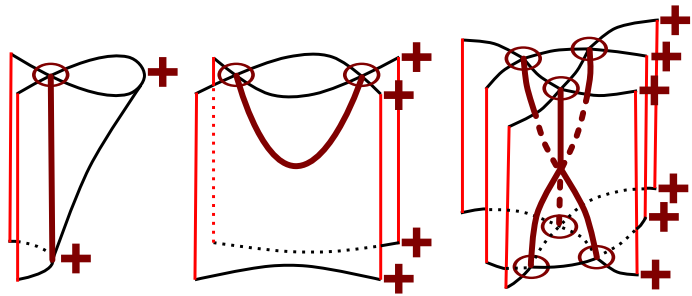}
  \caption{The virtual Reidemeister cobordisms.}
  \label{figure0-reide}
\end{figure}

One main point is the question what to do with the \textit{``open'' saddles}, i.e. saddles with no closed boundary. A possible solution is to define them to be zero. But this has two major problems. First the loss of information is big and second we would not have local properties as in the classical case (``tensoring'' of smaller parts), since an open saddle can, after closing some of his boundary circles, become either $m$, $\Delta$ or $\theta$. See Figure~\ref{figure0-order}. 
\begin{figure}[ht]
  \centering
     \includegraphics[scale=0.42]{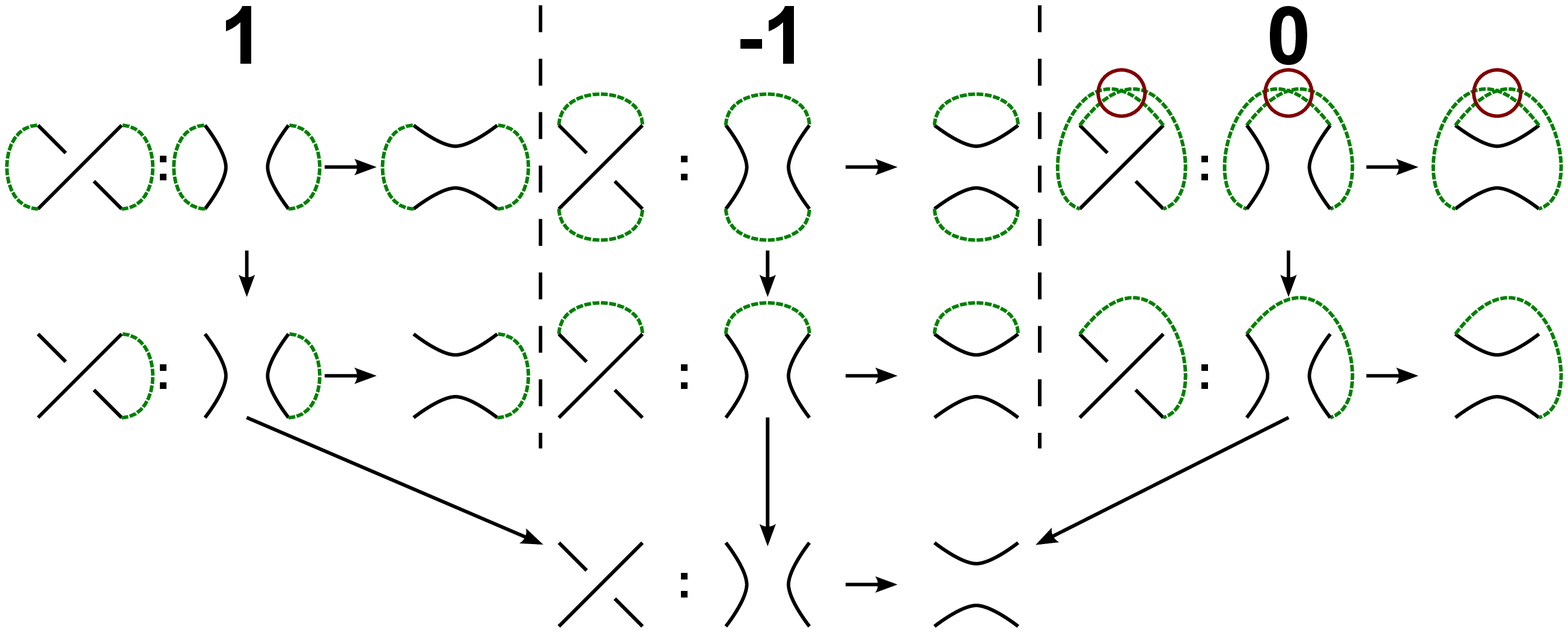}
  \caption{All of the closed cases give rise to the unclosed.}
  \label{figure0-order}
\end{figure}

Hence, an information mod 3 is missing. We therefore consider morphisms with an \textit{indicator}, i.e. an element of the set $\{0,+1,-1\}$. Then, after taking care of some technical difficulties, the concept extends from c-tangles to v-tangles in a suitable way. That is, we can ``tensor'' smaller pieces together as indicated in the Figure~\ref{figure0-main}.
\begin{figure}[ht]
\centering
     \includegraphics[scale=0.462]{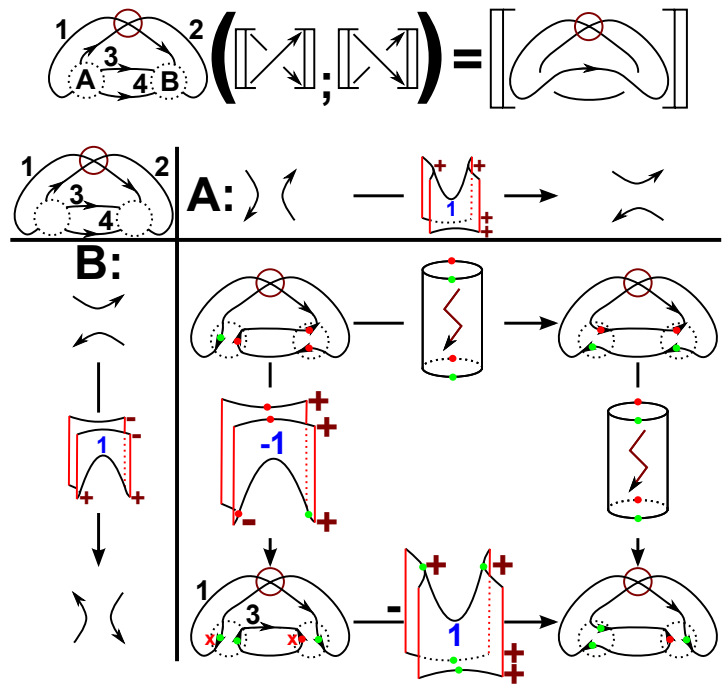}
  \caption{After we have fixed an orientation/numbering of the circuit diagram, we only have to compare whether the local orientations match (green) or mismatch (red) and compose if necessary with $\Phi^-_+$ (red). Iff we have a double mismatch at the top and bottom, then we add a bolt symbol.}
  \label{figure0-main}
\end{figure}

It should be noted that there are some technical points that make our construction only semi-local (a disadvantage that arises from the fact that ``non-orientability'' is not a local property). Note that indicators, if necessary, are pictured on the surfaces.
\vskip0.5cm
The outline of the Section~\ref{sec-vkh} is as follows.
\begin{itemize}
\item In Section~\ref{sec-vkhcat} we define the category of (possible non-orientable) cobordisms with boundary decorations. First in the ``closed'' case in Definition~\ref{defn-category} and then more general in the ``open'' case in Definition~\ref{defn-category3}. We also proof/recall some basic facts in Section~\ref{sec-vkhcat}.
\item In Section~\ref{sec-vkhcom} we define~\ref{defn-topcomplex} the virtual Khovanov complex for v-links. It is a v-link invariant (Theorem~\ref{thm-geoinvarianz}) and agrees with the construction in the c-case (Theorem~\ref{thm-classic}). There are two important things about the construction. The first is that there are many choices in the definition of the virtual complex, but we show in~\ref{lem-commutativeindependence} that different choices give isomorphic complexes. Second, it is not clear that the complex is a well-defined chain complex, but we show this fact in Theorem~\ref{theo-facescommute} and Corollary~\ref{cor-chaincomplex}. In order to show that the construction gives a well-defined chain complex we have to use a ``trick'', i.e. we use a move called \textit{virtualisation}, as shown in Figure~\ref{figure0-virt}, to reduce the question whether the faces of the virtual Khovanov cube are anticommutative to a finite and small number of so-called \textit{basic faces} (see Figure~\ref{figure-basic}).
\begin{figure}[ht]
\centering
     \includegraphics[scale=0.6]{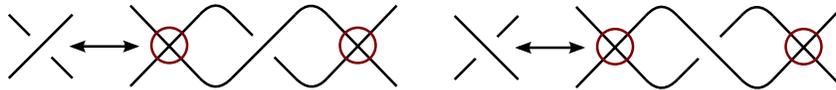}
  \caption{The virtualisation of a crossing.}
  \label{figure0-virt}
\end{figure}
\item In Section~\ref{sec-vkhfa} we show that our constructions can be compared to so-called \textit{skew-extended Frobenius algebras}~\ref{thm-tututheo}. With this we are able to classify all possible v-link homologies from our approach~\ref{thm-class}. We note that all the classical homologies are included. And we can therefore show in Corollary~\ref{cor-khovhomocat} that our construction is a categorification of the virtual Jones polynomial.
\item The Sections~\ref{sec-vkhtan} and~\ref{sec-vkhca} are analogues of the earlier sections, but for v-tangles.
\item The Section~\ref{sec-vkhapp} uses that our construction is semi-local~\ref{thm-semiloc}. As a result, we can still show that Lee's variant of Khovanov homology is in some sense degenerated~\ref{thm-leedeg}. This fact is one of the main ingredients to define Rasmussen's invariant in the classical case.
\item The Section~\ref{sec-vkhpro} gives some calculation results with a MATHEMATICA program written by the author. It is worth noting that we give examples of v-links with seven crossings which can not be distinguished by the virtual Jones polynomial, but by virtual Khovanov homology.
\item We have collected some open questions in the final Section~\ref{sec-vkhend}. 
\end{itemize}
\subsubsection{Notation}\label{notation}
We call $\smoothing$ the \textit{$0$-} and $\hsmoothing$ the \textit{$1$-resolution} of the crossing $\slashoverback$ for a given v-link diagram $L_D$ or v-tangle diagram $T^k_D$. For an oriented v-link diagram $L_D$ or v-tangle diagram $T^k_D$ we call $\overcrossing$ a \textit{positive} and $\undercrossing$ a \textit{negative} crossing. The \textit{number of positive crossings} is denoted by $n_+$ and the \textit{number of negative crossings} is denoted by $n_-$.
\vskip0.5cm
For a given v-link diagram $L_D$ or v-tangle diagram $T^k_D$ with $n$-numbered crossings we define a collection of closed curves and open strings $\gamma_a$ in the following way. Let $a$ be a word of length $n$ in the alphabet $\left\{0,1\right\}$. Then $\gamma_a$ is the collection of closed curves and open strings which arise, when one performs a $a_i$-resolution at the $i$-th crossing for all $i=1,\dots,n$. We call such a collection $\gamma_a$ the \textit{$a$-th resolution} of $L_D$ or $T^k_D$. All appearing v-circles should be numbered with consecutive numbers from $1,\dots,k_a$ in these resolutions, where $k_a$ is the total number of v-circles of the resolution $\gamma_a$.

We can choose an orientation for the different components of $\gamma_a$. We call such a $\gamma_a$ an \textit{orientated resolution}, i.e. every v-crossing of the resolution $\gamma_a$ should look like $\virtualoriented$. Then a local neighbourhood of a $0,1$-resolved crossing could for example look like $\uu$. We call these neighbourhoods \textit{oriented crossing resolutions}.

If we ignore orientations, then there are $2^n$ different resolutions $\gamma_a$ of $L_D$ or $T^k_D$. We say a resolution has length $m$ if it contains exactly $m$ 1-letters. That is $m=\sum_{i=1}^na_i$.

For two resolutions $\gamma_a$ and $\gamma_{a'}$ with $a_r=0$ and $a'_r=1$ for one fixed $r$ and $a_i=a'_i$ for $i\neq r$ we define a \textit{saddle between the resolutions $S$}. This means: Choose a small (no other crossing, classical or virtual, should be involved) neighbourhood $N$ of the $r$-th crossing and define a cobordism between $\gamma_a$ and $\gamma_{a'}$ to be the identity outside of $N$ and a saddle inside of $N$. Note that we, by a slight abuse of notation, call these cobordisms saddles although they contain in general some cylinder components.
\vskip0.5cm
From now on we consider \textit{faces} $F=(\gamma_{00},\gamma_{01},\gamma_{10},\gamma_{11})$ of four resolutions, as mentioned above, always \textit{together with the saddles} between the resolutions. We denote the saddles for example by $S_{0*}\colon\gamma_{00}\to\gamma_{01}$, where the position of the $*$ indicates the change $0\to 1$.

It should be noted that any v-link or v-tangle diagram should be oriented in the usual sense. But with a slight abuse of notation, we will suppress this orientation throughout the whole Section~\ref{sec-vkh}, since the afore mentioned oriented resolutions are main ingredients of our construction and easy to confuse with the usual orientations. Recall that these usual orientations are needed for the shifts in homology gradings, see for example~\cite{bn1}.
\vskip0.5cm
Sometimes we need a so-called \textit{spanning tree argument}, i.e. choose a spanning tree of a cube (as in Figure~\ref{figure0-spantree}) and change e.g. orientations of resolutions such that the edges of the tree change in a suitable way, starting at the rightmost leafs, then remove them and repeat. Notice that two cubes together with a chain map between them form again a bigger cube. It is worth noting that most of the spanning tree arguments work out in the end because of certain preconditions, e.g. the anticommutativity of faces.
\begin{figure}[ht]
  \centering
     \includegraphics[width=0.775\linewidth]{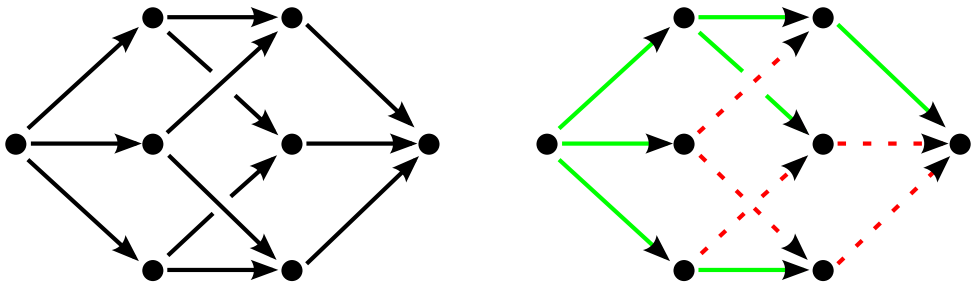}
  \caption{A Khovanov cube and a spanning tree of the cube (green edges).}
  \label{figure0-spantree}
\end{figure}

Moreover, we have collected some facts from homological algebra that we need in Section~\ref{sec-vkh} in the Section~\ref{sec-techhomalg} and in Section~\ref{sec-techcube}.
\subsection{The topological category}\label{sec-vkhcat}
\subsubsection{The topological category for v-links}
In this section we describe our topological category which we call $\ucob_R(\emptyset)$. This is a category of cobordisms between v-link resolutions in the spirit of Bar-Natan~\cite{bn2}, but we admit that the cobordisms are non-orientable as in~\cite{tutu}.
\vskip0.5cm
The basic idea of the construction is that the usual pantsup- and pantsdown-cobordisms do not satisfy the relation $m\circ\Delta=\theta^2$. But we need this relation for the face from~\ref{probcube}. This is the case, because we need an extra information for v-links, namely \textit{how} two cobordisms are glued together.

To deal with this problem, we decorate the boundary components of a cobordism with a formal sign $+,-$. With this construction $m_i\circ\Delta_j$ is sometimes $=\theta^2$ and sometimes $\neq\theta^2$, depending on $i,j=1,\dots,8$. The first case will occur iff $m_i\circ\Delta_j$ is a non-orientable surface.

One main idea of this construction is the usage of a cobordism $\Phi^-_+$ between two circles \textit{different} from the identity $\mathrm{id}^+_+$. See Figure~\ref{figure-idphi}.
\begin{figure}[ht]
  	 \xy
  	 (-40,0)*{\phantom{.}};
     (40,0)*{\includegraphics[scale=0.5]{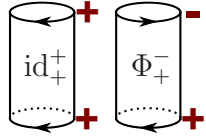}};
     (32,0)*{\mathrm{id}^+_+};
     (46,0)*{\Phi^-_+};
     \endxy
  \centering
  \caption{Glueing the boundary together as indicated can not be done without immersion in the case on the right.}
  \label{figure-idphi}
\end{figure}

Furthermore, we need relations between the decorated cobordisms. One of these relations identifies all boundary preserving homeomorphic cobordisms if their boundary decorations are all equal or are all different (up to a sign). Moreover, some of the standard relations of the category $\cob_R(\emptyset)$ (see for example in the book of Kock~\cite{ko}) should hold. We denote the category with the extra signs by $\ucob_R(\emptyset)$ and the category without the extra signs by $\ucob_R(\emptyset)^*$. Therefore, there will be two different cylinders in these categories.

Note that most of the constructions are easier for $\ucob_R(\emptyset)^*$ than for $\ucob_R(\emptyset)$. That is why we will only focus on the latter category and hope the reader does not have to many difficulties to do similar constructions for $\ucob_R(\emptyset)^*$ while reading this section. 
\vskip0.5cm
At the end of this section we will prove some basic relations (Lemma~\ref{lem-basiscalculations}) between the generators of our category. We also characterise the cobordisms of the face~\ref{probcube} (Proposition~\ref{prop-nonorientablefaces}).

It should be noted that, in order to extend the construction to v-tangle diagrams, we need some more extra notions. We will define them after Definition~\ref{defn-category} in an extra subsection in Definition~\ref{defn-category3} to avoid to many notions at once.
\vskip0.5cm
We start with the following definition. Beware that we consider v-circles as objects and cobordisms together with decorations. We denote the decorations by $+,-$ and illustrate them next to boundary components. Here $R$ denotes a commutative, unital ring of arbitrary characteristic.
\begin{defn}\label{defn-category}\textbf{(The category of cobordisms with boundary decorations)} We describe the category \textit{$\ucob_R(\emptyset)$} in six steps. Note that our category is $R-$pre-additive\footnote{Sometimes also called $R$-category, i.e. the set of morphisms form a $R$-module and composition is $R$-linear.}. The symbol $\amalg$ denotes the disjoint union.

\textbf{The objects:}

The \textit{objects} \textit{$\Ob(\ucob_R(\emptyset)$}) are disjoint unions of numbered \textit{v-circles}. We denote the objects by $\mathcal O=\amalg_{i\in I}\mathcal O_i$. Here $\mathcal O_i$ are the v-circles and $I$ is a finite, ordered index set. Note that, by a slight abuse of notation, we denote the objects by $\mathcal O$ to point out that the category can be seen as a $2$-category (but it is inconvenient for our purpose). The objects of the category are equivalence (modulo \textit{planar isotopies}) classes of four-valent graphs.

\textbf{The generators:}

The \textit{generators} of $\Mor(\ucob_R(\emptyset))$ are the eight cobordisms from Figure~\ref{figure1-1} plus topological equivalent cobordisms, but with all other possible boundary decorations (we do not picture them because one can obtain them using the ones shown after taking the relations below into account). Every orientable generator has a decoration from the set $\{+,-\}$ at the boundary components. We call these decorations the \textit{glueing number} (of the corresponding boundary component).
\begin{figure}[ht]
  	 \xy
  	 (-40,0)*{\phantom{.}};
     (40,0)*{\includegraphics[scale=0.5]{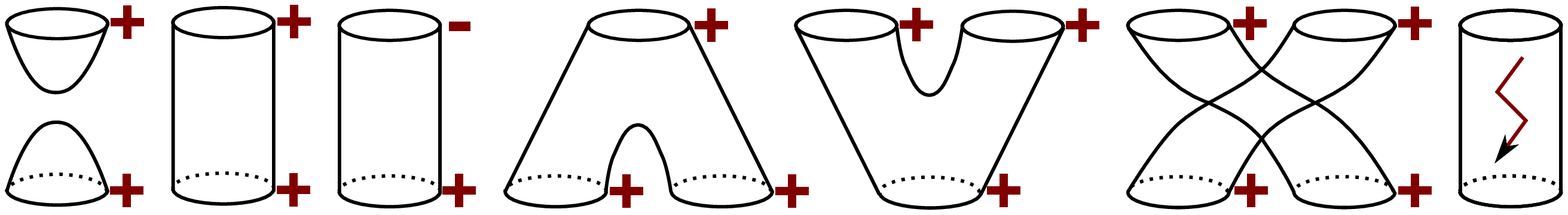}};
     (-27.75,-3.8)*{\varepsilon^+};
     (-27.75,3.8)*{\iota_+};
     (-7.15,-11)*{\mathrm{id}^+_+};
     (7.25,-11)*{\Phi^-_+};
     (54.5,-11)*{m^{++}_+};
     (30,-11)*{\Delta_{++}^+};
     (82.5,-11)*{\tau^{++}_{++}};
     (101.5,-11)*{\theta};
     \endxy
  \centering
  \caption{The generators of the set of morphisms. The cobordism on the right is the M\"obius cobordism, i.e. a two times punctured projective plane.}
  \label{figure1-1}
\end{figure}

We consider these cobordisms up to boundary preserving homeomorphisms (as abstract surfaces). Hence, between circles with v-crossings the (not pictured) generators are the same up to boundary preserving homeomorphisms, but immersed into $\bR^2\times[-1,1]$.
\vskip0.25cm
The eight cobordisms are (from left to right): a \textit{cap-cobordism} and a \textit{cup-cobordism} between the empty set and one circle and vice versa. Both are homeomorphic to a disc $D^2$ and both have a positive glueing number. We denote them by $\iota_+$ and $\varepsilon^+$ respectively.

Two \textit{cylinders} from one circle to one circle. The first has two positive glueing numbers and we denote this cobordism by $\mathrm{id}^+_+$. The second has a negative upper glueing number and a positive lower glueing number and we denote it by $\Phi^-_+$.

A \textit{multiplication-} and a \textit{comultiplication-cobordism} with only positive glueing numbers. Both are homeomorphic to a three times punctured $S^2$. We denote them by $m^+_{++}$ and $\Delta^+_{++}$.

A \textit{permutation-cobordism} between two upper and two lower boundary circles with only positive glueing numbers. We denote it by $\tau^{++}_{++}$.

A \textit{two times punctured projective plane}, also called \textit{M\"obius cobordism}. This cobordism is not orientable, hence it has no glueing numbers. We denote it by $\theta$.
\vskip0.25cm
The composition of the generators is given by glueing them together along their common boundary. In all pictures the upper cobordism is the $C$ in the composition $C^{\prime}\circ C$. The decorations are not changing at all (except that we remove the decorations if any connected component is non-orientable) \textit{before} taking the relations as in the equations~\ref{eq-combrel12},~\ref{eq-combrel3},~\ref{eq-commasss},~\ref{eq-unitrel},~\ref{eq-permrel1},~\ref{eq-permrel2} and~\ref{eq-frobandco} into account. Formally, \textit{before} taking quotients, the composition of the generators also needs internal decorations to remember if the generators where glued together alternating, i.e. minus to plus or plus to minus, or non-alternating. But after taking the quotients as indicated, these internal decorations are not needed any more. Hence, we suppress these internal decorations to avoid a too messy notation.

The reader should keep the informal slogan ``Composition with $\Phi^-_+$ changes the decoration'' in mind. 

\textbf{The morphisms:}

The \textit{morphisms} \textit{$\Mor(\ucob_R(\emptyset))$} are cobordisms between the objects in the following way. Note that we call a morphism non-orientable if any of its connected components is non-orientable.

We identify the collection of numbered v-circles with circles immersed into $\bR^2$. Given two objects $\mathcal O_1,\mathcal O_2$ with $k_1,k_2$ numbered v-circles, a morphism $C\colon\mathcal O_1\to\mathcal O_2$ is a surface immersed in $\bR^2\times[-1,1]$ whose boundary lies only in $\bR^2\times\{-1,1\}$ and is the disjoint union of the $k_1$ numbered v-circles from $\mathcal O_1$ in $\bR^2\times\{1\}$ and the disjoint union of the $k_2$ numbered v-circles from $\mathcal O_2$ in $\bR^2\times\{-1\}$. The morphisms are generated (as abstract surfaces) by the generators from above. It is worth noting that all possible boundary decorations can occur.

\textbf{The decorations:}

Given a $C\colon\mathcal O_1\to\mathcal O_2$ in $\Mor(\ucob_R(\emptyset))$, let us say that the v-circles of $\mathcal O_1$ are numbered from $1,\dots,k$ and the v-circles of $\mathcal O_2$ are numbered from $k+1,\dots,l$.

Every orientable cobordism has a decoration on the $i$-th boundary circle. This decoration is an element of the set $\{+,-\}$. We call this decoration of the $i$-th boundary component the \textit{$i$-th glueing number} of the cobordism.

Hence, the morphisms of the category are pairs $(C,w)$. Here $C\colon\mathcal O_1\to\mathcal O_2$ is a cobordism from $\mathcal O_1$ to $\mathcal O_2$ immersed in $\bR^2\times[-1,1]$ and $w$ is a string of length $l$ in such a way that the $i$-th letter of $w$ is the $i$-th glueing number of the cobordism or $w=0$ if the cobordism is non-orientable.

\textbf{Short hand notation:}

We denote a orientable, connected morphism $C$ by $C^{u}_{l}$. Here $u,l$ are words in the alphabet $\{+,-\}$ in such a way that the $i$-th character of $u$ (of $l$) is the glueing number of the $i$-th circle of the upper (of the lower) boundary. The construction above ensures that this notation is always possible. Therefore, we denote an arbitrary orientable morphism $(C,w)$ by
\[
C=C^{u_1}_{l_1}\amalg\cdots\amalg C^{u_k}_{l_k}.
\]
Here $C^{u_i}_{l_i}$ are its connected components and $u_i,l_i$ are words in $\{+,-\}$. For a non-orientable morphism we do not need any boundary decorations.

\textbf{The relations:}

There are two different types of relations, namely \textit{topological relations} and \textit{combinatorial relations}. The latter relations are described by the glueing numbers and the glueing of the cobordisms. The relations between the morphisms are the relations pictured below, i.e. the three \textit{combinatorial}~\ref{eq-combrel12} for the orientable and~\ref{eq-combrel3} for non-orientable cobordisms, \textit{commutativity} and \textit{cocommutativity} relations~\ref{eq-commasss}, \textit{associativity} and \textit{coassociativity} relations~\ref{eq-commasss}, \textit{unit and counit} relations~\ref{eq-unitrel}, \textit{permutation} relations~\ref{eq-permrel1} and~\ref{eq-permrel2}, a \textit{Frobenius relation} and the \textit{torus and M\"obius} relations~\ref{eq-frobandco} and different \textit{commutation} relations. Latter ones are not pictured, but all of them should hold with a plus sign. If the reader is unfamiliar with these relations, then we refer to the book of Kock~\cite{ko} and hope that it should be clear how to translate his pictures to our context (by adding some decorations). 
\vskip0.5cm
Beware that we have pictured several relations in some figures at once. We have separated them by a thick line.

Moreover, some of the relations contain several cases at once, e.g. in the right part of Equation~\ref{eq-frobandco}. In those cases it should be read: If the conditions around the equality sign are satisfied, then the equality holds.
\vskip0.5cm
The first combinatorial relations are
\begin{align}\label{eq-combrel12}
\xy(0,0)*{\includegraphics[scale=0.25]{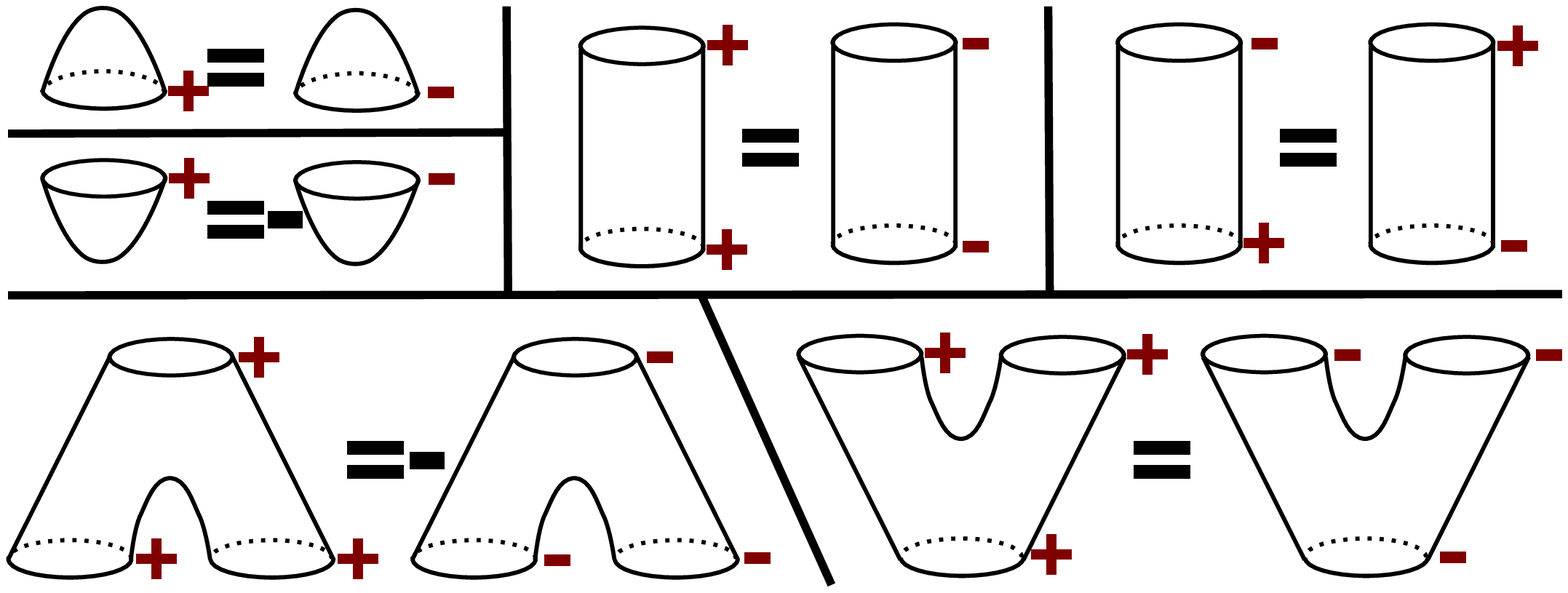}};\endxy\;\;\;\;\;\;\;\;\xy(0,0)*{\includegraphics[scale=0.25]{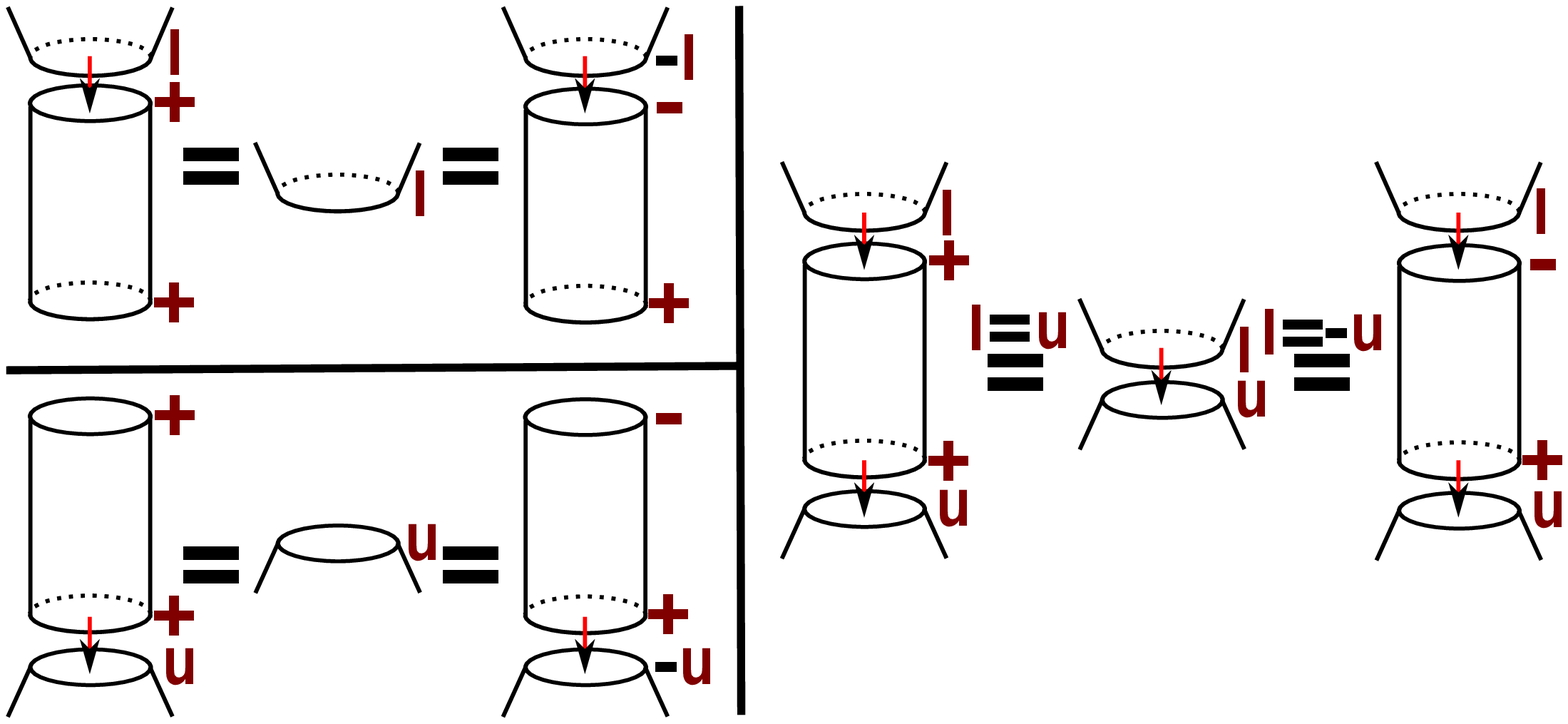}};\endxy
\end{align}
and the third for the non-orientable cobordisms is
\begin{align}\label{eq-combrel3}
\xy(0,0)*{\includegraphics[scale=0.25]{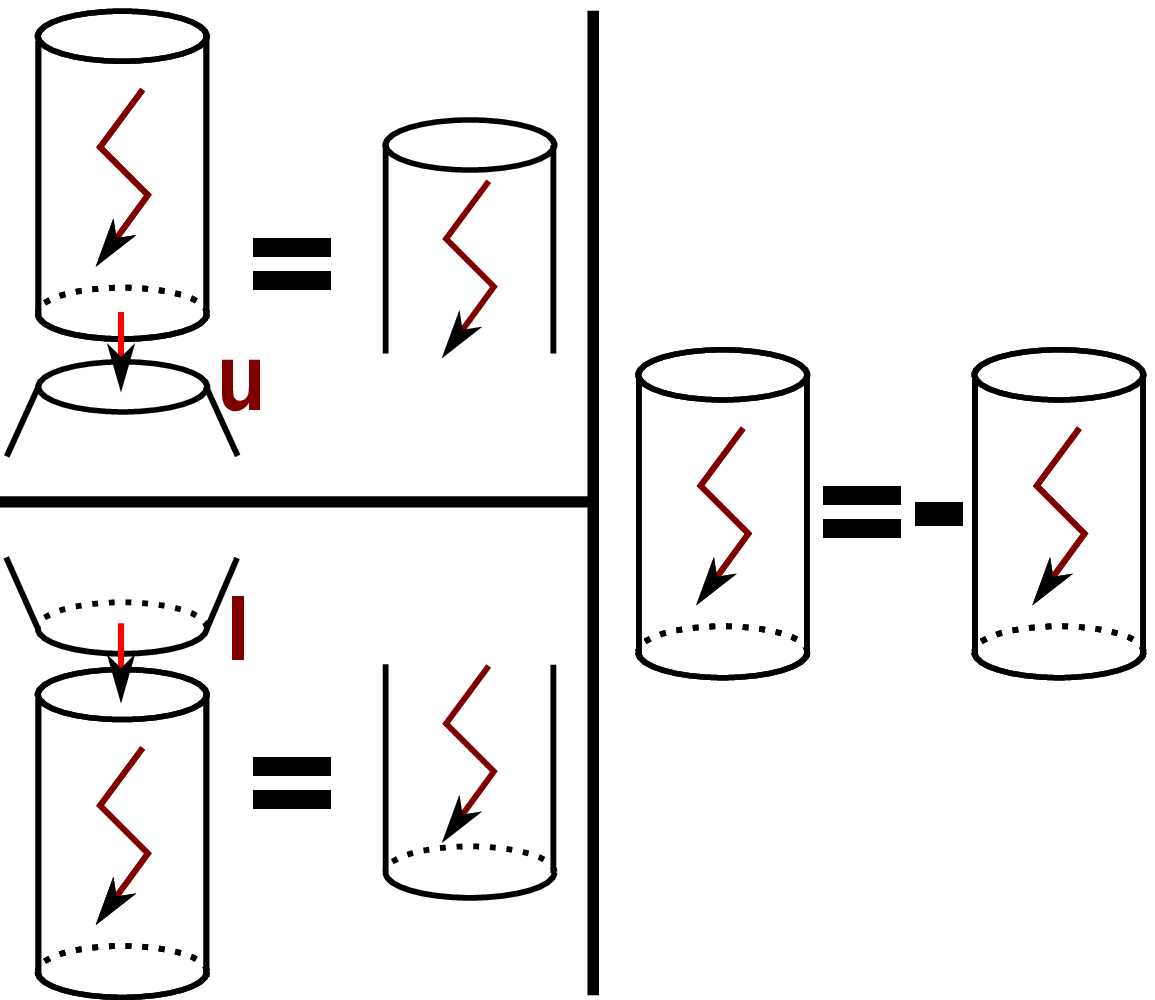}};\endxy
\end{align}
Note that the relation~\ref{eq-combrel3} above is not the same as $\theta=0$, since we work over rings of arbitrary characteristic. The (co)commutativity and (co)associativity relations are
\begin{align}\label{eq-commasss}
\xy(0,0)*{\includegraphics[scale=0.25]{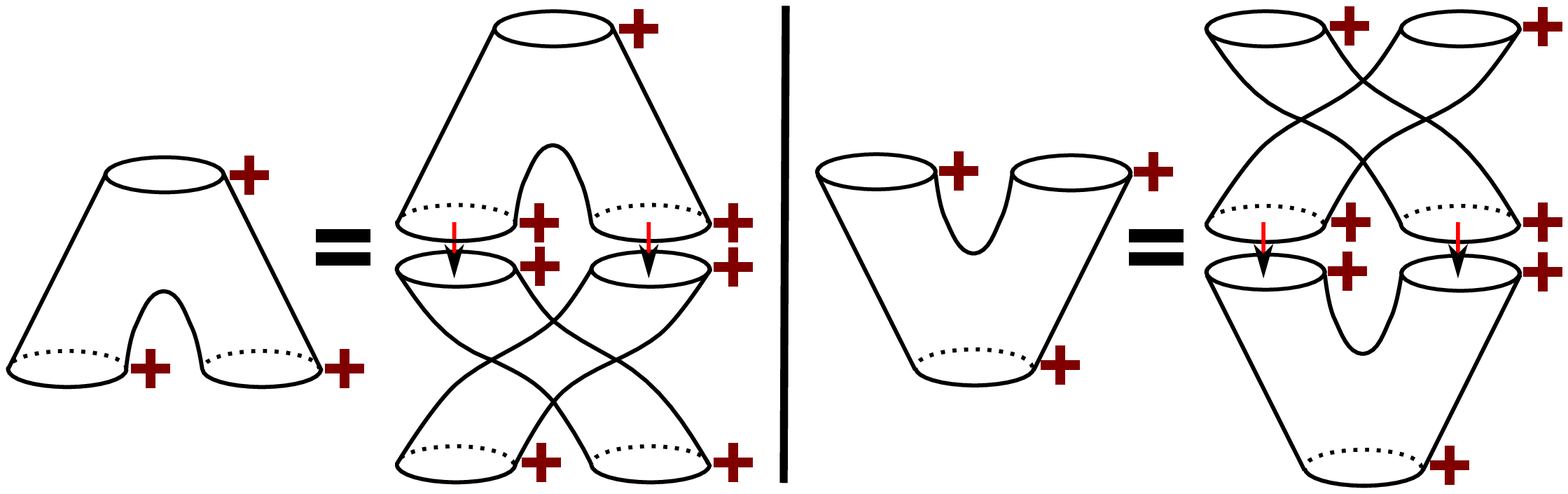}};\endxy\;\;\;\;\;\;\;\;\xy(0,0)*{\includegraphics[scale=0.25]{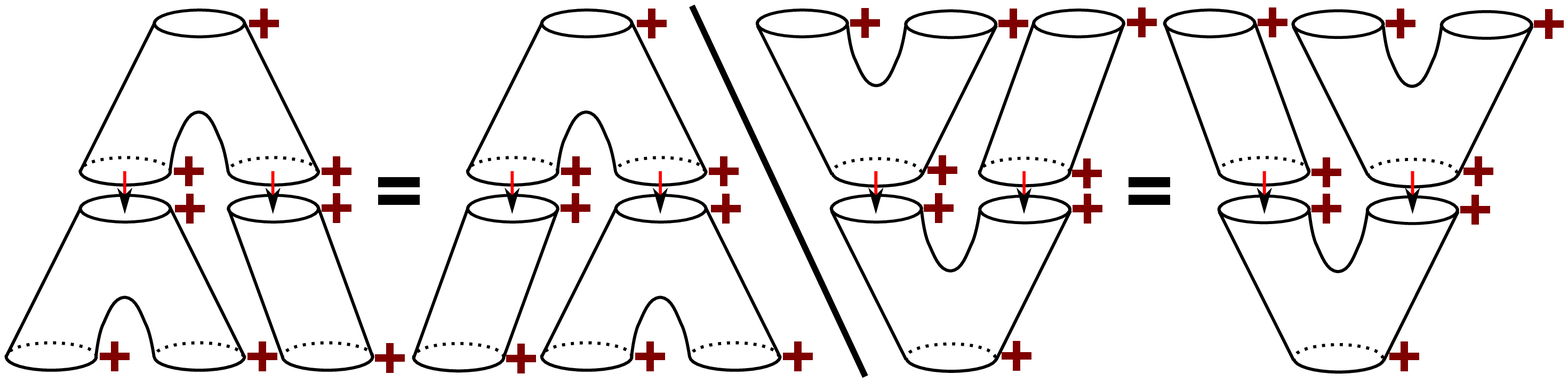}};\endxy
\end{align}
and the (co)unit relations are
\begin{align}\label{eq-unitrel}
\xy(0,0)*{\includegraphics[scale=0.25]{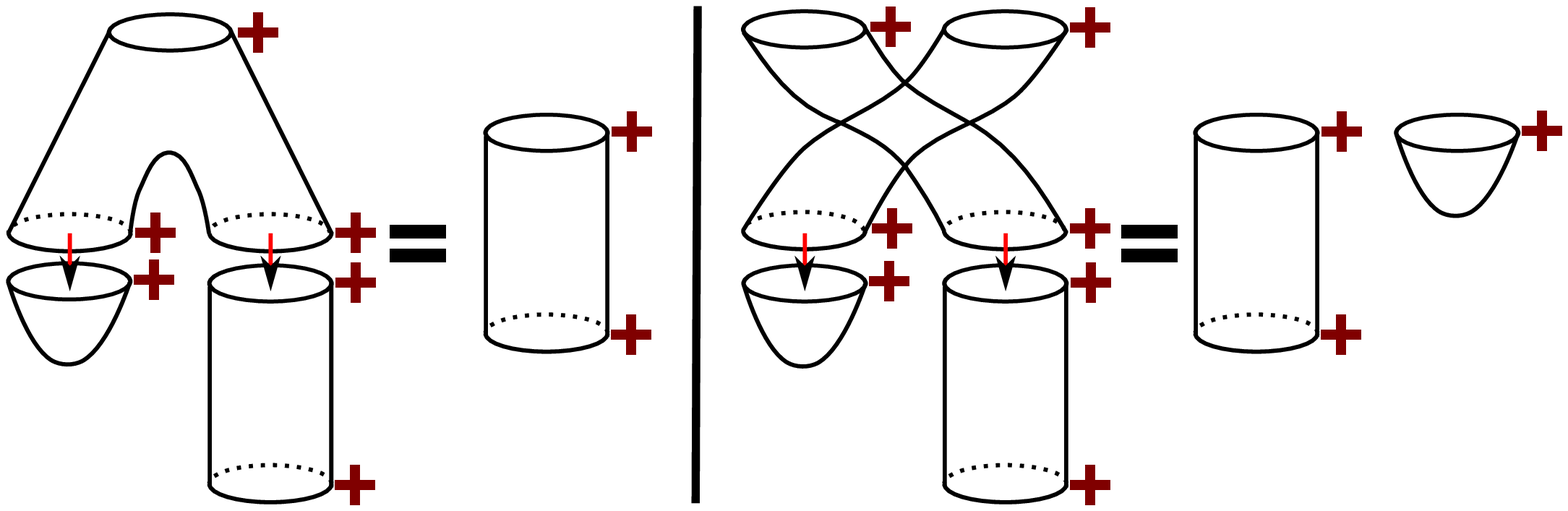}};\endxy\;\;\;\;\;\;\;\;\xy(0,0)*{\includegraphics[scale=0.25]{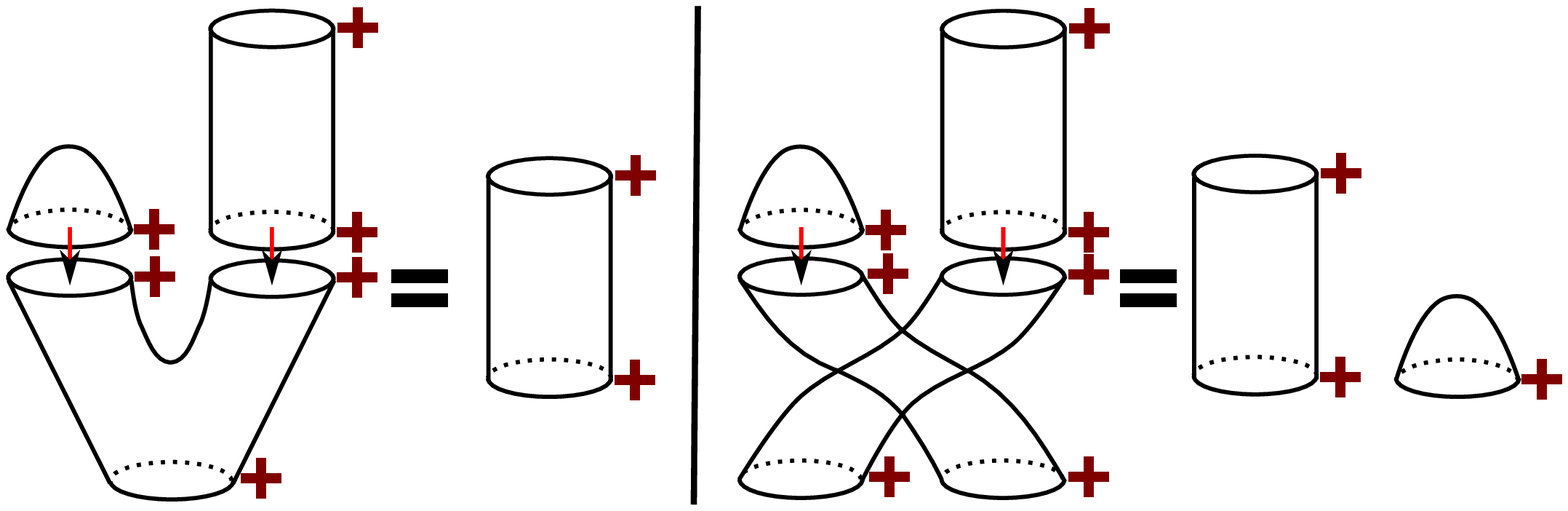}};\endxy
\end{align}
The first and second permutation relations are
\begin{align}\label{eq-permrel1}
\xy(0,0)*{\includegraphics[scale=0.25]{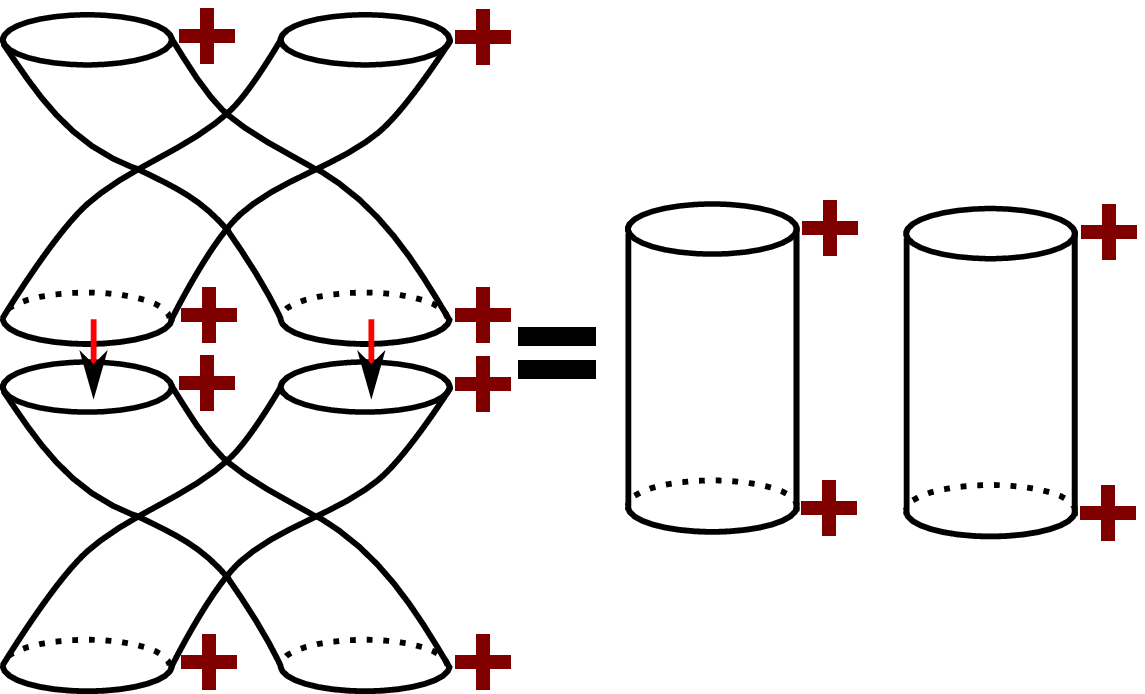}};\endxy\;\;\;\;\;\;\;\;\xy(0,0)*{\includegraphics[scale=0.25]{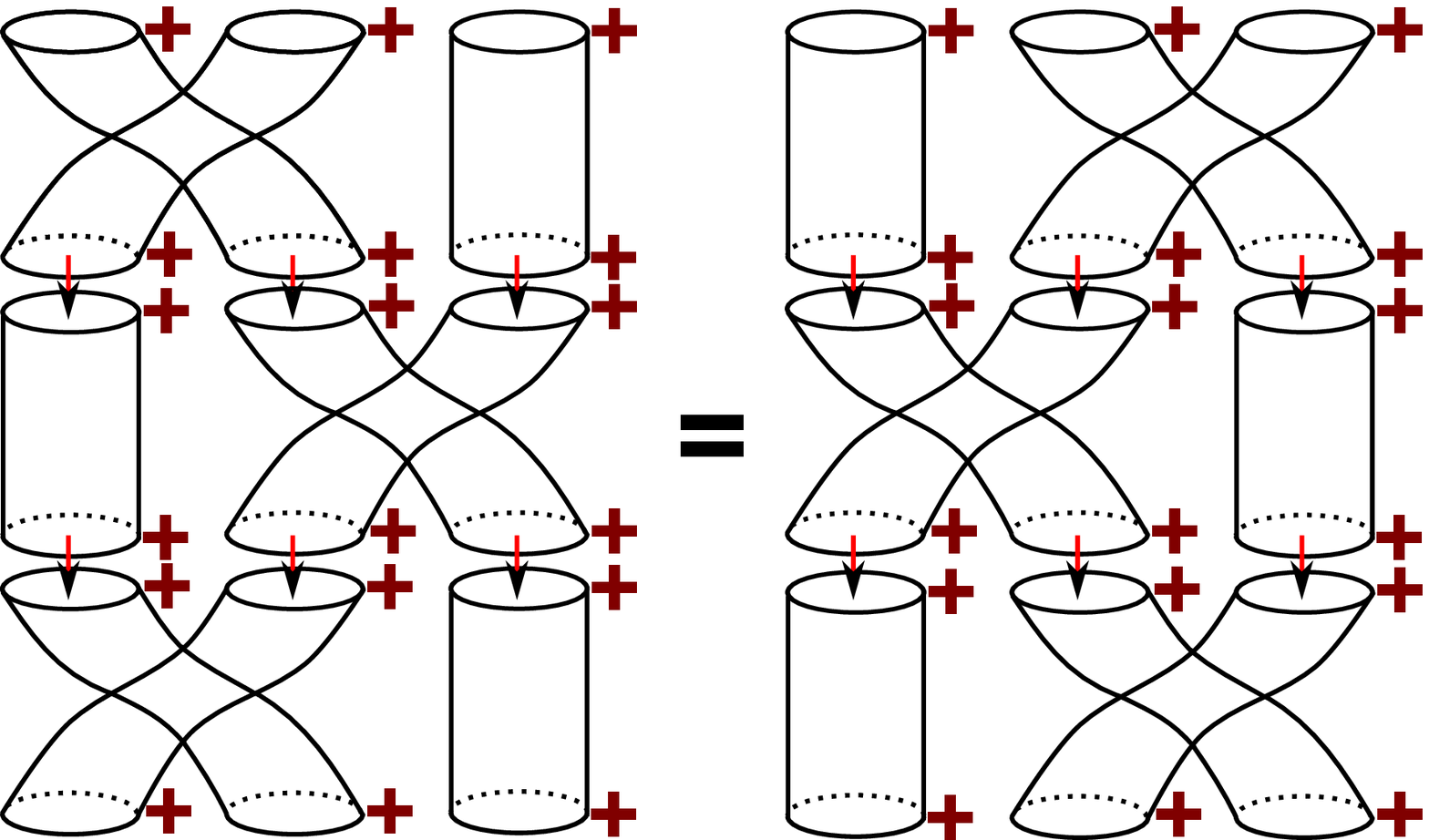}};\endxy
\end{align}
while the third permutation relation is
\begin{align}\label{eq-permrel2}
\xy(0,0)*{\includegraphics[scale=0.3]{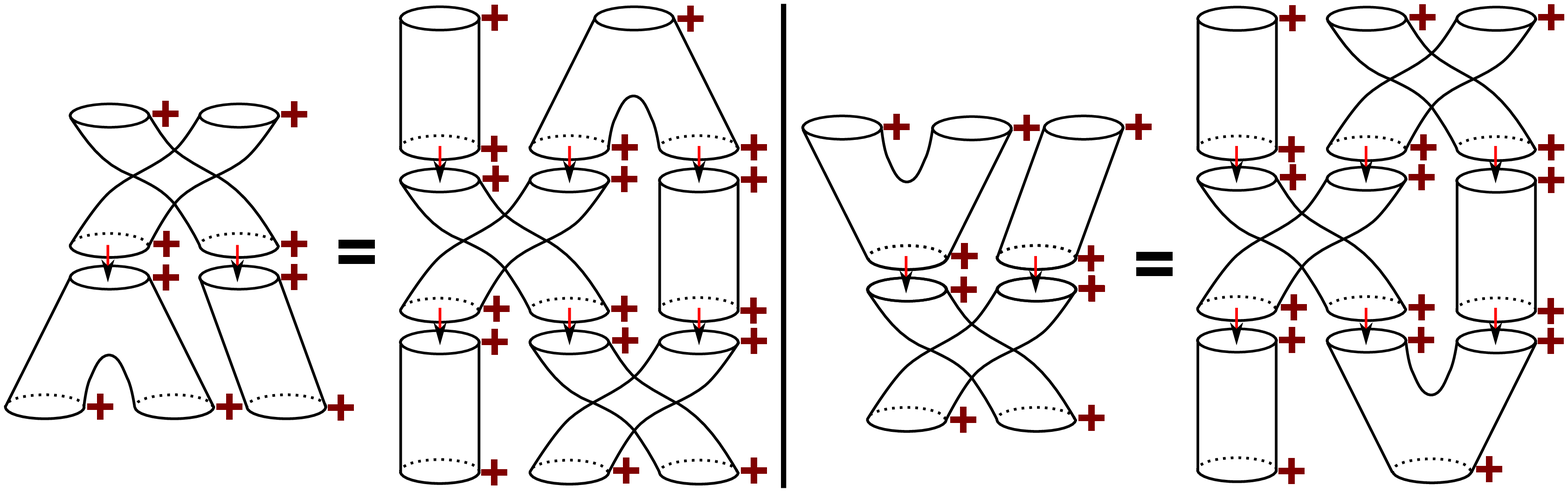}};\endxy
\end{align}
The important \textit{Frobenius, torus and M\"obius} relations are
\begin{align}\label{eq-frobandco}
\xy(0,0)*{\includegraphics[scale=0.25]{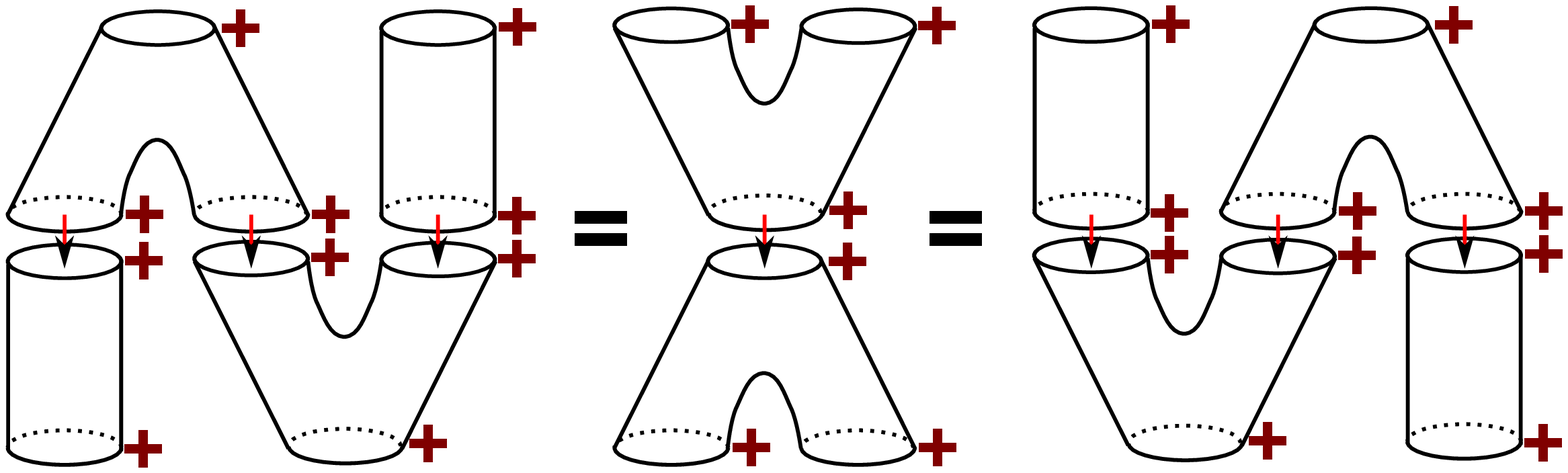}};\endxy\;\;\;\;\;\;\;\;\xy(0,0)*{\includegraphics[scale=0.3]{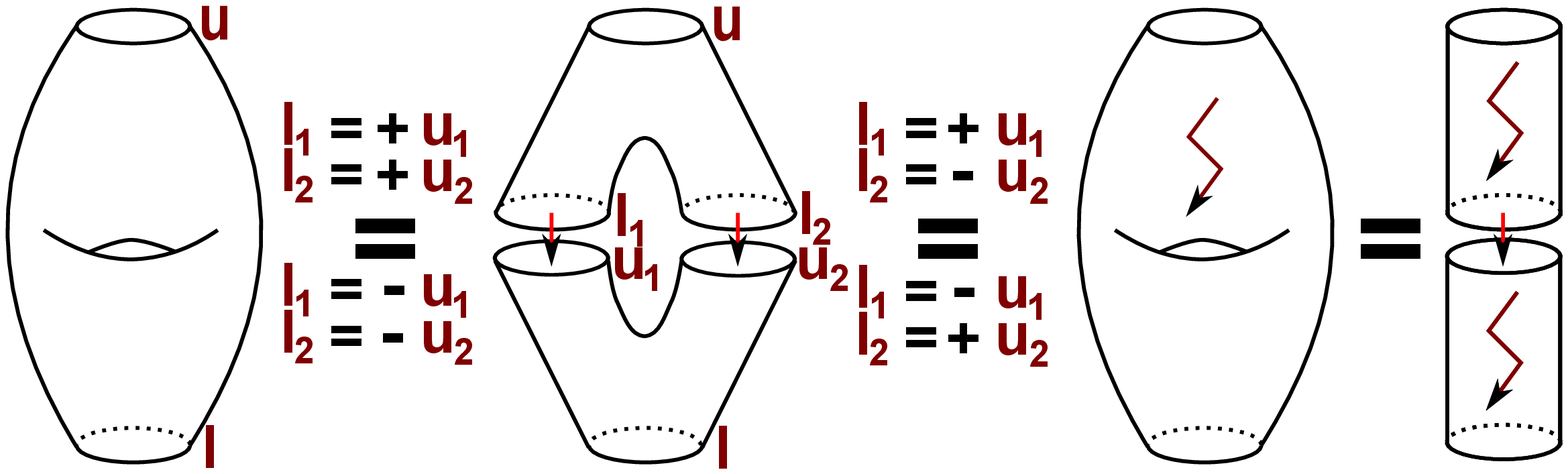}};\endxy
\end{align}
A $u$ or a $l$ means an arbitrary glueing number and $-u,-l$ are the glueing numbers $u$ or $l$ multiplied by $-1$. Furthermore, the bolt represent a non-orientable surfaces and not illustrated parts are arbitrary.

It follows from these relations, that the cobordism $\amalg_{i\in I}\mathrm{id}^+_+$ is the identity morphism between $|I|$ v-circles. The cobordism $\Phi^-_+$ changes the boundary decoration of a morphism. Hence, the category above contains all possibilities for the decorations of the boundary components.

The category $\ucob_R(\emptyset)^*$ is the same as above, but without all minus signs in the relations (we mean ``honest'' minus signs, i.e. the minus-decorations are still in use).
\vskip0.5cm
Both categories are strict monoidal categories (compare to~\ref{defn-monoidal}), since we are working with isotopy classes of cobordisms. The monoidal structure is be induced by the disjoint union $\amalg$. Moreover, both categories are symmetric. Note that they can be seen as $2$-categories, as explained in Example~\ref{ex-2cat}, but it is more convenient to see them as monoidal $1$-category.
\end{defn}
It is worth noting that the rest of this section can also be done for the category $\ucob_R(\emptyset)^*$ by dropping all the corresponding minus signs.

As in~\cite{bn2}, we define the category $\mat(\mathcal C)$ to be the \textit{category of formal matrices} over a pre-additive category $\mathcal C$, i.e. the objects $\Ob(\mat(\mathcal C))$ are ordered, formal direct sums of the objects $\Ob(\mathcal C)$ and the morphisms $\Mor(\mat(\mathcal C))$ are matrices of morphisms $\Mor(\mathcal C)$. The composition is defined by the standard matrix multiplication. This category is sometimes called the \textit{additive closure} of the pre-additive category $\mathcal C$.

Furthermore, as before in Section~\ref{sec-intro}, we define the category $\kom_b(\mathcal C)$ to be the \textit{category of formal, bounded chain complexes} over a pre-additive category $\mathcal C$. Denote the category modulo formal chain homotopy by $\kom_b(\mathcal C)^h$. More about such categories is collected in Section~\ref{sec-techhomalg}.

Furthermore, we define $\ucob_R(\emptyset)^l$, which has the same objects as the category $\ucob_R(\emptyset)$, but morphisms modulo the local relations from Figure~\ref{figureintroa-4}. We make the following definition.
\begin{defn}\label{defn-category2}
We denote by $\ukob_R$ the category $\kom_b(\mat(\ucob_R(\emptyset)))$. Here our objects are formal, bounded chain complexes of formal direct sums of the category of (possible non-orientable) cobordisms with boundary decorations. We define $\ukob_R^h$ to be the category $\ukob_R$ modulo formal chain homotopy. Furthermore, we define $\ukob_R^l$ and $\ukob_R^{hl}$ in the obvious way. The notations $\ucob_R(\emptyset)^{(l)}$ or $\ukob_R^{(h)(l)}$ mean that we consider all possible cases, namely with or without a $h$ and with or without a $l$.
\end{defn}
\vskip0.25cm
One effective way of calculation in $\ucob_R(\emptyset)$ is the usage of the \textit{Euler characteristic}\footnote{Here we consider our morphisms as surfaces.}. It is well-known that the Euler characteristic is invariant under homotopies and that it satisfies 
\[
\chi(C_2\circ C_1)=\chi(C_1)+\chi(C_2)-\chi(\mathcal O_2)\;\;\text{ and }\;\;\chi(C_1\amalg C_2)=\chi(C_1)+\chi(C_2)
\]
for any two cobordisms $C_1\colon\mathcal O_1\to\mathcal O_2$ and $C_2\colon\mathcal O_2\to\mathcal O_3$. Because the objects of $\ucob_R(\emptyset)$ are disjoint unions of v-circles, we note the following lemmata.
\begin{lem}\label{cor-euler1}
The Euler characteristic satisfies $\chi(C_1\circ C_2)=\chi(C_1)+\chi(C_2)$ for all morphisms $C_1,C_2$ of the category $\ucob_R(\emptyset)$.\qed
\end{lem}
\begin{lem}\label{cor-euler2}
The generators of the category $\ucob_R(\emptyset)$ satisfy $\chi(\mathrm{id}^+_+)=\chi(\mathrm{id}^-_-)=0$ and $\chi(\Phi^-_+)=\chi(\Phi^+_-)=0$ and $\chi(\Delta^+_{++})=\chi(m^{++}_+)=\chi(\theta)=-1$. The composition of a cobordism $C$ with $\mathrm{id}^+_+$ or $\Phi^-_+$ does not change $\chi(C)$. \qed
\end{lem}
It is worth noting that the Lemmata~\ref{cor-euler1} and~\ref{cor-euler2} ensure that the category $\ucob_R(\emptyset)$ can be seen as a \textit{graded category}, that is the grading of morphisms is the Euler characteristic. Recall that a saddle between v-circles is a saddle for a certain neighbourhood and the identity outside of it.
\begin{lem}\label{cor-saddleclass}
All saddles are homeomorphic to the following three cobordisms (and some extra cylinders for not affected components). Hence, after decorating the boundary components, we get nine different possibilities, if we fix the decorations of the cylinders to be $+$.
\begin{itemize}
\item[(a)] A two times punctured projective plane $\theta=\mathbb{RP}^2_2$ iff the saddle has two boundary circles.
\item[(b)] A pantsup-morphism $m$ iff the saddle is a cobordism from two circles to one circle.
\item[(c)] A pantsdown-morphism $\Delta$ iff the saddle is a cobordism from one circle to two circles.
\end{itemize}
\end{lem}
\begin{proof}
We note that an open saddle $S$ has $\chi(S)=-1$. Hence, after closing its boundary components, we get the statement.
\end{proof}
Now we deduce some basic relations between the basic cobordisms. Afterwards, we prove a proposition which is a key point for the understanding of the problematic face from~\ref{probcube}. Note the difference between the relations (b),(c) and (d),(e). Moreover, (k) and (l) are also very important.
\begin{lem}\label{lem-basiscalculations}
The following rules hold.
\begin{itemize}
\item[(a)] $\Phi^-_+\circ\Phi^-_+=\mathrm{id}^+_+\circ\mathrm{id}^+_+=\mathrm{id}^+_+$, $\tau^{++}_{++}\circ\tau^{++}_{++}=\mathrm{id}^{++}_{++}$.
\item[(b)] $(\Phi^-_+\amalg\Phi^-_+)\circ\Delta^+_{++}=\Delta^+_{--}=-\Delta^-_{++}=-\Delta^+_{++}\circ\Phi^-_+$.
\item[(c)] $(\Phi^-_+\amalg\mathrm{id}^+_+)\circ\Delta^+_{++}=\Delta^+_{-+}=-\Delta^-_{+-}=-(\mathrm{id}^+_+\amalg\,\Phi^-_+)\circ\Delta^+_{++}\circ\Phi^-_+$.
\item[(d)] $m^{++}_+\circ(\Phi^-_+\amalg\Phi^-_+)=m^{--}_+=m^{++}_-=\Phi^-_+\circ m^{++}_+$.
\item[(e)] $m^{++}_+\circ(\Phi^-_+\amalg\mathrm{id}^+_+)=m^{-+}_+=m^{+-}_-=\Phi^+_+\circ m^{++}_+\circ(\mathrm{id}^+_+\amalg\,\Phi^-_+)$.
\item[(f)] $m^{++}_+\circ\Delta^+_{++}=(\mathrm{id}^+_+\amalg\Delta^+_{++})\circ(m^{++}_+\amalg\mathrm{id}^+_+)$ (Frobenius relation).
\item[(g)] $m^{++}_{+}\circ(m^{++}_+\amalg\,\mathrm{id}^+_+) =m^{++}_{+}\circ(\mathrm{id}^+_+\amalg\,m^{++}_+)$ (associativity relation).
\item[(h)] $(\Delta^+_{++}\amalg\,\mathrm{id}^+_+)\circ\Delta^+_{++}=(\mathrm{id}^+_+\amalg\,\Delta^+_{++})\circ\Delta^+_{++}$ (associativity relation).
\item[(i)] $m^{++}_+\circ\tau^{++}_{++}\circ(\Phi^-_+\amalg\mathrm{id}^+_+)=m^{+-}_+$ (first permutation $\Phi$ relation).
\item[(j)] $(\Phi^-_+\amalg\mathrm{id}^+_+)\circ\tau^{++}_{++}\circ\Delta^+_{++}=\Delta^+_{+-}$ (second permutation $\Phi$ relation).
\item[(k)] $\theta\circ\Phi^-_+=\Phi^-_+\circ\theta=\theta$, $\theta=-\theta$ ($\theta$ relations).
\item[(l)] $\mathcal K=\theta^2$. Here $\mathcal K$ is a two times punctured Klein bottle.
\end{itemize}
\end{lem}
\begin{proof}
Most of the equations follow directly from the relations in Definition~\ref{defn-category} above. The rest are easy to check and therefore omitted.
\end{proof}
The following example illustrates that some cobordisms are in fact isomorphisms.
\begin{ex}\label{ex-vrmremoval}
The two cylinders $\mathrm{id}^+_+,\Phi^-_+$ are the only isomorphisms between two equal objects. Let us denote $\mathcal O_1$ and $\mathcal O_2$ two objects which differs only though a finite sequence of the virtual Reidemeister moves. The vRM-cobordisms from Figure~\ref{figure0-reide} induces isomorphisms $C\colon\mathcal O_1\to\mathcal O_2$. To see this we mention that the three cobordisms are isomorphisms, i.e. there inverses are the cobordisms which we obtain by turning the pictures upside down (use statement (a) of Lemma~\ref{lem-basiscalculations}).
\end{ex}
\begin{prop}\label{prop-nonorientablefaces}\textbf{(Non-orientable faces)}
Let $\Delta^{u}_{l_1l_2}$ and $m^{u'_1u'_2}_{l'}$ be the surfaces from Figure~\ref{figure1-1}. Then the following is equivalent.
\begin{itemize}
 \item[(a)] $m^{u'_1u'_2}_{l'}\circ\Delta^{u}_{l_1l_2}=\mathcal K$. Here $\mathcal K$ is a two times punctured Klein bottle.
 \item[(b)] $l_1=u'_1$ and $l_2=-u'_2$ or $l_1=-u'_1$ and $l_2=u'_2$.
\end{itemize}
Otherwise $m^{u'_1u'_2}_{l'}\circ\Delta^{u}_{l_1l_2}$ is a two times punctured torus $\mathcal T$. We call this the \textit{M\"obius relation}.
\end{prop}
\begin{proof}
Let us call $C$ the composition $C=m^{u'_1u'_2}_{l'_1}\circ\Delta^{u_1}_{l_1l_2}$. A quick computation shows $\chi(C)=-2$. Because $C$ has two boundary components, $C$ is either a 2-times punctured torus or a 2-times punctured Klein bottle and the statement follows from the torus and M\"obius relations in~\ref{eq-frobandco}.
\end{proof}
\subsubsection{The topological category for v-tangles}
In this part of Section~\ref{sec-vkhcat} we extend the notions above such that they can be used for v-tangles as well. As explained in Section~\ref{sec-vkhsum}, the most important difference is the usage of an extra decoration which we call the \textit{indicator}. The rest is (almost) the same as above. Again all definitions and statements can be done for an analogue of the category $\ucob_R(\emptyset)^*$. First we define/recall the notion of a \textit{virtual tangle (diagram)}, called \textit{v-tangle (diagram)}.
\begin{defn}\label{defn-vtangle}(\textbf{Virtual tangles}) A \textit{virtual tangle diagram with $k\in\bN$ boundary points} $T^k_D$ is a planar graph embedded in a disk $D^2$. This planar graph is a collection of \textit{usual vertices} and $k$-\textit{boundary vertices}. We also allow circles, i.e. closed edges without any vertices.

The usual vertices are all of valency four. Any of these vertices is either an overcrossing $\slashoverback$ or an undercrossing $\backoverslash$ or a virtual crossing $\virtual$. Latter is marked with a circle. The boundary vertices are of valency one and are part of the boundary of $D^2$.

As before, we call the crossings $\slashoverback$ and $\backoverslash$ \textit{classical crossings} or just \textit{crossings} and a virtual tangle diagram without virtual crossings a \textit{classical tangle diagram}.

A \textit{virtual tangle with $k\in\bN$ boundary points} $T^k$ is an equivalence class of virtual tangle diagrams $T^k_D$ module boundary preserving isotopies and \textit{generalised Reidemeister moves}.

We call a virtual tangle $T^k$ \textit{classical} if the set $T^k$ contains a classical tangle diagram. A v-string is a string starting and ending at the boundary without classical crossings. Moreover, we call a v-circle/v-string without virtual crossings a \textit{c-circle/c-string}.

The \textit{closure of a v-tangle diagram with *-marker} $\mathrm{Cl}(T^k_D)$ is a v-link diagram which is constructed by capping of neighbouring boundary points (starting from a fixed point marked with the *-marker and going counterclockwise) without creating new virtual crossings. For an example see Figure~\ref{figure1-2}.

There are exactly two, maybe not equivalent, closures of any v-tangle diagram. In the figure below the two closures are pictured using green edges.
\begin{figure}[ht]
  \centering
     \includegraphics[scale=0.45]{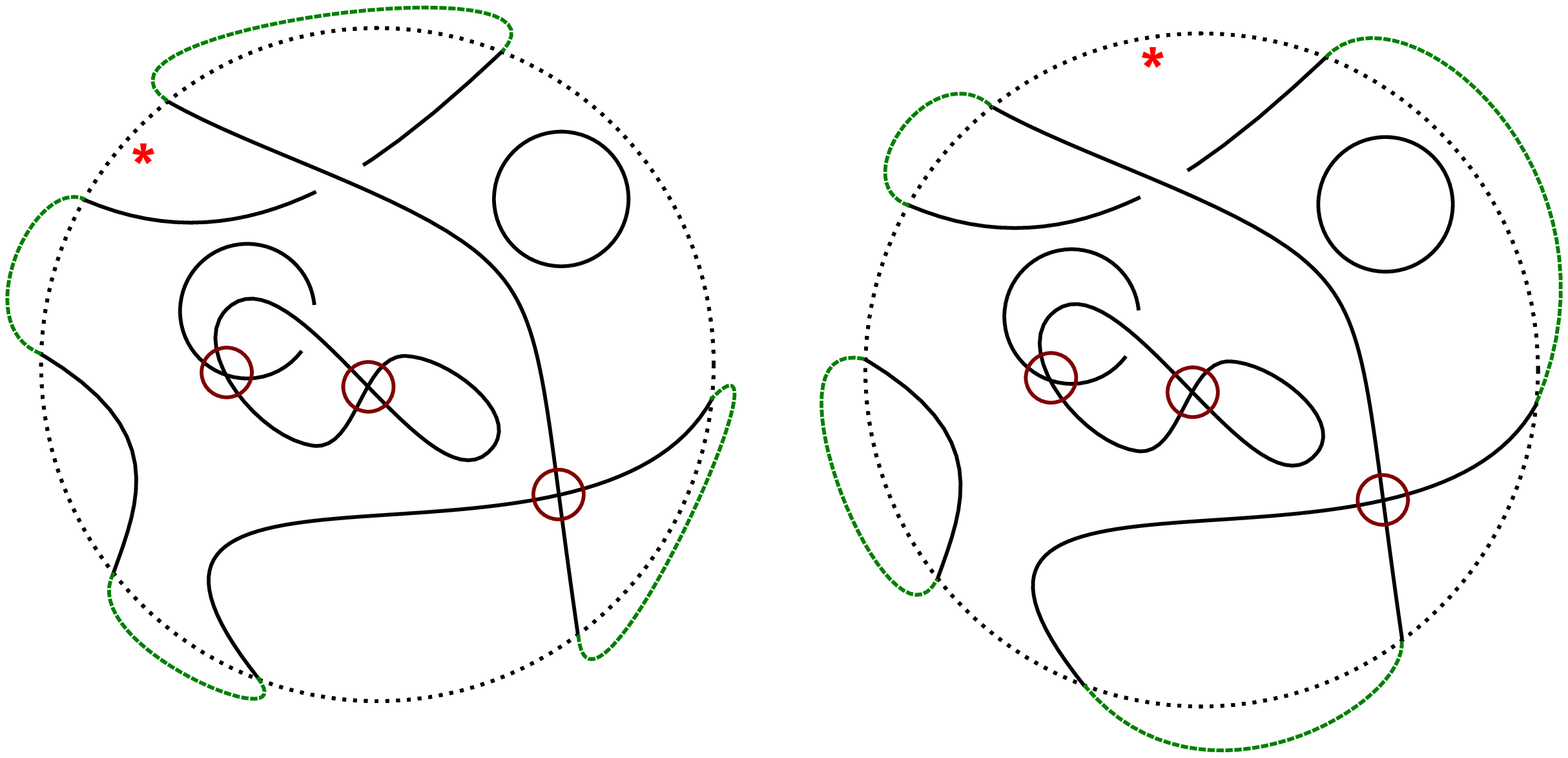}
  \caption{A *-marked v-tangle and two different closures.}
  \label{figure1-2}
\end{figure}

The notions of an \textit{oriented} virtual tangle diagram and of an \textit{oriented} virtual tangle are defined analogue (see also Section~\ref{sec-introa}). The latter modulo \textit{oriented} generalised Reidemeister moves and boundary preserving isotopies. From now on every v-tangle (diagram) is oriented. But we suppress this notion to avoid confusion with other (more important) notations.
\end{defn}
We define the category of \textit{open cobordisms with boundary decorations}. It is almost the same as in Definition~\ref{defn-category}, but the corresponding cobordisms could be open, i.e. they could have vertical boundary components, and are decorated with an extra information, i.e. a number in the set $\{0,+1,-1\}$ (exactly one, even for non-connected cobordisms). We picture the number $0$ as a bolt.
\begin{defn}\label{defn-category3}(\textbf{The category of open cobordisms with boundary decorations}) Let $k\in\bN$ and let $R$ be a commutative and unital ring. the category is $R-$pre-additive. The symbol $\amalg$ denotes the disjoint union.

\textbf{The objects:}

The \textit{objects} \textit{$\Ob(\ucob_R(k))$} are numbered v-tangle diagrams with $k$ boundary points without classical crossings. We denote the objects as $\mathcal O=\coprod_{i\in I}\mathcal O_i$. Here $\mathcal O_i$ are the v-circles or v-strings and $I$ is a finite, ordered index set. The objects of the category are equivalence (modulo \textit{boudary preserving, planar isotopies}) classes of four-valent graphs.

\textbf{The generators:}

The \textit{generators} of $\Mor(\ucob_R(k))$ are the cobordisms in Figure~\ref{figure1-3}. The cobordisms pictured are all between c-circles or c-strings. As before, we do not picture all the other possibilities, but we include them in the list of generators. 
\begin{figure}[ht]
	 \xy
  	 (-40,0)*{\phantom{.}};
     (40,0)*{\includegraphics[scale=0.5]{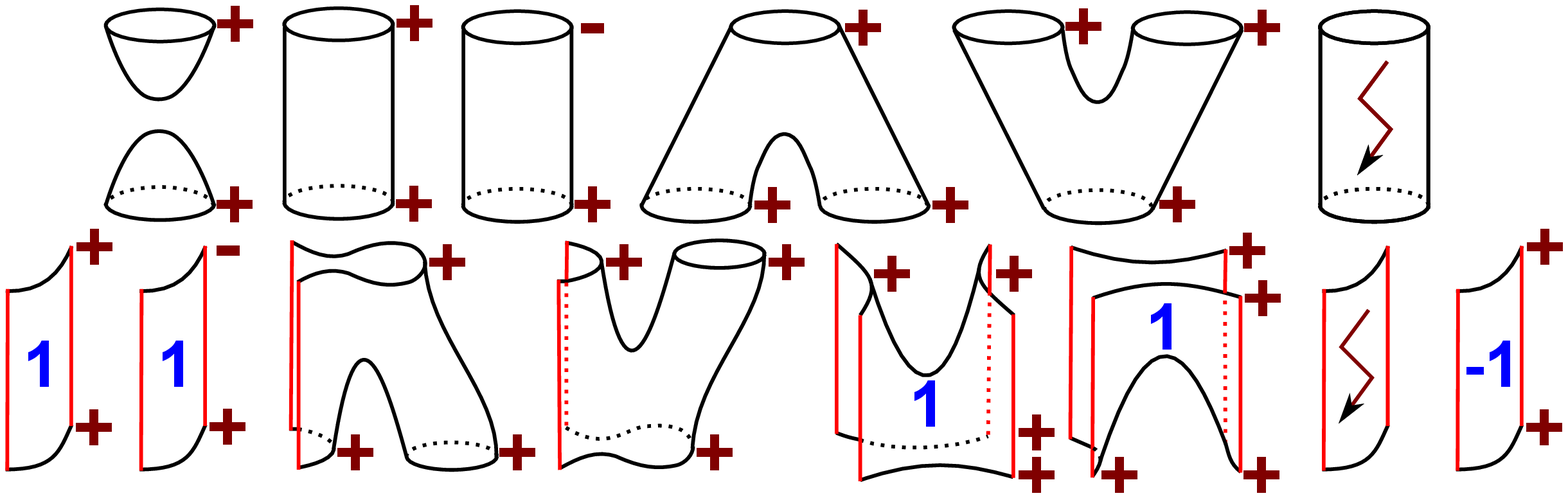}};
     (-15,7.8)*{\varepsilon^+};
     (-15,13.8)*{\iota_+};
     (6,21.5)*{\mathrm{id}^+_+};
     (19.5,21.5)*{\Phi^-_+};
     (67,21.5)*{m^{++}_+};
     (42.5,21.5)*{\Delta_{++}^+};
     (86.8,21.5)*{\theta};
     (-18.5,-20)*{\mathrm{id}(1)^+_+};
     (-7.5,-20)*{\Phi(1)^-_+};
     (9,-20)*{S^{+}_{++}};
     (30,-20)*{S^{++}_{+}};
     (71,-20)*{S^{++}_{++}};
     (52.5,-20)*{S^{++}_{++}};
     (85.5,-20)*{\theta};
     (98.5,-20)*{\mathrm{id}(-1)^+_+};
     \endxy
  \centering
  \caption{The generators for the set of morphisms.}
  \label{figure1-3}
\end{figure}

Every generator has a decoration from the set $\{0,+1,-1\}$. We call this decoration the \textit{indicator} of the cobordism. If no indicator is pictured, then it is $+1$. Indicators behave multiplicative.

Every generator with a decoration $\{+1,-1\}$ has extra decorations from the set $\{+,-\}$ at every horizontal boundary component. We call these decorations the \textit{glueing numbers} of the cobordism. The vertical boundary components are pictured in red.

We consider these cobordisms up to boundary preserving homeomorphisms (as abstract surfaces). Hence, between circles or strings with v-crossings the generators are the same up to boundary preserving homeomorphisms, but immersed into $D^2\times[-1,1]$.

We denote the different generators (from left to right; top row first) by $\iota_+$ and $\varepsilon^+$, $\mathrm{id}^+_+$ and $\Phi^-_+$, $\Delta^+_{++}$, $m^{++}_+$ and $\theta$, $\mathrm{id}(1)^+_+$ and $\Phi(1)^-_+$, $S^{+}_{++}$ and $S^{++}_{+}$, $S(1)^{++}_{++}$, $\theta$ and $\mathrm{id}(-1)^+_+$.

The composition of the generators formally needs again internal decorations to remember how they where glued together. But again we suppress them and hope the reader does not get confused. Moreover, as before, cobordisms with a $0$-indicator do not have any boundary decorations, i.e. they are deleted after glueing.

\textbf{The morphisms:}

The \textit{morphisms} \textit{$\Mor(\ucob_R(k))$} are cobordisms between the objects in the following way. We identify the collection of numbered v-circles/v-strings with circles/strings immersed into $D^2$.

Given two objects $\mathcal O_1,\mathcal O_2$ with $k_1,k_2\in\bN$ numbered v-circles or v-strings, then a morphism $\mathcal C\colon\mathcal O_1\to\mathcal O_2$ is a surface immersed in $D^2\times[-1,1]$ whose non-vertical boundary lies only in $D^2\times\{-1,1\}$ and is the disjoint union of the $k_1$ numbered v-circles or v-strings from $\mathcal O_1$ in $D^2\times\{1\}$ and the disjoint union of the $k_2$ numbered v-circles or v-strings from $\mathcal O_2$ in $D^2\times\{-1\}$. The morphisms are generated (as abstract surfaces) by the generators from above (see Figure~\ref{figure1-3}).

\textbf{The decorations:}

Every morphism has an \textit{indicator} from the set $\{0,+1,-1\}$.

Moreover, every morphism $C\colon\mathcal O_1\to\mathcal O_2$ in $\Mor(\ucob_R(k))$ is a cobordism between the numbered v-circles or v-strings of $\mathcal O_1$ and $\mathcal O_2$. Let us say that the v-circles or v-strings of $\mathcal O_1$ are numbered $i\in\{1,\dots,l_1\}$ and the v-circles or v-strings of $\mathcal O_2$ are numbered for $i\in\{l_1+1,\dots,l_2\}$.

Every cobordism with $+1,-1$ as an indicator has a decoration on the $i$-th boundary circle. This decoration is an element of the set $\{+,-\}$. We call the decoration of the $i$-th boundary component the \textit{$i$-th glueing number} of the cobordism.

Hence, the morphisms of the category are pairs $(C,w)$. Here $C\colon\mathcal O_1\to\mathcal O_2$ is a cobordism from $\mathcal O_1$ to $\mathcal O_2$ immersed into $D^2\times[-1,1]$ and $w$ is a string of length $l_2$ in such a way that the $i$-th letter of $w$ is the $i$-th glueing number of the cobordism and the last letter is the indicator or $w=0$ if the cobordism has $0$ as an indicator.

\textbf{Short hand notation:}

We denote a morphism $C$ with an indicator from $\{+1,-1\}$ which is a connected surfaces by $C^{u}_{l}(\mathrm{in})$. Here $u,l$ are words in the alphabet $\{+,-\}$ in such a way that the $i$-th character of $u$ (of $l$) is the glueing number of the $i$-th circle of the upper (of the lower) boundary. The number $\mathrm{in}$ is the indicator. The construction above ensures that this notation is always possible. Therefore we denote an arbitrary morphism as before by ($C^{u_i}_{l_i}$ are its connected components and $u_i,l_i$ are words in $\{+,-\}$)
\[
C(\pm 1)=(C^{u_1}_{l_1}\amalg\cdots\amalg C^{u_{k}}_{l_{k}})(\pm 1).
\]
For a morphism with $0$ as indicator we do not need any boundary decorations. With a slight abuse of notation, we denote all these cobordisms as the non-orientable cobordisms $\theta$.

\textbf{The relations:}

There are different relations between the cobordisms, namely \textit{topological relations} and \textit{combinatorial relations}. The latter relations are described by the glueing numbers and indicators of the cobordisms and the glueing of the cobordisms. The topological relations are not pictured but it should be clear how they should work. Moreover, we have only pictured the most important new relations below, but there should hold analogously relations as in Definition~\ref{defn-category}. The reader should read these relations is the same vein as before.

The most interesting new relations are the three combinatorial
\begin{align}\label{eq-combrel123}
\xy(0,0)*{\includegraphics[scale=0.25]{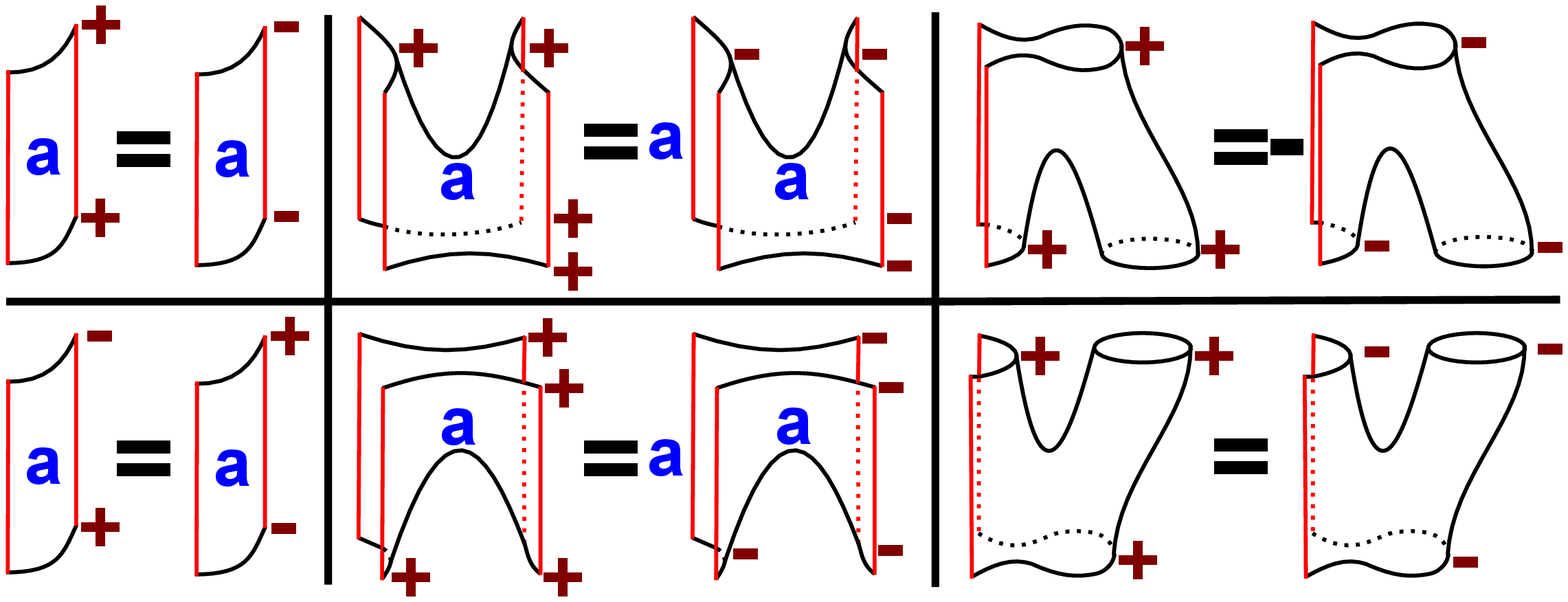}};\endxy\;\;\;\;\;\;\;\xy(0,0)*{\includegraphics[scale=0.25]{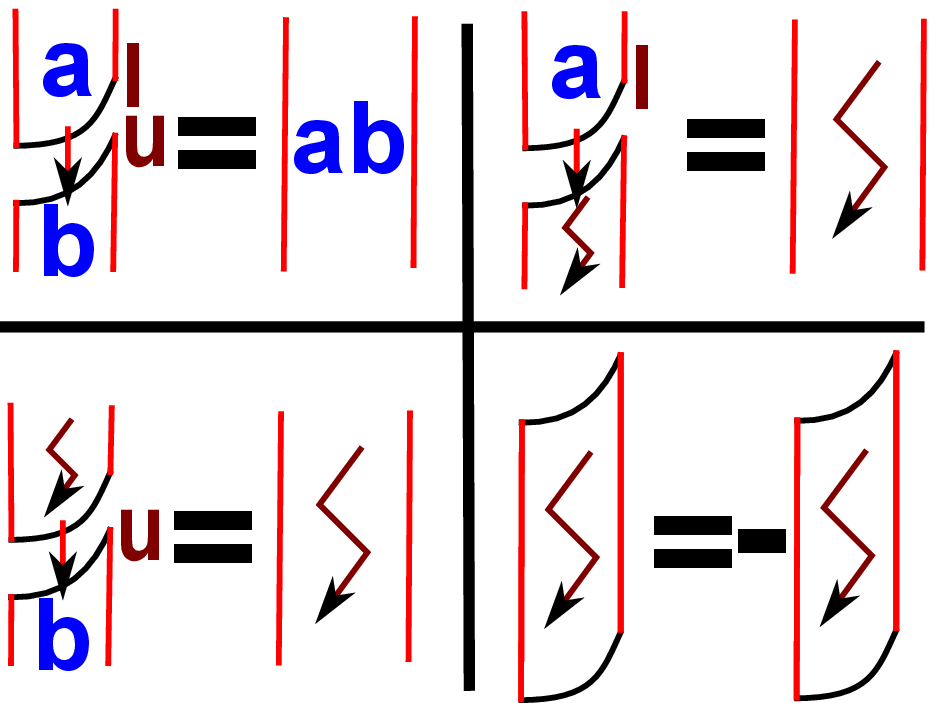}};\endxy\;\;\;\;\;\;\xy(0,0)*{\includegraphics[scale=0.25]{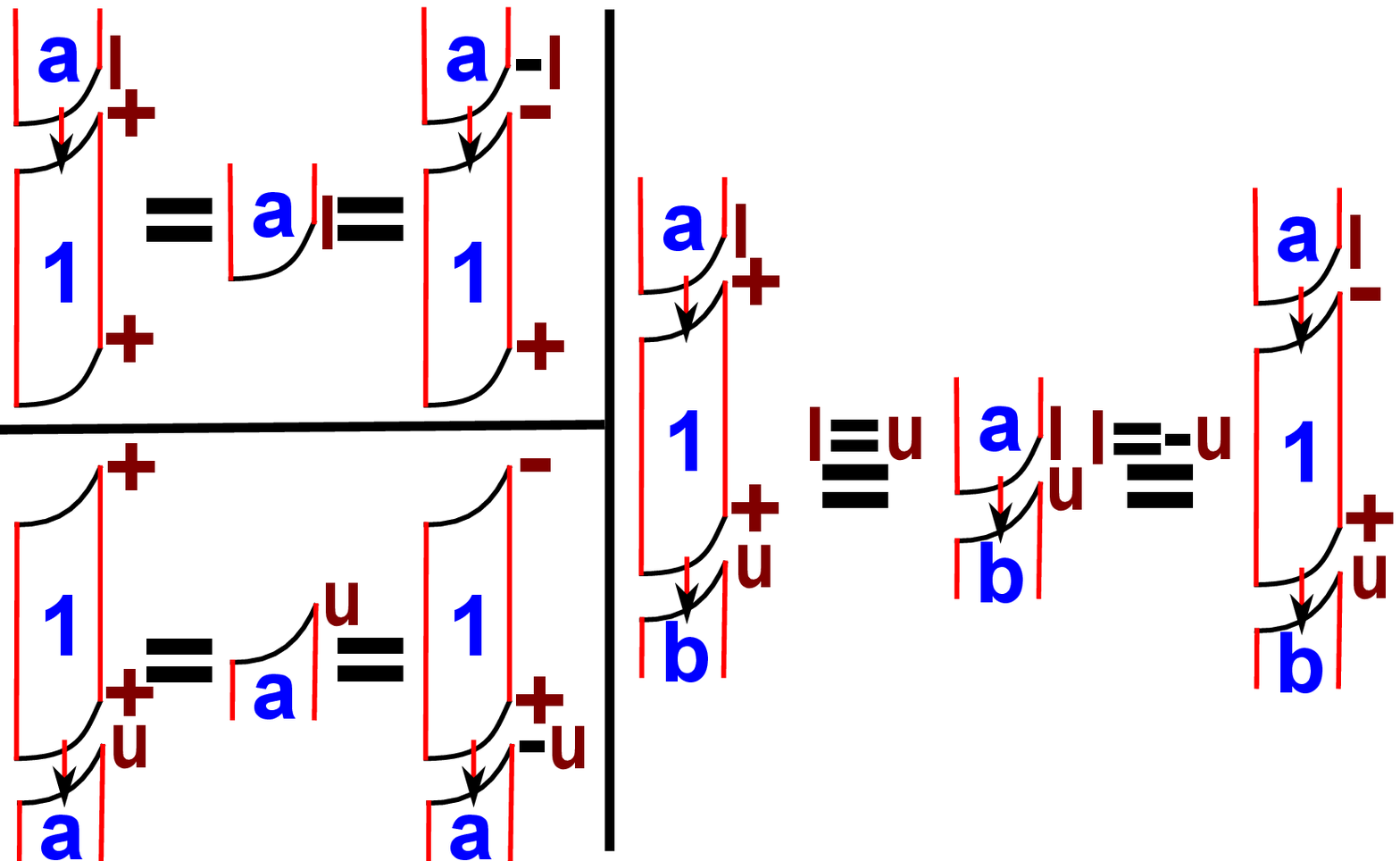}};\endxy
\end{align}
and the \textit{open} M\"obius relations (the glueing in these three cases is given by the glueing numbers, i.e. if there is an odd number of different glueing numbers, then the indicator is $0$ and just the product otherwise).
\begin{align}\label{eq-tanmoe}
\xy(0,0)*{\includegraphics[scale=0.325]{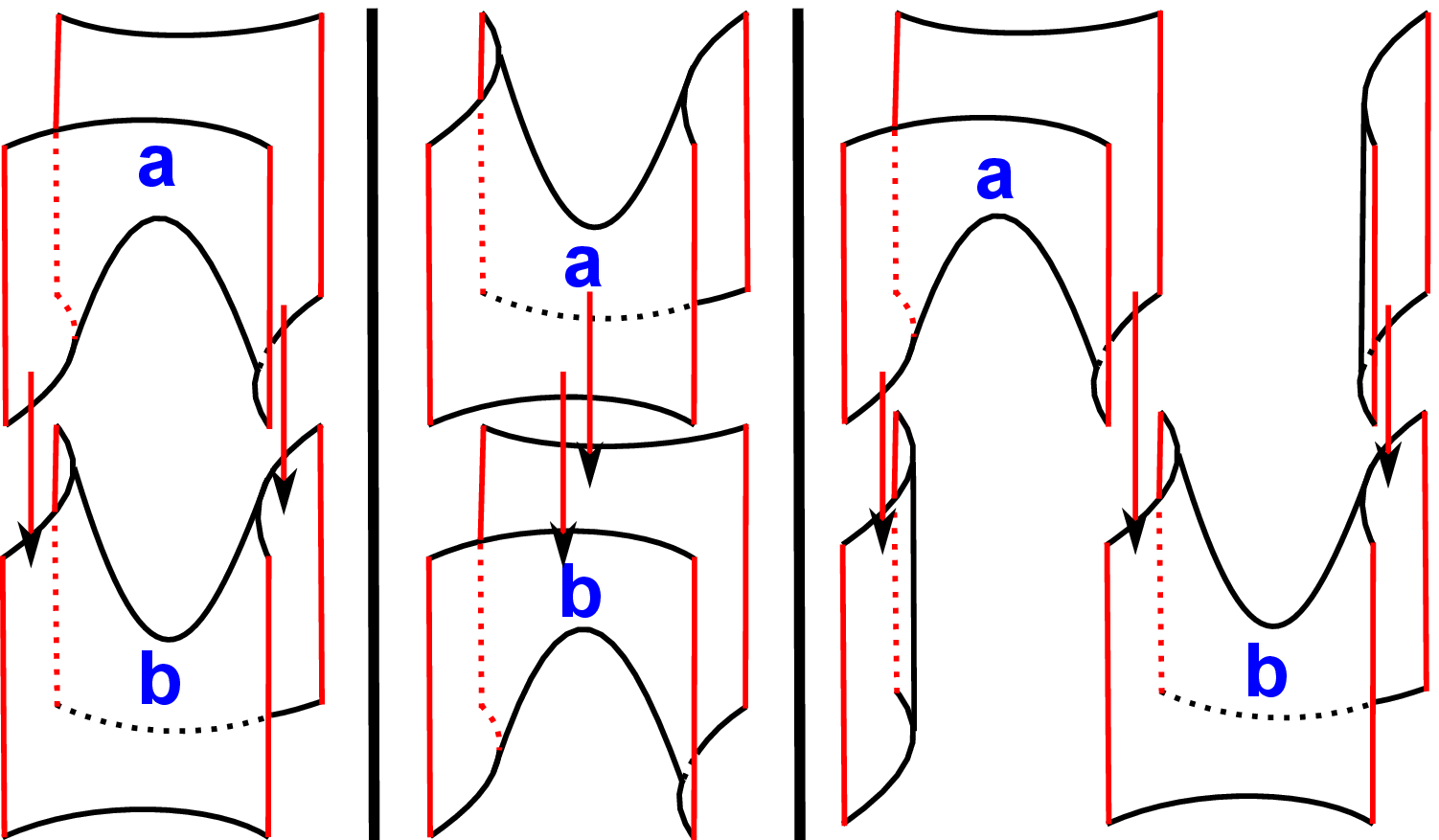}};\endxy
\end{align}
\vskip0.5cm
We define the category $\ucob_R(\omega)$ to be the category whose objects are $\bigcup_{k\in\bN}\Ob(\ucob_R(k))$ and whose morphisms are $\bigcup_{k\in\bN}\Mor(\ucob_R(k))$. Moreover, it should be clear how to convert the Definition~\ref{defn-category2} to the open case. Note that this category is also graded, but the degree function has to be a little bit more complicated (since glueing with boundary behaves different), that is the degree of a cobordism $C\colon\mathcal O_1\to\mathcal O_2$ is given by
\[
\mathrm{deg}(C)=\chi(C)-\frac{b}{2},\;\;\text{ where }b\text{ equals the number of vertical boundary components}.
\]
The reader should check that this definition makes the category graded, that is the degree of a composition is the degree of the sum of its factors.
\end{defn}
\vskip0.5cm
Note the following collection of formulas that follow from the relations. Recall that $\Phi^-_+$ and $\Phi(1)^-_+$ change the decorations and that $\theta$ and $\mathrm{id}(-1)^+_+$ change the indicators. With a slight abuse of notation, we suppress to write $\amalg$ if it is not necessary, i.e. for the indicator changes. Moreover, since $\Phi^-_+$ and $\Phi(1)^-_+$ satisfy similar formulas, we only write down the equations for $\Phi^-_+$ and hope that it is clear how the others look like.
\begin{lem}\label{lem-calc}
Let $\mathcal O,\mathcal O'$ be two objects in $\ucob_R(k)$. Let $C\colon\mathcal O\to\mathcal O'$ be a morphism that is connected, has $\mathrm{in}\in\{0,+1,-1\}$ as an indicator and $u$ and $l$ as decorated boundary strings. Then we have the following identities. We write $ C=C^u_l(\mathrm{in})$ as a short hand notation if the indicators and glueing numbers do not matter. It is worth noting that the signs in (d) are important.
\begin{itemize}
\item[(a)] $C\circ\mathrm{id}(-1)^+_+=\mathrm{id}(-1)^+_+\circ C$ (indicator changes commute).
\item[(b)] $C\circ\theta=\theta\circ C$ ($\theta$ commutes).
\item[(c)] $C(0)\circ\Phi^-_+=\Phi^-_+\circ C(0)$ (first decoration commutation relation).
\item[(d)] Let $u',l'$ denote the decoration change at the corresponding positions of the words $u,l$. Then we have \begin{align*}
C(\pm 1)^u_l\circ(\mathrm{id}^+_+\amalg\dots\amalg\Phi^-_+\amalg\dots\amalg\mathrm{id}^+_+) &=C(\pm 1)^{u'}_l=\pm C(\pm 1)^u_{l'}\\ &=\pm(\Phi^-_+\amalg\dots\amalg\mathrm{id}^+_+\amalg\dots\amalg\Phi^-_+)\circ C(\pm 1)^u_l
\end{align*}(second decoration commutation relation).
\end{itemize}
\end{lem}
\begin{proof}
Everything follows by a straightforward usage of the relations in Definition~\ref{defn-category3}.
\end{proof}
\subsection{The topological complex for virtual links}\label{sec-vkhcom}
We note that the present section splits into three part, i.e. we define the virtual Khovanov complex first and we show that it is an invariant of v-links that agrees with the classical Khovanov complex for c-links. We have collected the more technical points, e.g. it is not clear why Definition~\ref{defn-topcomplex} gives a well-defined chain complex independent of all involved choices, in the last part. It is rather technical and the reader may skip it on the first reading.
\subsubsection*{The definition of the complex}
In the present section we define the topological complex which we call the \textit{virtual Khovanov complex} $\bn{L_D}$ of an oriented v-link diagram $L_D$. This complex is an element of our category $\ukob_R$.

By Lemma~\ref{cor-saddleclass} we know that every saddle cobordism $S$ is homeomorphic to $\theta$, $m$ or $\Delta$ (disjoint union with cylinders for all v-cycles not affected by the saddle). We need extra information for the last two cases. We call these extra information the \textit{sign of the saddle} and the \textit{decoration of the saddle} (see Definitions~\ref{defn-sign} and~\ref{defn-deco}). 
\begin{defn}\label{defn-sign}(\textbf{The sign of a saddle}) We always want to read off signs or decorations for crossings that look like $\slashoverback$, but for a crossing $c$ in a general position there are two ways to rotate $c$ until it looks like $\slashoverback$ (which we call the \textit{standard position}). Since the sign depends on the two possibilities (see bottom row of Figure~\ref{figure-xmarker}), we choose an \textit{x-marker} as in Figure~\ref{figure-xmarker} for every crossing of $L_D$ and rotate the crossing in such a way that the markers match. 
\begin{figure}[ht]
  \centering
     \includegraphics[scale=0.8]{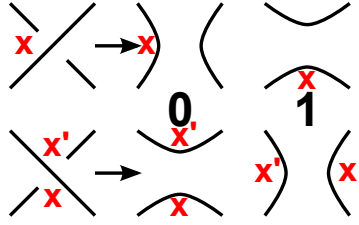}
  \caption{Top: The x-marker for a crossing in the standard position. Bottom: Two possible choices (one denoted by $x^{\prime}$) for a crossing not in the standard position.}
  \label{figure-xmarker}
\end{figure}

We can say now that every orientable saddle $S$ can be viewed in a unique way as a formal symbol $S\colon\smoothing\rightarrow\hsmoothing$. Then the saddle $S$ carries an extra sign determined in the following way.
\begin{itemize}
\item Recall that the v-circles of any resolution are numbered. Moreover, the x-marker for the resolutions in the source and target of $S$ should be at the position indicated in the top row of Figure~\ref{figure-xmarker}.
\item For a saddle $S\colon\gamma_a\to\gamma_{b}$ we denote the numbered v-circles of $\gamma_a,\gamma_{b}$ by $a_1,\dots, a_{k_a}$ and $b_1,\dots,b_{k_b}$ and the v-circles with the x-marker by $a^x_{i},b^x_{j}$.
\item Since the saddle $S$ is orientable, it either splits one v-circle or merges two v-circles. Hence, the two strings in the resolutions $\smoothing$ or $\hsmoothing$ are only different either in the target or in the source of $S$ and we denote the second affected v-circle by $b^y_{j^{\prime}}$ for a split and $a^y_{i^{\prime}}$ for a merge.
\item Then there exists two permutation $\sigma_1,\sigma_2$ for $S$, one for the source and one for the target, such that all $a_k\notin\{a^x_{i},a^y_{i^{\prime}}\}$ and all $b_{k^{\prime}}\notin\{b^x_{j},b^y_{j^{\prime}}\}$ are \textit{ordered ascending after} the (also ordered) $a^x_{i},a^y_{i^{\prime}}$ and $b^x_{j},b^y_{j^{\prime}}$.
\item Then we define the saddle sign $\mathrm{sgn}(S)$ by
\[
\mathrm{sgn}(S)=\mathrm{sgn}(\sigma_1)\cdot\mathrm{sgn}(\sigma_2).
\]
\end{itemize}
For completeness, we define the sign of a non-orientable saddle to be $0$. The \textit{sign $\mathrm{sgn}(F)$ of a face $F$} is then defined by the product of all the saddle signs of the saddles of $F$.
\end{defn}
\begin{ex}\label{ex-sign}
If we have a saddle $S$ between four v-circles numbered $u_1,u_2,u_3,u_4$ and three v-circles $l_1,l_2,l_3$ and the upper x-marker is on the v-circle number $2$ and the lower is on number $3$ and the second string of the upper part is number $1$, then the sign of $S$ is calculated by the product of the signs of the following two permutations.
\[
\sigma_1\colon(u_1,u_2,u_3,u_4)\mapsto (u_2,u_1,u_3,u_4)\;\;\text{and}\;\;\sigma_2\colon(l_1,l_2,l_3)\mapsto (l_3,l_1,l_2).
\]
\end{ex}
\vskip0.5cm
Before we can define the virtual Khovanov complex we need to define the saddle decorations.
\begin{defn}\label{defn-deco}(\textbf{Saddle decorations})
By Lemma~\ref{cor-saddleclass} again, we only have to define the decorations in three different cases. First choose an x-marker as in Definition~\ref{defn-sign} for all crossings and choose orientations for the two resolutions $\gamma_a,\gamma_{a^{\prime}}$. We say the formal saddle of the form
\[
S^{++}_{++}\colon\du\to\ler
\]
is the \textit{standard oriented saddle}. Moreover, every saddle looks locally like the standard oriented saddle, but with possible different orientations. Now we spread the decorations as follows.
\begin{itemize}
\item The non-orientable saddles do not get any extra decorations. It should be noted that locally non-alternating saddles, e.g. $S\colon \dd\to\lel$, are always non-orientable and vice versa.
\item The orientable saddles get a $+$ decoration at strings where the orientations agree and a $-$ where they disagree (after rotating it to the standard position defined above).
\item All cylinders of $S$ are $\mathrm{id}^+_+$ iff the corresponding unchanged v-circles of $\gamma_a$ and $\gamma_{a^{\prime}}$ have the same orientation and a $\Phi^-_+$ otherwise. 
\end{itemize}
To summarise we give the following table (we also give a way to denote the decorations for the saddles). We suppress the cylinders in the Table~\ref{tab-deco}, but we note that the last point of the list above, i.e. the decorations of the cylinders, is important and can not be avoided in our context.

In the Table~\ref{tab-deco} below we write $m,\Delta$ for the corresponding saddles $S$.
\begin{table}[ht]
\begin{center}
\begin{tabular}{|c|c||c|c|}
\hline String & Comultiplication & String & Multiplication \\ 
\hline $\du\to\ler$ & $\Delta^{+}_{++}$ & $\du\to\ler$ & $m_{+}^{++}$ \\ 
\hline $\du\to\rir$ & $\Delta^+_{-+}=\Phi_1\circ\Delta^+_{++}$ & $\uu\to\ler$ & $m_{+}^{-+}=m_{+}^{++}\circ\Phi_1$ \\ 
\hline $\du\to\lel$ & $\Delta^+_{+-}=\Phi_2\circ\Delta^+_{++}$ & $\dd\to\ler$ & $m_{+}^{+-}\circ\Phi_2$ \\ 
\hline $\du\to\ril$ & $\Delta^+_{--}=\Phi_{12}\circ\Delta^+_{++}$ & $\ud\to\ler$ & $m_{+}^{--}\circ\Phi_{12}$ \\ 
\hline $\ud\to\ler$ & $\Delta^-_{++}=\Delta^+_{++}\circ\Phi^-_+$ & $\du\to\ril$ & $\Phi^-_+\circ m_{-}^{++}$ \\ 
\hline $\ud\to\rir$ & $\Delta^-_{-+}=\Phi_1\circ\Delta^+_{++}\circ\Phi^-_+$ & $\uu\to\ril$ & $\Phi^-_+\circ m_{-}^{-+}\circ\Phi_1$ \\ 
\hline $\ud\to\lel$ & $\Delta^-_{+-}=\Phi_2\circ\Delta^+_{++}\circ\Phi^-_+$ & $\dd\to\ril$ & $\Phi^-_+\circ m_{-}^{+-}\circ\Phi_2$ \\ 
\hline $\ud\to\ril$ & $\Delta^-_{--}=\Phi_{12}\circ\Delta^+_{++}\circ\Phi^-_+$ & $\ud\to\ril$ & $\Phi^-_+\circ m_{-}^{--}\circ\Phi_{12}$ \\ 
\hline 
\end{tabular}
\caption{The decorations are spread based on the local orientations.}\label{tab-deco}
\end{center}
\end{table}
\end{defn}
At this point we are finally able to define the \textit{virtual Khovanov complex}. We call this complex \textit{the topological complex}.
\begin{defn}\label{defn-topcomplex}(\textbf{The topological complex}) For a v-link diagram $L_D$ with $n$ ordered crossings we define \textit{the topological complex} $\bn{L_D}$ as follows. We choose an x-marker for every crossing.
\begin{itemize}
\item For $i=0,\dots,n$ the $i-n_-$ \textit{chain module} is the formal direct sum of all $\gamma_a$ of length $i$. We consider the resolutions as elements of $\Ob(\ucob_R(\emptyset))$.
\item There are only morphisms between the chain modules of length $i$ and $i+1$.
\item If two words $a,a^{\prime}$ differ only in exactly one letter and $a_r=0$ and $a'_r=1$, then there is a morphism between $\gamma_a$ and $\gamma_{a'}$. Otherwise all morphisms between components of length $i$ and $i+1$ are zero.
\item This morphism $S$ is a \textit{saddle} between $\gamma_a$ and $\gamma_{a'}$.
\item We consider \textit{numbered} and \textit{oriented resolutions} (we choose them) and the saddles carry the \textit{saddle sings} and \textit{decorations} from the Definitions~\ref{defn-sign} and~\ref{defn-deco}.
\item We consider the saddles $S$ as elements of $\Mor(\ucob_R(\emptyset))$ where we interprete the saddle signs as scalars in $R$ and the saddle decorations as the corresponding boundary decorations.
\end{itemize}
\end{defn}
\begin{rem}\label{rem-topcom}
At this point it is not clear why we can choose the numbering of the crossings, the numbering of the v-circles, the x-markers and the orientation of the resolutions. Furthermore, it is not clear why this complex is a well-defined chain complex.

But we show in Lemma~\ref{lem-commutativeindependence} that the complex is independent of these choices, i.e. if $\bn{L_D}_1$ and $\bn{L_D}_2$ are well-defined chain complexes with different choices, then they are equal up to chain isomorphisms. Moreover, we show in Theorem~\ref{theo-facescommute} and Corollary~\ref{cor-chaincomplex} that the complex is indeed a well-defined chain complex. Hence, we see that
\[
\bn{L_D}\in\Ob(\ukob_R).
\]
For an example see Figure~\ref{figure0-big}. This figure shows the virtual Khovanov complex of a v-diagram of the unknot.
\end{rem}
\subsubsection*{The invariance}
There is a way to represent the topological complex of a v-link diagram $L_D$ as a cone of two v-links diagrams $L^0_D,L^1_D$. Here one fixed crossing of $L_D$ is resolved $0$ in $L^0_D$ and $1$ in $L^1_D$. Note that the cone construction, as explained in Definition~\ref{defn-cone}, works in our setting.

It should be noted that there is a saddle between any two resolutions that are resolved equal at all the other crossings of $L^0_D$ and $L^1_D$. This induces a chain map (as explained in the proof below) between the topological complex of $L^0_D$ and $L^1_D$. We denote this chain map by $\varphi\colon\bn{L^0_D}\to \bn{L^1_D}$.
\begin{lem}\label{lem-cone}
Let $L_D$ be a v-link diagram and let $c$ be a crossing of $L_D$. Let $L^0_D$ be the v-link where the crossing $c$ is resolved 0 and let $L^1_D$ be the v-link where the crossing $c$ is resolved 1. Then we have
\[
\bn{L_D}=\Gamma(\bn{L^0_D}\xrightarrow{\varphi}\bn{L^1_D}).
\]
\end{lem}
\begin{proof}
The proof is analogously to the proof for the classical Khovanov complex. The only new thing to prove is the fact that the map $\varphi$, which resolves the crossing, induces a chain map. This is true because we can take the induced orientation (from the orientations of the resolutions of $L^0_D$ and $L^1_D$) of the strings of $\varphi$. This gives us the glueing numbers for the morphisms of $\varphi$. Here we need the Lemma~\ref{lem-commutativeindependence} to ensure that all faces anticommute.
\end{proof}
\begin{ex}\label{ex-cone}
Let $L_D$ be the v-diagram of the unknot from Figure~\ref{figure0-big}. Then we have
\[
\bn{L_D}=\Gamma(\varphi\colon\bn{\jpg{12mm}{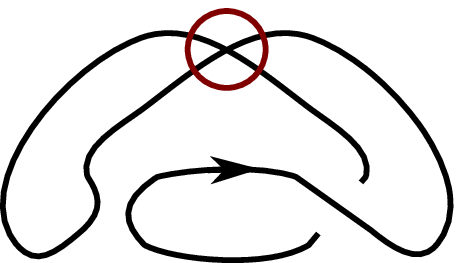}}\to\bn{\jpg{12mm}{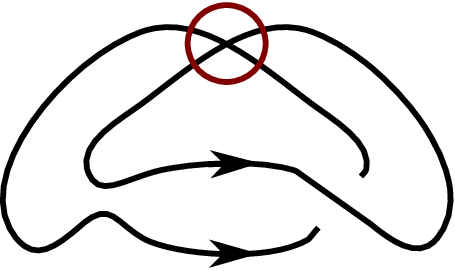}})=\Gamma(\varphi\colon L^0_D\to L^1_D).
\]
If we choose the orientation for the resolutions for the chain complexes $L^0_D,L^1_D$ to be the ones from Figure~\ref{figure0-big}, then the map $\varphi$ is of the form $\varphi=(\theta,m^{--}_+)$.
\end{ex}
As a short hand notation we only picture a certain part of a v-link diagram. The rest of the diagram can be arbitrary. Now we state the main theorem of this section.
\begin{thm}\label{thm-geoinvarianz}(\textbf{The topological complex is an invariant})
Let $L_D,L^{\prime}_D$ be two v-link diagrams which differs only through a finite sequence of isotopies and generalised Reidemeister moves. Then the complexes $\bn{L_D}$ and $\bn{L^{\prime}_D}$ are equal in $\ukob^{hl}_R$.
\end{thm}
\begin{proof}
We have to check invariance under the generalised Reidemeister moves from Figure~\ref{figureintroa-1}. We follow the original proof of Bar-Natan in~\cite{bn2} with some differences. The main differences are the following.
\begin{itemize}
\item[(1)] We have to ensure that our cobordisms have the adequate decorations. For this we number the v-circles in a way that the pictured v-circles have the lowest numbers and we use the orientations given below. It should be noted that Lemma~\ref{lem-commutativeindependence} ensures that we can use this numbering and orientations without problems. We mention that we do not care about the saddle signs to maintain readability because they only affect the anticommutativity of the faces. Hence, after adding some extra signs, the entire arguments work analogously.
\item[(2)] We have to check that the glueing of the cobordisms we give below works out correctly. This is a straightforward calculation using the relations in Lemma~\ref{lem-basiscalculations}.
\item[(3)] The proof of Bar-Natan uses the local properties of his construction. This is not so easy in our case. To avoid it we use some of the technical tools from homological algebra, i.e. Proposition~\ref{prop-sdr}.
\item[(4)] We have to check extra moves, i.e. the virtual Reidemeister moves vRM1, vRM2 and vRM3 and the mixed one mRM.
\end{itemize}
Recall that we have to use the Bar-Natan relations from Figure~\ref{figureintroa-4} here. Note that the Bar-Natan relations do not contain any boundary components. Therefore we do not need extra decorations for them. Because of this we can take the same chain maps as Bar-Natan (the cobordisms are the identity outside of the pictures). Furthermore, the whole construction is in $\ukob_R$.
\vskip0.5cm
The outline of the proof is as follows. For the RM1 and RM2 moves one has to show that the given maps induces chain homotopies, using the rules from Definition~\ref{defn-category} and Lemma~\ref{lem-basiscalculations} and the cone construction from Definition~\ref{defn-cone}. We note that we have to use the Proposition~\ref{prop-sdr} to get the required statement for the RM1 and RM2 moves. Then the RM3 move follows with the cone construction form the RM2 move. The vRM1, vRM2 and vRM3 moves follow from their properties explained in Example~\ref{ex-vrmremoval}. Finally, the invariance under the mRM move can be obtained by an instance of Proposition~\ref{prop-sdr}.
\vskip0.5cm
We consider oriented v-link diagrams. Thus, there are a lot of cases to check. But all cases for the RM1 and RM2 moves are analogously to the cases shown below, i.e. one case for the RM1 move and three cases for the RM2 move. Note that the mirror images work similar.

The case for the RM1 move is pictured below. For the RM2 move we show that the virtual Khovanov complexes of
\[
\bn{\jpg{8mm}{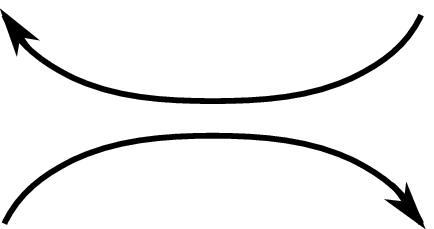}}\text{ and }\bn{\jpg{8mm}{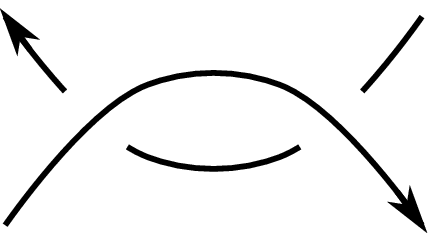}}\;\;\;\;\;\;\bn{\jpg{8mm}{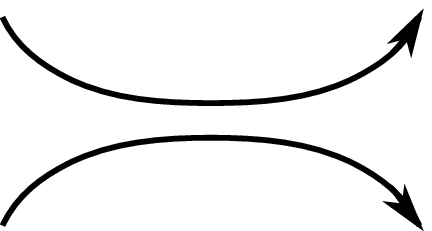}}\text{ and }\bn{\jpg{8mm}{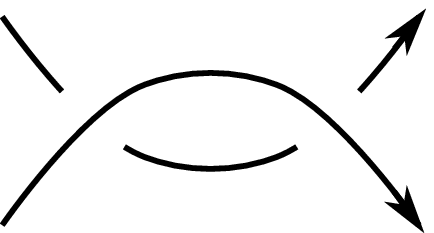}}
\]
are chain homotopic. Here both cases contain two different subcases. For the left case the upper left string can be connected to the upper right or to the lower left. For the other case the upper left string can be connected to the lower right or to the upper right. But the last case is analogously to the first. So we only consider the first three cases.

For the RM1 move we only have to resolve one crossing in the left picture and no crossing in the right. We choose the orientation in such a way that the saddle is a multiplication of the form $\ud\rightarrow\ril$. Thus it is the multiplication $m^{--}_-=m^{++}_+$.

For the RM2 move we have to resolve two crossings in the left picture and no crossing in the right. For the first two cases we choose the orientation in such a way that the corresponding saddles are of the form $\du\rightarrow\ler$ for the left crossing and of the form $\ler\rightarrow\du$ for the right crossing. Hence, we only have $\Delta^+_{++}=-\Delta^-_{--}$ and $m^{--}_-=m^{++}_+$ saddles in the possible complexes.

For the third case we choose the orientation in such a way that the corresponding saddles are of the form $\uu\rightarrow\ril$ or $\uu\rightarrow\rir$ for the left crossing and of the form $\ler\rightarrow\du$ or $\rir\rightarrow\uu$ for the right crossing. Hence, we only have $m^{-+}_-$, $\theta$, $\Delta^-_{--}$ and $\theta$ saddles in the possible complexes.

We give the required chain maps $F,G$ and the homotopy $h$. Note our abuse of notation, that is we denote the chain maps and homotopies and their parts with the same symbols. Moreover, the degree zero components are the leftmost non-trivial in the RM1 case and the middle non-trivial in the RM2 case.

One can prove that these maps are chain maps and that $F\circ G$ and $G\circ F$ are chain homotopic to the identity using the same arguments as Bar-Natan in~\cite{bn2} and the relations from Lemma~\ref{lem-basiscalculations}. We suppress the notation $\Gamma(\cdot)$ in the following. For the RM1 move we have
\[
\begin{xy}
  \xymatrix{
  \bn{\jpg{8mm}{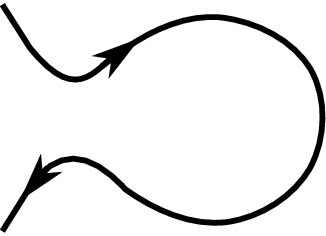}}: &  \bn{\jpg{8mm}{RM1-1}} \ar[rr]^0\ar@<2pt>[dd]^{F=\jpg{11mm}{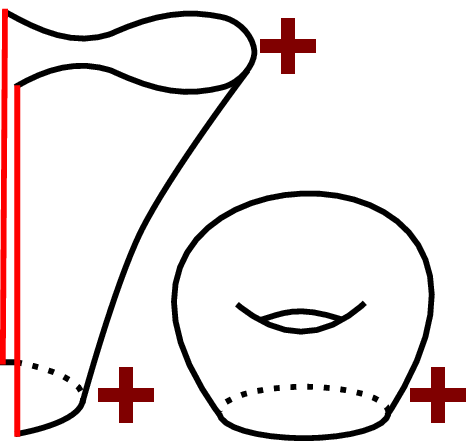}-\jpg{11mm}{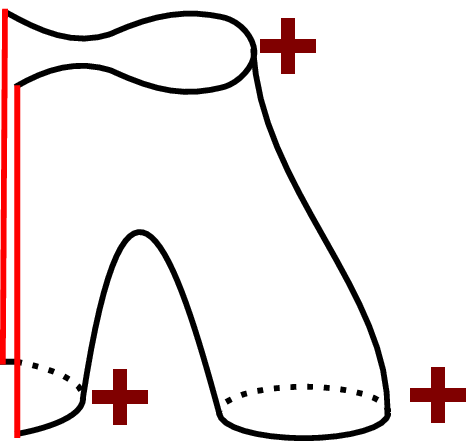}}   &  &   0\ar@<2pt>[dd]^0  \\
  & & & \\
  \bn{\jpg{8mm}{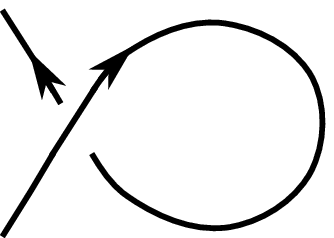}}: &  \bn{\jpg{8mm}{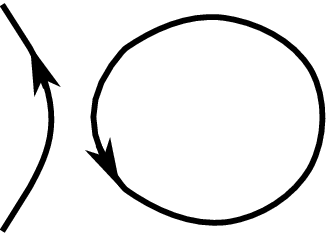}} \ar[rr]_{m^{++}_{+}}\ar@<2pt>[uu]^{G=\jpg{11mm}{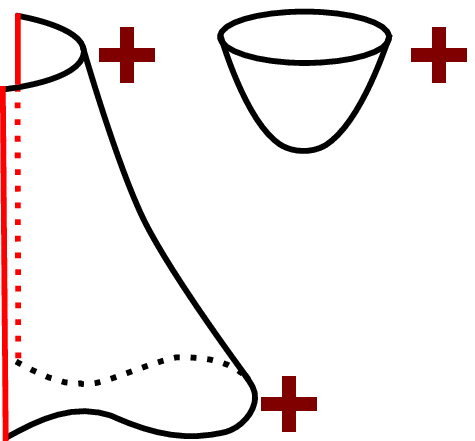}}           &  &   \bn{\jpg{8mm}{RM1-1}}\ar@<2pt>[uu]^0.   
  }
\end{xy}
\]
We also need to give an extra chain homotopy $h$. It is the one from below.
\[
h\colon\bn{\jpg{8mm}{RM1-1}}\to\bn{\jpg{8mm}{RM1-0}},\;\;h=-\jpg{11mm}{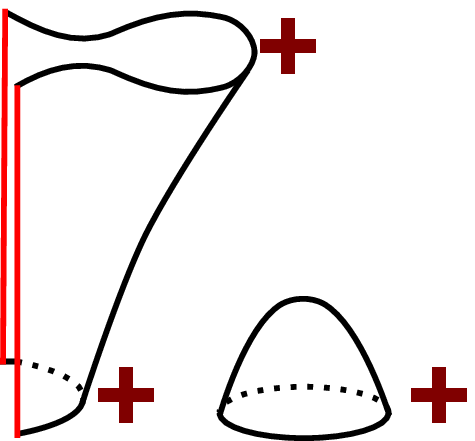}.
\]
An important observation is now that $G\circ F=\mathrm{id}$ and $h\circ F=0$. Beware our abuse of notation here, i.e. the parts of the homotopy $h$ and the chain map $F$ that can be composed are $0$. Thus, we are in the situation of Definition~\ref{defn-sdr} and can use Proposition~\ref{prop-sdr} to get
\[
\Gamma(\bn{\jpg{8mm}{RM1-1}})\simeq_h\Gamma(\bn{\jpg{8mm}{RM1}}).
\]
For the RM2 move the first two cases are
\[
\begin{xy}
  \xymatrix{
  \bn{\jpg{8mm}{RM2-1-a}}: &  0 \ar[rr]^0\ar@<2pt>[dd]^{0}   &  &   \bn{\jpg{8mm}{RM2-1-a}}\ar@<2pt>[dd]^{F=\begin{pmatrix}-\jpg{12mm}{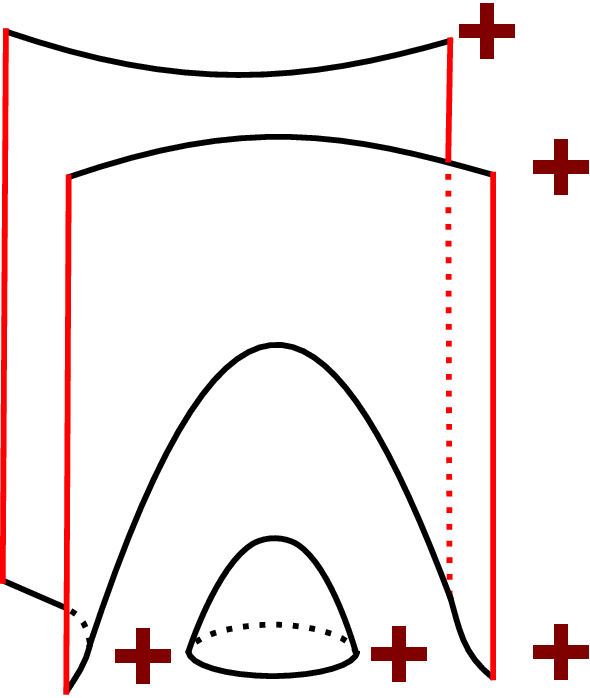} & \Phi^-_+\end{pmatrix}}\ar[rr]^0 & & 0\ar@<2pt>[dd]^{0} \\
  & & & & & \\
  \bn{\jpg{8mm}{RM2-a}}: &  \bn{\jpg{8mm}{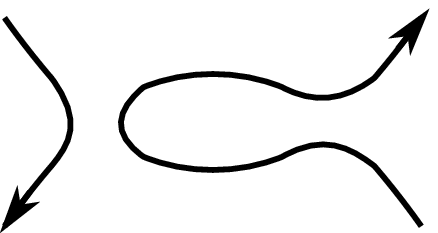}} \ar[rr]_/-1.0em/{d^{-1}}\ar@<2pt>[uu]^{0}           &  & \bn{\jpg{8mm}{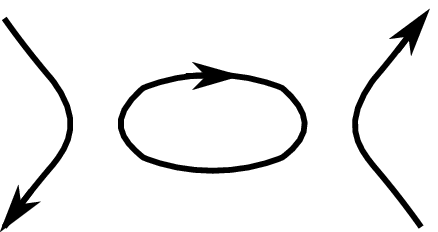}}\oplus\bn{\jpg{8mm}{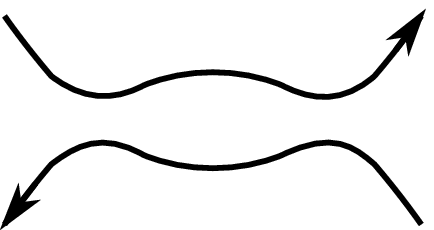}}\ar[rr]_/1.0em/{d^{0}}\ar@<2pt>[uu]^{G=\begin{pmatrix}\jpg{12mm}{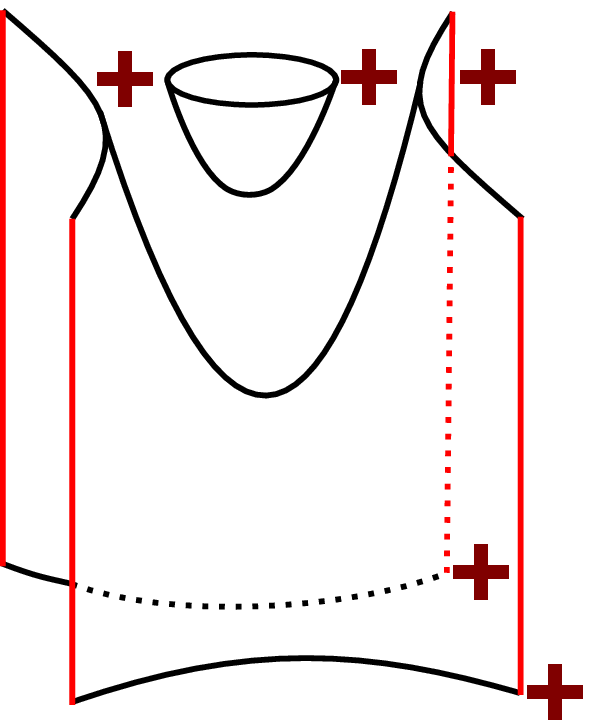} & \Phi^-_+ \end{pmatrix}^T} & & \bn{\jpg{8mm}{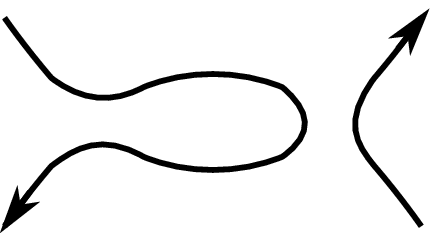}}\ar@<2pt>[uu]^0.   
  }
\end{xy}
\]
Here the differentials are either $d^{-1}=\begin{pmatrix}\Delta^-_{--} & \Delta^+_{++}\end{pmatrix}^T$ and $d^0=\begin{pmatrix}m^{++}_+ & m^{++}_+\end{pmatrix}$ in the second case or $d^{-1}=\begin{pmatrix}\mathrm{id}^+_+\amalg\,\Delta^-_{--} & m^{++}_+\end{pmatrix}^T$ and $d^0=\begin{pmatrix}m^{++}_+\amalg\,\mathrm{id}^+_+ & \Delta^-_{--}\end{pmatrix}$ in the first case. We can follow the proof of Bar-Natan again. Therefore, we need to give a chain homotopy. This chain homotopy is
\[
\begin{matrix}
h^{-1}\colon\bn{\jpg{8mm}{RM2-01-a}}\oplus\bn{\jpg{8mm}{RM2-10-a}}\to\bn{\jpg{8mm}{RM2-00-a}},\;h^{-1}=\begin{pmatrix}-\jpg{11mm}{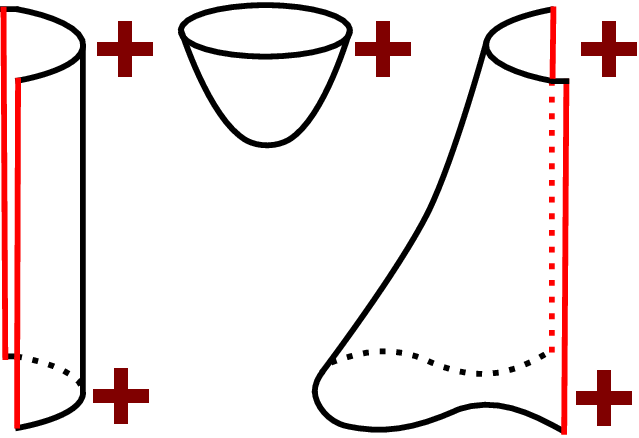} & 0 \end{pmatrix},\\h^0\colon\bn{\jpg{8mm}{RM2-11-a}}\to\bn{\jpg{8mm}{RM2-01-a}}\oplus\bn{\jpg{8mm}{RM2-10-a}},\;h^0=\begin{pmatrix}-\jpg{11mm}{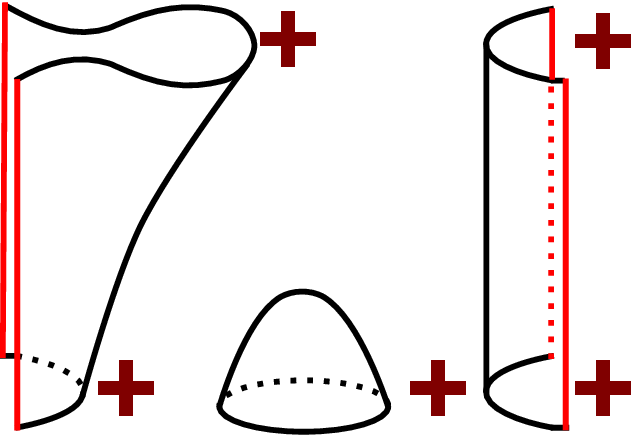} & 0 \end{pmatrix}^T.
\end{matrix}
\]
For the RM2 move the last case is
\[
\begin{xy}
  \xymatrix{
  \bn{\jpg{8mm}{RM2-1-b}}: &  0 \ar[rr]^0\ar@<2pt>[dd]^{0}   &  &   \bn{\jpg{8mm}{RM2-1-b}}\ar@<2pt>[dd]^{F=\begin{pmatrix}-\jpg{12mm}{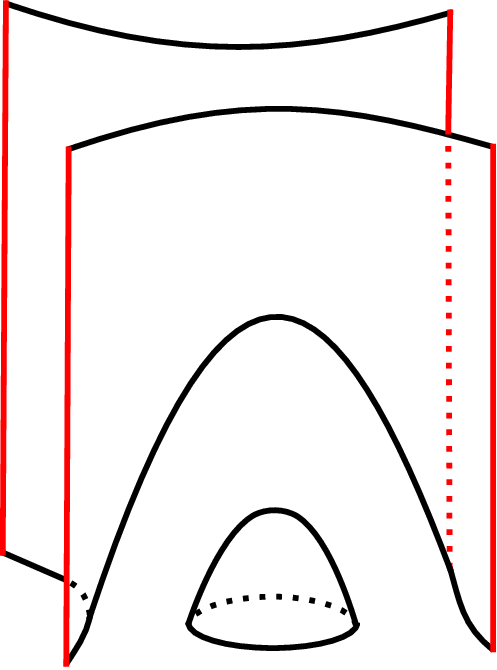} & \mathrm{id}^+_+\end{pmatrix}}\ar[rr]^0 & & 0\ar@<2pt>[dd]^{0} \\
  & & & & & \\
  \bn{\jpg{8mm}{RM2-b}}: &  \bn{\jpg{8mm}{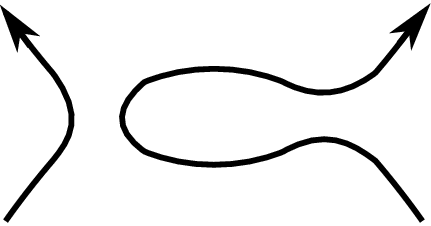}} \ar[rr]_/-1.0em/{d^{-1}}\ar@<2pt>[uu]^{0}           &  & \bn{\jpg{8mm}{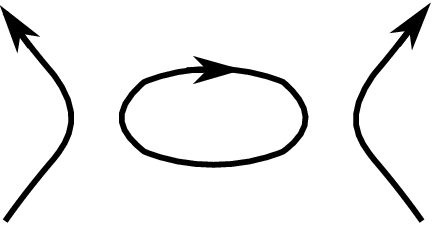}}\oplus\bn{\jpg{8mm}{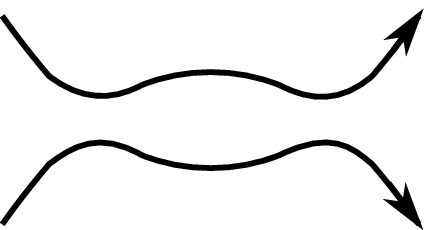}}\ar[rr]_/1.0em/{d^{0}}\ar@<2pt>[uu]^{G=\begin{pmatrix}\jpg{12mm}{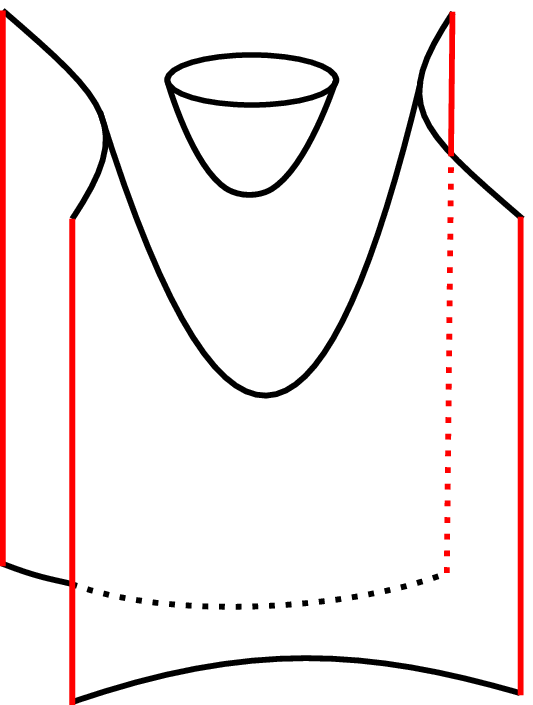} & \mathrm{id}^+_+ \end{pmatrix}^T} & & \bn{\jpg{8mm}{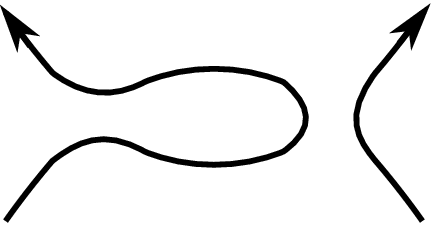}}\ar@<2pt>[uu]^0.   
  }
\end{xy}
\]
Here the differentials are either $d^{-1}=\begin{pmatrix}\Delta^-_{--} & \theta\end{pmatrix}^T$ and $d^0=\begin{pmatrix}m^{-+}_- & -\theta\end{pmatrix}$. Furthermore, saddles of the maps $F,G$ are also $\theta$ saddles. Hence, we do not need any decorations for them. The chain homotopy is defined by
\[
\begin{matrix}
h^{-1}\colon\bn{\jpg{8mm}{RM2-01-b}}\oplus\bn{\jpg{8mm}{RM2-10-b}}\to\bn{\jpg{8mm}{RM2-00-b}},\;h^{-1}=\begin{pmatrix}-\jpg{11mm}{RM2-h1} & 0 \end{pmatrix},\\h^0\colon\bn{\jpg{8mm}{RM2-11-b}}\to\bn{\jpg{8mm}{RM2-01-b}}\oplus\bn{\jpg{8mm}{RM2-10-b}},\;h^0=\begin{pmatrix}-\jpg{11mm}{RM2-h2} & 0 \end{pmatrix}^T.
\end{matrix}
\]
In all the cases it is easy to check that the given maps $F,G$ are chain homotopies. Furthermore, $G$ satisfies the conditions of a strong deformation retract, i.e. $G\circ F=\mathrm{id}$, $F\circ G=h^0\circ d^0+d^{-1}\circ h^{-1}$ and $h\circ F=0$. With the help of Proposition~\ref{prop-sdr} we get
\[
\bn{\jpg{8mm}{RM2-1-a}}\simeq_h\bn{\jpg{8mm}{RM2-a}}\;\;\;\text{and}\;\;\;\bn{\jpg{8mm}{RM2-1-b}}\simeq_h\bn{\jpg{8mm}{RM2-b}}.
\]
Because of this we can follow the proof of Bar-Natan again to show the invariance under the RM3 move. We skip this because this time it is completely analogously to the proof of Bar-Natan (with the maps from above).
\vskip0.5cm
The invariance under the virtual Reidemeister moves vRM1, vRM2 and vRM3 follow from Lemma~\ref{lem-virtualisation}. Therefore, the only move left is the mixed Reidemeister move mRM. We have
\[
\bn{\jpg{8mm}{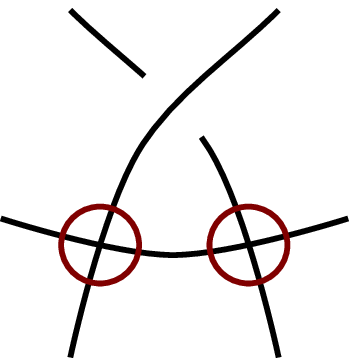}}=\Gamma(\bn{\jpg{8mm}{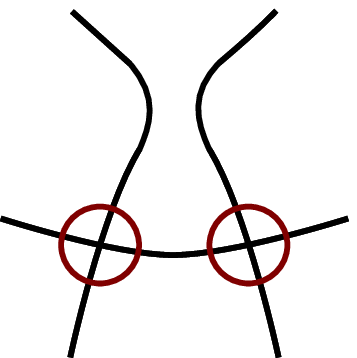}}\xrightarrow{\varphi}\bn{\jpg{8mm}{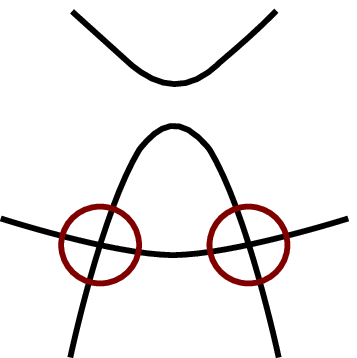}})
\]
and
\[
\bn{\jpg{8mm}{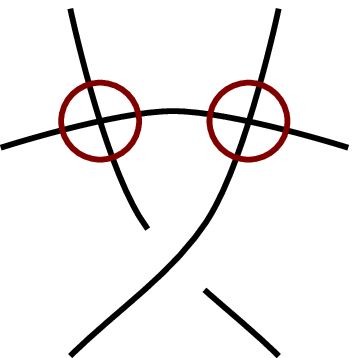}}=\Gamma(\bn{\jpg{8mm}{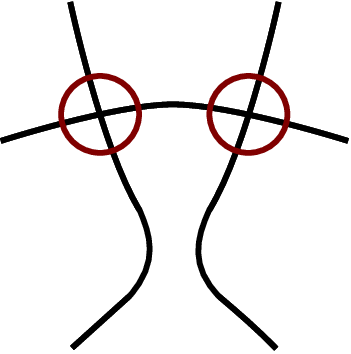}}\xrightarrow{\varphi'}\bn{\jpg{8mm}{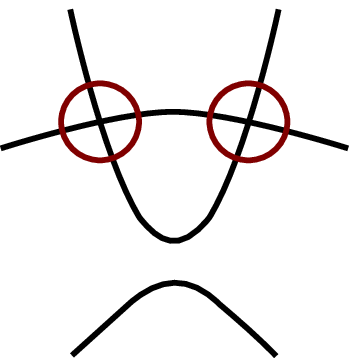}}).
\]
There is a vRM2 move in both rightmost parts of the cones. This move can be resolved. Hence, the complex changes only up to an isomorphism (see Lemma~\ref{lem-virtualisation}). Therefore, we have
\[
\bn{\jpg{8mm}{mRM}}\simeq\Gamma(\bn{\jpg{8mm}{mRM-0}}\xrightarrow{\varphi}\bn{\jpg{8mm}{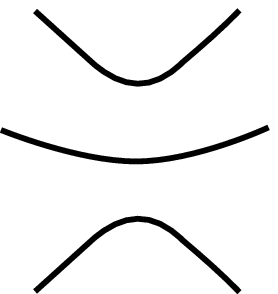}})
\]
and
\[
\bn{\jpg{8mm}{mRM2}}\simeq\Gamma(\bn{\jpg{8mm}{mRM2-0}}\xrightarrow{\varphi'}\bn{\jpg{8mm}{mRM-12}}).
\]
Thus, we see that the left and right parts of the cones are equal complexes. Hence, the complexes of two v-links diagrams which differ only through a mRM move are isomorphic. This finish the proof, because with the obvious chain homotopy $h=0$, isomorphisms induced by the v-Reidemeister cobordisms and Proposition~\ref{prop-sdr} again gives the desired
\[
\bn{\jpg{8mm}{mRM}}\simeq_h\bn{\jpg{8mm}{mRM2}}.
\] 
\end{proof}
A question which arises from Theorem~\ref{thm-geoinvarianz} is if the topological complex yields any new information for c-links (compared to the classical Khovanov complex). The following theorem answers this question negative, i.e. the complex from Definition~\ref{defn-topcomplex} is the classical complex up to chain isomorphisms. It should be noted that Theorem~\ref{thm-geoinvarianz} and Theorem~\ref{thm-classic} imply that our construction can be seen as an extension of Bar-Natans cobordism based complex to v-links. 

To see this we mention that the cobordisms $m^{++}_+,\Delta^+_{++}$ have the same behaviour as the classical (co)multiplications. Therefore, let $\bn{L_D}_c$ denote the classical Khovanov complex, i.e. every pantsup- or pantsdown-cobordisms are of the form $m^{++}_+,\Delta^+_{++}$ and we add the usual extra signs (e.g. see~\cite{bn1} or~\cite{kh1}). Beware that this complex is in general \textit{not} a chain complex for an arbitrary v-link diagram $L_D$. But it is indeed a chain complex for any c-link diagram, i.e. a diagram without v-crossings.
\begin{thm}\label{thm-classic}
Let $L_D$ be a c-link diagram. Then $\bn{L_D}$ and $\bn{L_D}_c$ are chain isomorphic.
\end{thm}
\begin{proof}
Because $L_D$ does not contain any v-crossing, the complex has no $\theta$-saddles. Moreover, every circle is a c-circle. Hence, we can orient them $+$ or $-$, i.e. counterclockwise or clockwise. We choose any numbering for the circles.

Because every circle is oriented clockwise or counterclockwise, every saddle $S$ is of the form $\ud\rightarrow\ril$ or $\du\rightarrow\ler$. Hence, every saddle is of the form $m^{++}_+=m^{--}_-$, $\Delta^+_{++}$ or $\Delta^-_{--}$. Thus, these maps are the classical maps (up to a sign).

We prove the theorem by a spanning tree argument, i.e. choose such a spanning tree. Start at the rightmost leaves and reorient the circles in such a way that the maps which belongs to the edges in the tree are the classical maps $m^{++}_+$ or $\Delta^+_{++}$. This is possible because we can use $m^{++}_+=m^{--}_-$ here. We do this until we reach the end.

We repeat the process rearranging the numbering in such a way that the corresponding maps have the same sign as in the classical Khovanov complex. This is possible because every face has an odd number of minus signs (if we count the sign from the relation $\Delta^-_{--}=-\Delta^+_{++}$).

Note that such rearranging does not affect the anticommutativity because of Lemma~\ref{lem-commutativeindependence}. Hence, after we reach the end every saddle is the classical saddle together with the classical sign. The change of orientations/numberings does not change the complex because of Lemma~\ref{lem-commutativeindependence}. This finishes the proof. 
\end{proof}
\begin{rem}\label{rem-gradings}
We could use the Euler characteristic to introduce the structure of a graded category on $\ucob_R(\emptyset)$ (and hence on $\ukob_R$).

The differentials in the topological complex from Definition~\ref{defn-topcomplex} have all $\deg=0$ (after a grade shift), because their Euler-characteristic is -1 (see Lemma~\ref{cor-euler2}). Then it is easy to prove that the topological complex is a v-link invariant under graded homotopy.
\end{rem}
\begin{rem}\label{rem-tutu}
If one does the same construction as above in the category $\ucob_R(\emptyset)^*$, then the whole construction becomes easier in the following sense. First one does not need to work with the saddle signs any more, i.e. the conplex will be a well-defined chain complex if one uses the same signs as in the classical case. Furthermore, most of the constructions and arguments to ensure that everything is a well-defined chain complex are not necessary or trivial, e.g. most parts of the next subsection are ``obviously'' true, and the rest of this section can be proven completely analogously. This construction leads us to an equivalent of the construction of Turaev and Turner~\cite{tutu}. Note that this version does not generalise the classical Khovanov homology. In order to get a bi-graded complex one seems to need a construction related to $\wedge$-products.
\end{rem}
\subsubsection*{The technical points of the construction}
In this subsection we give the arguments why the topological complex is well-defined and independent of all choices involved.

The following lemma ensures that we can choose the x-marker and the order of the v-circles in the resolutions without changing the total number of signs mod $2$ for all faces.
\begin{lem}\label{lem-xmarker}
Let $F$ be a face of the v-link diagram $L_D$ for a fixed choice of x-markers and orders for the v-circles of the resolutions of $L_D$. Let $F^{\prime},F^{\prime\prime}$ denote the same face, but $F^{\prime}$ with a different choice of x-markers and $F^{\prime\prime}$ with a different choice for the orderings. Then
\[
\mathrm{sgn}(F)=\mathrm{sgn}(F^{\prime})=\mathrm{sgn}(F^{\prime\prime}).
\]
\end{lem}
\begin{proof}
It is sufficient to show the statement if we only change one x-marker of one crossing $c$ or the numbers of only two v-circles with consecutive numbers in a fixed resolution $\gamma_a$. Moreover, the statement is clear if $c$ or $\gamma_a$ does not affect $F$ at all or one of the saddles is non-orientable.

Note that a change of the x-marker of $c$ effects exactly two saddles $S,S^{\prime}$ of $F$ and for both the number $\mathrm{sgn}(S),\mathrm{sgn}(S^{\prime})$ changes since, by definition, we demand that in the definition of the permutations $\sigma_1,\sigma_2$ from Definition~\ref{defn-sign} the two corresponding v-circles are ordered. Hence, the total change for the face is 0 mod 2.

If the numbering of the two v-circles changes in $\gamma_{00}$, then the sign of the permutation $\sigma_1$ changes for both saddles $S_{0*},S_{*0}$ but no changes for $S_{1*},S_{*1}$. Analogously for the $\gamma_{11}$ case.

In contrast, if the numbering of the two v-circles changes in $\gamma_{01}$, then the sign of the permutation $\sigma_1$ and $\sigma_2$ changes for $S_{*1}$ and $S_{0*}$ respectively, but no changes for $S_{*0},S_{1*}$. Analogously for the $\gamma_{10}$ case. Hence, no change for the face mod 2.
\end{proof}
\begin{lem}\label{lem-commutativeindependence}
Let $L_D$ be a v-link diagram and let $\bn{L_D}_1$ be its topological complex from Definition~\ref{defn-topcomplex} with arbitrary orientations for the resolutions. Let $\bn{L_D}_2$ be the complex with the same orientations for the resolutions except for one circle $c$ in one resolution $\gamma_a$. If a face $F_1$ from $\bn{L_D}_1$ is anticommutative, then the corresponding face $F_2$ from $\bn{L_D}_2$ is also anticommutative.

Moreover, if $\bn{L_D}_1$ is a well-defined chain complex, then it is isomorphic to $\bn{L_D}_2$, which is also a well-defined chain complex.

The same statement is true if the difference between the two complexes is the numbering of the crossings, the choice of the x-marker, rotations/isotopies of the v-link diagram or the fixed numbering of the v-circles in the resolutions.
\end{lem}
\begin{proof}
Assume that the face $F_1$ is anticommutative. Then the different orientations of the circle $c$ correspond to a composition of all morphisms of the face $F_2$ with this circle as a boundary component with $\Phi^-_+$.

Hence, the face $F_2$ is also anticommutative, because both outgoing (or incoming) morphisms of $F_2$ are composed with an extra $\Phi^-_+$ if the circle is in the first (or last) resolution of the faces. If it is in one of the middle resolutions, then we have to use the relation $\Phi^-_+\circ\Phi^-_+=\mathrm{id}^+_+$ from Lemma~\ref{lem-basiscalculations}. Note that it is important for this argument to work that cylinders between differently oriented v-circles are $\Phi^-_+$, as in Definition~\ref{defn-topcomplex}.

Thus, if the first complex is a well-defined chain complex, then the same is true for the second. The isomorphism is induced (using a spanning tree construction) by the isomorphism $\Phi^-_+$.

The second statement is true because the numbering of the crossings does not affect the cobordisms at all. Hence, the argument can be shown analogously to the classical case (see for example~\cite{kh1}), but it should be noted that our way of spreading signs does not depend on this ordering.

On the third point: That anticommutative faces stay anticommutative, if one changes between the two possible choices in Definition~\ref{defn-sign}, is part of Lemma~\ref{lem-xmarker}. The chain isomorphism is induced (using a spanning tree construction) by a sign permutation.

The penultimate statement follows directly from the definition of the saddle sign and decorations, while the latter statement is also part of Lemma~\ref{lem-xmarker} with the isomorphisms again induced (using a spanning tree construction) by a sign permutation.
\end{proof}
\begin{lem}\label{lem-virtualisation}
Let $L_D,L^{\prime}_D$ be v-link diagrams which differs only by a virtualisation of one crossing $c$. If a face $F$ is anticommutative in $\bn{L_D}$, then the corresponding face $F^{\prime}$ is anticommutative in $\bn{L^{\prime}_D}$. 

Moreover, if $\bn{L_D}$ is a well-defined chain complex, then it is isomorphic to $\bn{L^{\prime}_D}$, which is also a well-defined chain complex.

The same statement is true if $L_D$ and $L^{\prime}_D$ differs only by a vRM1, vRM2, vRM3 or mRM move.
\end{lem}
\begin{proof}
The statement about anticommutativity is clear, if one of the saddles which belongs to the crossing $c$ is non-orientable. This is true because of the relations from Equation~\ref{eq-combrel3} and Proposition~\ref{prop-nonorientablefaces}. Thus, we can assume that both saddles are orientable. Furthermore, it is clear that the two composition of the saddles are boundary preserving homeomorphic after the virtualisation. Hence, the only thing we have to ensure is that the decorations and signs work out correctly.

We use the Lemma~\ref{lem-commutativeindependence} here, i.e. we can choose the orientations and the numberings in such a way that the saddles which do not belong to the crossing $c$ have the same local orientations and numberings. We observe the following. The sign and the local orientations of a saddle can only change if the saddle belongs to the crossing $c$, i.e. the local orientations always changes (see Figure~\ref{figure-virt2}) and the sign changes precisely if the two strings in the bottom picture of Figure~\ref{figure-virt2} are part of two different v-circles.
\begin{figure}[ht]
  \centering
     \includegraphics[scale=0.75]{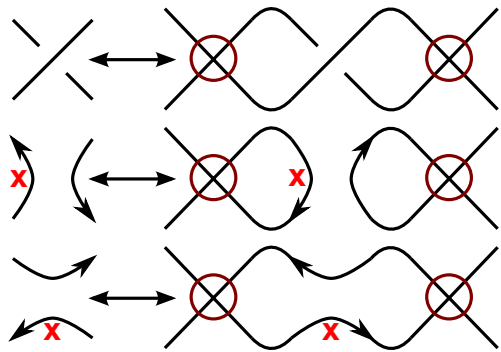}
  \caption{The behaviour of the x-marker and orientations under virtualisation.}
  \label{figure-virt2}
\end{figure}

A change of the local orientations multiplies an extra sign for comultiplication, but no extra sign for multiplication. This follows from the Table~\ref{tab-deco} and the relations from Equation~\ref{eq-combrel12}. Hence, the anticommutativity still holds if the two saddles which belong to the crossing $c$ are both multiplications or comultiplications, because their decorations and signs change in the same way.

If one is a multiplication and one is a comultiplication, then we have two cases, i.e. the multiplication gets an extra sign or not. The comultiplication always gets an extra sign because the local orientations change. But the multiplication will change its saddle sign iff the comultiplication does not change its saddle sign. Hence, the number of extra signs does not change modulo $2$. This ensures that the faces stays anticommutative.

That the face $F^{\prime}$ stays anticommutative after a vRM1, vRM2, vRM3 or mRM move follows because neither the local orientations nor the signs of any cobordism changes. Thus, all decorations and signs are the same.

The chain isomorphisms are induced by the vRM-cobordisms shown in Figure~\ref{figure0-reide}, morphisms of type $\Phi^-_+$ and identities. Recall that all these cobordisms are isomorphisms in our category.
\end{proof}
For the proof of the next lemma we refer the reader to the paper~\cite{ma2}. We call faces of the following type the \textit{basic (non-)orientable faces}.
\begin{figure}[ht]
  \begin{minipage}[c]{6,9cm}
	\centering
	\includegraphics[scale=0.35]{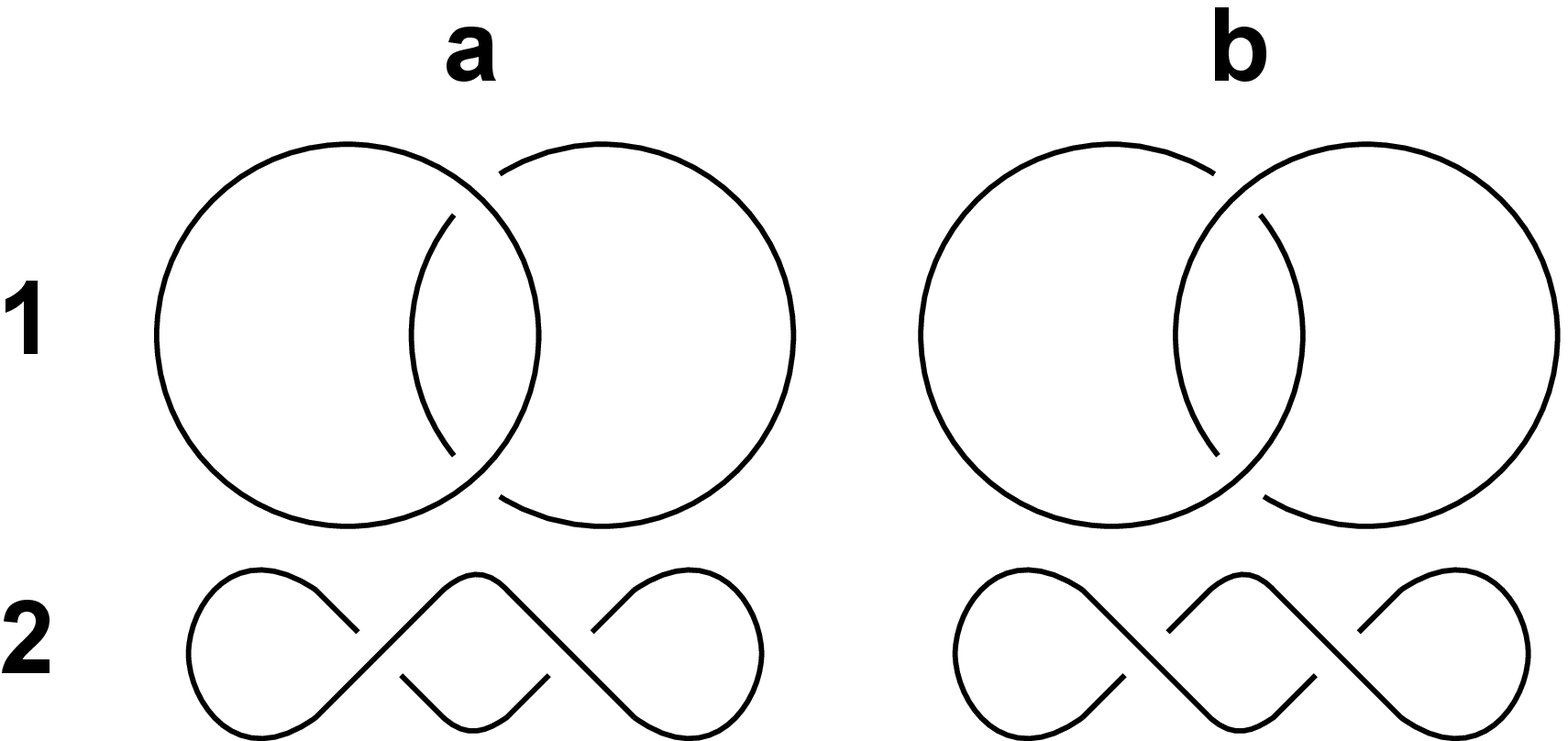}
\end{minipage}
\begin{minipage}[c]{6,9cm}
	\centering
	\includegraphics[scale=0.30]{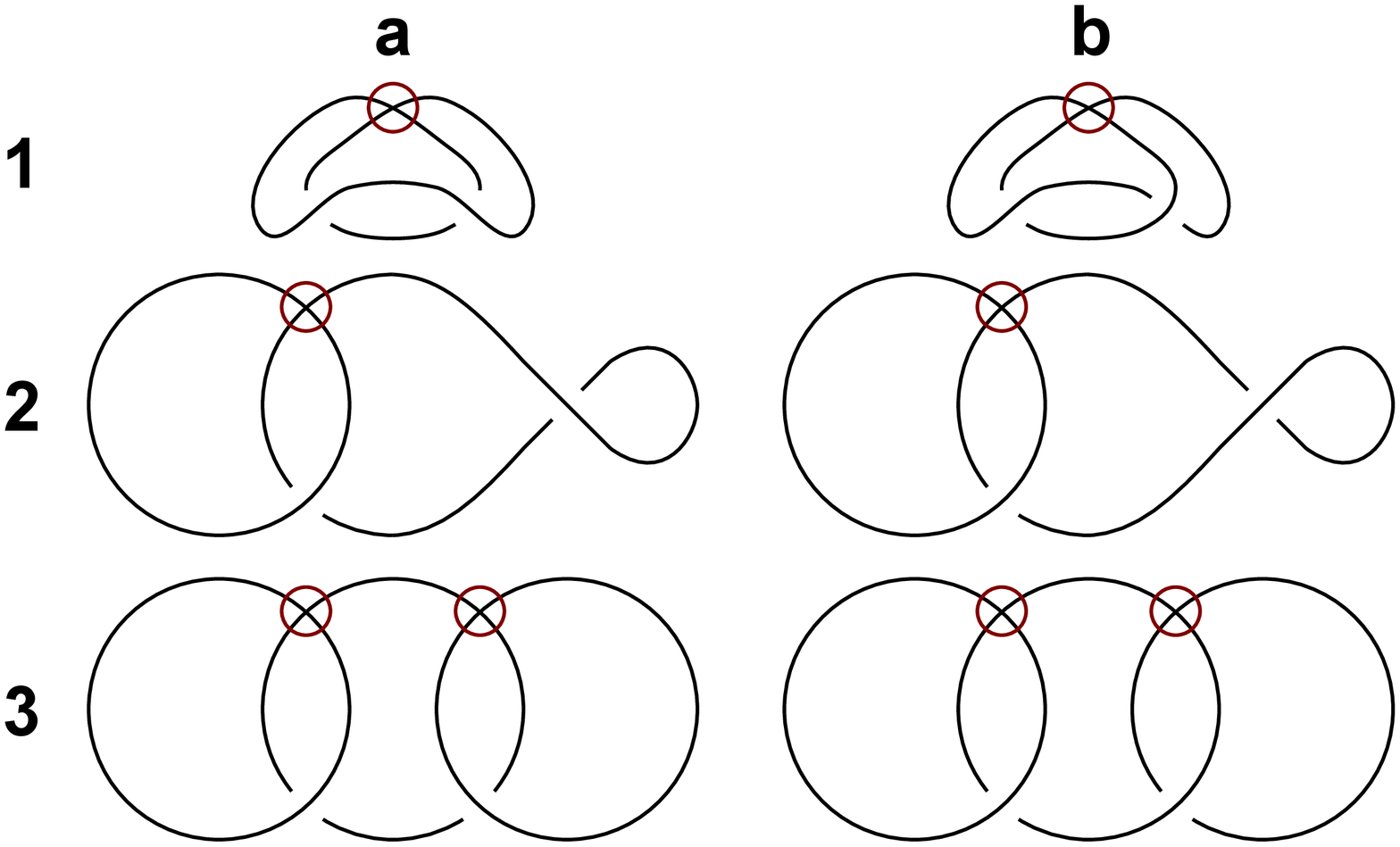}
\end{minipage}
\caption{Left: The basic orientable faces. Right: The basic non-orientable faces.}\label{figure-basic}
\end{figure} 
\begin{lem}\label{lem-basicfaces}
Let $L_D$ be a v-link diagram. Then $L_D$ can be reduced by a finite sequence of isotopies, vRM1, vRM2, vRM3, mRM moves and virtualisations to a v-link diagram $L^{\prime}_D$ in such a way that a fixed connected face of $L^{\prime}_D$ is isotopic to one of the basic faces from the Figure~\ref{figure-basic} (or to one of their mirror images) up to vRM1, vRM2, vRM3 moves on the face itself.\qed
\end{lem}
Note that these lemmata allow us to check arbitrary orientations on the basic faces with arbitrary numbering of crossings and components.
\begin{prop}\label{prop-standardfaces}
Let $L_D$ be a v-link diagram with a diagram which is isotopic to one of the projections from Figure~\ref{figure-basic}. Then $\bn{L_D}$ is a chain complex, i.e. the basic faces are anticommutative. Moreover, disjoint faces, i.e. faces such that the corresponding four-valent graph is unconnected, are always anticommutative.
\end{prop}
\begin{proof}
Because of Lemma~\ref{lem-commutativeindependence}, we only need to check that the faces are anticommutative for orientations of the resolutions of our choice with an arbitrary numbering. Then we are left with three different cases, i.e. the v-link diagram of $L_D$ is orientable, i.e. all saddles are orientable, or the face is non-orientable, i.e. two or four of the saddles are non-orientable or the face is disjoint.

For the first case we see that every resolution contains only c-circles. We prove the anticommutativity of the corresponding face for the following orientations of the resolutions. All appearing circles are numbered in ascending order from left to right or outside to inside. Moreover, the position of the x-marker does not affect our argument and we suppress it in this proof.

Because every resolution contains only c-circles, we choose a positive orientation for the circles except for the two nested circles that appear in two resolution of a face of type 1a or 1b. This is a clockwise orientation for all the non-nested circles and a counterclockwise orientation for the two nested circles. Hence, all appearing cylinders are identities.

It follows from this convention that every $0$-resolution (or $1$-resolution) $\smoothing$ of a crossing $\slashoverback$ (or a crossing $\backoverslash$) is of the form $\du$ and every $1$-resolution (or $0$-resolution) $\smoothing$ of a crossing $\slashoverback$ (or a crossing $\backoverslash$) is of the form $\ler$. Moreover, the only face with an even number of saddle signs is of type 1a.

All we need to do is compare these local orientations with the ones from Table~\ref{tab-deco}. We see that we have to check (indicated by the $\overset{!}{=}$) the following equations.
\begin{itemize}
\item $\Delta^{++}_+\circ m^-_{--}\overset{!}{=}-\Delta^{--}_-\circ m^+_{++}$ (face of type 1a).
\item $m^{++}_+\circ\Delta^+_{++}\overset{!}{=}m^{++}_+\circ\Delta^+_{++}$ (face of type 1b).
\item $(\Delta^+_{++}\amalg\,\mathrm{id}^+_+)\circ(\mathrm{id}^+_+\amalg\,m^{++}_+)\overset{!}{=}m^{++}_+\circ\Delta^{++}_+$ (face of type 2a).
\item $m^{++}_{+}\circ(m^{++}_-\amalg\,\mathrm{id}^+_+) =m^{++}_{+}\circ(\mathrm{id}^+_+\amalg\,m^{++}_+)$ (face of type 2b).
\end{itemize}
Most of these equations are easy to calculate. The reader should check that the cobordisms on the left and the right side of every equation are homeomorphic (using Proposition~\ref{prop-nonorientablefaces} and Lemma~\ref{lem-basiscalculations}).

Furthermore, the second equation is clear and the other three follows easy using the result of Lemma~\ref{lem-basiscalculations}. Hence, they are all anticommutative because only the first face has an even number of saddle signs.

The non-orientable faces of type 1b, 2a, 2b, 3a and 3b are easy to check. One can use the Euler characteristic here and the relations in Equation~\ref{eq-combrel3}.

The non-orientable face of type 1a is the face from~\ref{probcube}. Here we have to use Proposition~\ref{prop-nonorientablefaces}. We get two $\theta$-cobordisms and a $\Delta$- and a $m$-cobordism. Because of the relations in Equation~\ref{eq-combrel3} we can ignore the saddle signs.

Again we can choose an orientation for the resolutions. We can do this for example in the following way (compare to Figure~\ref{figure0-big}).
\begin{itemize}
 \item The first M\"obius strips are $\theta\colon\dd\rightarrow\lel$ and $\theta\colon\dd\rightarrow\rir$.
 \item The pantsdown is $\Delta^+_{-+}\colon\du\rightarrow\lel$ and the pantsup is $m^{--}_+\colon\ud\rightarrow\ler$.
\end{itemize}
We use Proposition~\ref{prop-nonorientablefaces} to see that this face is anticommutative.

The reader should check that all disjoint faces with only orientable saddles have an odd number of saddle signs. The disjoint faces with two or four non-orientable saddles anticommute because of the relations in Equation~\ref{eq-combrel3} and (k) of Lemma~\ref{lem-basiscalculations}.
\end{proof}
This proposition leads us to an important theorem and an easy corollary.
\begin{thm}\label{theo-facescommute}(\textbf{Faces commute}) Let $L_D$ be a v-link diagram. Let $\bn{L_D}$ be the complex from Definition~\ref{defn-topcomplex} with arbitrary possible choices. Then every face of the complex $\bn{L_D}$ is anticommutative.
\end{thm}
\begin{proof}
This is a direct consequence of the Proposition~\ref{prop-standardfaces} and the three Lemmata~\ref{lem-commutativeindependence},~\ref{lem-virtualisation} and~\ref{lem-basicfaces}.
\end{proof}
\begin{cor}\label{cor-chaincomplex}
The complex $\bn{L_D}$ is a chain complex. Thus, it is an object in the category $\ukob_R$.\qed 
\end{cor}
\subsection{Skew-extended Frobenius algebras}\label{sec-vkhfa}
We note that this section has three subsections. We construct the ``algebraic'' complex of a v-link diagram $L_D$ in the first subsection. It is an invariant of virtual links $L$, i.e. modulo the generalised Reidemeister moves from Figure~\ref{figure0-reide}. It should be noted that the notion of ``homology'' makes sense for the algebraic complex.

We describe the relation between uTQFTs and skew-extended Frobenius algebras in the second subsection. A relation of this kind was discovered by Turaev and Turner~\cite{tutu} for extended Frobenius algebras and the functors they use. Even though our construction is different, their ideas can be used in our context too. This is the main part of Theorem~\ref{thm-tututheo}. But our uTQFT correspond to \textit{skew-extended} Frobenius algebras, i.e. the map $\Phi$ is a skew-involution rather than an involution.

In the last subsection we are able to classify all aspherical uTQFTs which can be used to define v-links invariants, see Theorem~\ref{thm-class}. It is worth noting that we get an invariant for v-links which is an extension of the Khovanov complex for $R=\bZ$ or $R=\bQ$, see Corollary~\ref{cor-khovhomo}. Note that this includes that our construction can be seen as a categorification of the virtual Jones polynomial, see Corollary~\ref{cor-khovhomocat}. Moreover, we also get extension for other classical link homologies, see e.g. the Corollaries~\ref{cor-leekhovhomo} and~\ref{cor-bnhomo}.
\subsubsection*{The algebraic complex}
We denote any v-link diagram of the unknot with the symbol $\bigcirc$. Furthermore, we view v-circles, i.e. v-links without classical crossings, as disjoint circles immersed into $\bR^2$. Recall that $R$ is always a unital, commutative ring of arbitrary characteristic.
\begin{defn}\label{defn-utqft}(\textbf{uTQFT}) A \textit{(1+1)-dimensional unoriented TQFT} $\mathcal F$ (we call this a \textit{uTQFT}) is a strict, symmetric, covariant, $R$-pre-additive functor
\[
\mathcal F\colon\ucob_R(\emptyset)\rightarrow\RMOD.
\]
Here $\mathcal F(\bigcirc)$ is a finitely generated, free $R$-module. Let $\mathcal O,\mathcal O^{\prime}$ be two homeomorphic objects from $\ucob_R(\emptyset)$. Then $\mathcal F(\mathcal O)=\mathcal F(\mathcal O^{\prime})$ should hold. The functor $\mathcal F$ should also satisfy the following axioms.
\begin{itemize}
\item[(1)] Let $\mathcal O,\mathcal O^{\prime}$ be two disjoint objects in $\Ob(\ucob_R(\emptyset))$. Then there exists a natural (with respect to homeomorphisms) isomorphism between $\mathcal F(\mathcal O\amalg \mathcal O^{\prime})$ and $\mathcal F(\mathcal O)\otimes\mathcal F(\mathcal O^{\prime})$.
\item[(2)] The functor satisfies $\mathcal F(\emptyset)=R$.
\item[(3)] For a cobordism $C \colon \mathcal O \to \mathcal O^{\prime}\in\Mor(\ucob_R(\emptyset))$ the homomorphism $\mathcal F(C)$ is natural with respect to homeomorphisms of cobordisms.
\item[(4)] Let the cobordism $C\colon \mathcal O \to \mathcal O^{\prime}\in\Mor(\ucob_R(\emptyset))$ be a disjoint union of the two cobordisms $C_{1,2}$. Then $\mathcal F(C)=\mathcal F(C_1)\otimes\mathcal F(C_2)$ under the identification from axiom (1).
\end{itemize}
Two uTQFTs $\mathcal{F,F^{\prime}}$ are called \textit{isomorphic} if for each object of $\mathcal O\in\Ob(\ucob_R(\emptyset))$ there is an isomorphism $\mathcal F(\mathcal O)\rightarrow\mathcal F^{\prime}(\mathcal O)$, natural with respect to homeomorphisms of the objects and homeomorphisms of cobordisms, multiplicative with respect to disjoint union and the isomorphism assigned to $\emptyset$ is the identity morphism. 
\end{defn}
\begin{rem}
There are several things about the definition.
\begin{itemize}
\item Recall that our category is $R$-pre-additive. A uTQFT is a $R$-pre-additive functor. So we can extend this to a functor
\[
\mathcal F\colon\ukob_R\rightarrow\kom_b(\mat(\RMOD)),
\]
i.e. for every formal chain complex $(C_*,d_*)$ of objects of $\ucob_R(\emptyset)$, i.e. v-circles, the object $\mathcal F((C_*,d_*))$ is a chain complex of $R$-modules and for every formal chain map $f\colon (C_*,d_*)\to (C^{\prime}_*,d^{\prime}_*)$ of possible non-orientable, decorated cobordisms the morphism $\mathcal F(f)$ is a chain map of $R$-module homomorphisms.
\item A uTQFT $\mathcal F$ is a covariant functor. Hence, we see that $\mathcal F(\mathrm{id}^+_+)=\mathrm{id}$. Furthermore it is symmetric and hence $\mathcal F(\tau^{++}_{++})=\tau$. Here $\tau$ denotes the canonical permutation.
\item The permutation $\tau^{++}_{++}$ is natural. So we can assume that $A\otimes B$ and $B\otimes A$ are equal and not merely isomorphic.
\item For the definition of natural (converted to our setting) we refer the reader to~\cite{tutu}.
\end{itemize}
\end{rem}
\begin{defn}\label{defn-alcomplex}(\textbf{Algebraic complex}) Let $L_D$ be a v-link diagram. Then the \textit{algebraic complex of $L_D$} induced by the uTQFT $\mathcal F$ is the complex $\mathcal F(\llbracket L_D\rrbracket)$. 
\end{defn}
We prove the following important result. Here $L_D,L^{\prime}_D$ are a v-link diagrams. The proof is a direct consequence of Theorem~\ref{thm-geoinvarianz}.
\begin{thm}\label{thm-algcomplex}(\textbf{The algebraic complex is an invariant}) Let $\mathcal F$ a uTQFT which satisfies the Bar-Natan-relations of Figure~\ref{figureintroa-4}. Then the algebraic complex $\mathcal F(\llbracket L_D\rrbracket)$ is a v-link invariant in the following sense.

For two equivalent (up to the generalised Reidemeister moves) v-link diagrams $L_D,L^{\prime}_D$ the two chain complexes $\mathcal F(\llbracket L_D\rrbracket)$ and $\mathcal F(\llbracket L^{\prime}_D\rrbracket)$ are equal up to chain homotopy.\qed
\end{thm}
This theorem allows us to speak of \textit{the algebraic complex} $\mathcal F(\llbracket L\rrbracket)$ of any oriented v-link $L$. Furthermore, the category $\RMOD$ is abelian. Hence, the category $\kom_b(\mat(\RMOD))$ is also an abelian category. So unlike in the category $\ukob_R$, we have the notion of homology. We denote the homology of the algebraic chain complex by $H(\mathcal F(\bn L))$.
\subsubsection*{Skew-extended Frobenius algebras and uTQFTs}
We continue with the definition of an algebra that we call a \textit{skew-extended Frobenius algebra}. For a $R$-bialgebra $A$ with comultiplication $\Delta$ and counit $\varepsilon$ we call an $R$-algebra homomorphism $\Phi\colon A\to A$ a \textit{skew-involution} if it satisfies the following.
\begin{itemize}
\item[(a)] $\Phi^2=\mathrm{id}$ (involution).
\item[(b)] $(\Phi\otimes\Phi)\circ\Delta\circ\Phi=-\Delta$ and $\varepsilon\circ\Phi=-\varepsilon$ (skew-property).
\end{itemize}
\begin{defn}\label{defn-exfrob}(\textbf{Skew-extended Frobenius algebras}) A \textit{Frobenius algebra $A$ over $R$} is a unital, commutative algebra over $R$ which is projective and of finite type (as a $R$-module), together with a module homomorphism $\varepsilon\colon A\to R$, such that the bilinear form $\left\langle \cdot,\cdot\right\rangle$ defined by $\left\langle a,b\right\rangle=\varepsilon(ab)$ for all $a,b\in A$ is non-degenerate.

An \textit{skew-extended Frobenius algebra $A$ over $R$} is a Frobenius algebra together with a skew-involution of Frobenius algebras $\Phi\colon A\to A$ and an element $\theta\in A$ which satisfy the two equations below. Note that one can use $\varepsilon$ to define a comultiplication $\Delta$.
\begin{itemize}
\item[(1)] $\Phi(\theta a)=\theta a=\theta\Phi(a)$ for all $a\in A$.
\item[(2)] $(m\circ(\Phi\otimes\mathrm{id})\circ\Delta)(1)=\theta^2$. 
\end{itemize}
\end{defn}
\begin{nota}
Because of $1\in A$, we can define $\iota\colon R\to A$ by $1\mapsto 1$. We can write a Frobenius algebra uniquely as $\mathcal F=(R,A,\varepsilon,\Delta)$. Moreover, we can write such a skew-extended Frobenius algebra $\mathcal F$ uniquely as $\mathcal F=(R,A,\varepsilon,\Delta,\Phi,\theta)$.
\end{nota}
\begin{defn}\label{defn-isofrobenius}
Two skew-extended Frobenius algebras, denoted $\mathcal F_1=(R,A,\varepsilon,\Delta,\Phi,\theta)$ and $\mathcal F_2=(R,A^{\prime},\varepsilon^{\prime},\Delta^{\prime},\Phi^{\prime},\theta^{\prime})$, are called \textit{isomorphic} if there exists an isomorphism of Frobenius algebras $f\colon A\rightarrow A^{\prime}$, which satisfies $f(\theta)=\theta^{\prime}$ and $f\circ\Phi=\Phi^{\prime}\circ f$.

We call a Frobenius algebra \textit{aspherical} if $\varepsilon(\iota(1))=0$. Furthermore, we say it is a \textit{rank2-Frobenius algebra} if $A\cong 1\cdot R\oplus X\cdot R$ as $R$-modules.
\end{defn}
\begin{rem}\label{rem-categoryfrob}
The map $\varepsilon$ is called the \textit{counit} of $A$. It can be used to define a \textit{comultiplication} $\Delta\colon A\to A\otimes A$. We will call $m\colon A\otimes A\rightarrow A$ the \textit{multiplication} of $A$. The coproduct and the product make the two diagrams
\[
\begin{xy}
  \xymatrix{
 A\otimes A\ar[r]^{\mathrm{id}\otimes\Delta}\ar[d]_{\Delta\otimes\mathrm{id}}\ar[dr]|/.3em/{\Delta\circ m} & A\otimes A\otimes A\ar[d]^{m\otimes\mathrm{id}} & & A\ar[dr]|{\mathrm{id}}\ar[r]^{\Delta}\ar[d]_{\Delta} & A\otimes A\ar[d]^{\mathrm{id}\otimes\varepsilon}\\
 A\otimes A\otimes A\ar[r]_{\mathrm{id}\otimes m} & A\otimes A  & & A\otimes A\ar[r]_{\varepsilon\otimes\mathrm{id}} & A}
\end{xy}
\]
commutative. In a skew-extended Frobenius algebra the skew-involution $\Phi$ and the element $\theta$ make the two diagrams
\[
\begin{xy}
  \xymatrix{
& A\ar[rd]^{\Phi} & & & & A\otimes A\ar[rd]^{m^{\prime}} & \\
A\ar[ru]^{\cdot\theta}\ar[rr]_{\cdot\theta} & & A & & A\ar[ru]^{\Delta}\ar[rd]_{\cdot\theta} &  & A\\
 & & & & & A\ar[ru]_{\cdot\theta} & }
\end{xy}
\]
commutative (it is easy to check that the two equations from Definition~\ref{defn-exfrob} already imply the equation $(m\circ(\phi\otimes\mathrm{id})\circ\Delta)(a)=\theta^2 a$ for all $a\in A$). Here the map $\cdot\theta\colon A\to A$ is the multiplication with $\theta$ and the map $m^{\prime}\colon A\otimes A\to A$ is the map $(\Phi\otimes\mathrm{id})\circ m$.

We recognise that the lower right diagram is the problematic face from~\ref{probcube}. So the second equation from Definition~\ref{defn-exfrob} is a key point in the definition. 
\end{rem}
The following theorem is inspired by a corresponding theorem in~\cite{tutu}.
\begin{thm}\label{thm-tututheo}
The isomorphism classes of (1+1)-dimensional uTQFTs over $R$ are in bijective correspondence with the isomorphism classes of skew-extended Frobenius algebras over $R$.
\end{thm}
\begin{proof}
First let us consider a uTQFT $\mathcal F$ over $R$. We describe a way to get a skew-extended Frobenius algebra from it. Let us denote this algebra by $(R,A,\varepsilon,\Delta,\Phi,\theta)$.

We take $A=\mathcal F(\bigcirc)$ as our underlying $R$-module. Next we need a skew-involution $\Phi\colon A\to A$. We take the cylinder from Figure~\ref{figure1-1}. Set $\Phi=\mathcal F(\Phi^-_+)$.

The unit $\iota$ should be $\mathcal F(\iota_+)$. There is no further choice because $\iota_+=\iota_-$. The counit should be $\mathcal F(\varepsilon^+)$. Here we have a choice because $\varepsilon^+\neq\varepsilon^-$. But because of $\varepsilon^+=-\varepsilon^-$, both choices lead to isomorphic algebras.

Now we need a multiplication $m$ and a comultiplication $\Delta$. One may suspect, that we have different choices for either of them, namely the eight $m^{\pm\pm}_{\pm},\Delta^{\pm}_{\pm\pm}$. But the relations of a Frobenius algebra only allow one option. We discuss this now. It should be noted that the computations below can be done using Lemma~\ref{lem-basiscalculations}.
\begin{itemize}
\item The lower boundary components of $\Delta^u_{l_1l_2}$ must have the same glueing numbers as the boundary component of $\varepsilon^+$ because $\mathcal F(\varepsilon^+)$ should be the counit.
\item Because of the relation $\varepsilon\circ m\circ(\mathrm{id}\otimes\iota)=\varepsilon=\varepsilon\circ m\circ(\iota\otimes\mathrm{id})$, the lower boundary of $m^{u_1u_2}_l$ must have the same glueing number as the boundary component of $\varepsilon^+$. The same is true for the upper boundary (this means we need $m^{++}_+=m^{--}_-$).
\item Because of the relation $(\mathrm{id}\otimes\,m)\circ(\Delta\otimes\mathrm{id})=\Delta\circ m=(m\otimes\mathrm{id})\circ(\mathrm{id}\otimes\Delta)$, the $m^{u_1u_2}_{l}$ must have the same glueing number on the lower boundary as the upper boundary of $\Delta^u_{l_1l_2}$ (the reader should check that this is the only possible choice for the glueing numbers for $m^{u_1u_2}_{l}$ and $\Delta^u_{l_1l_2}$).
\end{itemize}
Therefore, we have $\mathcal F(\iota_+)=\iota$, $\mathcal F(\varepsilon^+)=\varepsilon$, $\mathcal F(m^{++}_+)=m$ and $\mathcal F(\Delta^+_{++})=\Delta$.

The last piece missing is the element $\theta\in A$. Consider a two times punctured projective plane $\mathbb{RP}^2_2$ (a punctured M\"obius strip). This is $\theta$ in our notation.

Then $\theta\circ\iota_+\colon\emptyset\to\bigcirc$ is a punctured projective plane (hence a M\"obius strip). Set $\theta=\mathcal F(\theta\circ\iota_+)(1)$. Because of the definition, this is an element of $\mathcal F(\bigcirc)=A$.

We have to prove the equations needed for a skew-extended Frobenius algebra, i.e. that $\iota$ is a unit, $\varepsilon$ is a counit, $\Phi$ is a skew-involution, $m$ ($\Delta$) is a (co)multiplication and the commutativity of the faces from Remark~\ref{rem-categoryfrob}.

This is a straightforward verification bases on the relations from Lemma~\ref{lem-basiscalculations} (we omit it here). This shows that every uTQFT has an underlying skew-extended Frobenius algebra.
\vskip0.5cm
For the other direction, i.e. if we assume that we have a skew-extended Frobenius algebra, we note that this algebra has an underlying ``classical'' Frobenius algebra. Therefore we get a TQFT $\mathcal F^{\prime}$ from this underlying Frobenius algebra. We want to use this TQFT to define a uTQFT $\mathcal F$. The TQFT $\mathcal F^{\prime}$ is a covariant functor
\[
\mathcal F^{\prime}\colon\cob_R(\emptyset)\rightarrow\RMOD.
\]
Let $\mathcal O$ be an object in $\ucob_R(\emptyset)$. This object gives us (modulo homeomorphisms) a corresponding object $\mathcal O^{\prime}$ in $\cob_R(\emptyset)$. We set $\mathcal F(\mathcal O)=\mathcal F^{\prime}(\mathcal O^{\prime})$. This assignment clearly satisfies that $\mathcal F(\bigcirc)$ is a finitely generated, free $R$-module and $\mathcal F(\mathcal O_1)=\mathcal F(\mathcal O_2)$ for two homeomorphic objects $\mathcal O_1,\mathcal O_2$.

Moreover, because $\mathcal F^{\prime}$ is a TQFT, this satisfies the first two axioms from our Definition~\ref{defn-utqft}. Now we need to define $\mathcal F(\mathcal C)$ for morphisms from $\ucob_R(\emptyset)$.
\vskip0.5cm
First we assume that $\mathcal C\colon\mathcal O_1\to\mathcal O_2$ is orientable and connected. Then we have a corresponding morphism in $\cob_R(\emptyset)$, i.e. the same without the boundary decorations, which we will denote by $\mathcal C^{\prime}\colon\mathcal O^{\prime}_1\to\mathcal O^{\prime}_2$.

We denote the cap-, cup-, pantsup- and pantsdown-cobordisms in the category $\cob_R(\emptyset)$ by $\iota,\varepsilon,m$ and $\Delta$ respectively. Let us define
\[
\mathcal F(\iota_+)=\mathcal F^{\prime}(\iota),\mathcal F(\varepsilon^+)=\mathcal F^{\prime}(\varepsilon),\mathcal F(m^{++}_+)=\mathcal F^{\prime}(m),\mathcal F(\Delta^+_{++})=\mathcal F^{\prime}(\Delta)\text{ and }\mathcal F(\Phi^-_+)=\Phi.
\]
The map $\Phi$ is the skew-involution in the skew-extended Frobenius algebra. Thus, we can define $\mathcal F(\mathcal C)$ in the following way. We decompose $\mathcal C^{\prime}$ into the basic pieces $\iota,\varepsilon,m,\Delta$. Then $\mathcal F^{\prime}(\mathcal C^{\prime})$ is independent of this decomposition because $\mathcal F^{\prime}$ is a TQFT. If we use the same decomposition for $\mathcal C$ (under the identification from above), we get a cobordism $\tilde{\mathcal C}$. For this cobordism we can define $\mathcal F(\tilde{\mathcal C})$. We see that we only have to change some of the boundary decorations of $\tilde{\mathcal C}$ to obtain $\mathcal C$. Hence, we have
\[
\mathcal C=\mathcal C_1\circ\tilde{\mathcal C}\circ\mathcal C_2,
\]
where $\mathcal C_1,\mathcal C_2$ are cylinders of the type $\mathrm{id}^+_+$ or $\Phi^-_+$. Hence, we can define
\[
\mathcal F(\mathcal C)=\mathcal F(\mathcal C_1)\circ\mathcal F(\tilde{\mathcal C})\circ\mathcal F(\mathcal C_2).
\]
That this is also independent of the decomposition follows from the fact that $\mathrm{id}^+_+,\Phi^-_+$ and the corresponding maps in the skew-extended Frobenius algebra are (skew-)involutions and a ``level-by-level''\footnote{One can for example verify the statement by induction on the number of generators in the decomposition.} change of decorations using the relations in Lemma~\ref{lem-basiscalculations}. Moreover, for a non-connected, orientable cobordism $\mathcal C$ we extend the definition from above multiplicatively.
\vskip0.5cm
For a non-orientable, connected cobordism $\mathcal C$ we have to define $\mathcal F(\theta)=\cdot\theta$ first. Here the map $\cdot\theta\colon A\to A$ is the multiplication with the element $\theta$ in our skew-extended Frobenius algebra. Hence, if we decompose $\mathcal C=\mathcal C_{or}\#n\mathbb{RP}^2$ into a (non-decorated) orientable part $\mathcal C_{or}$ and $n$-times a projective plane we define
\[
\mathcal F(\mathcal C)=\theta^n\mathcal F^{\prime}(\mathcal C_{or}).
\]
This is again independent of the decomposition of $\mathcal C_{or}$, because of the first relation in a skew-extended Frobenius algebra, namely $\Phi(\theta a)=\theta a=\theta\Phi(a)$ for all $a\in A$. Furthermore, it is independent from the decomposition $\mathcal C=\mathcal C_{or}\#n\mathbb{RP}^2$, because if we replace a $2-\mathbb{RP}^2$ with a torus $\mathcal T$, we see that $\mathcal F(\mathcal C_{or})$ is multiplied by a factor $(m\circ(\Phi\otimes\mathrm{id})\circ\Delta)(1)\theta^{n-2}$. Hence, using the second relation of the skew-extended Frobenius algebra, we get
\begin{align*}
\mathcal F(\mathcal C_{or}\#n\mathbb{RP}^2)&=(\theta^n)\mathcal F^{\prime}(\mathcal C_{or})=\theta^{n-2}(m\circ(\Phi\otimes\mathrm{id})\circ\Delta)(1)\mathcal F^{\prime}(\mathcal C_{or})=\mathcal F(\mathcal C_{or}\#\mathcal T\#(n-2)\mathbb{RP}^2).
\end{align*}
For a non-connected, non-orientable cobordism $\mathcal C$ we extend the definition from above multiplicatively. Hence, we only have to show the remaining axioms from the Definition~\ref{defn-utqft}. The reader should check these axioms (one could follow the end of the proof in~\cite{tutu}).
\end{proof}
\subsubsection*{Classification of v-link homologies}
From now on we use the notions uTQFT and skew-extended Frobenius algebra interchangeably.
\begin{prop}\label{cor-univalg}(\textbf{The universal skew-extended Frobenius algebra}) Every aspherical rank2-uTQFT comes from the rank2-uTQFT $\mathcal F_U=(R_U,A_U,\varepsilon_U,\Delta_U,\Phi_U,\theta_U)$ through base change. Here the ring $R_U$ is $R_U=\bZ[a,a^{-1},\alpha,\beta,\gamma,t]/\mathcal I$ with $\mathcal I$ is the ideal generated by the relations (we use the notation $h=a^{-1}\gamma-\alpha^2-\beta^2 t$ here)
\[
\alpha\gamma=\beta\gamma=2\alpha=2\beta=a^2\beta^2 h=0.
\]
Furthermore, the algebra is $A_U=R_U[X]/(X^2=t+ahX)$, the element $\theta_U\in R_U$ is given by $\theta_U=\alpha+\beta\cdot X$ and the maps will be the ones from Table~\ref{tabular-matrixtable}. The table is the following.
\vskip0.15cm
\begin{table}[ht]
\begin{center}
\begin{tabular}{|c||c|}
\hline
$\iota_U\colon R\to A,\;1\mapsto 1.$ & $\Phi_U\colon A\to A,\;\begin{cases}1\mapsto 1,\\X\mapsto \gamma-X.\end{cases}$\\
\hline
$\varepsilon_U\colon A\to R,\;1\mapsto 0,\,X\mapsto a.$ & $\cdot\theta_U\colon A\to A,\;\begin{cases}1\mapsto \alpha+\beta\cdot X,\\X\mapsto \beta t+(\alpha+a\beta h) \cdot X.\end{cases}$\\
\hline
\multicolumn{2}{|c|}{$m_U\colon A\otimes A\to A,\;\begin{cases}1\otimes1\mapsto 1,\,1\otimes X\mapsto X,\\X\otimes 1\mapsto X,\,X\otimes X\mapsto t+h\cdot X.\end{cases}$}\\
\hline
\multicolumn{2}{|c|}{$\Delta_U\colon A\to A\otimes A,\;\begin{cases}1\mapsto -h\cdot1\otimes 1+a^{-1}(1\otimes X+X\otimes 1),\\X\mapsto a^{-1}t\cdot 1\otimes 1+a^{-1}\cdot X\otimes X.\end{cases}$}\\
\hline
\end{tabular}
\caption{The maps for the generators from Figure~\ref{figure1-1}.}\label{tabular-matrixtable}
\end{center}
\end{table}
\end{prop}
\begin{proof}
We start by showing that the data given above give rise to a skew-extended Frobenius algebra, i.e. the satisfy the axioms given in Definition~\ref{defn-exfrob}. Note that the algebra $A_U$ is certainly a rank2-algebra over $R_U$, $\theta_U=\alpha+\beta X\in A_U$ and $(\varepsilon_U\circ\iota_U(1))=0$.

Moreover, it satisfies the axioms of an aspherical Frobenius algebra, since it, forgetting the new structure, coincides with the classical one given in~\cite{kh2}.

A direct computation verifies that $\Phi_U$ is a skew-involution, i.e.
\[
\Phi_U\circ \Phi_U=\mathrm{id_{A_U}},\;(\Phi_U\otimes\Phi_U)\circ\Delta_U\circ\Phi_U=-\Delta_U \;\text{and}\;\varepsilon_U\circ\Phi_U=-\varepsilon_U.
\]
Furthermore, a direct computation shows that $\cdot\theta_U$ and $\Phi_U$ also satisfy the axioms (a) and (b) from Definition~\ref{defn-exfrob}, i.e. the whole data is an aspherical rank2-uTQFT. 
\vskip0.5cm
Now assume that we have a given aspherical rank2-uTQFT $\mathcal F=(R,A,\varepsilon,\Delta,\Phi,\theta)$.

First we observe that a skew-extended Frobenius algebra $A$ has an underlying Frobenius algebra of rank two. Hence, $\iota$ has to be of the given form. Because it is also aspherical, i.e. $\varepsilon(\iota(1))=0$, we see that $\varepsilon(1)=1$ and $\varepsilon(X)=a\cdot 1$. The element $a\in R$ is invertible because of the relation
\[
(\varepsilon\otimes\mathrm{id})\circ\Delta=\mathrm{id}=(\mathrm{id}\otimes\text{ }\varepsilon)\circ\Delta.
\]
It is known (e.g.~\cite{kh2}) that such an algebra is of the form $A=R[X]/(X^2=t+ahX)$ with multiplication $m$ and comultiplication $\Delta$ from the table~\ref{tabular-matrixtable} above.
\vskip0.5cm
Next we look at the new structure. Because $\theta$ is an element of $A\cong 1\cdot R\oplus X\cdot R$ we find $\alpha,\beta\in R$ such that $\theta=\alpha+\beta X$. Using the multiplication we see that $X^2=t+ah\cdot X$. So an easy calculation shows that $\theta\cdot X=\beta t+(\alpha+a\beta h)X$ which gives us the map $\cdot\theta$ as above.

Because the map $\Phi\colon A\to A$ is not only $R$-linear, but also a skew-involution, we get $\Phi(1)=1$ and with $\varepsilon\circ\Phi=-\varepsilon$ we get $\Phi(X)=\gamma-X$. Using the first relation of a skew-extended Frobenius algebra we get the relations $\alpha\gamma=\beta\gamma=2\beta=0$ and $2(\alpha+a\beta h)=2\alpha=0$.

Using the second relation of a skew-extended Frobenius algebra, namely
\[
m\circ(\Phi\amalg\mathrm{id_{A}})\circ\Delta=(\cdot\theta)^2,
\]
we get the last two relations $ah=\gamma-a\alpha^2-a\beta^2t$ and $a^2\beta^2 h=0$.

These are all relations we get from the axioms of an aspherical rank2-uTQFT, i.e. any other axiom will also lead to one of these relations.
\end{proof}
\begin{rem}\label{rem-tutuextension}
The reader familiar with the paper of Turaev and Turner~\cite{tutu} will recognise that our \textit{universal skew-extended Frobenius algebra} $\mathcal F_U$ is different from the one from Turaev and Turner. But this is an advantage (see Corollary~\ref{cor-khovhomo}).

As mentioned before in the Remark~\ref{rem-tutu}, the version of Turaev and Turner can be obtained from our concept too. The difference again are the relations $\varepsilon^+=-\varepsilon^-$ and $\Delta^+_{++}=-\Delta^-_{--}$. This forces $\mathcal \Phi=F(\Phi^-_+)$ from the proof above to send $X\mapsto\gamma -X$ instead of $X\mapsto\gamma +X$ (but over $R$ with $\mathrm{char}(R)=2$ they coincide).
\end{rem}
The next corollary allows us to characterise the uTQFTs which lead to v-link homology.
\begin{cor}
\label{cor-bnraretrue}
Every aspherical rank2-uTQFT $\mathcal F$ satisfy the local relations from Figure~\ref{figureintroa-4}.
\end{cor}
\begin{proof}
View a sphere $S^2$ as a cobordism $S^2\colon\emptyset\rightarrow\emptyset$. Then $\mathcal F(S^2)=\mathcal F(\varepsilon^+)\circ\mathcal F(\iota_+)$. So we calculate $\mathcal F(S^2)=0$. Because of the axiom (4) from Definition~\ref{defn-utqft}, this is true for every cobordism with a sphere. Analogously view a torus $\mathcal T$ as a cobordism $\mathcal T\colon\emptyset\rightarrow\emptyset$. Thus, it is of the form $\mathcal F(\mathcal T)=\mathcal F(\varepsilon^+)\circ \mathcal F(m^{++}_+)\circ\mathcal F(\Delta^+_{++})\circ\mathcal F(\iota_+)$. An easy calculation with the maps of Table~\ref{tabular-matrixtable} shows, that $\mathcal F(\mathcal T)=2$. Because of the axiom (4), this is true for every cobordism with a torus.

The \textit{4Tu}-relation is algebraical just the formula
\[
\Delta_{12}\circ\iota+\Delta_{34}\circ\iota=\Delta_{13}\circ\iota+\Delta_{24}\circ\iota.
\]
Here $\Delta_{ij}:A\rightarrow A\otimes A\otimes A\otimes A$ is the map which sends an element $a\in A$ to an element $a_1\otimes a_2\otimes a_3\otimes a_4$ with $a_k=a$ for $k\neq i,j$ and $a_i,a_j$ the first respectively the second tensor factor of $\Delta(a)$ (see Figure~\ref{figure-4tu}).
\begin{figure}[ht]
 \centering
    \includegraphics[scale=0.525]{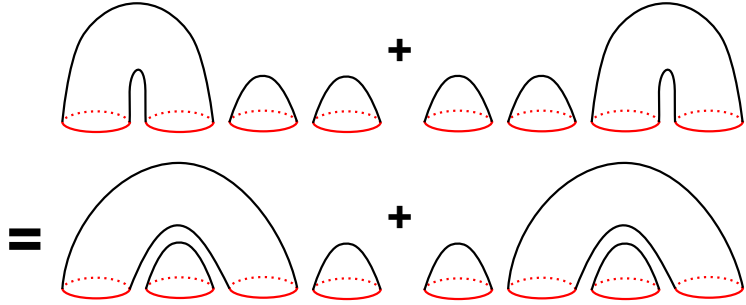}
 \caption{The relation $\Delta_{12}\circ\iota+\Delta_{34}\circ\iota=\Delta_{13}\circ\iota+\Delta_{24}\circ\iota.$}
 \label{figure-4tu}
\end{figure}

That this relation is true is also an easy calculation. Again axiom (4) gives us the global statement. Because this is true for the universal skew-extended Frobenius algebra $\mathcal F_U$, we get the statement for all aspherical rank2-uTQFTs from the Proposition~\ref{cor-univalg}.
\end{proof}
Because with an aspherical, rank2 skew-extended Frobenius algebra we can define a corresponding rank2-uTQFT which satisfies the Bar-Natan relations, we note the following two corollaries.
\begin{cor}\label{cor-invariant}
Every aspherical, rank2 uTQFT can be used to define a v-link invariant.\qed
\end{cor}
\begin{cor}\label{cor-khovhomo}(\textbf{The virtual Khovanov complex}) The above construction enables one to extend the Khovanov complex ($R_{Kh}=\bZ,A_{Kh}=\bZ[X]/(X^2=0,t=h=0)$) from c-links to v-links by setting $\alpha=\beta=\gamma=0$ and $a=1$.\qed
\end{cor}
From now on we denote by $\Kh (L)=\mathcal F_{Kh}(\bn L)$ the \textit{virtual Khovanov complex} of a v-link $L$ and by $H(\Kh (L))$ its homology.
\begin{rem}
It is possible to introduce gradings (by setting $\deg 1=1$ and $\deg X=-1$) for the complex from Corollary~\ref{cor-khovhomo}. This is true because the map $\cdot\theta=0$. In fact this is the only possibility where we can introduce gradings, because all maps in the Khovanov complex must decrease the grading by one. And this is only possible if $\cdot\theta\colon A\rightarrow A$ is equal zero.
\end{rem}
\begin{cor}\label{cor-khovhomocat}(\textbf{Categorification of the virtual Jones polynomial}) The virtual Khovanov complex~\ref{cor-khovhomo} is a categorification of the virtual Jones polynomial in the sense that its graded Euler characteristic gives the polynomial.
\end{cor}
\begin{proof}
The classical Jones polynomial is uniquely determined by the skein-relations. The same is true for the virtual Jones polynomial, see for example~\cite{ka1}. One can now easily check that the virtual Khovanov complex $\Kh(\cdot)$ satisfies these relations.
\end{proof}
There is also an extension of the Khovanov-Lee complex (see her paper~\cite{lee}) and two different extensions of Bar-Natan's variant ($R=\bZ/2,h=1,t=0$) (see his paper~\cite{bn2}).
\begin{cor}\label{cor-leekhovhomo}(\textbf{The virtual Khovanov-Lee complex}) The above construction enables us to extend the Khovanov-Lee complex ($R_{Lee}=\bZ$ and $A_{Lee}=\bZ[X]/(X^2=0,t=1,h=0)$) from c-links to v-links by setting $\alpha=\beta=\gamma=0$ and $a=1$.\qed
\end{cor}
\begin{cor}\label{cor-bnhomo}(\textbf{The virtual Khovanov-Bar-Natan complex}) The above construction enables us to extend Bar Natan's variant of Khovanov homology (this is the Frobenius algebra over the field $R_{BN}=\bZ/2$ with $A_{BN}=\bZ/2[X]/(X^2=0,t=0,h=1)$) from c-links to v-links in two different ways by setting $\alpha=\beta=0$ and $\gamma=a=1$ or by setting $\beta=\gamma=0$ and $\alpha=a=1$. The two extensions are non-isomorphic skew-extended Frobenius algebras.
\end{cor}
\begin{proof}
That these two skew-extended Frobenius algebras can be used as v-link homologies follows from Corollary~\ref{cor-invariant}. To see that they are non-isomorphic skew-extended Frobenius algebras we note that $\theta=0$ in the first case and $\theta=1$ in the second case. Because any isomorphism of skew-extended Frobenius algebras satisfies $f(1)=1$ and $f(\theta)=\theta'$, they are not isomorphic.
\end{proof}
We denote these three extensions by $\mathcal F_{Lee} (L)=\mathcal F_{Lee}(\bn L)$, $\mathcal F_{BN1}(L)=\mathcal F_{BN1}(\bn L)$ and $\mathcal F_{BN2}(L)=\mathcal F_{BN2}(\bn L)$ respectively.
\begin{prop}\label{prop-classiccomplex}
Let $L_D$ be a c-link diagram and let $\mathcal F$ be an aspherical rank2-uTQFT. Then the complex $\mathcal F(\bn{L_D})$  is the classical Khovanov complex (up to chain isomorphisms) which is obtained by using the underlying TQFT $\mathcal F^{\prime}$ of $\mathcal F$.

An similar statement is true for the Khovanov-Lee complex and the two different versions of the Khovanov-Bar-Natan complex.
\end{prop}
\begin{proof}
This is just the algebraic version of Theorem~\ref{thm-classic}.
\end{proof}
It is worth noting that, if $L$ is a c-link, then these three (and any other of the possibilities) are the classical complexes (up to chain homotopies) due to Theorem~\ref{thm-algcomplex}. Moreover, it should be noted that Corollary~\ref{cor-bnhomo} and Proposition~\ref{prop-classiccomplex} include that a v-link diagram $L_D$ with
\[
H(\mathcal F_{BN1}(L_D))_*\not\cong H(\mathcal F_{BN2}(L_D))_*
\]
can not be a v-diagram of a c-link, since a c-link diagram does not need the map $\cdot\theta$.
\vskip0.5cm
Because Khovanov showed (see~\cite{kh2}) that every TQFT which respects the first Reidemeister move must have an underlying $R$-module $A\cong 1\cdot R\oplus X\cdot R$ for an element $X\in A$, we also get the following theorem.
\begin{thm}\label{thm-class}(\textbf{Classification of aspherical uTQFTs}) The following statements are equivalent for an aspherical uTQFT.
\begin{itemize}
\item[(a)] It respects the first Reidemeister move RM1.
\item[(b)] It is a rank2-uTQFT.
\item[(c)] It can be obtained from the one of Proposition~\ref{cor-univalg}.
\item[(d)] It can be used as a v-link invariant.
\end{itemize} 
\end{thm}
With the work already done the proof is simple.
\begin{proof}
(a)$\Rightarrow$(b): This was done by Khovanov and stays true.

(b)$\Rightarrow$(c): This is just the Proposition~\ref{cor-univalg}.

(c)$\Rightarrow$(d): This is the Corollary~\ref{cor-invariant}.

(d)$\Rightarrow$(a): This is clear.
\end{proof}
\begin{rem}\label{rem-mant}
We conjecture that Manturov's $\bZ$-version~\cite{ma2} is a strictly weaker invariant than our extension of the Khovanov complex~\ref{cor-khovhomo} in the following sense. A v-link with ``lots'' of classical crossings is likely to have ``lots'' of faces of type 1b or its mirror image (see~\ref{figure-basic}). We call these faces the \textit{virtual trefoil faces}. In our construction the two multiplications (or comultiplications for the mirror image) are not the same, i.e. they have different boundary decorations as pictured for example in Figure~\ref{figure-big2}, since we take extra information of this face in account. In contrast, in Manturov's version they are just the same maps. It is worth noting that we use the extra information explicit in Section~\ref{sec-vkhapp}.
\end{rem}
\begin{rem}\label{rem-otherrelations}
At this state it is a fair question to ask why we use the relations (1) from Equation~\ref{eq-combrel12} (or the one without the signs for the variant of Turaev and Turner) for our cobordisms, i.e. why do we assume that $\Delta^{+}_{++}$ changes its sign under conjugation with $\Phi^-_+$ and not $m^{++}_+$ (or neither of them changes its sign for the variant of Turaev and Turner).

So what happens if we assume that $m^{++}_+$ changes its sign under conjugation with $\Phi^-_+$ (or both)? One can repeat the whole construction from Section~\ref{sec-vkhcat}, Section~\ref{sec-vkhcom} and this Section~\ref{sec-vkhfa} for these cases too. But this do not lead to anything new, i.e. if we assume that $m^{++}_+$ changes its sign, then we get an equivalent to the construction above and if we assume that both of them changes their signs, then we get an equivalent to the variant of Turaev and Turner again.
\end{rem}
\begin{rem}\label{rem-newthings}
Note that the classification of Theorem~\ref{thm-class} and the Table~\ref{tabular-matrixtable} include non-classical invariants. To be more precise, if we work for example over $R=\bQ$, then the relations force us to set $\theta=0$. But if we work over $R=\bZ/2$, then we have different choices for $\theta$. It should be noted, since c-links do not require the map $\cdot\theta$, these invariants can not appear in the classical setting. Note that in both of Manturov's versions~\cite{ma2} and~\cite{ma1} he sets $\theta=0$. 
\end{rem}
\subsection{The topological complex for virtual tangles}\label{sec-vkhtan}
We will define the \textit{topological complex of a v-tangle diagram} $T^k_D$ in this section. For this construction we use our notations for the \textit{saddle decorations} and \textit{saddle signs} of v-link diagrams $L_D$ from Section~\ref{sec-vkhcom}. Recall that a crucial ingredient for the construction of the topological complex were the \textit{decorations} of the saddles. Note that we work in a slightly different category now, i.e. the one from Definition~\ref{defn-category3}. Hence, we need signs, glueing numbers and indicators.

It is worth noting that the idea how to solve the problems that come with the observation summarised in Figure~\ref{figure0-order} in a non-trivial way (that is we do not define open saddles to be zero) is the following. Take the signs and decorations of a closure of the v-tangle diagram, since we already defined how to spread them for v-link diagrams in a ``good'' way. Note that this convention makes it easy to show analogous statements as in Section~\ref{sec-vkhcom}.
\vskip0.5cm
We note that this section has two subsection. We define the topological complex of a v-tangle diagram with a *-marker and show that it is v-tangle invariant in the first part, i.e. in Definition~\ref{defn-geocomplex} and Theorem~\ref{thm-invarianz}. In the second part we discuss how the position of the *-marker has influence on the topological complex. We can show~\ref{thm-marker} that in general the position of the *-marker gives rise to two different v-tangle invariances, but it agrees with the classical construction for c-tangles.
\subsubsection*{The topological complex for virtual tangles}
We start by explaining how we are going to extend the important notions of saddle sign and decorations to v-tangle diagrams.

Recall that $T^k_D$, as in Definition~\ref{defn-vtangle}, should denote a v-tangle diagram with $k\in\bN$ boundary points. Moreover, such diagrams should always have a *-marker on the boundary and let $\mathrm{Cl}(T^k_D)$ be the closure of the diagram. Recall that such diagrams come with x-markers.
\begin{defn}\label{defn-deco2}(\textbf{``Open'' saddle decorations})
Let $T^k_D$ be a v-tangle diagram with a *-marker on the boundary and let $\mathrm{Cl}(T^k_D)$ be the closure of the diagram. The \textit{saddle decorations} of the saddles of $T^k_D$ should be the ones induced by the saddle decorations of the closure. To be more precise.
\begin{itemize}
\item[(a)] The \textit{signs} of the saddles of $T^k_D$ should be the same as the signs of the corresponding saddles of $\mathrm{Cl}(T^k_D)$ as defined in Definition~\ref{defn-sign}.
\item[(b)] The \textit{indicators} of the saddles should be obtained from the corresponding saddles of $\mathrm{Cl}(T^k_D)$ as follows.
\begin{itemize}
\item Every orientable surface should carry an indicator $+1$ iff the number of upper boundary components of the saddle is two and a $-1$ iff the number is one.
\item Every non-orientable saddle gets a $0$ as an indicator.
\end{itemize}
\item[(c)] The \textit{glueing numbers} of the saddles of $T^k_D$ should be the same as the glueing numbers of the corresponding saddles of $\mathrm{Cl}(T^k_D)$ as defined in Definition~\ref{defn-deco}.
\end{itemize}
Note that saddles with a $0$-indicator do not have any boundary decorations. Everything together, i.e. boundary decorations, the saddle sign and the indicator, is called the \textit{saddle decorations} of $S$.
\end{defn}
Beware again that many choices are involved. But they do not change the complex up to chain isomorphisms as we show in Lemma~\ref{lem-everythingfine} in an analogon of Lemma~\ref{lem-commutativeindependence}.
\begin{defn}\label{defn-geocomplex}(Topological complex for v-tangles) For a v-tangle diagram $T^k_D$ with a *-marker on the boundary and with $n$ ordered crossings we define \textit{the topological complex} $\bn{T^k_D}$ as follows.
\begin{itemize}
\item For $i\in\{0,\dots,n\}$ the $i-n_-$ \textit{chain module} is the formal direct sum of all resolutions $\gamma_a$ of length $i$.
\item There are only morphisms between the chain modules of length $i$ and $i+1$.
\item If two words $a,a^{\prime}$ differ only in exactly one letter and $a_r=0$ and $a^{\prime}_r=1$, then there is a morphism between $\gamma_a$ and $\gamma_{a^{\prime}}$. Otherwise all morphisms between components of length $i$ and $i+1$ are zero.
\item This morphism is a \textit{saddle} between $\gamma_a$ and $\gamma_{a^{\prime}}$.
\item The saddles should carry the \textit{saddle decorations} from Definition~\ref{defn-deco2}.
\end{itemize}
\end{defn}
We note again that it is not clear at this point why we can choose the numbering of the crossings, the numbering of the v-circles and the orientation of the resolutions of the closure. Furthermore, it is not clear why this complex is a well-defined chain complex. But we show in Lemma~\ref{lem-everythingfine} that the complex is independent of these choices, i.e. if $\bn{L_D}_1$ and $\bn{L_D}_2$ are well-defined chain complexes with different choices, then they are equal up to chain isomorphisms. The same lemma ensures that the complex is a well-defined chain complex.

Another point that is worth mentioning is that the signs in our construction, in contrast to the classical Khovanov homology, do not depend on the order of the crossings of the diagram.
\vskip0.5cm
Beware that the position of the *-marker is important for v-tangle diagrams. But Theorem~\ref{thm-marker} ensures that the position is not important for c-tangles and v-links.

If it does not matter which of the possible two different chain complexes is which, i.e. it is just important that they could be different, then we denote them by $\bn{T^k_D}^*$ and $\bn{T^k_D}_*$ for a given v-tangle diagram $T^k_D$ without a chosen *-marker position. 
 
For an example see Figure~\ref{figure-tanglecomplex}. This figure shows the virtual Khovanov complex of a v-tangle diagram with two different *-marker positions. The vertical arrow between them indicates that they are (in this case) chain isomorphic. It is worth noting at this point that, as we show in Theorem~\ref{thm-marker}, they are always isomorphic if the diagram is a c-tangle diagram (as the two diagrams in the figure below).
\begin{figure}[ht]
  \centering
     \includegraphics[scale=0.425]{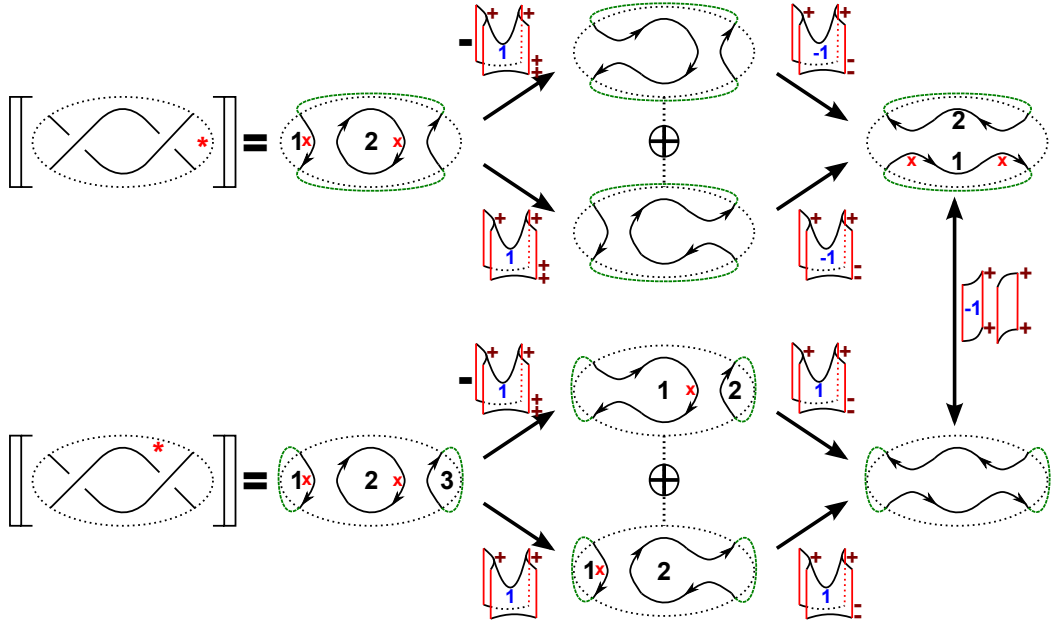}
  \caption{The complex of the same v-tangles with different *-marker positions. The two complexes are (in this case) isomorphic.}
  \label{figure-tanglecomplex}
\end{figure}
\begin{lem}\label{lem-everythingfine}
Let $T^k_D$ be a v-tangle diagram with a *-marker and let $\bn{T^k_D}_1$ be its topological complex from Definition~\ref{defn-geocomplex} with arbitrary orientations for the resolutions of the closure. Let $\bn{T^k_D}_2$ be the complex with the same orientations for the resolutions except for one circle $c$ in one resolution $\gamma_a$. If a face $F_1$ from $\bn{T^k_D}_1$ is anticommutative, then the corresponding face $F_2$ from $\bn{T^k_D}_2$ is also anticommutative.

Moreover, if $\bn{T^k_D}_1$ is a well-defined chain complex, then it is isomorphic to $\bn{T^k_D}_2$, which is also a well-defined chain complex.

The same statement is true if the difference between the two complexes is the numbering of the crossings, the choice of the x-marker for the calculation of the saddle signs or the fixed numbering of the v-circles of the closure. Moreover, the same is true for any rotations/isotopies of the v-tangle diagram.
\end{lem}
\begin{proof}
For v-tangle diagrams $T^k_D$ with $k=0$ the statement is the same as the corresponding statements in Lemma~\ref{lem-commutativeindependence} and Corollary~\ref{cor-chaincomplex}. Recall that the trick is to reduce all faces through a finite sequence of vRM1, vRM2, vRM3 and mRM moves in Figure~\ref{figure0-reide} and virtualisations from Figure~\ref{figure0-virt} to a finite number of different possible faces. Then one does a case-by-case check.

Because the saddles in the two chain complexes are topological the same, we only have to worry about the decorations. But the decorations are spread based on the closure of the v-tangle diagram and the relations from Definition~\ref{defn-category3} are build in such a way that the open cases behave as the closed ones.

Hence, we can use the statement for $k=0$ to finish the proof, since the only possible differences for $k>0$ are the indicators, but they only depend on the *-marker.
\end{proof}
In the same vein as in Section~\ref{sec-vkhcom} we obtain the following Corollary.
\begin{cor}\label{cor-chaincomplex2}
The complex $\bn{T^k_D}$ is a chain complex. Thus, it is an object in the category $\ukobk_R$.\qed 
\end{cor}
Hence, we can speak of \textit{the} topological complex $\bn{T^k_D}$ of the v-tangle diagram with a *-marker. The complex is by Corollary~\ref{cor-chaincomplex2} a well-defined chain complex.
\vskip0.5cm
The next theorem is very important but the proof itself is almost equal to the proof of Theorem~\ref{thm-geoinvarianz}. Therefore, we skip the details.
\begin{thm}\label{thm-invarianz}
Let $T^k_D,T'^k_D$ be two v-tangle diagrams with the same *-marker position which differ only through a finite sequence of isotopies and generalised Reidemeister moves. Then the complexes $\bn{T^k_D}$ and $\bn{T'^k_D}$ are equal in $\ukobk^{hl}_R$.
\end{thm}
\begin{proof}
We can copy the arguments of Theorem~\ref{thm-geoinvarianz}. The Lemma~\ref{lem-everythingfine} guarantees that we can choose the numbering and orientations without changing anything up to chain isomorphisms.

Beware that the chain homotopies in~\ref{thm-geoinvarianz} should all carry $+1$ as an indicator. Again, one can check that the involved chain homotopies satisfy the condition of a strong deformation retract.
\end{proof}
\subsubsection*{The *-marker and the classical complex}
We need some notions now. Note that they seem to be ad-hoc, but the main motivation is that in general the position of the *-marker is important. But to recover at least some local properties, as discussed in Section~\ref{sec-vkhca}, we need to identify basic parts of v-tangle diagrams such that the two complexes are isomorphic.
\vskip0.5cm
Let $T^k_D$ denote a v-tangle diagram. We call a part of $T^k_D$ a \textit{connected part} if it is connected as the four-valent graph by ignoring the v-crossings. We call a connected part of a v-tangle diagram \textit{fully internal} if it is not adjacent to the boundary. See Figure~\ref{figure-internal}. The left v-tangle diagram has one connected part, which is not fully internal, and the right v-tangle diagram has two connected parts, one fully internal and one not fully internal.
\begin{figure}[ht]
  \centering
     \includegraphics[scale=0.56]{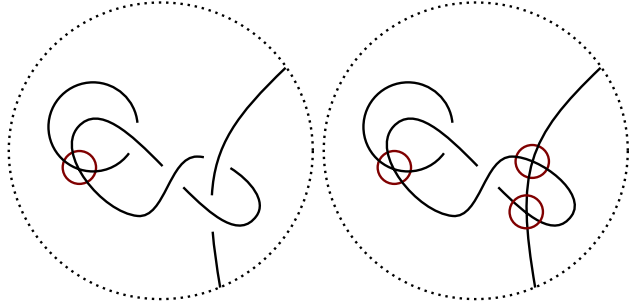}
  \caption{The left v-tangle diagram is not fully internal, but the right diagram has a fully internal component (the two internal v-circles).}
  \label{figure-internal}
\end{figure}

A v-crossing is called \textit{negligible} if it is part of a fully internal component, e.g. all v-crossings of the right v-tangle diagram in Figure~\ref{figure-internal} are negligible. Note that, by convention, negligible v-crossings are never part (for all resolutions) of any string that touches the boundary.

We call a v-tangle diagram $T^k_D$ \textit{nice} if there is a finite sequence of vRM1, vRM2, vRM3 and mRM moves and virtualisations such that every v-crossing is negligible, e.g. every v-link diagram is nice and every c-tangle diagram is nice.
\vskip0.1cm
An example of a not nice v-tangle diagram is shown in Figure~\ref{figure-notnice}. Note that the complexes are not chain homotopic.
\begin{figure}[ht]
  \centering
     \includegraphics[scale=0.5]{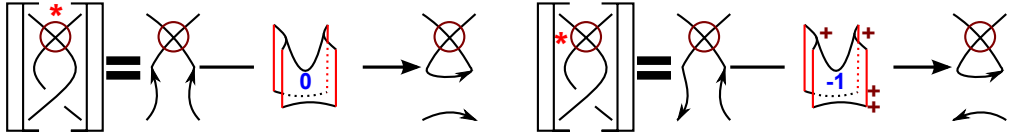}
  \caption{A counterexample. The diagram is not a nice v-tangle diagram.}
  \label{figure-notnice}
\end{figure}

We note that for a v-tangle diagram $T^k_D$ the chain complexes $\bn{T^k_D}^*$ and $\bn{T^k_D}_*$ are ``almost'' the same, i.e. they have the same vertices, but possible different edges (which are still in the same positions). The next lemma makes the observation precise.

It is worth noting that the Khovanov cube of a v-tangle diagram with $n$ crossings has $2^{n-1}n$ saddles. We number these saddles and the numbering in the lemma below should be the same for the two complexes.
\begin{lem}\label{lem-factorbyisos}
Let $T^k_D$ be a v-tangle diagram with $n$ crossings. Let ${}_iS(\mathrm{in})^*$ and ${}_iS(\mathrm{in})_*$ denote the numbered saddles of $\bn{T^k_D}^*$ and of $\bn{T^k_D}_*$. If $T^k_D$ is a nice v-tangle diagram, then we have for all $i=1,\dots,2^{n-1}n$ a factorisation of the form ${}_iS(\mathrm{in})^*=\alpha\circ {}_iS(\mathrm{in})_*\circ\beta$ for two invertible cobordisms $\alpha,\beta$.
\end{lem}
\begin{proof}
It is clear that the saddles are topological equivalent. So we only need to consider the decorations. The main point is the following observation. Of the four outer (two on both sides) cobordisms in the bottom row of Figure~\ref{figure1-3}, i.e. $\mathrm{id}(1)^+_+$, $\Phi(1)^-_+$, $\mathrm{id}(0)$ and $\mathrm{id}(-1)^+_+$, only the third is not invertible. The first is the identity, the second and fourth are their own inverses. The third is not invertible because the $0$-indicator can not be changed to a $\pm 1$-indicator.

Note that neither the vRM1, vRM2, vRM3 and mRM moves nor a virtualisation change the indicator of a saddle cobordism. Hence, it is sufficient to show the statement for a v-tangle diagram with only negligible v-crossings. From the observation above it is enough to show that every saddle gets a $0$-indicator in one closure iff it gets a $0$-indicator in the other closure.

The only possible way that a saddle gets an indicator from $\{+1,-1\}$ for one closure and a $0$-indicator for the other closure is the rightmost case in Figure~\ref{figure0-order}. But for this case the existence of a non-negligible v-crossing is necessary. Hence, we get the statement.    
\end{proof}
\begin{prop}\label{prop-nicechainiso}
Let $T^k_D$ be a v-tangle diagram. If $T^k_D$ is nice, then $\bn{T^k_D}^*$ and $\bn{T^k_D}_*$ are chain isomorphic.
\end{prop}
\begin{proof}
Let $T^k_D$ be a nice v-tangle diagram. Then Lemma~\ref{lem-factorbyisos} ensures that every saddle is the same, up to isomorphisms, in $\bn{T^k_D}^*$ and $\bn{T^k_D}_*$. Furthermore, Lemma~\ref{lem-everythingfine} ensures that both are well-defined chain complexes. Hence, the number of signs of every face is odd (also counting the ones from the decorations).

Thus, we can use a spanning tree argument to construct the chain isomorphism explicit, i.e. start at the rightmost leafs of a spanning tree of the Khovanov cube and change the orientations of the resolutions at the corresponding vertices such that the unique outgoing edges of the tree has the same decorations in both cases (Lemma~\ref{lem-everythingfine} ensures that nothing changes modulo chain isomorphisms). Continue along the vertices of the spanning tree, but remove already visited leafs. This construction generates a chain isomorphism.

Next repeat the whole process, but change the indicators and afterwards the signs. It is worth noting that Lemma~\ref{lem-everythingfine} ensures that the two processes will never run into ambiguities or problems and Lemma~\ref{lem-factorbyisos} ensures that they will generate chain isomorphisms.

The chain isomorphism that we need is the composition of the three isomorphisms constructed before. See for example Figure~\ref{figure-tanglecomplex}. 
\end{proof}
\begin{thm}\label{thm-marker}(\textbf{Two different chain complexes}) Let $T^k_D$ be a v-tangle diagram with two different *-marker positions. Let $\bn{T^k_D}^*$ and $\bn{T^k_D}_*$ be the topological complex from Definition~\ref{defn-geocomplex} for the two positions. Then the two complexes are equal in $\ukobk^{hl}_R$ if the v-tangle has $k=0$ or is a c-tangle.
\end{thm}
\begin{proof}
We can use the Proposition~\ref{prop-nicechainiso} above for a v-tangle diagram with $k=0$. Moreover, we can choose a diagram without virtual crossing for a c-tangle without changing anything up to chain homotopies, because of Theorem~\ref{thm-geoinvarianz}. Then we can use the Proposition~\ref{prop-nicechainiso} again.
\end{proof}
\begin{rem}\label{rem-disums}
Note that the whole construction can be done with an \textit{arbitrary} closure of a v-tangle diagram, i.e. cap of in any possible way without creating new c- or v-crossings. The direct sum of all possibilities is then a v-tangle invariant. Or one can even allow v-crossings and take direct sums over all possibilities again. But since both is inconvenient for our purpose, we do not discuss it in detail here. 
\end{rem}
\begin{rem}\label{rem-gradings2}
Again, we could use the Euler characteristic to introduce the structure of a grading on $\ucob_R(k)$ (and hence on $\ukobk_R$). The differentials in the topological complex from Definition~\ref{defn-geocomplex} have all $\deg=0$ (after a grade shift), because their Euler characteristic is $-1$. Then it is easy to prove that the topological complex is a v-tangle invariant under graded homotopy.
\end{rem}
\subsection{Circuit algebras}\label{sec-vkhca}
In the present section we describe the notion of a \textit{circuit algebra}. A circuit algebra is almost the same as a planar algebra, but we allow virtual crossings.

Planar algebras were introduced by Jones~\cite{vj} to study subfactors. In our setting, they were for example studied by Bar-Natan in the case of classical Khovanov homology~\cite{bn2}. Hence, we can use most of his constructions in our context, too. A crucial difference is that we need to \textit{decorate} our \textit{circuit diagrams}. This is necessary because our cobordisms are also decorated.
\vskip0.25cm
We start the section with the definition of a (decorated) circuit diagram. In the whole section every v-tangle diagram should have a *-marker. We call a v-tangle diagram \textit{decorated} if it has an orientation, a number (same numbers are allowed), one coloured (green and red) dot  for each of its v-circle/v-string and we call a cobordisms \textit{decorated} if it has gluing numbers and an indicator. In the following we use the notion $\omega^*$ to illustrate that we consider all possibilities for $k\in\bN$ together.

\begin{defn}\label{defn-circuit diagram}
Let $D^2_o$ denote a disk embedded into $\bR^2$, the so-called \textit{outside disk}. Let $\mathrm{I}_k$ denote disks $D^2$ embedded into $\bR^2$ such that for all $k\in\{0,\dots,m-1\}$ the disk $\mathrm{I}_k$ is also embedded into $D^2_o$ without touching the boundary of $D^2_o$, i.e. $\mathrm{I}_k\subset D^2_o\subset \bR^2$, $\mathrm{I}_k\cap\mathrm{I}_{k^{\prime}}=\emptyset$ for $k\neq k^{\prime}$ and $\mathrm{I}_k\cap\partial D^2_o=\emptyset$. We denote $\mathcal D_m=D^2_o-(\mathrm{I}_0\cup\cdots\cup\mathrm{I}_{m-1})$. These $\mathrm{I}_k$ are called \textit{input disks}.

A \textit{circuit diagram with $m$ input disks} $\mathcal{CD}_m$ is a planar graph embedded into $\mathcal D_m$ with only vertices of valency one and four in such a way that every vertex of valency one is in $\partial\mathcal D_m$ and every vertex of valency four is in $\mathrm{Int}(\mathcal D_m)$. All vertices of valency four are marked with a v-crossing. Again we allow circles, i.e. closed edges without any vertices. A \textit{*-marked circuit diagram} is the same, but with $m+1$ extra *-marker for every boundary component of $\mathcal D_m$. Moreover, we call the vertices at $\partial D^2_o$ the \textit{outer boundary points}.

See for example Figure~\ref{figure-deccircuit}, i.e. the figure shows a *-marked (decorated) circuit diagram with three input disks.
\begin{figure}[ht]
  \centering
     \includegraphics[scale=0.45]{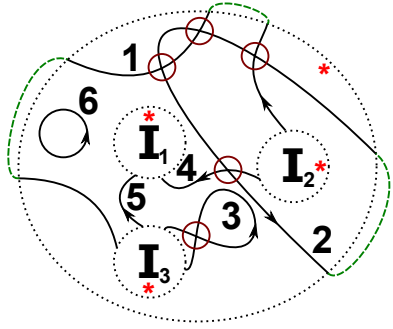}
  \caption{A decorated circuit diagram with three input disks.}
  \label{figure-deccircuit}
\end{figure}
\end{defn}

A \textit{closure} of a *-marked circuit diagram with $m$ input disks $\mathrm{Cl}(\mathcal{CD}_m)$ is a circuit diagram with $m$ input disks and without any outer boundary points which is constructed from $\mathcal{CD}_m$ by capping of neighbouring strings starting from the \textit{outer} *-marker and proceeding counterclockwise. Note that we only cap of the outside disk and not the small inside disks.

A \textit{decoration} for a *-marked circuit diagram is a tuple of a numbering and an orientation of the strings of the diagram in such a way that its also a numbering and orientation of the closure. We call a circuit diagram together with a decoration a \textit{decorated circuit diagram}. See for example Figure~\ref{figure-deccircuit}. The decoration of the circuit diagram in this figure is also a decoration for the closure (the diagram together with the green lines).
\vskip0.5cm
We can realise the definition of a \textit{(decorated) circuit algebra} with these notions. Recall that our v-tangle diagrams should always be oriented with the usual orientations but we suppress these again to maintain readability.
\begin{defn}\label{defn-circuit algebra}(\textbf{Circuit algebra}) Let $\mathrm{T^{\prime}}(k)$ be the set of (decorated) v-tangle diagrams with $k$ boundary points and a *-marker and let $\mathrm{T}(k)$ denote the quotient by boundary preserving isotopies and generalised Reidemeister moves.

Furthermore, let $\mathcal{CD}_m$ denote a (decorated) circuit diagram with $m$ input disks and $k^{\prime}$ outer boundary points in such a way that the $j$-th input disk has $k_j$ numbered boundary points. 

Because $\mathcal{CD}_m$ has no c-crossings, this induces operations
\[
\mathcal{CD}_m\colon \mathrm{T^{\prime}}(k_0)\times\cdots\times \mathrm{T^{\prime}}(k_{m-1})\to \mathrm{T^{\prime}}(k^{\prime})\text{ and }\mathcal{CD}_m\colon \mathrm{T}(k_0)\times\cdots\times \mathrm{T}(k_{m-1})\to \mathrm{T}(k^{\prime})
\]
by placing the $i$-th v-tangle diagram from $\mathrm{T^{(\prime)}}(k_i)$ in the $i$-th boundary component of $\mathcal{CD}_m$, i.e. glue the v-tangle inside in such a way that the *-markers match. See the right side of Figure~\ref{figure-tanglealg}. 

There is an identity operation on $\mathrm{T^{(\prime)}}(k)$ (it is of the form $\jpg{5mm}{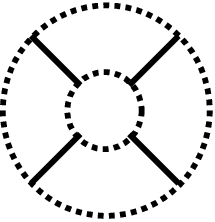}$) and the operations are compatible in a natural way (``associative''). We call a set of sets $\mathfrak C(\omega^*)$ with operations $\mathcal{CD}_m$ as above a \textit{circuit algebra}, provided that the identity and associativity from above hold.

If the operators and elements are decorated, first with numbers and orientations and latter with any kind of suitable decorations, then we call a set of sets $\mathfrak C(\omega^*)$ as before a \textit{decorated circuit algebra}. Note that in this case we have to define how the decorations change after glueing, e.g. we can run into ambiguities.
\end{defn}
We should note that Figure~\ref{figure-tanglealg} below also illustrates how a v-tangle diagram induces a decorated circuit diagram (choices for the decorations are involved).
\begin{figure}[ht]
  \centering
     \includegraphics[scale=0.45]{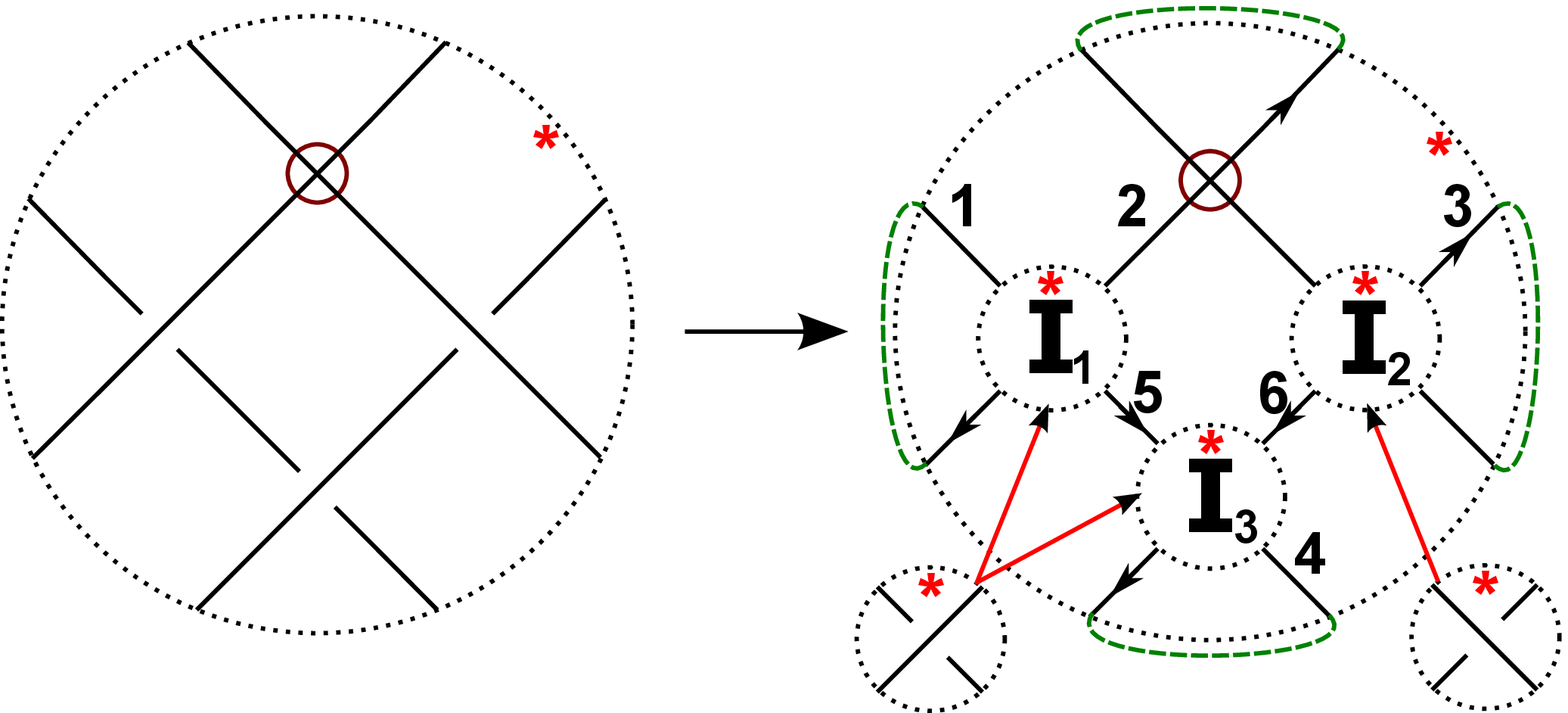}
  \caption{A decorated circuit diagram induced by a v-tangle.}
  \label{figure-tanglealg}
\end{figure}
\vskip0.15cm
Here are some examples. The reader may also check the corresponding section in~\cite{bn2}.
\begin{ex}\label{ex-circuit}
The first example is the set $\Ob(\ucob(\omega^{*}))$ from Definition~\ref{defn-category}, i.e. v-tangles diagrams with $k\in\bN$ boundary points, an extra *-marker, but without c-crossings. This is a sub-circuit algebra of the circuit algebra that allows c-crossings.

But we want to view it as a decorated circuit algebra, denoted by $\Ob_d(\ucob(\omega^{*}))$, i.e. the elements are \textit{decorated} v-tangle diagram (all possible decorations). We have to define the operations in more detail now, since we can run into ambiguities, see top row of Figure~\ref{figure-gluedeco}.
\begin{figure}[ht]
  \centering
     \includegraphics[scale=0.6]{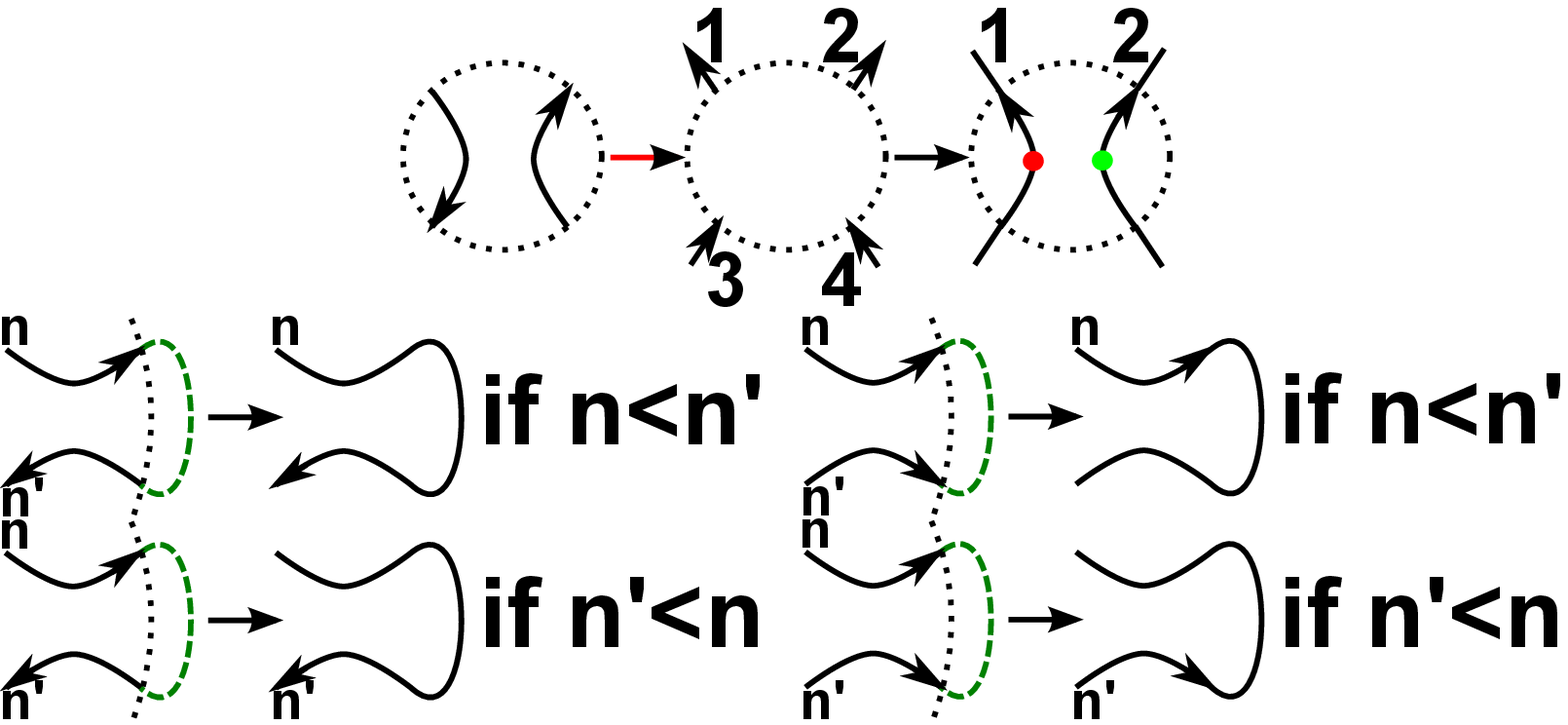}
  \caption{The operation in the decorated circuit algebra.}
  \label{figure-gluedeco}
\end{figure}

First, we can run into ambiguities if the decorations of the operator(s) that are glued together do not match. In this case we define the new decoration based on the rule \textit{``lower first''}, i.e. the new number is the lower and the new orientation is the one from the lower numbered string. See the two lower rows of Figure~\ref{figure-gluedeco}. Not all four cases are pictured, but we hope that it should be clear how the other two work. Moreover, it is worth noting that the order of these local steps does not affect the end result, since, by construction, the lowest number of all strings that are connected and its orientation will always determine the output. 

Furthermore, if we glue a decorated v-tangle diagram in an input disk, then we run into ambiguities if the shared decorations, i.e. the orientations and numbers, do not match. In this case we change the decorations of the v-tangle diagram (as above). We add in a \textit{red dot $r$} if we have to change the orientation and a \textit{green dot $g$} otherwise. This is pictured in the top row of Figure~\ref{figure-gluedeco}. As above, in order to make it well-defined, one has to allow the dots to change stepwise, i.e. one uses the ``multiplication'' rules $g\cdot g=g=r\cdot r$ and $g\cdot r=r=r\cdot g$ for dots on the same string.

The reader should check that this gives rise to a (well-defined!) decorated circuit algebra.

Another important example is the whole collection $\Mor_d(\ucob(\omega^{*}))$ from Definition~\ref{defn-category}, i.e. decorated cobordisms (all possible decorations) with $k\in\bN$ vertical boundary lines and an extra *-marker. We want to view this example as a decorated circuit algebra again.

Hence, we have to define the operations. The most important point is the question how to handle the decorations again, because it should be clear how to glue a cobordism with $m$ vertical boundary lines into $\mathcal{CD}_m\times [-1,1]$. This time we have to define the behaviour of two decorations, i.e. the glueing numbers and indicators. The glueing numbers are treated as the orientations before, i.e. use the ``lower-first'' rule. The indicators (recall that they are just numbers of $\{0,+1,-1\}$) are multiplied. Recall that a cobordism with a $0$-indicator does not get any glueing numbers. We simply remove them in this case. To be more precise, we note make the following definition of the operation of $\mathcal{CD}_m$ on $\Ob_d(\ucob(\omega^{*}))$ and $\Mor_d(\ucob(\omega^{*}))$ (compare to Figure~\ref{figure-gluedeco}).
\begin{itemize}
\item A cobordism with a $+$ glueing number (or $-$) is composed with $\Phi^-_+$ iff the decorated v-tangle diagram (short: diagram) gets a red dot (or green) at the corresponding position.
\item A cobordism is composed with a $0$-indicator surface iff the strings of the diagram get identified at the bottom and top resolution at the corresponding position.
\item A cobordism with a $1$-indicator is composed with a $1$/$-1$-indicator surface iff the strings of the diagram get identified at the bottom/top resolution at the corresponding position.
\item A cobordism with a $-1$-indicator is composed with a $1$/$-1$-indicator surface iff the strings of the diagram get identified at the top/bottom resolution at the corresponding position.
\end{itemize} 
These rules define a new decoration for the new cobordism. The reader should check again that this gives rise to a (well-defined!) decorated circuit algebra (to see this we note that everything behaves multiplicative as the elements of $\{+1,-1\}\cong\bZ/2\bZ$ or has a $0$-indicator).
\end{ex}
\vskip0.1cm
We summarise the notions in a definition. Recall that v-tangle diagrams are decorated with orientations, numbers and coloured dots and cobordisms have glueing numbers and an indicator.
\begin{defn}\label{defn-dotcalculus}(\textbf{Dot-calculus}) Let $\mathcal{CD}_m$ denote a decorated circuit diagram with $m$ input disks and $k'$ outer boundary points in such a way that the $j$-th input disk has $k_j$ numbered boundary points. Then $\mathcal{CD}_m$ induces an associative and unital (as above) operation on decorated v-tangle diagrams (with a corresponding number of boundary points) by the ``lower first''-rule, i.e. if the orientation does not match, then the lower number induces the new orientation. Put a red dot $r$ on every string that has its orientation changed and a green dot $g$ otherwise (two dots on the same v-string are multiplied by the convention $g=1,r=-1$). We call this the \textit{v-tangle dot-calculus}.
\vskip0.15cm
Moreover, $\mathcal{CD}_m$ induces an associative and unital (as above) operation on decorated cobordisms (with a corresponding number of boundary lines) by the ``lower first''-rule, i.e. if the orientation does not match, then the lower number induces the new orientation. Put a red dot on every string that has its orientation changed and a green dot otherwise and compose the corresponding boundary with a $+$ glueing number (or $-$) with $\Phi^-_+$ iff the string has a red dot (or a green dot), multiply indicators via identity surfaces with corresponding indicators $+1$/$-1$ iff the v-tangle numbers get identified at the bottom/top (or vice versa for surfaces with a $-1$-indicator) resolution at the corresponding position, multiply with an $0$-identity iff in both resolutions the strings are identified (do everything repeatedly using the rules as explained above). We call this the \textit{dot-calculus}. The reader should compare the notions above with Figure~\ref{figure0-main}.
\end{defn}
\vskip0.15cm
Recall that we assume that all v-tangle diagrams have already a fixed *-marker. A v-tangle diagram $T^k_D$ gives rise to a decorated circuit diagram $\mathcal{CD}_{T^k_D}$ as already illustrated in Figure~\ref{figure-tanglealg}. If the diagram has $m$ crossings, denoted by $cr_1,\dots,cr_m$, then we choose a neighbourhood of the $cr_i$ without any other crossings and a *-marker for all $cr_i$. We obtain by this procedure $m$ v-tangle diagrams with one crossings and four boundary components, denoted by an abuse of notation by $cr_1,\dots,cr_m$, and we call these crossing diagrams \textit{associated} to $T^k_D$.
\begin{defn}\label{defn-assocomp}
Let $T^k_D$ be a v-tangle diagram with $m$ crossings and let $\mathcal{CD}_{T^k_D}$ and $cr_1,\dots,cr_m$ be its associated decorated diagram and crossings. Then the \textit{tensored complex}
\[
\mathcal{CD}_{T^k_D}(cr_1,\dots,cr_m)=(C_*,c_*)
\]
is defined as follows. Let $(C_j,c_j)$ with $j\in\{1,\dots,m\}$ be the topological complex of the $cr_j$ and defined in Definition~\ref{defn-geocomplex} such that the unique saddle is of the form $c_j\colon\smoothing\to\hsmoothing$ for any suitable orientation (without a sign). Let $\alpha_i,\beta_i$ denote the compositions of the morphisms that we compose after applying the circuit diagram on cobordisms (see Example~\ref{ex-circuit} and Definition~\ref{defn-dotcalculus} above), i.e. the red dots induce a composition with $\Phi^-_+$ (or with a $0$-identity surface in the degenerated case) and a change in the numbering induces a composition with a cobordisms that changes indicators.

Therefore, we denote the operation of $\mathcal{CD}_{T^k_D}$ on cobordisms, i.e. the dot-calculus, by $\alpha\circ\mathcal{CD}_{T^k_D}\circ\beta$ to illustrate the difference to the classical case. We skip this notion for the objects to maintain readability. The $i$-th chain module is
\[
C^i=\bigoplus_{i=j_0+\dots+j_{m-1}}\mathcal{CD}_{T^k_D}(C_0^{j_0},\dots,C_{m-1}^{j_{m-1}})
\]
and the differentials are
\[
c|_{\mathcal{CD}_{T^k_D}(C_0^{j_0},\dots,C_{m-1}^{j_{m-1}})}=\sum_{i=0}^{m-1}\alpha\circ\mathcal{CD}_{T^k_D}(\mathrm{Id}_{C_0^{j_0}},\dots,c_i,\dots,\mathrm{Id}_{C_{m-1}^{j_{m-1}}})\circ\beta.
\]
Note that this complex does not have any extra signs and will in general not be a chain complex.
\end{defn}
\vskip0.15cm
We call a Khovanov cube of type p, if all its faces are commutative up to a unit of $R$, and a projectivisation of such a cube is given by identifying morphisms up to units. We denote the latter usually with a superscript $P$. Details are in Section~\ref{sec-techcube}.

It should be noted that the choice of the *-markers in the definition of $\mathcal{CD}_{T^k_D}(cr_1,\dots,cr_m)$ or the choice of the decorations for $\mathcal{CD}_{T^k_D}$ is not important for our purpose (and we will suppress the difference). To be more precise, we give the following lemma. We should note that it is not clear at this point why the complexes are $m$ cubes of type p. But we show it in Theorem~\ref{thm-pcomplex} below.
\begin{lem}\label{lem-pcomplex}
Let $\mathcal{CD}_{T^k_D}(cr_1,\dots,cr_m)$ and $\mathcal{CD}^{\prime}_{T^k_D}(cr_1,\dots,cr_m)$ denote two different choices for the *-markers of the $cr_j$. Then the two complexes are equal.

Moreover, if the difference between $\mathcal{CD}_{T^k_D}(cr_1,\dots,cr_m)$ and $\mathcal{CD}^{\prime}_{T^k_D}(cr_1,\dots,cr_m)$ is the choice of decorations, either for the circuit diagram or for the saddles of the complexes of $cr_j$, then the two complexes are isomorphic as $m$ cubes of type p.
\end{lem}
\begin{proof}
This is the case because the $c_j$ has always an indicator $+1,-1$ in the definition of the complex $(C_j,c_j)$ and never a $0$-indicator. Moreover, the result depends only on the position of the *-marker for $T^k_D$, since the involved operations only depend how strings are connected.

The second statement can be verified analogously as in Lemma~\ref{lem-everythingfine}.
\end{proof}
By a slight abuse of notation, we denote the topological complex by $\bn{T^k_D}$, although some choices are involved (but they do ``not matter'', see Lemma~\ref{lem-everythingfine}).
\begin{thm}\label{thm-pcomplex}(\textbf{Semi-locality I}) Let $T^k_D$ be a v-tangle diagram with $m$ crossings. Let $\bn{T^k_D}$ be (one of) its topological complex(es) from Definition~\ref{defn-geocomplex} and let $\mathcal{CD}_{T^k_D}(cr_1,\dots,cr_m)$ be its tensored complex from Definition~\ref{defn-assocomp}. Then $\mathcal{CD}_{T^k_D}(cr_1,\dots,cr_m)$ is a $m$-cube of type p and
\[
\mathcal{CD}_{T^k_D}(cr_1,\dots,cr_m)=\bn{T^k_D}^P
\]
for a suitable choice of orientations for the resolutions of $\bn{T^k_D}$.
\end{thm}
\begin{proof}
This is true because the dot-calculus is exactly build in such a way that the resulting saddles have some glueing numbers induced by a suitable choice of orientations of the resolutions. To be more precise, it is clear that the construction from Definition~\ref{defn-assocomp} gives rise to a $m$-cube as explained in Section~\ref{sec-techcube}.

Moreover, since we do not spread any formal signs in the construction from Definition~\ref{defn-assocomp}, the only thing we can expect is that the corresponding cube will be of type p, i.e. faces commute up to a sign. So we only have to care that the glueing numbers and indicators work out as claimed.

That the glueing numbers work out follows from the definition of the dot-calculus, since the orientation of the lowest numbered string will always determine the result and the decorations of the circuit diagram are also decorations of the closure, i.e. we can use Theorem~\ref{theo-facescommute} to see that the glueing numbers work out as claimed (up to a formal sign).

Moreover, the indicators of the saddles are spread based on a topological information, namely how certain strings are connected in the closure of the diagram $T^k_D$. Hence, since we have fixed the *-marker positions, these indicators are the same for $\mathcal{CD}_{T^k_D}(cr_1,\dots,cr_m)$ and any of the $\bn{T^k_D}$. Note that it is important that the indicators at the beginning are all $+1,-1$, since we can not change a $0$ using the conventions above.

This proves the statement, since there is a choice of orientations of the resolutions such that all saddles of $\mathcal{CD}_{T^k_D}(cr_1,\dots,cr_m)$ and $\bn{T^k_D}$ are equal up to a sign. 
\end{proof}
Given a Khovanov cube, then an \textit{edges assignment (with signs)} of this cube is a choice of extra signs for some of the saddles. We denote such an assignment using $\epsilon$ as a superscript, see Definition~\ref{defn-edge}.
\begin{cor}\label{cor-pcomplex}
There is an edge assignment such that $\mathcal{CD}^{\epsilon}_{T^k_D}(cr_1,\dots,cr_m)$ is a chain complex. Moreover, there is a chain isomorphism between $\mathcal{CD}^{\epsilon}_{T^k_D}(cr_1,\dots,cr_m)$ and $\bn{T^k_D}$ (for all possible choices involved in the definition of latter).  
\end{cor}
\begin{proof}
The first statement follows from Theorem~\ref{thm-pcomplex}, Theorem~\ref{theo-facescommute} and Lemma~\ref{lem-edge}. The second from Lemma~\ref{lem-everythingfine}.
\end{proof}
We note that Theorem~\ref{thm-pcomplex} and Corollary~\ref{cor-pcomplex} allows us to be ``sloppy'' when it comes to signs.
\vskip0.5cm
It is a natural question if one can generalise the statement of Theorem~\ref{thm-pcomplex}, since in the classical case one can allow arbitrary c-tangle diagrams as inputs. In fact, we do not know the answer in general. The main problem is that ``non-orientablity'' is not a local property.

We can make an analogously definition as in Definition~\ref{defn-assocomp}, but we allow the $cr_1,\dots,cr_m$ to be \textit{not nice} v-tangle diagrams with one crossing. We denote them by $cr^{\prime}_1,\dots,cr^{\prime}_m$ to illustrate the difference and we call the corresponding complex \textit{generalised tensored complex}. An example is shown in Figure~\ref{figure-counter}. Even this slight generalisation has unsatisfying properties.
\begin{thm}\label{thm-semiloc}(\textbf{Semi-locality II}) Let $T^k_D$ be a v-tangle diagram with $m$ crossings. And let
\[
\mathcal{CD}_{T^k_D}(cr^{\prime}_1,\dots,cr^{\prime}_m)=(C_*,c_*)
\]
be its generalised tensored complex.
\begin{itemize}
\item[(a)] The complex $(C_*,c_*)$ is a complex of type p, i.e. faces commute up to a unit of $R$.
\item[(b)] Let $\bn{T^k_D}$ denote (one of) its topological complex. Then we do not have a suitable choice for $\bn{T^k_D}$ in general such that
\[
\mathcal{CD}_{T^k_D}(cr^{\prime}_1,\dots,cr^{\prime}_m)=\bn{T^k_D}^P.
\]
\item[(c)] The complexes $(C_*,c_*)$ and $\bn{T^k_D}$ are not p-homotopic (see Definition~\ref{defn-chainhomotopy2}) in general.
\end{itemize}
\end{thm}
\begin{proof}
(a) This statement can be verified analogously to Theorem~\ref{thm-pcomplex}, since, if the corresponding saddles have an $+1,-1$ indicator, as in Theorem~\ref{thm-pcomplex}, then one can copy the arguments from before. If it has a $0$-indicator, then the arguments are even easier to verify, since we do not need any decorations in this case.
\vskip0.5cm
(b)+(c) This is true, because a surfaces with a $0$-indicator can not be changed to a surfaces with a $\pm 1$-indicator, since indicators behave multiplicatively. For an explicit example see Figure~\ref{figure-counter}, i.e. the two complexes are not p-homotopy equivalent, since we can not change the $0$-indicator.
\begin{figure}[ht]
  \centering
     \includegraphics[scale=0.41]{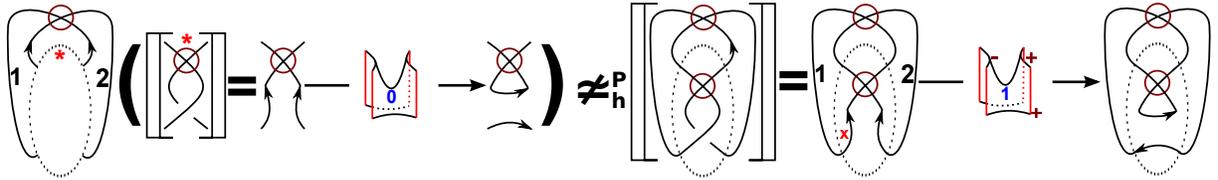}
  \caption{A counterexample. The diagram is not a nice v-tangle diagram.}
  \label{figure-counter}
\end{figure}

Note that this includes that no choice will make them equal as complexes of type p.
\end{proof}
\vskip0.15cm
It should be noted again that the whole discussion in this section could be done with oriented (in the usual sense) v-tangle diagrams and oriented (decorated) circuit algebras. But to maintain readability we only refer to~\cite{bn2}, i.e. the reader can adopt the notions there and use them in our context without any difficulties.
\begin{rem}\label{rem-mod2}
It should be noted that the constructions presented in this section can be extended relatively easily to work in an even better way if one works over rings of characteristic $2$, e.g. over the ring $R=\bZ/2\bZ$.

This is the case because all appearing problems are in some sense ``sign problems''. If one works over $R=\bZ/2\bZ$, then, for example, the indicators are not necessary and most constructions will work analogously to the classical case (see~\cite{bn2}).
\end{rem}
\begin{rem}\label{rem-fastcalc}
One application of the local construction in the classical case is a way to calculated the classical Khovanov homology of a c-link with $n$ crossings in approximately $2^{\sqrt{n}}$ instead of $2^n$ of the ``brute force method'', see~\cite{bn3}. In the view of Theorem~\ref{thm-semiloc}, one has to be very careful if one tries to copy the method given in~\cite{bn3}. Another idea is missing. 
\end{rem}
\subsection{An application: Degeneration of Lee's variant}\label{sec-vkhapp}
This section splits into three subsections. We explain the main motivations in the first and we are going to show that some facts about the classical Khovanov-Lee link homology are still true in the context of v-links (e.g. see Theorem~\ref{thm-leedeg}) in the last subsection. In order to do so, we identify the two generators with so-called non-alternating resolution~\ref{thm-nonalternating} in the second subsection. We note that these correspond to colourings in the c-case.
\vskip0.5cm
The approach (we follow~\cite{bnsm}) to show that the degeneration is still true is the following. First we define two orthogonal idempotents in our category, which we call \textit{down and up}. Then we can go to the \textit{Karoubi envelope} of our category, denoted by $\KAR(\ukobk_R)$.

The idea of the Karoubi envelope is to find a ``completion'' of a category such that every idempotent splits. It is named after the french mathematician Karoubi, but it already appears in an earlier work by Freyd~\cite{pf}. Note that it is sometimes called \textit{idempotent completion}. Then we show that the topological complex of a simple crossing (as a v-tangle), if considered in $\KAR(\ukob_R(k))$, is homotopy equivalent to a very simple complex with only $0$-morphisms. After that we use the semi-local constructions from Section~\ref{sec-vkhca} to finish the proof.

In the whole section let $R$ denote a commutative and unital ring such that $2$ is invertible, e.g. $R=\bZ\left[\frac{1}{2}\right]$. Moreover, throughout the whole section, we denote the topological complex by $\bn{\cdot}$ and its algebraic version by $\mathcal F(\bn{\cdot})$ or short by $\mathcal F(\cdot)$, e.g. we denote Lee's version by
\[
\mathcal F_{\mathrm{Lee}}(\cdot)=\mathcal F(\bn{\cdot}_{\mathrm{Lee}}).
\]
Note that, in order for the signs to work out correct, we have to fix x-marker positions. In the whole section we, by convention, say that the x-marker for $\slashoverback$ is at the left side and for $\backoverslash$ is on the top.

Moreover, recall that the topological picture of Lee's variant is given by the \textit{dot-relations} in Figure~\ref{figure-dotrel} with $t=1$, while the graded case of the Khovanov complex is $t=0$. Recall (see~\cite{bn1}) that $\frac{1}{2}\in R$ allows us to use the \textit{dot-relation} in Figure~\ref{figure-dotrel} instead of the local relations of Figure~\ref{figureintroa-4}. We give an example of the Lee complex of a v-knot in Example~\ref{ex-leedeg}.
\begin{figure}[ht]
  \centering
     \includegraphics[width=0.55\linewidth]{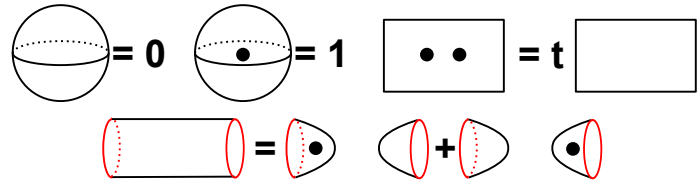}
  \caption{The dot-relations. A dot is a short hand notation for $\frac{1}{2}$-times a handle.}
  \label{figure-dotrel}
\end{figure}

\subsubsection*{Main observations}
Recall (see~\cite{lee} for the classical and~\ref{cor-khovhomo} for the virtual case) that \textit{Lee's variant} for v-links is given by the filtered algebra $A=A_{\mathrm{Lee}}=R[X]/(X^2=1)$ and the following maps.
\[
m^{++}_+\colon A\otimes A\to A,\;\begin{cases}1\otimes 1\phantom{.}\mapsto 1, & X\otimes X\mapsto 1, \\ 1\otimes X\mapsto X, & X\otimes 1\phantom{.}\mapsto X\end{cases}
\]
for the multiplication and
\[
\Delta_{++}^+\colon A\to A\otimes A,\;\begin{cases}1\phantom{.}\mapsto 1\otimes X+X\otimes 1,\\ X\mapsto 1\otimes 1+X\otimes X\end{cases},\;\theta\colon A\to A,\; \begin{cases}1\phantom{.}\mapsto 0,\\ X\mapsto 0\end{cases}
\]
for the comultiplication and $\cdot\theta$. Furthermore, the very important map $\Phi^-_+$ given by
\[
\Phi^-_+\colon A\to A,\;\begin{cases}1\phantom{.}\mapsto 1,\\ X\mapsto -X.\end{cases}
\]

Lee's variant has a remarkable property in the classical case, i.e. Lee showed that her variant just ``counts'' the number of components of the c-link, i.e. she showed that (for $R=\bQ$) the homology of a $n$-component link $L$ is
\[
H(\mathcal F_{\mathrm{Lee}}(L))\cong\bigoplus_{2^n}\bQ.
\]
So on the first hand this seems to be a ``boring'' invariant. But Rasmussen~\cite{ra} used this degeneration in a masterfully way to define the \textit{Rasmussen invariant} of a c-knot (as in Section~\ref{sec-intro}).

Therefore, a natural question is if this degeneration of Lee's variant is still true for v-links. In this section we show that this is indeed the case. It is worth noting that this is an unexpected result, since $\theta=0$ for $2^{-1}\in R$ (see the relations in Definition~\ref{defn-category2}). Hence, there are ``tons'' of $0$-morphisms in the complex. But these $0$-morphisms also come with isomorphisms ``in a lot of'' cases.

The following example for the Lee complex of a v-knot is a blueprint of this effect. It is very important, as indicated in Example~\ref{ex-leedeg} below, that our construction keeps track of the \textit{extra information} how the cobordisms are glued together depending on the orientations of the v-circle diagrams in the resolutions. We note that, even though the orientations can be read of locally, this information has some ``global character''.
\begin{ex}\label{ex-leedeg}
Consider the diagram of the virtual trefoil $L_D$ given in Figure~\ref{figure-big2}. In this example the number of negative crossings is zero, i.e. the leftmost object is the $0$-degree part.
\begin{figure}[ht]
  \centering
     \includegraphics[width=0.575\linewidth]{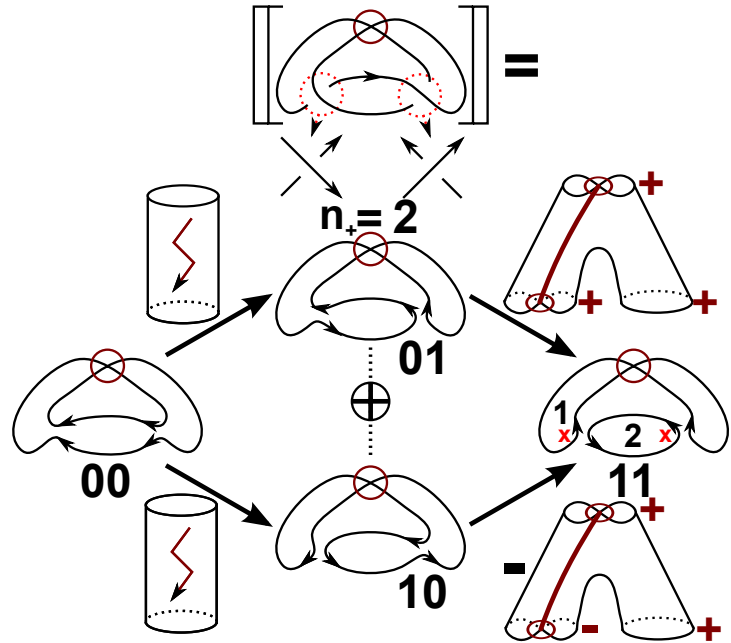}
  \caption{The Lee complex of the v-trefoil. We note that the first map is a $0$-morphism, but the second is an isomorphism.}
  \label{figure-big2}
\end{figure}

Let us consider $R=\bQ$. Then $\theta=0$ and therefore the first two maps are $0$-morphisms. But note that the two right morphisms are not the same, i.e. one is $\Delta^+_{++}$ and the other is $\Delta^+_{-+}$. So on the algebraic level we get, using the maps from before, the following complex, if we fix $B_1=\{1,X\}$ as a basis for $A$ and $B_2=\{1\otimes 1,1\otimes X,X\otimes 1,X\otimes X\}$ for $A\otimes A$.
\[
\begin{xy}
  \xymatrix{
      A\ar[rrr]^{\begin{pmatrix}
      0 & 0 \\
      0 & 0 \\
      0 & 0 \\
      0 & 0 
      \end{pmatrix}} &  & & A\oplus A\ar[rrr]^{\begin{pmatrix}
      0 & 1 & 0 & -1\\
      1 & 0 & 1 & 0\\
      1 & 0 & -1 & 0\\
      0 & 1 & 0 & 1
      \end{pmatrix}} & & & A\otimes A.
      }
\end{xy}
\]
An easy calculations shows that the second matrix is an isomorphism. Hence, the homology of the virtual trefoil is only non-trivial for $k=0$, i.e.
\[
H_k(\mathcal F_{\mathrm{Lee}}(L_D))=\begin{cases}\bQ\oplus\bQ, & \text{if}\;k=0,\\ 0, &\text{else}.\end{cases}
\]
Another example is the v-knot in Figure~\ref{figure-rasexample}, e.g. with the pictured orientation and numbering of the circles from left to right, the three outgoing morphisms from resolution $000$ to $001$, $010$ and $100$ are (up to, in this case, not important signs) the morphisms $m^{+-}_-$, $m^{++}_+$ and $m^{--}_-$, i.e. one alternating and two non-alternating. Hence, the kernel is trivial. The reader should check that the rest also works out in the same fashion as before.
\end{ex}
\subsubsection*{Non-alternating resolutions}
We prove the following interesting result about the number of decorations of v-link resolutions with the ``colours'' down and up. Note that we call an oriented resolution $\mathrm{Re}$ of a v-link diagram \textit{non-alternating} if it is of the form $\uu$ or $\dd$ at the corresponding positions of the saddles. Recall that all the v-link diagrams should be oriented and that such a diagram with $n\in\bN_{>0}$ components has $2^n$ different orientations $\mathrm{Or}_1,\dots,\mathrm{Or}_{2^n}$.

We note that one can also colour the resolutions with ``honest'' colours, say red and green, in such a way that the colour changes at every v-crossing. We call this a \textit{colouring of a v-link resolution} if at the corresponding saddle-position the colours are different, i.e. (red,green) or (green,red). The reader should compare this with the coloured dots in Figure~\ref{figure0-main}.
\begin{thm}\label{thm-nonalternating}(\textbf{Non-alternating resolutions}) Let $L_D$ denote a v-link diagram with $n\in\bN_{>0}$ components. There are bijections of sets
\begin{align*}
\{\mathrm{Or}\mid \mathrm{Or}\;\text{is an orientation of}\;L_D\}&\simeq\{\mathrm{Re}\mid \mathrm{Re}\;\text{is a non-alternating resolution of}\;L_D\}\\
&\simeq\{\mathrm{Co}\mid \mathrm{Co}\,\text{is a coloured resolution of}\;L_D\}.
\end{align*}
If $L_D$ is a v-knot diagram, i.e. $n=1$, then the two non-alternating resolutions are in homology degree $0$. A similar statement holds for the coloured resolutions.
\end{thm}
\begin{proof}
With a slight abuse of notation let us denote the first two sets by $\mathrm{Or}$ and $\mathrm{Re}$. To show the existence of a bijection we construct an explicit map $f\colon\mathrm{Or}\to\mathrm{Re}$ and its inverse.

Given an orientation $\mathrm{Or}$ of the v-link diagram $L_D$, the map $f$ should assign the resolution $\mathrm{Re}$ which is obtained by replacing every oriented crossing of the form $\overcrossing$ and $\undercrossing$ with $\uu$ (and the same for rotations). This is clearly an injection.

Now, given a non-alternating resolution $\mathrm{Re}$, we assign to it an orientation of $L_D$ in the following way. At any non-alternating part of the form $\uu$ and $\dd$ replace the non-alternating part with the corresponding oriented crossing $\overcrossing$ and $\undercrossing$ (or a rotation in the $\dd$ case).

Note that both maps are well-defined and that these two maps are clearly inverses for a v-knot diagram. Moreover, the corresponding non-alternating resolutions are in homology degree $0$, since all $n_+$-crossings are resolved $0$ and all $n_-$-crossings are resolved $1$ in this procedure.   

To see the second bijection use a checker-board colouring of the v-link diagram. Then start at any point of the non-alternating resolution and use the right-hand rule, i.e. the index finger follows the orientation and the string should get the colour of the face on the side of the thumb. As above, one checks that all $n_+$-crossings are resolved $0$ and all $n_-$-crossings are resolved $1$.
\end{proof}
\begin{cor}\label{cor-numbernonalt}
Let $L_D$ be a v-link diagram with $n$ components. Then it has $2^n$ non-alternating resolutions.
\end{cor}
\begin{proof}
Such a diagram has $2^n$ possible orientations. Then the bijection of Theorem~\ref{thm-nonalternating} finishes the proof.
\end{proof}
\begin{ex}\label{ex-nonalt}
Let $L_D$ be the v-knot diagram in Figure~\ref{figure-rasexample}.
\begin{figure}[ht]
  \centering
     \includegraphics[width=0.55\linewidth]{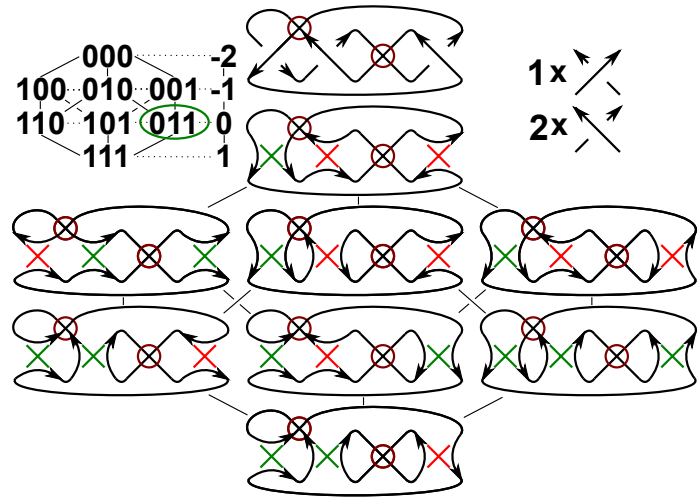}
  \caption{There are exactly two non-alternating resolutions, i.e. the one pictured and the one with all orientations reversed.}
  \label{figure-rasexample}
\end{figure}

Then only the $011$ resolution of the v-knot diagram allows a non-alternating resolution. Moreover, the orientation of the diagram induces this non-alternating resolution by replacing the three crossings with $\dd,\uu$ and $\uu$. The other orientation induces the non-alternating resolution $\uu,\dd$ and $\dd$. Note that, by construction, these resolutions are in homology degree $0$. A computation as in Example~\ref{ex-leedeg} shows that these two non-alternating resolutions give the only two generators of the homology, i.e.
\[
H_k(\mathcal F_{\mathrm{Lee}}(L_D))=\begin{cases}\bQ\oplus\bQ, & \text{if}\;k=0,\\ 0, &\text{else}.\end{cases}
\]
\end{ex}
\subsubsection*{Degeneration}
We start by recalling the motivation, definition and some basic properties of the \textit{Karoubi envelope} of a pre-additive category $\mathcal C$. We denote the envelope as before by $\KAR(\mathcal C)$.

For any category the notion of an idempotent morphisms, i.e. an arrow with $e\circ e=e$, makes sense. Moreover, in a pre-additive category the notion $\mathrm{id}-e$ also makes sense. A classical trick in modern algebra is to use an idempotent, e.g. in $\mathrm{End}_K(V)$ for a given $K$-vector space $V$, to split the algebra into
\[
\mathrm{End}_K(V)\cong \mathrm{im}(e)\oplus\mathrm{im}(\mathrm{id}-e).
\]
Hence, it is a natural question to ask if on can ``split'', given an idempotent $e$, an object of a category $\mathcal O$ in the same way, i.e.
\[
\mathcal O\cong \mathrm{im}(e)\oplus \mathrm{im}(\mathrm{id}-e).
\]
The main problem is that the notion of an ``image'' of an arrow could possibly not exist in an arbitrary category. The Karoubi envelope is an extension of a category such that for a given idempotent $e$ the notions $\mathrm{im}(e)$ makes sense. Therefore, one can ``split'' a given object in the Karoubi envelope that could be indecomposable in the category itself.
\begin{defn}\label{defn-karoubi}
Let $\mathcal C$ be a category and let $e,e^{\prime}\colon\mathcal O\to\mathcal O$ denote idempotents in $\mathrm{Mor}(\mathcal C)$. The \textit{Karoubi envelope of $\mathcal C$}, denoted $\KAR(\mathcal C)$, is the following category.
\begin{itemize}
\item Objects are ordered pairs $(\mathcal O,e)$ of an object $\mathcal O$ and an idempotent $e$ of $\mathcal C$.
\item Morphisms $f\colon(\mathcal O,e)\to(\mathcal O',e')$ are all arrows $f\colon\mathcal O\to\mathcal O'$ of $\mathcal C$ such that the equation $f=f\circ e=e'\circ f$ holds.
\item Compositions are defined in the obvious way. The identity of an object is $e$ itself.
\end{itemize}
It is straightforward to check that this is indeed a category. We denote an object $(\mathcal O,e)$ by $\mathrm{im}(e)$, the \textit{image} of the idempotent $e$. Moreover, we identify the objects of $\mathcal C$ with their image via the embedding functor
\[
\iota\colon\mathcal C\to\KAR(\mathcal C),\;\mathcal O\mapsto (\mathcal O,\mathrm{id}).
\] 
\end{defn}
Note that, if $\mathcal C$ is pre-additive, then $\mathrm{id}-e$ is also an idempotent and, under the identification above, we can finally write
\[
\mathcal O\cong \mathrm{im}(e)\oplus \mathrm{im}(\mathrm{id}-e).
\]
The following proposition is well-known (see e.g.~\cite{bnsm}). The proposition allows us to shift the problem if two chain complexes are homotopy equivalent to the Karoubi envelope. Recall that $\kom(\mathcal C)$ denotes the category of formal chain complexes.
\begin{prop}\label{prop-karoubi}
Let $(C,c),(D,d)$ be two objects, i.e. formal chain complexes, of $\kom(\mathcal C)$. If the two objects are homotopy equivalent in $\kom(\KAR(\mathcal C))$, then the two objects are also homotopy equivalent in $\kom(\mathcal C)$.\qed
\end{prop}
\vskip0.5cm
We define the two orthogonal idempotents $\mathrm{u},\mathrm{d}$ now and show some basic, but very important, properties afterwards.

We call the idempotents \textit{``down and up''}. The reader should be careful not to confuse them with the orientations on the resolutions or the colourings of Theorem~\ref{thm-nonalternating}, i.e. latter colours change at v-crossings, but ``down and up'' do not change.
\begin{defn}\label{defn-idempotent}
We call the two cobordisms in Figure~\ref{figure-idem} the \textit{``down and up'' idempotents}. We denote them by $\mathrm{d}$ and $\mathrm{u}$.
\begin{figure}[ht]
  \centering
     \includegraphics[width=0.6\linewidth]{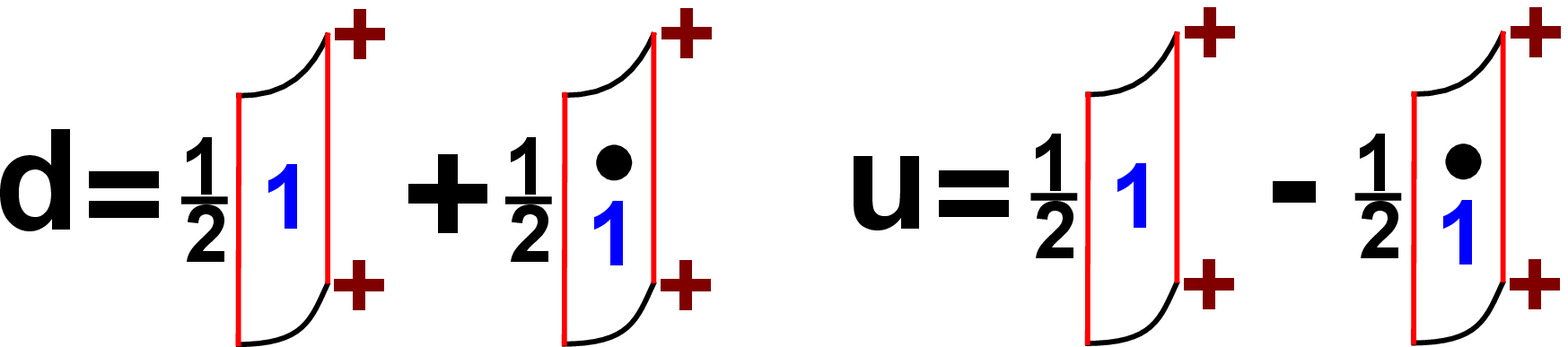}
  \caption{The two idempotents up and down.}
  \label{figure-idem}
\end{figure}

Recall that the dot represents $\frac{1}{2}$-times a handle.
\end{defn}
It is worth noting that (e) is very important. Moreover, we write $\Phi^-_+$ instead of $\Phi^-_+(1)$.
\begin{lem}\label{lem-idem}
The cobordisms $\mathrm{d},\mathrm{u}$ satisfy the following identities.
\begin{enumerate}
\item[(a)] $\mathrm{d}^2=\mathrm{d}$ and $\mathrm{u}^2=\mathrm{u}$ (idempotent).
\item[(b)] $\mathrm{d}\circ\mathrm{u}=0=\mathrm{u}\circ\mathrm{d}$ (orthogonal).
\item[(c)] $\mathrm{d}+\mathrm{u}=\mathrm{id}$ (complete).
\item[(d)] $\mathrm{id}_{\mathrm{dot}}\circ\mathrm{d}=\mathrm{d}$ and $\mathrm{id}_{\mathrm{dot}}\circ\mathrm{u}=-\mathrm{u}$ (Eigenvalues).
\item[(e)] $\Phi^-_+\circ\mathrm{d}=\mathrm{u}\circ\Phi^-_+$ and $\mathrm{d}\circ\Phi^-_+=\Phi^-_+\circ\mathrm{u}$ (change of orientations).
\item[(f)] $[\mathrm{d},\Phi^-_+]=\mathrm{id}(1)_{\mathrm{dot}}=-[\mathrm{u},\Phi^-_+]$ (Commutator relation).
\end{enumerate}
\end{lem}
\begin{proof}
All equations are straightforward to prove. One has to use the dot-relations from Figure~\ref{figure-dotrel} and the relations from Definition~\ref{defn-category}.

In (d)+(f) the surface $\mathrm{id}(1)_{\mathrm{dot}}$ denotes an identity with an extra dot and $+1$ as an indicator.

Beware that the dot represents $\frac{1}{2}$-times a handle. This forces a sign change after composition with the cobordism $\Phi^-_+$. The reader should compare this with the relations in Definition~\ref{defn-category3}.
\end{proof}
Now we take a look at the Karoubi envelope. The discussion above shows that there is an isomorphism
\[
\none\simeq\down\oplus\up.
\]
With this notation we get
\[
\smoothinga\simeq\ddp\oplus\dup\oplus\udp\oplus\uup\;\;\text{ and }\;\;\hsmoothinga\simeq\ddpp\oplus\dupp\oplus\udpp\oplus\uupp.
\]
Recall that the standard orientation for the complex $\bn{\overcrossing}$ is (see e.g. Figure~\ref{figure0-main})
\[
\bn{\overcrossing}=\du\xrightarrow{S(1)^{++}_{++}}\ler.
\]
In order to avoid mixing the notions of the down and up-colours and the orientations we denote this complex simply as $\bn{\overcrossing}^{++}_{++}$, i.e. standard orientations for all strings. Moreover, under the convention left=first superscript, right=second superscript, bottom=first subscript and top=second subscript, a notation like $\bn{\overcrossing}^{+-}_{-+}$ makes sense, i.e. act by $\Phi^-_+$ at the corresponding positions. The following theorem is a main observation of this section. It is worth noting again that (e) of Lemma~\ref{lem-idem} is crucial for the theorem.
\begin{thm}\label{thm-karoubi}
In $\ukobk_R$, there are sixteen chain homotopies (only four are illustrated, but it should be clear how the rest works)
\[
\begin{xy}
  \xymatrix{
      \bn{\overcrossing}^{++}_{++}\simeq_h\dup\oplus\udp\xrightarrow{0}\dupp\oplus\udpp, & \bn{\overcrossing}^{+-}_{-+}\simeq_h\ddp\oplus\uup\xrightarrow{0}\ddpp\oplus\uupp,\\
      \bn{\overcrossing}^{++}_{+-}\simeq_h\dup\oplus\udp\xrightarrow{0}\ddpp\oplus\uupp, & \bn{\overcrossing}^{+-}_{++}\simeq_h\ddp\oplus\uup\xrightarrow{0}\dupp\oplus\udpp.
      }
\end{xy}
\]
Moreover, similar formulas hold for $\bn{\undercrossing}$. 
\end{thm}
\begin{proof}
We use the observations from above, i.e. in the Karoubi envelope the differential of $\bn{\overcrossing}^{++}_{++}$ is a $4\times 4$-matrix of saddles. Hence, for $\bn{\overcrossing}^{++}_{++}$ we get (for simplicity write $S=S(1)^{++}_{++}$ and $S_{\mathrm{d}}$ and $S_{\mathrm{u}}$ for the saddle under the action of down and up)
\[
\begin{xy}
  \xymatrix{
       \text{\phantom{.}}\ddp\oplus\dup\oplus\udp\oplus\uup\ar[rrr]^{\begin{pmatrix}
      S_{\mathrm{d}} & 0 & 0 & 0\\
      0 & 0 & 0 & 0\\
      0 & 0 & 0 & 0\\
      0 & 0 & 0 & S_{\mathrm{u}}
      \end{pmatrix}} & & &  \text{\phantom{.}}\ddpp\oplus\dupp\oplus\udpp\oplus\uupp.
      }
\end{xy}
\]
This is true, because all other saddles are killed by the orthogonality relations of the colours down and up, i.e. (b) of Lemma~\ref{lem-idem}.

Note that both non-zero saddles are invertible, i.e. their inverses are the saddles
\[
\frac{1}{2}(S\colon \hsmoothing \to \smoothing)_{\mathrm{d}}\;\;\text{and}\;\;-\frac{1}{2}(S\colon \hsmoothing \to \smoothing)_{\mathrm{u}}
\]
with only $+$ as boundary decorations. To see this one uses Lemma~\ref{lem-idem} and the neck cutting relation. Thus, we get
\[
\bn{\overcrossing}^{++}_{++}\simeq_h\dup\oplus\udp\xrightarrow{0}\dupp\oplus\udpp.
\]
To prove the rest of the statements one has to use the relation (e) of Lemma~\ref{lem-idem}, i.e. the only surviving objects change according to the action of $\Phi^-_+$. We note again that this is a very important observation, i.e. with a different action of $\Phi^-_+$ this would not be true any more.

For $\bn{\undercrossing}^{++}_{++}$ one can simply copy the arguments from before.
\end{proof}
The following corollary is an application of the semi-local properties of our construction, i.e. we use Theorem~\ref{thm-pcomplex} and Corollary~\ref{cor-pcomplex} to avoid the usage of signs and Lemma~\ref{lem-pcomplex} to see that the involved choices do not matter up to chain isomorphisms.
\begin{cor}\label{cor-importantsaddles}
Let $T^k_D$ be a v-tangle diagram with $m$ crossings. Then $\bn{T^k_D}$ is chain homotopic to a chain complex $(C_*,c_*)$ with only $0$-differentials and objects coloured by the orientations of the resolutions of $T^k_D$, i.e. if a resolution of $\bn{T^k_D}$ is locally of the form
\[
\du\;\;\text{ or }\;\;\dd\;\;\text{ or }\;\;\uu\;\;\text{ or }\;\;\ud
\]
then $(C_*,c_*)$ is locally of the form
\[
\dup\oplus\udp\;\;\text{ or }\;\;\ddp\oplus\uup\;\;\text{ or }\;\;\ddp\oplus\uup\;\;\text{ or }\;\;\dup\oplus\udp.
\]
\end{cor}
\begin{proof}
We note that we work in the Karoubi envelope, but with Proposition~\ref{prop-karoubi} we see that we are free to do so. Moreover, as stated above, we do not care about signs or choices at this place.

Then the statement follows from Theorem~\ref{thm-karoubi} together with the Theorem~\ref{thm-pcomplex} in Section~\ref{sec-vkhca}. To be more precise, we copy the arguments from Theorem~\ref{thm-karoubi} for the saddles of the complex $\bn{T^k_D}$ with a $+1,-1$-indicator. Note that these saddles have an extra action of $\Phi^-_+$ at some of its boundary components. That is why the parts of $(C_*,c_*)$ are locally as illustrated above.

Moreover, the saddles with a $0$-indicator are $0$-morphisms for $\frac{1}{2}\in R$ and there local decomposition is the one given above, since they will, by construction, always be between non-alternating parts of the resolutions and due to the orthogonality relations for up and down, i.e. (b) of Lemma~\ref{lem-idem}, the $\dup$ and $\udp$ parts will be therefore killed (the only possibility how they close is as the rightmost case of Figure~\ref{figure0-order}).
\end{proof}
\vskip0.25cm
As an application of the Theorems~\ref{thm-nonalternating} and~\ref{thm-karoubi} above, we get the desired statement for v-link diagrams. That is, we have the following.
\begin{thm}\label{thm-leedeg}(\textbf{Degeneration}) Let $L_D$ denote a $n$-component v-link diagram. Then $\bn{L_D}_{\mathrm{Lee}}$ is homotopy equivalent (in $\ukob_R$) to a chain complex with only zero differentials and $2^n$ generators given by the $2^n$ non-alternating resolutions.

If $n=1$, i.e. $L_D$ is a v-knot diagram, then the two generators are in homology degree $0$.
\end{thm}
\begin{proof}
We will suppress the notion of the x-markers and the formal signs of the morphisms to maintain readability. Moreover, we will choose a specific orientation for the resolutions. We can do both freely because of Lemma~\ref{lem-commutativeindependence}.

The main part of the proof will be to choose the orientations in a ``good'' way and use Corollary~\ref{cor-importantsaddles}. Moreover, with Theorem~\ref{thm-karoubi}, we see that the complex will be homotopy equivalent to a complex with only $0$-differentials. Hence, the only remaining thing is to show that the number of generators will work out as claimed.

Note that, if a resolution contains a lower part of a multiplication or a upper part of a comultiplication, then by Corollary~\ref{cor-importantsaddles}, this resolution is killed, because these will always be alternating, e.g. $\du$, but will connect as the $\pm 1$ cases of Figure~\ref{figure0-order} (the strings are closed with an even number of v-crossings). Moreover, we can ignore top and bottom parts of $\theta$, since they will always be non-alternating.
\vskip0.5cm
Now we define the \textit{dual graph of a resolution}, denoted $\mathcal D$, as follows. Recall that a resolution is a four valent graph without any c-crossings. Any edge of this graph is a vertex of $\mathcal D$. Two vertices are connected with a labelled edge iff they are connected by a v-crossing $\virtual$ or a $\smoothing$ (or rotations) that is a top part of a multiplication or a bottom part of a comultiplication. First edges should be labelled $v$, the second type of edges should get a labelling that corresponds to the given orientation of the resolution, that is an ``a'' for alternating orientations and a ``n'' otherwise. We will work with the simple graph of that type, i.e. remove circles or parallel edges of the same type. See Figure~\ref{figure-dual}, i.e. the figure shows two resolutions from Figure~\ref{figure-rasexample} and their dual graphs. Note that the leftmost $\smoothing$ of the 011 resolution is part of a $\theta$. 
\begin{figure}[ht]
  \centering
     \includegraphics[width=0.7\linewidth]{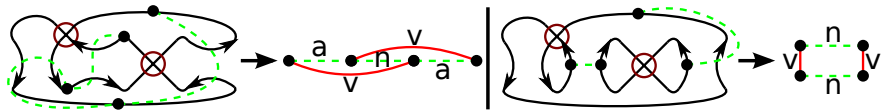}
  \caption{The resolutions 000 and 011 and their dual graphs.}
  \label{figure-dual}
\end{figure}

The advantage of this notation is that the question of surviving resolutions simplifies to the question of a colouring of the dual graph, i.e. a colouring of the dual graph is a colouring with two colours, say red and green, such that any $v$-labelled or $a$-labelled edge has two equally coloured adjacent vertices, but any $n$-labelled edge has two equal colours at adjacent vertices. The reader should compare this to Theorem~\ref{thm-karoubi} and Corollary~\ref{cor-importantsaddles}.

Then, because of Corollary~\ref{cor-importantsaddles}, a resolution will have surviving generators iff it does not contain lower parts of multiplication or upper parts of comultiplications and, given an orientation of the resolution, it allows such a colouring. For example, the left resolution in Figure~\ref{figure-dual} does not allow such an colouring, but the right does.
\vskip0.5cm
Recall that the number of crossings is finite. Hence, we can choose an orientation of any resolution such that the number $m$ of alternating crossings is minimal. The rest is just a case-by-case check, i.e. we have the following three cases.
\begin{itemize}
\item[(i)] The dual graph of the resolution is a tree, i.e. no circles.
\item[(ii)] All circles in the dual graph have an even number of $v$-labelled edges.
\item[(iii)] There is one circle in the dual graph with an odd number of $v$-labelled edges.
\end{itemize}
If $m=0$, i.e. the resolution is non-alternating, we get exactly the claimed number of generators, since there are, by construction, no lower parts of multiplication or upper parts of comultiplications and the dual graph is of type (i) or (ii) and in both cases the graph can be coloured.

So let $m>0$ and let $c$ be an alternating crossing in a resolution $\mathrm{Re}$. The whole resolution is killed if the $c$ is a lower part of a multiplication or an upper part of a comultiplication. Hence, we can assume that all alternating crossings of $\mathrm{Re}$ are either top components of multiplications or bottom components of comultiplications.

So we only have to check the three cases from above. If the resolution is one of type (i), then it is possible to choose the orientations in such a way that all crossings are non-alternating, i.e. this would be a contradiction to the minimality of $m$.

If the resolution is of type (ii), then the resolution only survives, i.e. the dual graph allows a colouring, iff the number of other alternating crossings in every circle is even. But in this case one can also choose an orientation with a lower number of non-alternating crossings. Hence, we would get a contradiction to the minimality of $m$ again. An analogous argument works in the case of type (iii), i.e. contradicting the minimality of $m$ again\footnote{It is worth noting that these arguments work because of the well-known fact that a graph allows a $2$-colouring iff it has no circles of odd length.}.   

Hence, only non-alternating resolutions generate non-vanishing objects and any non-alternating resolution will create exactly two of these. Thus, with Theorem~\ref{thm-nonalternating}, the statement follows.
\end{proof}
\begin{ex}\label{ex-nonalt2}
As an example how the Theorem~\ref{thm-leedeg} works consider the v-knot diagram of Example~\ref{ex-nonalt} again.

The theorem tells us that the resolution 000 should not contribute to the number of generators, i.e. it should get killed. To see this, we first note that in the Karoubi envelope there are $4^3$ different direct summands of coloured (with the idempotents down $d$ and up $u$) versions of the resolution, i.e. four for each crossing. But most of them are killed by the orthogonality of $d$ and $u$, i.e. the two components of the resolution need to have the same colour. Hence, we have the four remaining summands as shown in Figure~\ref{figure-ras2}.
\begin{figure}[ht]
  	 \xy
  	 (-40,0)*{\phantom{.}};
     (40,0)*{\includegraphics[scale=0.7]{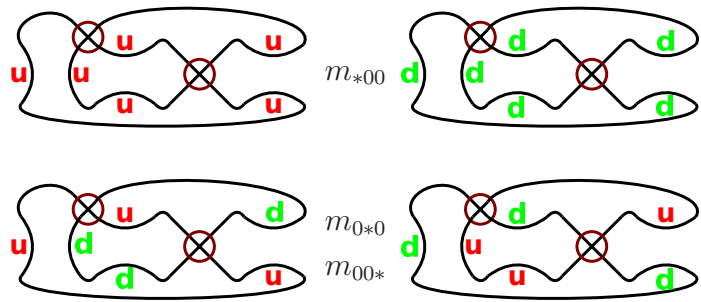}};
     (40,10.5)*{m_{*00}};
     (40,-9.5)*{m_{0*0}};
     (40,-15.5)*{m_{00*}};
     \endxy
  \centering
  \caption{The remaining four coloured versions of resolution 000.}
  \label{figure-ras2}
\end{figure}

Let us denote the three multiplications with this resolution $\gamma_{000}$ as source by
\[
m_{*00}\colon\gamma_{000}\to\gamma_{100}\;\;\text{and}\;\;m_{0*0}\colon\gamma_{000}\to\gamma_{010}\;\;\text{and}\;\;m_{00*}\colon\gamma_{000}\to\gamma_{001}.
\]
If we choose the orientation for $\gamma_{000}$ as indicated in the Figure~\ref{figure-rasexample}, we see that
\[
m_{*00}\colon\dd\to\ril\;\;\text{and}\;\;m_{0*0}\colon\du\to\ler\;\;\text{and}\;\;m_{00*}\colon\ud\to\ril.
\]
We can now use Corollary~\ref{cor-importantsaddles} to see that the only remaining parts for the three multiplications are as follows.
\[
m_{*00}\colon\ddp\oplus\uup\to\dupp\oplus\udpp,
\]
for $m_{*00}$ and for the other two
\[
m_{0*0},m_{*00}\colon\dup\oplus\udp\to\dupp\oplus\udpp.
\]
Hence, the pick two distinct coloured versions as illustrated in Figure~\ref{figure-ras2}. Therefore, there are no surviving generators for the 000 resolution.

It should be noted that changing for example the orientation of the leftmost v-circle in Figure~\ref{ex-nonalt} does not affecting the result, since Lemma~\ref{lem-commutativeindependence} ensures that the resulting complexes are isomorphic.

And in fact such a change leads to
\[
m_{*00}\colon\ud\to\ril\;\;\text{and}\;\;m_{0*0}\colon\uu\to\ler\;\;\text{and}\;\;m_{00*}\colon\uu\to\ril.
\]
Hence, the $m_{*00}$ and the two multiplications $m_{0*0},m_{*00}$ still pick out different coloured versions of the resolution 000. Therefore, there will not be any surviving generators for this case either.
\end{ex}
\vskip0.2cm
We finish by using the functor $\mathcal F_{\mathrm{Lee}}$ to get the corresponding statement in the category $\RMOD$. The reader may compare this to the results in the classical case, e.g. see Proposition 2.4 in~\cite{mtv}.
\begin{prop}\label{prop-algebraicdegen}
Let $L_D$ denote a $n$-component v-link diagram. Then we have the following.
\begin{itemize}
\item[(a)] If $R=\bZ$, then there is an isomorphism
\[
H(\mathcal F_{\mathrm{Lee}}(L_D),\bZ)\cong\bigoplus_{2^n}\bZ\oplus \mathrm{Tor},
\]
where $\mathrm{Tor}$ is all torsion. Moreover, he only possible torsion is $2$-torsion.
\item[(b)] If $R=\bQ$ or $R=\bZ\left[\frac{1}{2}\right]$, then there is an isomorphism
\[
H(\mathcal F_{\mathrm{Lee}}(L_D),R)\cong\bigoplus_{2^n}R.
\]
\end{itemize}
\end{prop}
\begin{proof}
The statement (b) follows from Theorem~\ref{thm-leedeg} above. Recall that the whole construction requires that $2$ is invertible.

For (a) recall the universal coefficients theorem. i.e. there is a short exact sequence
\[
0\rightarrow H_*(\mathcal F_{\mathrm{Lee}}(L_D),\bZ)\otimes_{\bZ} R\rightarrow H_*(\mathcal F_{\mathrm{Lee}}(L_D),R)\rightarrow \mathrm{Tor}(H_{*+1}(\mathcal F_{\mathrm{Lee}}(L_D),\bZ)),R)\rightarrow 0.
\]
Therefore, (a) follows from (b) with $R=\bQ$, since the Tor-functor will vanish in this case and from (b) with $R=\bZ\left[\frac{1}{2}\right]$. Hence, this shows the proposition.
\end{proof}
\subsection{Computer talk}\label{sec-vkhpro}
In this section we show some basic calculations with a computer program we have written. The program is a MATHEMATICA (see~\cite{wolf}) package called \textit{vKh.m}. There is also a notebook called \textit{vKh.nb}. Both and some calculation results are available online, i.e. on the authors homepage\footnote{{\tt \tiny http://xwww.uni-math.gwdg.de/dtubben/vKh.htm}}.

The input data is a v-link diagram in a \textit{circuit notation}, i.e. the classical \textit{planar diagram notation}, but we allow v-crossings. Hence, the input data is a string of labelled $X$, i.e crossings are presented by symbols $X_{ijkl}$ where the numbers are obtained by numbering the edges of the v-link diagram and the
edges around the crossing start counting from the lower incoming and proceeding counterclockwise. We denote such a diagram by CD[X[i,j,k,l],...,X[m,n,o,p]].

After starting MATHEMATICA and loading our package vKh.m, we type in the unknot from Figure~\ref{figure0-big}, the classical and virtual trefoil. Our notation follows the notation of Green in his nice \textit{table of virtual knots}\footnote{J.~Green, A Table of Virtual Knots, {\tt \tiny http://www.math.toronto.edu/drorbn/Students/GreenJ/} (2004)}.

{\In Unknot:= CD[X[1,3,2,4], X[2,1,3,4]]; Knot21 := CD[X[1,3,2,4], X[4,2,1,3]];\\ Knot36 := CD[X[1,5,2,4], X[5,3,6,2], X[3,1,4,6]];}\vskip 2mm

Let us denote the elements $1,X\in A=\bZ[X]/X^2=0$ by 1=vp[i] and X=vm[i] and tensors of these elements multiplicatively. Here the module $A$ should belong to the $i$-th v-circle. Moreover, we denote by the word $a$, whose letters are from the alphabet $\{0,1,*\}$ with exactly one $*$-entry, the cobordism starting at the resolution $\gamma_{*=0}$ and going to the resolution $\gamma_{*=1}$. Let us check the different morphisms.

{\In {d2[Unknot, "0*"],d2[Unknot, "*0"], d2[Unknot, "1*"],d2[Unknot, "*1"]}}
\begingroup\Out {{vp[1] -> vm[2] vp[1] - vm[1] vp[2], vm[1] -> vm[1] vm[2]}, {vp[1] -> 0,\\vm[1] -> 0}, {vp[1] -> 0, vm[1] -> 0}, {vp[1] vp[2] -> -vp[1], vm[2] vp[1] ->\\-vm[1], vm[1] vp[2] -> -vm[1], vm[1] vm[2] -> 0}}\endgroup\vskip 2mm

We see that the two orientable morphisms are $\Delta^+_{-+}$ and $-m^{--}_-=-m^{++}_+$. With the command KhBracket[Knot,r] we generate the $r$-th module of the complex (here for simplicity without gradings). Moreover, with d[Knot][KhBracket[Knot,r]] we calculate the image of the $r$-th differential for the whole module. Let us check the output.

{\In {KhBracket[Unknot, 0], KhBracket[Unknot, 1], KhBracket[Unknot, 2]}}
\begingroup\Out {{v[0, 0] vm[1], v[0, 0] vp[1]}, {v[0, 1] vm[1] vm[2], v[0, 1] vm[2] vp[1],\\v[0, 1] vm[1] vp[2], v[0, 1] vp[1] vp[2], v[1, 0] vm[1], v[1, 0] vp[1]},\\{v[1, 1] vm[1], v[1, 1] vp[1]}}\endgroup
{\In {d[Unknot][KhBracket[Unknot, 0]], d[Unknot][KhBracket[Unknot, 1]]}}
\begingroup\Out {{v[0, 1] vm[1] vm[2], v[0, 1] vm[2] vp[1] - v[0, 1] vm[1] vp[2]}, {0,\\-v[1, 1] vm[1], -v[1, 1] vm[1], -v[1, 1] vp[1], 0, 0}}\endgroup\vskip 2mm

It is easy to check that the composition $d_1\circ d_0$ is indeed zero.

{\In d[Unknot][d[Unknot][KhBracket[Unknot, 0]]]}
\begingroup\Out {0, 0}\endgroup\vskip 2mm

Let us check this for the other two knots, too.

{\In d[Knot21][d[Knot21][KhBracket[Knot21, 0]]]}
\begingroup\Out {0, 0, 0, 0}\endgroup
{\In {d[Knot36][d[Knot36][KhBracket[Knot36, 0]]], d[Knot36][d[Knot36]\\ [KhBracket[Knot36, 1]]]}}
\begingroup\Out {{0, 0, 0, 0, 0, 0, 0, 0}, {0, 0, 0, 0, 0, 0, 0, 0, 0, 0, 0, 0}}\endgroup\vskip 2mm

Now let us check for the trefoil how the signs of the morphisms work out.

{\In {sgn[Knot36, "00*"], sgn[Knot36, "0*0"], sgn[Knot36, "*00"]}}
\begingroup\Out {1, -1, 1}\endgroup
{\In {sgn[Knot36, "01*"], sgn[Knot36, "10*"], sgn[Knot36, "0*1"], 
 sgn[Knot36, "1*0"], sgn[Knot36, "*01"], sgn[Knot36, "*10"]}}
\begingroup\Out {1, 1, 1, -1, -1, -1}\endgroup
{\In {sgn[Knot36, "11*"], sgn[Knot36, "1*1"], sgn[Knot36, "*11"]}}
\begingroup\Out {1, 1, 1}\endgroup\vskip 2mm

We observe that all of the six different faces have an odd number of signs. For example the face $F1=(\gamma_{000},\gamma_{001}\oplus\gamma_{010},\gamma_{011})$ gets a sign from the morphism $d_{0*0}$. Furthermore, the face $F2=(\gamma_{100},\gamma_{101}\oplus\gamma_{110},\gamma_{111})$ gets a sign from the morphism $d_{1*0}$.

The first face is of type 2b and the second is of type 1b. Hence, after a virtualisation the latter should have an even number of signs, but the first should have an odd number signs. Let's check this. First we define a new knot diagram which we obtain by performing a virtualisation on the second crossing of the trefoil.

{\In Knot36v := CD[X[1, 4, 2, 5], X[2, 5, 3, 6], X[3, 6, 4, 1]];}
{\In {sgn[Knot36v, "00*"], sgn[Knot36v, "0*0"], sgn[Knot36v, "*00"]}}
\begingroup\Out {1, -1, 1}\endgroup
{\In {sgn[Knot36v, "01*"], sgn[Knot36v, "10*"], sgn[Knot36v, "0*1"], 
 sgn[Knot36v, "1*0"], sgn[Knot36v, "*01"], sgn[Knot36v, "*10"]}}
\begingroup\Out {1, 1, 1, -1, -1, -1}\endgroup
{\In {sgn[Knot36v, "11*"], sgn[Knot36v, "1*1"], sgn[Knot36v, "*11"]}}
\begingroup\Out {1, -1, 1}\endgroup\vskip 2mm

Indeed only the sign of the morphism $d_{1*1}$ is different now. Hence, the face $F1$ still has an odd number, but the face $F2$ has an even number off signs. This should cancel with the extra sign of the pantsdown morphism $d_{1*1}$.

{\In {d[Knot36v][d[Knot36v][KhBracket[Knot36v, 0]]], d[Knot36v][d[Knot36v]\\ [KhBracket[Knot36v, 1]]]}}
\begingroup\Out {{0, 0, 0, 0, 0, 0, 0, 0}, {0, 0, 0, 0, 0, 0, 0, 0, 0, 0, 0, 0}}\endgroup\vskip 2mm

Let us look at some calculation results for the four knots. The output is Betti[q,t], i.e. the dimension of the homology group in quantum degree $q$ and homology degree $t$. The unknot should have trivial homology.

{\In vKh[Unknot]}
\begingroup\Out {Betti[-1,-2] = 0, Betti[-1,0] = 0, Betti[0,-2] = 0, Betti[0,-1] = 1,\\Betti[0,0]= 0, Betti[0,1] = 1, Betti[0,1] = 0, Betti[1,0] = 0, Betti[1,2] = 0}\endgroup
\begingroup\Out 1/q + q\endgroup\vskip 2mm

For the other outputs we skip the Betti-numbers. One can read them off from the polynomial. The trefoil and its virtualisation have the same output (as they should).

{\In vKh[Knot21]}
\begingroup\Out 1/q^3 + 1/q + 1/(q^6 t^2) + 1/(q^2 t)\endgroup
{\In vKh[Knot36]}
\begingroup\Out 1/q^3 + 1/q + 1/(q^9 t^3) + 1/(q^5 t^2)\endgroup
{\In vKh[Knot36v]}
\begingroup\Out 1/q^3 + 1/q + 1/(q^9 t^3) + 1/(q^5 t^2)\endgroup\vskip 2mm

Let us check that the graded Euler characteristic is the Jones polynomial\footnote{To simplify the outputs we have avoided to include the orientation of the v-links in the input, i.e. every output needs a grading shift.}.

{\In Factor[(vKh[Knot21] /. t -> -1)/(q + q^-1)]}
\begingroup\Out (1 - q^2 + q^3)/q^5\endgroup
{\In Factor[(vKh[Knot36] /. t -> -1)/(q + q^-1)]}
\begingroup\Out (-1 + q^2 + q^6)/q^8\endgroup\vskip 2mm

Another observation is the following. The map $\Phi^-_+$ sends $1$ to itself, but $X$ to $-X$. Hence, there is a good change for 2-torsion. Let us check. Here Tor[q,t] denotes the $\bZ/p\bZ$-rank minus the $\bZ$-rank (both graded) of Betti[q,t]$\otimes\bZ/p\bZ$. Even the v-trefoil has 2-torsion, but no 3-torsion.

{\In vKh[Knot21,2]}
\begingroup\Out {Tor[-2,-6] = 0, Tor[-2,-4] = 0, Tor[-2,-2] = 0, Tor[-1,-4] = 1,\\Tor[-1,-2] = 0, Tor[0,-3] = 0, Tor[0,-1] = 0}\endgroup
\begingroup\Out 1/(q^4 t)\endgroup
{\In vKh[Knot21,3]}
\begingroup\Out {Tor[-2,-6] = 0, Tor[-2,-4] = 0, Tor[-2,-2] = 0, Tor[-1,-4] = 0,\\Tor[-1,-2] = 0, Tor[0,-3] = 0, Tor[0,-1] = 0}\endgroup
\begingroup\Out 0\endgroup\vskip 2mm

There seems to be a lot of 2-torsion!

{\In Knot32 := CD[X[2, 6, 3, 1], X[4, 2, 5, 1], X[5, 3, 6, 4]];}
{\In vKh[Knot32]}
\begingroup\Out 1/q^2 + 1/q + q + 1/(q^5 t^2) + 1/(q t) + q^2 t\endgroup
{\In vKh[Knot32,2]}
\begingroup\Out 1/(q^3 t) + t\endgroup\vskip 2mm

Because the virtual Khovanov complex is invariant under virtualisation, there are many examples of non-trivial v-knots with trivial Khovanov complex.

{\In Knot459 := CD[X[2, 8, 3, 1], X[4, 2, 5, 1], X[3, 6, 4, 7], X[5, 8, 6, 7]];}
{\In vKh[Knot32]}
\begingroup\Out 1/q + q\endgroup\vskip 2mm

Let us try an harder example. We mention that the faces are all anticommutative, hence the composition of the differentials is zero. 

{\In Knot53 := CD[X[1, 9, 2, 10], X[2, 10, 3, 1], X[5, 4, 6, 3], X[7, 4, 8, 5],\\X[8, 7, 9, 6]];}
{\In vKh[Knot53]}
\begingroup\Out 2 + 1/q^3 + 1/q^2 + 1/q + 1/(q^7 t^3) + 1/(q^6 t^2) + 1/(q^5 t^2) \\+ 1/(q^3 t^2) + 2/(q^4 t) + 1/(q^2 t) + 1/(q t) + t/q + q^2 t + q^3 t^2\endgroup
{\In vKh[Knot53,2]}
\begingroup\Out 2/q^2 + 1/(q^5 t^2) + 1/(q^4 t) + 1/(q^3 t) + t + q t^2\endgroup
{\In \{d[Knot53][d[Knot53][KhBracket[Knot53, 0]]], d[Knot53][d[Knot53][KhBracket\\ [Knot53, 1]]], d[Knot53][d[Knot53][KhBracket[Knot53, 2]]], d[Knot53][d[Knot53]\\ [KhBracket[Knot53, 3]]]\}}
\begingroup\Out {{0, 0, 0, 0, 0, 0, 0, 0}, {0, 0, 0, 0, 0, 0, 0, 0, 0, 0, 0, 0, 0, 0, 0,\\0, 0, 0, 0, 0, 0, 0, 0, 0}, {0, 0, 0, 0, 0, 0, 0, 0, 0, 0, 0, 0, 0, 0, 0, 0, 0, 0,\\0, 0, 0, 0, 0, 0, 0, 0, 0, 0, 0, 0, 0, 0}, {0, 0, 0, 0, 0, 0, 0, 0, 0, 0, 0, 0, 0,\\0, 0, 0, 0, 0, 0, 0, 0, 0, 0, 0, 0, 0, 0, 0}}\endgroup\vskip 2mm

The virtual Khovanov complex is strictly stronger than the virtual Jones polynomial. The first example appears for v-links with seven crossings. Let's check two examples.

{\In Example1 := CD[X[1, 4, 2, 3], X[2, 10, 3, 11], X[4, 9, 5, 10], X[11, 5, 12, 6],\\X[6, 1, 7, 14], X[12, 8, 13, 7], X[13, 9, 14, 8]]; Example2 := CD[X[1, 4, 2, 3],\\X[2, 11, 3, 10], X[4, 10, 5, 9], X[14, 5, 1, 6], X[6, 12, 7, 11], X[13, 7, 14, 8],\\X[12, 8, 13, 9]]; Example3 := CD[X[1, 4, 2, 3], X[2, 11, 3, 10], X[4, 9, 5, 10],\\X[13, 5, 14, 6], X[6, 11, 7, 12], X[14, 8, 1, 7], X[12, 8, 13, 9]]; Example4 :=\\CD[X[1, 4, 2, 3], X[2, 11, 3, 10], X[4, 10, 5, 9], X[14, 5, 1, 6], X[6, 13, 7, 14],\\X[11, 7, 12, 8], X[12, 8, 13, 9]];}\vskip 2mm

So let us see what our program calculates.

{\In \{vKh[Example1], vKh[Example2], vKh[Example3], vKh[Example4]\}}
\begingroup\Out {2 + 1/q + q + 2 q^2 + 1/(q^3 t^2) + 2/(q^2 t) + q/t + 2 q t + 2 q^4 t\\+ q^3 t^2 + 2 q^5 t^2 + q^7 t^3, 2 + 1/q + q + 2 q^2 + q^3 + 1/(q^3 t^2) + 2/(q^2 t)\\+ q/t + 2 q t + q^2 t + q^3 t + 2 q^4 t + q^2 t^2 + q^3 t^2 + 2 q^5 t^2 + q^6 t^2 + q^6 t^3 + q^7 t^3, 2/q^2 + 1/q + 3 q + 1/(q^6 t^3) + 2/(q^5 t^2) + 1/(q^2 t^2) + 2/(q^3 t)\\+ 2/(q t) + t + 2 q^2 t + q^4 t^2, 1 + 2/q^2 + 2/q + 3 q + 1/(q^6 t^3) + 2/(q^5 t^2)\\+ 1/(q^4 t^2) + 1/(q^2 t^2) + 1/t + 1/(q^4 t) + 2/(q^3 t) + 2/(q t) + t + t/q + 2 q^2 t + q^3 t + q^3 t^2 + q^4 t^2}\endgroup\vskip 2mm

Good news: Example1 and Example2 have the same virtual Jones polynomial ($t=-1$), but different virtual Khovanov homology, i.e. Example2 has the six extra terms (compared to Example1) $q^2t$, $q^2t^2$, $q^3$, $q^3t$, $q^6t^2$ and $q^6t^3$. They all cancel if we substitute $t=-1$. An analogously effect happens for Example3 and Example4. Furthermore, our calculations suggest that this repeats frequently for v-knots with seven or more crossings.

The command line GausstoCD converts \textit{signed Gauss Code} to a CD representation. The signed Gauss code has to start with the first overcrossing. To get the mirror image we can use the rule from below. For example the virtual trefoil and its mirror are not equivalent.
{\In Knot21gauss := "O1-O2-U1-U2-";}
{\In GuasstoCD[Knot21gauss]}
\begingroup\Out CD[X[1, 4, 2, 3], X[2, 1, 3, 4]]\endgroup
{\In GuasstoCD[Knot21gauss] /. X[i_,j_,k_,l_] :> X[i,l,k,j]}
\begingroup\Out CD[X[1, 3, 2, 4], X[2, 4, 3, 1]]\endgroup
{\In \{vKh[GausstoCD[Knot21gauss]], 
 vKh[GausstoCD[Knot21gauss] /. X[i_, j_, k_, l_] :> X[i, l, k, j]]\}}
\begingroup\Out {q + q^3 + q^2 t + q^6 t^2, 1/q^3 + 1/q + 1/(q^6 t^2) + 1/(q^2 t)}\endgroup\vskip 2mm

We used this to calculate the virtual Khovanov homology for all different v-knots with less or equal five crossings. The input was the list of v-knots from Green's virtual knot table. The results are available on the author's website (as mentioned before). One could visualise the polynomial with the function \textit{Ployplot}. It creates an output as in the figures~ \ref{figure-calc1},~\ref{figure-calc2}.
\begin{figure}[ht]
\begin{minipage}[c]{6,9cm}
	\centering
	\includegraphics[scale=0.425]{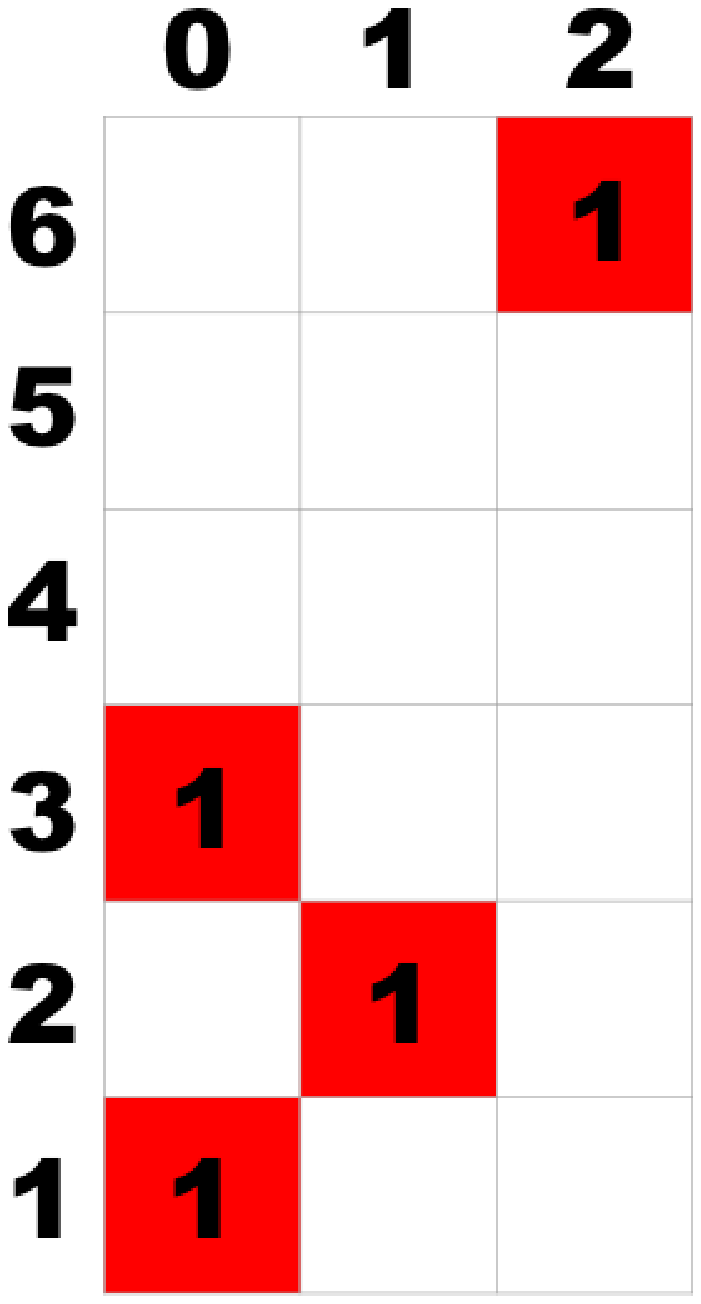}
	\caption{Homology of the v-trefoil.}
	\label{figure-calc1}
\end{minipage}
\begin{minipage}[c]{6,9cm}
	\centering
	\includegraphics[scale=0.425]{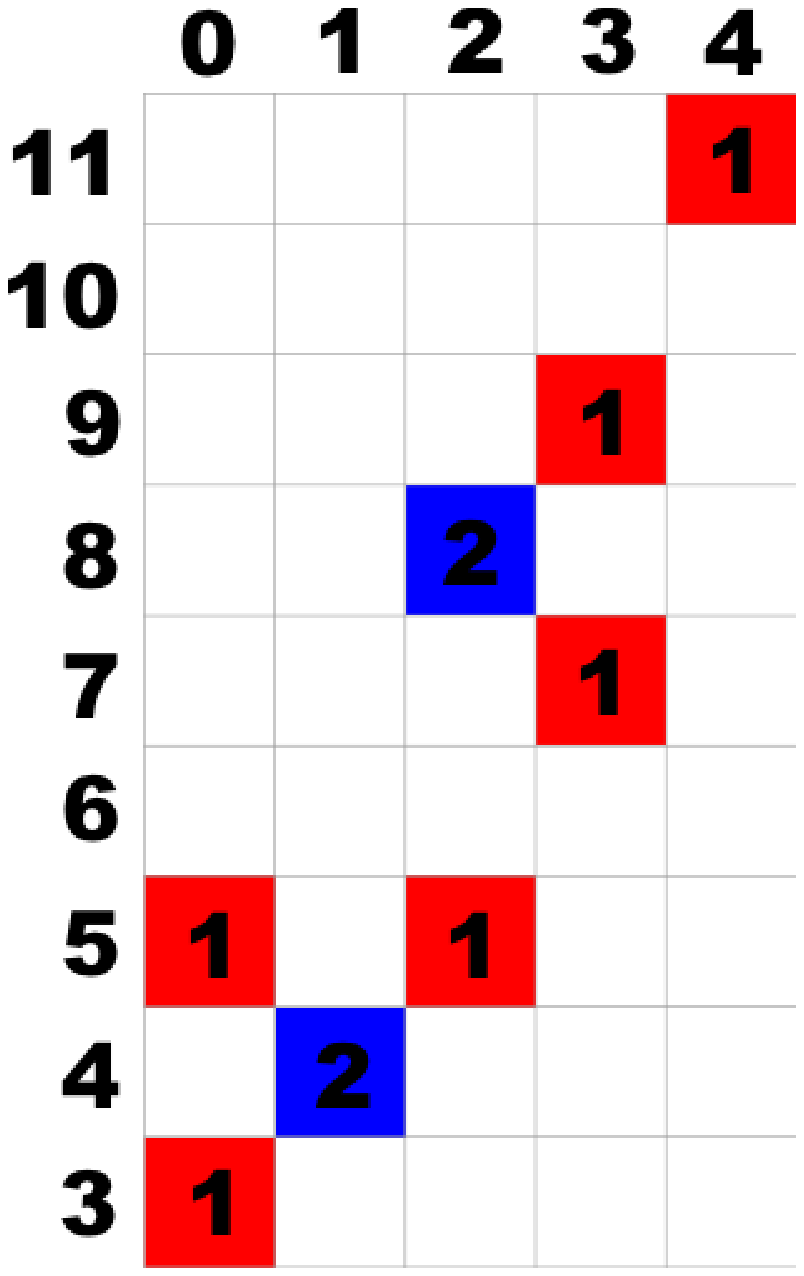}
	\caption{Homology of the v-knot 4.1.}
	\label{figure-calc2}
\end{minipage}
\end{figure}
{\In Knot41 := "O1-O2-U1-U2-O3-O4-U3-U4-"; vKh[GausstoCD[Knot41]]}
\begingroup\Out q^3 + q^5 + 2 q^4 t + q^5 t^2 + 2 q^8 t^2 + q^7 t^3\\ + q^9 t^3 + 
 q^11 t^4]\endgroup\vskip 2mm
The output of this v-knot and of the mirror of the virtual trefoil is shown in the figures~\ref{figure-calc1},~\ref{figure-calc2}. In these pictures the quantum degree is on the y-axis and the homology degree on the x-axis.
\subsection{Open issues}\label{sec-vkhend}
Here are some open problems that we have observed. Note that nowadays the results about classical Khovanov homology form a highly studied and rich field. So there are much more open questions related to our construction.
\begin{itemize}
\item It is quite remarkable that one has to use a ``$\wedge$-product like'' construction to define even, virtual Khovanov homology. An interpretation of this fact is missing.
\item Our complex is an extension of the classical (even) Khovanov complex. We shortly discuss a method which could lead to an extension of odd Khovanov homology~\cite{ors}. Even and odd Khovanov homology differ over $\bQ$ but are equal over $\bZ/2$.
\item Secondly we discuss the relationship between the virtual Khovanov complex and the categorification of the higher quantum polynomials ($n\ge 3$) from Khovanov in~\cite{kh3} and Mackaay and Vaz in~\cite{mv1} and Mackaay, Sto\v{s}i\'{c} and Vaz in~\cite{msv}.
\item The results from Section~\ref{sec-vkhapp} could lead to an extension of the Rasmussen invariant to virtual knots.
\end{itemize}
On the second point: The reader familiar with the paper of Ozsv\'{a}th, Rasmussen and Szab\'{o} may have already identified our map
\[
F_{Kh}((\Phi^-_+\amalg\mathrm{id}^+_+)\circ\Delta^+_{++})\colon A\to A\otimes A
\]
to be the comultiplication which they use.

One main difference between the even and odd Khovanov complex is the usage of this map instead of the standard map $F_{Kh}(\Delta^+_{++})$ and the structure of an exterior algebra instead of direct sums. Furthermore, there are commutative and anticommutative faces in the odd Khovanov complex. But because every cube has an even number of both types of faces, there is a sign assignment which makes every face anticommute. One major problem is the question how to handle unorientable faces, because these faces can be counted as commutative or anticommutative. Furthermore, one should admit that faces of type 1a and 1b can be commutative or anticommutative. Hence, there is still much work to do.  

On the third point: The key idea in the categorification of the $\mathfrak{sl}(n)$-polynomial for $n\ge 3$ is the usage of so-called foams. This very interesting approach due to Khovanov, Mackaay, Sto\v{s}i\'{c} and Vaz (see in their papers~\cite{kh3},~\cite{mv1} and~\cite{msv}).

So in the virtual case one should use a topological construction with virtual webs and decorated, possible non-orientable foams (immersed rather than embedded). So their concept to categorify the $\mathfrak{sl}(n)$-polynomials for $n=3$ should lift to v-links. This needs further work (the sign assignment seem to be the main point), but seems to be very interesting. The $n>3$ case is indeed more complicated. In their paper Mackaay, Sto\v{s}i\'{c} and Vaz (see~\cite{msv}) use a special formula, the so-called Kapustin-Li formula, to find the adapted relations. But this formula only works in the orientable case and it has no straightforward extension to the non-orientable case. But hopefully the collection of relations they use is already enough to show invariance under the vRM1, vRM2, vRM3 and the mRM moves. At least in the case $n=2$ the local relations are enough to show the invariance.
\newpage
\section{The $\mathfrak{sl}_3$ web algebra}\label{sec-web}
\subsection{A brief summary}\label{sec-websum}
We summarise the results of Section~\ref{sec-web} now. Recall that Section~\ref{sec-web} is a slight adaptation of a preprint of Mackaay, Pan and the author~\cite{mpt}.

It is worth noting that we have collected some facts about higher categories (see Section~\ref{sec-techhigher}), Grothendieck groups (see Sections~\ref{sec-techgrgr} and~\ref{sec-techgrgrcat}) and cellular algebras (see Section~\ref{sec-techcell}) in later sections. They are far from being complete, but the reader can hopefully find some useful information therein.
\vskip0.5cm
Recall that we consider the $\mathfrak{sl}_3$ analogue of Khovanov's arc algebras and we call them \textit{web algebras} and denote them by $K_S$, where $S$ is a sign string (string of $+$ and $-$ signs). We prove the following main results regarding $K_S$.
\begin{enumerate}
\item $K_S$ is a graded, symmetric Frobenius algebra (Theorem~\ref{thm:frob}).
\item We give an explicit degree preserving algebra isomorphism 
between the cohomology ring of the Spaltenstein variety 
$X^{\lambda}_{\mu}$ and $Z(K_S)$, where $\lambda$ and $\mu$ are 
two weights determined by $S$ (Theorem~\ref{thm:center}). 
\item Let $V_S=V_{s_1}\otimes\cdots\otimes V_{s_n}$, where $V_+$ is the 
basic $U_q(\mathfrak{sl}_3)$-representation and $V_-$ its dual. 
Kuperberg~\cite{kup} proved that $W_S$, the space of $\mathfrak{sl}_3$ webs 
whose boundary is determined by $S$, is isomorphic to 
$\mathrm{Inv}_{U_q(\mathfrak{sl}_3)}(V_S)$, the space of invariant tensors in 
$V_S$. Choose an arbitrary 
$k\in\mathbb{N}$ and let $n=3k$. By $q$-skew Howe duality, which 
we will explain at the beginning of Section~\ref{sec-webhowe}, 
we know that 
\[
V_{(3^k)}\cong \bigoplus_S W_S.
\]
Here $V_{(3^k)}$ denotes the irreducible $U_q(\mathfrak{gl}_n)$-module 
with highest weight $(3^k)$. Recall that this can be restricted to an irreducible $U_q(\mathfrak{sl}_n)$-module. The direct sum on the right-hand side 
is taken over all \textit{enhanced sign sequences} of length $n$, 
which are in bijective correspondence to the semi-standard Young tableaux with 
$k$ rows and 3 columns. 
\vskip0.5cm
In Section~\ref{sec-webhoweb} we categorify 
this result. Let $R_{(3^k)}$ be the cyclotomic Khovanov-Lauda-Rouquier 
algebra (cyclotomic 
KLR algebra for short) with highest weight $(3^k)$. Brundan and 
Kleshchev~\cite{bk2} (see also~\cite{kaka},
~\cite{lv},~\cite{vv} and~\cite{web1}) proved that 
\[
K^{\oplus}_0(R_{(3^k)}\text{-}\mathrm{p\textbf{Mod}}_{\mathrm{gr}})
\cong V_{(3^k)}^{\mathbb{Z}},
\] 
where the latter is the integral form of $V_{(3^k)}$. 

We prove (in Proposition~\ref{prop:cataction}) that there exists an  
exact, degree preserving categorical $\Ucat$-action on 
\[
\bigoplus_S K_S\text{-}\mathrm{\textbf{Mod}}_{\mathrm{gr}},
\] 
where $\Ucat$ is Khovanov and Lauda's diagrammatic categorification of 
$\U$. This categorical action can be restricted to 
\[
\bigoplus_S K_S\text{-}\mathrm{p\textbf{Mod}}_{\mathrm{gr}}.
\] 
By a general result due to Rouquier~\cite{rou} (in fact a slight variation of Rouquier's result, see Section~\ref{sec-techhigherrep}), which 
we recall in Proposition~\ref{prop:rouquier}, we get
\begin{equation}
\label{eq:rouquierequiv}
R_{(3^k)}\text{-}\mathrm{p\textbf{Mod}}_{\mathrm{gr}}\cong \bigoplus_S 
K_S\text{-}\mathrm{p\textbf{Mod}}_{\mathrm{gr}}.
\end{equation}
\item In particular, 
this proves that the split Grothendieck groups of both categories are 
isomorphic (Corollary~\ref{cor:equivalence}). It follows that we have 
\[
K^{\oplus}_0\left(K_S\text{-}\mathrm{p\textbf{Mod}}_{\mathrm{gr}}\right)\cong W_S^{\mathbb{Z}},\]
for any $S$. Again, the superscript $\mathbb{Z}$ denotes the integral form. 
\item As proved in Corollary~\ref{cor:equivalence}, the equivalence 
in~\eqref{eq:rouquierequiv} implies that 
$R_{(3^k)}$ and $\bigoplus_S K_S$ 
are Morita equivalent (Proposition~\ref{prop:morita}), i.e. we have 
\begin{equation}
\label{eq:rouquiermorita}
R_{(3^k)}\text{-}\mathrm{\textbf{Mod}}_{\mathrm{gr}}\cong \bigoplus_S 
K_S\text{-}\mathrm{\textbf{Mod}}_{\mathrm{gr}}.
\end{equation}
\item In Corollary~\ref{cor:moritacellular}, we show 
that~\eqref{eq:rouquiermorita} implies that 
$K_S$ is a graded cellular algebra, for any $S$. This observation relies on several facts, see Section~\ref{sec-techcell} for more details.
\item We show that the graded, indecomposable, projective $K_S$-modules 
correspond to the dual canonical basis elements in 
$\mathrm{Inv}(V_S)$ (Theorem~\ref{thm:dualcan}).
\item In a new section we give an isotopy invariant, homogeneous basis of $K_S$ (Theorem~\ref{thm-basis2} and Corollary~\ref{cor-basis}).  
\end{enumerate}
The first result is easy to prove and similar to the case for $H_n$. 
Some of the other results are much harder to prove for $K_S$ 
than their analogues are for $H_n$ 
(e.g. see Remark~\ref{rem:counter2}). In order to prove the second and 
the penultimate result, we introduce a 
``new trick'', i.e. we use a deformation of $K_S$, called $G_S$. This deformation 
is induced by Gornik's~\cite{gornik} deformation of Khovanov's original 
$\mathfrak{sl}_3$ foam relations. One big difference 
between $G_S$ and $K_S$ is that the former algebra is \textit{filtered} whereas 
the latter is \textit{graded}. As a matter of fact, 
$K_S$ is the associated graded algebra of $G_S$. The usefulness of 
$G_S$ relies on the fact that $G_S$ is semisimple as an algebra, 
i.e. forgetting the filtration (see Proposition~\ref{prop:Gsemisimple}).
\vskip0.5cm
Let us explain the connection to existing work in the literature. There are 
two diagrammatic approaches which give $\mathfrak{sl}_3$ link homologies, e.g. 
there is Khovanov's original approach using foams~\cite{kh3} and 
there is Webster's approach~\cite{web1},~\cite{web2} using a generalisation of the 
cyclotomic KLR-algebras. In Proposition 4.4 in~\cite{web2}, 
Webster 
proved that both link homologies are isomorphic, but the proof is quite 
sophisticated and relies on Mazorchuk and Stroppel's approach~\cite{mast} to 
link homology using functors and natural transformations on certain 
blocks of category $\mathcal{O}$. Our results 
of Section~\ref{sec-web} might help to give an elementary and direct proof that 
Khovanov and Webster's $\mathfrak{sl}_3$ link homologies are isomorphic. A very recent and related approach is due to Lauda, Queffelec and Rose in~\cite{lqr}.

As we explained in more detail in Section~\ref{sec-introb}, it 
should not be too hard to generalise our results of Section~\ref{sec-web} to the case 
for $\mathfrak{sl}_n$, with $n\geq 4$, using matrix factorisations instead 
of foams. This could be helpful to show that 
Webster's $\mathfrak{sl}_n$ link homology 
is isomorphic to Khovanov and Rozansky's link homology~\cite{kr1}. For 
$n\geq 4$, Webster has conjectured this result to hold, but he 
has not proved it (see his remarks below Proposition 4.4 in~\cite{web2}).
\vskip0.5cm
Another question is how $K_S\text{-}\mathrm{p\textbf{Mod}}_{\mathrm{gr}}$ is related to 
(a subcategory of) $R_S\text{-}\mathrm{\textbf{Mod}}_{\mathrm{gr}}$, where $R_S$ is 
Webster's~\cite{web1} generalisation of the cyclotomic KLR-algebra which 
categorifies $V_S$. In one of the later versions of~\cite{web1}, 
Webster has added a section (Section 4.3) on the categorification of 
skew Howe duality within his framework of generalised cyclotomic KLR-algebras. 
\vskip0.5cm
As mentioned in Section~\ref{sec-introb}, in~\cite{fkk}, Fontaine, Kamnitzer and Kuperberg study so-called spiders i.e. given a sign string $S$, 
the \textit{Satake fiber} $F(S)$, denoted 
$F(\overrightarrow{\lambda})$ in~\cite{fkk}, is isomorphic to 
the Spaltenstein variety $X^{\lambda}_{\mu}$ mentioned above. 
Here we point out the difference in these notations 
that otherwise might confuse the reader, i.e. the $\lambda$ in~\cite{fkk} is 
equal to $\mu$ in our notation, which is also equivalent to $S$. 
Given a web $w$ with boundary corresponding to $S$, Fontaine, Kamnitzer and 
Kuperberg also define a 
variety $Q(D(w))$, called the \textit{web variety}. One obvious question 
is the following (asked to us by Kamnitzer).
\begin{question}\label{question}
For any two basis webs $u,v\in B_S$, does there exist 
a isomorphism of graded vector spaces  
\[
H^*(Q(D(u)))\otimes_{F(S)}H^*(Q(D(v)))\cong {}_uK_v,
\]  
that gives rise to an isomorphisms of graded algebras
\[
\bigoplus_{u,v\in B_S}{}_uK_v\cong \bigoplus_{u,v\in B_S}H^*(Q(D(u)))\otimes_{F(S)}H^*(Q(D(v))),
\]
where the left side is the decomposition of $K_S$ of Section~\ref{sec-webalg} and 
the product on 
\[
\bigoplus_{u,v\in B_S} H^*(Q(D(u)))\otimes_{F(S)}H^*(Q(D(v)))
\] 
is given by convolution.   
\end{question}
If the answer to this question is affirmative, then that would be 
the $\mathfrak{sl}_3$ analogue of the result, 
due to Stroppel and Webster~\cite{sw}, 
relating $H_n$ to the intersection cohomology of the $(n,n)$-Springer fiber. 
Our Theorem~\ref{thm:center} is a first step towards proving 
Kamnitzer's conjecture. 
We also note that, in~\cite{kh5}, Khovanov showed that the center of $H_n$ is 
isomorphic to the ordinary cohomology of the $(n,n)$-Springer fiber, before 
Stroppel and Webster proved the more general result.
\vskip0.5cm
Another related point is to prove similar results as Brundan and Stroppel~\cite{bs1} showed for the $\mathfrak{sl}_2$ analogues, denoted $H_n$, of our algebra $K_S$. For example, they showed that $H_n$ is a graded cellular algebra by constructing an explicit cellular basis. Using the explicit basis, they also constructed the quasi-hereditary cover of $H_n$. We note that the results of the new section, i.e. Section~\ref{sec-webbase}, are the first step to generalise their work. 
\vskip0.5cm
This section is organised as follows.
\begin{enumerate}
\item We recall the definitions 
and some fundamental properties of $\mathfrak{sl}_3$ webs in Section~\ref{sec-webbasica}, $\mathfrak{sl}_3$ foams in Section~\ref{sec-webbasicb} and categorified quantum 
algebras and their categorical representations in Section~\ref{sec-webbasicc}.  
\item In Section~\ref{sec-webalg}, we define $K_S$ and prove the 
first of our aforementioned main results.
\item After recalling some notation in Section~\ref{sec-webcenter}, we first study the relation between 
column strict tableaux and webs with 
flows, i.e. in Section~\ref{sec-webflow}. Using this relation, we prove our second main result in Section~\ref{sec-webcenterb}.
\item In Section~\ref{sec-webhowe}, we recall Howe duality. 
\item In the Sections~\ref{sec-webhowea} and~\ref{sec-webhoweb}, we explain Howe duality (we recall Howe duality in Section~\ref{sec-webhowe}) in our 
context and categorify the case relevant in Section~\ref{sec-web}. This leads to the 
other main results. 
\item The Sections~\ref{sec-webcenter},~\ref{sec-webflow} and~\ref{sec-webcenterb} and the three Sections~\ref{sec-webhowe},~\ref{sec-webhowea} and~\ref{sec-webhoweb} are largely independent of each other. However, the proof of 
Theorem~\ref{thm:center} requires Proposition~\ref{prop:morita} and 
the proof of Proposition~\ref{prop:unitriang}, which is a key ingredient 
for the proof of Theorem~\ref{thm:dualcan}, requires Lemma~\ref{lem:dimZG}.
\item In the new Section~\ref{sec-webbase} we give in Theorem~\ref{thm-basis} a method, if one makes certain choices, to obtain a homogeneous basis of $K_S$. We conjecture that certain choices will give a graded cellular basis (in the sense of Section~\ref{sec-techcell}) of $K_S$. Furthermore, in Theorem~\ref{thm-basis2} and Corollary~\ref{cor-basis}, we explain how this procedure can be used to give an isotopy invariant basis.
\item Note that Section~\ref{sec-appendix} replaces the Appendix of our paper~\cite{mpt}.
\end{enumerate}
\subsection{Basic definitions and background: Webs}\label{sec-webbasica}
In~\cite{kup}, Kuperberg describes the representation theory of 
$U_{q}(\mathfrak{sl}_3)$ using oriented trivalent graphs, possibly with 
boundary, called \textit{webs}. Boundaries of webs
consist of univalent vertices (the ends of oriented edges), which we will 
usually put on a horizontal line (or various horizontal lines), e.g. such a web is shown below.
\begin{align}
\xy
   (0,0)*{\includegraphics[width=140px]{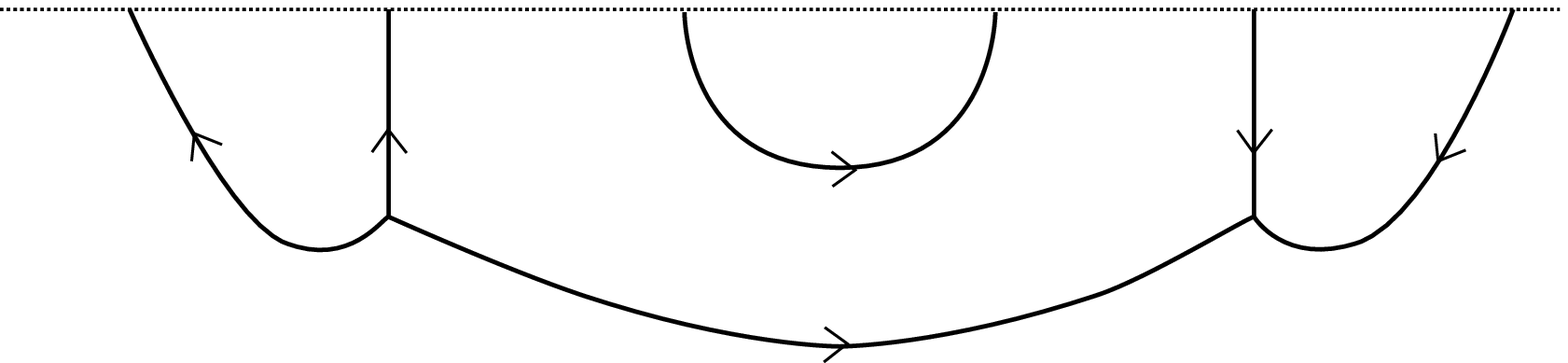}};
\endxy
\end{align}
We say that a web has $n$ free strands if the number of non-trivalent vertices 
is exactly $n$. In this way, the boundary of a web can be 
identified with a \textit{sign string} $S=(s_1,\ldots,s_n)$, with $s_i=\pm$, 
such that upward oriented boundary edges get a ``$+$'' and downward oriented 
boundary edges a ``$-$'' sign. Webs without boundary are called 
\textit{closed} webs. 

Any web can be obtained from the following 
elementary webs by glueing and disjoint union.
\begin{align}
\xy
   (0,0)*{\includegraphics[width=280px]{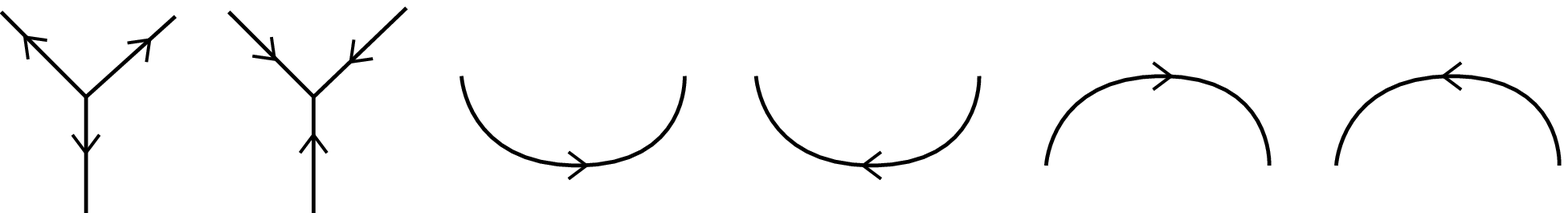}};
\endxy
\end{align}
Fixing a boundary $S$, we can form the 
$\mathbb{C}(q)$-vector space $W_S$, spanned by all webs with boundary $S$, 
modulo the following set of local relations (due to Kuperberg~\cite{kup}).
\begin{align}
\label{eq:circle}
\xy(0,0)*{\includegraphics[width=20px]{res/figs/intro/cirrem}}\endxy\;\; &=\;\; [3]\\
\label{eq:digon}
\xy(0,0)*{\includegraphics[width=70px]{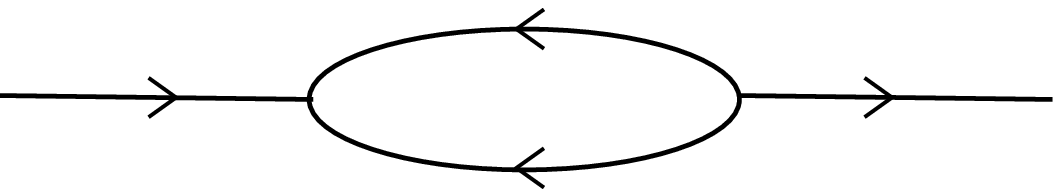}}\endxy\;\; &=\;\; [2]\;\; \xy(0,0)*{\includegraphics[width=50px]{res/figs/intro/line}}\endxy\\
\label{eq:square}
\xy(0,0)*{\includegraphics[width=50px]{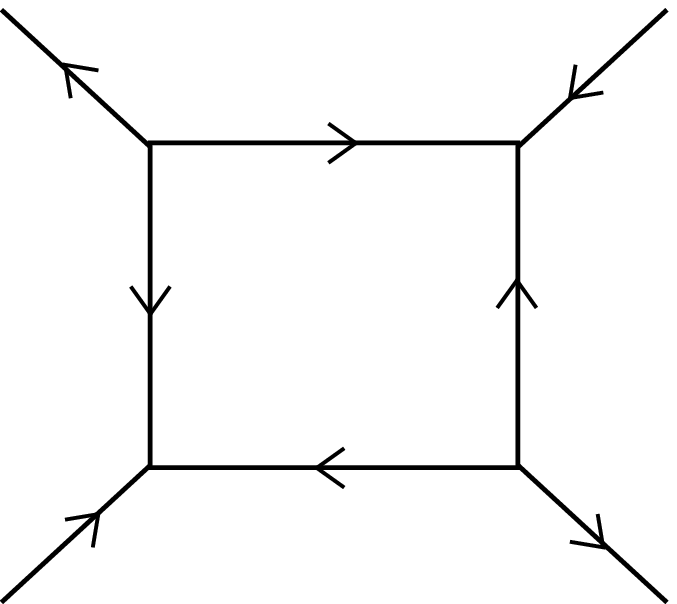}}\endxy\;\; &=\;\;\xy(0,0)*{\includegraphics[height=45px]{res/figs/intro/vert}}\endxy + \;\;\xy(0,0)*{\includegraphics[width=45px]{res/figs/intro/horiz}}\endxy
\end{align}
Recall that 
$$[a]=\frac{q^a-q^{-a}}{q-q^{-1}}=q^{a-1}+q^{a-3}+\cdots+q^{-(a-1)}\in\mathbb{N}[q,q^{-1}]$$ 
denotes the \textit{quantum integer}.

By abuse of notation, we will call all elements of $W_S$ webs. 
From relations~\eqref{eq:circle},~\eqref{eq:digon} and~\eqref{eq:square} it 
follows that any element in $W_S$ is a linear 
combination of webs with the same boundary and without circles, digons or 
squares. These are called \textit{non-elliptic webs}. 
As a matter of fact, the non-elliptic webs form a basis of $W_S$, which 
we call $B_S$. 
Therefore, we will simply call them \textit{basis webs}.
\vskip0.5cm
Let $W_S^{\mathbb{Z}}$ be the free $\mathbb{Z}[q,q^{-1}]$-submodule of 
$W_S$ generated by $B_S$. We call this the \textit{integral form} of 
the web space. 
\vskip0.5cm
Following Brundan and Stroppel's~\cite{bs1} notation for arc diagrams, 
we will write $w^*$ to denote the web obtained by reflecting a given 
web $w$ horizontally and reversing all 
orientations.
\begin{align}
\xy
 (0,0)*{\includegraphics[width=140px]{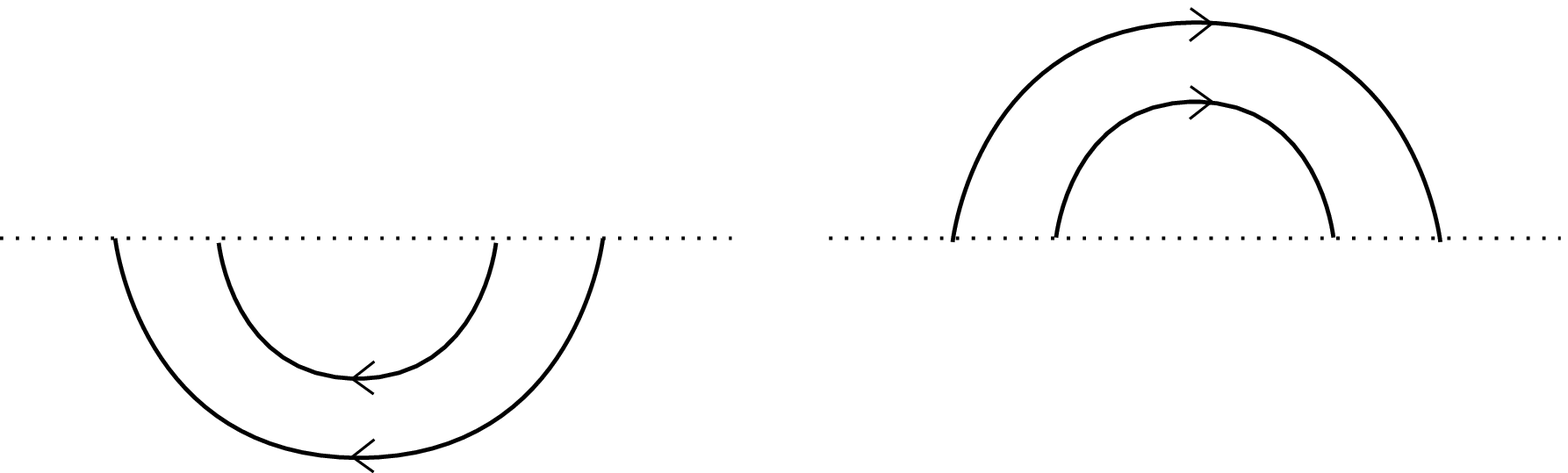}};
 (-13.5,2)*{w};
 (14,-2)*{w^*};
\endxy
\end{align}
By $uv^*$, we mean the planar diagram containing the disjoint union of $u$ and 
$v^*$, where $u$ lies vertically above $v^*$.
\begin{align}
\xy
 (0,0)*{\includegraphics[width=70px]{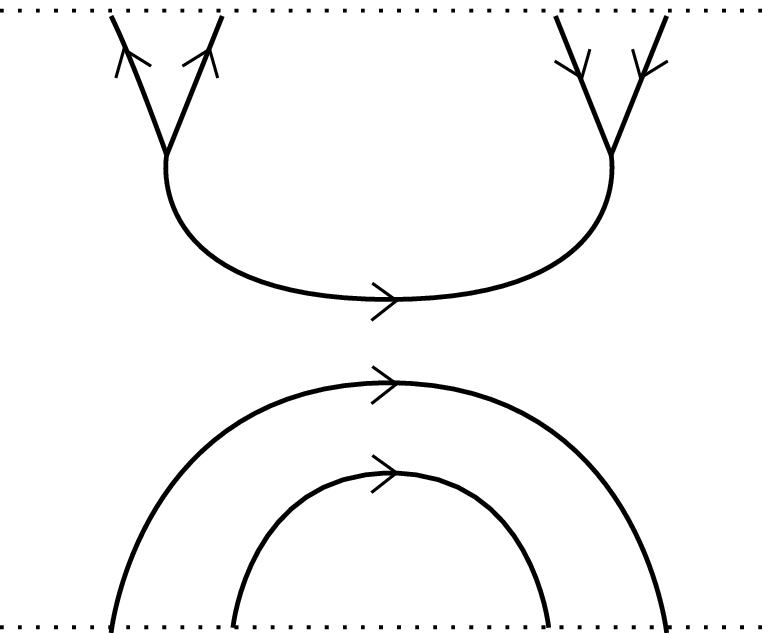}};
 (12,5)*{u};
 (12.5,-5)*{v^*};
\endxy
\end{align}
By $v^*u$, we shall mean the closed web obtained by glueing 
$v^*$ on top of $u$, when such a construction is possible (i.e. the number of free strands and orientations on the strands match).
\begin{align}\label{closed}
\xy
 (0,0)*{\includegraphics[width=70px]{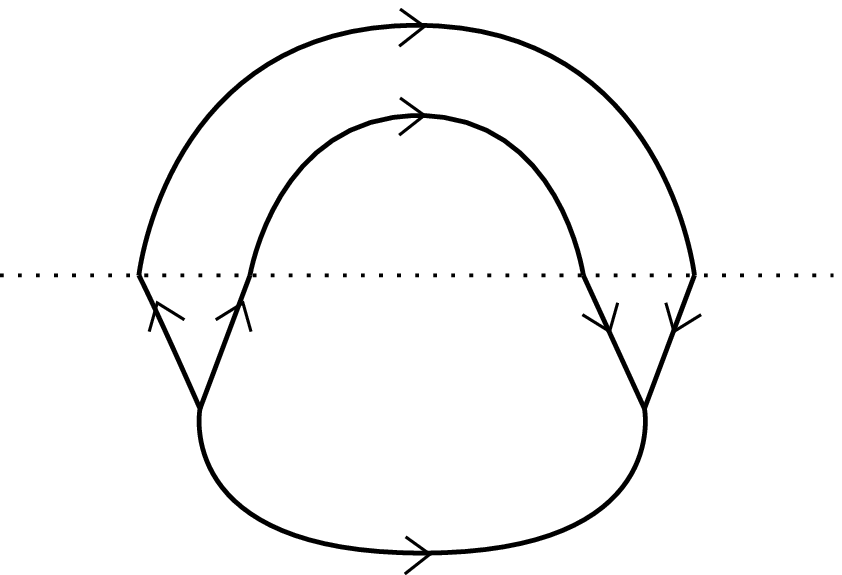}};
 (12,-4)*{u};
 (12.5,5)*{v^*};
\endxy
\end{align}
In the same vein, by $v_1^*u_1v_2^*u_2$ we denote the following web.
\begin{align}
\xy
 (0,0)*{\includegraphics[width=70px]{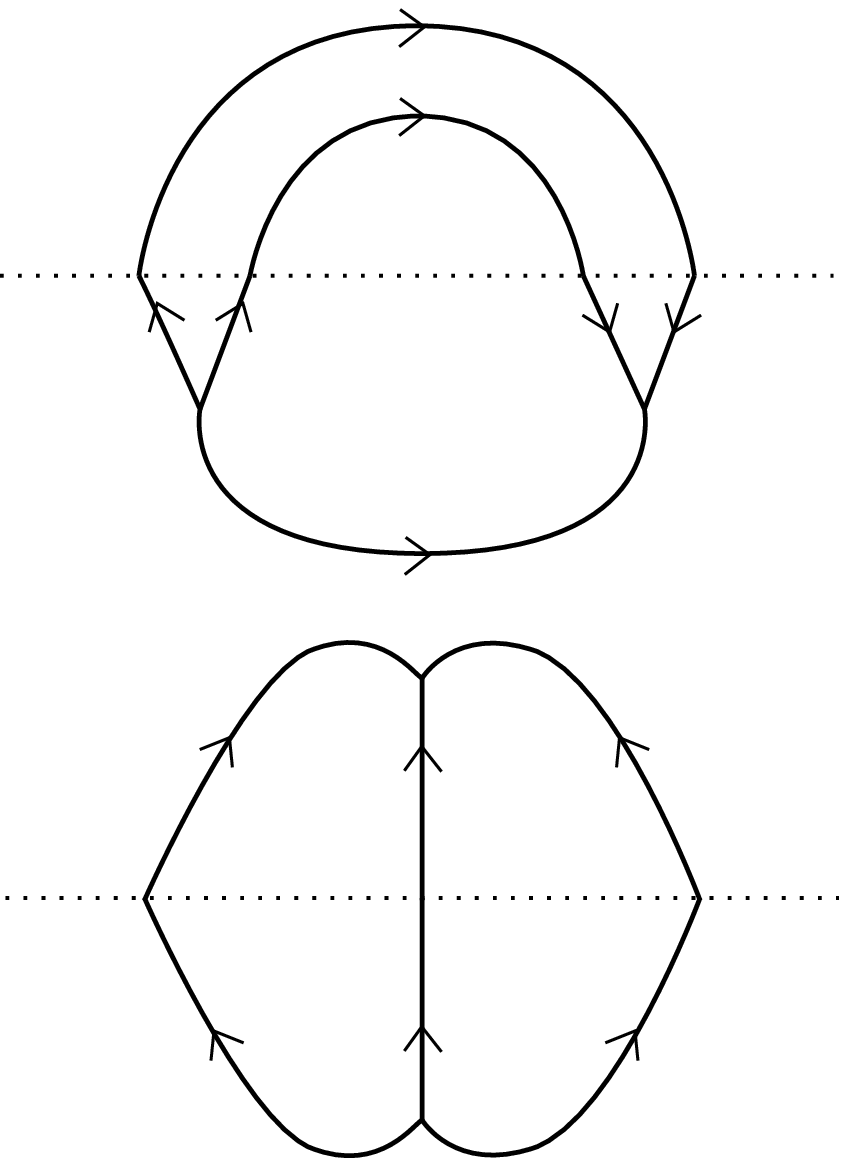}};
 (12.2,5)*{u_1};
 (12.4,-5.5)*{v_2^*};
 (12.2,-13)*{u_2};
 (12.4,13)*{v_1^*};
\endxy
\end{align}
\vskip0.5cm
To make the connection with the representation theory of 
$U_{q}(\mathfrak{sl}_3)$, we recall that 
a sign string $S=(s_1,\ldots,s_n)$ corresponds to 
\[
V_S=V_{s_1}\otimes \cdots\otimes V_{s_n},
\]
where $V_{+}$ is the fundamental representation and $V_{-}$ its dual. The 
latter is also isomorphic to $V_+\wedge V_+$, a fact which we will need 
later on. Both $V_+$ and $V_-$ have dimension three. Khovanov and Kuperberg~\cite{kk} 
use a particular basis for $V_+$, denoted $\{e^+_1,e^+_2,e^+_3\}$, and also one for $V_-$, denoted $\{e^-_1,e^-_2,e^-_3\}$. In this interpretation, 
webs correspond to intertwiners and 
\[
W_S\cong \mathrm{Inv}(V_S).
\]
Therefore, the elements of $B_S$ give a basis of $\mathrm{Inv}(V_S)$. 
However, this basis is not equal to the usual tensor basis. In 
Theorem 2 of~\cite{kk}, Kuperberg and Khovanov prove an important result 
concerning the change of basis matrix, which we will reproduce in 
Theorem~\ref{thm:upptriang}.
\vskip0.5cm
Kuperberg showed in~\cite{kup} (see also~\cite{kk}) that basis webs are indexed 
by closed weight lattice paths in the dominant Weyl chamber of 
$\mathfrak{sl}_3$. It is well-known that any path in the 
$\mathfrak{sl}_3$-weight lattice can be presented by a pair consisting of a 
sign string $S=(s_1,\ldots,s_n)$ and a 
\textit{state string} $J=(j_1,\ldots,j_n)$, with $j_i\in \{-1,0,1\}$ for all 
$1\leq i\leq n$. Given a pair $(S,J)$ representing a closed dominant path, 
a unique basis web (up to isotopy) is determined by a set of 
inductive rules called 
the \textit{growth algorithm}. We briefly recall the algorithm as 
described in~\cite{kk}. In fact, the algorithm can be applied to any path, but 
we will only use it for closed dominant paths. 

\begin{defn} \label{growth}
\textbf{(The growth algorithm)} 
Given $(S,J)$, a web $w^S_J$ is recursively generated by the following rules.
\begin{enumerate}
\item Initially, the web consists of $n$ parallel strands whose orientations are given by the sign string. If $s_{i} = +$, then the $i$-th strand is oriented 
upwards; if $s_{i} = -$, it is oriented downwards.
\item The algorithm builds the web downwards. Suppose we have already applied 
the algorithm $k-1$ times. For the $k$-th step, do the following. 
If the bottom boundary string contains a neighbouring pair of edges 
matching the top of one of 
the following webs (called H, arc and Y respectively), then glue 
the corresponding H, arc or Y to the relevant bottom boundary edges.
\begin{figure}[H]
  \centering
    \includegraphics[width=260px]{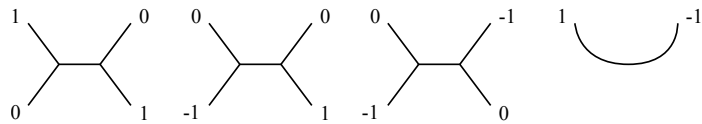}
    \caption{Top strands have different signs.}
\end{figure}
\begin{figure}[H]
  \centering
    \includegraphics[width=180px]{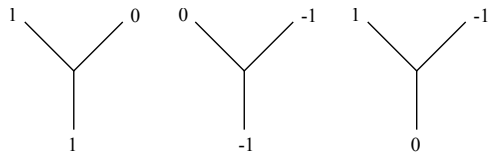}
    \caption{Top strands have same sign.}
\end{figure}
\end{enumerate}
These rules apply for any compatible orientation of the edges in the webs. 
Therefore, we have drawn them without any specific orientations. Below, 
whenever we write down an equation involving webs without orientations, 
we mean that the equation holds for all possible orientations. 
For future use, we will call the rules above the \textit{H, arc and Y-rule}. 
The growth algorithm stops if no further rules can be applied. 
\end{defn}
If $(S,J)$ represents a closed dominant path, then the growth algorithm 
produces a basis web. 

For example, the growth algorithm converts $S=(+-+-+++)$ and 
$J=(1,1,0,0,-1,0,-1)$ into the following basis web.
\begin{align}
\xy
 (0,0)*{\includegraphics[width=140px]{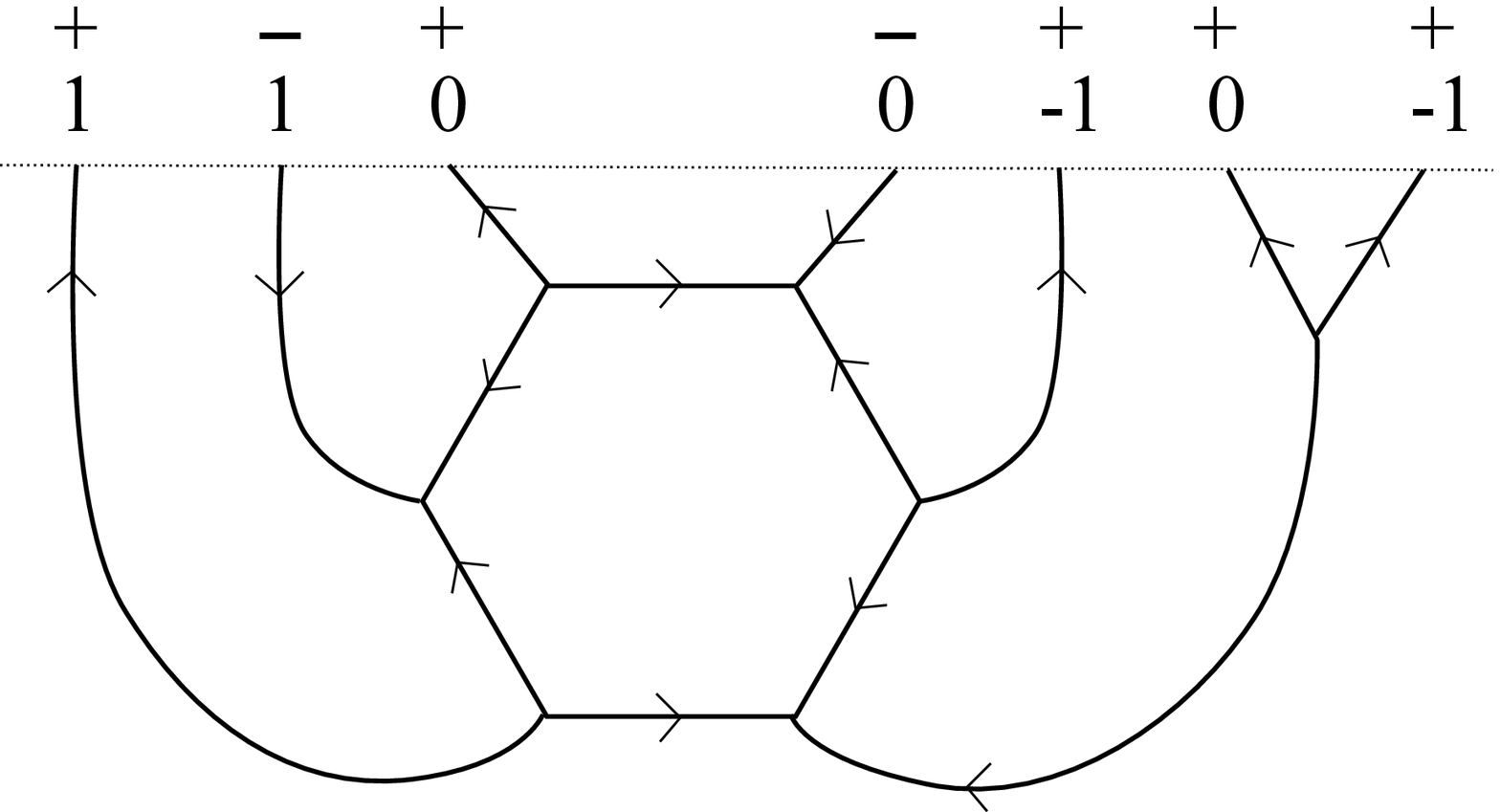}};
\endxy
\end{align}
In addition, the growth algorithm has an inverse, 
called the \textit{minimal cut path algorithm}~\cite{kk}, which we will 
not use here. 
\vskip0.5cm
Following Khovanov and Kuperberg in~\cite{kk}, we define a \textit{flow} $f$ on 
a web $w$ to be an oriented subgraph that contains exactly two of the three edges 
incident to each trivalent vertex. The connected components of the 
flow are called the \textit{flow lines}. The orientation of the flow lines 
need not agree with the orientation of $w$. Note that if $w$ is closed, then 
each flow line is a closed cycle. At the boundary, the flow lines 
can be represented by a state string $J$. By convention, at the $i$-th 
boundary edge, we set $j_i=+1$ if the flow line is oriented upward, $j_i=-1$ if 
the flow line is oriented downward and $j_i=0$ there is no flow line. The same 
convention determines a state for each edge of $w$. 
\begin{rem}
\label{rem:3color}
Every flow determines a unique 3-colouring of $w$, with colours $-1,0,1$, 
satisfying the property that, for any trivalent vertex of $w$, 
the colours of the three incident edges are all distinct. 
These colourings are called \textit{admissible} in~\cite{gornik}. 

Conversely, any such 3-colouring determines a unique flow on $w$. This 
correspondence determines a bijection between flows and admissible 
3-colourings on $w$.   

This remark will be important in Section~\ref{sec-webhowea} and in Section~\ref{sec-webhoweb}.
\end{rem}
We will also say that any flow $f$ that is compatible with a given state string $J$ 
on the boundary of $w$ \textit{extends} $J$. 

Given a web with a flow, denoted $w_f$, Khovanov and Kuperberg~\cite{kk} 
attribute a \textit{weight} to each trivalent vertex and each arc in $w_f$, 
as in Figures~\ref{weights} and~\ref{weights2}. 
The total weight of $w_f$ is by definition the sum of the 
weights at all trivalent vertices and arcs.
\begin{align}\label{weights}
  & \xy
 (0,0)*{\includegraphics[width=310px]{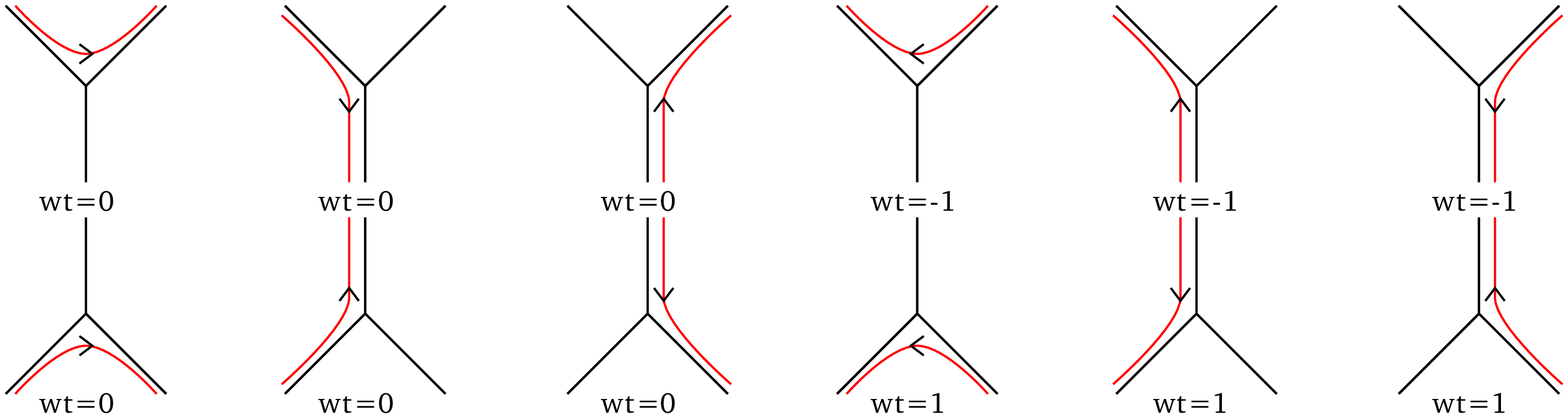}};
\endxy\\
   \label{weights2}
  & \xy
 (0,0)*{\includegraphics[width=310px]{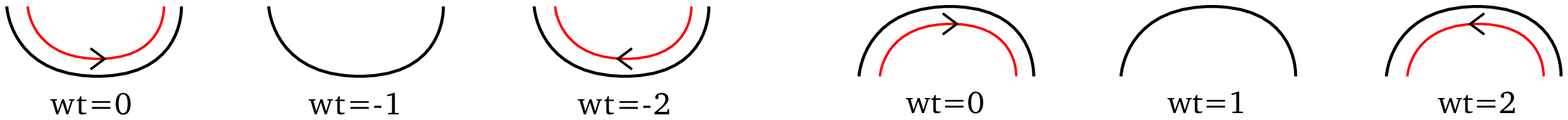}};
\endxy
\end{align}
For example, the following web has weight $-3$.
\begin{align}
\xy
 (0,0)*{\includegraphics[width=150px]{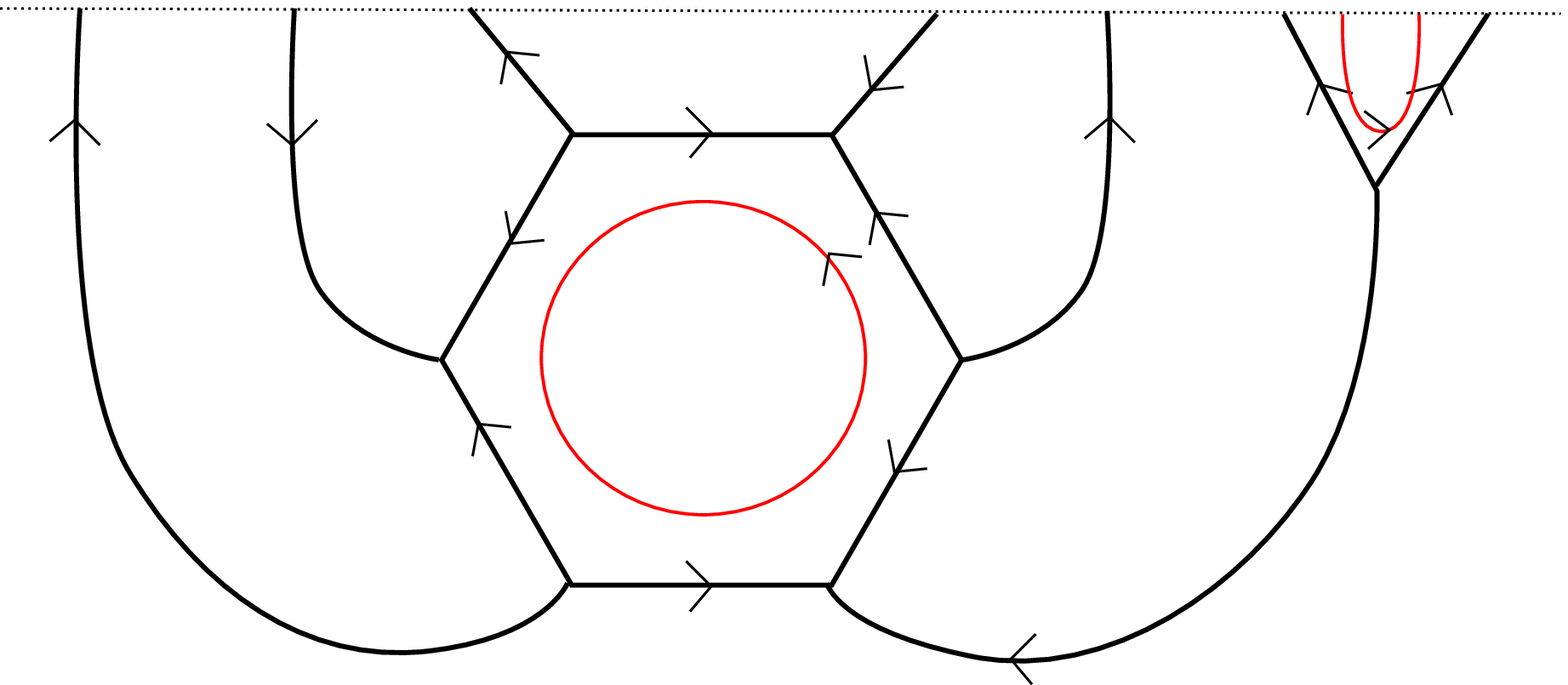}};
\endxy
\end{align}
We can extend the table in~\eqref{weights} and~\eqref{weights2} to calculate 
weights determined by flows on $H$'s, so that it becomes easier to 
compute the weight of $w_f$ when $w$ is expressed using the growth algorithm 
(Definition~\ref{growth}).

\begin{defn}\cite{kk}\label{defn-cano}
\textbf{(Canonical flows on basis webs)} Given a basis web $w$ expressed 
using the growth algorithm. 
We define the \textit{canonical flow} on $w$ by the following rules.
\begin{align} \label{canrule}
\xy
 (0,0)*{\includegraphics[width=250px]{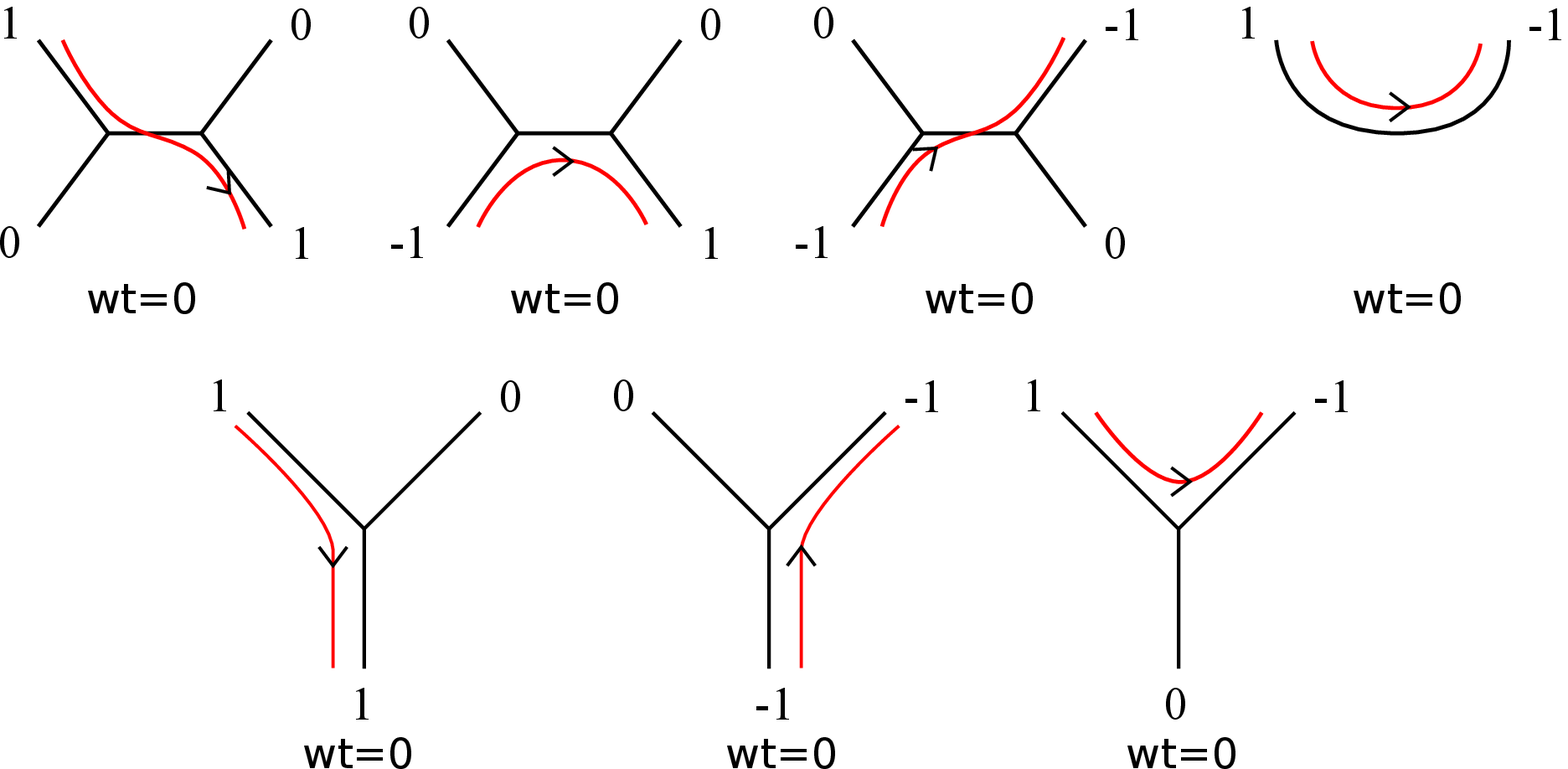}};
\endxy
\end{align}
The canonical flow does not depend on the 
particular instance of the growth algorithm that we have chosen to obtain 
$w$. 
\end{defn}
Observe that the definition of the canonical flows implies the following lemma. 

\begin{lem}
\label{lem:canflowzero}
A basis web with the canonical flow has weight zero.
\end{lem}
Given $(S,J)$, let 
\[
e^S_J=e^{s_1}_{j_1}\otimes\cdots\otimes e^{s_n}_{j_n}.
\] 
Khovanov and Kuperberg prove the following  
theorem (Theorem 2 in~\cite{kk}), which will be important for us 
in Section~\ref{sec-webhoweb}.
\begin{thm}\label{thm:upptriang}(\textbf{Khovanov-Kuperberg})
Given $(S,J)$, we have 
\[ 
w^S_J=e^S_J+\sum_{J^{\prime}<J}c(S,J,J^{\prime})e^S_{J^{\prime}}
\]
for some coefficients $c(S,J,J^{\prime})\in\mathbb{N}[q,q^{-1}]$, where the 
state strings $J$ and $J^{\prime}$ are ordered lexicographically.
\end{thm}
\begin{rem} 
\label{rem:counter}
Khovanov and Kuperberg~\cite{kk} show that $B_S$ is not 
equal to the dual canonical basis of $W_S$. This follows from the 
fact that $c(S,J,J^{\prime})\not\in q\mathbb{N}[q^{-1}]$ in general. 
In their Section 8, they give explicit counterexamples of elements $w\in B_S$ which 
admit non-canonical weight zero flows.

We note that the web $w$ shown in Example~\ref{ex-reso3} later in Section~\ref{sec-webbase} is related to one of the counterexamples of Khovanov and Kuperberg~\cite{kk}.
\end{rem}
\subsection{Basic definitions and background: Foams}\label{sec-webbasicb}
In this section we review the category called $\foamt$ of $\mathfrak{sl}_3$-foams 
introduced by Khovanov in~\cite{kh3}. As a matter of fact, we will also need 
a deformation of Khovanov's original category, due to Gornik~\cite{gornik} 
in the context of matrix factorizations, and studied in~\cite{mv1} in the 
context of foams. Therefore, we introduce a parameter $c\in \mathbb{C}$ in 
$\foamt^c$, just 
as in~\cite{mv1}, such that we get Khovanov's original category for $c=0$ and, 
for any $c\ne 0$, the category $\foamt^c$ is isomorphic to 
Gornik's deformation (his original deformation was for $c=1$). 
A big difference between these two 
specializations is that $\foamt^c$ is graded for $c=0$ and filtered for 
any $c\ne 0$. In fact, for any $c\ne 0$, the associated graded category of 
$\foamt^c$ is isomorphic to $\foamt^0$.  

We recall the following definitions as they appear in~\cite{mv1}. We note that the diagrams accompanying these definitions are taken, also, from~\cite{mv1}.

A \textit{pre-foam} is a cobordism with singular arcs between two webs. 
A singular arc in a pre-foam $U$ is the set of points of $U$ which have a neighborhood 
homeomorphic to the letter Y 
times an interval. Note that singular arcs are disjoint. 
Interpreted as morphisms, we read pre-foams from bottom to top by convention. Thus, pre-foam composition consists of placing one 
pre-foam on top of the other. The orientation of the singular arcs is, by convention, as in 
the diagrams below (called the \textit{zip} and the 
\textit{unzip} respectively).
\begin{equation*}
\figins{-20}{0.5}{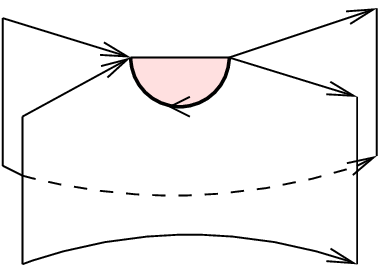}
\mspace{35mu}\mspace{35mu}
\figins{-20}{0.5}{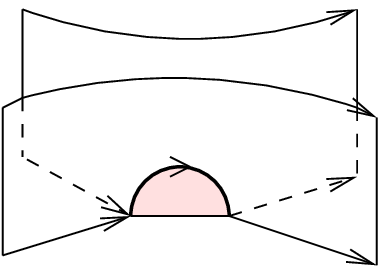}\ 
\end{equation*}
We allow pre-foams to have dots that can move freely about the facet on which they 
belong, but we do not allow a dot to cross singular arcs. 

By a \textit{foam}, we mean a formal $\mathbb{C}$-linear combination of 
isotopy classes of pre-foams modulo 
the ideal generated by the set of relations $\ell=(3D,NC,S,\Theta)$ and 
the \textit{closure relation}, 
as described below. 
\begin{gather*}
\figins{-7}{0.25}{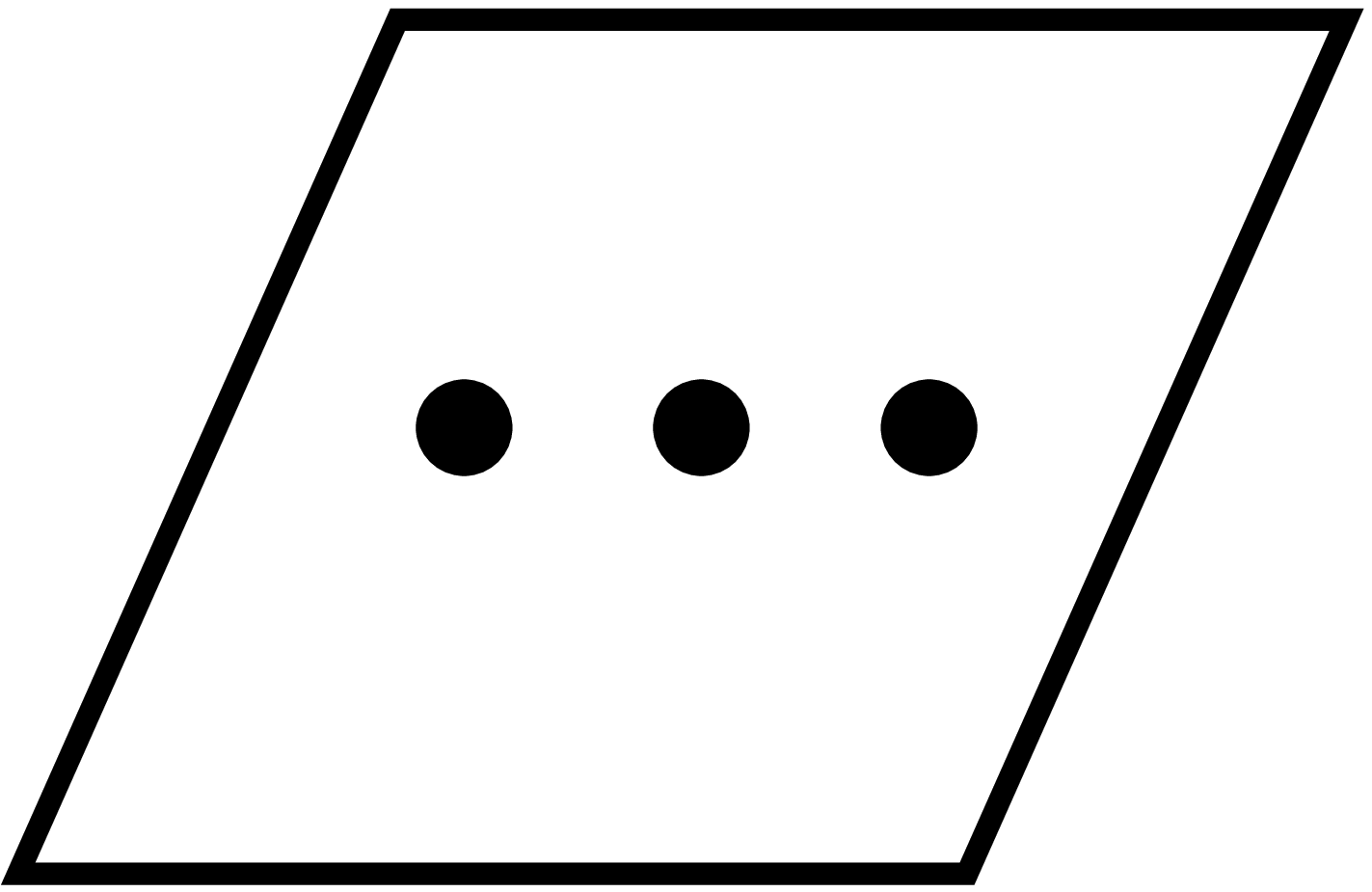}
= c\figins{-7}{0.25}{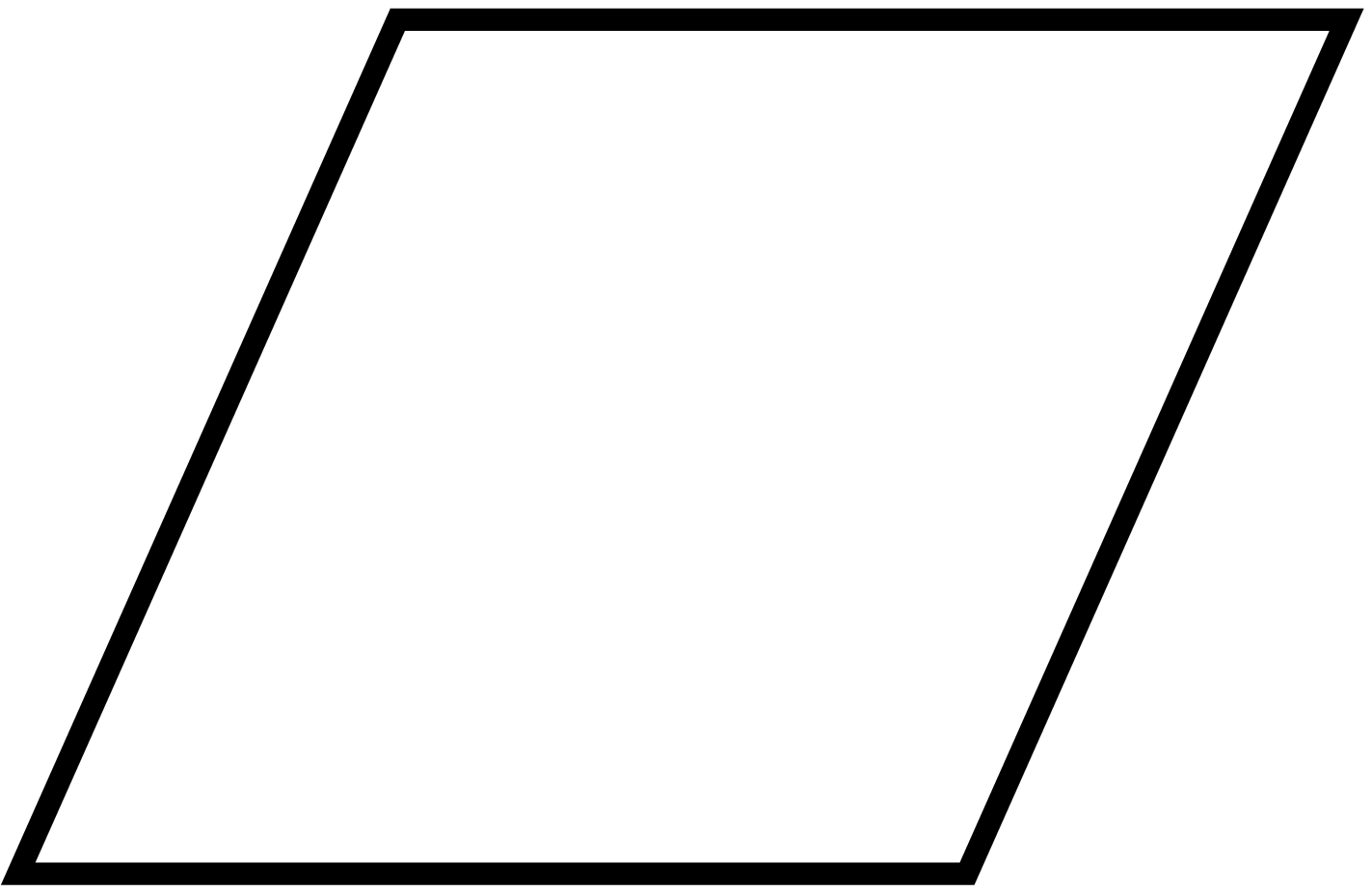}
\tag{3D}
\\[1.5ex]\displaybreak[0]
 \figwhins{-22}{0.65}{0.26}{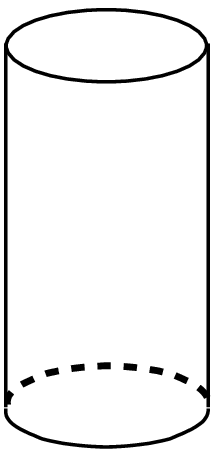}=
-\figwhins{-22}{0.65}{0.26}{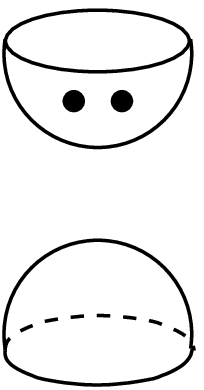}
-\figwhins{-22}{0.65}{0.26}{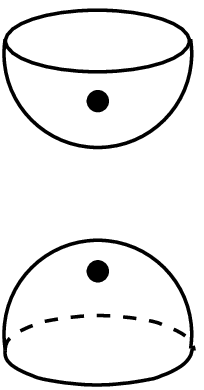}
-\figwhins{-22}{0.65}{0.26}{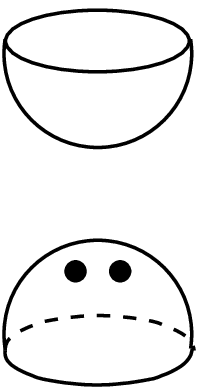}
\tag{NC}\label{eq:cn}
\\[1.5ex]\displaybreak[0]
\figins{-8}{0.3}{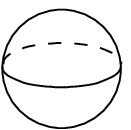}=
\figins{-8}{0.3}{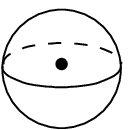}=0,\quad
\figins{-8}{0.3}{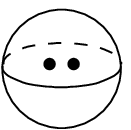}=-1
\tag{S}
\\[1.5ex]\displaybreak[0]
\labellist
\small\hair 2pt
\pinlabel $\alpha$ at 3 33
\pinlabel $\beta$ at -3 17
\pinlabel $\delta$ at 3 5
\endlabellist
\centering
\figins{-10}{0.4}{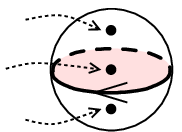} 
=\begin{cases}
\ \ 1, & (\alpha,\beta,\delta)=(1,2,0)\text{ or a cyclic permutation}, \\ 
-1, & (\alpha,\beta,\delta)=(2,1,0)\text{ or a cyclic permutation}, \\ 
\ \ 0, & \text{else}.
\end{cases}
\tag{$\Theta$}\label{eq:theta}
\end{gather*}
\begin{quote}
The \textit{closure relation}, i.e. any $\mathbb{C}$-linear combination of foams with the same boundary,
is equal to zero if and only if any way of capping off these foams with a 
common foam yields a $\mathbb{C}$-linear combination of closed foams 
whose evaluation is zero. 
\end{quote}

The relations in $\ell$ imply the following identities (for detailed proofs see~\cite{kh3}).
\begin{align}
\figwhins{-22}{0.65}{0.30}{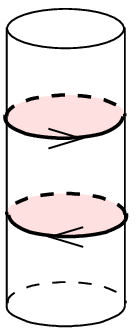}
& = -\ 
\figwhins{-22}{0.65}{0.30}{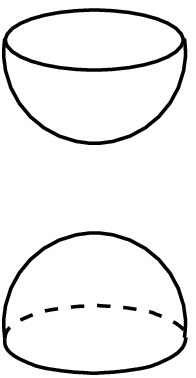}
\tag{Bamboo}\label{eq:bamboo}
\\[1.2ex]\displaybreak[0]
\figwhins{-22}{0.65}{0.30}{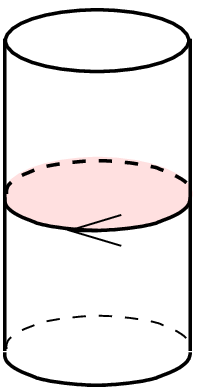}
& =\
\figwhins{-22}{0.65}{0.30}{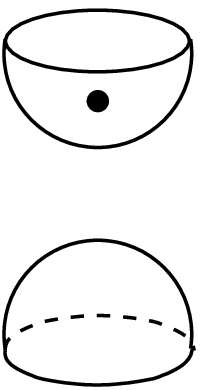}-
\figwhins{-22}{0.65}{0.30}{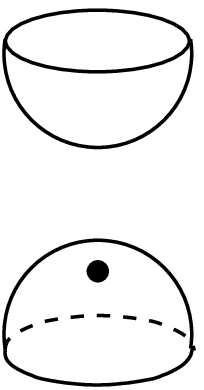}
 \tag{RD}\label{eq:rd}
\\[1.2ex]\displaybreak[0]
\figins{-12}{0.4}{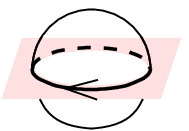} 
& =\ 0
\tag{Bubble}\label{eq:bubble}
\\[1.2ex]\displaybreak[0]
\figins{-20}{0.6}{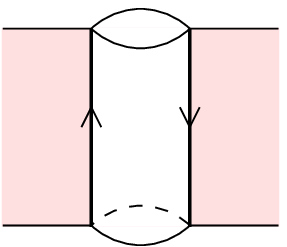}
& = 
\figins{-26}{0.75}{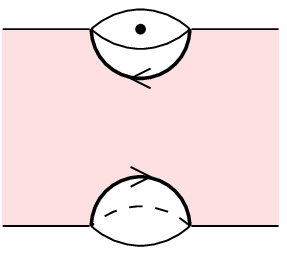}-
\figins{-26}{0.75}{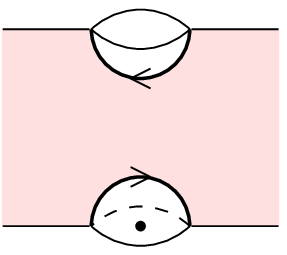}
\tag{DR}\label{eq:dr}
\\[1.2ex]\displaybreak[0]
\figins{-28}{0.8}{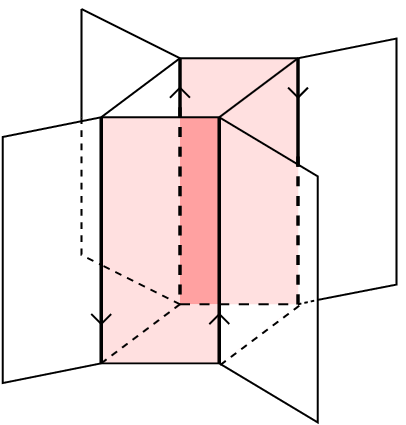}
&=
-\ \figins{-28}{0.8}{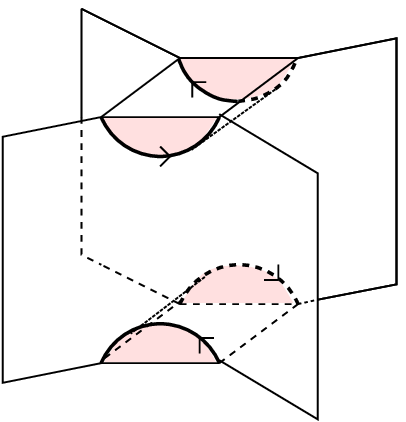}
-\figins{-28}{0.8}{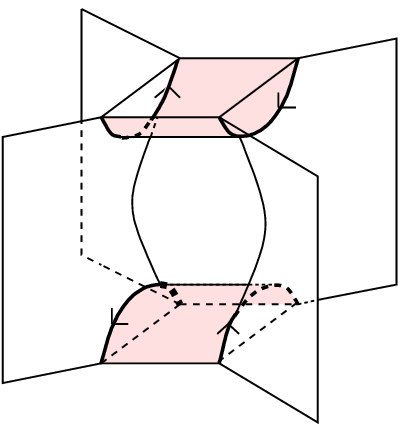}
\tag{SqR}\label{eq:sqr}
\end{align}

\begin{equation}\tag{Dot Migration}\label{eq:dotm}
\begin{split}
\figins{-22}{0.6}{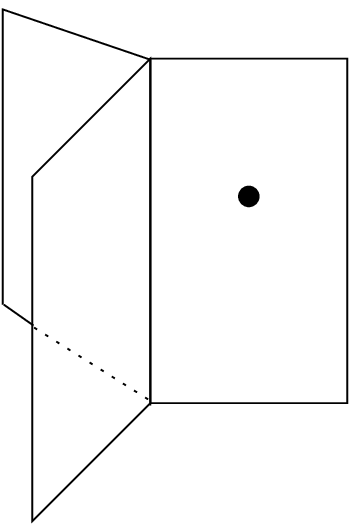}
\,+\,
\figins{-22}{0.6}{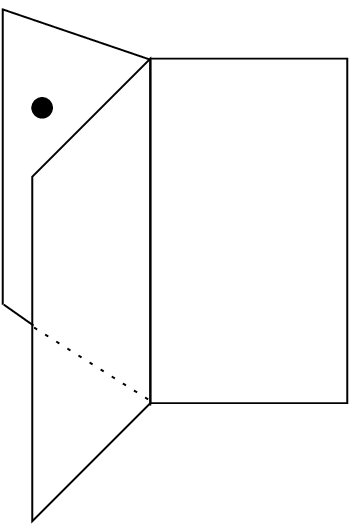}
\,+\,
\figins{-22}{0.6}{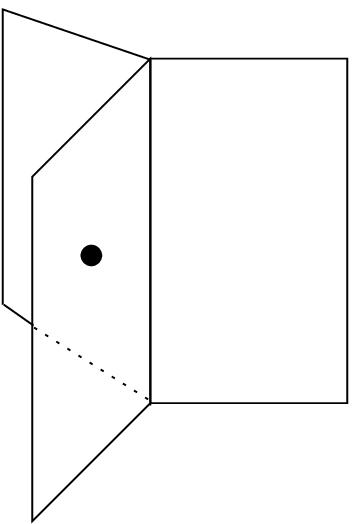}
\, &= 0
\\[1.2ex]
\figins{-22}{0.6}{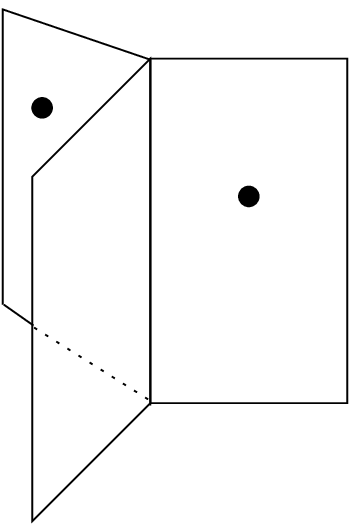}
\,+\,
\figins{-22}{0.6}{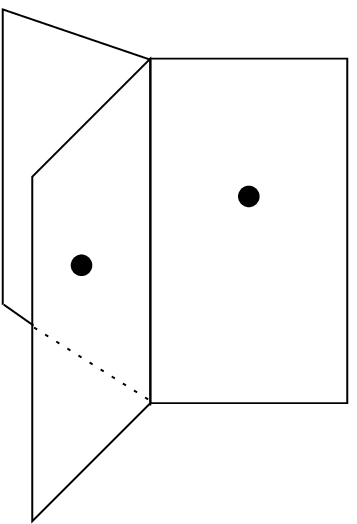}
\,+\,
\figins{-22}{0.6}{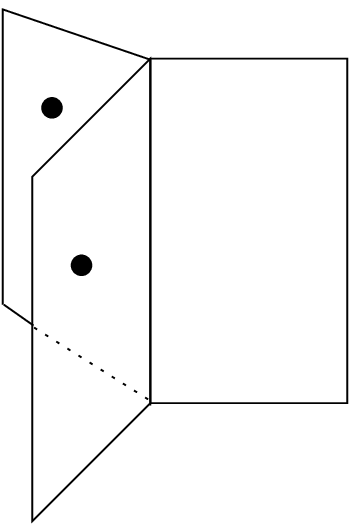}
\, &=\ 0
\\[1.2ex]
\figins{-22}{0.6}{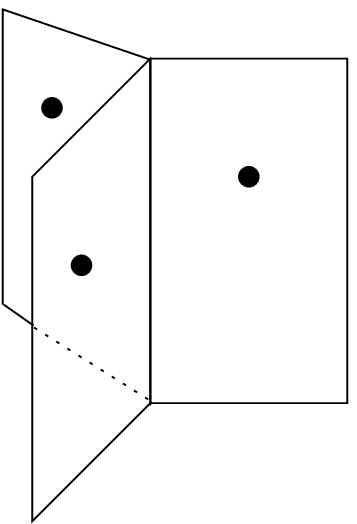}
\, &=c\;\figins{-22}{0.6}{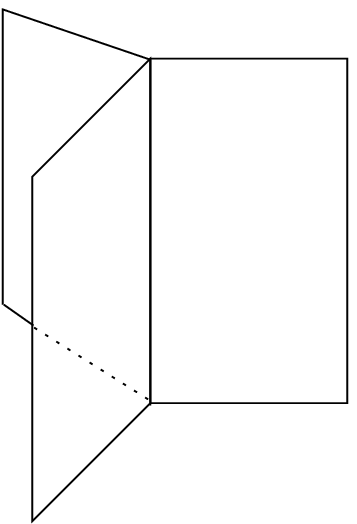}
\end{split}
\end{equation}
\begin{defn}
For any $c\in \mathbb{C}$, let $\foamt^c$ be the category whose 
objects are webs $\Gamma$ lying inside a horizontal strip in 
$\mathbb{R}^2$, which is bounded by the lines $y=0,1$ containing the boundary points of $\Gamma$. 
The morphisms of $\foamt^c$ are $\mathbb{C}$-linear combinations of foams lying inside the horizontal strip bounded by $y=0,1$ times the unit interval. We require that the vertical boundary of each foam is a set (possibly empty) of vertical lines.
\end{defn}
The \textit{$q$-grading} of a foam $U$ is defined as 
\[
q(U) = \chi(\partial U) - 2\chi(U) + 2d+b,
\] 
where $\chi$ denotes the Euler characteristic, $d$ is 
the number of dots on $U$ and $b$ is the number of vertical boundary 
components. This makes $\foamt^0$ into a graded category. For any 
$c\ne 0$, this makes $\foamt^c$ into a filtered category, whose 
associated graded category is isomorphic to $\foamt^0$. 
\begin{defn}\cite{kh3} \label{foamhom} \textbf{(Foam homology)} Given a web 
$w$ the \textit{foam homology} of $w$ is the complex vector space, 
$\F^c(w)$, spanned by all foams  
\[U\colon\emptyset \to w\]
in $\foamt^c$.
\end{defn}
The complex vector space $\F^c(w)$ is filtered/graded by 
the $q$-grading on foams and has rank 
$\langle w\rangle$, where $\langle w\rangle $ is the \textit{Kuperberg bracket} 
computed recursively by the rules below.
\begin{enumerate}
\item $\left\langle w \amalg \raisebox{-2mm}
{\includegraphics[width=15px]{res/figs/intro/cirrem}}\right\rangle = [3]\langle w\rangle$.
\item $\langle  \raisebox{-0.5mm}{\includegraphics[width=45px]{res/figs/section21/dgrem}}\rangle  
= [2]\langle  \raisebox{0.5mm}{\includegraphics[width=30px]{res/figs/intro/line}}\rangle $.
\item $\left\langle \raisebox{-3mm}{\includegraphics[width=25px]{res/figs/section21/sqrem}}
\right\rangle  = 
\left\langle  \raisebox{-3mm}{\includegraphics[width=15px]{res/figs/intro/vert}}\right
\rangle  + \left\langle  \raisebox{-1.9mm}{\includegraphics[width=23px]{res/figs/intro/horiz}}
\right\rangle $.
\end{enumerate}
The relations above correspond to the decomposition of $\F^c(w)$ into direct 
summands. The idempotents corresponding to these direct summands are the 
terms on the r.h.s. of the relations (NC), (DR) and (SqR), respectively.
For any $c\ne 0$, the complex vector space $\F^c(w)$ is filtered and 
its associated graded vector space is $\F^0(w)$. See~\cite{kh3},~\cite{mv1} for details. 
\begin{rem}
\label{rem:flowsfoams}
Given $u,v\in B_S$, the observations above 
and Theorem~\ref{thm:upptriang} 
show that there exists a homogeneous basis of 
$\F^0(u^*v)$ parametrised by the flows on 
$u^*v$. We have, in fact, constructed such a basis, but it is not unique. See Section~\ref{sec-webbase}.
There is 
also no ``preferred choice'', unless one requires the basis to have other 
nice properties, e.g. in the $\mathfrak{sl}_2$ case, 
Brundan and Stroppel prove that there is a cellular basis of $H_n$ (in fact, it is also graded cellular~\ref{ex-cellular2}). The 
construction of a ``good'' basis of the $\mathfrak{sl}_3$ web algebra 
$K_S$ (and similarly for Gornik's deformation $G_S$) 
is still work in progress and will, 
hopefully, be the contents of a paper. Although we do not need 
such a ``good'' basis here, it is important that the reader keep this remark 
in mind while reading Section~\ref{sec-webhoweb}. 
\end{rem}
\subsection{Basic definitions and background: Quantum 2-algebras}\label{sec-webbasicc}
\subsubsection{The quantum general and special linear algebras}
First we recall the quantum general and special linear algebras. Most parts in this section are copied from section two and three in~\cite{msv1}. Note that, in contrast to~\cite{msv1}, we work over $\mathbb{C}(q)$ instead of $\mathbb{Q}(q)$.

The $\mathfrak{gl}_n$-weight lattice is isomorphic to $\bZ^n$. Let 
$\epsilon_i=(0,\ldots,1,\ldots,0)\in \bZ^n$, with $1$ being on the $i$-th 
coordinate, and $\alpha_i=\epsilon_i-\epsilon_{i+1}
=(0,\ldots,1,-1,\ldots,0)\in\bZ^{n}$, for 
$i=1,\ldots,n-1$. Recall that the Euclidean inner product on $\bZ^n$ is defined by  
$(\epsilon_i,\epsilon_j)=\delta_{i,j}$. 
   
\begin{defn} For $n\in\bN_{>1}$ the \textit{quantum general linear algebra} 
${\mathbf U}_q(\mathfrak{gl}_n)$ is 
the associative unital $\bC(q)$-algebra generated by $K_i$ and $K_i^{-1}$, for $1,\ldots, n$, 
and $E_{\pm i}$ (beware that some authors use $F_i$ instead of $E_{-i}$), for $i=1,\ldots, n-1$, subject to the relations
\begin{gather*}
K_iK_j=K_jK_i,\quad K_iK_i^{-1}=K_i^{-1}K_i=1,
\\
E_iE_{-j} - E_{-j}E_i = \delta_{i,j}\dfrac{K_iK_{i+1}^{-1}-K_i^{-1}K_{i+1}}{q-q^{-1}},
\\
K_iE_{\pm j}=q^{\pm (\epsilon_i,\alpha_j)}E_{\pm j}K_i,
\\
E_{\pm i}^2E_{\pm j}-(q+q^{-1})E_{\pm i}E_{\pm j}E_{\pm i}+E_{\pm j}E_{\pm i}^2=0,
\qquad\text{if}\quad |i-j|=1,
\\
E_{\pm i}E_{\pm j}-E_{\pm j}E_{\pm i}=0,\qquad\text{else}.
\end{gather*} 
\end{defn}

\begin{defn} 
\label{defn:qsln}
For $n\in\bN_{>1}$ the \textit{quantum special linear algebra} 
${\mathbf U}_q(\mathfrak{sl}_n)\subseteq {\mathbf U}_q(\mathfrak{gl}_n)$ is 
the unital $\bC(q)$-subalgebra generated by $K_iK^{-1}_{i+1}$ and 
$E_{\pm i}$, for $i=1,\ldots, n-1$.
\end{defn}

Recall that the ${\mathbf U}_q(\mathfrak{sl}_n)$-weight lattice is 
isomorphic to $\bZ^{n-1}$. Suppose that $V$ is a 
${\mathbf U}_q(\mathfrak{gl}_n)$-weight representation with 
weights $\lambda=(\lambda_1,\ldots,\lambda_n)\in\bZ^n$, i.e. 
\[
V\cong \bigoplus_{\lambda}V_{\lambda},
\] 
and $K_i$ acts as multiplication by 
$q^{\lambda_i}$ on $V_{\lambda}$. Then $V$ is also a 
${\mathbf U}_q(\mathfrak{sl}_n)$-weight representation with weights 
$\overline{\lambda}=(\overline{\lambda}_1,\ldots,\overline{\lambda}_{n-1})\in
\bZ^{n-1}$ such that 
$\overline{\lambda}_j=\lambda_j-\lambda_{j+1}$ for $j=1,\ldots,n-1$.
 
Conversely, given a ${\mathbf U}_q(\mathfrak{sl}_n)$-weight 
representation with weights $\mu=(\mu_1,\ldots,\mu_{n-1})$, there is not a 
unique choice of ${\mathbf U}_q(\mathfrak{gl}_n)$-action on $V$. We can 
fix this by choosing the action of $K_1,\cdots, K_n$. In terms of weights, this 
corresponds to the observation that, for any $d\in\bZ$, the equations 
\begin{align}
\label{eq:sl-gl-wts1}
\lambda_i-\lambda_{i+1}&=\mu_i,\\
\label{eq:sl-gl-wts2}
\qquad \sum_{i=1}^{n}\lambda_i&=d,
\end{align}  
determine $\lambda=(\lambda_1,\ldots,\lambda_n)$ uniquely, 
if there exists a solution to~\eqref{eq:sl-gl-wts1} and~\eqref{eq:sl-gl-wts2} 
at all. To fix notation, we 
define the map $\phi_{n,d}\colon \bZ^{n-1}\to \bZ^{n}\cup \{*\}$ by 
\[
\phi_{n,d}(\mu)=\lambda, 
\]
if~\eqref{eq:sl-gl-wts1} and~\eqref{eq:sl-gl-wts2} have a solution, and we
put $\phi_{n,d}(\mu)=*$ otherwise.   

Note that ${\mathbf U}_q(\mathfrak{gl}_n)$ and 
${\mathbf U}_q(\mathfrak{sl}_n)$ are both Hopf algebras, which implies that 
the tensor product of two of their representations is a representation again. 

Both ${\mathbf U}_q(\mathfrak{gl}_n)$ and ${\mathbf U}_q(\mathfrak{sl}_n)$ 
have plenty of non-weight representations, but we will not discuss them here. 
Therefore we can restrict our attention to the 
Beilinson-Lusztig-MacPherson~\cite{blm} idempotent version of these 
quantum groups, denoted $\Ugl$ and $\U$ respectively. It is worth noting, as explained in (f) of Example~\ref{ex-2cat}, that such algebras can be seen as $1$-categories.
 
To understand their definition, recall that $K_i$ acts as $q^{\lambda_i}$ on the 
$\lambda$-weight space of any weight representation. 
For each $\lambda\in\bZ^n$ adjoin an idempotent $1_{\lambda}$ to 
${\mathbf U}_q(\mathfrak{gl}_n)$ and add 
the relations
\begin{align*}
1_{\lambda}1_{\mu} &= \delta_{\lambda,\nu}1_{\lambda},   
\\
E_{\pm i}1_{\lambda} &= 1_{\lambda\pm\alpha_i}E_{\pm i},
\\
K_i1_{\lambda} &= q^{\lambda_i}1_{\lambda}.
\end{align*}
\begin{defn} 
\label{defn:Uglndot}
The idempotented quantum general linear algebra is defined by 
\[
\Ugl=\bigoplus_{\lambda,\mu\in\bZ^n}1_{\lambda}{\mathbf U}_q(\mathfrak{gl}_n)1_{\mu}.
\]
\end{defn}
\noindent Let $I=\{1,2,\ldots,n-1\}$. In the sequel we use \textit{signed sequences} 
$\ii=(\alpha_1i_1,\ldots,\alpha_mi_m)$, 
for any $m\in\bN$, $\alpha_j\in\{\pm 1\}$ and $i_j\in I$. 
We denote the set of signed sequences 
by $\sseq$.

For such an $\ii=(\alpha_1 i_1,\ldots,\alpha_{n-1}i_{n-1})$ we define
\[
E_{\ii}=E_{\alpha_1 i_1}\cdots E_{\alpha_{n-1} i_{n-1}}
\] 
and we define $\ii_{\Lambda}\in\bZ^n$ to be the $n$-tuple such that 
\[
E_{\ii}1_{\mu}=1_{\mu + \ii_{\Lambda}}E_{\ii}.
\]

Similarly, for ${\mathbf U}_q(\mathfrak{sl}_n)$, adjoin an idempotent $1_{\mu}$ 
for each $\mu\in\bZ^{n-1}$ and add the relations
\begin{align*}
1_{\mu}1_{\nu} &= \delta_{\mu,\nu}1_{\lambda},   
\\
E_{\pm i}1_{\mu} &= 1_{\mu\pm\overline{\alpha}_i}E_{\pm i},\quad\text{with}\;
\overline{\alpha}_i=\alpha_i-\alpha_{i+1},
\\
K_iK^{-1}_{i+1}1_{\mu} &= q^{\mu_i}1_{\mu}.
\end{align*}
\begin{defn} The idempotented quantum special linear algebra is defined by 
\[
\U=\bigoplus_{\mu,\nu\in\bZ^{n-1}}1_{\mu}{\mathbf U}_q(\mathfrak{sl}_n)1_{\nu}.
\]
\end{defn}
\noindent Note that $\Ugl$ and $\U$ are both non-unital algebras, 
because their units 
would have to be equal to the infinite sum of all their idempotents.
 
Furthermore, the only ${\mathbf U}_q(\mathfrak{gl}_n)$ 
and ${\mathbf U}_q(\mathfrak{sl}_n)$-representations which factor through 
$\Ugl$ and $\U$ respectively are the weight representations. 
Finally, note that there is no embedding of $\U$ into $\Ugl$, because 
there is no embedding of the $\mathfrak{sl}_n$-weights into the 
$\mathfrak{gl}_n$-weights.  
\vskip0.5cm
Finally, recall the \textit{integral forms} of these quantum algebras. For each 
$i=1,\ldots, n-1$ and each $a\in\mathbb{N}$, define the \textit{divided power}
$$E_{\pm i}^{(a)}=\dfrac{E_{\pm i}^a}{[a]!}.$$
\begin{defn}
Let $\UZ(\mathfrak{gl}_n)\subset \Ugl$ and $\UZ(\mathfrak{sl}_n)\subset 
\U$ be the $\mathbb{Z}[q,q^{-1}]$-subalgebras generated by the divided 
powers $E_{\pm i}^{(a)}1_{\lambda}$.  
\end{defn}
\subsubsection{The $q$-Schur algebra}
Let $d\in\bN$ and let $V$ be the natural $n$-dimensional representation of 
${\mathbf U}_q(\mathfrak{gl}_n)$. Define  
\[
\Lambda(n,d)=\{\lambda\in \bN^n\mid\,\, 
\sum_{i=1}^{n}\lambda_i=d\}\quad\text{and}
\]  
\[
\Lambda^+(n,d)=\{\lambda\in\Lambda(n,d)\mid d\geq 
\lambda_1\geq\lambda_2\geq\cdots 
\geq\lambda_n\geq 0\}.
\] 
Recall that the weights in $V^{\otimes d}$ are precisely the elements of 
$\Lambda(n,d)$, and that the highest weights are the elements of 
$\Lambda^+(n,d)$.  
The highest weights correspond exactly to the irreducibles $V_{\lambda}$ 
that show up in the decomposition of $V^{\otimes d}$. 

We can define the $q$-Schur algebra as follows. 
\begin{defn}
The \textit{$q$-Schur algebra} $S_q(n,d)$ is the image 
of the representation $\psi_{n,d}$ defined by
$$\psi_{n,d}\colon {\mathbf U}_q(\mathfrak{gl}_n)\to 
\End_{\mathbb{C}}(V^{\otimes d}).$$
\end{defn}
Recall that for $\lambda\in\Lambda^+(n,d)$, the 
${\mathbf U}_q(\mathfrak{gl}_n)$-action on $V_{\lambda}$ factors through 
the projection given by $\psi_{n,d}\colon {\mathbf U}_q(\mathfrak{gl}_n)\to S_q(n,d)$. 
This way we obtain all irreducible representations of $S_q(n,d)$. Note that 
this also implies that all representations of $S_q(n,d)$ have a 
weight decomposition. As a matter of fact, it is well-known that 
\[
S_q(n,d)\cong \prod_{\lambda\in\Lambda^+(n,d)}\End_{\mathbb{C}}(V_{\lambda}).
\]
Therefore $S_q(n,d)$ is a finite dimensional, semisimple, 
unital algebra and its dimension is equal to 
\[
\sum_{\lambda\in\Lambda^+(n,d)}\dim(V_{\lambda})^2=\binom{n^2+d-1}{d}. 
\]
Since $V^{\otimes d}$ is a weight representation, 
$\psi_{n,d}$ gives rise to a homomorphism 
$\Ugl\to S_q(n,d)$, for 
which we use the same notation. This map is still surjective and 
Doty and Giaquinto, in Theorem 2.4 of~\cite{dogi}, showed that 
the kernel of $\psi_{n,d}$ is equal to the ideal generated by all 
idempotents $1_{\lambda}$ such that 
$\lambda\not\in\Lambda(n,d)$. Clearly the kernel of $\psi_{n,d}$ is isomorphic to
$S_q(n,d)$. By the above observations, 
we see that $S_q(n,d)$ has a Serre presentation. As a matter of fact, 
by Corollary 4.3.2 in~\cite{chgi}, this presentation is simpler than 
that of $\Ugl$, i.e. one does not need to impose the last two Serre relations, 
involving cubical terms, because they are implied by the other relations 
and the finite 
dimensionality.  
\begin{lem} 
$S_q(n,d)$ is isomorphic to the associative, unital 
$\bC(q)$-algebra generated by $1_{\lambda}$, for $\lambda\in\Lambda(n,d)$, 
and $E_{\pm i}$, for $i=1,\ldots,n-1$, subject to the relations
\begin{align}
\label{eq:schur1} 1_{\lambda}1_{\mu} &= \delta_{\lambda,\mu}1_{\lambda},
\\[0.5ex]
\label{eq:schur2}
\sum_{\lambda\in\Lambda(n,d)}1_{\lambda} &= 1,
\\[0.5ex]
\label{eq:schur3}
E_{\pm i}1_{\lambda} &= 1_{\lambda\pm\alpha_i}E_{\pm i},\quad\text{with}\;\alpha_i=\epsilon_i-\epsilon_{i+1}=(0,\dots ,1,-1,\dots ,0),
\\[0.5ex]
\label{eq:schur4}
E_iE_{-j}-E_{-j}E_i &= \delta_{ij}\sum\limits_{\lambda\in\Lambda(n,d)}
[\overline{\lambda}_i]1_{\lambda}.
\end{align}
We use the convention that $1_{\mu}X1_{\lambda}=0$, if $\mu$ 
or $\lambda$ is not contained in $\Lambda(n,d)$. Again $[a]$ denotes the 
$q$-integer from before.
\end{lem}

Although there is no embedding of $\U$ into $\Ugl$, the projection 
\[
\psi_{n,d}\colon{\mathbf U}_q(\mathfrak{gl}_n)\to S_q(n,d)
\] 
can be restricted to ${\mathbf U}_q(\mathfrak{sl}_n)$ and is still surjective. 
This gives rise to the surjection 
\[
\psi_{n,d}\colon \U\to S_q(n,d),
\]
defined by 
\begin{equation}
\label{eq:psi}
\psi_{n,d}(E_{\pm i}1_{\lambda})=E_{\pm i}1_{\phi_{n,d}(\lambda)},
\end{equation}
where $\phi_{n,d}$ was defined below equations~\eqref{eq:sl-gl-wts1} and~\eqref{eq:sl-gl-wts2}. By convention we put $1_{*}=0$.   
\vskip0.5cm 
Just for completeness, let us also recall the integral form of the 
$q$-Schur algebra. 
\begin{defn}
Define $S_q^{\mathbb{Z}}(n,d)\subset S_q(n,d)$ to be the $\mathbb{Z}[q,q^{-1}]$-subalgebra generated by the divided powers $E_{\pm i}^{(a)}1_{\lambda}$.  
\end{defn}
\subsubsection{The general and special quantum 2-algebras}
We note that a lot of this section is copied from~\cite{msv1}. The reader can find even more details there.

Let $\Ucat$ be Khovanov and Lauda's~\cite{kl5} 
diagrammatic categorification of $\U$. It is worth noting, as explained in (f) of Example~\ref{ex-2cat}, that such a categorification has to be a $2$-category.
 
In~\cite{msv1} it was shown that there is a quotient 2-category of 
$\Ucat$, denoted $\Scat(n,n)$, which categorifies $S_q(n,n)$. Note that ``categorifies'' should be in the sense of Section~\ref{sec-techgrgrcat}. 

We recall the definition of these categorified quantum algebras and some notions from above. 
As before, let $I=\{1,2,\ldots,n-1\}$. Again, we use \textit{signed sequences} 
$\ii=(\alpha_1i_1,\ldots,\alpha_mi_m)$, 
for any $m\in\bN$, $\alpha_j\in\{\pm 1\}$ and $i_j\in I$, and 
the set of signed sequences is denoted $\sseq$. 
For $\ii=(\alpha_1i_1,\ldots,\alpha_mi_m)\in\sseq$ we 
define $\ii_{\Lambda}=\alpha_1 (i_1)_{\Lambda}+\cdots+\alpha_m (i_m)_{\Lambda}$, 
where 
\[
(i_j)_{\Lambda}=(0,0,\ldots,1,-1,0\ldots,0),
\]
such that the vector starts with $i_j-1$ and ends with $k-1-i_j$ zeros. 
We also define the symmetric $\bZ$-valued bilinear form on $\bC[I]$ 
by $i\cdot i=2$, $i\cdot (i+1)=-1$ and $i\cdot j=0$, for $\vert i-j\vert>1$. 
Recall that $\overline{\lambda}_i=\lambda_i-\lambda_{i+1}$.
\vskip0.5cm
We first recall the definition, given in~\cite{msv1}, of the 2-category which 
conjecturally categorifies $\Ugl$. It is 
a straightforward adaptation of Khovanov and Lauda's $\Ucat$. 
\begin{defn} \label{def_glcat} $\glcat$ is an
additive $\bC$-linear 2-category. The 2-category $\glcat$ consists of
\begin{itemize}
  \item Objects are $\lambda\in\bZ^n$.
\end{itemize}
The hom-category $\glcat(\lambda,\lambda')$ between two objects 
$\lambda$, $\lambda'$ is an additive $\bC$-linear category 
consisting of the following.
\begin{itemize}
  \item Objects\footnote{We refer to objects of the category
$\glcat(\lambda,\lambda')$ as 1-morphisms of $\glcat$.  Likewise, the morphisms of
$\glcat(\lambda,\lambda')$ are called 2-morphisms in $\glcat$. Compare to Section~\ref{sec-techhigher}.} of
$\glcat(\lambda,\lambda')$, i.e. a 1-morphism in $\glcat$ from $\lambda$ to $\lambda'$,
is a formal finite direct sum of 1-morphisms
  \[
 \mathcal{E}_{\ii} \onel\{t\} = \onelp \mathcal{E}_{\ii} \onel\{t\}
= \mathcal{E}_{\alpha_1 i_1}\dotsm\mathcal{E}_{\alpha_m i_m} \onel\{t\}
  \]
for any $t\in \bZ$ and signed sequence $\ii\in\sseq$ such that 
$\lambda'=\lambda+\ii_{\Lambda}$ and $\lambda$, $\lambda'\in\bZ^n$. 
  \item Morphisms of $\glcat(\lambda,\lambda')$, i.e. for 1-morphisms $\mathcal{E}_{\ii} \onel\{t\}$
and  $\mathcal{E}_{\jj} \onel\{t'\}$ in $\glcat$, the hom
sets $\glcat(\mathcal{E}_{\ii} \onel\{t\},\mathcal{E}_{\jj} \onel\{t'\})$ of
$\glcat(\lambda,\lambda')$ are graded $\bC$-vector spaces given by linear
combinations of degree $t-t'$ diagrams, modulo certain relations, built from
compo\-sites of the following.
\begin{enumerate}[(i)]
  \item  Degree zero identity 2-morphisms $1_x$ for each 1-morphism $x$ in
$\glcat$; the identity 2-morphisms $1_{\mathcal{E}_{+i} \onel}\{t\}$ and
$1_{\mathcal{E}_{-i} \onel}\{t\}$, for $i \in I$, are represented graphically by
\[
\begin{array}{ccc}
  1_{\mathcal{E}_{+i} \onel\{t\}} &\quad  & 1_{\mathcal{E}_{-i} \onel\{t\}} \\ \\
   \lambda + i_{\Lambda}\xy
 (0,0)*{\includegraphics[width=09px]{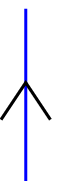}};
 (0,8)*{\scriptstyle i};
 (0,-8)*{\scriptstyle i};
 \endxy\lambda
 & &
 \;\;   
   \lambda - i_{\Lambda}\xy
 (0,0)*{\includegraphics[width=09px]{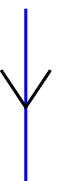}};
 (0,8)*{\scriptstyle i};
 (0,-8)*{\scriptstyle i};
 \endxy\lambda,
\\ \\
   \;\;\text{ {\rm deg} 0}\;\;
 & &\;\;\text{ {\rm deg} 0}\;\;
\end{array}
\]
for any $\lambda + i_{\Lambda} \in\bZ^n$ and any 
$\lambda - i_{\Lambda} \in \bZ^n$, respectively.

More generally, for a signed sequence $\ii=(\alpha_1i_1, \alpha_2i_2, \ldots
\alpha_mi_m)$, the identity $1_{\mathcal{E}_{\ii} \onel\{t\}}$ 2-morphism is
represented as
\begin{equation*}
\begin{array}{ccc}
  \lambda + \ii_{\Lambda}\xy
 (0,0)*{\includegraphics[width=0.7px]{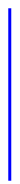}};
 (0,8)*{\scriptstyle i_1};
 (0,-8)*{\scriptstyle i_1};
 \endxy\,\xy
 (0,0)*{\includegraphics[width=0.7px]{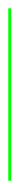}};
 (0,8)*{\scriptstyle i_2};
 (0,-8)*{\scriptstyle i_2};
 \endxy\cdots\xy
 (0,0)*{\includegraphics[width=0.7px]{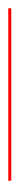}};
 (0,8)*{\scriptstyle i_m};
 (0,-8)*{\scriptstyle i_m};
 \endxy\lambda,
\end{array}
\end{equation*}
where the strand labelled $i_{k}$ is oriented up if $\alpha_{k}=+$
and oriented down if $\alpha_{k}=-$. We will often place labels with no
sign on the side of a strand and omit the labels at the top and bottom.  The
signs can be recovered from the orientations on the strands. 

\item Recall that $-\cdot -$ is the bilinear form from above. For each $\lambda \in \bZ^n$ the 2-morphisms 
\[
\begin{tabular}{|l|c|c|c|c|}
\hline
 {\bf Notation:} \xy (0,-5)*{};(0,7)*{}; \endxy& 
\xy
 (0,0)*{\includegraphics[width=7px]{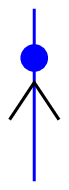}};
 (2,-5)*{\scriptstyle i,\lambda};
 \endxy &
\xy
 (0,0)*{\includegraphics[width=7px]{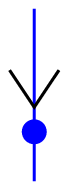}};
 (2,-5)*{\scriptstyle i,\lambda};
 \endxy  &  \xy
 (0,0)*{\includegraphics[width=20px]{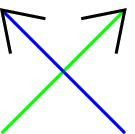}};
 (8,-4.5)*{\scriptstyle i,j,\lambda};
 \endxy  
 & 
\xy
 (0,0)*{\includegraphics[width=20px]{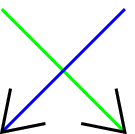}};
 (8,-4.5)*{\scriptstyle i,j,\lambda};
 \endxy\\
 \hline
 {\bf 2-morphism:} \xy (0,-5)*{};(0,9)*{}; \endxy&   \, \xy
 (0,0)*{\includegraphics[width=09px]{res/figs/section23/upsimpledot}};
 (1.5,-5)*{\scriptstyle i};
 (3,0)*{\scriptstyle\lambda};
 (-5,0)*{\scriptstyle\lambda+i_{\Lambda}};
 \endxy\,
 &
    \, \xy
 (0,0)*{\includegraphics[width=09px]{res/figs/section23/downsimpledot}};
 (1.5,-5)*{\scriptstyle i};
 (-3,0)*{\scriptstyle\lambda};
 (5,0)*{\scriptstyle \lambda+i_{\Lambda}};
 \endxy\,
 &
   \xy
 (0,1)*{\includegraphics[width=25px]{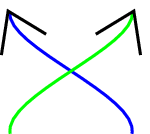}};
 (-5,-3)*{\scriptstyle i};
 (5,-3)*{\scriptstyle j};
 (5.5,0)*{\scriptstyle\lambda};
 \endxy\,
 &
   \xy
 (0,1)*{\includegraphics[width=25px]{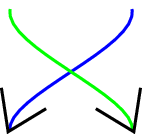}};
 (-5,-4)*{\scriptstyle i};
 (5,-4)*{\scriptstyle j};
 (5.5,0)*{\scriptstyle\lambda};
 \endxy\,
\\ & & & &\\
\hline
 {\bf Degree:} & \;\;\phantom{.a-}\text{$i\cdot i$}\;\;\phantom{.a-}
 &\;\;\phantom{.a-}\text{$i\cdot i$}\;\;\phantom{.a-}& \;\phantom{.a}\text{$-i\cdot j$}\;\;\phantom{..-}
 & \;\,\phantom{.a}\text{$-i\cdot j$}\;\,\phantom{..-} \\
 \hline
\end{tabular}
\]

\[
\begin{tabular}{|l|c|c|c|c|}
\hline
 {\bf Notation:} \xy (0,-5)*{};(0,7)*{}; \endxy& \xy
 (0,0)*{\includegraphics[width=20px]{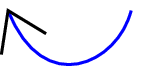}};
 (6,-2)*{\scriptstyle i,\lambda};
 \endxy &
 \xy
 (0,0)*{\includegraphics[width=20px]{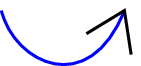}};
 (6,-2)*{\scriptstyle i,\lambda};
 \endxy  &  \xy
 (0,0)*{\includegraphics[width=20px]{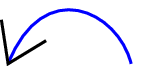}};
 (6,-2)*{\scriptstyle i,\lambda};
 \endxy  
 & \xy
 (0,0)*{\includegraphics[width=20px]{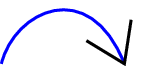}};
 (6,-2)*{\scriptstyle i,\lambda};
 \endxy  \\
 \hline
 {\bf 2-morphism:} \xy (0,-5)*{};(0,7)*{}; \endxy&   \xy
 (0,0)*{\includegraphics[width=25px]{res/figs/section23/leftcup}};
 (0,-3)*{\scriptstyle i};
 (5,0)*{\scriptstyle\lambda};
 \endxy\,
 &
    \xy
 (0,0)*{\includegraphics[width=25px]{res/figs/section23/rightcup}};
 (0,-3)*{\scriptstyle i};
 (5,0)*{\scriptstyle\lambda};
 \endxy\,
 &
   \xy
 (0,0)*{\includegraphics[width=25px]{res/figs/section23/leftcap}};
 (0,3)*{\scriptstyle i};
 (5,0)*{\scriptstyle\lambda};
 \endxy\,
 &
   \xy
 (0,0)*{\includegraphics[width=25px]{res/figs/section23/rightcap}};
 (0,3)*{\scriptstyle i};
 (5,0)*{\scriptstyle\lambda};
 \endxy\,
\\ & & & &\\
\hline
 {\bf Degree:} \xy (0,-1)*{};(0,5)*{}; \endxy& \phantom{.m}\text{$1-\overline{\lambda}_i$}\phantom{.m}
 &\phantom{.m}\text{$1+\overline{\lambda}_i$}\phantom{.m}&\phantom{.m}\text{$1+\overline{\lambda}_i$}\phantom{.m}
 & \phantom{.m}\text{$1-\overline{\lambda}_i$}\phantom{.m} \\
 \hline
\end{tabular}
\]
\end{enumerate}

\item Bi-adjointness and cyclicity are shown below.
\begin{enumerate}[(i)]
\item\label{it:sl2i}  $\mathbf{1}_{\lambda+i_{\Lambda}}\mathcal{E}_{+i}\onel$ and
$\onel\mathcal{E}_{-i}\mathbf{1}_{\lambda+i_{\Lambda}}$ are bi-adjoint, up to grading shifts:
\begin{equation} \label{eq_biadjoint1}
  \xy
 (0,0)*{\includegraphics[width=50px]{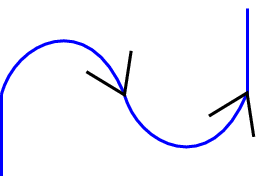}};
 (-4,5.5)*{\scriptstyle \lambda+i_{\Lambda}};
 (4,-5.5)*{\scriptstyle \lambda};
 \endxy
    \; =
    \;
\, \xy
 (0,0)*{\includegraphics[width=09px]{res/figs/section23/upsimple}};
 (1.5,-5)*{\scriptstyle i};
 (3,0)*{\scriptstyle\lambda};
 (-5,0)*{\scriptstyle\lambda+i_{\Lambda}};
 \endxy\,
\qquad \quad  \xy
 (0,0)*{\includegraphics[width=50px]{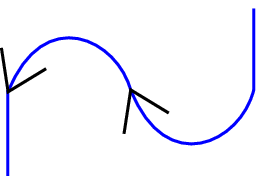}};
 (-4,5.5)*{\scriptstyle \lambda};
 (5,-5.5)*{\scriptstyle \lambda+i_{\Lambda}};
 \endxy
    \; =
    \;
\, \xy
 (0,0)*{\includegraphics[width=09px]{res/figs/section23/downsimple}};
 (1.5,-5)*{\scriptstyle i};
 (-3,0)*{\scriptstyle\lambda};
 (5,0)*{\scriptstyle\lambda+i_{\Lambda}};
 \endxy\,
\end{equation}

\begin{equation} \label{eq_biadjoint2}
  \xy
 (0,0)*{\includegraphics[width=50px]{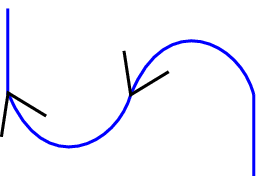}};
 (-4,-5.5)*{\scriptstyle \lambda+i_{\Lambda}};
 (4,5.5)*{\scriptstyle \lambda};
 \endxy
    \; =
    \;
\, \xy
 (0,0)*{\includegraphics[width=09px]{res/figs/section23/upsimple}};
 (1.5,-5)*{\scriptstyle i};
 (3,0)*{\scriptstyle\lambda};
 (-5,0)*{\scriptstyle\lambda+i_{\Lambda}};
 \endxy\,
\qquad \quad  \xy
 (0,0)*{\includegraphics[width=50px]{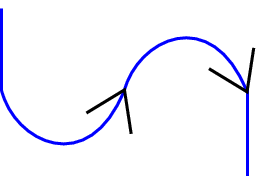}};
 (-4,-5.5)*{\scriptstyle \lambda};
 (5,5.5)*{\scriptstyle \lambda+i_{\Lambda}};
 \endxy
    \; =
    \;
\, \xy
 (0,0)*{\includegraphics[width=09px]{res/figs/section23/downsimple}};
 (1.5,-5)*{\scriptstyle i};
 (-3,0)*{\scriptstyle\lambda};
 (5,0)*{\scriptstyle\lambda+i_{\Lambda}};
 \endxy\,
\end{equation}
\item
\begin{equation} \label{eq_cyclic_dot}
  \xy
 (0,0)*{\includegraphics[width=50px]{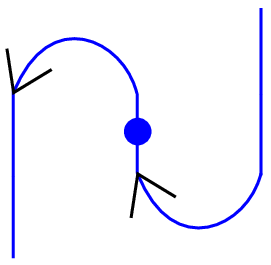}};
 (-4,8)*{\scriptstyle \lambda+i_{\Lambda}};
 (4,-8)*{\scriptstyle \lambda};
 \endxy
 \; =
    \;
\, \xy
 (0,0)*{\includegraphics[width=09px]{res/figs/section23/downsimpledot}};
 (1.5,-5)*{\scriptstyle i};
 (-3,0)*{\scriptstyle\lambda};
 (5,0)*{\scriptstyle\lambda+i_{\Lambda}};
 \endxy\,
    \; =
    \;
  \xy
 (0,0)*{\includegraphics[width=50px]{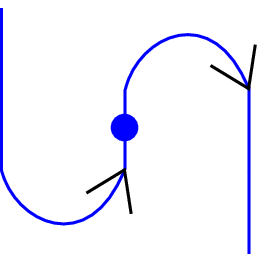}};
 (-4,-8)*{\scriptstyle \lambda};
 (4,8)*{\scriptstyle \lambda+i_{\Lambda}};
 \endxy    
\end{equation}
\item All 2-morphisms are cyclic with respect to the above bi-adjoint
   structure. This is ensured by the relations~\eqref{eq_cyclic_dot}, and, for arbitrary $i,j$, the
   relations
\begin{equation} \label{eq_cyclic_cross-gen}
  \xy
 (0,0)*{\includegraphics[width=125px]{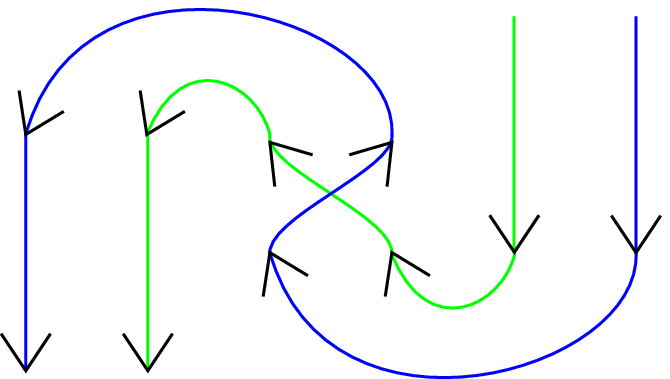}};
 (-12,-14)*{\scriptstyle j};
 (-20.3,-14)*{\scriptstyle i};
 (12,14)*{\scriptstyle j};
 (20.3,14)*{\scriptstyle i};
 (6,-0.5)*{\scriptstyle\lambda};
 \endxy\,
 \; =
    \;
\,\xy
 (0,1)*{\includegraphics[width=25px]{res/figs/section23/downcrosscurved}};
 (-5,-4)*{\scriptstyle i};
 (5,-4)*{\scriptstyle j};
 (-5.5,0)*{\scriptstyle\lambda};
 \endxy
    \; =
    \;
  \xy
 (0,0)*{\includegraphics[width=125px]{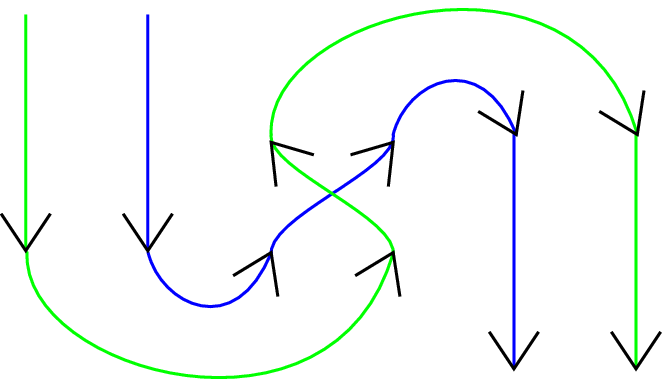}};
 (-12,14)*{\scriptstyle i};
 (-20.3,14)*{\scriptstyle j};
 (12,-14)*{\scriptstyle i};
 (20.3,-14)*{\scriptstyle j};
 (6,-1)*{\scriptstyle\lambda};
 \endxy\,.    
\end{equation}
Note that we can take either the first or the last diagram above as the 
definition of the up-side-down crossing. The cyclic condition on 
2-morphisms, expressed by~\eqref{eq_cyclic_dot} and~\eqref{eq_cyclic_cross-gen}, ensures that diagrams related by isotopy represent
the same 2-morphism in $\glcat$.

It will be convenient to introduce degree zero 2-morphisms as shown below.
\begin{equation} \label{eq_crossl-gen}
  \xy
 (0,0)*{\includegraphics[width=25px]{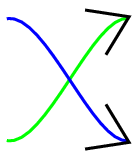}};
 (-5.5,4)*{\scriptstyle i};
 (-5.5,-4)*{\scriptstyle j};
 (5.5,0)*{\scriptstyle\lambda};
 \endxy\,
 \; =
    \;
 \xy
 (0,1)*{\includegraphics[width=75px]{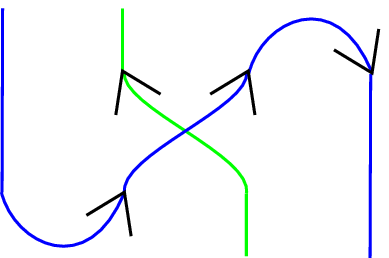}};
 (-4.3,11.5)*{\scriptstyle j};
 (-13.3,11.5)*{\scriptstyle i};
 (3.8,-9.5)*{\scriptstyle j};
 (12.5,-9.5)*{\scriptstyle i};
 (14,0)*{\scriptstyle\lambda};
 \endxy\,
    \; =
    \;
  \xy
 (0,0)*{\includegraphics[width=75px]{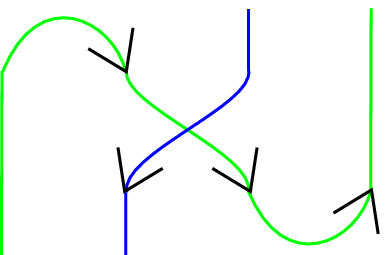}};
 (-4,-10.5)*{\scriptstyle i};
 (-13,-10.5)*{\scriptstyle j};
 (4,10.5)*{\scriptstyle i};
 (13,10.5)*{\scriptstyle j};
 (14,0)*{\scriptstyle\lambda};
 \endxy\,   
\end{equation}

\begin{equation} \label{eq_crossr-gen}
  \,\xy
 (0,0)*{\includegraphics[width=25px]{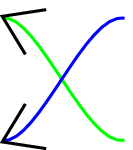}};
 (-5.5,0)*{\scriptstyle\lambda};
 (5.5,4)*{\scriptstyle i};
 (5.5,-4)*{\scriptstyle j};
 \endxy
 \; =
    \;
 \,\xy
 (0,1)*{\includegraphics[width=75px]{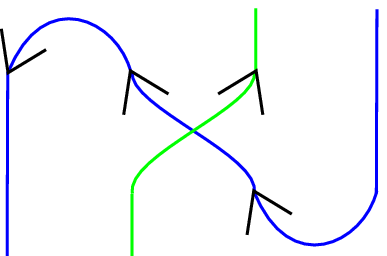}};
 (-4,-9.5)*{\scriptstyle j};
 (-13,-9.5)*{\scriptstyle i};
 (4.3,11.5)*{\scriptstyle j};
 (13.3,11.5)*{\scriptstyle i};
 (-14,0)*{\scriptstyle\lambda};
 \endxy\,
    \; =
    \;
  \,\xy
 (0,0)*{\includegraphics[width=75px]{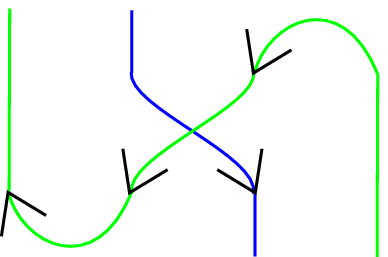}};
 (-4,10.5)*{\scriptstyle i};
 (-13,10.5)*{\scriptstyle j};
 (4,-10.5)*{\scriptstyle i};
 (13,-10.5)*{\scriptstyle j};
 (-14,0)*{\scriptstyle\lambda};
 \endxy\, ,    
\end{equation}
where the second equality in~\eqref{eq_crossl-gen} and~\eqref{eq_crossr-gen}
follow from~\eqref{eq_cyclic_cross-gen}.  

\item All dotted bubbles of negative degree are zero. That is,
\begin{equation} \label{eq_positivity_bubbles}
 \xy
 (0,0)*{\includegraphics[width=30px]{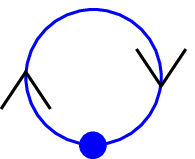}};
 (0,-5.5)*{\scriptstyle m};
 (-6,0)*{\scriptstyle i};
 (6,0)*{\scriptstyle\lambda};
 \endxy
  = 0,
 \qquad
  \text{if $m<\llambda_i-1$}, \qquad
 \xy
 (0,0)*{\includegraphics[width=30px]{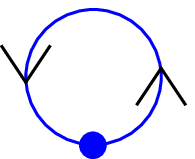}};
 (0,-5.5)*{\scriptstyle m};
 (-6,0)*{\scriptstyle i};
 (6,0)*{\scriptstyle\lambda};
 \endxy = 0,\quad
  \text{if $m< -\llambda_i-1$}
\end{equation}
for all $m \in \bZ_+$, where a dot carrying a label $m$ denotes the
$m$-fold iterated vertical composite of $\xy
 (0,0)*{\includegraphics[width=5px]{res/figs/section23/upsimpledot}};
 (3,-3)*{\scriptstyle i,\lambda};
 \endxy$ or
$\xy
 (0,0)*{\includegraphics[width=5px]{res/figs/section23/downsimpledot}};
 (3,-3)*{\scriptstyle i,\lambda};
 \endxy$ depending on the orientation.  A dotted bubble of degree
zero equals $\pm 1$ as illustrated below.
\begin{equation}\label{eq:bubb_deg0}
 \xy
 (0,0)*{\includegraphics[width=30px]{res/figs/section23/circleclock}};
 (0,-5.5)*{\scriptstyle m};
 (-6,0)*{\scriptstyle i};
 (6,0)*{\scriptstyle\lambda};
 \endxy
  = (-1)^{\lambda_{i+1}}, \quad \text{for $\llambda_i \geq 1$,}
  \qquad \quad
  \xy
 (0,0)*{\includegraphics[width=30px]{res/figs/section23/circlecounter}};
 (0,-5.5)*{\scriptstyle m};
 (-6,0)*{\scriptstyle i};
 (6,0)*{\scriptstyle\lambda};
 \endxy
  = (-1)^{\lambda_{i+1}-1}, \quad \text{for $\llambda_i \leq -1$.}
\end{equation}
\item For the following relations we employ the convention that all summations
are increasing, so that a summation of the form $\sum_{f=0}^{m}$ is zero 
if $m < 0$.
\begin{eqnarray}
\label{eq:redtobubbles}
  \;\;\;\;\qquad\text{$\,\xy
 (0,0)*{\includegraphics[width=50px]{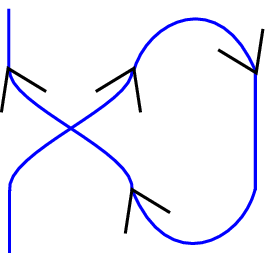}};
 (-8,0)*{\scriptstyle i};
 (9,0)*{\scriptstyle \lambda};
 \endxy$} \; = \; -\sum_{f=0}^{-\llambda_i}
   \xy
 (0,0)*{\includegraphics[width=9px]{res/figs/section23/upsimpledot}};
 (1,-4.5)*{\scriptstyle i};
 (5.5,5)*{\scriptstyle -\overline{\lambda}_i-f};
 \endxy\,\xy
 (0,0)*{\includegraphics[width=30px]{res/figs/section23/circleclock}};
 (0,-5.5)*{\scriptstyle \overline{\lambda}_i-1+f};
 (-6,0)*{\scriptstyle i};
 (6,0)*{\scriptstyle\lambda};
 \endxy
\quad\text{  and  } \quad
  \text{$\xy
 (0,0)*{\includegraphics[width=50px]{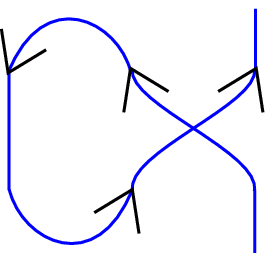}};
 (-9.5,0)*{\scriptstyle \lambda};
 (8,0)*{\scriptstyle i};
 \endxy$} \; = \;
 \sum_{g=0}^{\llambda_i}
   \xy
 (0,0)*{\includegraphics[width=30px]{res/figs/section23/circlecounter}};
 (0,-5.5)*{\scriptstyle -\overline{\lambda}_i-1+g};
 (-6,0)*{\scriptstyle i};
 (6,0)*{\scriptstyle\lambda};
 \endxy\,\xy
 (0,0)*{\includegraphics[width=9px]{res/figs/section23/upsimpledot}};
 (1,-4.5)*{\scriptstyle i};
 (5.5,5)*{\scriptstyle \overline{\lambda}_i-g};
 \endxy
\end{eqnarray}

\begin{equation}
\label{eq:EF}
 \xy
 (0,0)*{\includegraphics[width=9px]{res/figs/section23/upsimple}};
 (-1,-3.5)*{\scriptstyle i};
 (-2.5,0)*{\scriptstyle\lambda};
 \endxy\,\xy
 (0,0)*{\includegraphics[width=9px]{res/figs/section23/downsimple}};
 (1,-3.5)*{\scriptstyle i};
 (2.5,0)*{\scriptstyle\lambda};
 \endxy
  =
 \xy
 (0,0)*{\includegraphics[width=17px]{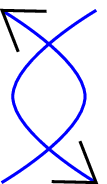}};
 (-3.9,-5.5)*{\scriptstyle i};
 (3.8,-5.5)*{\scriptstyle i};
 (3.2,0)*{\scriptstyle\lambda};
 \endxy
   - 
\sum_{f=0}^{\llambda_i-1} \sum_{g=0}^{f}
    \xy
 (0,0)*{\includegraphics[width=30px]{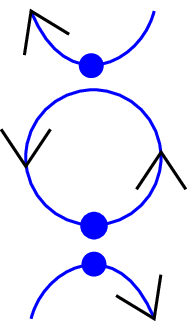}};
 (-6,0)*{\scriptstyle\lambda};
 (8,5)*{\scriptstyle \overline{\lambda}_i-1-f};
 (7,-4)*{\scriptstyle -\overline{\lambda}_i-1+g};
 (-5,-5)*{\scriptstyle f-g};
 (-4,4)*{\scriptstyle i};
 \endxy
 \quad\text{and}\quad
 \xy
 (0,0)*{\includegraphics[width=9px]{res/figs/section23/downsimple}};
 (-1,-3.5)*{\scriptstyle i};
 (-2.5,0)*{\scriptstyle\lambda};
 \endxy\,\xy
 (0,0)*{\includegraphics[width=9px]{res/figs/section23/upsimple}};
 (1,-3.5)*{\scriptstyle i};
 (2.5,0)*{\scriptstyle\lambda};
 \endxy
 = 
 \xy
 (0,0)*{\includegraphics[width=17px]{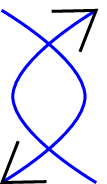}};
 (-3.9,-5.5)*{\scriptstyle i};
 (3.8,-5.5)*{\scriptstyle i};
 (3.2,0)*{\scriptstyle\lambda};
 \endxy
   - 
\sum_{f=0}^{-\llambda_i-1} \sum_{g=0}^{f}
    \xy
 (0,0)*{\includegraphics[width=30px]{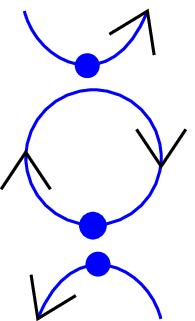}};
 (7,5)*{\scriptstyle -\overline{\lambda}_i-1-f};
 (-6,0)*{\scriptstyle\lambda};
 (8,-4)*{\scriptstyle \overline{\lambda}_i-1+g};
 (-5,-5)*{\scriptstyle f-g};
 (-4,4)*{\scriptstyle i};
 \endxy
\end{equation}
for all $\lambda\in \bZ^n$
(see~\eqref{eq_crossl-gen} and~\eqref{eq_crossr-gen} for the definition of sideways
crossings). 
Notice that for some values of 
$\lambda$ the dotted
bubbles appearing above have negative labels. A composite of $\xy
 (0,0)*{\includegraphics[width=5px]{res/figs/section23/upsimpledot}};
 (3,-3)*{\scriptstyle i,\lambda};
 \endxy$
or $\xy
 (0,0)*{\includegraphics[width=5px]{res/figs/section23/downsimpledot}};
 (3,-3)*{\scriptstyle i,\lambda};
 \endxy$ with itself a negative number of times does not make
sense. These dotted bubbles with negative labels, called \textit{fake bubbles}, are
formal symbols inductively defined by the equation
\begin{equation}
\left(\xy
 (0,0)*{\includegraphics[width=30px]{res/figs/section23/circlecounter}};
 (-4,4)*{\scriptstyle i};
 (4,4)*{\scriptstyle \lambda};
 (-1,-5)*{\scriptstyle -\overline{\lambda}_i-1};
 \endxy t^{0}+\cdots+\xy
 (0,0)*{\includegraphics[width=30px]{res/figs/section23/circlecounter}};
 (-4,4)*{\scriptstyle i};
 (4,4)*{\scriptstyle \lambda};
 (0,-5)*{\scriptstyle -\overline{\lambda}_i-1+r};
 \endxy t^{r}+\cdots\right)\left(\xy
 (0,0)*{\includegraphics[width=30px]{res/figs/section23/circleclock}};
 (-4,4)*{\scriptstyle i};
 (4,4)*{\scriptstyle \lambda};
 (-0.1,-5)*{\scriptstyle \overline{\lambda}_i-1};
 \endxy t^{0}+\cdots+\xy
 (0,0)*{\includegraphics[width=30px]{res/figs/section23/circleclock}};
 (-4,4)*{\scriptstyle i};
 (4,4)*{\scriptstyle \lambda};
 (0,-5)*{\scriptstyle \overline{\lambda}_i-1+r};
 \endxy t^{r}+\cdots\right)=-1
 \label{eq_infinite_Grass}
\end{equation}
and the additional condition
\[
\xy
 (0,0)*{\includegraphics[width=30px]{res/figs/section23/circleclock}};
 (-4,4)*{\scriptstyle i};
 (4,4)*{\scriptstyle \lambda};
 (0,-5)*{\scriptstyle -1};
 \endxy=(-1)^{\lambda_{i+1}}\quad\text{and}\quad\xy
 (0,0)*{\includegraphics[width=30px]{res/figs/section23/circlecounter}};
 (-4,4)*{\scriptstyle i};
 (4,4)*{\scriptstyle \lambda};
 (0,-5)*{\scriptstyle -1};
 \endxy=(-1)^{\lambda_{i+1}-1},\quad\text{if}\;\overline{\lambda}_i=0.
\]
Although the labels are negative for fake bubbles, one can check that the overall
degree of each fake bubble is still positive, so that these fake bubbles do not
violate the positivity of dotted bubble axiom. The above equation, called the
infinite Grassmannian relation, remains valid even in high degree when most of
the bubbles involved are not fake bubbles.  

\item NilHecke relations are the following.
\begin{equation}\label{eq_nil_rels}
 \xy
 (0,0)*{\includegraphics[width=17px]{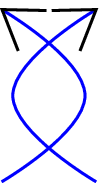}};
 (-3.9,-5.5)*{\scriptstyle i};
 (3.8,-5.5)*{\scriptstyle i};
 (3.2,0)*{\scriptstyle\lambda};
 \endxy=0,\quad\xy
 (0,0)*{\includegraphics[width=50px]{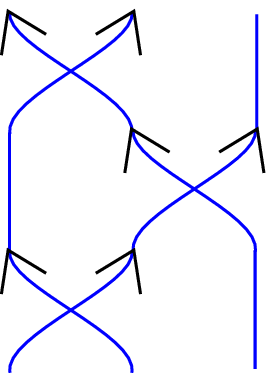}};
 (-9.5,-11.5)*{\scriptstyle i};
 (1,-11.5)*{\scriptstyle i};
 (7.3,-11.5)*{\scriptstyle i};
 (10,0)*{\scriptstyle \lambda};
 \endxy=\xy
 (0,0)*{\includegraphics[width=50px]{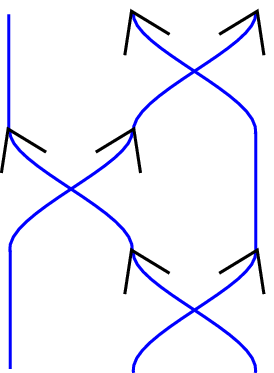}};
 (-7.3,-11.5)*{\scriptstyle i};
 (-1,-11.5)*{\scriptstyle i};
 (9.5,-11.5)*{\scriptstyle i};
 (10,0)*{\scriptstyle \lambda};
 \endxy
\end{equation}

\begin{equation}\label{eq_nil_dotslide}
 \xy
 (0,0)*{\includegraphics[width=9px]{res/figs/section23/upsimple}};
 (-1,-3.5)*{\scriptstyle i};
 \endxy\,\xy
 (0,0)*{\includegraphics[width=9px]{res/figs/section23/upsimple}};
 (1,-3.5)*{\scriptstyle i};
 (2.5,0)*{\scriptstyle\lambda};
 \endxy\;=\;\xy
 (0,1)*{\includegraphics[width=25px]{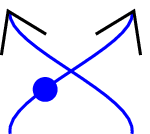}};
 (-5,-3)*{\scriptstyle i};
 (5,-3)*{\scriptstyle i};
 (5.5,0)*{\scriptstyle\lambda};
 \endxy\;-\xy
 (0,1)*{\includegraphics[width=25px]{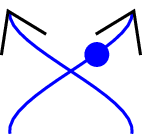}};
 (-5,-3)*{\scriptstyle i};
 (5,-3)*{\scriptstyle i};
 (5.5,0)*{\scriptstyle\lambda};
 \endxy\;=\xy
 (0,1)*{\includegraphics[width=25px]{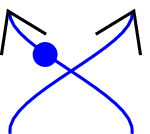}};
 (-5,-3)*{\scriptstyle i};
 (5,-3)*{\scriptstyle i};
 (5.5,0)*{\scriptstyle\lambda};
 \endxy\;-\xy
 (0,1)*{\includegraphics[width=25px]{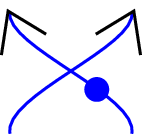}};
 (-5,-3)*{\scriptstyle i};
 (5,-3)*{\scriptstyle i};
 (5.5,0)*{\scriptstyle\lambda};
 \endxy.
\end{equation}
\end{enumerate}

\item For $i \neq j$:
\begin{equation} \label{eq_downup_ij-gen}
 \xy
 (0,0)*{\includegraphics[width=17px]{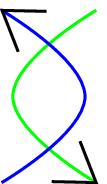}};
 (-3.9,-5.5)*{\scriptstyle i};
 (3.8,-5.5)*{\scriptstyle j};
 (3.2,0)*{\scriptstyle\lambda};
 \endxy=\xy
 (0,0)*{\includegraphics[width=9px]{res/figs/section23/upsimple}};
 (-1,-3.5)*{\scriptstyle i};
 \endxy\,\xy
 (0,0)*{\includegraphics[width=9px]{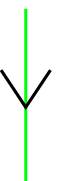}};
 (1,-3.5)*{\scriptstyle j};
 (2.5,0)*{\scriptstyle\lambda};
 \endxy\quad\text{and}\quad\xy
 (0,0)*{\includegraphics[width=17px]{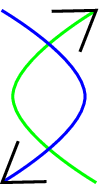}};
 (-3.9,-5.5)*{\scriptstyle i};
 (3.8,-5.5)*{\scriptstyle j};
 (3.2,0)*{\scriptstyle\lambda};
 \endxy=\xy
 (0,0)*{\includegraphics[width=9px]{res/figs/section23/downsimple}};
 (-1,-3.5)*{\scriptstyle i};
 \endxy\,\xy
 (0,0)*{\includegraphics[width=9px]{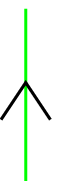}};
 (1,-3.5)*{\scriptstyle j};
 (2.5,0)*{\scriptstyle\lambda};
 \endxy
\end{equation}

\item \begin{enumerate}[(i)]
\item For $i \neq j$:
\begin{equation}\label{eq_r2_ij-gen}
 \xy
 (0,0)*{\includegraphics[width=17px]{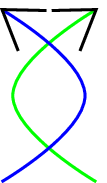}};
 (-3.9,-5.5)*{\scriptstyle i};
 (3.8,-5.5)*{\scriptstyle j};
 (3.2,0)*{\scriptstyle\lambda};
 \endxy =\begin{cases}\phantom{(i-j)(d}
 \xy
 (0,0)*{\includegraphics[width=9px]{res/figs/section23/upsimple}};
 (-1,-3.5)*{\scriptstyle i};
 \endxy\,\xy
 (0,0)*{\includegraphics[width=9px]{res/figs/section23/upsimplegreen}};
 (1,-3.5)*{\scriptstyle j};
 (2.5,0)*{\scriptstyle\lambda};
 \endxy, & \text{if}\,i\cdot j=0,\\[3ex]
 (i-j)\left(\xy
 (0,0)*{\includegraphics[width=9px]{res/figs/section23/upsimpledot}};
 (-1,-3.5)*{\scriptstyle i};
 \endxy\,\xy
 (0,0)*{\includegraphics[width=9px]{res/figs/section23/upsimplegreen}};
 (1,-3.5)*{\scriptstyle j};
 (2.5,0)*{\scriptstyle\lambda};
 \endxy-\xy
 (0,0)*{\includegraphics[width=9px]{res/figs/section23/upsimple}};
 (-1,-3.5)*{\scriptstyle i};
 \endxy\,\xy
 (0,0)*{\includegraphics[width=9px]{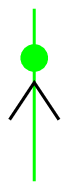}};
 (1,-3.5)*{\scriptstyle j};
 (2.5,0)*{\scriptstyle\lambda};
 \endxy\right), & \text{if}\,i\cdot j=-1.
 \end{cases}
\end{equation}
Notice that $(i-j)$ is just a sign, which takes into account the standard 
orientation of the Dynkin diagram.

\begin{equation}\label{eq_dot_slide_ij-gen}
\xy
 (0,1)*{\includegraphics[width=25px]{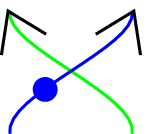}};
 (-5,-3)*{\scriptstyle i};
 (5,-3)*{\scriptstyle j};
 (5.5,0)*{\scriptstyle\lambda};
 \endxy\;=\xy
 (0,1)*{\includegraphics[width=25px]{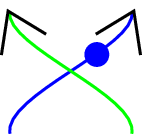}};
 (-5,-3)*{\scriptstyle i};
 (5,-3)*{\scriptstyle j};
 (5.5,0)*{\scriptstyle\lambda};
 \endxy\quad\text{and}\quad\xy
 (0,1)*{\includegraphics[width=25px]{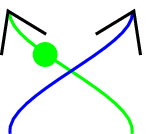}};
 (-5,-3)*{\scriptstyle i};
 (5,-3)*{\scriptstyle j};
 (5.5,0)*{\scriptstyle\lambda};
 \endxy\;=\xy
 (0,1)*{\includegraphics[width=25px]{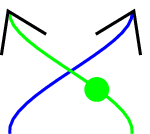}};
 (-5,-3)*{\scriptstyle i};
 (5,-3)*{\scriptstyle j};
 (5.5,0)*{\scriptstyle\lambda};
 \endxy.
\end{equation}

\item Unless $i = k$ and $i \cdot j=-1$:
\begin{equation}\label{eq_r3_easy-gen}
 \xy
 (0,0)*{\includegraphics[width=50px]{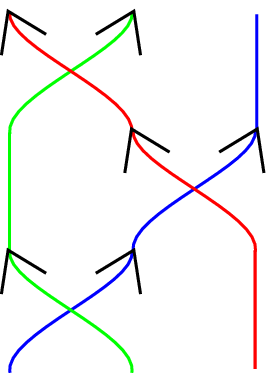}};
 (-9.1,-11.5)*{\scriptstyle i};
 (1,-11.5)*{\scriptstyle j};
 (7.1,-11.5)*{\scriptstyle k};
 (10,0)*{\scriptstyle \lambda};
 \endxy=\xy
 (0,0)*{\includegraphics[width=50px]{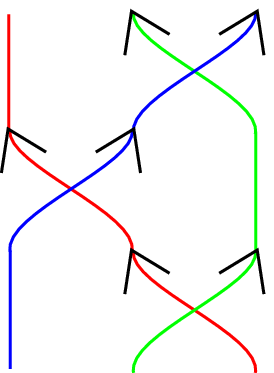}};
 (-9.3,-11.5)*{\scriptstyle i};
 (-1,-11.5)*{\scriptstyle j};
 (9.1,-11.5)*{\scriptstyle k};
 (10,0)*{\scriptstyle \lambda};
 \endxy
\end{equation}

\item For $i \cdot j =-1$:
\begin{equation}\label{eq_r3_hard-gen}
\xy
 (0,0)*{\includegraphics[width=50px]{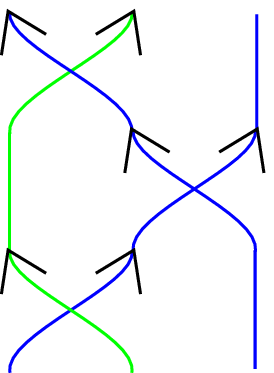}};
 (-9.1,-11.5)*{\scriptstyle i};
 (1,-11.5)*{\scriptstyle j};
 (7.1,-11.5)*{\scriptstyle i};
 (10,0)*{\scriptstyle \lambda};
 \endxy -\xy
 (0,0)*{\includegraphics[width=50px]{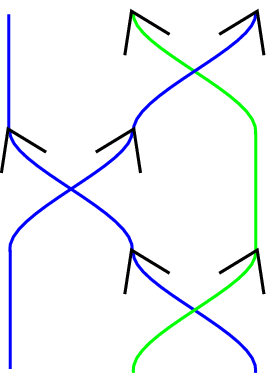}};
 (-9.3,-11.5)*{\scriptstyle i};
 (-1,-11.5)*{\scriptstyle j};
 (9.1,-11.5)*{\scriptstyle i};
 (10,0)*{\scriptstyle \lambda};
 \endxy =
 (i-j)\xy
 (0,0)*{\includegraphics[width=60px]{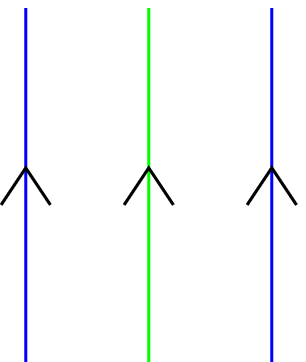}};
 (-9.4,-11.5)*{\scriptstyle i};
 (1,-11.5)*{\scriptstyle j};
 (7.7,-11.5)*{\scriptstyle i};
 (12,0)*{\scriptstyle \lambda};
 \endxy.
\end{equation}
\end{enumerate}
\item The additive, linear composition functor $\glcat(\lambda,\lambda')
\times \glcat(\lambda',\lambda'') \to \glcat(\lambda,\lambda'')$ is given on
1-morphisms of $\glcat$ by
\begin{equation}
  \mathcal{E}_{\jj}\mathbf{1}_{\lambda'}\{t'\} \times \mathcal{E}_{\ii}\onel\{t\} \mapsto
  \mathcal{E}_{\jj\ii}\onel\{t+t'\}
\end{equation}
for $\ii_{\Lambda}=\lambda-\lambda'$, and on 2-morphisms of $\glcat$ by juxtaposition of
diagrams, e.g.
\[
\left(\xy
 (0,0)*{\includegraphics[width=95px]{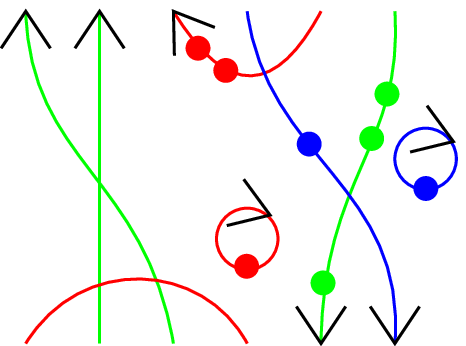}};
 (-16,0)*{\scriptstyle \lambda};
 (19,0)*{\scriptstyle\lambda'};
 \endxy\right)\times\left(\xy
 (0,0)*{\includegraphics[width=35px]{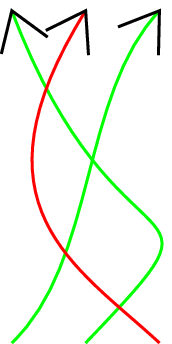}};
 (-8,0)*{\scriptstyle \lambda'};
 (8,0)*{\scriptstyle\lambda''};
 \endxy\right)\mapsto\xy
 (0,0)*{\includegraphics[width=125px]{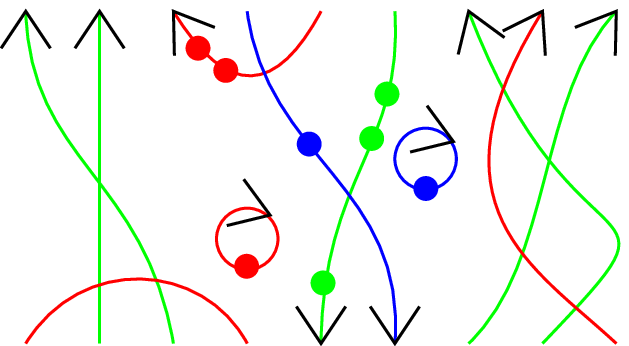}};
 (-20,0)*{\scriptstyle \lambda};
 (22,0)*{\scriptstyle\lambda''};
 \endxy.
\]
\end{itemize}
\end{defn}
This concludes the definition of $\glcat$. 
\vskip0.5cm
Note that for two $1$-morphisms $x$ and $y$ in 
$\glcat$ the 2-hom-space $\HomGL(x,y)$ 
only contains 2-morphisms of degree zero and is therefore finite dimensional. 
Following Khovanov and Lauda we introduce the graded 2-hom-space 
\[
\HOMGL(x,y)=\bigoplus_{t\in\bZ}\HomGL(x\{t\},y),
\]
which is infinite dimensional. We also define the $2$-category 
$\glcat^*$ which has the same objects and $1$-morphisms as $\glcat$, 
but for two $1$-morphisms $x$ and $y$ the vector space of 2-morphisms is 
defined by 
\begin{equation}
\label{eq:ast}
\glcat^*(x,y)=\HOMGL(x,y).
\end{equation}

Note that $\Ucat$ is defined just as $\glcat$, but labelling 
all the regions of the diagrams with $\mathfrak{sl}_n$-weights, i.e. 
elements of 
$\mathbb{Z}^{n-1}$. One also has to renormalise the signs of the 
left cups and caps, so that the bubble relations all become 
dependent on the $\mathfrak{sl}_n$-weights. For more details, see~\cite{msv1}.
\subsubsection{The $q$-Schur 2-algebra}
The categorification (in the sense of Section~\ref{sec-techgrgrcat}) of $S_q(n,n)$ is now obtained from 
$\glcat$ by taking a quotient. 
\begin{defn}
The 2-category $\Scat(n,n)$ is the quotient of $\glcat$ by the ideal 
generated by all 2-morphisms containing a region with a label not in 
$\Lambda(n,n)$. 
\end{defn}
We remark that we only put real bubbles, whose interior has a label outside 
$\Lambda(n,n)$, equal to zero. To see what happens to a fake bubble, one 
first has to write it in terms of real bubbles with the opposite orientation 
using the infinite Grassmannian relation~\eqref{eq_infinite_Grass}.

A main result of~\cite{msv1}, given in Theorem 7.11 in that paper, is the following.  
\begin{thm}
Let $K^{\oplus}_0(\dot{\mathcal S}(n,n))$ denote the split Grothendieck 
group of the Karoubi envelope of ${\mathcal S}(n,n)$. The 
$\mathbb{Z}[q,q^{-1}]$-linear map 
\[
\gamma_S\colon S_q^{\mathbb{Z}}(n,n)\to K^{\oplus}_0(\dot{\mathcal S}(n,n)),
\]
determined by 
\[
\gamma_S(E_{\ii}1_{\lambda})=[\mathcal{E}_{\ii}1_{\lambda}]
\]
is an isomorphism of algebras. 
\end{thm}

Recall also (see Definition 4.1 in~\cite{msv1}) that there is an 
essentially surjective and full additive 2-functor
\[
\Psi_{n,n}\colon \Ucat\to \Scat(n,n),
\]
whose precise definition is not relevant here. Up to signs related to 
cups and caps, it is obtained by mapping any string diagram to itself 
and applying $\phi_{n,n}$ to the labels of the regions. By convention, any diagram 
with a region labelled $*$ is taken to be zero. 
It is important to note that 
\[
K^{\oplus}_0(\Psi_{n,n})\colon K^{\oplus}_0(\UcatD)\otimes_{\mathbb{Z}[q,q^{-1}]}\mathbb{C}(q)
\to K^{\oplus}_0(\ScatD(n,n))\otimes_{\mathbb{Z}[q,q^{-1}]}\mathbb{C}(q)
\]
corresponds to the aforementioned surjective homomorphism
\[
\psi_{n,n}\colon \U\to S_q(n,n).
\] 
\subsubsection{The cyclotomic KLR-algebras}
In this subsection, we recall the definition of the cyclotomic KLR-algebras, 
due to Khovanov and Lauda~\cite{kl1},~\cite{kl3} and, independently, 
to Rouquier~\cite{rou}. We also recall two important results about them. 

Fix $\nu\in\mathbb{Z}_{\leq 0}[I]$. Let $\mathrm{Seq}(\nu)$ be the set of all 
sequences $\ii=(-i_1,-i_2,\cdots,-i_m)$, such that $i_k\in I$ for 
each $k$ and $\nu_j=\#\{k\mid i_k=j\}$. 
\begin{defn} For any $\ii,\jj\in \mathrm{Seq}(\nu)$ and any 
$\mathfrak{gl}_n$-weight $\lambda\in\mathbb{Z}^{n}$, let 
\[
{}_{\ii}R(\nu)_{\jj}\subset 
\mathrm{End}_{\glcat}({\mathcal E}_{\jj}1_{\lambda},{\mathcal E}_{\ii}1_{\lambda})
\]  
be the subalgebra containing only diagrams which are oriented downwards. So, 
only strands oriented downwards with dots and crossings are allowed. No 
strands oriented upwards, no cups and no caps. The relations 
in $\glcat$ involving only downward strands do not depend on $\lambda$. 
Therefore, the definition above makes sense. In~\cite{kl1}, 
the authors do not label the regions of the diagrams. 

Then $R(\nu)$ is defined as 
\[
R(\nu)=\bigoplus_{\ii,\jj\in \mathrm{Seq}(\nu)} {}_{\ii}R(\nu)_{\jj}.
\]

The ring $R$ is defined as 
\[
R=\bigoplus_{\nu\in\mathbb{Z}_{\leq 0}[I]}R(\nu).
\]
\end{defn}
As remarked above, the definition of $R(\nu)$ does not depend on $\lambda$. 
However, when we use a particular $\lambda$, we will write $R(\nu)1_{\lambda}$. 

Note that $R(\nu)$ is unital, whereas $R$ has infinitely many idempotents. 

Let $R(\nu)\text{-}\mathrm{p\textbf{Mod}}_{\mathrm{gr}}$ be the category of graded, 
finitely-generated, projective $R(\nu)$-modules and define 
\[
R\text{-}\mathrm{p\textbf{Mod}}_{\mathrm{gr}}=
\bigoplus_{\nu\in\mathbb{Z}_{\leq 0}[I]}R(\nu)\text{-}\mathrm{p\textbf{Mod}}_{\mathrm{gr}}.
\]

In Proposition 3.18 in~\cite{kl1}, Khovanov and Lauda showed that 
$R\text{-}\mathrm{p\textbf{Mod}}_{\mathrm{gr}}$ categorifies the negative half of 
$\U$ and $R(\nu)\text{-}\mathrm{p\textbf{Mod}}_{\mathrm{gr}}$ categorifies the 
$\nu$-root space. 
\vskip0.5cm
We can now recall the definition of the \textit{cyclotomic KLR-algebras}. The reader can find more details in~\cite{kl1} or~\cite{rou}, for example. 
\begin{defn}
Choose a dominant $\Ugl$-weight $\lambda\in\Lambda(n,n)^+$. 
Let $R(\nu;\lambda)$ be the quotient algebra 
of $R(\nu)1_{\lambda}$ by the ideal generated by all diagrams of the form
\begin{align*}
\xy
(0,0)*{\includegraphics[width=45px]{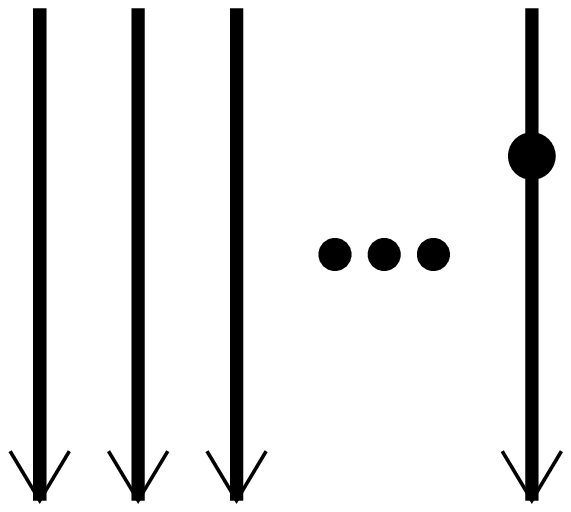}};
(-6.5,-8.5)*{\scriptstyle i_1};
(-3.5,-8.5)*{\scriptstyle i_2};
(-0.5,-8.5)*{\scriptstyle i_3};
(7.5,-8.5)*{\scriptstyle i_m};
(10,5)*{\scriptstyle \overline{\lambda}_m};
(8.8,0.3)*{\scriptstyle\lambda};
\endxy.
\end{align*}
Define 
\[
R_{\lambda}=\bigoplus_{\nu\in\mathbb{Z}_{\leq 0}[I]}R(\nu;\lambda).
\]
\end{defn}
Recall that $\overline{\lambda}_m=\lambda_m-\lambda_{m+1}$, the $m$-th entry 
of the $\mathfrak{sl}_n$-weight corresponding to $\lambda$. 
 
Note that we mod out by relations involving dots on the last strand, rather 
than the first strand as in~\cite{kl1}. This is to make the definition 
compatible with the other definitions in Section~\ref{sec-web}. 

It turns out that $R_{\lambda}$ is a finite dimensional, unital algebra. Let 
$R_{\lambda}\text{-}\mathrm{p\textbf{Mod}}_{\mathrm{gr}}$ be its category of 
finite dimensional, graded, projective modules and let $K^{\oplus}_0(R_{\lambda})=K^{\oplus}_0(R_{\lambda}\text{-}\mathrm{p\textbf{Mod}}_{\mathrm{gr}})$ 
be the split Grothendieck group of that category.

There is a graded categorical action (in the sense of Section~\ref{sec-techhigherrep}) of $\Ucat$ on 
$R_{\lambda}\text{-}\mathrm{p\textbf{Mod}}_{\mathrm{gr}}$ and  
Brundan and Kleshchev~\cite{bk2} (see also~\cite{kaka},
~\cite{lv},~\cite{vv} and~\cite{web1}) 
proved a conjecture by Khovanov and Lauda, i.e. the following holds.
\begin{thm}
\label{thm:bk}
We have 
\[
K^{\oplus}_0(R_{\lambda})\cong V_{\lambda}^{\mathbb{Z}}
\]
as $\UZ(\mathfrak{sl}_n)$-modules. Here 
$V_{\lambda}^{\mathbb{Z}}$ is the irreducible $\UZ(\mathfrak{sl}_n)$-module with 
highest weight $\overline{\lambda}$.  
\end{thm}

Rouquier's showed that, in a certain sense, 
$R_{\lambda}$ is the universal categorification of 
$V_{\lambda}$. For a proof of the following 
result, see Lemma 5.4, Proposition 5.6 and Corollary 5.7 in~\cite{rou}.

\begin{prop}
\label{prop:rouquier}
Let $\mathcal V$ be any additive idempotent complete category, which allows 
an integrable graded categorical action by $\Ucat$ (for the precise definition 
see~\cite{rou}). Suppose $V_h$ is a highest weight object in $\mathcal V$, i.e 
an object that is killed by ${\mathcal E}_{+i}$, for all $i\in I$, and 
$\mathrm{End}_{\mathcal V}(V_h)\cong \mathbb{C}$. Suppose also 
that any object in $\mathcal V$ is a direct summand of $XV_h$, for some 
object $X\in\Ucat$. Then there exists an equivalence of 
categorical $\Ucat$-representations
\[
\Phi\colon R_{\lambda}\text{-}\mathrm{p\textbf{Mod}}_{\mathrm{gr}}\to {\mathcal V}.
\]
\end{prop}
There are some subtle differences between 
Rouquier's approach to categorification and Khovanov and Lauda's, compare to Section~\ref{sec-techhigherrep}. However, 
Proposition~\ref{prop:rouquier} holds in both set-ups, as already remarked by 
Webster in Section 1.4 in~\cite{web1}.

The proof of Proposition~\ref{prop:rouquier} consists of Rouquier's remarks in Section 5.1.2 and of the 
contents of his proofs of Lemma 5.4 and Proposition 5.6 in~\cite{rou}, which 
only rely on the assumptions in the statement of our 
Proposition~\ref{prop:rouquier} and the fact that $\mathcal{E}_{+i}$ and 
$\mathcal{E}_{-i}$ are bi-adjoint in $\Ucat$, for any $i\in I$.

The precise 
definition of the units and the counits, i.e. the cups and the caps, 
is not relevant for 
the validity of the proof. Note that we have included the hypothesis 
$$
\mathrm{End}_{\mathcal V}(V_h)\cong \mathbb{C}
$$  
in Proposition~\ref{prop:rouquier}, which is not one of 
Rouquier's assumptions. There are categorifications of 
$V_{\lambda}$ without that property, see Conjecture 7.16 in~\cite{msv1} for 
example. However, in order to get a categorification which is 
really equivalent to $R_{\lambda}\text{-}\mathrm{p\textbf{Mod}}_{\mathrm{gr}}$, 
i.e. with hom-spaces of the same graded dimension, 
one needs to add that assumption because it holds in the latter category. 

We do not need the precise definition of $\Phi$ here. 
In order to contain the length of this thesis within reasonable boundaries, 
we will not explain it here.
\subsection{The $\mathfrak{sl}_{3}$ web algebra ${\mathcal W}_S^c$}\label{sec-webalg}
\setcounter{subsubsection}{1}
For the rest of this section, let $S$ be a fixed sign string of length $n$. 
We are going to define the \textit{web algebra} ${\mathcal W}_S^c$.
\begin{defn}
\textbf{(Web algebra)}\label{defn:webalg} For $u,v\in B_S$, we define 
\[
{}_{u}\mathcal{W}^c_{v}=\F^c(u^*v)\{n\},
\] 
where $\{n\}$ denotes a grading shift upwards in degree by $n$.

The \textit{web algebra} ${\mathcal W}_S^c$ is defined by 
\[
{\mathcal W}_S^c=\bigoplus_{u,v\in B_S}{}_{u}\mathcal{W}^c_{v}.
\]
 
The multiplication on ${\mathcal W}_S^c$ is defined by taking 
\[
{}_{u}{\mathcal W}^c_{v_1}\otimes {}_{v_2}{\mathcal W}^c_{w} \to 
{}_{u}{\mathcal W}^c_{w}
\]
to be zero, if $v_1\ne v_2$, and by the map to be defined in Definition~\ref{multfoam}, 
if $v_1=v_2=v$. 
\end{defn}
\begin{rem}
In Proposition~\ref{prop:multqgrade} we prove that the 
multiplication foam always has
degree $n$, so the degree shift in the definition above 
makes ${\mathcal W}_S^0$ into a graded algebra and, for any $c\ne 0$, it 
makes ${\mathcal W}_S^c$ into a filtered algebra. 
\end{rem}
\begin{defn} \textbf{(Multiplication of closed webs)}
\label{multfoam}
The \textit{multiplication} 
\[
{}_{u}{\mathcal W}^c_{v}\otimes {}_{v}\mathcal{W}^c_{w} \to 
{}_{u}{\mathcal W}^c_{w}
\] 
is induced by the \textit{multiplication foam}  
\[
m_{u,v,w}\colon u^*vv^*w \xrightarrow{Id_{u^*}m_{v}Id_{w}}u^*w,
\]
where $m_v\colon vv^*\to \Ver_{n}$, with $\Ver_{n}$ being the web of $n$ 
parallel oriented vertical line segments, is defined by the following 
inductive algorithm.
\begin{enumerate}
\item Express $v$ using the growth algorithm, label each level of the 
growth algorithm starting from zero. Then form $vv^*$.
\item At the $k$th level in the growth algorithm, \textit{resolve} 
the corresponding pair 
of arc, H or Y-rules in $v$ and $v^*$ by applying the foams.
\begin{align}
 \xy(0,0)*{\label{multrules}\includegraphics[width=280px, height=45px]{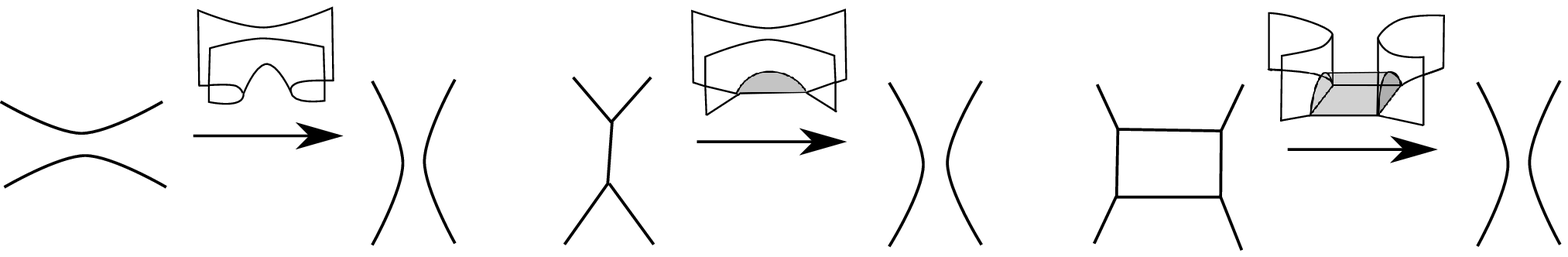}};\endxy
\end{align}
\end{enumerate}
Note that at the last level in the growth algorithm of $v$, only pairs of arcs 
are present.
\end{defn}
\begin{ex}
Let $w$ and $v$ be the following webs. 
\begin{align}
   \xy(0,0)*{\includegraphics[width=180px]{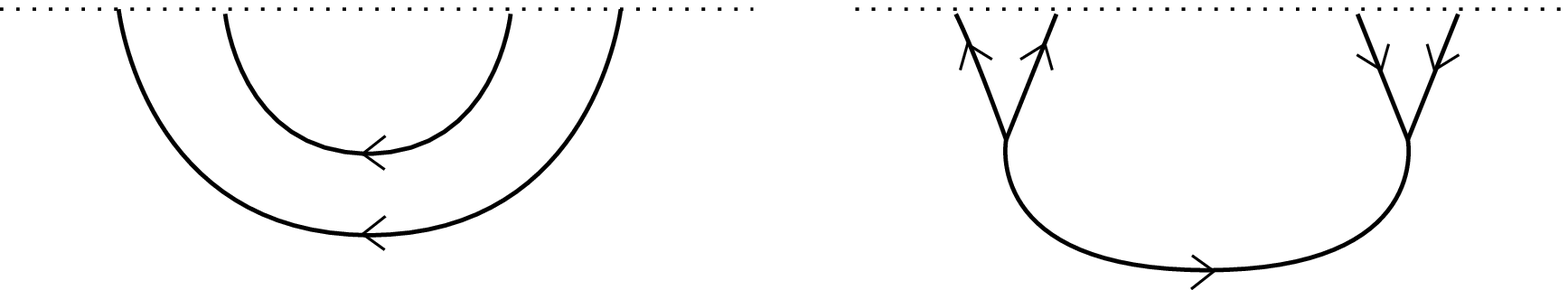}};(-5,0)*{w};
 (28,0)*{v};\endxy
\end{align}
The multiplication foam $m_{w,v,v}$ is given by the following steps.
\begin{align}
   \xy(0,0)*{\includegraphics[width=300px]{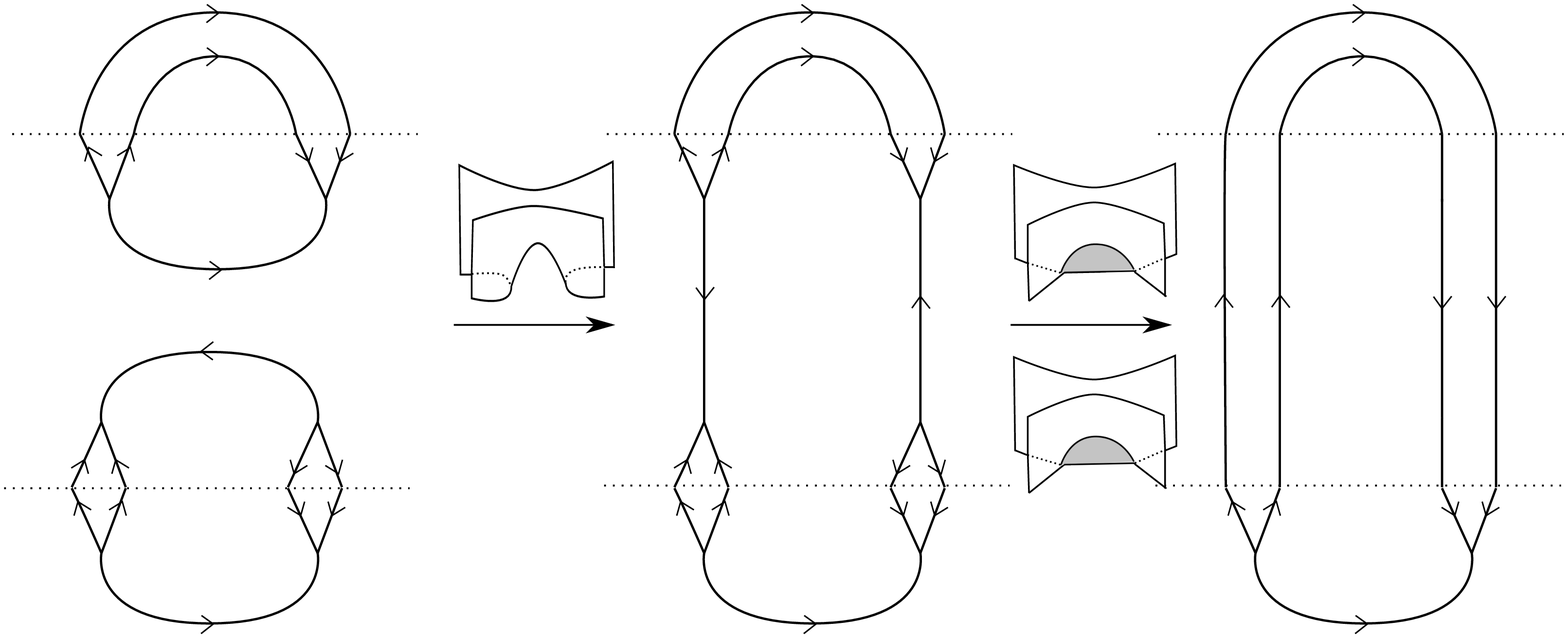}};(51,17)*{w^*};
 (-28,-17)*{v};(-26.3,17)*{w^*};
 (-27.2,-7)*{v^*};(-27,8)*{v};
 (50,-17)*{v};\endxy
\end{align}
\end{ex}
\begin{prop}
The foam $m_v$ in Definition~\ref{multfoam} only depends on the isotopy type of $v$.
\end{prop}
\begin{proof} We have to show that $m_v$ is 
independent of the way $v$ is expressed using the growth algorithm 
(Definition~\ref{growth}). Let $G_{1}$ and $G_{2}$ be two different expressions of 
$v$ using the growth algorithm. We have to compare $G_1$ and $G_2$ 
walking backwards in the growth algorithm. Note that we only have to worry 
about two consecutive steps in the same region of $v$. Reordering steps in 
``distant'' regions of $v$ corresponds to an isotopy which simply alters 
the height function on $m_v$. With these observations, the only possible 
remaining difference between the last two steps in $G_1$ and $G_2$ is 
the following.
\begin{align}\label{Yreplace}
	\xy(0,0)*{\includegraphics[width=80px]{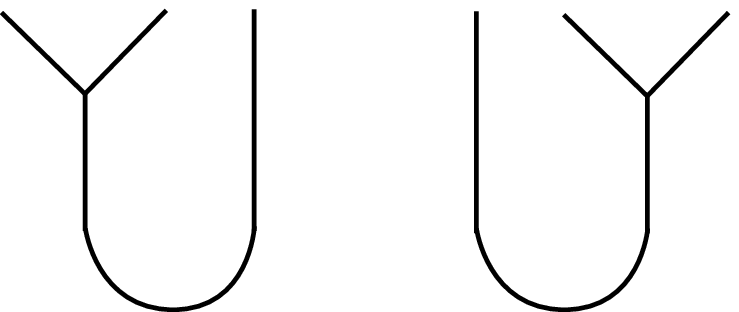}};\endxy
\end{align}
If the last two steps in $G_1$ and $G_2$ are equal, we have to go further back 
in the growth algorithm. Besides two-step differences of the same 
sort as above, we can encounter another one of the following sort.
\begin{align}\label{Hreplace}
   \xy(0,0)*{\includegraphics[width=80px]{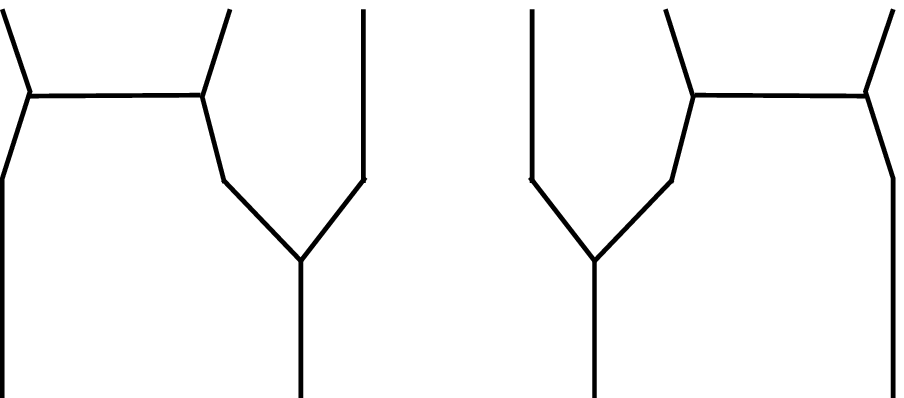}};\endxy
\end{align}
We have to check that the above two-step differences in $G_1$ and $G_2$ 
correspond to equivalent foams. In the first case, the foams in the 
multiplication algorithm are given by 
\begin{figure}[H]
   \centering
     \includegraphics[width=200px]{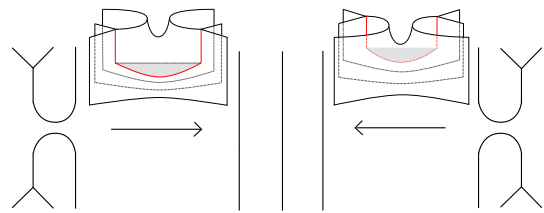}
     \caption{A possible local difference between $m_{G_{1}}$ and $m_{G_{2}}$.}
     \label{Yfoam}
\end{figure}

In the second case, we get 
\begin{figure}[H] 
   \centering
    \includegraphics[width=230px]{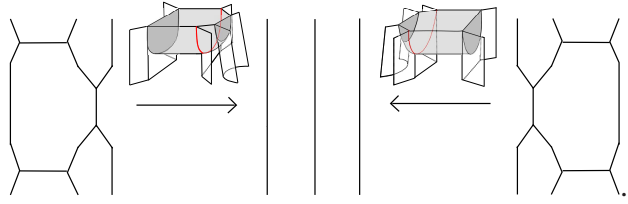}.
    \caption{The other possible local difference between $m_{G_{1}}$ and $m_{G_{2}}$.}
    \label{Hfoam}
\end{figure}

The two foams in Figure~\ref{Yfoam} are isotopic - one foam can be produced 
from the other by sliding the red singular arc over the saddle as illustrated below.
\begin{align}
   \xy(0,0)*{\includegraphics[width=150px]{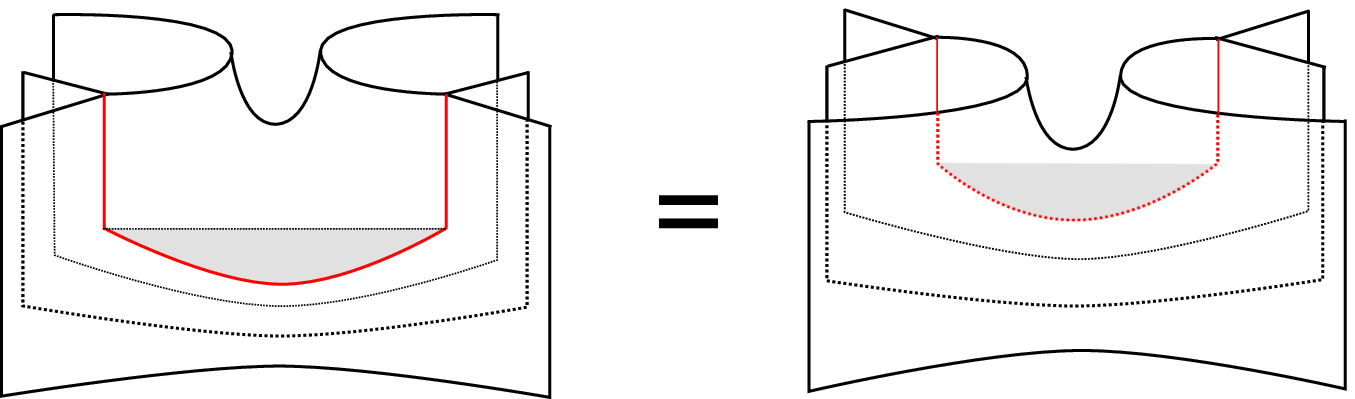}};\endxy
\end{align}
The two foams in Figure~\ref{Hfoam} are also isotopic - one foam can be 
produced from the other by moving the red singular arc to the right or to the 
left as illustrated below.
\begin{align}
   \xy(0,0)*{\includegraphics[width=160px]{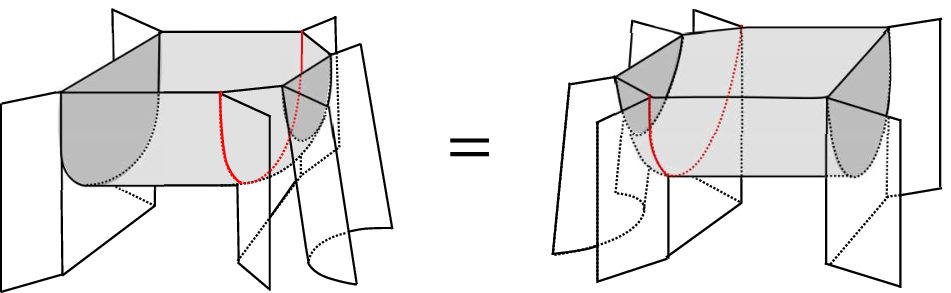}};\endxy 
\end{align}
The cases above are the only possible ones, so their verification provides 
the proof. 
\end{proof}
\begin{prop} \label{prop:multqgrade}
The foam $m_v$ has $q$-grading $n$.
\end{prop}
\begin{proof}
We proceed by backward induction on the level of the growth algorithm 
expressing $v$. At the final level of the growth algorithm, 
the only possible rule is the arc rule. Resolving a corresponding pair of 
arcs in $v$ and $v^*$ results in two new vertical strands and is obtained 
by a saddle point cobordism, which has $q$-grading 2. 

Let $n_{k}$ be the number of vertical strands and $m_v^{k}$ be the foam after 
resolving the last $k$ rules in the growth algorithm of $v$. Suppose that 
$n_k$ is equal to the $q$-degree of $m_v^k$. In the next 
step of the multiplication we can have three cases.
\begin{enumerate}
\item The resolution of a pair of arc rules. In this case we have 
$n_{k+1}=n_k+2$ and $m_v^{k+1}$ is obtained from $m_v^k$ by adding a saddle, 
which adds 2 to the $q$-grading.
\item The resolution of a pair of Y rules. In this case we have 
$n_{k+1}=n_k+1$ and $m_v^{k+1}$ is obtained from $m_v^k$ by adding an unzip, 
which adds 1 to the $q$-grading.
\item The resolution of a pair of H rules. In this case we have 
$n_{k+1}=n_k$ and $m_v^{k+1}$ is obtained from $m_v^k$ by adding a square foam, 
which adds $0$ to the $q$-grading. 
\end{enumerate}
\end{proof}
There is a useful alternative definition of $\mathcal{W}^c_S$, 
which we give below. 
As a service to the reader, we state it as a lemma and prove that it really is 
equivalent to our definition above. Both definitions have their advantages 
and disadvantages, so it is worthwhile to catalogue both in this thesis. 
\begin{lem}
\label{lem:webalgaltern}
For any $c\in\mathbb{C}$ and any $u,v\in B_S$, we have a 
grading preserving isomorphism 
\[
\foamt^c(u,v)\cong{}_{u}\mathcal{W}^c_{v}.
\] 

Using this isomorphism, the multiplication  
\[
{}_{u}{\mathcal W}^c_{v}\otimes {}_{v'}{\mathcal W}^c_{w}\to 
{}_{u}{\mathcal W}^c_{w}
\]
corresponds to the composition 
\[
\foamt^c(u,v)\otimes\foamt^c(v',w)\to\foamt^c(u,w),
\]
if $v=v'$, and is zero otherwise. 
\end{lem}
\begin{proof}
The isomorphism of the first claim is sketched in the following figure.
\begin{align*}
	\xy(0,0)*{\includegraphics[width=180px]{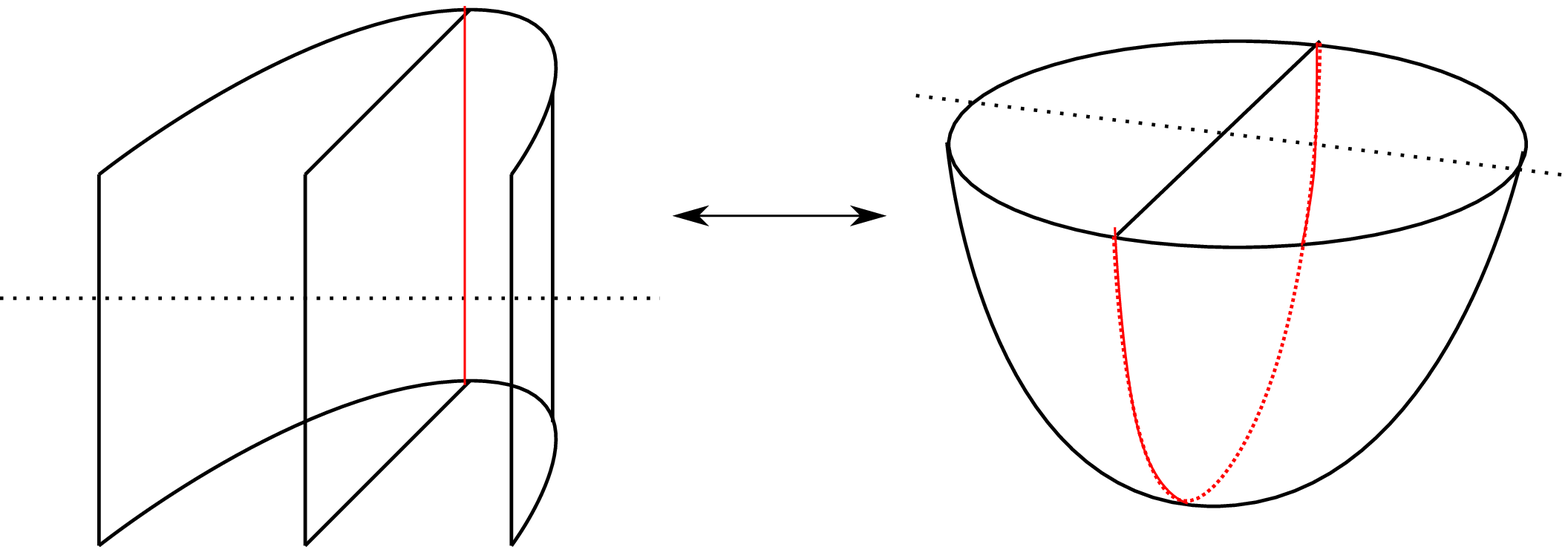}};(-5.5,5)*{v^*};
 (28.5,10.5)*{v^*};(-6,-6)*{v};
 (23,-1)*{v};\endxy
\end{align*}
The proof of the second claim follows from analysing what the 
isomorphism does to the resolution of a pair of arc, Y or H-rules in the 
multiplication foam. This is done below. 
\begin{align*}
	\xy(0,0)*{\includegraphics[width=180px]{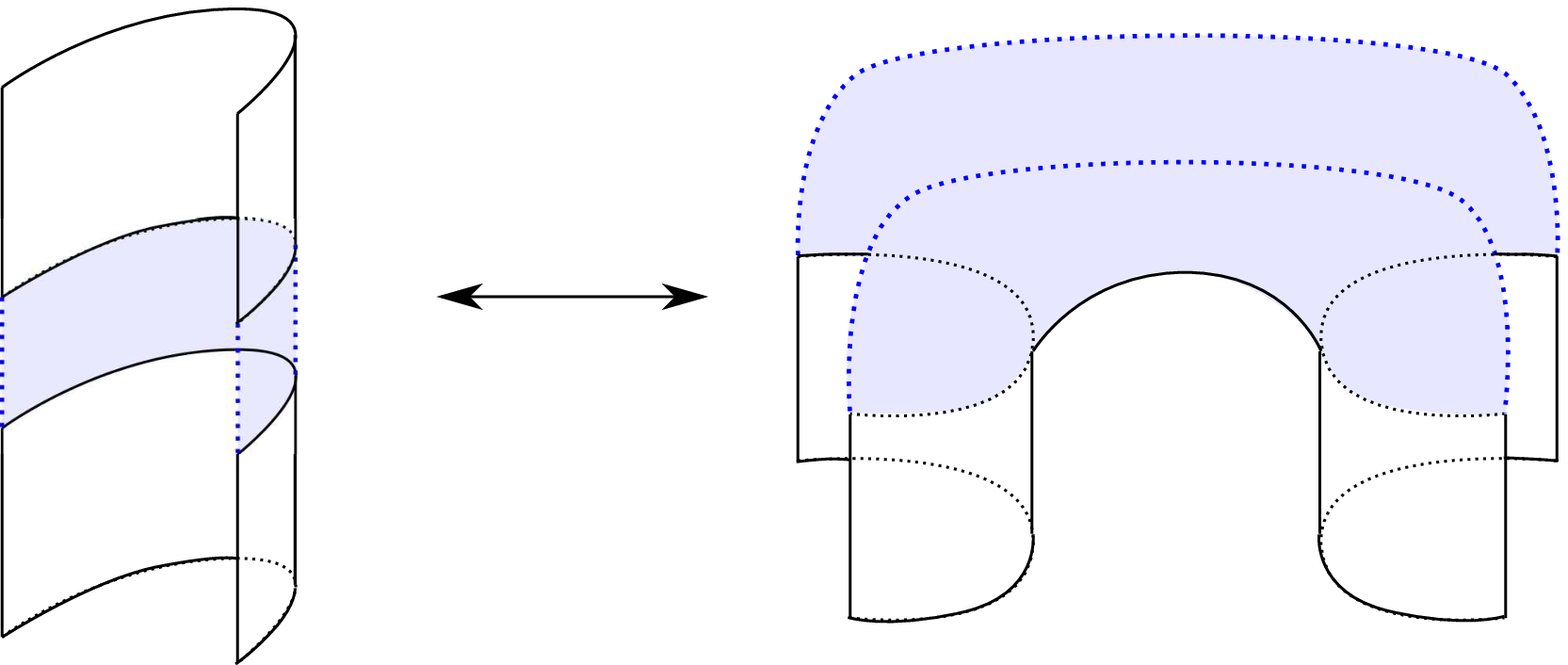}};(-17,8)*{f_2};
 (-1,-7.5)*{f_1};(-17,-7)*{f_1};
 (33,-7.5)*{f_2};\endxy\\
	\xy(0,0)*{\includegraphics[width=180px]{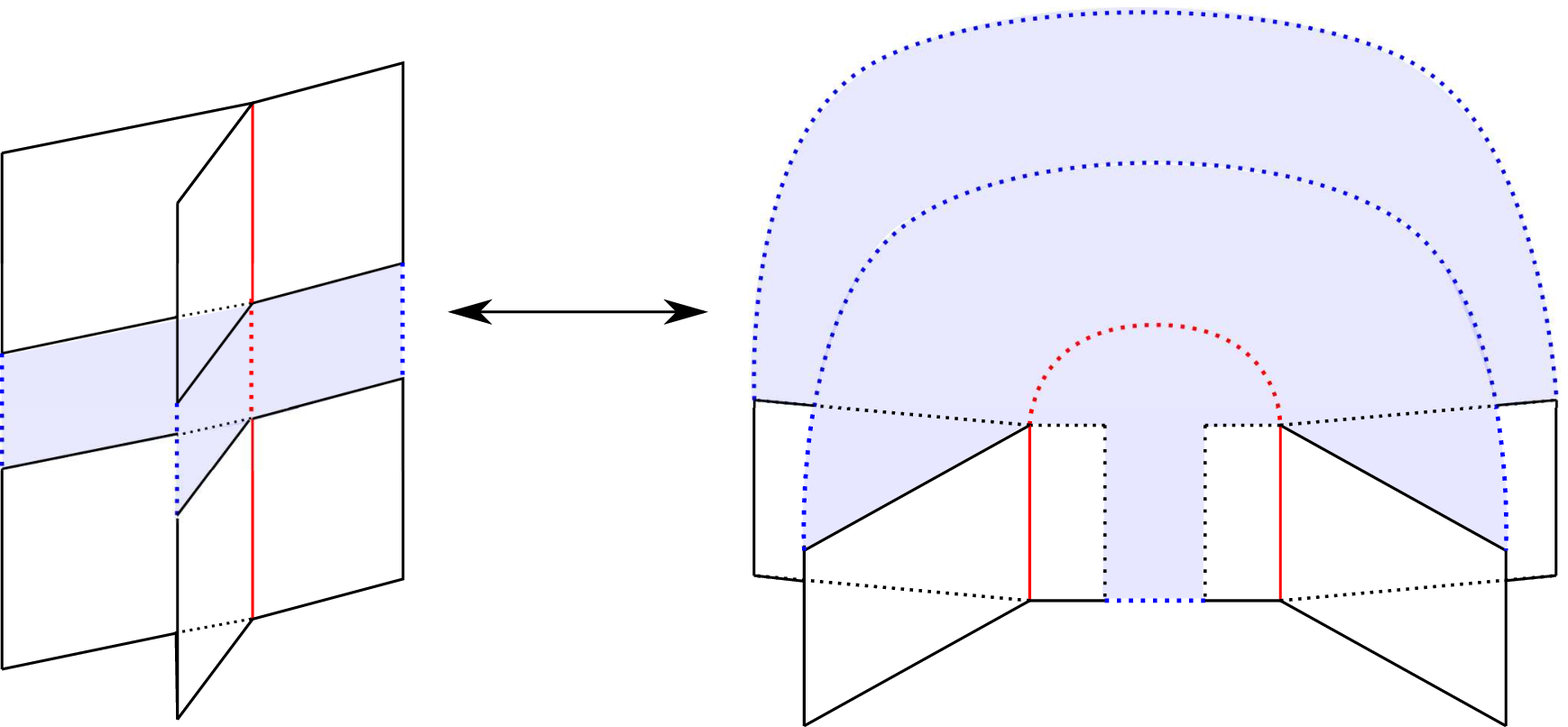}};(-13,9)*{f_2};
 (-4,-9.5)*{f_1};(-13,-5.5)*{f_1};
 (33,-9.5)*{f_2};\endxy\\
	\xy(0,0)*{\includegraphics[width=180px]{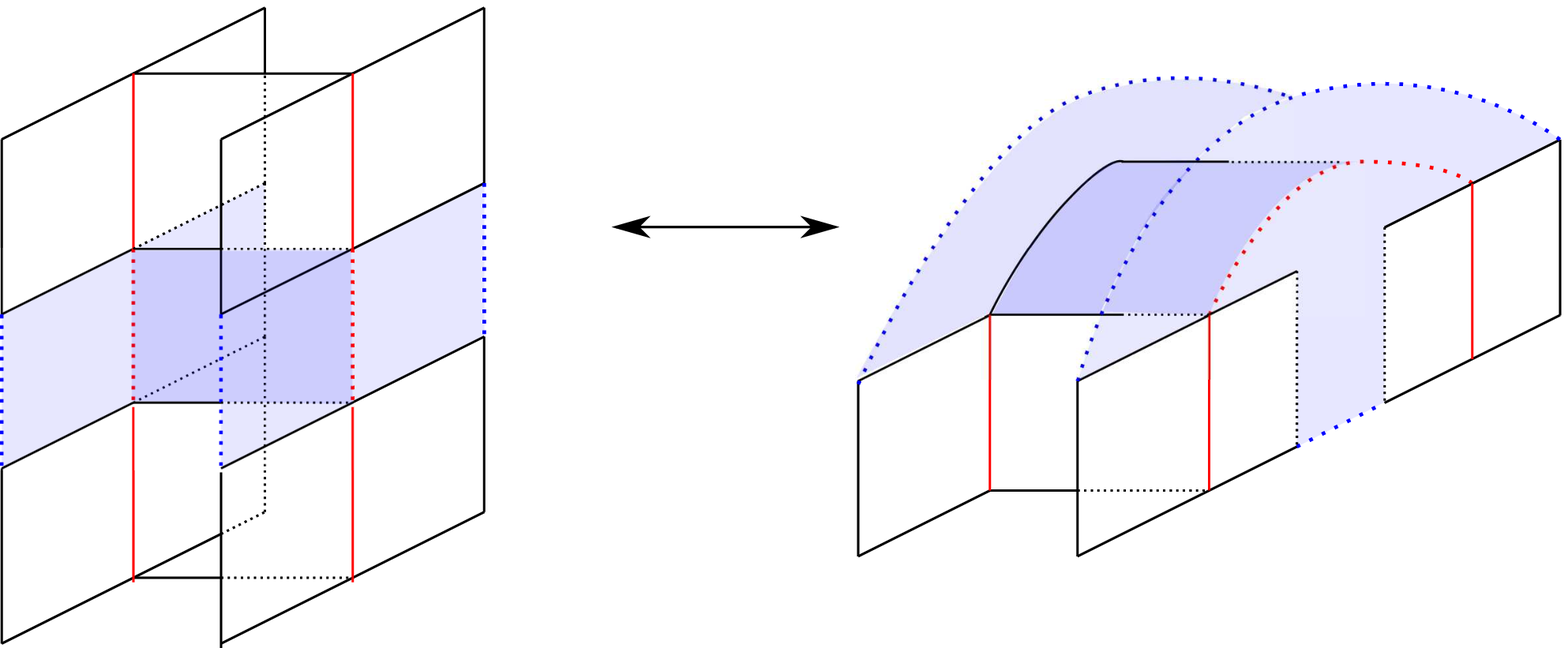}};(-9,10)*{f_2};
 (20,-8.6)*{f_1};(-9,-4.5)*{f_1};
 (30,-3)*{f_2};\endxy
\end{align*}
\end{proof}
Note that Lemma~\ref{lem:webalgaltern} implies that ${\mathcal W}_S^c$ is 
associative and unital, something that is not immediately clear from 
Definition~\ref{defn:webalg}. 
For any $u\in B_S$, the identity $1_u\in\foamt^c(u,u)$ defines an idempotent. 
We have 
\[
1=\sum_{u\in B_S} 1_u\in {\mathcal W}_S^c.
\]
Alternatively, one can see ${\mathcal W}_S^c$ as a category whose 
objects are the elements in $B_S$ such that the module of morphisms 
between $u\in B_S$ and $v\in B_S$ is given by $\foamt^c(u,v)$. 
In this thesis we will mostly see 
${\mathcal W}_S^c$ as an algebra, but will sometimes refer to the category 
point of view. 
\vskip0.5cm
In this thesis, we will study ${\mathcal W}_S^c$ for two special 
values of $c\in\mathbb{C}$. 
\begin{defn}\label{defn-khgo}
Let $K_S$ and $G_S$ be the complex algebras obtained from ${\mathcal W}_S^c$ 
by setting $c=0$ and $c=1$, respectively. We call them 
\textit{Khovanov's web algebra} and \textit{Gornik's web algebra}, respectively, 
to distinguish them throughout Section~\ref{sec-web}.  
\end{defn}
Note that $G_S$ is a filtered algebra. Its associated 
graded algebra is $K_S$. By Lemma~\ref{lem:webalgaltern}, 
both $K_S$ and $G_S$ are finite dimensional, unital, associative algebras. 
They also have similar decompositions as shown below.
\[
K_S=\bigoplus_{u,v\in B_S}{}_uK_v\,,\quad\quad G_S=\bigoplus_{u,v\in B_S}{}_uG_v.
\] 
We now recall the definition of complex, 
graded and filtered Frobenius algebras. 
Let $A$ be a finite dimensional, graded, complex algebra and let  
$\mathrm{hom}_{\mathbb{C}}(A,\mathbb{C})$ be the complex vector space of 
grading preserving 
maps. The \textit{dual} of $A$ is defined by      
\[
A^{\vee}=\bigoplus_{n\in\mathbb{Z}}\mathrm{hom}_{\bC}(A,\mathbb{C}\{n\}),
\] 
where $\{n\}$ denotes an upward degree shift of size $n$.
Note that $A^{\vee}$ is also a graded module, such that 
\begin{equation}
\label{eq:dualgrading}
(A^{\vee})_i=(A_{-i})^{\vee},
\end{equation}
for any $i\in \mathbb{Z}$.
Then $A$ is called a \textit{graded, 
symmetric Frobenius algebra of Gorenstein parameter} $\ell$, 
if there exists an isomorphism of graded $(A,A)$-bimodules 
\[
A^{\vee}\cong A\{-\ell\}.
\] 
If $A$ is a complex, finite dimensional, filtered algebra, 
let $\mathrm{hom}_{\mathbb{C}}(A,\mathbb{C})$ be the complex vector space of 
filtration preserving 
maps. The \textit{dual} of $A$ is defined by     
\[
A^{\vee}=\bigoplus_{n\in\mathbb{Z}}\mathrm{hom}_{\bC}(A,\mathbb{C}\{n\}),
\]
where $\{n\}$ denotes an upward suspension of size $n$.
Note that $A^{\vee}$ is also a filtered module, such that 
\begin{equation}
\label{eq:dualfiltration}
(A^{\vee})_i=(A_{-i})^{\vee},
\end{equation}
for any $i\in \mathbb{Z}$.
Then $A$ is called a \textit{filtered, 
symmetric Frobenius algebra of Gorenstein parameter} $\ell$, 
if there exists an isomorphism of filtered $(A,A)$-bimodules 
\[
A^{\vee}\cong A\{-\ell\}.
\] 
For more information on graded Frobenius algebras, see~\cite{ue} and 
the references therein, for example. We do not have a good reference for 
filtered Frobenius algebras, but it is a straightforward generalisation of 
the graded case. We explain some basic results on the character theory of 
filtered and graded, symmetric Frobenius algebras in 
Section~\ref{sec-appendix}. 
\begin{thm}\label{thm:frob} For any sign string $S$ of length $n$, 
the algebra $K_S$ is a graded, symmetric Frobenius algebra 
and $G_S$ is a filtered, symmetric Frobenius algebra, both of Gorenstein 
parameter $2n$. 
\end{thm}
\begin{proof} First, let $c=0$. 
We take, by definition, the trace form 
\[
\mathrm{tr}\colon K_S\to \mathbb{C}
\] 
to be zero on ${}_uK_v$, when $u\ne v\in B_S$. 
For any $v\in B_S$, we define 
\[
\mathrm{tr}\colon {}_vK_v\to\mathbb{C}
\] 
by closing any foam $f_v$ with $1_v$, e.g. as pictured below.
\begin{align*}
	\xy(0,0)*{\includegraphics[width=60px]{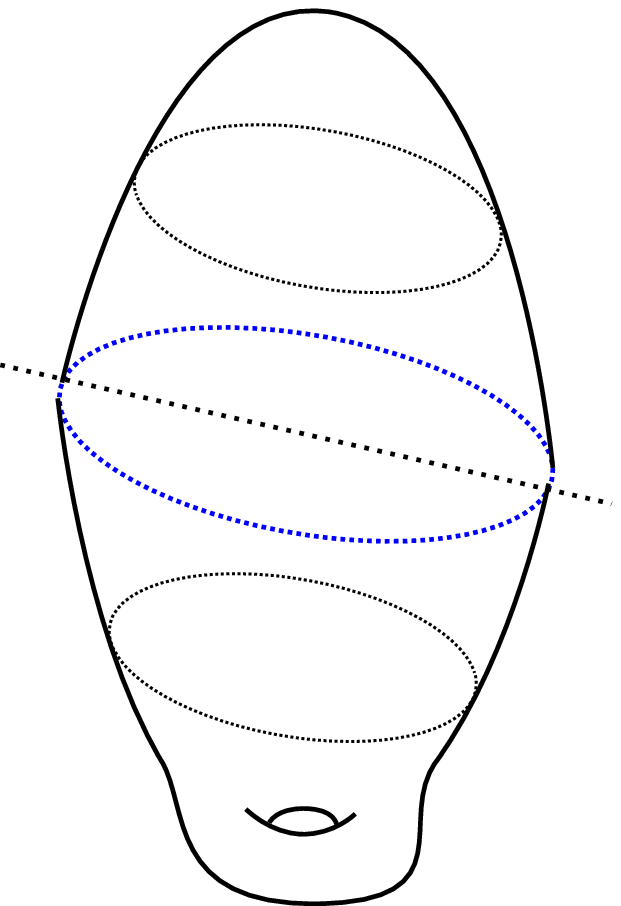}};(-9,10)*{1_v};
 (-9,-8)*{f_v};(10.5,3)*{v^*};
 (9.5,-3.5)*{v};\endxy
\end{align*}  
Equivalently, in $\foamt^0(v,v)$, closing $f_v$ by $1_v$,
\begin{align*}
	\xy(0,0)*{\includegraphics[width=180px]{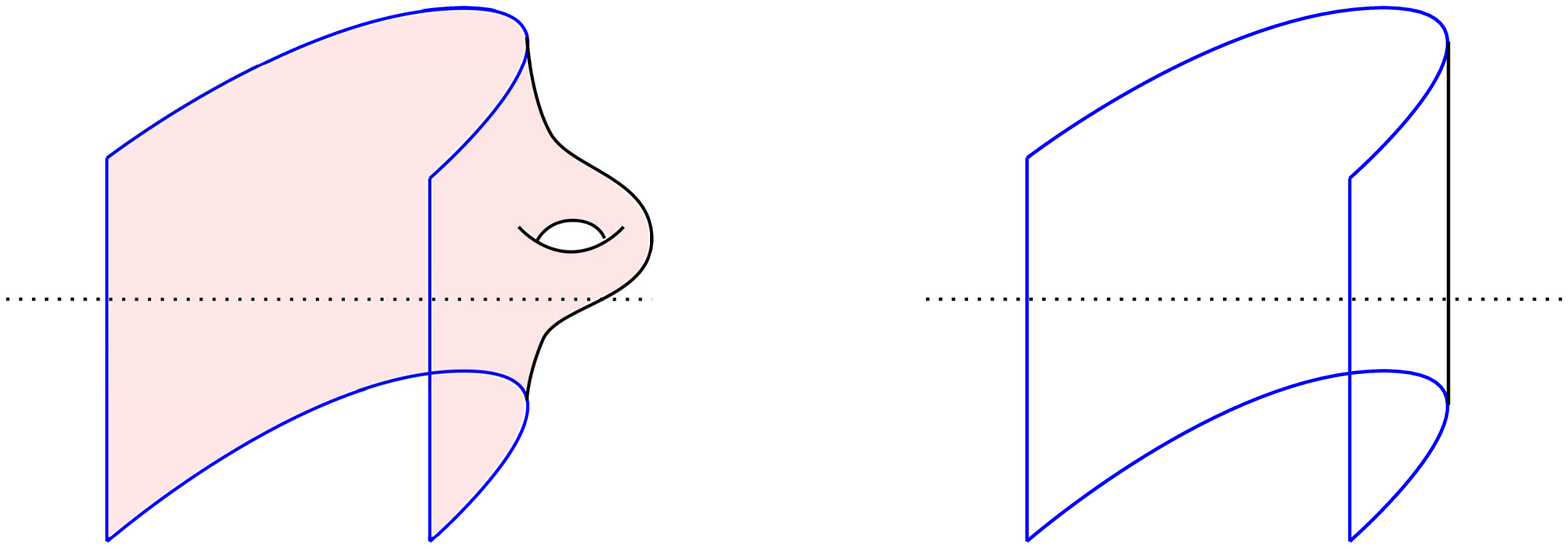}};(-30,2)*{f_v};(-4.5,5)*{v^*};
 (-5,-5.5)*{v};(8,2)*{1_v};(30,5)*{v^*};
 (29.5,-5.5)*{v};\endxy
\end{align*}  
gives
\begin{align*}
	\xy(0,0)*{\includegraphics[width=95px]{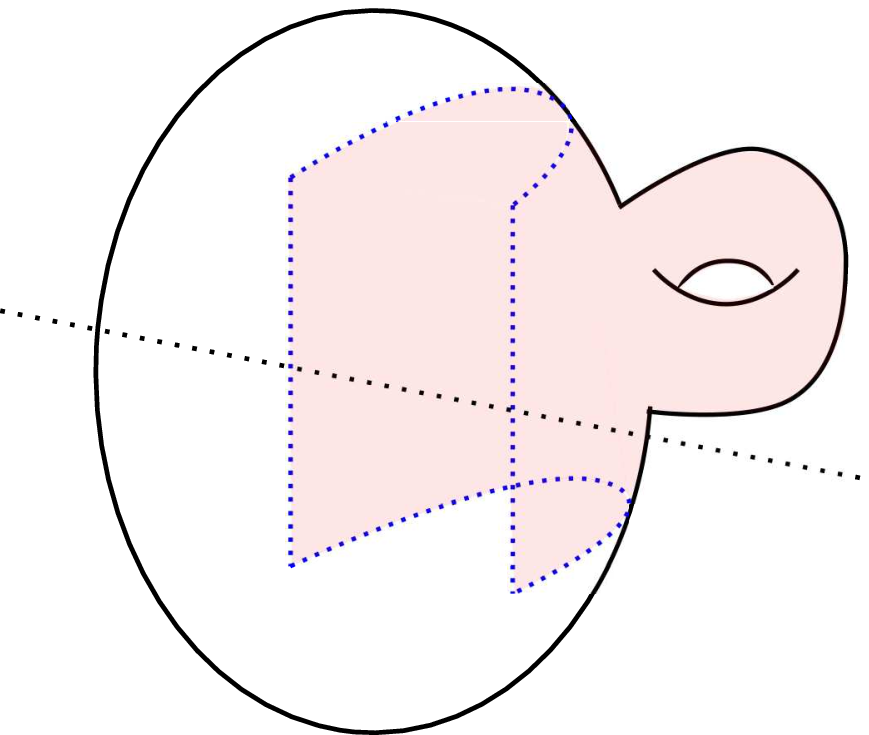}};(-15,4.5)*{1_v};
 (19,4.5)*{f_v};\endxy
\end{align*}
The fact that the trace form is non-degenerate follows immediately 
from the closure relation in Section~\ref{sec-webbasicb}.

The fact that $\mathrm{tr}(gf)=\mathrm{tr}(fg)$ holds follows from sliding 
$f$ around the closure until it appears on the other side of $g$, e.g. as shown below.
\begin{align*}
\xy(0,1)*{\includegraphics[width=90px]{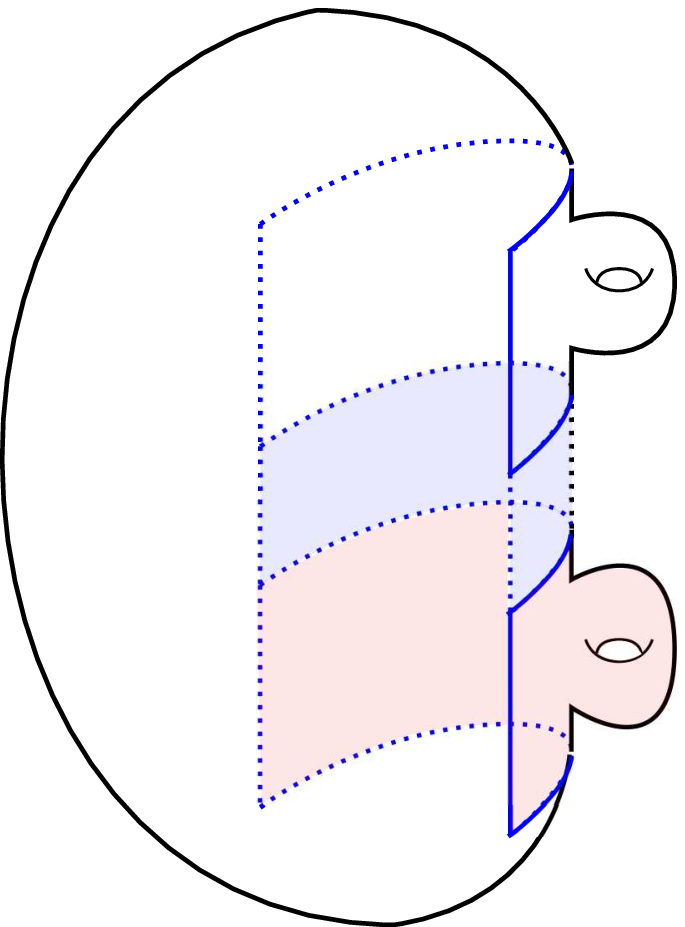}};(-18.5,1)*{1_u};
 (18.5,1)*{1_v};(17.5,-8)*{f};
 (17.5,9)*{g};\endxy\quad=\quad\xy(0,0)*{\includegraphics[width=88px]{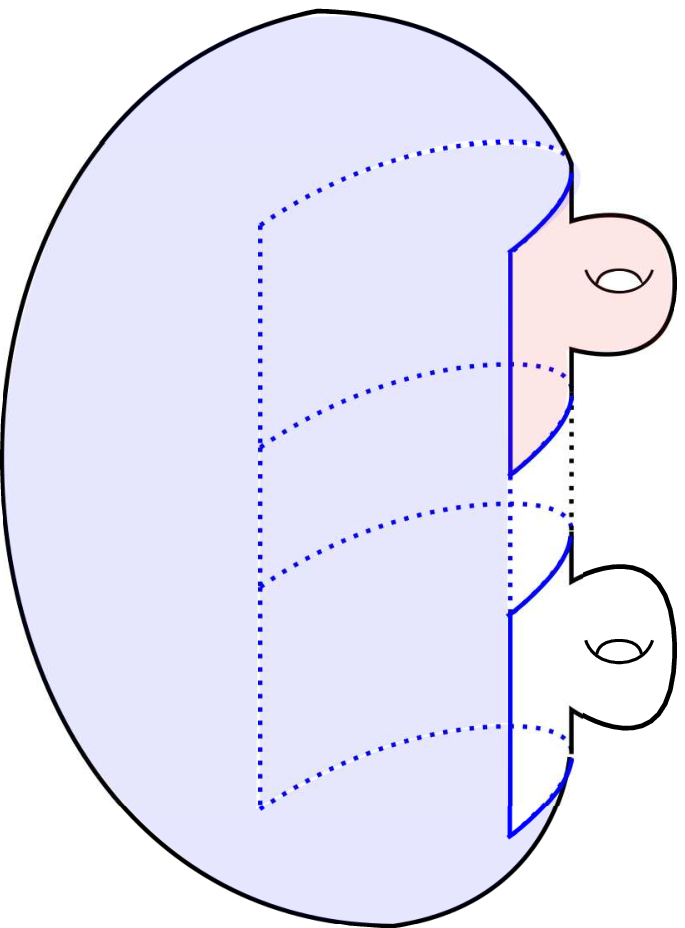}};(-18,0)*{1_v};
 (18.5,0)*{1_u};(17.5,-9)*{g};
 (17.5,8)*{f};\endxy
\end{align*}    
Note that a closed foam can only have non-zero evaluation if it has 
degree zero. Therefore, for any $u\in B_S$ and any 
two homogeneous elements $f\in \F^0(u^*v)$ and $g\in \F^0(v^*u)$, we have 
$\mathrm{tr}(fg)\ne 0$ unless $\deg(f)=-\deg(g)$. By the shift in  
\[
{}_uK_v=\mathcal{F}^0(u^*v)\{n\}
\]
and by~\eqref{eq:dualgrading}, 
this implies that 
the non-degenerate trace form on $K_S$ gives rise to a graded 
$(K_S,K_S)$-bimodule isomorphism   
\begin{equation}
\label{eq:frobK}
K_S^{\vee}\cong K_S\{-2n\}.
\end{equation}
Now, let $c=1$. Then the construction above also gives a 
non-degenerate bilinear form on $G_S$. Moreover, it induces 
a filtration preserving bijective $\mathbb{C}$-linear map 
of filtered $(G_S,G_S)$-bimodules 
\begin{equation}
\label{eq:frobG}
G_s\{-2n\}\to G_s^{\vee}.
\end{equation}
The associated graded map is precisely the isomorphism in~\eqref{eq:frobK}. By Proposition~\ref{prop:Srid}, this implies that 
the map in~\eqref{eq:frobG} is a strict isomorphism 
of filtered $(G_S,G_S)$-bimodules.   
\end{proof}
We now explain some of Gornik's results, which are relevant for $G_S$. 
Recall that $R^1_{u^*v}$ is the commutative ring associated to $u^*v$, 
generated by the edge variables of $u^*v$ and mod out by the ideal, which,  
for each trivalent vertex in $u^*v$, is generated by the relations 
\begin{equation}
\label{eqn:R1}
x_1+x_2+x_3=0,\qquad x_1x_2+x_1x_3+x_2x_3=0,\qquad x_1x_2x_3=1,
\end{equation}
where $x_1,x_2$ and $x_3$ are the edge variables around the vertex. 
The algebra $R^1_{u^*v}$ acts on ${}_uG_v$ in such a way 
that each edge variable corresponds to adding a dot on the incident facet.
See~\cite{gornik},~\cite{kh3} or~\cite{mv1} for the precise definition and more details. 

In what follows, 3-colourings will always be assumed to be admissible and 
we therefore omit the adjective. Theorem 3 in~\cite{gornik} proves the following.   
\begin{thm}\label{thm:Gornik}(\textbf{Gornik})
There is a complete set of orthogonal idempotents $e_T\in R^1_{u^*v}$, indexed 
by the 3-colourings $T$ of $u^*v$. The number of 3-colourings of $u^*v$ is 
exactly equal to $\dim_q({}_uG_v)$. 

These idempotents are not filtration preserving, 
but as an $R^1_{u^*v}$-module (i.e. forgetting the filtration on ${}_uG_v$ 
and its left ${}_uG_u$ and right ${}_vG_v$-module structures) 
we have  
\[
{}_uG_v\cong \bigoplus_T \mathbb{C}e_T.
\]
\end{thm}
Let us have a closer look at Gorniks idempotents.  
First of all, in the proof of Theorem 3 in~\cite{gornik} Gornik notes 
that for any edge $i$ and any 3-colouring $T$ of $u^*v$, we have 
\begin{equation}
\label{eqn:edgeaction}
x_ie_T=\zeta^{T_i}e_T\in R^1_{u^*v},
\end{equation}
where $\zeta$ is a primitive third root of unity, $x_i$ is the edge variable 
and $T_i$ the colour of the edge 
(see (4) in~\cite{mv1} for this result in the context of foams). 

Furthermore, a 3-colouring of $u^*v$ 
actually corresponds to a pair of 3-colourings of $u$ and $v^*$ 
that match at the boundary. Of course, there is a bijective correspondence 
between 3-colourings of $v$ and $v^*$, so we see that a 3-colouring of 
$u^*v$ corresponds to a matching pair of 3-colourings of $u$ and $v$. Recall that 3-colourings can be seen as flows.

Recall that ${}_uG_v$ is a left ${}_uG_u$-module and a right ${}_vG_v$-module.  
Let $T_1$ and $T_2$ be a pair of matching 3-colourings of $u$ and $v$, 
respectively, which together give a 3-colouring $T$ of $u^*v$. 
Then the action of $e_T$ on any $f\colon u\to v$ can be written as 
\[
e_{T_1}fe_{T_2}.
\]
To show that this notation really makes sense, define 
\textit{Gornik's symmetric idempotent} associated to $T_1$ as
\[
e_{u,T_1}=e_{T_1}1_ue_{T_1}.
\]
So we let the Gornik idempotent associated to the symmetric 3-colouring 
of $u^*u$, given by $T_1$ both on $u$ and $u^*$, act on $1_u$. Then 
we have 
\[
e_{T_1}fe_{T_2}=e_{u,T_1}fe_{v,T_2},
\]
where on the right-hand side we really mean composition. 

We immediately see that 
\[
e_{T_1}1_ue_{T_2}=0\Leftrightarrow T_1\ne T_2
\]
and 
\[
e_{T_1}1_ue_{T_1}e_{T_2}1_ue_{T_2}=\delta_{1,2} e_{T_1}1_ue_{T_1}\quad\text{and}\quad 
\sum_{T}e_{T}1_ue_{T}=1_u,
\]
where the sum is over all 3-colourings of $u$. 
This shows that the $e_{u,T}$, for all 3-colourings 
$T$ of a given $u\in B_S$, are orthogonal idempotents in ${}_uG_u$.  
It also implies that 
\[
e_{T_1}1_ue_{T_1}=e_{T_1}1_u=1_ue_{T_1},
\]
so it is enough to label just the source or just the target of $1_u$.
For this purpose, we define $R^1_u$ to be ``half'' of $R^1_{u^*u}$, i.e. the 
subring which is only generated by the edge variables of $u$. To be precise, 
we have 
\[
R^1_{u^*u}\cong R^1_u\otimes_{S}R^1_u,
\]
where $\otimes_{S}$ indicates that we impose the relation 
$x\otimes 1=1\otimes x$, for any $x$ corresponding to a boundary edge of $u$. 

If $u$ has no closed cycles, then all the 3-colourings of 
$u^*u$ are symmetric, because they are completely determined by the colours on 
the boundary of $u$. In that case 
\[
e_T\mapsto e_{u,T}
\]
defines an isomorphism of algebras $R^1_u\cong {}_uG_u$. In particular, 
${}_uG_u$ is commutative. This is not true in general, but we can prove 
the following.
\begin{lem}
\label{lem:embedding} 
For any $u\in B_S$, the map 
\[
x\mapsto x1_u
\]
defines a strict embedding of filtered $R^1_u$-modules
\[
\iota\colon R^1_u\to {}_uG_u.
\]
In particular, we see that $(R^1_u)_0\cong\text{Im}(\iota)_0\cong\mathbb{C}1_u$.
\end{lem}
\begin{proof}
The map is clearly a homomorphism of filtered algebras. 

The relations~\eqref{eq:dotm} correspond precisely to 
the relations in $R^1_u$, because the only singular edges in $1_u$ are the 
ones corresponding to the trivalent vertices of $u$. 
This shows that it is a strict embedding. 
\end{proof}
For any $u\in B_S$, we define the graded ring 
\[
R^0_u=E(R^1_u).
\]
This ring is the one which appears 
in Khovanov's original paper~\cite{kh3}. In $R^0_u$ we have the relations
\begin{equation}
\label{eqn:R0}
x_1+x_2+x_3=0,\qquad x_1x_2+x_1x_3+x_2x_3=0,\qquad x_1x_2x_3=0.
\end{equation}
The reader should compare them to~\eqref{eqn:R1}.

There are no analogues of the Gornik idempotents in $R^0_u$, but we do 
have an analogue of Lemma~\ref{lem:embedding}.
\begin{lem}
\label{lem:gradedembedding} 
For any $u\in B_S$, the map 
\[
x\mapsto x1_u
\]
defines an embedding of graded $R^0_u$-modules
\[
E(\iota)\colon R^0_u\to {}_uK_u.
\]
In particular, we see that $(R^0_u)_0\cong\text{Im}(E(\iota))_0\cong\mathbb{C}1_u$.
\end{lem}
 Another interesting consequence of Theorem~\ref{thm:Gornik} is the following. 
\begin{prop}
\label{prop:Gsemisimple}
As a complex algebra, i.e. without taking the filtration into account, 
$G_S$ is semisimple.
\end{prop}
\begin{proof}
For any $u\in B_S$ and any 3-colouring $T$ of $u$, define the projective 
$G_S$-module 
\[
P_{u,T}=(G_S)e_{u,T},
\]
where $e_{u,T}$ is Gornik's symmetric idempotent in $G_S$ defined above. 
Theorem~\ref{thm:Gornik} and our subsequent analysis of Gornik's idempotents 
show that the $P_{u,T}$ form a complete set of indecomposable projective 
$G_S$-modules. Furthermore, we have 
\[
\text{Hom}_{G_S}(P_{u,T},P_{v,T'})\cong e_{u,T}(G_S)e_{v,T'}\cong 
\begin{cases}
\mathbb{C},&\quad\text{if}\; T\;\text{and}\;T'\;\text{match at}\; S,\\
\{0\},&\quad\text{else}.
\end{cases}
\]  
This shows that $P_{u,T}\cong P_{v,T'}$ if and only if 
$T$ and $T'$ match at the common boundary. It also shows that 
if $P_{u,T}\not\cong P_{v,T'}$, then 
\[
\text{Hom}_{G_S}(P_{u,T},P_{v,T'})=\text{Hom}_{G_S}(P_{v,T'},P_{u,T})=\{0\}.
\]
Finally, it shows that each $P_{u,T}$ has only one composition factor, i.e. 
$P_{u,T}$ is irreducible. 

It is well-known that this implies that $G_S$ is semisimple; see 
Proposition 1.8.5 in~\cite{be} for example.   
\end{proof}
By Proposition~\ref{prop:Gsemisimple}, it is clear that for each 
$u\in B_S$ and each colouring $T$ of $u$, the corresponding block 
in $G_S$ is isomorphic to $\mathrm{End}(P_{u,T})$. 
In Section~\ref{sec-webcenter}, we will determine the central idempotents 
of $G_S$.
\subsection{The center of the web algebra and the cohomology ring of the Spaltenstein variety}\label{sec-webcenter}
For the rest of this section (and the following two sections), choose arbitrary but fixed non-negative integers 
$n\geq 2$ and $k\leq n$, such that $d=3k\geq n$. Let 
\[
\Lambda(n,d)=\left\{\mu\in \mathbb{N}^n\mid \sum_{i=1}^n\mu_i=d\right\}
\]
be the set of \textit{compositions} of $d$ of length $n$. 
By $\Lambda^+(n,d)\subset\Lambda(n,d)$ we denote the 
subset of \textit{partitions}, i.e. all $\mu\in\Lambda(n,d)$ such that 
\[
\mu_1\geq\mu_2\geq \ldots\geq \mu_n\geq 0.
\] 
Also for the rest of this section (and the following two sections), choose an arbitrary but fixed sign string 
$S$ of length $n$. We associate to $S$ a unique element 
$\mu=\mu_S\in\Lambda(n,d)$, such that    
\[
\mu_i=
\begin{cases}
1,&\quad \mathrm{if}\quad s_i=+,\\
2,&\quad \mathrm{if}\quad s_i=-. 
\end{cases}
\] 
Let $\Lambda(n,d)_{1,2}\subset \Lambda(n,d)$ be the subset of compositions 
whose entries are all $1$ or $2$. For any sign string $S$, we have 
$\mu_S\in \Lambda(n,d)_{1,2}$. 

Let $\lambda=(3^k)\in \Lambda(n,d)$. Let $\mathrm{Col}_{\mu}^{\lambda}$ be the 
set of column strict tableaux of shape $\lambda$ and type $\mu$, both 
of length $n$. It is well-known that there is a bijection between 
$\mathrm{Col}_{\mu}^{\lambda}$ and the tensor basis of 
\[
V_{\mu}=V_{\mu_1}\otimes \cdots\otimes V_{\mu_n},
\]
where $V_1=V_+$ and $V_2=V_1\wedge V_1\cong V_-$ (see Section 3 in~\cite{mast}, 
for example). However, we are interested in tensors as summands in 
the decomposition of elements in $B_S$. Therefore, we prove 
Proposition~\ref{prop:tableauxflows} in Section~\ref{sec-webflow}. 
The reader, who is not interested in the details of the proof of this 
proposition, can choose to skip Section~\ref{sec-webflow} at a first reading and 
just read the statement of the proposition.
\subsection{Tableaux and flows}\label{sec-webflow}
Let $p_S$ be the number of positive entries and $n_S$ the number of negative 
entries of $S$. By definition, we have that $d=p_S+2n_S$. The key idea 
in this section is to reduce all proofs to the case where $n_S=0$. 

\begin{defn} \label{def:posstr}
Fix any state string $J$ of length $n$, we define a new state string 
$\hat{J}$ of length $d$ by the following algorithm.
\begin{enumerate}
\item Let ${}_{0}\hat{J}$ be the empty string.
\item For $1\leq i\leq n$, let $_{i}\hat{J}$ be the result of 
concatenating $j_{i}$ to $_{i-1}\hat{J}$ if $\mu_{i}=1$. 
If $\mu_{i}=2$ then
\begin{enumerate}
\item concatenate $(1, 0)$ to $_{i-1}\hat{J}$ if $j_{i} = 1$.
\item concatenate $(0, -1)$ to $_{i-1}\hat{J}$ if $j_{i} = -1$.
\item concatenate $(1, -1)$ to $_{i-1}\hat{J}$ if $j_{i} = 0$.
\end{enumerate}
\end{enumerate}
We set $\hat{J} = {}_{n}\hat{J}$. Lastly, for any $c\in \{-1,0,1\}$, 
we define $\hat{J}^{c}$ to be the number of entries in $\hat{J}$ that is 
equal to $c$.
\end{defn} 

\begin{prop} 
\label{prop:tableauxflows}
There is a bijection between $\mathrm{Col}_{\mu}^{\lambda}$ and the 
set of state strings $J$ such that there exists a $w\in B_S$ and 
a flow $f$ on $w$ which extends $J$. 
\end{prop}
\noindent The proof of Proposition~\ref{prop:tableauxflows} follows 
directly from 
Lemmas~\ref{lem:ttostr} and~\ref{lem:strtoflow}.

\begin{lem} \label{lem:ttostr}
There is a bijection between $\mathrm{Col}_{\mu}^{\lambda}$ and 
state strings $J$ of length $n$ such that 
\begin{equation}\label{eqn:conds}
\hat{J}^{-1} = \hat{J}^{0} = \hat{J}^{1}.
\end{equation}
where the $\hat{J}^{c}$ are as defined in Definition~\ref{def:posstr}.
\end{lem}

\begin{proof}
Given a state string $J$ satisfying (\ref{eqn:conds}), we first give an 
algorithm to build a 3-column tableau $Y_J$, filled 
with integers from 1 to $n$. Afterwards, we show that $Y_J$ has 
shape $\lambda$. 

Begin by labelling the three columns with $1, 0$ and $-1$, 
reading from left to right. We are going to build up $Y_J$ from top to bottom. 
Start by taking $Y_J$ to be the empty tableau. Then, from $i=1$ to $i=n$, 
do the following.
\begin{enumerate}
\item If $\mu_{i}=1$, add one box labelled $i$ to column $j_i$ in $Y_J$.  
\item If $\mu_{i} = 2$, add two boxes labelled $i$ to columns $c_{1}$ and 
$c_{2}$, such that $c_1\ne c_2$ and $c_1+c_2=j_i$.
\end{enumerate}

We have to show that $Y_J$ belongs to $\mathrm{Col}_{\mu}^{\lambda}$. 
Since the algorithm builds up from top to bottom, $Y_J$ is strictly column 
increasing. To see that $Y_J$ has shape $\lambda$, we need to show 
that every row in $Y_J$ has three entries. Observe that the number of 
filled boxes in column $c$ of $Y_J$ is exactly equal to $\hat{J}^{c}$. 
Since we have assumed condition~\eqref{eqn:conds}, all three columns 
have the same length, therefore every row in $Y_J$ must have 
exactly three entries.
\vskip0.5cm
Conversely, let $T\in\mathrm{Col}_{\mu}^{\lambda}$. We define a state string 
$J$ as follows.
\[
j_{i}=\sum_{i\mathrm{\,appears\,in\,column\,c}} c.
\]
Since $\mu$ corresponds to a sign string and $T$ is column strict, 
we see that, for each $1\leq i\leq n$, $i$ 
can appear at most twice in $T$ but never twice in the same 
column. Thus, $j_{i}\in\{-1,0,1\}$, i.e. $J$ is a state string. It follows from 
the definition of $\hat{J}$ that $\hat{J}^{c}$ is equal to the length of 
column $c$ of $T$. Since $T$ is of shape $\lambda$, the number of boxes in 
each column is the same. Hence, condition (\ref{eqn:conds}) holds for $J$. 

It is straightforward to check that the above two constructions are inverse 
to each other and therefore determine a bijection. 
\end{proof}

\begin{lem} \label{lem:strtoflow}
A state string $J$ corresponds to the boundary state of a flow on a 
web $w\in B_S$ if and only if condition~\eqref{eqn:conds} holds for $J$.
\end{lem}

\begin{proof}
Let $w\in B_S$ be equipped with a flow with boundary state string 
$J$. We are going to show that $J$ satisfies condition~\eqref{eqn:conds} 
by induction on $n$. For $n=2$, $w$ can only be an arc. 
In this case it is simple to check that all flows on $w$ have corresponding 
boundary state strings satisfying condition~\eqref{eqn:conds}. 

For $n>2$, we express $w$ using the growth algorithm in an arbitrary, but 
fixed way, with the restriction that only one rule is applied per level. 
Let $_{k}J$ denote the boundary state string at the beginning of the $k$-th 
level in the growth algorithm and $_{k}\hat{J}$ the associated string as in 
Definition~\ref{def:posstr}. Similarly, let $_{k}\mu$ denote the composition 
corresponding to the sign string at the $k$-th level. 
Let us compare $_{k+1}J$ and $_{k}J$. They can 
only differ in the following ways.
\begin{enumerate}
\item In case an arc-rule is applied at the $k$-th level, $_{k}J$ 
can be obtained from $_{k+1}J$ by inserting the substring 
$(1,-1)$, $(0, 0)$ or $(-1,1)$ 
between the $i$-th and $i+1$-th entries in $_{k+1}J$. $_{k}\mu$ can be obtained 
from $_{k+1}\mu$ by inserting the substring $(1, 2)$ or $(2, 1)$ 
between the $i$-th and $i+1$-th entries in $_{k+1}\mu$.
\begin{align} \label{fig:arcflow}
   \xy(0,0)*{\includegraphics[width=150px]{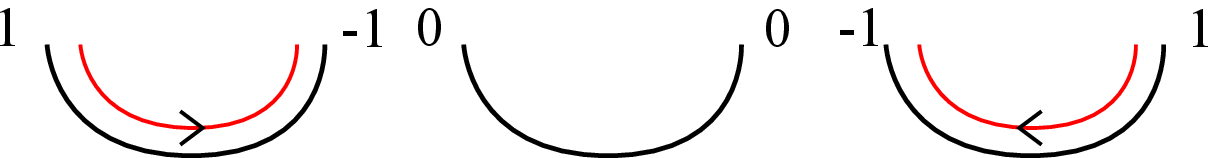}};\endxy
\end{align}
\item In case a Y-rule is applied, $_{k}J$ can be obtained from $_{k+1}J$ 
by replacing the $i$-th entry in $_{k+1}J$ with a length two substring 
whose sum is equal to the $i$-th entry. $_{k}\mu$ can be obtained from 
$_{k+1}\mu$ by replacing the $i$-th entry in $_{k+1}\mu$ with the 
substring $(3-{}_{k+1}\mu_{i}, 3-{}_{k+1}\mu_{i})$.
\begin{align} \label{fig:yflow}
   \xy(0,0)*{\includegraphics[width=300px]{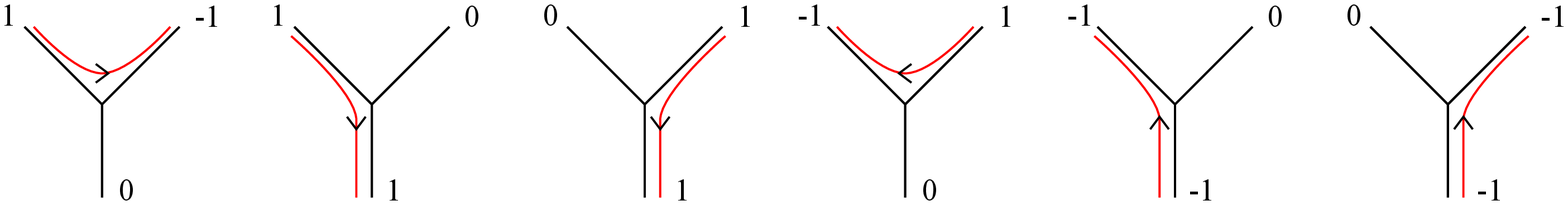}};\endxy
\end{align}
\item In case an H-rule is applied, $_{k}\mu$ can be obtained from 
$_{k+1}\mu$ by replacing a substring $(1, 2)$ or $(2, 1)$, at the $i$-th and 
$(i+1)$-th position in $_{k+1}\mu$, with 
$(3-{}_{k+1}\mu_{i}, 3-{}_{k+1}\mu_{i+1})$. $_{k}J$ can be obtained from 
$_{k+1}J$ by replacing a substring of length two in $_{k+1}J$ at the $i$-th 
and $(i+1)$-th position according to the schema.
\begin{align} \label{fig:hflow}
  &\xy(0,0)*{\includegraphics[width=220px]{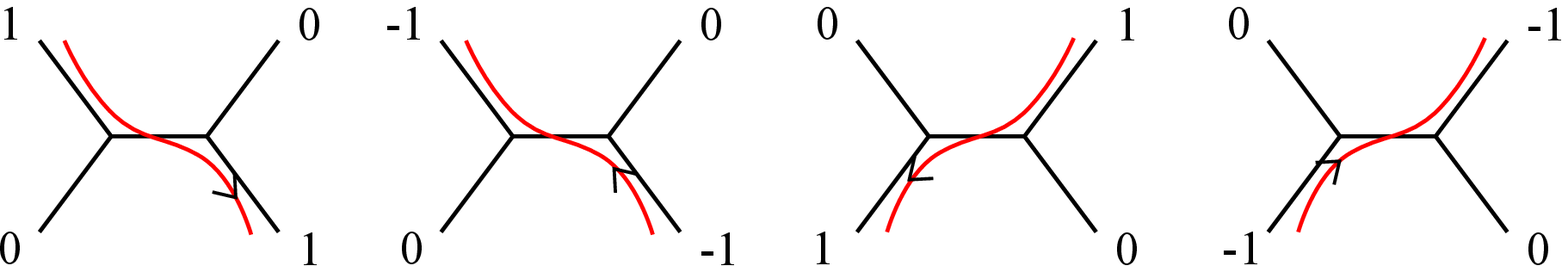}};\endxy\\
   &\xy(0,0)*{\includegraphics[width=220px]{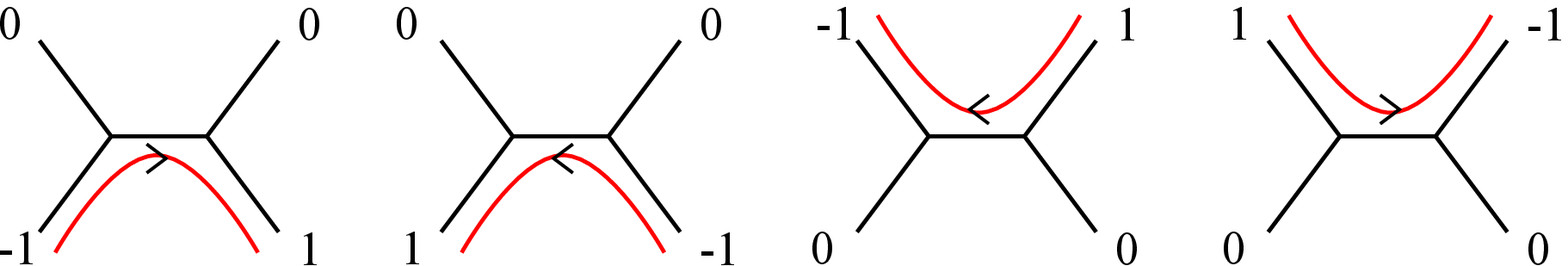}};\endxy
\end{align}
\end{enumerate}
It is straightforward to check that $_{k+1}J$ satisfies 
condition~\eqref{eqn:conds}, with composition $_{k+1}\mu$, if and only if 
$_{k}J$ does, with composition $_{k}\mu$. Take, for example, an instance where 
a Y-rule is applied; suppose also that the $i$-th entry in $_{k+1}\mu$ is 
2 and the $i$-th entry in $_{k+1}J$ is 0. Thus, the $i$-th entry in 
$_{k+1}J$ contributes a pair $(1, -1)$ to $_{k+1}\hat{J}$. 
By~\eqref{fig:yflow},  $_{k}J$ is 
obtained from $_{k+1}J$ by replacing the $i$-th entry in $_{k+1}J$ 
with $(1, -1)$ and the $i$-th entry in $_{k+1}\mu$ with $(1, 1)$. 
We see that $_{k}\hat{J}$ is in fact exactly equal to $_{k+1}\hat{J}$. 
Therefore $_{k+1}\hat{J}$ satisfies condition~\eqref{eqn:conds} if and only if  
$_{k}\hat{J}$ does. Similar analysis apply to all cases 
in~\eqref{fig:arcflow},~\eqref{fig:yflow} and~\eqref{fig:hflow}.

Let $k$ be the first level in the growth algorithm of $w$ where a Y or an 
arc-rule is applied. From the $(k+1)$-th level down we have a 
non-elliptic web $w'$ with flow, whose boundary state string $_{k+1}J$ and 
composition $_{k+1}\mu$ both have length less than $n$. Thus, by our 
induction hypothesis, $_{k+1}J$, with composition $_{k+1}\mu$, satisfies 
condition~\eqref{eqn:conds}. 

By the above argument, then ${}_{i}J$ also satisfy condition~\eqref{eqn:conds}, 
for any $0\leq i\leq k$. In particular, $J={}_0J$ satisfies that condition, 
which is what we had to prove.  

\vskip0.5cm

Conversely, let $J$ satisfy condition~\eqref{eqn:conds}, with composition 
$\mu$. We show, by induction on $n$, that there is a $w\in B_S$ 
with flow whose boundary state string is exactly $J$. More specifically, 
we first construct a $w\in W_S$ and then show that $w$ is non-elliptic, i.e. 
$w\in B_S$. 

For $n=2$, then $w$ must be an arc. 
It is simple to check that if $J$ satisfies condition~\eqref{eqn:conds}, 
$J$ is the boundary state of a flow on an arc. 

For $n>2$, suppose it is possible to apply an arc or Y-rule to 
the pair $\mu$ and $J$, depicted in~\eqref{fig:arcflow} 
and~\eqref{fig:yflow}. Then we obtain a new pair $\mu'$ and 
$J'$ with length less than $n$. Thus, by induction, there exist a web $w'\in 
W_{s'}$ and flow extending $J'$. Gluing the arc or Y on top 
of $w'$ results in a web $w\in W_S$ with a flow extending $J$.

Suppose, then, that it is not possible to apply an arc or Y-rule to 
$\mu$ and $J$. This means that one of the following must hold.
\begin{enumerate}
\item \label{case1} $\mu$ does not contain a substring of type 
$(1,2)$ or $(2,1)$ and $J = (1,...,1)$, $J=(-1,...,-1)$ or $J=(0,...,0)$.
\item \label{case2}$\mu$ contains at least one substring of the form $(1, 2)$ 
or $(2, 1)$. For every substring in $\mu$ of the form $(1, 2)$ or $(2, 1)$, 
the corresponding substring in $J$ is $(\pm1, \pm1)$, $(0, 1)$ or $(1,0)$. 
For every substring in $\mu$ of the form $(1, 1)$ or $(2,2)$, 
the corresponding substring in $J$ is $(1,1)$, $(-1,-1)$ or $(0,0)$.
\end{enumerate}
Case 1 contradicts the assumption that $J$ satisfies 
condition~\eqref{eqn:conds}. 

Case~\ref{case2} contains several subcases, each of which contains details 
which are slightly different. However, the general idea is the same for all of 
them and is very simple, i.e. apply H-moves until you can apply an arc or a Y-rule 
and finish the proof by induction. 

We first suppose, without loss of generality, that 
$\mu$ contains a substring $(\mu_{i}, \mu_{i+1})=(1, 2)$ and that the 
corresponding substring in $J$ is $(j_{i}, j_{i+1})=(1, 1)$ 
(the subcase for $(j_{i}, j_{i+1})=(-1, -1)$ is analogous). 
We see that $(1, 1)$ in $J$ contributes a substring $(1, 0, 1)$ to $\hat{J}$. 
Thus, our assumption that $\hat{J}$ satisfies condition~\eqref{eqn:conds} 
implies that $\hat{J}$ contains at least one more entry equal to $-1$. 
This means that for some $r\neq i, i+1$, $1\leq r\leq n$, one of the 
following is true.
\begin{itemize}
\item[(a)] $j_{r} = -1$, $\mu_{r} = 1$, denoted for brevity by
\begin{align} 
   \xy(0,0)*{\includegraphics[width=70px]{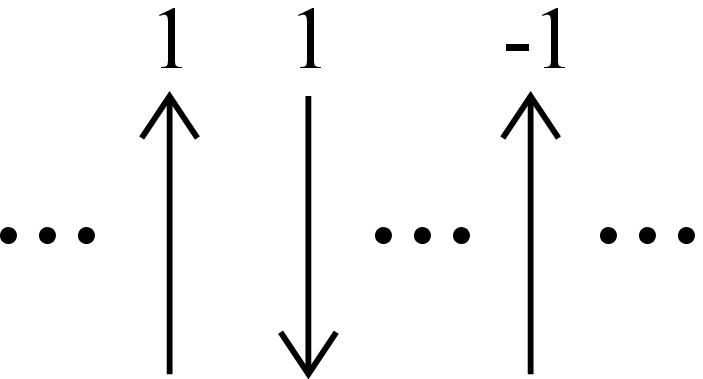}};\endxy
\end{align}
\item[(b)] $j_r = -1$, $\mu_{r} = 2$, denoted
\begin{align} 
   \xy(0,0)*{\includegraphics[width=70px]{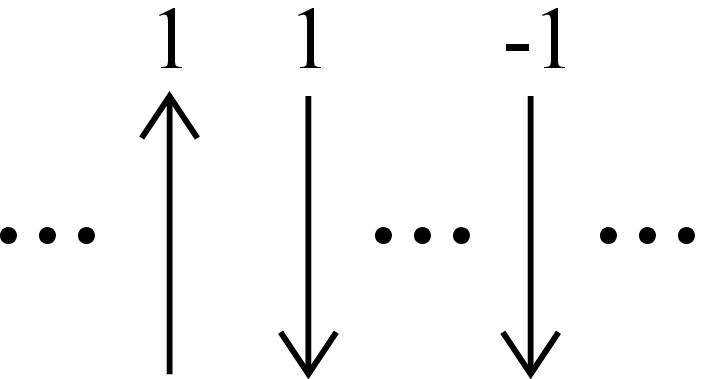}};\endxy
\end{align}
\item[(c)] $j_r = 0$, $\mu_{r} = 2$, denoted
\begin{align} 
   \xy(0,0)*{\includegraphics[width=70px]{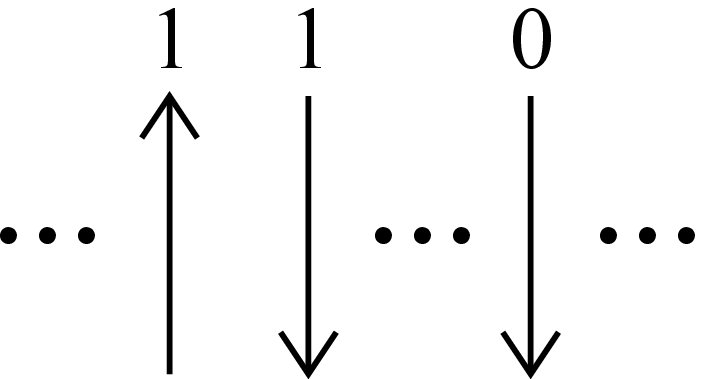}};\endxy
\end{align}
\end{itemize}
Without loss of generality, let us assume $i+1<r$. 
Consider subcases (a) and (b). If $j_{m}\neq 0$ for all $i+1<m<r$, then it 
is possible to apply an arc or Y-move to $J$ and $\mu$, contrary to our 
assumption in case~\ref{case2}. Thus, in all three scenarios above it suffices 
to analyse the following two configurations.
\begin{align} \label{fig:110}
   \xy(0,0)*{\includegraphics[width=70px]{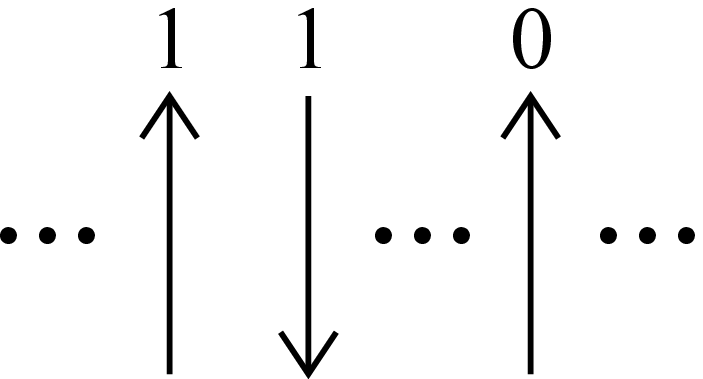}};\endxy && 			\xy(0,0)*{\includegraphics[width=70px]{res/figs/section41/udd110}};\endxy
\end{align}
Let $i+1< r \leq n$ be smallest integer where $j_r = 0$. We must have that 
$\mu_{r-1}=3-\mu_{r}$ and $j_{r-1}=\pm 1$. For any other values of $\mu_{r-1}$ 
and $j_{r-1}$ we would be able to apply an arc or a Y-move, 
contradicting our assumptions for case~\ref{case2}. In both situations, 
we can apply an H-rule to the substrings $(j_{r-1}, j_{r})$ and 
$(\mu_{r-1}, \mu_{r})$ as shown below.
\begin{align} \label{fig:hmove}
   \xy(0,0)*{\includegraphics[width=80px]{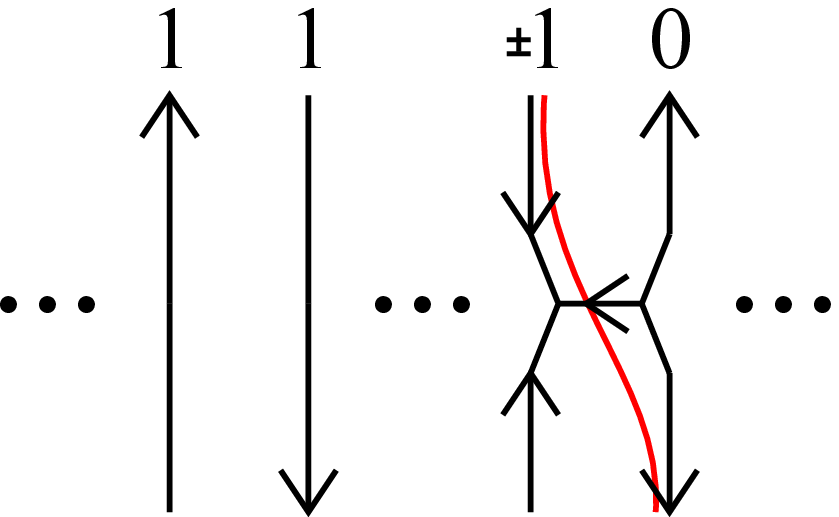}};\endxy &&
   \xy(0,0)*{\includegraphics[width=80px]{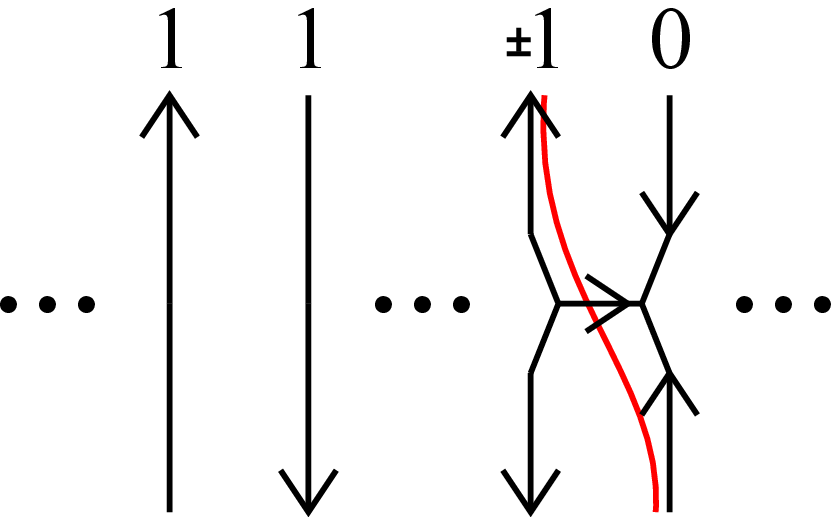}};\endxy
\end{align}
This results in new sign and state strings, each with length $n$ satisfying 
condition~\eqref{eqn:conds}. The application of the H-rule in~\eqref{fig:hmove} 
moves the zero at the $r$-th position to the $r-1$ position. 
Either we can now apply an arc or Y-rule to the new strings or 
by repeatedly applying an H-rule in the manner of~\eqref{fig:hmove}, 
we obtain one of the following pairs.
\begin{align} 
   \xy(0,0)*{\includegraphics[width=45px]{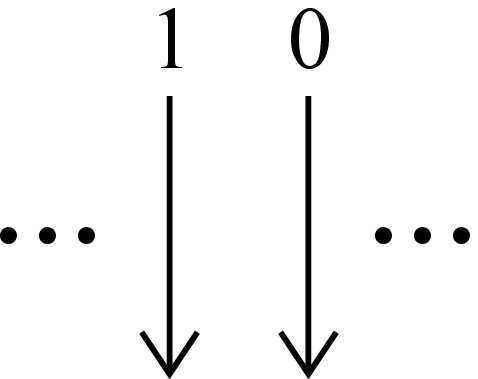}};\endxy &&
\xy(0,0)*{\includegraphics[width=45px]{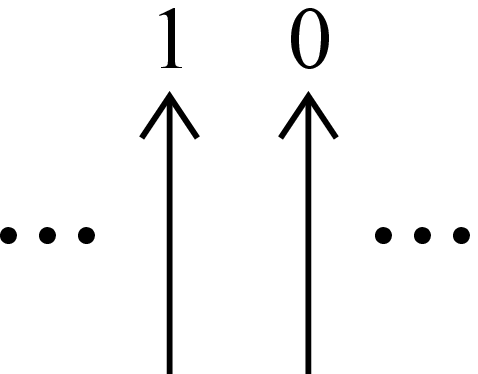}};\endxy
\end{align}
To either of the above diagrams we can apply a Y-rule, after which we 
can use induction. 

To complete our analysis of case~\ref{case2}, now suppose, without loss of 
generality, that $\mu$ contains a substring $(\mu_{i}, \mu_{i+1})=(1, 2)$ and 
that the corresponding substring in $J$ is $(j_{i}, j_{i+1})=(1, 0)$ 
(the subcases for $(0,\pm 1)$ or $(-1,0)$ are analogous).  

We see that $(1, 0)$ in $J$ contributes a substring $(1, 1,-1)$ 
to $\hat{J}$. Thus, our assumption that $\hat{J}$ satisfies 
condition~\eqref{eqn:conds} implies that $\hat{J}$ contains at least one more 
entry equal to $-1$. This means that for some $r$, with $1\leq r\leq n$, 
one of the following is true.
\begin{itemize}
\item[(a)] $j_{r} = -1$, $\mu_{r} = 1$, denoted for brevity by
\begin{align} 
   \xy(0,0)*{\includegraphics[width=70px]{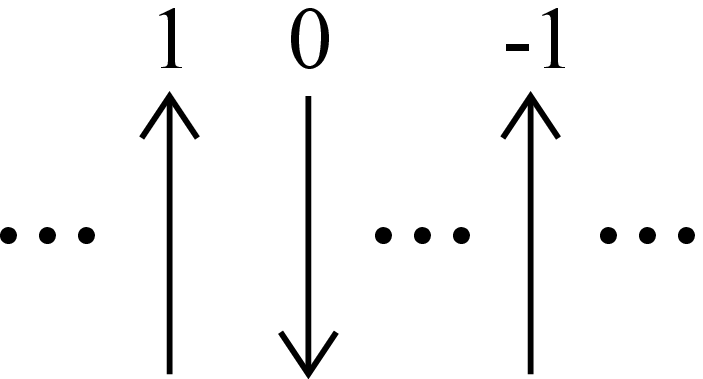}};\endxy
\end{align}
\item[(b)] $j_{r} = -1$, $\mu_{r} = 2$, denoted
\begin{align} 
   \xy(0,0)*{\includegraphics[width=70px]{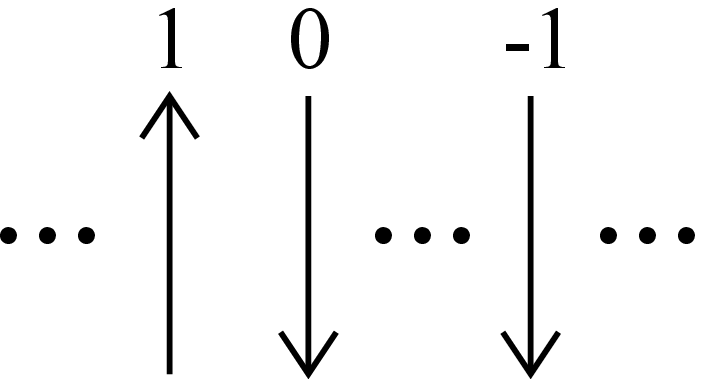}};\endxy
\end{align}
\item[(c)] $j_{r} = 0$, $\mu_{r} = 2$, denoted
\begin{align} 
   \xy(0,0)*{\includegraphics[width=70px]{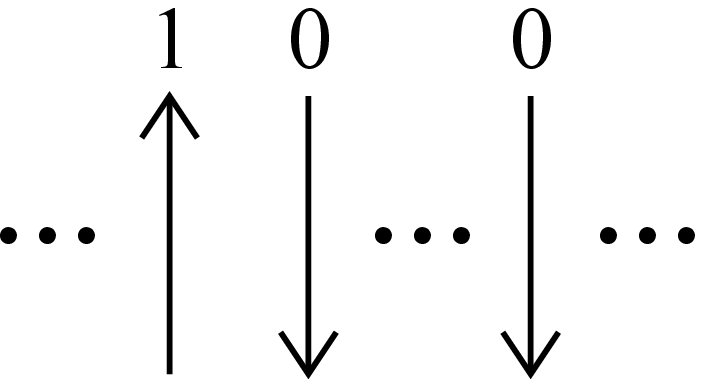}};\endxy
\end{align}
\end{itemize}
For subcases (a) and (b), if $\mu_{i+2}=1$ and $j_{i+2} =-1$, we may apply an 
H-rule to $(\mu_{i+1}, \mu_{i+2})$, $(j_{i+1}, j_{i+2})$ to obtain a new 
pair $\mu'$ and $J'$. Subsequently we can apply a Y-rule 
to the $i$-th and $(i+1)$-th entries of $\mu'$ and $J'$ as illustrated below.
\begin{align}
   \xy(0,0)*{\includegraphics[width=60px]{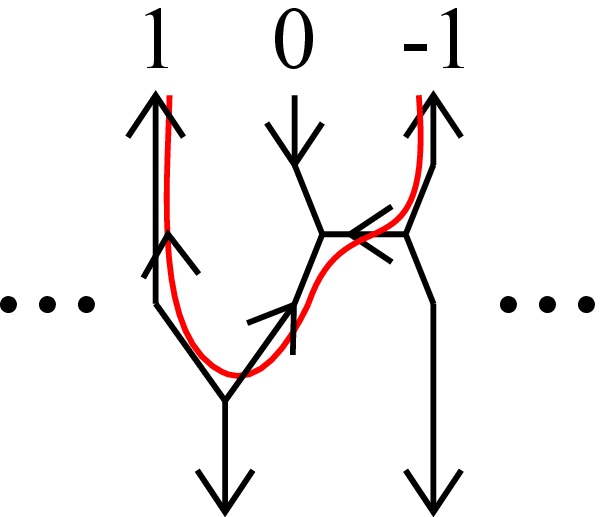}};\endxy 
\end{align}
After applying the Y-rule, we can use induction. 

Otherwise, we can show, just as before, that all three scenarios 
above reduce to an analysis of the following two configurations.
\begin{align} \label{fig:100}
   \xy(0,0)*{\includegraphics[width=70px]{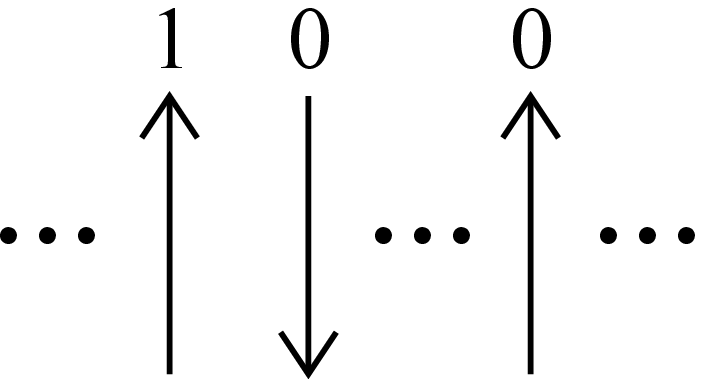}};\endxy && \xy(0,0)*{\includegraphics[width=70px]{res/figs/section41/udd100}};\endxy
\end{align}
That is, we can assume that $\mu$ contains a substring 
$(\mu_{i}, \mu_{i+1})=(1, 2)$ with the corresponding substring in $J$ being 
$(j_{i}, j_{i+1})=(1, 0)$, and for some $0<r\neq i+1<n$ we have $j_{r} = 0$. 
In particular, this tells us that there exist a $0<r\neq i+1<n$ such that 
$\mu_{r}=1$ and $j_{r} = 0$.
\begin{align}
   \xy(0,0)*{\includegraphics[width=70px]{res/figs/section41/udu100}};\endxy
\end{align}
This has to hold because otherwise $\hat{J}$ cannot satisfy 
condition~\eqref{eqn:conds}. Let us assume $r$ to be the smallest integer 
such that $i+1<r$, $\mu_{r}=1$ and $j_{r} = 0$. By our assumption that we 
cannot apply an arc or Y-rule to $J$ and $\mu$, we see that 
$\mu_{r-1} = 2$ and $j_{r-1} = \pm 1$. Applying an H-rule to 
$(j_{r-1}, j_{r})$ and $(\mu_{r-1}, \mu_{r})$ 
\begin{align}
   \xy(0,0)*{\includegraphics[width=80px]{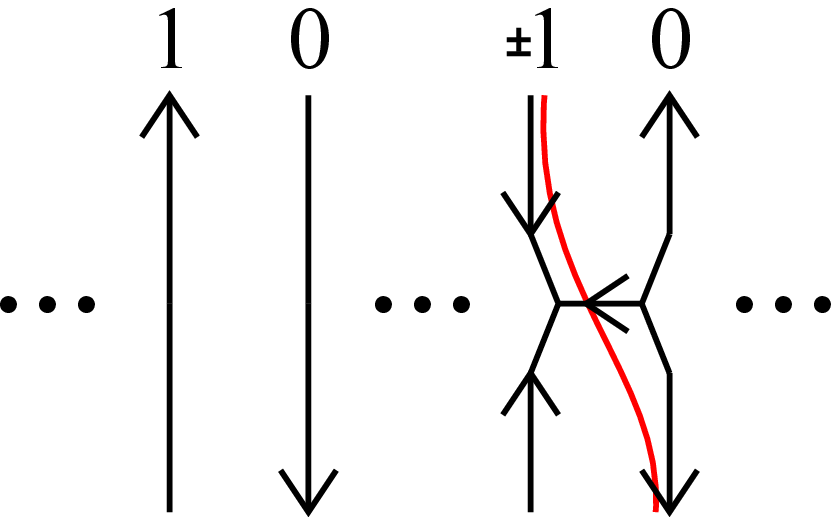}};\endxy
\end{align}
results in new sign and state strings, also with length $n$, satisfying 
condition~\eqref{eqn:conds}. The application of the H-rule in the above 
case moves the zero at the $r$-th position to the $r-1$-th position. 
Either we can now apply an arc or a Y-rule to the new sign and state strings, 
or by repeatedly apply an H-rule in the manner of~\eqref{fig:hmove}, 
we obtain a pair like below.
\begin{align} 
   \xy(0,0)*{\includegraphics[width=45px]{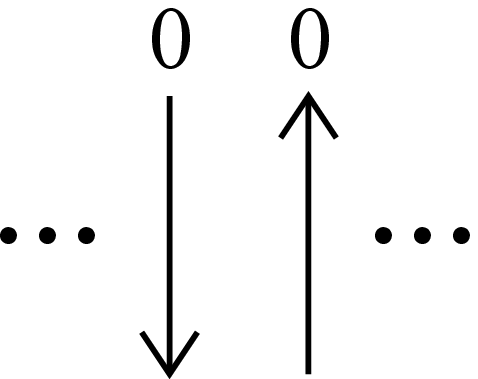}};\endxy
\end{align}
to which we can apply an arc-rule. Finally, apply induction.
\vskip0.5cm
It remains to show that the web $w$ (with flow) produced from the above algorithm is an element of $B_s$, that is, $w$ does not contain digons or squares. 
We note that, just as in~\cite{kk}, in the expression of $w$ using 
the arc, Y and H-rules, digons can only appear as the result of applying an 
arc-rule to the bottom of an H-rule, i.e. we have
\begin{align} 
   \xy(0,0)*{\includegraphics[width=25px]{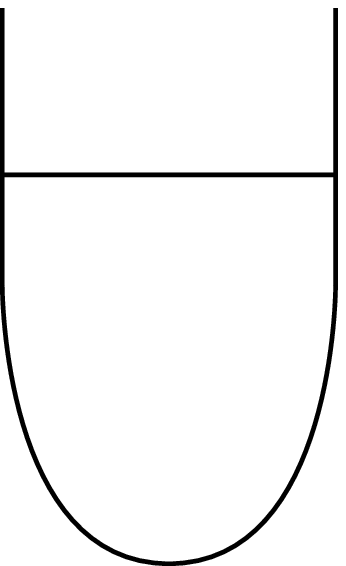}};\endxy
\end{align}

A square can only result from the following sequence of arc, Y and H-rules.
\begin{align} 
   \xy(0,0)*{\includegraphics[width=250px]{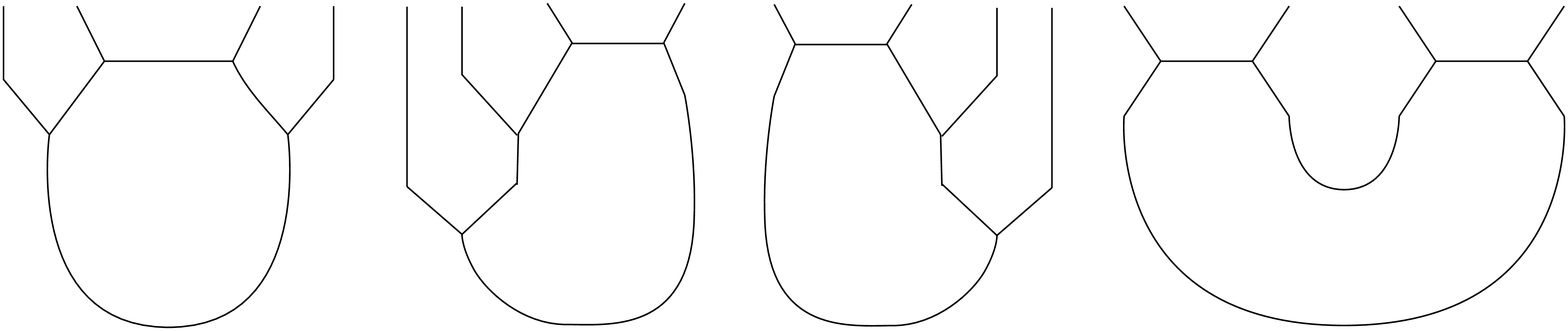}};\endxy
\end{align}
Note that in the above, we do not consider the case in which we apply 
an H-rule to the bottom of another H-rule. This is because such a case cannot arise in 
our construction of $w$.

Recall that in our inductive construction of $w$, we only apply H-rules 
equipped with the following flows.
\begin{align} 
   \xy(0,0)*{\includegraphics[width=90px]{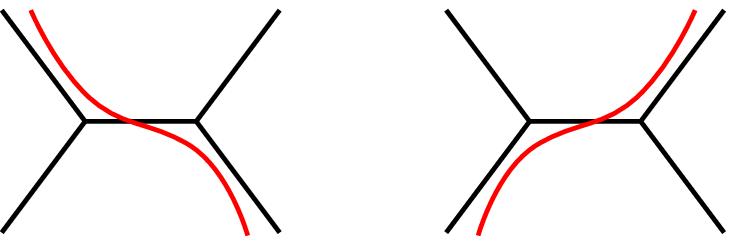}};\endxy
\end{align}

We can immediately see that is it not possible to apply an arc-rule with flow 
to the bottom of such an H-rule as shown below.
\begin{align} 
   \xy(0,0)*{\includegraphics[width=30px]{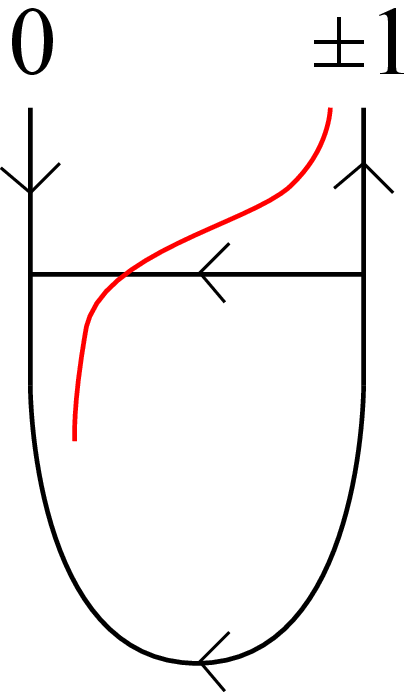}};\endxy
\end{align}

Since we only use the above two H-rules with flow, the induced flows on 
squares are as follows.
\begin{align} 
   \xy(0,0)*{\includegraphics[width=130px]{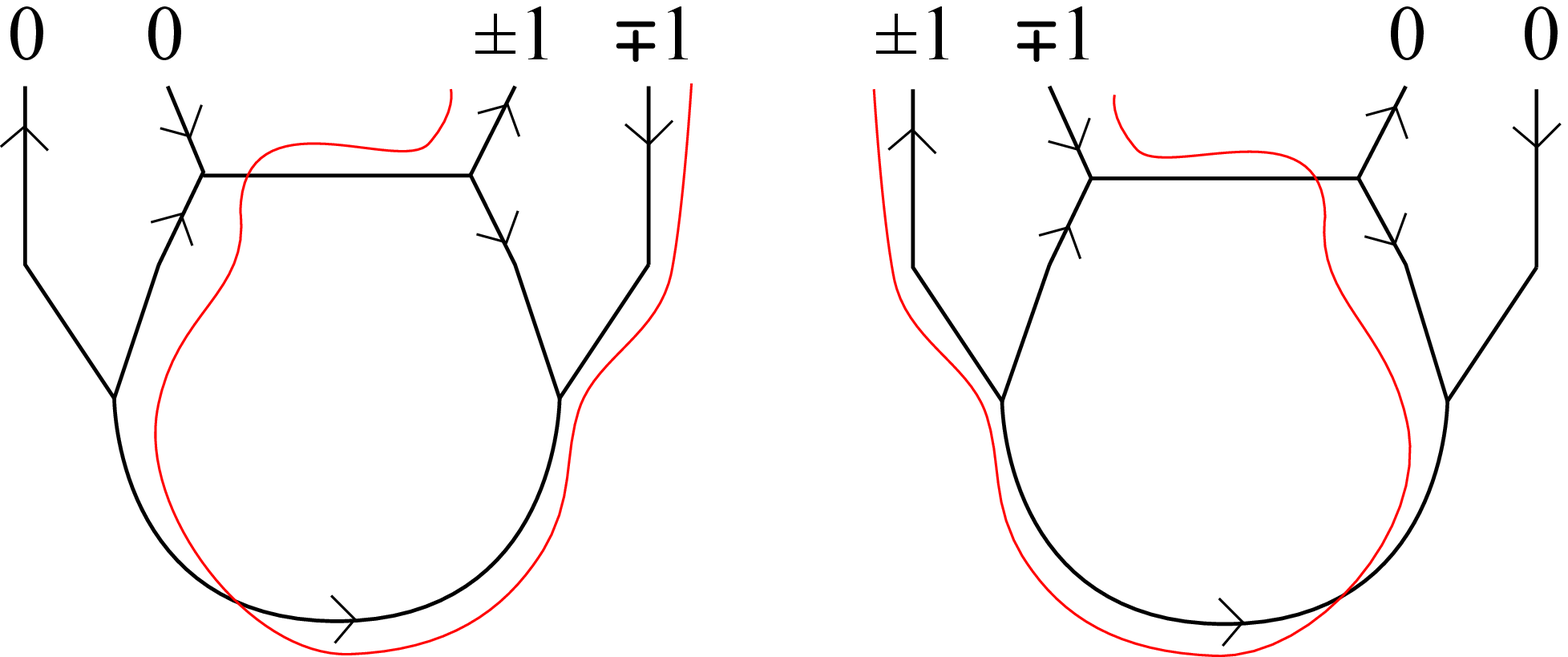}};\endxy
\end{align}

\begin{align} 
   \xy(0,0)*{\includegraphics[width=130px]{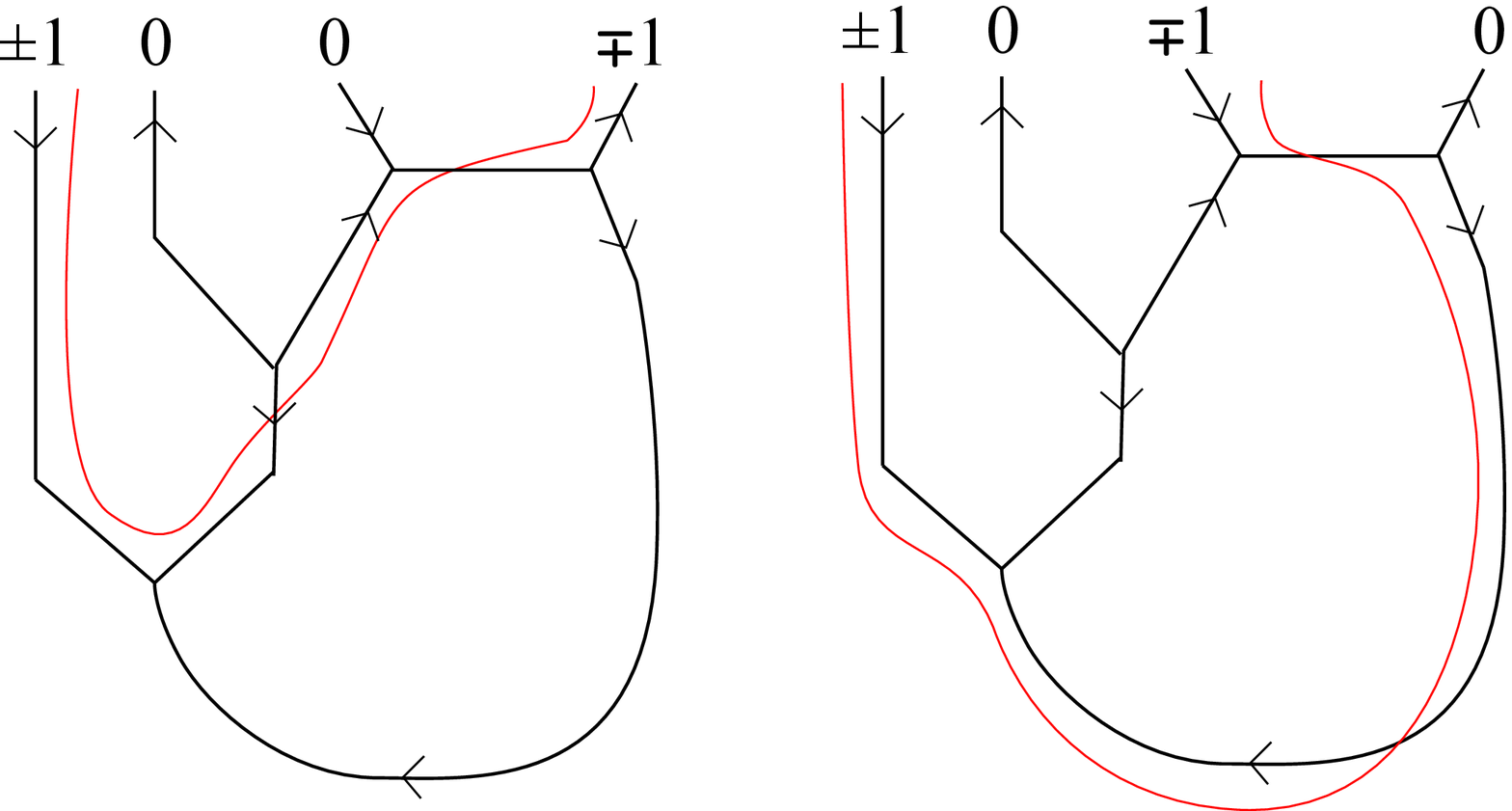}};\endxy
\end{align}

\begin{align} 
   \xy(0,0)*{\includegraphics[width=180px]{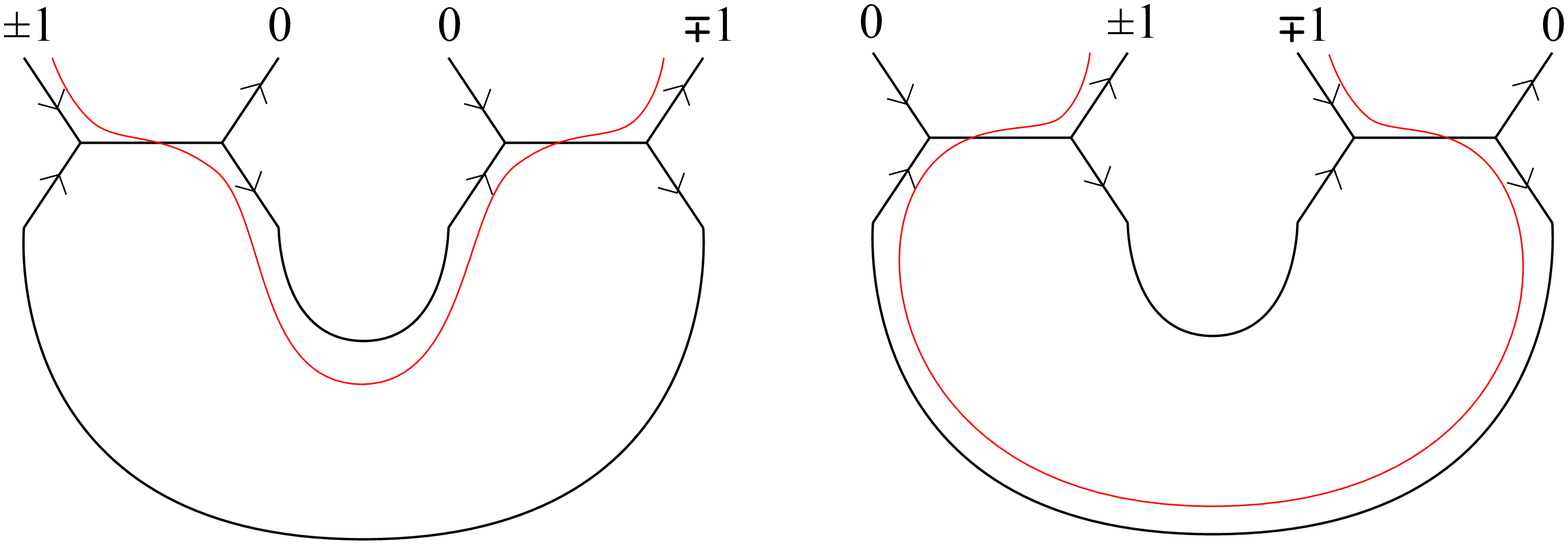}};\endxy
\end{align}
In each case, one can check that it is possible to apply an arc-rule to the 
state and sign strings 
(the same analysis applies to the cases where the faces above are given 
the opposite edge orientations). 
However, recall that an H-rule is used in our construction 
only in the case for which it is not possible to apply any other rules 
to the boundary. This implies that none of the above faces 
can appear during the construction of $w$.
\end{proof}

Implicit in the proof of Lemma~\ref{lem:strtoflow} is a procedure to construct,
from a state string $J$ satisfying condition~\eqref{eqn:conds},  
a non-elliptic web $w$ with flow extending $J$, such that $\partial w=\mu$. Note
that this procedure is not deterministic. That is, it is possible to produce 
different webs with flows extending $J$ by making different choices in the 
construction. 

\begin{ex}
\label{ex:tableauxflows}
The procedure is exemplified below. If we choose to replace the substrings as indicated in the right figure, the tableau on the left gives rise to the web with flow next to it.  
\begin{align} 
   \xy(0,0)*{\includegraphics[width=80px]{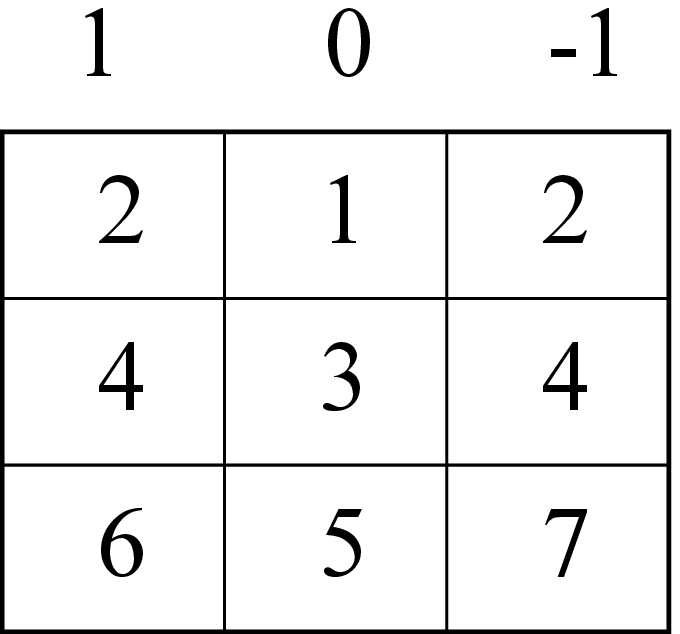}};\endxy && \xy(0,0)*{\includegraphics[width=160px]{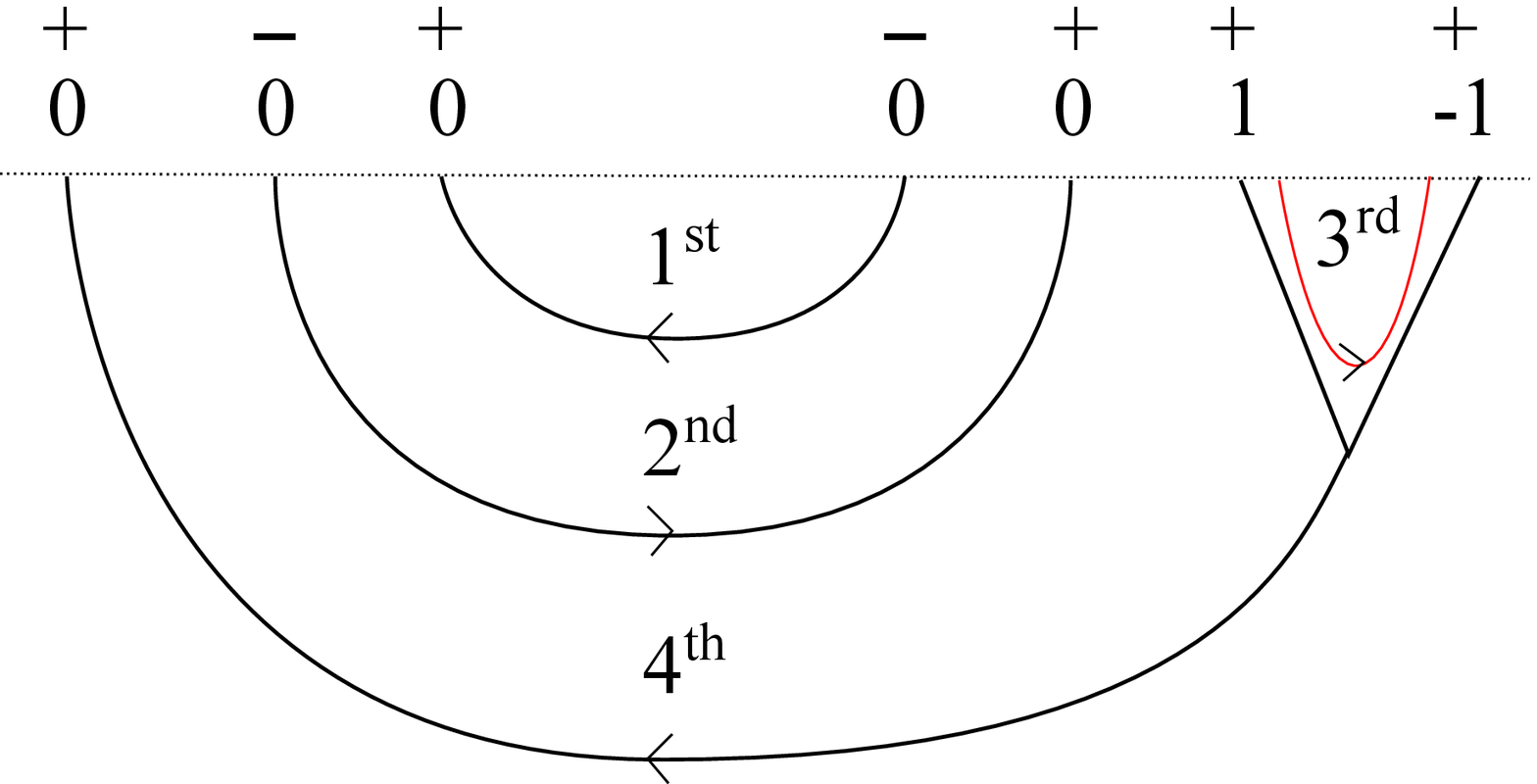}};\endxy
\end{align}
However, for other choices the same tableau generates the following web with flow.
\begin{align} 
   \xy(0,0)*{\includegraphics[width=160px]{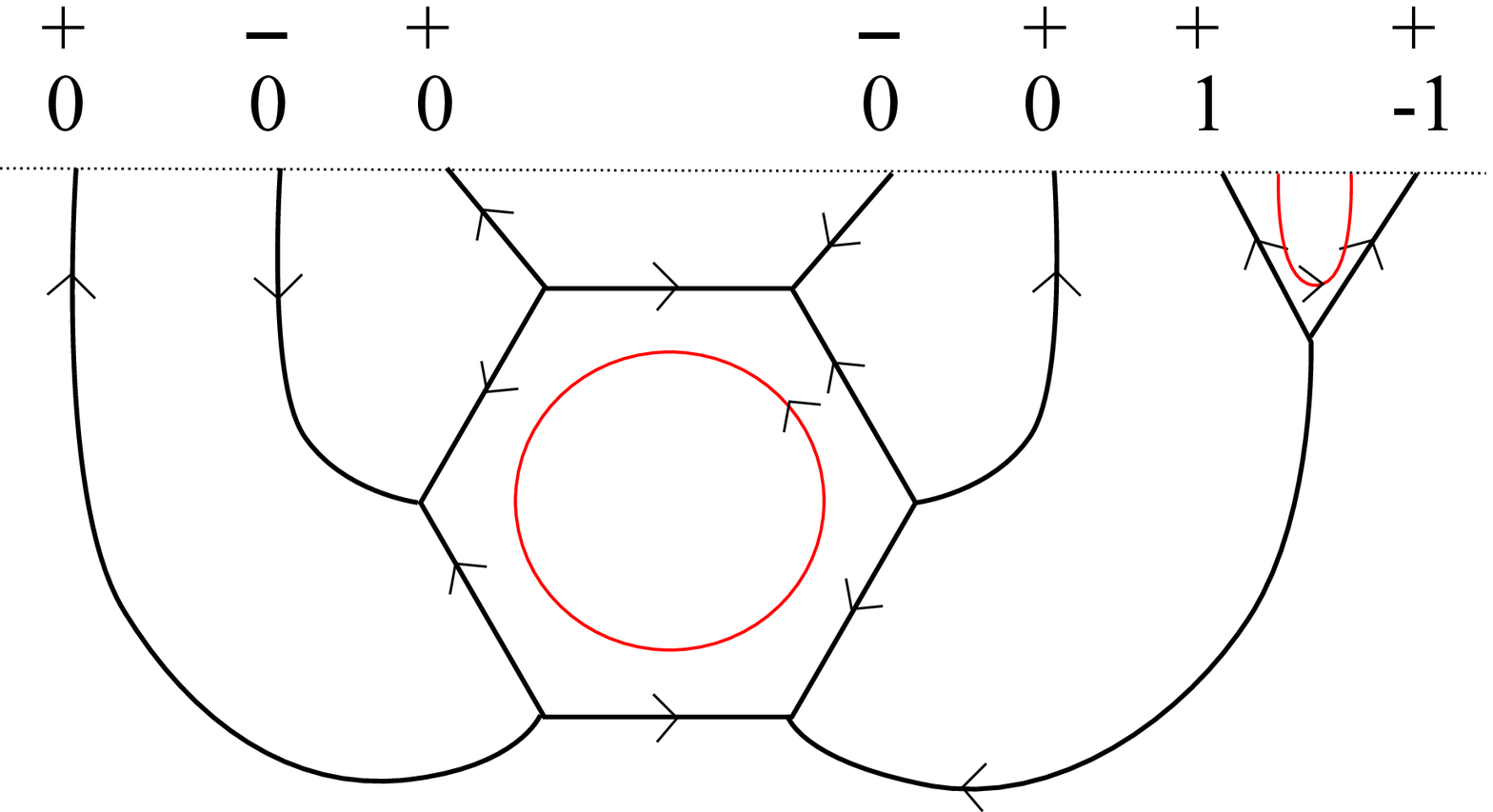}};\endxy
\end{align}
As a matter of fact, we could also invert the orientation of the flow in the internal cycle. The resulting web with flow 
would still correspond to the same tableau. 
\end{ex}

However, when we restrict to semi-standard tableaux, the 
procedure gives a unique web with flow, the canonical flow. 
One can check that the procedure implicit in Lemma~\ref{lem:strtoflow}, 
restricts to the same bijection between $\mathrm{Std}_{\mu}^{\lambda}$ and 
non-elliptic webs as defined by Russell in~\cite{ru}.
\subsection{$Z(G_S)$ and $E(Z(G_S))$}\label{sec-webcenterb}
In this section, $S$ continues to be a fixed sign string of length $n$. 
Moreover, we continue to use some of the other notations and conventions from 
the previous sections as well, e.g. $d=3k\geq n$ etc.  
Let $\mu$ be the composition associated to $S$ and 
let $S_{\mu}$ be the corresponding parabolic subgroup of the 
symmetric group $S_d$. 

Let $Z(K_S)$ be the center of $K_S$ and let $X^{\lambda}_{\mu}$ be 
the Spaltenstein variety, with the notation as in~\cite{bros}. If $n_s=0$, 
then $X^{\lambda}_{\mu}=X^{\lambda}$, the latter being the Springer fiber 
associated to $\lambda$.\footnote{When comparing to Khovanov's result 
for $\mathfrak{sl}_2$, the reader should be aware that he labels the 
Springer fiber by $\lambda^T$, the transpose of $\lambda$.} 

In Theorem~\ref{thm:center}, we are going to prove that $H^*(X^{\lambda}_{\mu})$ 
and $Z(K_S)$ are isomorphic as graded algebras. 
\vskip0.5cm
Recall the following result by Tanisaki~\cite{ta}. Let 
$P=\mathbb{C}[x_1,\ldots,x_d]$ and let $I^{\lambda}$ be the ideal generated by
\begin{equation}
\label{eqn:SpringerI}
\left\{
e_r(i_1,\ldots,i_m)\;\middle\vert\;
\begin{array}{c}
m\geq 1, 1\leq i_1< \cdots < i_m\leq d\\
r> m-\lambda_{d-m+1}-\cdots-\lambda_n 
\end{array}
\right\}, 
\end{equation}
where $e_r(i_1,\ldots,i_m)\in P$ is the $r$-th elementary symmetric 
polynomial. Write 
\[
R^{\lambda}=P/I^{\lambda}.
\]
Tanisaki showed that 
\[
H^*(X^{\lambda})\cong R^{\lambda}.
\]

Note that $S_{\mu}$ acts on $P$ by permuting the variables, but it maps 
$I^{\lambda}$ to itself. Let $P^{\mu}=P^{S_{\mu}}\subset P$ be the subring of 
polynomials which are invariant under $S_{\mu}$. For 
$1\leq i_1\leq \cdots\leq i_m\leq n$ and $r\geq 1$, we let 
$e_r(\mu, i_1,\ldots,i_m)$ denote the $r$-th elementary 
symmetric polynomials in the variables $X_{i_1}\cup\cdots\cup X_{i_m}$, where 
\[
X_{p}=
\left\{x_k\mid \mu_1+\cdots+\mu_{p-1}+1\leq k\leq \mu_1+\cdots+\mu_p\right\}.
\]
So, we have 
\[
e_r(\mu, i_1,\ldots,i_m)=\sum_{r_1+\cdots+r_m=r}e_{r_1}(\mu;i_1)\cdots 
e_{r_m}(\mu;i_m).
\]
If $r=0$, we set $e_r(\mu, i_1,\ldots,i_m)=1$ and if $r<0$, 
we set $e_r(\mu, i_1,\ldots,i_m)=0$. 
Let $I^{\lambda}_{\mu}$ be the ideal generated by 
\begin{equation}
\label{eqn:SpaltensteinI}
\left\{
e_r(\mu, i_1,\ldots,i_m)\;\middle\vert\;
\begin{array}{l}
m\geq 1, 1\leq i_1< \cdots < i_m\leq d\\
r> m-\mu_{i_1}+\cdots+\mu_{i_m}-\lambda_{l+1}-\cdots-\lambda_n \\
\mathrm{where}\;l=\#\{i\mid \mu_i>0, i\ne i_1,\ldots, i_m\}
\end{array}
\right\}.
\end{equation}
Note that $I^{\lambda}_{\mu}\subseteq I^{\lambda}$ holds. Write 
\[
R^{\lambda}_{\mu}=P^{\mu}/I^{\lambda}_{\mu}.
\]
Brundan and Ostrik~\cite{bros} proved that 
\[
H^*(X^{\lambda}_{\mu})\cong R^{\lambda}_{\mu}.
\]

First we want to show that $R^{\lambda}_{\mu}$ acts on $K_S$. Clearly, 
$P^{\mu}$ acts on $K_S$, by converting polynomials into dots on the facets 
meeting $S$. 

\begin{lem}
\label{lem:I}
The ideal $I^{\lambda}_{\mu}$ annihilates any foam in $K_S$. 
\end{lem}
\begin{proof}
The following argument demonstrates that it suffices to show this for the 
case when $n_s=0$.
Let $u,v\in B_S$. For each $1\leq i\leq n$ with $s_i=-$, glue a $Y$ onto 
the $i$-th boundary edge of $u$ and $v$, respectively. Call these 
new webs $\hat{u}$ and $\hat{v}$, respectively. 
Note that $\partial\hat{u}=\partial\hat{v}=\hat{S}$, 
where $\hat{S}=(+^d)$. 
Let $f\in {}_uK_v$ be any foam. For each $1\leq i\leq n$ with $s_i=-$, 
glue a digon foam on top of the 
$i$-th facet of $f$ meeting $S$. The new foam $\hat{f}$, 
obtained in this way, belongs to ${}_{\hat{u}}K_{\hat{v}}$. Note that we can 
reobtain $f$ by capping off $\hat{f}$ with dotted digon foams. Any 
polynomial $p\in I^{\lambda}_{\mu}\subseteq I^{\lambda}$ acting on $f$ 
also acts on $\hat{f}$. So, if we know that $p\hat{f}=0$, then it follows that 
$pf=0$.

Thus, without loss of generality, assume that $n_s=0$. 
We are now going to show that $I^{\lambda}$ annihilates $K_S$.  

As follows from Definition in~\eqref{eqn:SpringerI}, 
$I^{\lambda}$ is generated by the elementary symmetric polynomials 
$e_r(x_{i_1},\ldots,x_{i_m})$, for the following values of $m$ and $r$.
\[
\begin{array}{lll}
m=2n+1 &;&r>2n-2,\\
m=2n+2 &;&  r>2n-4,\\
\vdots &\vdots& \vdots\\
m=3n-1 &;& r>2,\\
m=3n & ;& r>0.
\end{array}
\]
Note that for $m=3n$, we simply get all completely symmetric polynomials of 
positive degree in the variables $x_1,\ldots,x_d$. Any such polynomial $p$ 
annihilates any foam $f\in {}_uK_v$, because 
by the complete symmetry of $p$, the dots can all be moved to the 
three facets around one singular edge.
The relations~\eqref{eq:dotm} then show that $p$ kills $f$.

Now suppose $m=3n-\ell$, for $\ell>0$. So we must have $r>2\ell$. The 
argument we are going give does 
not depend on the particular choice of 
$i_1,\ldots,i_m\subseteq \{1,2,\ldots,d\}$, so, without loss of 
generality, let us assume that $(i_1,\ldots,i_m)=(1,\ldots,m)$. 

Let $f$ be any foam in ${}_uK_v$. 

First assume that $\ell=1$, then we have 
\begin{align*}
&e_r(x_1,\ldots,x_{d-1})f\\
=&-e_{r-1}(x_1,\ldots,x_{d-1})x_df\\
=&e_{r-2}(x_1,\ldots,x_{d-1})x_d^2f\\
\vdots&\\
=&(-1)^r x_d^rf.
\end{align*}
All these equations follow from the fact that, for any $j>0$, we have  
\[
e_j(x_1,\ldots,x_{d})=e_j(x_1,\ldots,x_{d-1})+e_{j-1}(x_1,\ldots,x_{d-1})x_d,
\]
and the fact that $e_j(x_1,\ldots,x_{d})f=0$, as we proved above in the 
previous case for $m=3n$.  
Since in this case we have $r>2$, we see that 
\[
(-1)^r x_d^rf=0,
\]
by Relation (3D). This finishes the proof for this case. 

In general, for $\ell\geq 1$, we get that 
$e_r(x_1,\ldots,x_{d-\ell})f$ is equal to a linear combination of 
terms of the form 
\[
x_{d-\ell+1}^{r_1}x_{d-\ell+2}^{r_2}\cdots x_{d}^{r_{\ell}}f,
\]
with $r_1+\cdots+r_{\ell}=r$. Since $r>2\ell$, there exists 
a $1\leq j\leq \ell$ such that $r_j>2$, in each term. 
So each term kills $f$, by Relation (3D). This 
finishes the proof.
\end{proof}

Note that Lemma~\ref{lem:I} shows that there is a well-defined homomorphism of 
graded algebras 
$c_S\colon R^{\lambda}_{\mu}\to Z(K_S)$, defined by 
\[
c_S(p)=p1.
\] 

Similarly, there is a filtration preserving homomorphism 
\[
P^{\mu}\to Z(G_S)
\]
defined by 
$p\mapsto p1$. This homomorphism does not descend to 
$R^{\lambda}_{\mu}$, because the relations in $G_S$ are deformations 
of those in $K_S$, but the associated graded homomorphism 
maps $E(P^{\mu})$ to $E(Z(G_S))$ and we have 
\[
E(P^{\mu}1)=R^{\lambda}_{\mu}1.
\]
Before giving our following result, we recall that Brundan and Ostrik~\cite{bros} 
showed that 
\[
\dim H^*(X^{\lambda}_{\mu})=\#\mathrm{Col}^{\lambda}_{\mu}.
\]
They actually gave a concrete basis, but we do not need it here. 

\begin{lem}
\label{lem:dimZG}
We have 
\[
\dim Z(G_S)=\#\mathrm{Col}^{\lambda}_{\mu}.
\] 
\end{lem}
\begin{proof}
Let $J$ be any state-string satisfying condition~\eqref{eqn:conds}. 
We define 
\begin{equation}
\label{eq:centralidempotent}
z_J=\sum_{u\in B_S} \sum_{T} e_{u,T}\in G_S,
\end{equation}
where the second sum is over all 3-colourings of $u$ extending $J$.

First we show that $z_J\in Z(G_S)$. For any $u,v\in B_S$, 
let $f\in {}_uG_v$. Choose two arbitrary 
compatible colourings $T_1$ and $T_2$ of $u$ and $v$, respectively. 
Assume that $e_{T_1}fe_{T_2}\ne 0$. Then we have 
\[
z_Je_{T_1}fe_{T_2}=
\begin{cases}
e_{T_1}fe_{T_2},&\quad\text{if}\quad T_1\;\text{extends}\;J,\\
0,&\quad\text{else}. 
\end{cases}
\]  
We also have 
\[
e_{T_1}fe_{T_2}z_J=
\begin{cases}
e_{T_1}fe_{T_2},&\quad\text{if}\quad T_2\;\text{extends}\;J,\\
0,&\quad\text{else}. 
\end{cases}
\]  
This shows that $z_J\in Z(G_S)$, because $T_1$ and $T_2$ are compatible, 
and so $T_1$ extends $J$ if and only if $T_2$ extends $J$.

Note that 
\[
\sum_J z_J=1\quad\text{and}\quad z_Jz_{J'}=\delta_{J,J'}z_J.
\]
In particular, the $z_J$'s are linearly independent. 
\vskip0.5cm
For any state-string $J$ satisfying condition~\eqref{eqn:conds}, 
the central idempotent $z_J$ belongs to $P^{\mu}1$. In order to see this, 
first note that, for 
any $u\in B_S$, the element
\[
z_J1_u=\sum_{T\;\mathrm{extends}\;J} e_{u,T},
\]
belongs to $P^{\mu}1_u$. This holds, because only the colours of the boundary 
edges of $u$ are fixed. We can sum over all possible 3-colourings of the other 
edges, which implies that these edges only contribute a factor $1$ to 
$z_J1_u$. Furthermore, we see that $z_J1_u=p_J1_u$, for a fixed polynomial 
$p_J\in P^{\mu}$, i.e. $p_J$ is independent of $u$. Therefore, we have 
\[
z_J=\sum_{u\in B_S} p_J1_u=p_J1\in P^{\mu}1.
\]
It remains to show that $Z(G_S)z_J=\mathbb{C}z_J$. 
Let $z\in Z(G_S)$. By the orthogonality of Gornik's symmetric idempotents, 
we have 
\[
z=\sum_{u,T}e_{u,T}ze_{u,T}.
\] 
By Theorem~\ref{thm:Gornik}, we know that 
\[
e_{u,T}ze_{u,T}=\lambda_{u,T}(z) e_{u,T},
\]
for a certain $\lambda_{u,T}(z)\in\mathbb{C}$. Therefore, we have 
\[
z=\sum_{u,T}\lambda_{u,T}(z)e_{u,T}\in \bigoplus_{u,T}\mathbb{C}e_{u,T}.
\]
By Lemma~\ref{lem:strtoflow}, we know that $z_J\ne 0$. This shows that 
\[
\{z_J\mid J\;\text{satisfying condition~\eqref{eqn:conds}}\}
\] 
forms a basis of $Z(G_S)$. By Proposition~\ref{prop:tableauxflows}, 
the claim of the lemma follows. 
\end{proof}

\begin{thm}
\label{thm:center}
The degree preserving algebra homomorphism
$$c_S\colon R^{\lambda}_{\mu_S}\to Z(K_S)$$
is an isomorphism.
\end{thm}
\begin{proof}
In Corollary~\ref{cor:moritacenter} it will be shown that 
$$\dim H^*(X^{\lambda}_{\mu_S})=\dim Z(K_S),$$ 
so it suffices to show that $c_S$ is injective.   

Lemma~\ref{lem:I} shows that (as graded complex algebras)
\[
R^{\lambda}_{\mu}1\subset Z(K_S).
\] 
As already mentioned above, Brundan and Ostrik~\cite{bros} showed that 
\[
H^*(X^{\lambda}_{\mu})\cong R^{\lambda}_{\mu}
\] 
as graded complex algebras. 

The proof of Lemma~\ref{lem:dimZG} shows that the filtration preserving 
homomorphism 
\[
P^{\mu}\to Z(G_S),
\]
defined by 
$p\mapsto p1$, is surjective. Note the $E(\cdot)$ is not a map. However, 
a filtered algebra $A$ and its associated graded $E(A)$ 
are isomorphic as vector spaces. In particular, they satisfy 
\[
\dim A =\dim E(A). 
\]
Therefore, since $p\mapsto p1$ is a surjection of vector spaces, we have 
\[
\dim Z(G_S)=\dim E(Z(G_S))=\dim E(P^{\mu}1)=\dim P^{\mu}1.
\]
Recall that $E(P^{\mu}1)=R^{\lambda}_{\mu}1$ and 
$\dim Z(G_S)=\dim R^{\lambda}_{\mu}$. This shows 
\[
\dim R^{\lambda}_{\mu}1=\dim P^{\mu}1=\dim Z(G_S)=\dim R^{\lambda}_{\mu},
\]
which implies that the map $c_S$ is injective. 
\end{proof}
\subsection{Web algebras and the cyclotomic KLR algebras: Howe duality}\label{sec-webhowe}
Our main references for Howe duality are~\cite{ho2} and~\cite{ho1}, 
where the reader can find the proofs of the results, which we recall below, 
and other details.

Let us briefly explain Howe duality.\footnote{We follow Kamnitzer's exposition in ``The ubiquity of Howe duality'', which is online available at https://sbseminar.wordpress.com/2007/08/10/the-ubiquity-of-howe-duality/.} The two natural actions of $\mathrm{GL}_m=\mathrm{GL}(m,\mathbb{C})$ and of $\mathrm{GL}_n=\mathrm{GL}(n,\mathbb{C})$ on $\mathbb{C}^m\otimes \mathbb{C}^n$ commute and the two groups are each others commutant. We say that the actions of $\mathrm{GL}_m$ and $\mathrm{GL}_n$ are \textit{Howe dual}. 

More interestingly, their actions on the symmetric powers
\[
S^p\left(\mathbb{C}^m\otimes \mathbb{C}^n\right)
\]
and on the alternating powers
\[
\Lambda^p\left(\mathbb{C}^m\otimes \mathbb{C}^n\right)
\]
are also Howe dual, for any $p\in\mathbb{N}$. These are called the 
\textit{symmetric} and the \textit{skew} Howe duality of $\mathrm{GL}_m$ and 
$\mathrm{GL}_n$, respectively. In this thesis, we are only considering 
the skew Howe duality. 

The skew Howe duality implies that 
we have the following decomposition into irreducible 
$\mathrm{GL}_m\times \mathrm{GL}_n$-modules.
\begin{equation}
\label{eq:decomp}
\Lambda^p\left(\mathbb{C}^m\otimes \mathbb{C}^n\right)\cong 
\bigoplus_{\lambda} V_{\lambda}\otimes W_{\lambda'},
\end{equation}
where $\lambda$ ranges over all partitions with $p$ boxes and at most 
$m$ rows and $n$ columns and $\lambda'$ is the transpose of $\lambda$.

Here $V_{\lambda}$ is the unique irreducible $\mathrm{GL}_m$-module of 
highest weight $\lambda$ and $W_{\lambda'}$ is the unique irreducible 
$\mathrm{GL}_n$-module of highest weight $\lambda'$.

Without giving a full 
proof of~\eqref{eq:decomp}, which can be found in Section 4.1 of~\cite{ho2}, 
we note that it is easy to write down 
the highest weight vectors in the decomposition of 
\[
\Lambda^p\left(\mathbb{C}^m\otimes \mathbb{C}^n\right).
\]
Define 
\[
\epsilon_{ij}=\epsilon_i\otimes \epsilon_j,
\]
for any $1\leq i\leq m$ and $1\leq j\leq n$. Here the $\epsilon_i$ and the 
$\epsilon_j$ are the canonical basis elements of $\mathbb{C}^m$ and $\mathbb{C}^n$ 
respectively. Let $\lambda$ be one of the highest $\mathrm{GL}_m$ weights 
in~\eqref{eq:decomp}. Write $\lambda=(\lambda_1,\ldots,\lambda_m)$ with 
$n\geq \lambda_1\geq 
\lambda_2\geq \cdots\geq \lambda_m\geq 0$. Then  
\begin{eqnarray*}
v_{\lambda,\lambda'}&=&\left(\epsilon_{11}\wedge \cdots\wedge \epsilon_{1\lambda_1}\right)\wedge
  \left(\epsilon_{21}\wedge \cdots\wedge \epsilon_{2\lambda_2}\right)\wedge
  \left(\epsilon_{m1}\wedge \cdots\wedge \epsilon_{m\lambda_m}\right)\\
&=&\pm\left(\epsilon_{11}\wedge \cdots\wedge \epsilon_{\lambda'_1 1}\right)\wedge
  \left(\epsilon_{12}\wedge \cdots\wedge \epsilon_{\lambda'_2 2}\right)\wedge
  \left(\epsilon_{1n}\wedge \cdots\wedge \epsilon_{\lambda'_n n}\right)
\end{eqnarray*}
is a highest $\mathrm{GL}_m\times\mathrm{GL}_n$ weight. 
By convention, we exclude factors $\epsilon_{ij}$ for which $\lambda_i=0$ or 
$\lambda_j'=0$.

Now restrict to $\mathrm{SL}_m$ and assume that $p=mk$, 
for some $k\in\mathbb{N}$. 
By Schur's lemma, the decomposition in~\eqref{eq:decomp} implies that 
\begin{equation}
\label{eq:howehom1}
\mathrm{Inv}_{\mathrm{GL}_m}\left(\Lambda^p\left(\mathbb{C}^m\otimes 
\mathbb{C}^n\right)\right)\cong 
\mathrm{Hom}_{\mathrm{SL}_m}\left(\mathbb{C},\Lambda^p\left(\mathbb{C}^m\otimes 
\mathbb{C}^n\right)\right)\cong W_{(k^m)},
\end{equation}  
where $\mathbb{C}$ denotes the trivial representation. 

Decompose 
\[
\mathbb{C}^n\cong \mathbb{C}\epsilon_1\oplus \mathbb{C}\epsilon_2\oplus 
\cdots\oplus \mathbb{C}\epsilon_n
\]
into its one-dimensional $\mathfrak{gl}_n$-weight spaces. Then we have 
\begin{equation}
\label{eq:decomp2}
\Lambda^p\left(\mathbb{C}^m\otimes \mathbb{C}^n\right)\cong 
\bigoplus_{(p_1,\ldots,p_n)\in \Lambda(n,p)}\Lambda^{p_1}\left(\mathbb{C}^m\right) 
\otimes \Lambda^{p_2}\left(\mathbb{C}^m\right) \otimes \cdots\otimes 
\Lambda^{p_n}\left(\mathbb{C}^m\right) 
\end{equation}
as $\mathrm{GL}_m\times T$-modules, where $T$ is the diagonal torus in 
$\mathrm{GL}_n$.  
\vskip0.5cm
This decomposition implies that  
\begin{equation}
\label{eq:howehom2}
\mathrm{Inv}_{\mathrm{SL}_m}\left(\Lambda^{p_1}\left(\mathbb{C}^m\right) 
\otimes \Lambda^{p_2}\left(\mathbb{C}^m \right)\otimes \cdots\otimes 
\Lambda^{p_n}\left(\mathbb{C}^m\right)\right)\cong W(p_1,\ldots,p_n), 
\end{equation}
where $W(p_1,\ldots,p_n)$ denotes the 
$(p_1,\ldots,p_n)$-weight space of $W_{(k^m)}$.  
\vskip0.5cm
It is worth noting that Cautis has written down a $q$-version of skew Howe duality in Section 6.1 
in~\cite{caut} (see also~\cite{ckm}). 
We do not recall his general explanation here.

Instead, in the next section, 
we use Kuperberg's webs to give a $q$-version 
of the isomorphism in~\eqref{eq:howehom2}, for $U_q(\mathfrak{sl}_3)$ 
and $U_q(\mathfrak{gl}_n)$ with $n=3k$ and $k\in\mathbb{N}$ arbitrary 
but fixed.
\vskip0.5cm
We also categorify this instance of $q$-skew Howe duality, as we will explain 
after the next section.
\subsection{Web algebras and the cyclotomic KLR algebras: The uncategorified story}\label{sec-webhowea}
\subsubsection{Enhanced sign sequences}
In this section we slightly generalise the notion of a sign sequence/string. We call this generalisation \textit{enhanced sign sequence} or \textit{enhanced sign string}. Note that, with a slight abuse of notation, we use $\hat{S}$ for sign strings and $S$ for enhanced sign string throughout the whole section. 
\begin{defn} An \textit{enhanced sign sequence/string} is 
a sequence $S=(s_1,\ldots, s_n)$ with entries $s_i\in\{\circ,-1,+1,\times\}$, for 
all $i=1,\ldots n$. The corresponding weight $\mu=\mu_S\in\Lambda(n,d)$ is 
given by the rules
\[
\mu_i=
\begin{cases}
0,&\quad\text{if}\;s_i=\circ,\\
1,&\quad\text{if}\;s_i=1,\\
2,&\quad\text{if}\;s_i=-1,\\
3,&\quad\text{if}\;s_i=\times.
\end{cases}
\] 
Let $\Lambda(n,d)_3\subset \Lambda(n,d)$ be the subset of weights 
with entries between $0$ and $3$. Recall that $\Lambda(n,d)_{1,2}$ denotes the subset of weights with only $1$ and $2$ as entries.
\end{defn}

Let $n=d=3k$. For any enhanced sign string $S$ such that 
$\mu_S\in\Lambda(n,n)_3$, we define $\hat{S}$ to be the sign sequence 
obtained from $S$ by deleting all entries that are equal to $\circ$ or 
$\times$ and keeping the linear ordering 
of the remaining entries. Similarly, for any $\mu\in\Lambda(n,n)_3$, 
let $\hat{\mu}$ be the weight obtained from $\mu$ by deleting all entries 
which are equal to $0$ or $3$. Thus, if $\mu=\mu_S$, for a certain enhanced 
sign string $S$, then $\hat{\mu}=\mu_{\hat{S}}$. 
Note that $\hat{\mu}\in\Lambda(m,d)_{1,2}$, for a certain 
$0\leq m\leq n$ and $d=3(k-(n-m))$. 

Note that for any semi-standard tableau $T\in \mathrm{Std}^{(3^k)}_{\mu}$, 
there is a unique 
semi-standard tableau $\hat{T}\in \mathrm{Std}^{(3^{k-(n-m)})}_{\hat{\mu}}$, 
obtained by deleting any cell in $T$ whose label appears three times and 
keeping the linear ordering of the remaining cells within each column. 

Conversely, let $\mu'\in\Lambda(m,d)_{1,2}$, 
with $m\leq n$ and $d=3(k-(n-m))$. In general, there is more than one 
$\mu\in \Lambda(n,n)_3$ such that $\hat{\mu}=\mu'$, but at least one. Choose one of them, 
say $\mu_0$. Then, given any $T'\in \mathrm{Std}^{(3^{k-(n-m)})}_{\mu'}$, there is 
a unique $T\in \mathrm{Std}^{(3^k)}_{\mu_0}$ such that 
$\hat{T}=T'$.

The construction of $T$ is as follows. Suppose that 
$i$ is the smallest number such that $(\mu_0)_i=3$.
\begin{itemize}
\item[(1)] In each column $c$ of $T'$, there is a unique vertical position such that 
all cells above that position have label smaller than $i$ and 
all cells below that position have label greater than $i$. Insert a new 
cell labeled $i$ precisely in that position, for each column $c$.
\item[(2)] In this way, we obtain a new tableau of shape $(3^{k-(n-m)+1})$. It is easy to 
see that this new tableau is semi-standard. Now apply this procedure 
recursively for each $i=1,\ldots, n$, such that $(\mu_0)_i=3$.
\item[(3)] In this way, we obtain a tableau $T$ of shape $(3^k)$. Since in each step the 
new tableau that we get is semi-standard, we see that $T$ belongs to 
$\mathrm{Std}^{(3^k)}_{\mu_0}$.  
\end{itemize} 
Note also that $\hat{T}=T'$. This shows that 
for a fixed $\mu\in\Lambda(n,n)_3$, we have a bijection
\[
\mathrm{Std}^{(3^k)}_{\mu}\ni T\; \longleftrightarrow\; \hat{T}\in
\mathrm{Std}^{(3^{k-(n-m)})}_{\hat{\mu}}.
\]
Given an enhanced sign sequence $S$, such that $\mu_S\in\Lambda(n,n)_3$, 
we define 
\[
W_S=W_{\hat{S}}.
\]
In other words, as a vector space $W_S$ does not depend on the $\circ$ and 
$\times$-entries of $S$. However, they do play an important role below.   
Similarly, we define 
\[
B_S=B_{\hat{S}}\quad\text{and}\quad K_S=K_{\hat{S}}.
\]
\subsubsection{An instance of $q$-skew Howe duality}
Let $V_{(3^k)}$ be the irreducible 
$U_q(\mathfrak{gl}_n)$-module of highest weight $(3^k)$. 
By restriction, $V_{(3^k)}$ is also a $U_q(\mathfrak{sl}_n)$-module 
and, since it is a weight representation, it is a 
$\dot{\mathbf U}(\mathfrak{sl}_n)$-module, too.
It is well-known (see~\cite{fu} and~\cite{ma}) for example) that 
\[
\dim V_{(3^k)}=\sum_{\mu\in\Lambda(n,n)_3}\#\mathrm{Std}^{(3^k)}_{\mu}.
\] 
Note that a tableau of shape $(3^k)$ can only be semi-standard if its 
filling belongs to $\Lambda(n,n)_3$, so strictly speaking we could 
drop the $3$-subscript. More precisely, if 
\[
V_{(3^k)}=\bigoplus_{\mu\in\Lambda(n,n)_3} V_{(3^k)}(\mu)
\]
is the $U_q(\mathfrak{gl}_n)$-weight space decomposition of $V_{(3^k)}$, then 
\[
\dim V_{(3^k)}(\mu)=\#\mathrm{Std}^{(3^k)}_{\mu}.
\] 
Note that the action of $U_q(\mathfrak{gl}_n)$ on $V_{(3^k)}$ descends to 
$S_q(n,n)$ and recall that there exists a surjective algebra homomorphism 
\[
\psi_{n,n}\colon \U\to S_q(n,n).
\]
The action of $\U$ on $V_{(3^k)}$ is equal to the pull-back of the 
action of $S_q(n,n)$ via $\psi_{n,n}$. 

Define 
\[
W_{(3^k)}=\bigoplus_{S\in \Lambda(n,n)_3} W_S.
\]
Below, we will show that $S_q(n,n)$ acts on $W_{(3^k)}$. Pulling 
back the action via $\psi_{n,n}$, we see that $W_{(3^k)}$ is a 
$\U$-module. We will also show that 
\[
W_{(3^k)}\cong V_{(3^k)}
\]
as $S_q(n,n)$-modules, and therefore also as $\U$-modules, and that 
$W_S$ corresponds to the $\mu_S$-weight space of $V_{(3^k)}$. 
\vskip0.5cm
Let us define the aforementioned left action of $S_q(n,n)$ on $W_{(3^k)}$. 
The reader should compare this action to the categorical action on the 
objects in Section 4.2 in~\cite{msv1}. Note that our conventions in 
this thesis are different from those in~\cite{msv1}. 
 
\begin{defn}
\label{defn:phi}
Let 
\[
\phi\colon S_q(n,n)\to \mathrm{End}_{\mathbb{C}(q)}\left(W_{(3^k)}\right)
\] 
be the homomorphism of $\mathbb{C}(q)$-algebras defined by glueing 
the following webs on top of the elements in $W_{(3^k)}$.
\begin{align*}
1_{\lambda}&\mapsto
\;\xy
(0,0)*{\includegraphics[width=60px]{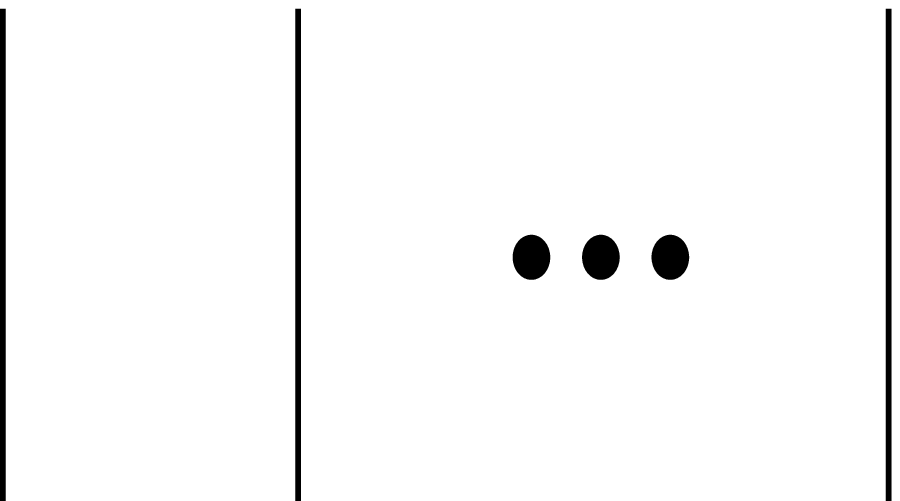}};
(-10,-8)*{\scriptstyle\lambda_1};
(-3,-8)*{\scriptstyle\lambda_{2}};
(10,-8)*{\scriptstyle\lambda_n};
\endxy
\\[0.5ex]
E_{\pm i}1_{\lambda}&\mapsto
\;\xy
(-0,0)*{\includegraphics[width=150px]{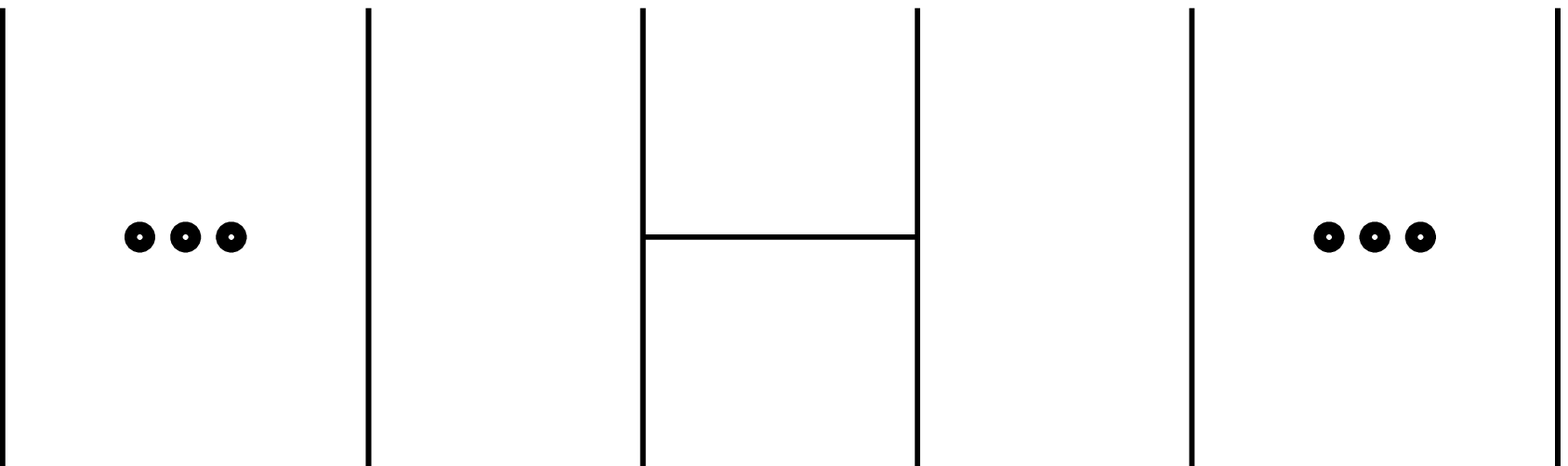}};
(-27,-10)*{\scriptstyle\lambda_1};
(-14,-10)*{\scriptstyle\lambda_{i-1}};
(-5,-10)*{\scriptstyle\lambda_{i}};
(5,-10)*{\scriptstyle\lambda_{i+1}};
(-5,10)*{\scriptstyle\lambda_{i}\pm 1};
(5,10)*{\scriptstyle\lambda_{i+1}\mp 1};
(14,-10)*{\scriptstyle\lambda_{i+2}};
(27,-10)*{\scriptstyle\lambda_n};
\endxy
\end{align*}
We use the convention that vertical edges labeled 1 are oriented upwards, 
vertical edges labeled 2 are oriented downwards and edges labeled 0 or 3 
are erased. The orientation of the horizontal edges is uniquely determined by 
the orientation of the vertical edges. With these conventions, one can 
check that the horizontal edge is always oriented from right 
to left for $E_{+i}$ and from left to right for $E_{-i}$.  

Furthermore, let $\lambda\in\Lambda(n,n)$ and let $S$ be any 
sign string such that $\mu_S\in\Lambda(n,n)_3$. For any $w\in W_S$, 
we define 
\[
\phi(1_{\lambda})w=0,\quad\text{if}\quad\mu_S\ne\lambda.
\]
By 
$\phi(1_{\lambda})w$ we mean the left action of $\phi(1_{\lambda})$ on $w$. 
In particular, for any $\lambda>(3^k)$, we have $\phi(1_{\lambda})=0$ in 
$\mathrm{End}_{\mathbb{C}(q)}\left(W_{(3^k)}\right)$.  
\end{defn}

Let us give two examples to show how these conventions work. We only write 
down the relevant entries of the weights and only draw the important edges. 
We have 
\begin{align*}
E_{+1}1_{(22)}&\mapsto
\;\xy
(-0,0)*{\includegraphics[width=30px]{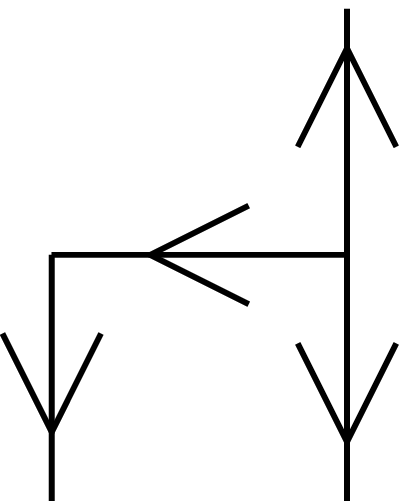}};
(-4,-8)*{\scriptstyle 2};
(4,-8)*{\scriptstyle 2};
(-4,8)*{\scriptstyle 3};
(4,8)*{\scriptstyle 1};
\endxy
\\[0.5ex]
E_{-2}E_{+1}1_{(121)}&\mapsto
\;\xy
(-0,0)*{\includegraphics[width=60px]{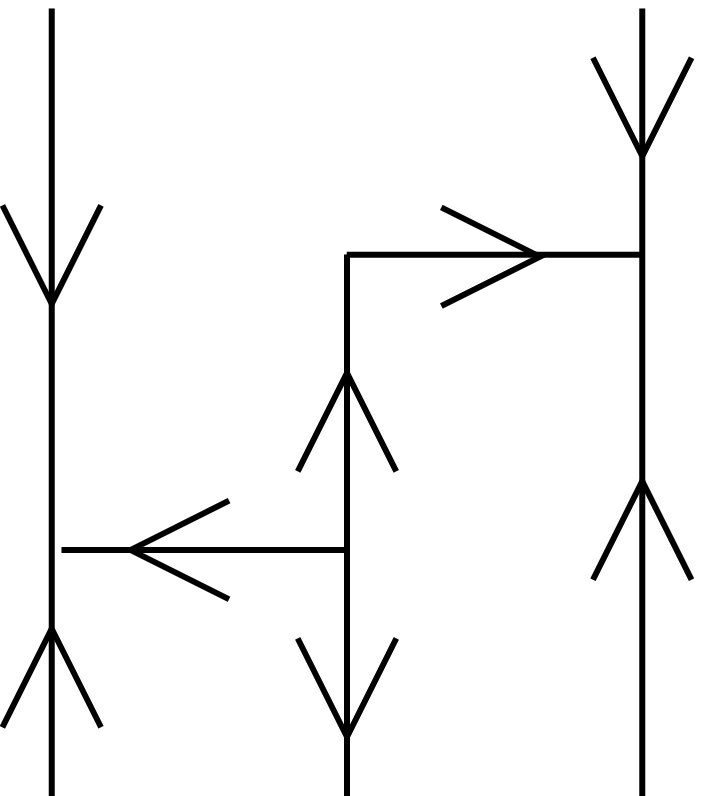}};
(-9,-14)*{\scriptstyle 1};
(0,-14)*{\scriptstyle 2};
(9,-14)*{\scriptstyle 1};
(-9,14)*{\scriptstyle 2};
(0,14)*{\scriptstyle 0};
(9,14)*{\scriptstyle 2};
\endxy
\end{align*}

\begin{rem} Note that the introduction of enhanced sign strings is necessary 
for the definition of $\phi$ to make sense. Although as a vector space 
$W_S$ does not depend on the entries of $S$ which are equal to 
$\circ$ or $\times$, the $S_q(n,n)$-action on $W_S$ does depend on them. 
\end{rem}

\begin{rem}
A more general version of the map $\phi$ was studied later in the paper~\cite{ckm} by Cautis, Kamnitzer and Morrison.
\end{rem}

\begin{lem} 
The map $\phi$ in Definition~\ref{defn:phi} is well-defined. 
\end{lem}
\begin{proof}
It follows immediately from its definition that $\phi$ preserves the three 
relations~\eqref{eq:schur1},~\eqref{eq:schur2} and~\eqref{eq:schur3}.  

Checking case by case, one can easily show that $\phi$ 
preserves~\eqref{eq:schur4} by using the relations~\eqref{eq:circle}, \eqref{eq:digon} and~\eqref{eq:square}. We do just one example and 
leave the other cases to the reader. The figure 
below shows the image of the relation 
\[
E_1E_{-1}1_{(21)}-E_{-1}E_11_{(21)}=1_{(21)}
\]
under $\phi$.   
\begin{align*}
\xy
(0,1)*{\includegraphics[width=35px]{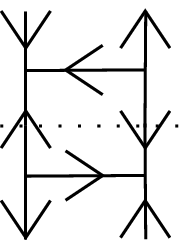}};
(-4.5,-9)*{\scriptstyle 2};
(4,-9)*{\scriptstyle 1};
(-4.5,11)*{\scriptstyle 2};
(4,11)*{\scriptstyle 1};
\endxy
\;-\;
\xy
(0,0)*{\includegraphics[width=35px]{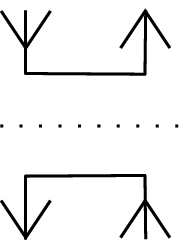}};
(-4.5,-10)*{\scriptstyle 2};
(4,-10)*{\scriptstyle 1};
(-4.5,10)*{\scriptstyle 2};
(4,10)*{\scriptstyle 1};
\endxy
\;=\;
\xy
(0,0)*{\includegraphics[width=35px]{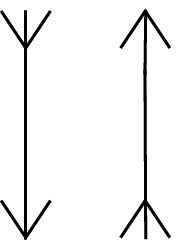}};
(-4.5,-10)*{\scriptstyle 2};
(4,-10)*{\scriptstyle 1};
(-4.5,10)*{\scriptstyle 2};
(4,10)*{\scriptstyle 1};
\endxy
\end{align*}
This relation is exactly the third Kuperberg relation in~\eqref{eq:square}.
\end{proof}

\begin{lem} The map $\phi$ gives rise to an isomorphism 
\label{lem:isoirrep}
\[
\phi\colon V_{(3^k)}\to W_{(3^k)}
\]
of $S_q(n,n)$-modules.  
\end{lem}
\begin{proof}
Note that the empty web $w_h=w_{(3^k)}$, which generates 
$W_{(\times^k,\circ^{2k})}\cong \mathbb{C}(q)$, is a highest weight vector. 

The map $\phi$ induces a surjective homomorphism of 
$S_q(n,n)$-modules 
\[
\phi\colon S_q(n,n)1_{(3^k)}\to W_{(3^k)},
\]
defined by 
\[
\phi(x1_{(3^k)})=\phi(x)w_h.
\]

As we already remarked above, we have
\begin{align*}
\dim V_{(3^k)}&=\sum_{\mu_S\in\Lambda(n,n)_3} 
\#\mathrm{Std}^{(3^k)}_{\mu_S}\\
&=\sum_{\mu_S\in\Lambda(n,n)_3} \dim W_S
=\dim W_{(3^k)}.
\end{align*} 
Therefore, we have 
\[
V_{(3^k)}\cong \phi\left(S_q(n,n)\right)w_h\cong W_{(3^k)}, 
\]
which finishes the proof. 
\vskip0.2cm
It is well-known that 
\[
V_{(3^k)}\cong 
S_q(n,n)1_{(3^k)}/(\mu > (3^k)),
\]
where $(\mu > (3^k))$ is the ideal generated by all 
elements of the form $x1_{\mu}y1_{(3^k)}$, with $x,y\in S_q(n,n)$ 
and $\mu$ is some weight greater than $(3^k)$. This quotient of 
$S_q(n,n)$ is an example of a so called \textit{Weyl module}. 
We see that the kernel of $\phi$ is also 
equal to $(\mu > (3^k))$. 
\end{proof}
\vskip0.2cm
We want to explain two more facts about the 
isomorphism in Lemma~\ref{lem:isoirrep}, which we will need later.  

Recall that there is an inner product on $V_{(3^k)}$. 
First of all, there is a $\bC$-linear and $q$-antilinear involution on $\mathbb{C}(q)$ determined by 
\[
\overline{aq^n}=\overline{a}q^{-n},
\]
for any $a\in\mathbb{C}$. Here $\overline{a}$ denotes the complex conjugate 
of $a$. Recall Lusztig's $q$-antilinear (antilinear means 
w.r.t. to the involution above) algebra anti-involution 
$\tau$ on $S_q(n,n)$ defined by 
\[
\tau(1_{\lambda})=1_{\lambda},\;\;
\tau(1_{\lambda+\alpha_i}E_i1_{\lambda})= 
q^{-1-\overline{\lambda}_i}1_{\lambda}E_{-i}1_{\lambda+\alpha_i},\;\;
\tau(1_{\lambda}E_{-i}1_{\lambda+\alpha_i})= q^{1+\overline{\lambda}_i}
1_{\lambda+\alpha_i}E_i1_{\lambda}.
\] 
The $q$-\textit{Shapovalov form} $\langle\;\cdot\;,\;\cdot\;\rangle$ on 
$V_{(3^k)}$ is the unique $q$-sesquilinear form such that 
\begin{enumerate}
\item $\langle v_h,v_h\rangle =1$, for a fixed highest weight vector $v_h$.
\item $\langle x v, v' \rangle=\langle v,\tau(x) v'\rangle$, for any 
$x\in S_q(n,n)$ and any $v,v'\in V_{(3^k)}$.
\item $f\langle v,v'\rangle =\langle v\overline{f},v'\rangle=\langle 
v,v'f\rangle$, for any $f\in \mathbb{C}(q)$ and any $v,v'\in V_{(3^k)}$.
\end{enumerate}  

We can also define an inner product on $W_{(3^k)}$, using the 
Kuperberg bracket. Let $S$ be any enhanced sign string $S$, 
such that $\mu_S\in\Lambda(n,n)_3$. Denote the length of the sign string 
$\hat{S}$ by $\ell(\hat{S})$.  
\begin{defn}
\label{defn:normkuperform}
Define the $q$-sesquilinear 
\textit{normalised Kuperberg form} by 
\begin{itemize}
\item $\langle w_h,w_h\rangle=1$, for a fixed highest weight vector $w_h$.
\item $\langle u,v\rangle=q^{\ell(\hat{S})}\langle u^*v\rangle_{\mathrm{Kup}}$, for any $u,v\in B_S$.
\item $\langle f(q)u,g(q)v\rangle=\overline{f(q)}g(q)\langle u,v\rangle$, for any $u,v\in B_S$ and $f(q),g(q)\in \mathbb{C}(q)$. 
\end{itemize}
\end{defn}
The following lemma motivates the normalisation of the Kuperberg form. 
\begin{lem}
\label{lem:phiisometry}
The isomorphism of $S_q(n,n)$-modules 
\[
\phi\colon V_{(3^k)}\to W_{(3^k)}
\]
is an isometry. 
\end{lem}
\begin{proof}
First note that  
\[
\langle (E_{\pm i}u)^*v\rangle_{\mathrm{Kup}}=
\langle u^*E_{\mp i}v\rangle_{\mathrm{Kup}},
\]
for any $u,v\in W_S$ and any $i=1,\ldots, n$, which is exactly (2) from above. 
This shows that the result of the lemma holds up to normalisation. 

Our normalisation of the Kuperberg form matches the 
normalisation of the $q$-Shapovalov form. One can easily check this 
case by case. Let us just do two examples. Let $i=1$. Then one has
$E_11_{(a,b,\ldots)}=1_{(a+1,b-1,\ldots)}E_1$. If 
$(a,b,\ldots)\in\Lambda(n,n)_3$ such that $a-b=-1$, 
then 
\[
\ell(\widehat{(a,b)})=\ell(\widehat{(a+1,b-1)}),
\]
where $\ell$ indicates the length of the sign sequence. 
This matches 
\[
\tau(E_11_{(a,b)})=1_{(a,b)}E_{-1}.
\]  
If $(a,b)=(2,1)$, then $E_11_{(2,1,\ldots)}=1_{(3,0,\ldots)}E_1$. Note 
that 
\[
\ell(\widehat{(2,1,\ldots)})=\ell(\widehat{(3,0,\ldots)})+2.
\]
This $+2$ 
cancels exactly with the $-2$, which appears as the exponent of 
$q$ in 
\[
\tau(E_11_{(2,1,\ldots)})=q^{-2}1_{(2,1,\ldots)}E_{-1}.
\]    
\end{proof}

We will need one more fact about $\phi$. For any $i=1,\ldots,n$ 
and any $a\in\mathbb{N}$, let  
\[
E_{\pm i}^{(a)}=\dfrac{E_{\pm i}^a}{[a]!}
\]
denote the \textit{divided power} in $S_q(n,n)$. 
Recall the following relations for the divided powers.
\begin{eqnarray}
\label{eq:divpow1}
E_{\pm i}^{(a)}E_{\pm i}^{(b)}1_{\lambda}&=&\qbin{a+b}{a}E_{\pm i}^{(a+b)}1_{\lambda},\\
\label{eq:divpow2}
E_{+i}^{(a)}E_{-i}^{(b)}1_{\lambda}&=&
\sum_{j=0}^{\min(a,b)}\qbin{a-b+\lambda_i-\lambda_{i+1}}{j}
E_{-i}^{(b-j)}E_{+i}^{(a-j)}1_{\lambda},\\
\label{eq:divpow3}
E_{-i}^{(b)}E_{+i}^{(a)}1_{\lambda}&=&
\sum_{j=0}^{\min(a,b)}\qbin{b-a-(\lambda_i-\lambda_{i+1})}{j}
E_{+i}^{(a-j)}E_{-i}^{(b-j)}1_{\lambda}.
\end{eqnarray}
Here $[a]!$ denotes the \textit{quantum factorial} and $\qbin{a}{b}$ 
denotes the \textit{quantum binomial}.

The images of the divided powers under 
\[
\phi\colon S_q(n,n)\to \mathrm{End}(W_{(3^k)})
\]
are easy to compute. For example, we have (for simplicity, we only 
draw two of the strands and write $E=E_{+i}$)
\[
\phi(E^21_{(0,2)})=
\;\xy
(0,0)*{\includegraphics[width=30px]{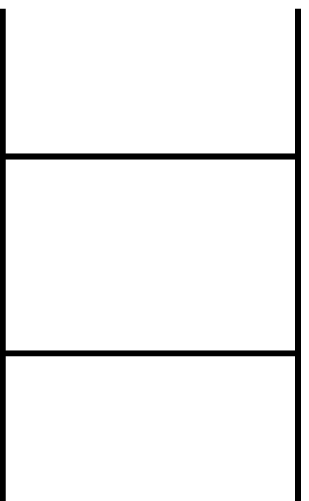}};
(-6,-10)*{\scriptstyle 0};
(6,-10)*{\scriptstyle 2};
(-7,0)*{\scriptstyle 1};
(7,0)*{\scriptstyle 1};
(-6,10)*{\scriptstyle 2};
(6,10)*{\scriptstyle 0};
\endxy 
\;
=
\;\xy
(0,0)*{\includegraphics[width=40px]{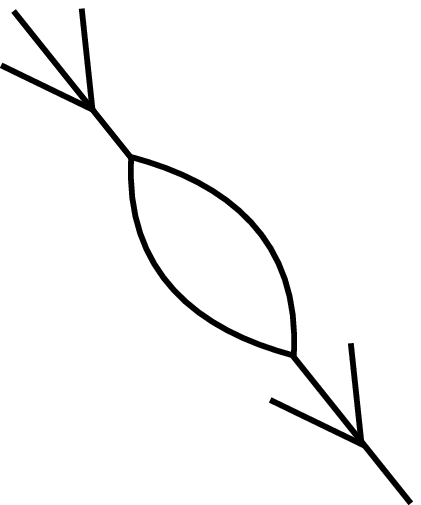}};
(6,10)*{\scriptstyle \circ};
(-6,10)*{\scriptstyle -};
(-6,-10)*{\scriptstyle \circ};
(6,-10)*{\scriptstyle -};
\endxy 
=
\;
[2]
\;\xy
(0,0)*{\includegraphics[width=40px]{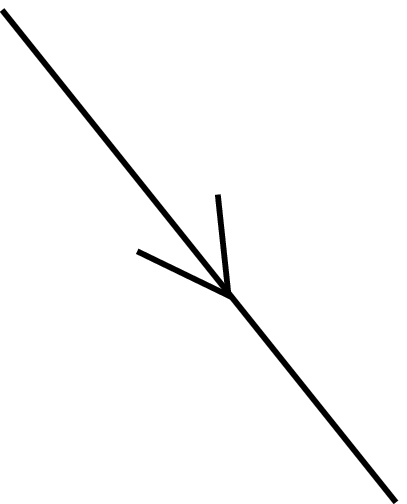}};
(6,10)*{\scriptstyle \circ};
(-6,10)*{\scriptstyle -};
(-6,-10)*{\scriptstyle \circ};
(6,-10)*{\scriptstyle -};
\endxy. 
\]
Therefore, we get 
\[
\phi(E^{(2)}1_{(0,2)})=
\;\xy
(0,0)*{\includegraphics[width=40px]{res/figs/section51/DivEE02.eps}};
(6,10)*{\scriptstyle \circ};
(-6,10)*{\scriptstyle -};
(-6,-10)*{\scriptstyle \circ};
(6,-10)*{\scriptstyle -};
\endxy. 
\]
Another interesting example is 
\[
\phi(E^{2}1_{(0,3)})=
\;\xy
(-0,0)*{\includegraphics[width=30px]{res/figs/section51/HHweb.eps}};
(-6,-10)*{\scriptstyle 0};
(6,-10)*{\scriptstyle 3};
(-7,0)*{\scriptstyle 1};
(7,0)*{\scriptstyle 2};
(-6,10)*{\scriptstyle 2};
(6,10)*{\scriptstyle 1};
\endxy
\;
=
\;
[2]
\;\xy
(0,0)*{\includegraphics[width=40px]{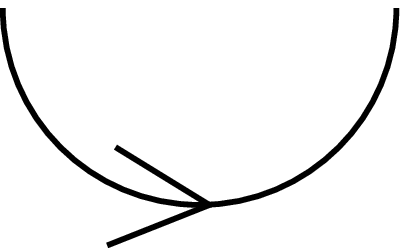}};
(6,10)*{\scriptstyle +};
(-6,10)*{\scriptstyle -};
(-6,-10)*{\scriptstyle \circ};
(6,-10)*{\scriptstyle \times};
\endxy, 
\]
which shows that  
\[
\phi(E^{(2)}1_{(03)})=
\;
\xy
(0,0)*{\includegraphics[width=40px]{res/figs/section51/rightcup.eps}};
(6,10)*{\scriptstyle +};
(-6,10)*{\scriptstyle -};
(-6,-10)*{\scriptstyle \circ};
(6,-10)*{\scriptstyle \times};
\endxy. 
\]
The final example we will consider is $\phi(E^{(3)}1_{(0,3)})$.
We see that 
\[
\phi(E^{3}1_{(0,3)})=
\;\xy
(-0,0)*{\includegraphics[width=30px]{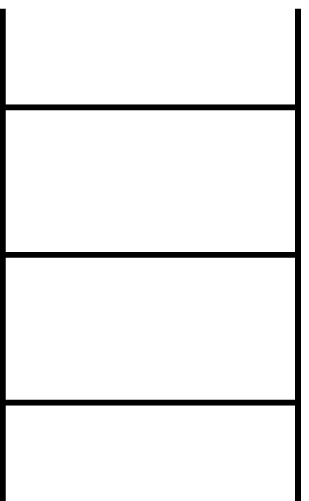}};
(-6,-10)*{\scriptstyle 0};
(6,-10)*{\scriptstyle 3};
(-7,-3)*{\scriptstyle 1};
(7,-3)*{\scriptstyle 2};
(-7,3)*{\scriptstyle 2};
(7,3)*{\scriptstyle 1};
(-6,10)*{\scriptstyle 3};
(6,10)*{\scriptstyle 0};
\endxy
\;
=
\;
\xy
(0,0)*{\includegraphics[width=40px]{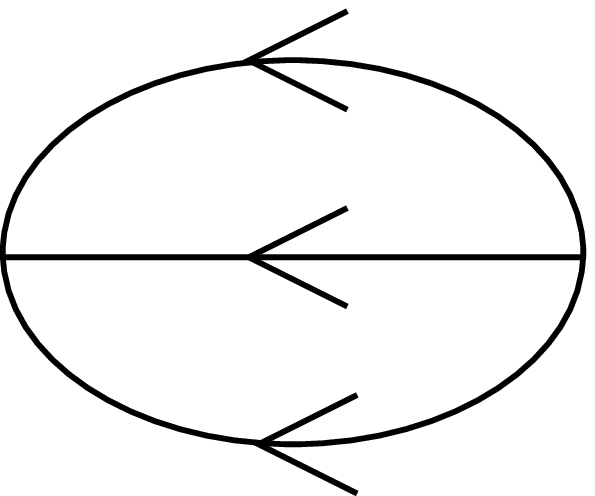}};
(-6,10)*{\scriptstyle \times};
(6,10)*{\scriptstyle \circ};
(-6,-10)*{\scriptstyle \circ};
(6,-10)*{\scriptstyle \times};
\endxy
\;
=
\;
[3]!
\xy
(-6,10)*{\scriptstyle \times};
(6,10)*{\scriptstyle \circ};
(-6,-10)*{\scriptstyle \circ};
(6,-10)*{\scriptstyle \times};
\endxy.
\]
Thus, we have 
\[
\phi(E^{(3)}1_{(0,3)})=\;\xy
(-6,10)*{\scriptstyle \times};
(6,10)*{\scriptstyle \circ};
(-6,-10)*{\scriptstyle \circ};
(6,-10)*{\scriptstyle \times};
\endxy,
\]
which is the unique empty web from $(\circ,\times)$ to 
$(\times,\circ)$.

Note that~\eqref{eq:divpow2} and~\eqref{eq:divpow3} imply that, 
for any $a\in\mathbb{N}$, we have 
\begin{equation}
\label{eq:adjust1}
E_{-i}^{(a)}E_{+i}^{(a)}1_{(\ldots,0,a,\ldots)}=1_{(\ldots,0,a,\ldots)}
\quad\text{and}\quad 
E_{+i}^{(a)}E_{-i}^{(a)}1_{(\ldots,a,0,\ldots)}=1_{(\ldots,a,0,\ldots)}
\end{equation}
in $S_q(n,n)$. Similarly, 
let $S_q(n,n)/I$, where $I$ denotes the two-sided ideal generated by 
all $1_{\mu}$ such that $\mu>(3^k)$. Again by~\eqref{eq:divpow2} 
and~\eqref{eq:divpow3}, we have
\begin{equation}
\label{eq:adjust2}
E_{-i}^{(3-a)}E_{+i}^{(3-a)}1_{(\ldots,a,3,\ldots)}=1_{(\ldots,a,3,\ldots)}
\quad\text{and}\quad 
E_{+i}^{(3-a)}E_{-i}^{(3-a)}1_{(\ldots,3,a,\ldots)}=1_{(\ldots,3,a,\ldots)}
\end{equation}
in $S_q(n,n)/I$.
One can check that $\phi$ maps the two sides of the equations 
in~\eqref{eq:adjust1} and~\eqref{eq:adjust2} to isotopic diagrams. 
For example, $\phi$ maps 
\[
E_{-}^{(2)}E_{+}^{(2)}1_{(0,2)}=1_{(0,2)}
\]
to 
\[
\xy
(0,0)*{\includegraphics[width=22px]{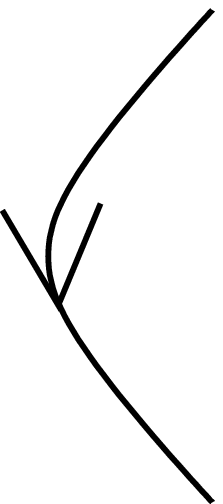}};
(-6,12)*{\scriptstyle \circ};
(6,12)*{\scriptstyle -};
(-6,0)*{\scriptstyle -};
(6,0)*{\scriptstyle \circ};
(-6,-12)*{\scriptstyle \circ};
(6,-12)*{\scriptstyle -};
\endxy
\;
=
\;
\xy
(6,0)*{\includegraphics[width=11.3px]{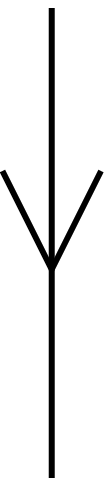}};
(-6,12)*{\scriptstyle \circ};
(6,12)*{\scriptstyle -};
(-6,-12)*{\scriptstyle \circ};
(6,-12)*{\scriptstyle -};
\endxy.
\]

\begin{rem}
Let 
\[
W_{(3^k)}^{\mathbb{Z}}=\bigoplus_{\mu_S\in\Lambda(n,n)_3}W_S^{\mathbb{Z}}
\]
be the integral form. Then the remarks above show that the action in 
Definition~\ref{defn:phi} restricts to a well-defined action of 
$S_q^{\mathbb{Z}}(n,n)$ on 
$W_{(3^k)}^{\mathbb{Z}}$. Therefore, the isomorphism in Lemma~\ref{lem:isoirrep} 
restricts to a well-defined isomorphism between the integral forms 
\[
V_{(3^k)}^{\mathbb{Z}}\cong W_{(3^k)}^{\mathbb{Z}}.
\]
\end{rem}

The proof of the following lemma is based on an algorithm, which we call 
\textit{enhanced inverse growth algorithm}. The result is needed later to 
show surjectivity in Theorem~\ref{thm:equivalence}. 
\begin{lem}
\label{lem:phisurj}
Let $S$ be any enhanced sign string such that $\mu_S\in\Lambda(n,n)_3$. 
For any $w\in B_S$, there exists a product of divided powers $x$, such that
\[
\phi(x1_{(3^k)})=w.
\] 
\end{lem}
\begin{proof}
Choose any $w\in B_S$. We consider $w\in B_{(\times^k,\circ^{2k})}^S$, i.e. 
a non-elliptic web with (empty) lower boundary determined by 
$(\times^k,\circ^{2k})$ and upper boundary determined by $S$. 
Express $w$ using the growth algorithm, 
in an arbitary way. Suppose there are $m$ steps in this instance of the 
growth algorithm. The element $x$ is built up in $m+2$ steps, i.e. 
an initial step, one step for each step in the growth algorithm, and a last 
step. During the construction of $x$, we always keep track of 
the $\circ$s and $\times$s. At each step 
the strands of $w$ are numbered according to their position in $x$. 

If the $H$, $Y$ or arc-move is applied to two non-consecutive strands, 
we first have to apply some divided powers, 
as in~\eqref{eq:adjust1} and~\eqref{eq:adjust2}, 
to make them consecutive. Let $x_k\in S_q(n,n)$ be the element 
assigned to the $k$-th step and let $\mu^k$ be the weight after the 
$k$-step, i.e. $x_k=1_{\mu^{k-1}}x_k1_{\mu^k}$. 
The element $x$ we are looking for is the product of all $x_k$. 
\begin{enumerate}
\item Take $x_0=1_{\mu_S}$.
\item Suppose that the $k$-th step in the growth algorithm is applied to 
the strands $i$ and $i+r$, for some $r\in\mathbb{N}_{>0}$. This means that 
the entries of $\mu^{k-1}$ satisfy $\mu_j\in\{0,3\}$, 
for all $j=i+1,\ldots,i+r-1$. Let $x_k'$ be the product of divided powers 
which ``swap'' the $(\mu_{i+1},\ldots,\mu_{i+r-1})$ and $\mu_{i+r}$. So, we 
first swap $\mu_{i+r-1}$ and $\mu_{i+r}$, then $\mu_{i+r-2}$ and $\mu_{i+r}$ 
etc. Now, the rule in the growth algorithm, still corresponding to the 
$k$-th step, can be applied to the strands $i$ and $i+1$. 
\item Suppose that it is an $H$-rule. If the bottom of the $H$ is a pair 
(up-arrow down-arrow), then take $x_k=x_k'E_{+i}$. If the bottom of the 
$H$ is a pair (down-arrow up-arrow), then take $x_k=X_k'E_{-i}$. 
\item Suppose that the rule, corresponding to the $k$-th step in 
the growth algorithm, is a $Y$-rule. If the bottom strand of $Y$ is oriented 
downward, then take $x_k=x_k'E_{-i}$. If it is oriented upward, take 
$x_k=x_k'E_{+i}$. Note that these two choices are not unique. They depend on 
where you put $0$ or $3$ in $\mu^k$. The choice we made corresponds to taking 
$(\mu^k_i,\mu^k_{i+1})=(2,0)$ in the first case and 
$(\mu^k_i,\mu^k_{i+1})=(1,3)$ in the second case. Other choices would be 
perfectly fine and would lead to equivalent elements in $S_q(n,n)1_{(3^k)}/
(\mu > (3^k))$. 
\item Suppose that the rule, corresponding to the $k$-th step in 
the growth algorithm, is an arc-rule. If the arc is oriented clockwise, 
take $x_k=x_k'E_{-i}^{(2)}$. If the arc is oriented counter-clockwise, 
take $x_k=x_k'E_{-i}$. Again, these choices are not unique. They correspond 
to taking $(\mu^k_i,\mu^k_{i+1})=(3,0)$ in both cases. 
\item After the $m$-th step in the growth algorithm, which is the last one, 
we obtain $\mu^m$, which is a sequence of $3$s and $0$s. Let $x_{m+1}$ be 
the product of divided powers which reorders the entries of $\mu^m$, so that 
$\mu^{m+1}=(3^k)$. 
\item Take $x=1_{\mu_S}x_1x_2\cdots x_{m+1}1_{(3^k)}\in S_q(n,n)$. Note that 
$x$ is of the form $E_{\ii}1_{(3^k)}$.      
\end{enumerate}
From the analysis of the images of the divided powers under $\phi$, it is 
clear that 
\[
\phi(x)=w.
\] 
\end{proof}

We do a simple example to illustrate Lemma~\ref{lem:phisurj}.
Let 
\[
w=
\;
\xy
(0,0)*{\includegraphics[width=40px]{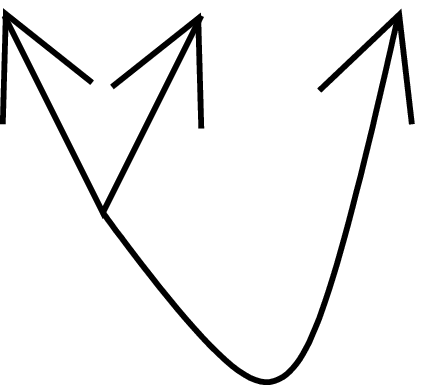}};
(-6,10)*{\scriptstyle 1};
(0,10)*{\scriptstyle 1};
(6,10)*{\scriptstyle 1};
\endxy
\] 
Then the algorithm in the proof of Lemma~\ref{lem:phisurj} gives 
\[
x=1_{(111)}E_{-1}E_{-2}E_{-1}1_{(300)},
\]
or as a picture (read from bottom to top)
\[
\xy
(0,0)*{\includegraphics[width=50px]{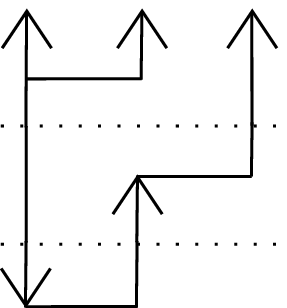}};
(-9,12)*{\scriptstyle 1};
(0.3,12)*{\scriptstyle 1};
(9,12)*{\scriptstyle 1};
(-15,5)*{\scriptstyle{E_{-1}}};
(-9,2)*{\scriptstyle 2};
(0.3,2)*{\scriptstyle 0};
(9,2)*{\scriptstyle 1};
(-15,-1.5)*{\scriptstyle{E_{-2}}};
(-9,-5.5)*{\scriptstyle 2};
(0.3,-5.5)*{\scriptstyle 1};
(9,-5.5)*{\scriptstyle 0};
(-15,-9)*{\scriptstyle{E_{-1}}};
(-9,-12)*{\scriptstyle 3};
(0.3,-12)*{\scriptstyle 0};
(9,-12)*{\scriptstyle 0};
\endxy.
\]

We are now ready to start explaining the categorified story.
\subsection{Web algebras and the cyclotomic KLR algebras: And its categorification}\label{sec-webhoweb}
Let us denote with $K_S\text{-}\mathrm{p\textbf{Mod}}_{\mathrm{gr}}$ 
the category of all finite dimensional, projective, unitary, 
graded $K_S$-modules and 
$K^{\oplus}_0(K_S)=K^{\oplus}_0(K_S\text{-}\mathrm{p\textbf{Mod}}_{\mathrm{gr}})$ its 
split Grothendieck group. 
Recall that a \textit{unitary} module is one on which the identity of $K_S$ 
acts as the identity operator. In what follows, 
it will sometimes be useful to consider homomorphisms 
of arbitrary degree, so we define
\[
\mathrm{HOM}_{B}(M,N)=\bigoplus_{t\in\mathbb{Z}}\text{hom}_B(M,N\{t\}),
\]
for any finite dimensional, associative, unital, graded algebra $B$ and 
any finite dimensional, unitary, graded $B$-modules $M$ and $N$. Note that for 
almost all 
$t\in\mathbb{Z}$ we have $\text{hom}_B(M,N\{t\})=\{0\}$, so 
$\text{HOM}_B(M,N)$ is still finite dimensional.

Moreover, we need the following notions throughout the rest of the section.

Suppose that $S$ is an enhanced sign string 
such that $\mu_S\in\Lambda(n,n)_3$. For any $u\in B_S$, let 
\[
P_u=\bigoplus_{w\in B_S} {}_wK_u.
\]
Then we have   
\[
K_S=\bigoplus_{u\in B_S} P_u,
\]
and so $P_u$ is an object in $K_S\text{-}\mathrm{p\textbf{Mod}}_{\mathrm{gr}}$, 
for any $u\in B_S$. Note that, for any $u,v\in B_S$, we have  
\[
\mathrm{HOM}(P_u,P_v)\cong {}_uK_v,
\]
where an element in ${}_{u^{\prime}}K_{v^{\prime}}$ acts on $P_u$ by composition on the 
left-hand side. 

Similarly, we can define 
\[
{}_uP=\bigoplus_{w\in B_S} {}_uK_w,
\]
which is a right graded, projective $K_S$-module. 

\begin{rem}
Just one warning. The reader should not confuse 
$P_u$ with $P_{u,T}$ in Section~\ref{sec-webalg}.
\end{rem} 
\subsubsection{The definition of ${\mathcal W}_{(3^k)}$}
Recall that $S$ denotes an enhanced sign string. Define  
\[
K_{(3^k)}=\bigoplus_{\mu_S\in\Lambda(n,n)_3} K_S
\]
and
\[
{\mathcal W}_{(3^k)}=K_{(3^k)}\text{-}\mathrm{p\textbf{Mod}}_{\mathrm{gr}}\cong 
\bigoplus_{\mu_S\in\Lambda(n,n)_3} 
K_S\text{-}\mathrm{p\textbf{Mod}}_{\mathrm{gr}}.
\]
The main goal of this section is to show that 
there exists a categorical $\Ucat$-action on ${\mathcal W}_{(3^k)}$ and 
that 
\[
{\mathcal W}_{(3^k)}\cong {\mathcal V}_{(3^k)}
\]
as $\Ucat$-2-representations as explained in Section~\ref{sec-techhigherrep}. 

This will imply that 
\[
K^{\oplus}_0({\mathcal W}_{(3^k)})\cong V_{(3^k)}^{\mathbb{Z}}.
\]
Note that 
\[
K^{\oplus}_0({\mathcal W}_{(3^k)})\cong \bigoplus_{\mu_S\in\Lambda(n,n)_3}K^{\oplus}_0(K_S).
\]
We will show that this corresponds exactly to 
the $U_q(\mathfrak{gl}_n)$-weight space decomposition of $V_{(3^k)}$. 
In particular, this will show that 
\begin{equation}
\label{eq:fundequality}
K^{\oplus}_0(K_S)\cong W_S,
\end{equation}
for any enhanced sign sequence $S$ such that $\mu_S\in \Lambda(n,n)_3$. 
\vskip0.5cm
First, we have to recall the definitions of 
\textit{sweet} bimodules.
\subsubsection{Sweet bimodules}
Note that the following definitions and results are 
the $\mathfrak{sl}_3$-analogues of those in Section 2.7 in~\cite{kh4}.

\begin{defn} Given rings $R_1$ and $R_2$, a $(R_1,R_2)$-bimodule $N$ is 
called \textit{sweet} if it is finitely generated and projective as a 
left $R_1$-module and as a right $R_2$-module. 
\end{defn}
If $N$ is a sweet $(R_1,R_2)$-bimodule, then the functor 
\[
N\otimes_{R_2}- \colon 
R_2\text{-}\mathrm{\textbf{Mod}}\to R_1\text{-}\mathrm{\textbf{Mod}}
\]
is exact and sends projective modules to projective modules. Given a 
sweet $(R_1,R_2)$-bimodule $M$ and a sweet $(R_2,R_3)$-bimodule $N$, then the 
tensor product $M\otimes_{R_2}N$ is a sweet $(R_1,R_3)$-bimodule.

Let $S$ and $S'$ be two enhanced sign strings. Then $\widehat{B}_S^{S'}$ 
denotes the set of all webs whose boundary is divided 
into a lower part, determined by $S$, and an upper part, determined by $S'$. 
Here we mean one diagram when we say web, not a linear combination of 
diagrams. Let $B_S^{S'}\subset \widehat{B}_S^{S'}$ be 
the subset of non-elliptic webs. 

For any $w\in \widehat{B}_S^{S'}$, define a finite dimensional, graded 
$(K_{S'},K_{S})$-bimodule $\Gamma(w)$ by 
\[
\Gamma(w)=\bigoplus_{u\in B_{S'},v\in B_{S}}{}_u\Gamma(w)_v,
\]
with 
\[
{}_u\Gamma(w)_v={\mathcal F}^c(u^*wv)\{n\},
\] 
where $n$ is the length of $S'$. The left and right actions of $K_S$ on 
$\Gamma(w)$ are defined by applying the multiplication foam in~\ref{multfoam} 
to 
\[
{}_rK_u\otimes{}_u\Gamma(w)_v\to {}_r\Gamma(w)_v\quad\text{and}\quad 
{}_u\Gamma(w)_v\otimes {}_vK_r\to {}_u\Gamma(w)_r.
\]
Let $w\in\widehat{B}_S^{S'}$. Then $w=c_1w_1+\cdots +c_mw_m$, for 
certain $w_i\in B_S^{S'}$ and $c_i\in \mathbb{N}[q,q^{-1}]$. Since all 
relations which are satisfied by the Kuperberg bracket have categorical 
analogues for foams, this shows that 
\[
\Gamma(w)\cong c_1\Gamma(w_1)\oplus \cdots\oplus c_m\Gamma(w_m),
\]
where the multiplication by the $c_i$ is interpreted in the usual way using 
direct sums and grading shifts.  

We have the following analogue of Proposition 3 in~\cite{kh4}.
\begin{prop}\label{prop-sweet}
For any $w\in\widehat{B}_S^{S'}$, the graded $(K_{S'},K_{S})$-bimodule 
$\Gamma(w)$ is sweet. 
\end{prop} 
\begin{proof}
As a left $K_S$-module, we have 
\[
\Gamma(w)\cong \bigoplus_{v\in B_{S}} \Gamma(w)_v,
\]
where 
\[
\Gamma(w)_v=\bigoplus_{u\in B_{S'}}{}_u\Gamma(w)_v.
\]
So, as far as the left action is concerned, it suffices to show that 
$\Gamma(w)_v$ is a left projective $K_{S'}$-module. Note that, as a left 
$K_{S'}$-module, we have  
\[
\Gamma(w)_v\cong \bigoplus_{u\in B_{S'}}{\mathcal F}^0(wv).
\] 
Then $wv=c_1u_1+\cdots+c_1u_m$, for certain 
$u_i\in B_{S'}$ and $c_i\in\mathbb{N}[q,q^{-1}]$. By the remarks above, this 
means that 
\[
{\mathcal F}^0(wv)\cong c_1P_{u_1}\oplus \cdots\oplus c_mP_{u_m},
\]
which proves that $\Gamma(w)$ is projective as a left $K_{S'}$-module. 

The proof that $\Gamma(w)$ is projective as a right $K_{S}$-module is 
similar.   
\end{proof}

It is not hard to see that (see for example~\cite{kh4}), for any $w\in \widehat{B}_S^{S'}$ and 
$w'\in \widehat{B}_{S'}^{S''}$, we have 
\begin{equation}
\label{eq:sweettensor}
\Gamma(ww')\cong \Gamma(w)\otimes_{K_{S'}}\Gamma(w').
\end{equation}
\begin{lem} Let $w, w'\in \widehat{B}_{S}^{S'}$. An isotopy between 
$w$ and $w'$ induces an isomorphism between $\Gamma(w)$ and $\Gamma(w')$. 
Two isotopies between $w$ and $w'$ induce the same isomorphism if and only 
if they induce the same bijection between the connected components of $w$ 
and $w'$.  
\end{lem}

\begin{lem}
Let $w,w'\in \widehat{B}_{S}^{S'}$ and let $f\in \foamt^0(w,w')$ be a foam 
of degree $t$. Then $f$ induces a bimodule map 
\[
\Gamma(f)\colon \Gamma(w)\to \Gamma(w')
\] 
of degree $t$.
\end{lem}
\begin{proof}
Note that, for any $u\in B_{S'}$ and $v\in B_{S}$, the foam $f$ induces a 
linear map 
\[
{\mathcal F}^0(1_{u^*}f1_v)\colon {\mathcal F}^0(u^*wv)\to \F^0(u^*w'v),
\]
by glueing $1_{u^*}f1_v$ on top of any element in 
${\mathcal F}^0(u^*wv)=\foamt^0(\emptyset, u^*wv).$ 
This map has degree $t$, e.g. the identity has degree $0$ 
because the multiplication in $K_S$ is degree preserving. 
By taking the direct sum over all 
$u\in B_{S'}$ and $v\in B_{S}$, we get a linear map 
\[
\Gamma(f)\colon \Gamma(w)\to \Gamma(w').
\]
The shifts in the definition of $\Gamma(w)$ and $\Gamma(w')$, given by the 
length $n$ of $w$ and the length $m$ of $w'$, imply that 
$\deg \Gamma(f)=t$.
\vskip0.5cm 
The fact that $\Gamma(f)$ is a left $K_S$-module map follows from 
the following observation. For any $u\in B_S$ and $v\in B_{S'}$, 
the linear map ${\mathcal F}^0(1_{u^*}f1_v)$ corresponds to the linear map 
\[
\foamt^0(u,wv)\to \foamt^0(u,w'v)
\]
determined by horizontally composing with $f1_v$ on the right-hand side. 
This map clearly commutes with any composition on the left-hand side. 

Analogously, the linear map ${\mathcal F}^0(1_{u^*}f1_v)$ corresponds to 
the linear map  
\[
\foamt^0(w^*u,v)\to \foamt^0((w')^*u,v)
\]
determined by horizontally composing with $f^*1_u$ on the left-hand side. 
This map clearly commutes with any composition on the right-hand side. 

These two observations show that $\Gamma(f)$ is a $(K_{S'},K_{S})$-bimodule map. 
\end{proof}
It is not hard to see that, for any $f\in \foamt^0(w,w')$ and 
$g\in\foamt^0(w',w'')$, we have 
\[
\Gamma(fg)=\Gamma(f)\Gamma(g).
\]
Similarly, for any given $u_1,u_2\in \widehat{B}_{S}^{S'}$ and 
$u_1',u_2'\in \widehat{B}_{S'}^{S''}$ and any given 
$f\in\foamt^0(u_1,u_2)$ and $f'\in\foamt^0(u'_1,u'_2)$, we have 
a commuting square 
\[
\begin{xy}
  \xymatrix{
      \Gamma(u_1u_1') \ar[rr]^{\Gamma(f\circ f')} \ar[d]_{\cong}    & &   \Gamma(u_2u_2') \ar[d]^{\cong}  \\
      \Gamma(u_1)\otimes_{K_{S'}}\Gamma(u_1') \ar[rr]_{\Gamma(f)\otimes \Gamma(f')}             & &   \Gamma(u_2)\otimes_{K_{S'}}\Gamma(u_2')   
  }
\end{xy}
\]
where the vertical isomorphisms are as in~\eqref{eq:sweettensor}.

\subsubsection{The categorical $\Scat(n,n)$-action on ${\mathcal W}_{(3^k)}$}
We are now going to use sweet bimodules to define a categorical action of 
$\Scat(n,n)$ on $\mathcal{W}_{(3^k)}$ in the sense of Section~\ref{sec-techhigherrep}. For the definition of this action, 
we will consider $\Scat(n,n)$ to be a monoidal category 
rather than a 2-category. Like always, everything should be strict. The reader should compare this to Section~\ref{sec-techhigher}.  
\begin{defn}
\label{defn:cataction}
\textbf{On objects:} The categorical action of any object 
$\mathcal{E}_{\ii}1_{\lambda}$ in $\Scat(n,n)$ on $\mathcal{W}_{(3^k)}$ 
is defined by tensoring with the sweet bimodule (see Proposition~\ref{prop-sweet})
\[
\Gamma\left(\phi\left(E_{\ii}1_{\lambda}\right)\right).
\]
Recall that $\phi\colon S_q(n,n)\to \mathrm{End}_{\mathbb{C}(q)}(W_{(3^k)})$ 
was defined in Definition~\ref{defn:phi}.
\vskip0.5cm
\textbf{On morphisms:} We give a list of the foams associated to the 
generating morphisms of $\Scat(n,n)$. 
Applying $\Gamma$ to these foams determines 
the natural transformations associated to the morphisms of $\Scat(n,n)$. 

As before, we only draw the most important part of the foams, 
omitting partial identity foams. Note our conventions.
\begin{itemize}
\item[(1)]We read the regions of the morphisms in 
$\Scat(n,n)$ from right to left and the morphisms themselves from bottom to 
top.
\item[(2)]The corresponding foams we read from bottom to top and from front to 
back.
\item[(3)]Vertical front edges labeled $1$ are assumed to be oriented upward and 
vertical front edges labeled $2$ are assumed to 
be oriented downward.
\item[(4)]The convention for the orientation of the back edges is 
precisely the opposite.
\item[(5)]A facet is labeled 0 or 3 if and only if its boundary has edges 
labeled 0 or 3. 
\end{itemize}
In the list below, we always assume that $i<j$. 
Finally, all facets labeled 0 or 3 in the images below have to be erased, 
in order to get real foams. For any $\lambda>(3^k)$, the image of 
the elementary morphisms below is taken to be zero, by convention.

{\allowdisplaybreaks
\begin{align*}
\xy
(4,1.5)*{\includegraphics[width=9px]{res/figs/section23/upsimpledot}};
(7,-4)*{{\scriptstyle i,\lambda}};
\endxy &\mapsto 
\;\xy
(0,0)*{\includegraphics[width=60px]{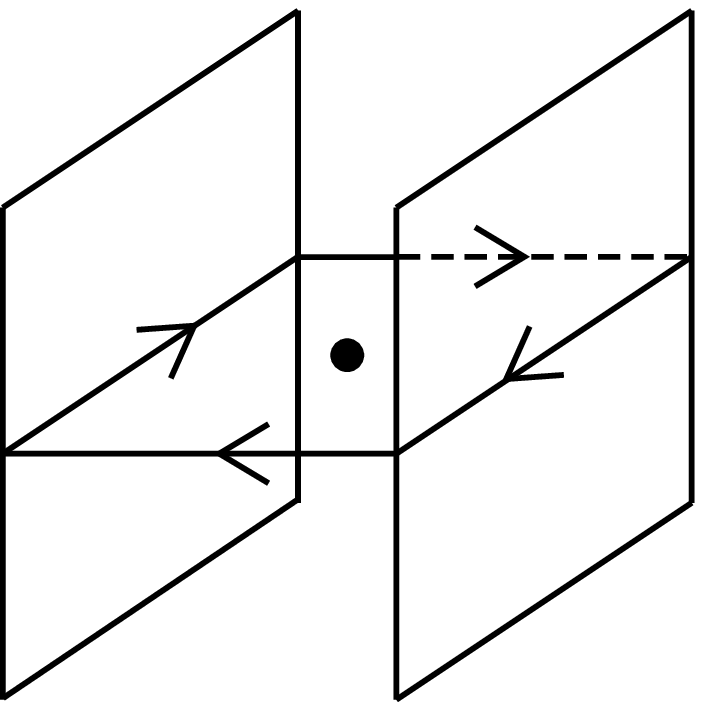}};
(-8,-12)*{\scriptstyle\lambda_i};
(6,-12)*{\scriptstyle\lambda_{i+1}};
\endxy
\\
\xy
(4,1.5)*{\includegraphics[width=9px]{res/figs/section23/downsimpledot}};
(7,-4)*{{\scriptstyle i,\lambda}};
\endxy &\mapsto 
\;\xy
(0,0)*{\includegraphics[width=60px]{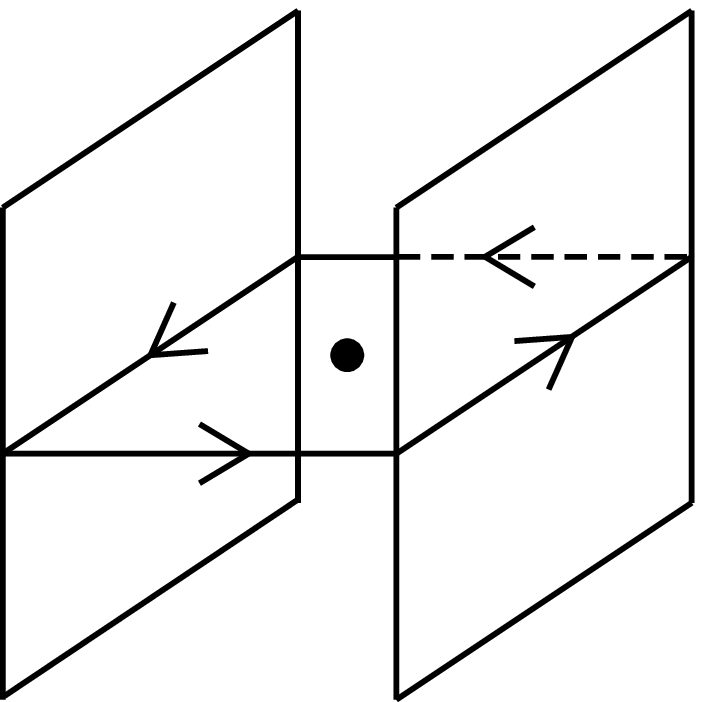}};
(-8,-12)*{\scriptstyle\lambda_i};
(6,-12)*{\scriptstyle\lambda_{i+1}};
\endxy
\\ \xy
(6,1.5)*{\includegraphics[width=20px]{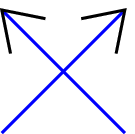}};
(12.5,-4)*{{\scriptstyle i,i,\lambda}};
\endxy  &\mapsto
\;\;-\;\;
\xy
(0,0)*{\includegraphics[width=60px]{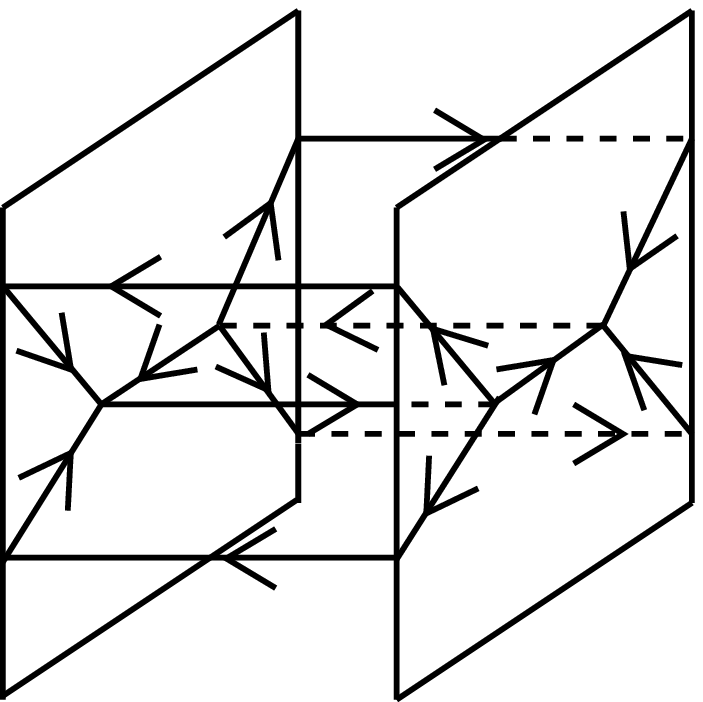}};
(-8,-12)*{\scriptstyle\lambda_i};
(6,-12)*{\scriptstyle\lambda_{i+1}};
\endxy
\\
\xy
(6,1.5)*{\includegraphics[width=20px]{res/figs/section23/upcross}};
(12.5,-4)*{{\scriptstyle i,i+1,\lambda}};
\endxy &\mapsto
\;\;(-1)^{\lambda_{i+1}}\;\;
\xy
(0,0)*{\includegraphics[width=100px]{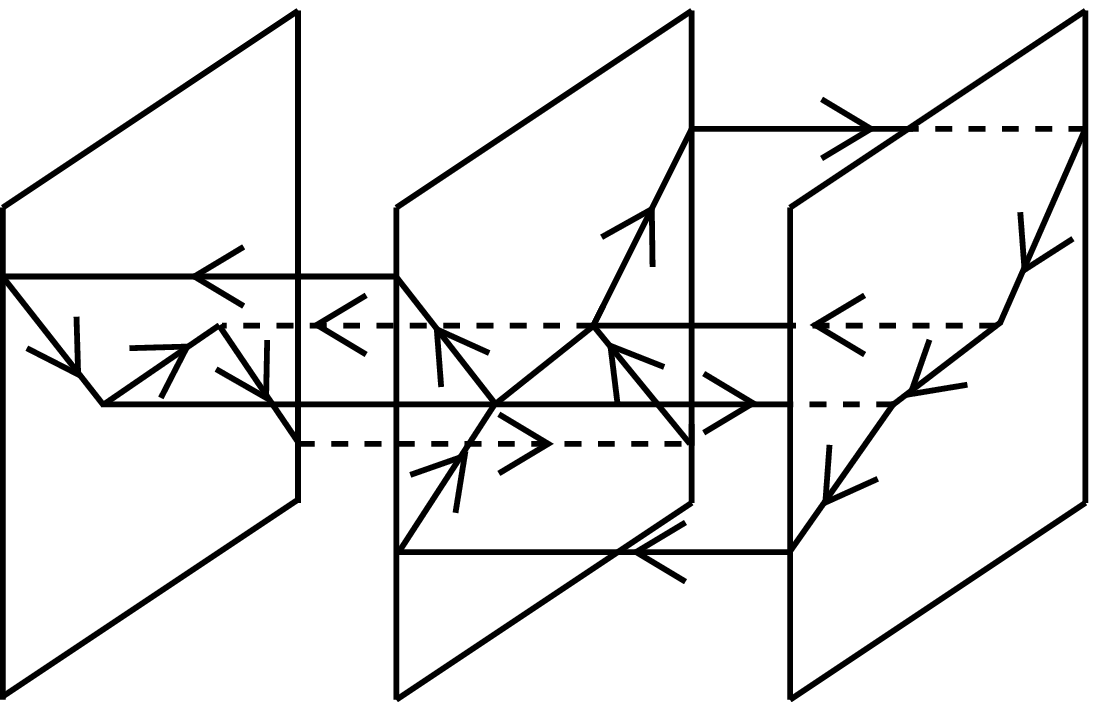}};
(-14,-12)*{\scriptstyle\lambda_i};
(0,-12)*{\scriptstyle\lambda_{i+1}};
(13,-12)*{\scriptstyle\lambda_{i+2}};
\endxy
\\
\xy
(6,1.5)*{\includegraphics[width=20px]{res/figs/section23/upcross}};
(12.5,-4)*{{\scriptstyle i+1,i,\lambda}};
\endxy &\mapsto
\;
\xy
(0,0)*{\includegraphics[width=100px]{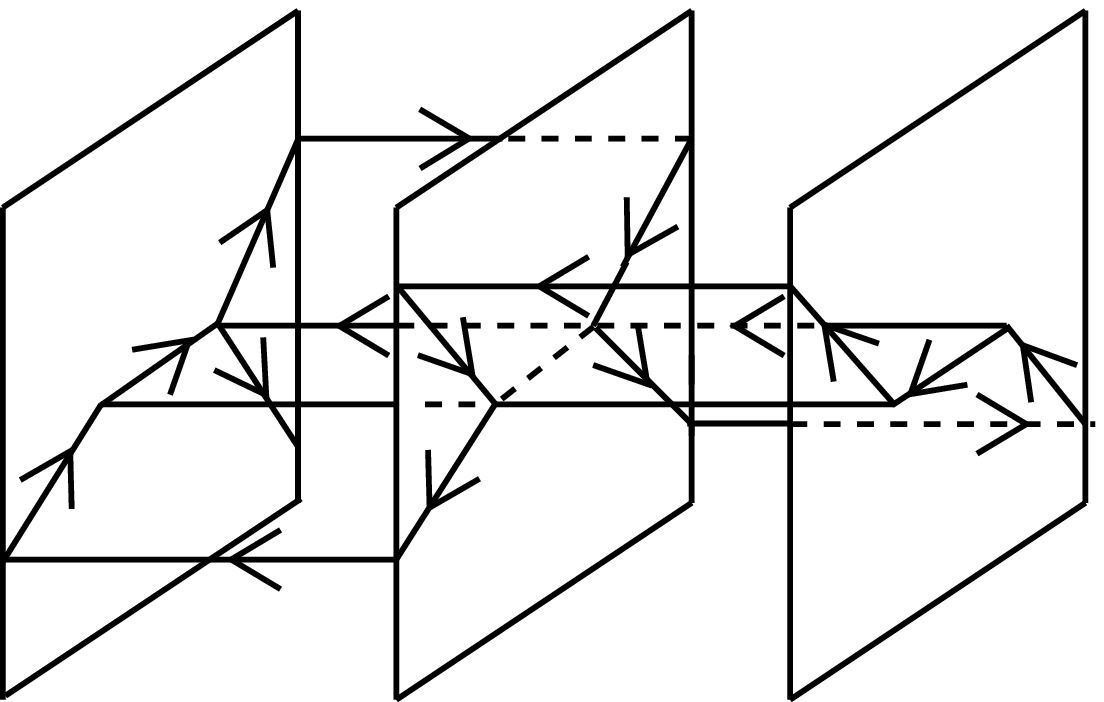}};
(-14,-12)*{\scriptstyle\lambda_i};
(0,-12)*{\scriptstyle\lambda_{i+1}};
(13,-12)*{\scriptstyle\lambda_{i+2}};
\endxy
\\
\xy
(6,1.5)*{\includegraphics[width=20px]{res/figs/section23/upcross}};
(12.5,-4)*{{\scriptstyle i,j,\lambda}};
\endxy &\mapsto
\;
\xy
(0,0)*{\includegraphics[width=130px]{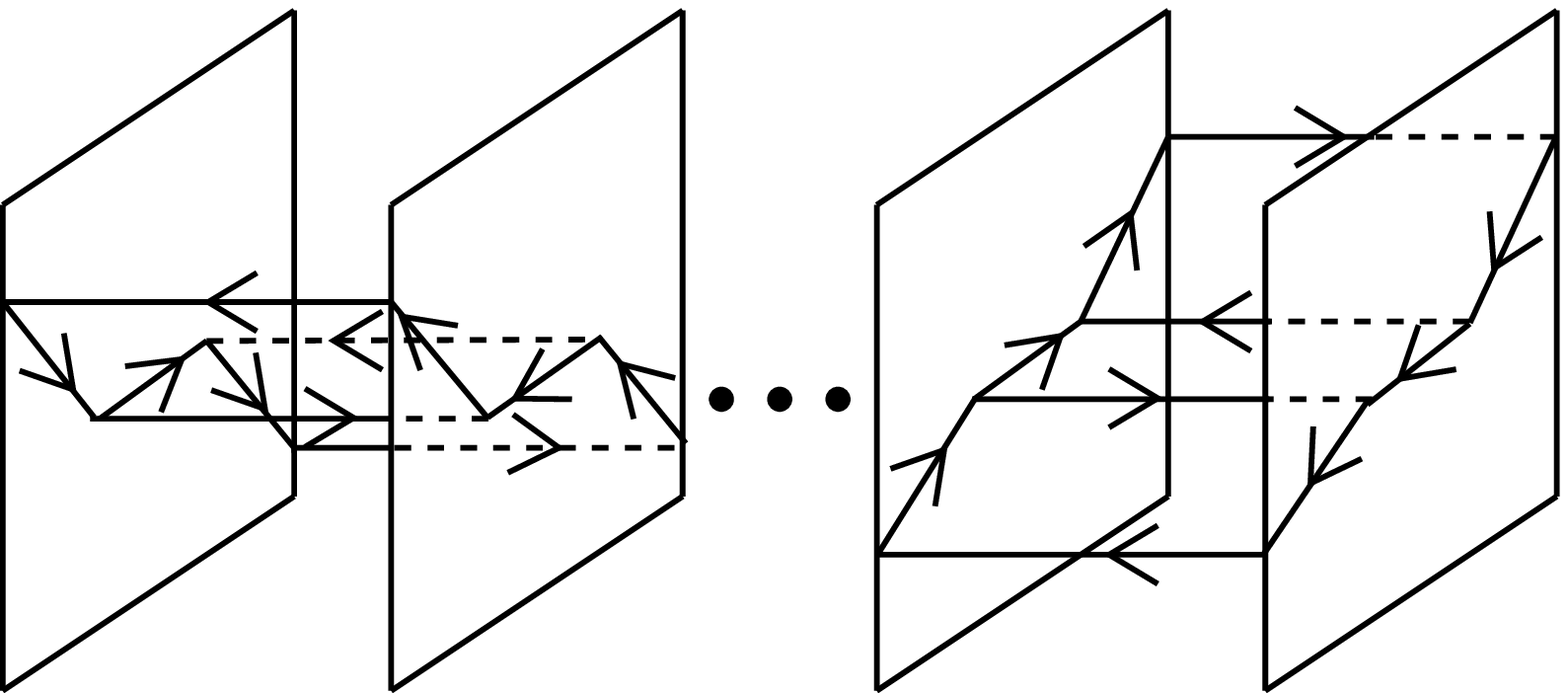}};
(-22,-12)*{\scriptstyle\lambda_i};
(-7,-12)*{\scriptstyle\lambda_{i+1}};
(4,-12)*{\scriptstyle\lambda_{j}};
(19,-12)*{\scriptstyle\lambda_{j+1}};
\endxy
\\
\xy
(6,1.5)*{\includegraphics[width=20px]{res/figs/section23/upcross}};
(12.5,-4)*{{\scriptstyle j,i,\lambda}};
\endxy &\mapsto
\;
\xy
(0,0)*{\includegraphics[width=130px]{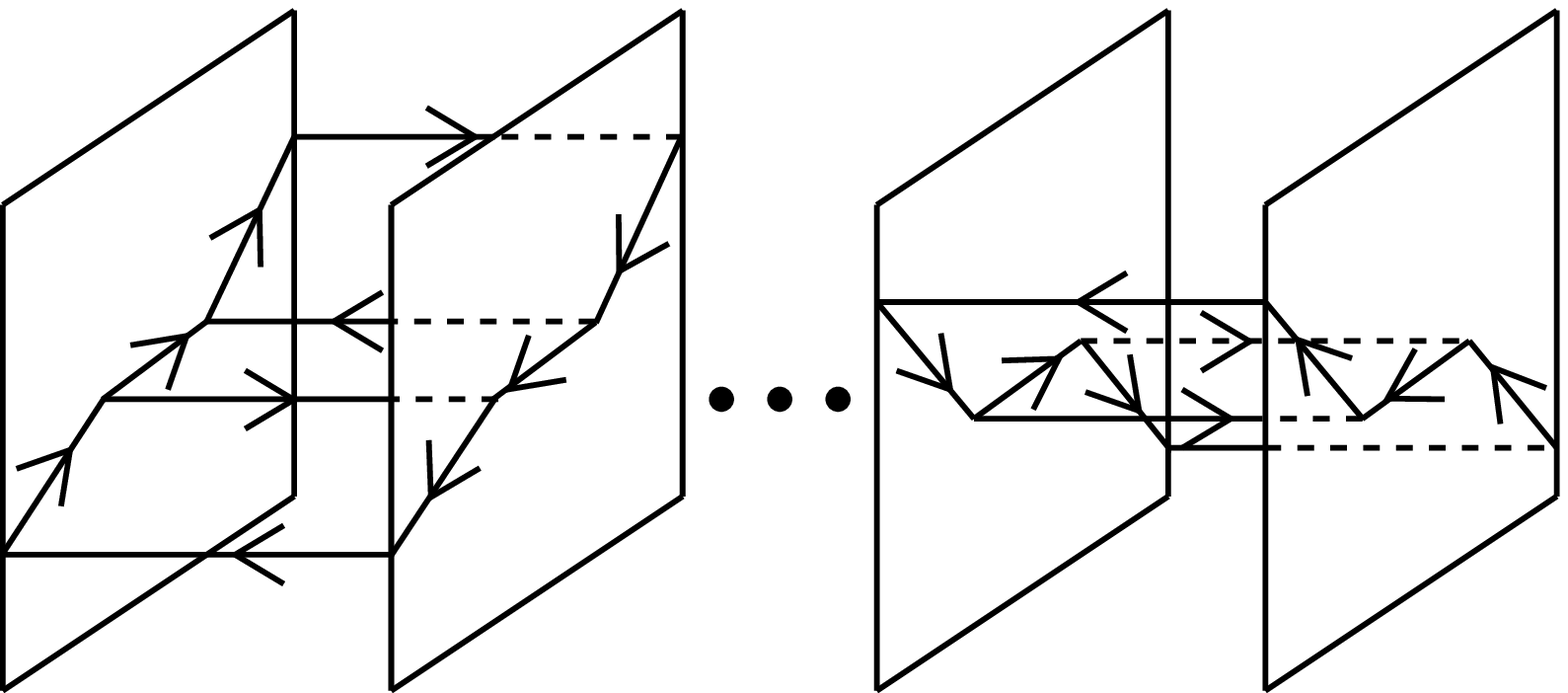}};
(-22,-12)*{\scriptstyle\lambda_i};
(-7,-12)*{\scriptstyle\lambda_{i+1}};
(4,-12)*{\scriptstyle\lambda_{j}};
(19,-12)*{\scriptstyle\lambda_{j+1}};
\endxy
\\
\xy
(8,0)*{\includegraphics[width=25px]{res/figs/section23/rightcup}};
(12.5,-4)*{{\scriptstyle i,\lambda}};
\endxy &\mapsto
\;\xy
(0,0)*{\includegraphics[width=60px]{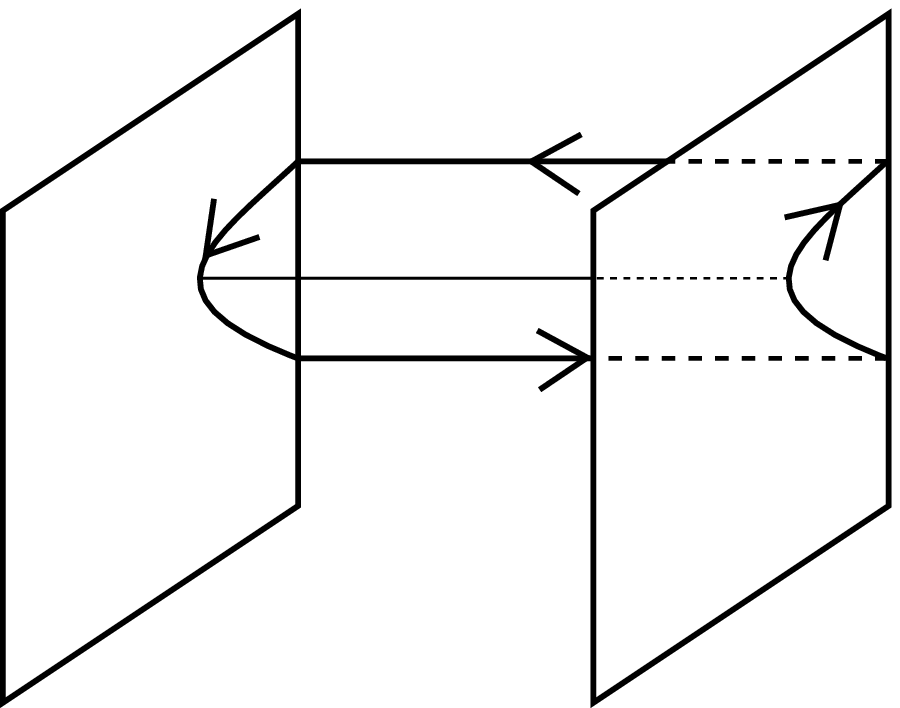}};
(-8,-10)*{\scriptstyle\lambda_i};
(8,-10)*{\scriptstyle\lambda_{i+1}};
\endxy
\\
\xy
(8,0)*{\includegraphics[width=25px]{res/figs/section23/leftcup}};
(12.5,-4)*{{\scriptstyle i,\lambda}};
\endxy &\mapsto
(-1)^{\lfloor\frac{\lambda_i}{2}\rfloor+\lceil\frac{\lambda_{i+1}}{2}\rceil}\;\xy
(0,0)*{\includegraphics[width=60px]{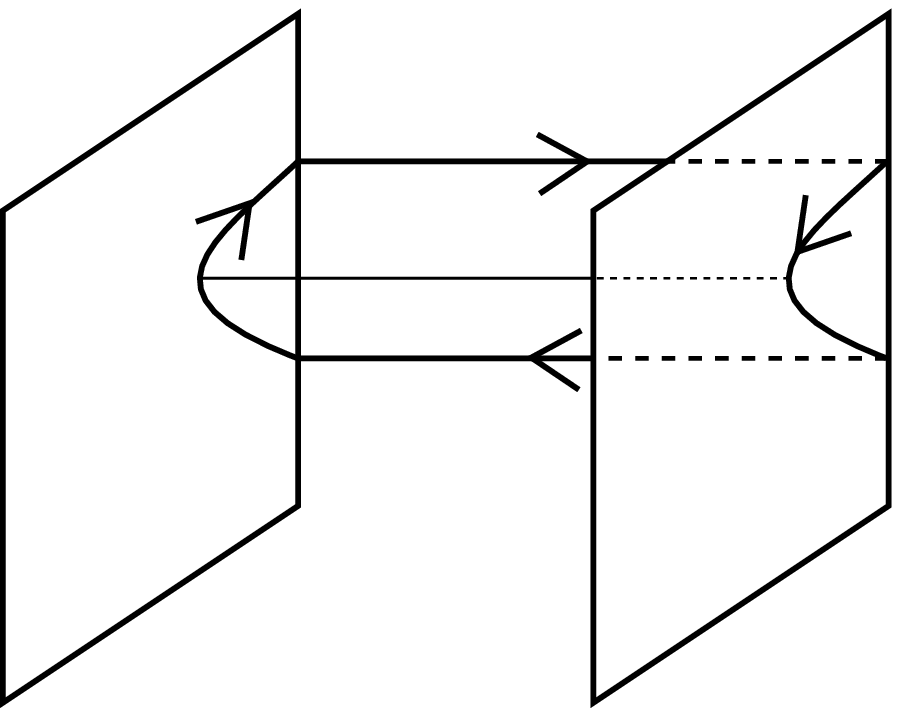}};
(-8,-10)*{\scriptstyle\lambda_i};
(8,-10)*{\scriptstyle\lambda_{i+1}};
\endxy
\end{align*}
\begin{align*}
\xy
(8,0)*{\includegraphics[width=25px]{res/figs/section23/leftcap}};
(12.5,-4)*{{\scriptstyle i,\lambda}};
\endxy &\mapsto 
(-1)^{\lceil\frac{\lambda_i}{2}\rceil+\lfloor\frac{\lambda_{i+1}}{2}\rfloor}\;\xy
(0,0)*{\includegraphics[width=60px]{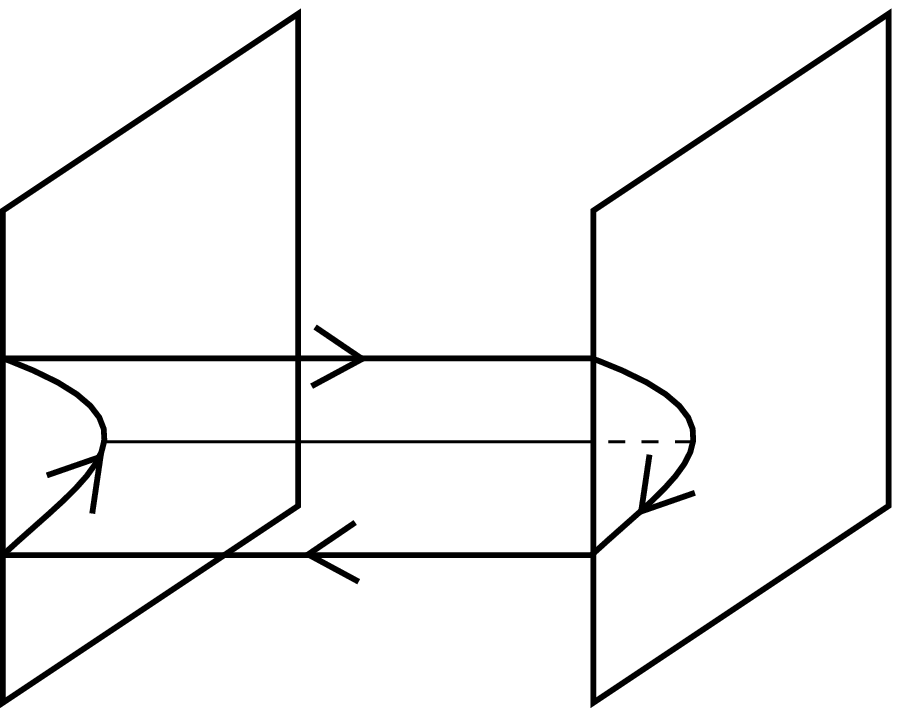}};
(-8,-10)*{\scriptstyle\lambda_i};
(8,-10)*{\scriptstyle\lambda_{i+1}};
\endxy
\\
\xy
(8,0)*{\includegraphics[width=25px]{res/figs/section23/rightcap}};
(12.5,-4)*{{\scriptstyle i,\lambda}};
\endxy &\mapsto
\;\xy
(0,0)*{\includegraphics[width=60px]{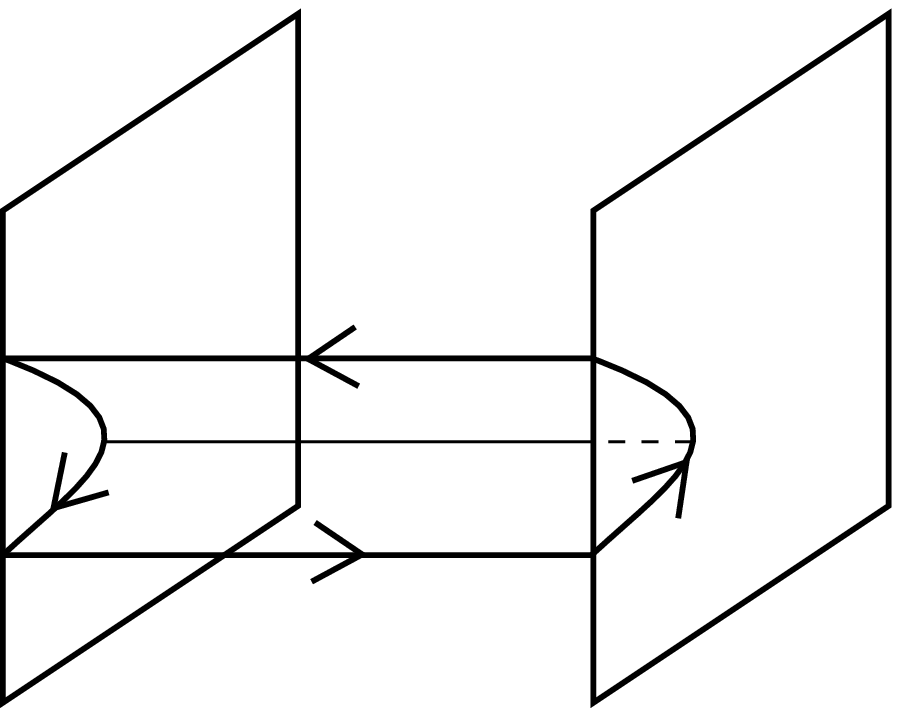}};
(-8,-10)*{\scriptstyle\lambda_i};
(8,-10)*{\scriptstyle\lambda_{i+1}};
\endxy
\end{align*} 
}
\end{defn}  

\begin{prop} 
\label{prop:cataction}
The formulas in Definition~\ref{defn:cataction} 
determine a well-defined graded categorical action of 
$\Scat(n,n)$ on ${\mathcal W}_{(3^k)}$ in the sense of Section~\ref{sec-techhigherrep}. 
\end{prop}
\begin{proof}
A tedious but straightforward case by case check, for each 
generating morphism and each $\lambda$ which give a non-zero foam, 
shows that each of the foams in 
Definition~\ref{defn:cataction} has the 
same degree as the elementary morphism in $\Scat(n,n)$ to which it is 
associated. Note that it is important to erase the facets labeled 0 or 3, 
before computing the degree of the foams. We do just one example here. We 
have 
\begin{align*}
\xy
(8,0)*{\includegraphics[width=25px]{res/figs/section23/leftcup}};
(12.5,-4)*{{\scriptstyle i,(12)}}
\endxy &\mapsto
-\;\xy
(0,0)*{\includegraphics[width=80px]{res/figs/section52/lcupfoam}};
(-12,-12)*{\scriptstyle 1};
(7,-12)*{\scriptstyle 2};
(-5.5,2.5)*{\scriptstyle 0};
(12.5,2.5)*{\scriptstyle 3};
\endxy=f
&\text{and}\quad \deg(\xy
(8,0)*{\includegraphics[width=25px]{res/figs/section23/leftcup}};
(12.5,-4)*{{\scriptstyle i,(12)}}
\endxy)=2.
\end{align*}
We see that $f$ has one facet labeled 0 and another 
labeled 3, so those two facets have to be erased. Therefore, $f$ has 
12 vertices, 14 edges and 3 faces, i.e.
\[
\chi(f)=12-14+3=1.
\]
The boundary of $f$ has 12 vertices and 12 edges, so 
\[
\chi(\partial f)=12-12=0.
\]
Note that the two circular edges do not belong to $\partial f$, because 
the circular facets have been removed. 
In this section we draw the foams horizontally, so 
$b$ is the number of horizontal edges at the top and the bottom of $f$, 
which go from the front to the back. Thus, for $f$ we have 
\[
b=4.
\] 
Altogether, we get 
\[
q(f)=0-2+4=2.
\]
\vskip0.5cm
In order to show that the categorical action is well-defined, one has to 
check that it preserves all the relations in Definition~\ref{def_glcat}. 
Modulo 2 this was done in the proof of Theorem 4.2 in~\cite{mack}. At the 
time there was a small issue about the signs in~\cite{kl5}, which prevented 
the author to formulate and prove Theorem 4.2 in~\cite{mack} over 
$\mathbb{C}$. That issue has now been solved (see~\cite{msv1} and~\cite{kl4} 
for more information) and in this thesis we use the sign conventions 
from~\cite{msv1}, which are compatible with those from~\cite{kl4}.  
We laboriously checked all these relations again, but now over $\mathbb{C}$ 
and with the signs above. The arguments are exactly the same, 
so let us not repeat them one by one here. Instead, we first explain 
how we computed the signs for the categorical 
action above and why they give the desired result over $\mathbb{C}$. 
After that, we will do an example. 
For a complete case by case check, we refer to the arguments used in the 
proof of Theorem 4.2 in~\cite{mack}. The reader should check that our signs above 
remove the sign ambiguities in that proof. 

One can compute the signs above as follows. First check the relations 
only involving strands of one colour, i.e. the $\mathfrak{sl}_2$-relations. 
The first thing to notice is that 
the foams in the categorical action do not satisfy 
relation~\eqref{eq_nil_dotslide}; 
for all $\lambda$, which give a non-zero foam, the sign is wrong. 
Therefore, one is forced to multiply the foam associated to 
\[
\xy
(6,1.5)*{\includegraphics[width=20px]{res/figs/section52/upcrossblue}};
(12.5,-4)*{{\scriptstyle i,i,\lambda}};
\endxy
\]
by $-1$, for all $\lambda$. 

After that, compute the foams 
associated to the degree zero bubbles (real bubbles, not fake bubbles) and 
adjust the signs of the images of the left cups and caps accordingly. 
This way, most of the signs of the images of the 
left cups and caps get determined. The remaining ones can be determined by 
imposing the zig-zag relations in~\eqref{eq_biadjoint1} 
and~\eqref{eq_biadjoint2}. 
  
Of course, one could also choose to adjust the signs of the images of 
the right cups and caps. That would determine a categorical 
action that is naturally isomorphic to the one in this thesis. 
  
After these signs have been determined, one can check that all 
$\mathfrak{sl}_2$-relations are preserved by the categorical action. 

The next and final step consists in determining the signs of
\[
\xy
(6,1.5)*{\includegraphics[width=20px]{res/figs/section23/upcross}};
(12.5,-4)*{{\scriptstyle i,j,\lambda}};
\endxy,
\]
for $i\ne j$. First one can check that cyclicity is already preserved. 
The relations in~\eqref{eq_cyclic_cross-gen} are preserved by the corresponding 
foams, which are all isotopic, with our sign choices for the foams 
associated to the left cups and caps. Therefore, cyclicity does not 
determine any more signs. 

The relations in~\eqref{eq_downup_ij-gen} are preserved on the nose, for 
$i=j$ and $|i-j|>1$. For $|i-j|=1$, they are only preserved up to a sign. 
Note that, since the corresponding foams are all isotopic, the signs 
actually come from the sign choice for the foams associated to the 
left cups and caps. Thus, whenever the total sign 
in the image of~\eqref{eq_downup_ij-gen} becomes negative, one has 
to change the sign of one of the two crossings (not of both of course). 
Our choice has been to change the sign of the foam associated to 
\[
\xy
(6,1.5)*{\includegraphics[width=20px]{res/figs/section23/upcross}};
(12.5,-4)*{{\scriptstyle i,i+1,\lambda}};
\endxy,
\]
whenever necessary. Any other choice, consistent with all the previous 
sign choices, leads to a naturally isomorphic categorical action. 
It turns out that the sign has to be equal to $(-1)^{\lambda_{i+1}}$, after 
checking for all $\lambda$. 

After this, one can check that all relations involving two or three colours 
are preserved by the categorical action. Note that we have not specified an 
image for the fake bubbles. As stressed repeatedly in~\cite{kl5}, fake bubbles 
do not exist as separate entities. They are merely formal symbols, used as 
computational devices to keep the computations involving real bubbles 
tidy and short. As we are using $\mathfrak{sl}_3$-foams in this thesis, 
most of the dotted bubbles are mapped to zero. Therefore, under the 
categorical action it is very easy to convert the fake bubbles 
in the relations in Definition~\ref{def_glcat} 
into linear combinations of real bubbles, using the 
infinite Grassmannian relation~\eqref{eq_infinite_Grass}. 
Thus, there is no need to use fake bubbles in this thesis.  
\vskip0.5cm
Finally, let us do two examples; one involving only one colour and another 
involving two colours. 

The left side of the equation in~\eqref{eq:EF}, for $i=1$ and 
$\lambda=(1,2)$ (the other entries are omitted for simplicity), becomes
\begin{align*}
\xy
(0,0)*{\includegraphics[width=120px]{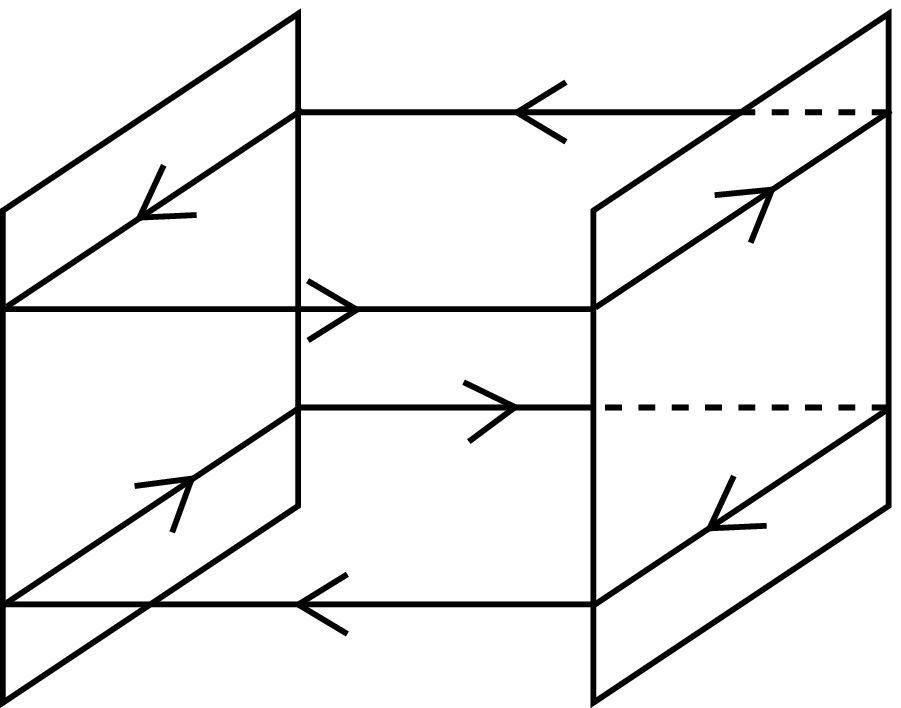}};
(-23,-18)*{\scriptstyle 1};
(5,-18)*{\scriptstyle 2};
(-23,-4)*{\scriptstyle 2};
(5,-4)*{\scriptstyle 1};
(-23,7)*{\scriptstyle 1};
(5,7)*{\scriptstyle 2};
\endxy
\;=\;
-\;\xy
(0,0)*{\includegraphics[width=120px]{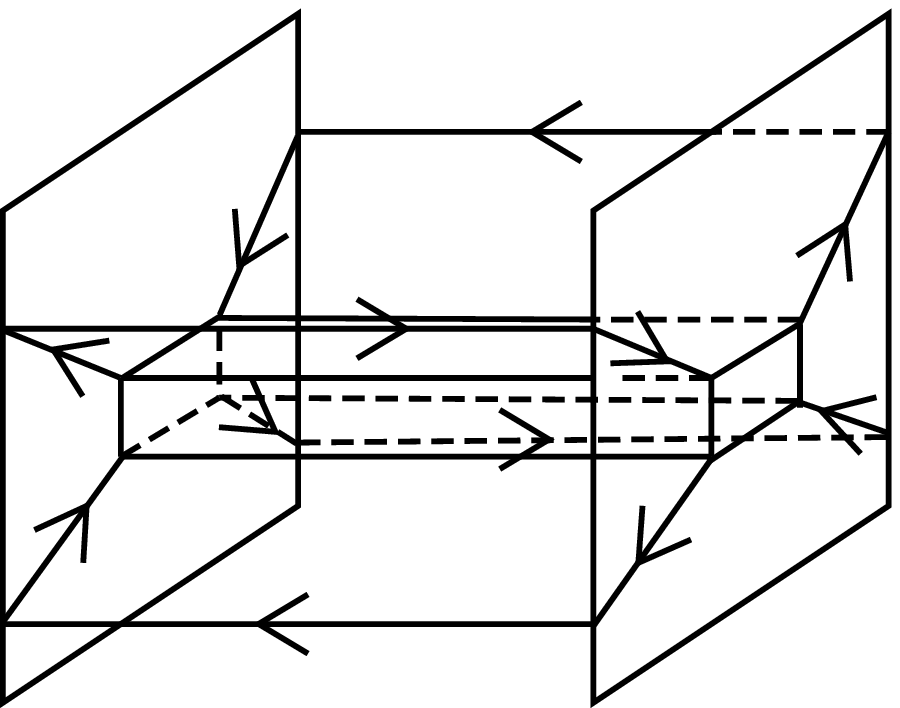}};
(-23,-18)*{\scriptstyle 1};
(5,-18)*{\scriptstyle 2};
(-19,-6)*{\scriptstyle 2};
(8,-7)*{\scriptstyle 1};
(-23,7)*{\scriptstyle 1};
(5,7)*{\scriptstyle 2};
(-14,-2.3)*{\scriptstyle 3};
(14,-2)*{\scriptstyle 0};
(-8.5,3)*{\scriptstyle 2};
(19,3)*{\scriptstyle 1};
\endxy
\;-\;
\xy
(0,0)*{\includegraphics[width=120px]{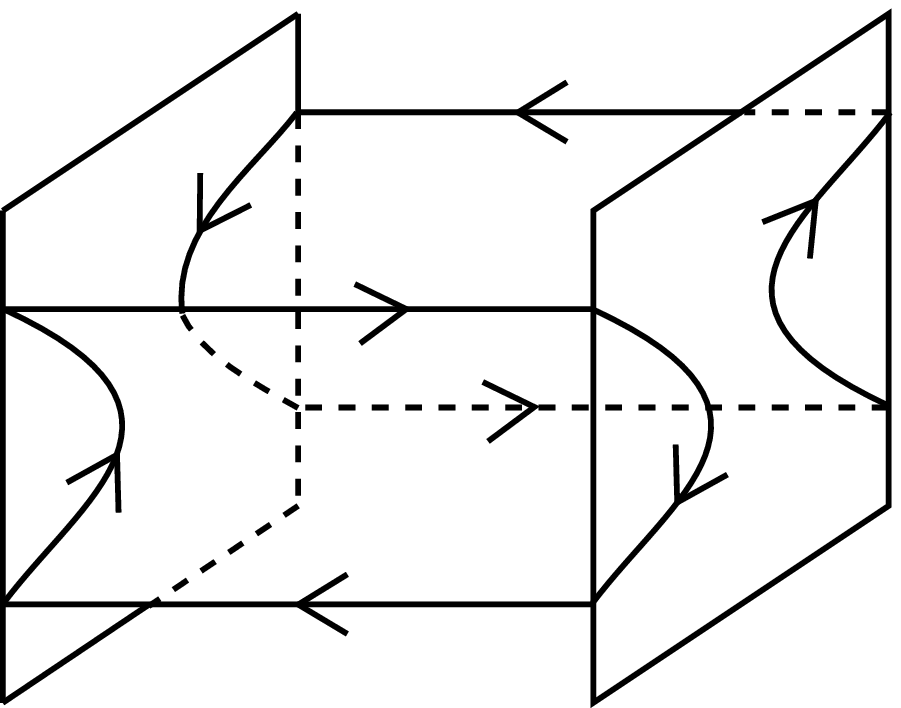}};
(-23,-18)*{\scriptstyle 1};
(5,-18)*{\scriptstyle 2};
(-19,-6)*{\scriptstyle 2};
(8,-7)*{\scriptstyle 1};
(-23,7)*{\scriptstyle 1};
(5,7)*{\scriptstyle 2};
(-8.5,4)*{\scriptstyle 2};
(19,4)*{\scriptstyle 1};
\endxy.
\end{align*}  
This foam equation is precisely the relation (SqR). 
Note that the signs match perfectly, because we have 
\[
\mathrm{sign}\left(\xy
(8,0)*{\includegraphics[width=25px]{res/figs/section23/leftcap}};
(12.5,-4)*{{\scriptstyle i,(12)}}
\endxy\right)=+\quad\text{and}\quad 
\mathrm{sign}\left(\xy
(8,0)*{\includegraphics[width=25px]{res/figs/section23/leftcup}};
(12.5,-4)*{{\scriptstyle i,(12)}}
\endxy\right)=-.
\]
\vskip0.5cm
The equation in~\eqref{eq_r2_ij-gen}, for $(i,j)=(1,2)$ and 
$\lambda=(121)$ (the other entries are omitted for simplicity), becomes
\begin{align*}
\xy
(0,0)*{\includegraphics[width=120px]{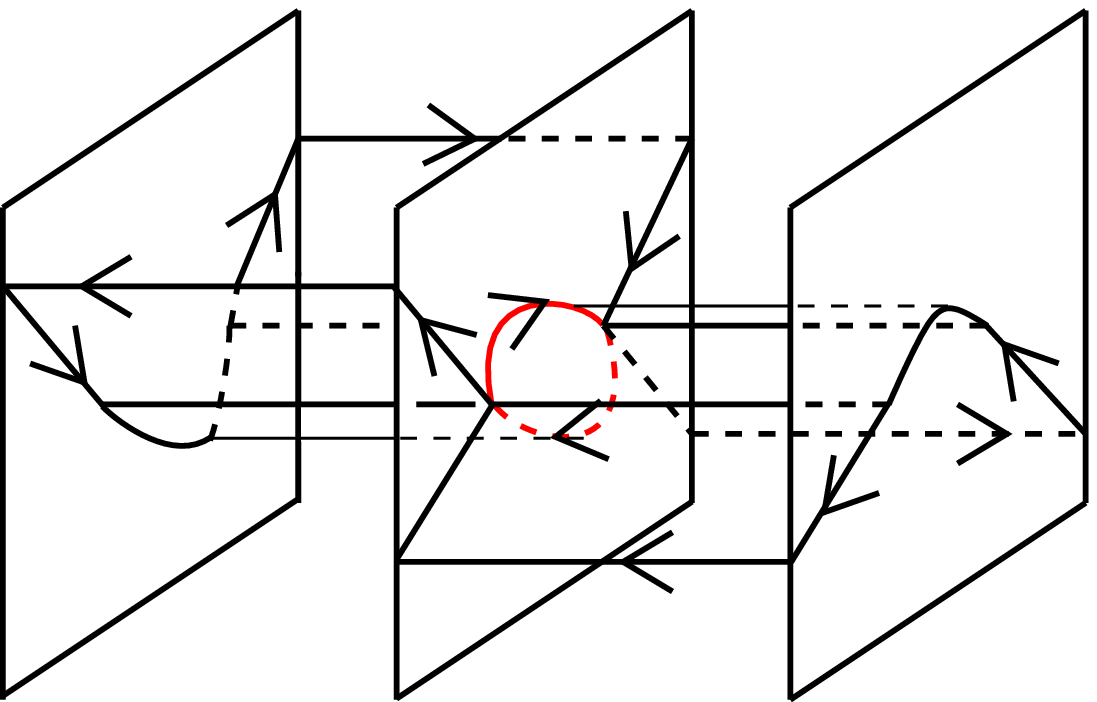}};
(-17,-13)*{\scriptstyle 1};
(-2,-13)*{\scriptstyle 2};
(13,-13)*{\scriptstyle 1};
(-13,13)*{\scriptstyle 2};
(2,13)*{\scriptstyle 2};
(17,13)*{\scriptstyle 0};
(0.5,-0.5)*{\scriptstyle 1};
(-4.5,-3.5)*{\scriptstyle 3};
(4.5,3)*{\scriptstyle 3};
\endxy
\;=\;
\xy
(0,0)*{\includegraphics[width=120px]{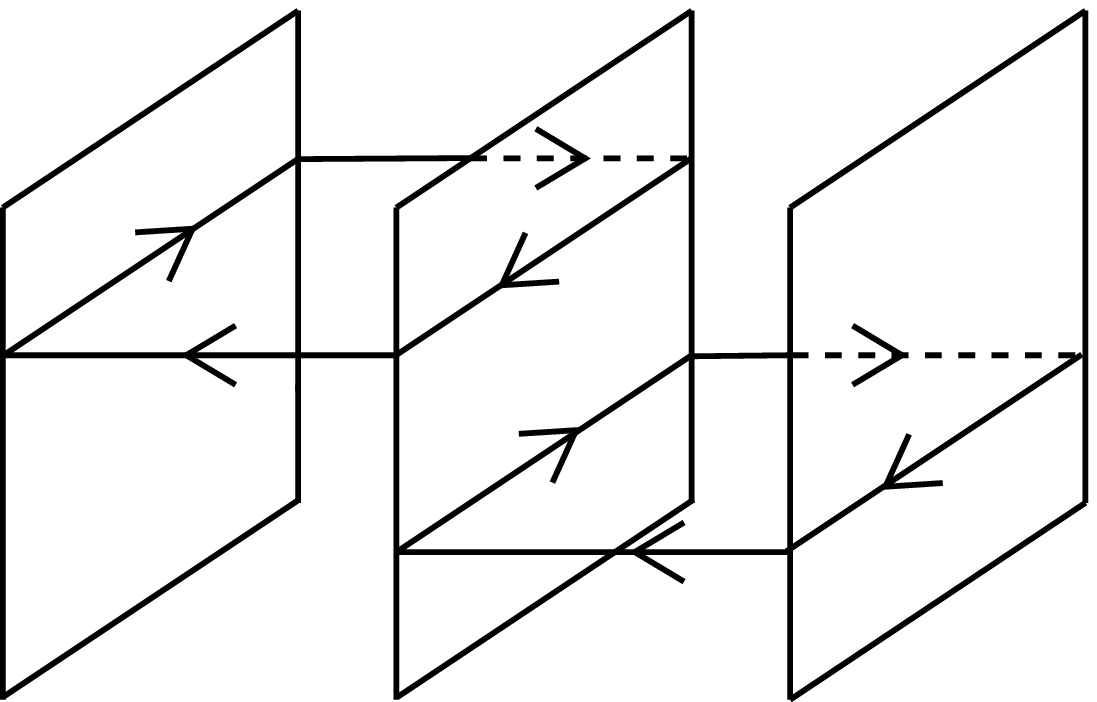}};
(-17,-13)*{\scriptstyle 1};
(-2,-13)*{\scriptstyle 2};
(13,-13)*{\scriptstyle 1};
(-13,13)*{\scriptstyle 2};
(2,13)*{\scriptstyle 2};
(17,13)*{\scriptstyle 0};
(0.5,-0.5)*{\scriptstyle 3};
(7.5,-4)*{\bullet};
\endxy
\;-\;
\xy
(0,0)*{\includegraphics[width=120px]{res/figs/section52/R2ijrhsfoam}};
(-17,-13)*{\scriptstyle 1};
(-2,-13)*{\scriptstyle 2};
(13,-13)*{\scriptstyle 1};
(-13,13)*{\scriptstyle 2};
(2,13)*{\scriptstyle 2};
(17,13)*{\scriptstyle 0};
(0.5,-0.5)*{\scriptstyle 3};
(-7.5,4)*{\bullet};
\endxy.
\end{align*} 
To see that this holds, apply the (RD) relation to the foam on the l.h.s., 
in order to remove the disc bounded by the red singular circle on the middle 
sheet. 
\end{proof}

Let $W_h\cong \mathbb{C}$ be the unique indecomposable, projective, graded 
$K_{(\times^k,\circ^{2k})}$-module of degree zero. 
Recall that $K_{(\times^k,\circ^{2k})}$ is 
generated by the empty diagram, so $W_h$ is indeed one-dimensional. It is the 
categorification of $w_h$, the highest weight vector in $W_{(3^k)}$. 

Note that we can pull back the categorical action on ${\mathcal W}_{(3^k)}$ 
via 
\[
\Psi_{n,n}\colon \Ucat\to \Scat(n,n).
\]
We are now able to prove one of our main results. Recall that $\mathcal V$ is any additive, idempotent complete category, which allows an integrable, graded categorical action by $\Ucat$ in the sense of Section~\ref{sec-techhigherrep}.
\begin{thm}
\label{thm:equivalence}
There exists an equivalence of categorical $\Ucat$-representations   
\[
\Phi\colon {\mathcal V}_{(3^k)}\to {\mathcal W}_{(3^k)}.
\]
\end{thm}   
\begin{proof}
As we already mentioned above, we have  
\[
\mathrm{End}_{{\mathcal W}_{(3^k)}}(W_h)\cong \mathbb{C}.
\]
Let $Q$ be any indecomposable object in ${\mathcal W}_{(3^k)}$. 
There exists an enhanced sign string $S$ such that $Q$ belongs to 
$K_S\text{-}\mathrm{p\textbf{Mod}}_{\mathrm{gr}}$. Therefore, there exists a 
basis web $w\in B_S$ and a $t\in\mathbb{Z}$, such that $Q$ is a graded 
direct summand of $P_w\{t\}$. Without loss of generality, we may assume that 
$t=0$.   

By Lemma~\ref{lem:phisurj} and Proposition~\ref{prop:cataction}, 
there exists an object of 
$X$ in $\Scat(n,n)$ such that $Q$ is a direct summand of $XW_h$. 
This holds, because in $\ScatD(n,n)$, the Karoubi envelope of $\Scat(n,n)$, 
the divided powers correspond to direct summands of ordinary powers. 
For more details on the categorification of the divided powers 
see~\cite{kl5} and~\cite{klms}.

Proposition~\ref{prop:rouquier} now proves the existence of $\Phi$.  
\end{proof}

An easy consequence of Theorem~\ref{thm:equivalence} is the following.
\begin{cor}
\label{cor:equivalence}
By Theorem~\ref{thm:equivalence}, the $S_q^{\mathbb{Z}}(n,n)$-module map 
\[
K^{\oplus}_0(\Phi)\colon K^{\oplus}_0({\mathcal V}_{(3^k)})\to K^{\oplus}_0({\mathcal W}_{(3^k)})
\]
is an isomorphism. 
\end{cor}

The following consequence of Theorem~\ref{thm:equivalence} is very important 
and we thank Ben Webster for explaining its proof. 
\begin{prop}
\label{prop:morita} 
The graded algebras $K_{(3^k)}$ and $R_{(3^k)}$ are (graded) Morita equivalent. 
\end{prop}
\begin{proof}
We are going to show that, for each weight $\mu_S$ 
which shows up in the weight decomposition of $V_{(3^k)}$, the graded algebras 
$K_S$ and $R(\mu_S-\lambda,\lambda)$ are Morita equivalent. This proves 
the proposition after taking direct sums.

Let $\mu_S$ be a weight 
which shows up in the weight decomposition of $V_{(3^k)}$. Define 
\[
\Theta_{\mu_S}=\bigoplus_{\ii\in\mathrm{Seq}(\mu_S-\lambda)}
\mathcal{E}_{\ii}W_h \in K_S\text{-}\mathrm{p\textbf{Mod}}_{\mathrm{gr}}.
\] 
In the proof of Theorem~\ref{thm:equivalence}, we already showed that 
every object in 
$K_S\text{-}\mathrm{p\textbf{Mod}}_{\mathrm{gr}}$ is a direct summand 
of $XW_h$ for some object $X\in\Scat(n,n)$. By the biadjointness of 
the $\mathcal{E}_i$ and $\mathcal{E}_{-i}$ in $\Scat(n,n)$ and the fact that 
$W_h$ is a highest weight object, it is not hard to see that 
$XW_h$ itself is a direct summand of a finite direct sum of 
degree shifted copies of 
$\Theta_{\mu_S}$. This shows that 
every object in $K_S\text{-}\mathrm{p\textbf{Mod}}_{\mathrm{gr}}$ is a direct summand 
of a finite direct sum of degree shifted copies of $\Theta_{\mu_S}$. Since $K_S$ is a finite dimensional, complex algebra, every 
finite dimensional, graded $K_S$-module has a projective cover and is therefore 
a quotient of a finite direct sum of degree-shifted copies of $\Theta_{\mu_S}$. 
This shows that $\Theta_{\mu_S}$ is a projective generator of 
$K_S\text{-}\mathrm{\textbf{Mod}}_{\mathrm{gr}}$. 

Theorem~\ref{thm:equivalence} also shows that 
\[
\mathrm{End}_{K_S}\left(\Theta_{\mu_S}\right)\cong R(\mu_S-\lambda,\lambda)
\]
holds. 

By a general result due to Morita, it follows that 
the above observations imply that 
$K_S$ and $R(\mu_S-\lambda,\lambda)$ are Morita 
equivalent. For a proof see Theorem 5.55 in~\cite{rot}, for example.
\end{proof}

We can draw two interesting conclusions from Proposition~\ref{prop:morita}.

In~\cite{bk2}, Brundan and Kleshchev defined an explicit isomorphism between 
blocks of cyclotomic Hecke algebras and cyclotomic KLR-algebras. 
Theorem 3.2 in~\cite{bru2} implies that the center of the 
cyclotomic Hecke algebra, which under Brundan and Kleshchev's isomorphism 
correponds to $R(\mu_S-\lambda,\lambda)$, has the same dimension as 
$H^*(X^{(3^k)}_{\mu_S})$.

\begin{cor}
\label{cor:moritacenter}
The center of $K_S$ is isomorphic to the center of $R(\mu_S-\lambda,\lambda)$. 
In particular, we have 
\[
\dim Z(K_S)=\dim Z(R(\mu_S-\lambda,\lambda))=\dim 
H^*(X^{(3^k)}_{\mu_S}).
\]
\end{cor}
\begin{proof}
We only have to prove the first statement, which follows from 
the well-known fact that 
Morita equivalent algebras have 
isomorphic centers. For a proof see for example Corollary 18.42 in~\cite{lam}.
\end{proof}

In Theorem~\ref{thm:center} we used 
Corollary~\ref{cor:moritacenter} to give an explicit 
isomorphism  
\[
H^*(X^{(3^k)}_{\mu_S})\to Z(K_S).
\] 

\begin{rem}
Just for completeness, we remark that the aforementioned results 
in~\cite{bru2} and~\cite{bk2} together with the results in~\cite{bru}, 
which we have not explained, imply that 
\[
H^*(X^{(3^k)}_{\mu_S})\cong Z(R(\mu_S-\lambda,\lambda)),
\] 
so we have not proved anything new about 
$Z(R(\mu_S-\lambda,\lambda))$.
\end{rem}

Another interesting consequence of Proposition~\ref{prop:morita} is 
the following.
\begin{cor}
\label{cor:moritacellular}
$K_S$ is a graded cellular algebra. 
\end{cor}
\begin{proof}
In Corollary 5.12 in~\cite{hm}, Hu and Mathas proved that 
$R(\mu_S-\lambda,\lambda)$ is a graded cellular algebra. 

In~\cite{kx1}, K\"{o}nig and Xi showed that ``being a cellular algebra'' is 
a Morita invariant property, provided that the algebra is defined over 
a field whose characteristic is not equal to two. Moreover, as we explained in Section~\ref{sec-techcell}, this is also true in the graded setting.  

These results together with Proposition~\ref{prop:morita} 
prove that $K_S$ is a graded cellular algebra. 
\end{proof}

The precise definition of a graded cellular algebra can be found 
in~\cite{hm} or Section~\ref{sec-techcell}. In a follow-up paper, we 
intend to discuss the cellular basis of $K_S$ in detail and use it to 
derive further results on the representation theory of $K_S$. For an isotopy invariant, homogeneous basis see Section~\ref{sec-webbase}.  

\begin{rem} 
Corollary~\ref{cor:moritacellular} is the $\mathfrak{sl}_3$ analogue of 
Corollary 3.3 in~\cite{bs1}, which proves that Khovanov's arc algebra $H_m$ 
is a graded cellular algebra. Compare also to the Example~\ref{ex-cellular2}.

It is ``easy'' to give a 
cellular basis of $H_m$. The proof of cellularity follows from checking a 
small number of cases by hand. 
For $K_S$, we tried to mimick that approach, but had to give up because the 
combinatorics got too complex.      
\end{rem}

\subsubsection{The Grothendieck group of $\mathcal{W}_{(3^k)}$}
Recall that $W_S^{\mathbb{Z}}$ has an inner product defined 
by the normalised Kuperberg form 
(see Definition~\ref{defn:normkuperform}). 
The \textit{Euler form}
\[
\langle [P],[Q]\rangle=\dim_q \mathrm{HOM}(P,Q)
\]
defines a $\mathbb{Z}[q,q^{-1}]$-sesquilinear form on $K^{\oplus}_0(K_S)$.

\begin{lem}
\label{lem:grothendieck}
Let $S$ be an enhanced sign sequence. Take 
\[
\gamma_S\colon W_S^{\mathbb{Z}}\to K^{\oplus}_0(K_S)
\] 
to be the $\mathbb{Z}[q,q^{-1}]$-linear 
map defined by 
\[
\gamma_S(u)=[P_u],
\]
for any $u\in B_S$. Then $\gamma_S$ is an isometric embedding.   

This implies that the $\mathbb{Z}[q,q^{-1}]$-linear map
\[
\gamma_W=\bigoplus_{\mu(S)\colon\Lambda(n,n)_3}\gamma_S
\] 
defines an isometric embedding 
\[
\gamma_W\colon W_{(3^k)}^{\mathbb{Z}}\to K^{\oplus}_0(\mathcal{W}_{(3^k)})
.
\] 
\end{lem}
\begin{proof}
Note that the normalised Kuperberg form, because of the 
relations~\ref{eq:circle},~\ref{eq:digon} and~\ref{eq:square}, and the 
Euler form are non-degenerate.   
For any pair $u,v\in B_S$, we have  
\[
\dim_q \mathrm{HOM}(P_u,P_v)=\dim_q {}_uK_v=
q^{\ell(\hat{S})}\langle u^*v\rangle.
\] 
The factor $q^{\ell(\hat{S})}$ is a consequence of the grading shift in the 
definition of ${}_uK_v$.

Thus, $\gamma_S$ is an isometry. Since the normalised Kuperberg form is 
non-degenerate, this implies that $\gamma_S$ is an embedding. 
\end{proof}

\begin{rem} 
\label{rem:counter2}
It is well-known that 
$K^{\oplus}_0(K_S)$ is the free $\mathbb{Z}[q,q^{-1}]$-module 
generated by the isomorphism 
classes of the indecomposable, projective $K_S$-modules (see Example~\ref{ex-grothen2}). 
In Section 5.5 in~\cite{mn}, Morrison and Nieh showed that 
$P_u$ is not necessarily indecomposable (see also~\cite{rob}). This 
is closely related to 
the contents of Remark~\ref{rem:counter}, as Morrison and Nieh showed.  
Therefore, the surjectivity of $\gamma_W$ is not immediately clear and we need 
the results of the previous sections to establish it below. 

The $\mathfrak{sl}_2$ case is much simpler. The projective modules analogous to 
the $P_u$ are all indecomposable. See Proposition 2 in~\cite{kh4} for 
the details. 
\end{rem}

\begin{thm}
\label{thm:equivalence2}
The map  
\[
\gamma_W\colon W_{(3^k)}^{\mathbb{Z}}\to K^{\oplus}_0({\mathcal W}_{(3^k)})
\] 
is an isomorphism of $S_q^{\mathbb{Z}}(n,n)$-modules. 

This also implies that, for each sign string $S$, the map  
\[
\gamma_S\colon W_S^{\mathbb{Z}}\to K^{\oplus}_0(K_S)
\]
is an isomorphism. 
\end{thm}
\begin{proof}
The proof of the theorem is only a matter of assembling already known pieces.

By Proposition~\ref{prop:cataction}, 
$\gamma_W$ intertwines the $S_q^{\mathbb{Z}}(n,n)\cong 
K^{\oplus}_0(\ScatD(n,n))$ actions. 

We already know that $\gamma_W$ is an embedding, 
by Lemma~\ref{lem:grothendieck}. 

Note that, by Theorem~\ref{thm:bk}, Lemma~\ref{lem:isoirrep} and 
Corollary~\ref{cor:equivalence}, 
we have the following commuting square   
\[
\begin{xy}
  \xymatrix{
      V_{(3^k)}^{\mathbb{Z}} \ar[r]^{\gamma_V} \ar[d]_{\phi}    &   K^{\oplus}_0({\mathcal V}_{(3^k)}) \ar[d]^{K^{\oplus}_0(\Phi)}  \\
      W_{(3^k)}^{\mathbb{Z}} \ar[r]_{\gamma_W}             &   K^{\oplus}_0({\mathcal W}_{(3^k)})   
  }
\end{xy}
\]
We already know that $\gamma_V$, $\phi$ and $K^{\oplus}_0(\Phi)$ are isomorphisms. 
Therefore, $\gamma_W$ has to be an isomorphism. This shows that 
$K_S$ indeed categorifies the $\mu_S$-weight space 
of $V_{(3^k)}$. 

Recall that we have not explained the definition of $\gamma_V$ nor 
Rouquier's definition of $\Phi$. However, for general reasons, 
$\gamma_V$ has to send the highest weight vector $v_h\in V_{(3^k)}^{\mathbb{Z}}$ 
to the class of the highest weight object in $\mathcal{V}_{(3^k)}$ and 
$\Phi$ has to send that highest weight object 
to the highest weight object in $\mathcal{W}_{(3^k)}$. 
This shows that the images of the highest weight vector 
$v_h\in V_{(3^k)}^{\mathbb{Z}}$ around the two sides of the square are equal. 
Since all maps involved are $S_q^{\mathbb{Z}}(n,n)$ intertwiners, 
it follows that the square indeed commutes.      
\end{proof}
\vskip0.5cm
A good question is how to find the graded, indecomposable, 
projective modules of $K_S$. Before answering that question, 
we need a result on the 3-colourings of webs. 

Let $w\in B_S$. Recall that there is a bijection between the flows on $w$ 
and the $3$-colourings of $w$, as already mentioned in Remark~\ref{rem:3color}. 
Call the 3-colouring corresponding to the canonical flow of $w$, the 
\textit{canonical 3-colouring}, denoted $T_w$. 
\begin{lem}
\label{lem:tech}
Let $u, v\in B_S$. If there 
is a 3-colouring of $v$ which matches $T_u$ and a 3-colouring of $u$ 
which matches $T_v$ on the common boundary $S$, then $u=v$. 
\end{lem} 
\begin{proof}
This result is a direct consequence of Theorem~\ref{thm:upptriang}. Recall 
that there is a partial order on flows, and therefore on 3-colourings by 
Remark~\ref{rem:3color}. This ordering is induced by the lexicographical 
order on the state-strings on $S$, which are induced by the flows. Note that 
two matching colourings of $u$ and $v$ have the same order, by definition. 
On the other hand, Theorem~\ref{thm:upptriang} implies that 
any 3-colouring of $u$, respectively $v$, has order less than or equal to 
that of $T_u$ and $T_v$ respectively. 
Therefore, if there exists a 3-colouring of $v$ matching $T_u$, the order of 
$T_u$ must be less or equal than that of $T_v$. 

Thus, if there 
exists a 3-colouring of $v$ matching $T_u$ 
and a 3-colouring of $u$ matching $T_v$, then $T_u$ and $T_v$ must have the 
same order. This implies that $u=v$, because canonical 3-colourings are uniquely 
determined by their order and the corresponding canonical flows determine 
the corresponding basis webs uniquely by the growth algorithm.   
\end{proof}

\begin{prop}
\label{prop:unitriang}
For each $u\in B_S$, there exists a unique graded, indecomposable, projective 
$K_S$-module $Q_u$, such that 
\[
P_u\cong Q_u\oplus\bigoplus_{J_v < J_u} d(S,J_u,J_v)\; Q_v.
\]
Here $J_u$ is the state string associated to the canonical flow on $u$, 
the coefficients $d(S,J_u,J_v)$ belong to $\mathbb{N}[q,q^{-1}]$ and 
indicate direct sums and degree shifts as usual, 
and the state strings are ordered lexicographically. 
\end{prop}
Note that we need a lot of the results from the Sections~\ref{sec-webalg},~\ref{sec-webcenterb} and~\ref{sec-webhoweb} to prove the proposition.
\begin{proof}
Let $u\in B_S$. Then there is a complete decomposition of $1_u$ into orthogonal, 
primitive idempotents 
\[
1_u=e_1+\cdots+e_r.
\]
By Theorem~\ref{thm:sjodin} and Corollary~\ref{cor:sjodin}, we can lift this 
decomposition to $G_S$. 
We do not introduce any 
new notation for this lift, trusting that the reader will not get confused 
by this slight abuse of notation. 

Let $z_u\in Z(G_S)$ be the central idempotent corresponding to $J_u$, as 
defined in the proof of Lemma~\ref{lem:dimZG}. We claim that there is a unique 
$1\leq i\leq r$, such that 
\begin{equation}
\label{eq:orto1}
z_ue_j=\delta_{ij}z_u1_u,
\end{equation}
for any $1\leq j\leq r$. 

Let us prove this claim. Note that 
\begin{equation}
\label{eq:orto2}
z_u1_u=e_{u,T_u},
\end{equation}
where $T_u$ is the canonical colouring of $u$, i.e. 
$J_u$ only allows one compatible colouring of $u$, which is $T_u$. 
Since $e_{u,T_u}\ne 0$, courtesy of Lemma~\ref{lem:embedding}, this implies that 
\begin{equation}
\label{eq:dimone}
z_u{}_uG_u={}_uG_u z_u=z_u{}_uG_uz_u=e_{u,T_u}G_Se_{u,T_u}\cong \mathbb{C},
\end{equation}
by Theorem~\ref{thm:Gornik}. 

We also see that there has to exist at least one 
$1\leq i_0\leq r$ such that $z_ue_{i_0}\ne 0$. Then, by~\eqref{eq:dimone}, 
there exists a non-zero $\lambda_{i_0}\in \mathbb{C}$, such that 
\[
z_ue_{i_0}=\lambda_{i_0}z_u1_u=\lambda_{i_0}e_{u,T_u}.
\]  
For any 
$1\leq i,j\leq r$, we have  
\[
z_ue_iz_ue_j=z_u^2e_ie_j=z_u\delta_{ij}e_i.
\]
This implies that $i_0$ is unique and $\lambda_{i_0}=1$. 
In order to see that this is true, suppose there exist 
$1\leq i_0\ne j_0\leq r$ such that $z_ue_{i_0}\ne 0$ and $z_ue_{j_0}\ne 0$. 
By~\eqref{eq:dimone}, there exist non-zero 
$\lambda_{i_0},\lambda_{j_0}\in\mathbb{C}$ such that 
\[
z_ue_{i_0}=\lambda_{i_0}z_u1_u\quad\text{and}\quad z_ue_{j_0}=\lambda_{j_0}
z_u1_u.
\]
However, this is impossible, because we get  
\[
z_ue_{i_0}z_ue_{j_0}=\lambda_{i_0}\lambda_{j_0}z_u1_u\ne 0,
\] 
which contradicts the orthogonality of $z_ue_{i_0}$ and $z_ue_{j_0}$. 

Thus, for each $u\in B_S$, there is a unique primitive idempotent 
$e_u\in \mathrm{End}_{\mathbb{C}}(P_u)$ that is not killed by $z_u$, 
when lifted to 
$G_S$. We define $Q_u$ to be the corresponding graded, indecomposable, 
projective $K_S$-module  
\[
Q_u=K_Se_u,
\]
which is clearly a direct summand of $P_u=(K_S)1_u$. 
\vskip0.5cm
Let us now show that, for any $u,v\in B_S$, we have  
\[
Q_u\cong Q_v\Leftrightarrow u=v.
\]
If $u=v$, we obviously have $Q_u\cong Q_v$. 
Let us prove the other implication. Suppose $Q_u\cong Q_v$. From 
the above, recall that $e_u$ and $e_v$ can be lifted to $G_S$. By a slight 
abuse of notation, call these lifted idempotents $e_u$ and $e_v$ again. 
We have 
\[
z_ue_u=e_{u,T_u}\ne 0\quad\text{and}\quad z_ve_v=e_{v,T_v}\ne 0.
\]
Since $Q_u\cong Q_v$, we then also have 
\[
z_ue_v\ne 0\quad\text{and}\quad z_ve_u\ne 0.
\]
This can only hold if $T_u$ gives a 3-colouring of $v$ and $T_v$ a 3-colouring 
of $u$. By Lemma~\ref{lem:tech}, this implies that $u=v$.
\vskip0.5cm

Since 
\[
\mathrm{rk}_{\mathrm{Z}[q,q^{-1}]}K^{\oplus}_0(K_S) 
=\mathrm{rk}_{\mathrm{Z}[q,q^{-1}]} W_S^{\mathbb{Z}}=\#B_S,
\]
by Theorem~\ref{thm:equivalence2}, 
the above shows that 
\[
\left\{[Q_u]\mid u\in B_S\right\}
\]
is a basis of the free $\mathbb{Z}[q,q^{-1}]$-module 
$K^{\oplus}_0(K_S)$. For any $u,v\in B_S$, we 
have  
\[
z_u1_u=z_ue_u\quad\text{and}\quad z_v1_u=0,\;\text{if}\; J_v>J_u.
\]
The second claim follows from the fact that there are no admissible 
3-colourings of $u$ greater than $J_u$. The proposition now follows. 
\end{proof}

\begin{rem} Proposition~\ref{prop:unitriang} proves the conjecture 
about the decomposition of $1_u$, which Morrison and Nieh 
formulate in their Conjectures 5.14 and 5.15 in~\cite{mn} and in the text below them.  
\end{rem}

Before giving the last result of this section, we briefly recall 
some facts about the \textit{dual canonical basis} of $W_S^{\mathbb{Z}}$. 
For more details see~\cite{fkhk} and~\cite{kk}. There exists a $q$-antilinear 
involution $\tilde{\psi}$ on $V_S^{\mathbb{Z}}$ (in~\cite{fkhk} and~\cite{kk} 
this involution is denoted $\psi'$ and $\Phi$, respectively). 
For any sign string $S$ and any state 
string $J$, there exists a unique element $e^S_{\heartsuit J}\in V_S^{\mathbb{Z}}$ which 
is invariant under $\tilde{\psi}$ and such that 
\begin{equation}
\label{eq:dualcaneq}
e^S_{\heartsuit J}=e^S_J+\sum_{J'<J}c(S,J,J')e^{S}_{J'},
\end{equation}
with $c(S,J,J')\in q\mathbb{Z}[q]$. Note 
that $q=v^{-1}$ in~\cite{fkhk} and~\cite{kk}. 
The $e^S_J$ are the elementary tensors, which were defined 
in~\ref{thm:upptriang}. The basis $\left\{e^S_{\heartsuit J}\right\}$ 
is called the \textit{dual canonical basis} of $V_S^{\mathbb{Z}}$. Restriction to the 
dominant closed paths $(S,J)$ 
gives the dual canonical basis of $W_S^{\mathbb{Z}}$ (see Theorem 3 
in~\cite{kk} and the comments below it).

We have not given a definition of $\tilde{\psi}$, but we 
note that $\tilde{\psi}$ is completely determined by Proposition 2 in~\cite{kk}, i.e. we recall the following. 
\begin{prop}\label{prop:basisbarinvariant}(\textbf{Khovanov-Kuperberg})
Each basis web $w\in B_S$ is invariant under $\tilde{\psi}$. 
\end{prop}

The above definition is hard to check directly for $\{[Q_u]\mid u\in B_S\}$. 
Therefore, let us recall Webster's~\cite{web3} very general definition of a 
canonical basis of a free $\mathbb{Z}[q,q^{-1}]$-module $M$. 
Our $q$ corresponds to $q^{-1}$ 
in~\cite{web3}. A \textit{pre-canonical structure} on 
$M$ is a choice of the following. 
\begin{itemize}
\item A $q$-antilinear ``bar involution'' $\psi\colon M\to M$.
\item A sesquilinear inner product $\langle -,-\rangle\colon M\times M\to 
\mathbb{Z}((q))$. 
\item A ``standard basis'' $\{a_c\}_{c\in C}$ with partially ordered index 
set $(C,<)$ 
such that 
\begin{equation}
\label{eq:standardbasis}
\psi(a_c)\in a_c+\sum_{c'<c}\mathbb{Z}[q,q^{-1}]a_{c'}.
\end{equation}
\end{itemize}
A basis $\{b_c\}_{c\in C}$ is called \textit{canonical} if the following is satisfied. 
\begin{enumerate}
\item Each vector $b_c$ is invariant under $\psi$.
\item Each vector $b_c$ belongs to $a_c+\sum_{c'<c}\mathbb{Z}[q,q^{-1}]a_{c'}$.
\item The vectors $b_c$ are \textit{almost orthonormal} in the sense that 
\begin{equation}
\label{eq:almostortho}
\langle b_c,b_{c'}\rangle\in \delta_{c,c'}+q\mathbb{Z}[q].
\end{equation}
\end{enumerate}
If a canonical basis exists, for a given pre-canonical structure, then it 
is unique by Theorem 26.3.1 in~\cite{lu}. In particular, the dual 
canonical basis is ``canonical'' in the above sense, w.r.t. to a 
pre-canonical structure which we will discuss below. 
We note that the same basis can be canonical w.r.t. different 
pre-canonical structures. 
\vskip0.5cm
Let us show now how Lusztig's canonical basis on $V_{(3^k)}^{\mathbb{Z}}\cong 
K^{\oplus}_0(\mathcal{V}_{(3^k)})$ is mapped to a basis in 
$W_{(3^k)}^{\mathbb{Z}}\cong K^{\oplus}_0(\mathcal{W}_{(3^k)})$, which 
is also canonical according to Webster's definition. After doing that, 
we will prove that the latter basis is exactly the dual canonical basis 
defined in~\cite{fkk} and~\cite{kk}. 

First the pre-canonical structures.  
\begin{itemize}
\item As Brundan and Kleshchev showed in~\cite{bk2} and Webster recalled in 1.2 in~\cite{web3}, 
the bar involution on $K^{\oplus}_0(\mathcal{V}_{(3^k)})$ is induced by 
Khovanov and Lauda's~\cite{kl5} contravariant functor 
\[
\psi\colon R_{(3^k)}\to R_{(3^k)},
\]
given by reflecting the diagrams in the $x$-axis and inverting 
their orientation. On objects this functor sends 
$\mathcal{F}_i\{t\}$ to $\mathcal{F}_i\{-t\}$. 

Using our equivalence 
\[
\Phi\colon \mathcal{V}_{(3^k)}\to \mathcal{W}_{(3^k)}
\] 
from Theorem~\ref{thm:equivalence}, we get a contravariant functor 
\[
\psi\colon \mathcal{W}_{(3^k)}\to \mathcal{W}_{(3^k)}
\]
given by reflecting the foams in the vertical $yz$-plane, i.e. the plane 
parallel to the front and the back of the foams in 
Definition~\ref{defn:cataction}, and inverting the orientation of 
their edges. 

We have 
\[
\psi(P_u)\cong P_u,\quad\text{for any non-elliptic web}\; u.
\]
It might seem confusing that $P_u$ is again a left and not a right 
$K_{(3^k)}$-module. The reason is that any $f\in K_{(3^k)}$ acts 
on $\psi(P_u)$ by multiplication on the right with $\psi(f)$. 
Since $\psi$ is contravariant, this gives a left action again.  

Using the isomorphism 
\[
\psi\colon W_{(3^k)}^{\mathbb{Z}}\to K^{\oplus}_0(\mathcal{W}_{(3^k)})
\] 
to pull back $K^{\oplus}_0(\psi)$, we get a bar involution on $W_{(3^k)}^{\mathbb{Z}}$ which 
fixes the non-elliptic webs. By Proposition~\ref{prop:basisbarinvariant}, 
we see that this bar involution is equal to $\tilde{\psi}$.
\item As remarked by Webster in the introduction of~\cite{web3}, the inner 
product on $K^{\oplus}_0(\mathcal{V}_{(3^k)})$ is given by the Euler form 
\[
\langle [P],[Q]\rangle =\dim_q\left(\mathrm{HOM}(P,Q)\right).
\]
Pulling back the Euler form via the isomorphism 
\[
\gamma\colon V_{(3^k)}^{\mathbb{Z}}\to K^{\oplus}_0(\mathcal{V}_{(3^k)})
\]
gives the $q$-Shapovalov form.  

Our isomorphism  
\[
K^{\oplus}_0(\Phi)\colon K^{\oplus}_0(\mathcal{V}_{(3^k)})\to K^{\oplus}_0(\mathcal{W}_{(3^k)})
\] 
from Theorem~\ref{thm:equivalence2} is an isometry intertwining 
the Euler forms. Furthermore, the Euler form on the latter 
Grothendieck group corresponds to 
the normalised Kuperberg form on $W_{(3^k)}^{\mathbb{Z}}$, by 
Lemma~\ref{lem:phiisometry}.
\item For our purpose, we are only interested in a standard basis 
of $K^{\oplus}_0(\mathcal{W}_{(3^k)})$. We take $\{[P_u]\}$, where the $u$ 
are the non-elliptic webs in $W_{(3^k)}^{\mathbb{Z}}$. The partial ordering is given 
by the lexicographical ordering of the state-strings for each $S$. By 
Proposition~\ref{prop:basisbarinvariant}, we see that $\{[P_u]\}$ 
satisfies~\eqref{eq:standardbasis}.
\end{itemize} 
Now, let us have a look at the canonical bases in 
$K^{\oplus}_0(\mathcal{V}_{(3^k)})$ and $K^{\oplus}_0(\mathcal{W}_{(3^k)})$, 
which both satisfy Webster's definition. 
\begin{enumerate}
\item The canonical basis elements in $K^{\oplus}_0(\mathcal{V}_{(3^k)})$ are 
the classes of the indecomposable projective $R_{(3^k)}$-modules, 
with their gradings suitably normalized (which is all that we need). These elements correspond precisely to 
Lusztig's canonical basis elements in $V_{(3^k)}^{\mathbb{Z}}$, 
as shown by Brundan and Kleshchev~\cite{bk2}. In particular, they satisfy 
the three conditions for a canonical basis in Webster's list.

Our equivalence in Theorem~\ref{thm:equivalence} maps the indecomposable 
objects in $\mathcal{V}_{(3^k)}$ to the indecomposable objects in 
$\mathcal{W}_{(3^k)}$, which we had called $Q_u$. 
In particular, this shows that $\tilde{\psi}([Q_u])=q^a[Q_u]$ for some suitable value $a$. Since, isomorphisms on Grothendieck groups intertwine the $\tilde{\psi}$, we see that $q^a[Q_u]$ is $\tilde{\psi}$ invariant. But since, as a consequence of the Propositions~\ref{prop:unitriang} and~\ref{prop:basisbarinvariant}, $[Q_u]$ is already $\tilde{\psi}$ invariant, we see that $a=0$. 
\item The $[Q_u]$ also satisfy the second condition in Webster's list, 
as follows from inverting the change of basis matrix in 
Proposition~\ref{prop:unitriang}.
\item The third condition in Webster's list, for the $[Q_u]$, 
follows from the fact that Lusztig's canonical basis elements $[P_u]$ 
satisfy that condition and the fact that the isomorphism 
\[
K^{\oplus}_0(\Phi)\colon K^{\oplus}_0(\mathcal{V}_{(3^k)})\to K^{\oplus}_0(\mathcal{W}_{(3^k)}),
\] 
with $\Phi$ as in Theorem~\ref{thm:equivalence2}, 
maps Lusztig's canonical basis isometrically onto $\{Q_u\}$.   
\end{enumerate}

\begin{thm}
\label{thm:dualcan}
The basis 
\[
\left\{[Q_u]\mid u\in B_S\right\}
\]
corresponds to the dual canonical basis 
of $\mathrm{Inv}(V_S^{\mathbb{Z}})$, under the isomorphisms 
\[ 
\mathrm{Inv}(V_S^{\mathbb{Z}})\cong W_S^{\mathbb{Z}}\cong K^{\oplus}_0(K_S).
\]
\end{thm}
\begin{proof}
The remarks above prove that $\{[Q_u]\mid u\in B_S\}$ 
is a canonical basis in the sense 
of Webster's definition. What remains to be proven, is that it is exactly 
the dual canonical basis defined by Frenkel, Khovanov and Kirillov 
Jr.~\cite{fkhk} (and in Theorem 3 in~\cite{kk}).

As we demonstrated above, the bar involution on $K^{\oplus}_0(\mathcal{W}_{(3^k)})$ 
is exactly the bar involution for the dual canonical basis in~\cite{fkhk} 
and~\cite{kk}. 

As we will explain below, the 
normalised Kuperberg $\mathbb{Z}[q,q^{-1}]$-sesquilinear form on 
$K^{\oplus}_0(\mathcal{W}_{(3^k)})$, given in Definition~\ref{defn:normkuperform} and 
denoted by $\langle-,-\rangle_{\mathrm{Kup}}$ in this proof, 
is exactly the one corresponding to the 
pre-canonical structure used in~\cite{fkhk} and~\cite{kk}. 

Since there is at most one canonical basis for any given 
pre-canonical structure on $K^{\oplus}_0(\mathcal{W}_{(3^k)})$, this proves that the 
two bases are equal. 
\vskip0.5cm
For completeness, let us explain why $\langle-,-\rangle_{\mathrm{Kup}}$ is 
exactly equal to the $\mathbb{Z}[q,q^{-1}]$-sesquilinear 
inner product that is used \textit{implicitly} 
in~\cite{fkhk} and~\cite{kk}. The form that is used \textit{explicitly} 
in~\cite{fkhk} 
and~\cite{kk} is actually Lusztig's $\mathbb{Z}[q,q^{-1}]$-bilinear form, 
denoted $(-,-)_{\mathrm{Lusz}}$ in this proof and 
defined in Section 19.1.1 in~\cite{lu} for irreducible modules and 
extended factorwise to tensor 
products in Section 27.3 of that same book. 

Therefore, we 
first have to recall how the $\mathbb{Z}[q,q^{-1}]$-bilinear forms from above are 
related to $\mathbb{Z}[q,q^{-1}]$-sesquilinear forms.
Given a $\mathbb{Z}[q,q^{-1}]$-bilinear form $(-,-)$ on $V_S^{\mathbb{Z}}$, we can define a 
$\mathbb{Z}[q,q^{-1}]$-sesquilinear form on $V_S^{\mathbb{Z}}$ which is 
$\mathbb{Z}[q,q^{-1}]$-antilinear in the first variable, by 
\begin{equation}
\label{eq:sesquilversusbil}
\langle x,y \rangle=(\tilde{\psi}(x),y),
\end{equation}
where $\tilde{\psi}$ is the $\mathbb{Z}[q,q^{-1}]$ anti-involution 
mentioned above. 
This is exactly how Khovanov and Lauda defined their 
$\mathbb{Z}[q,q^{-1}]$-sesquilinear form on $\U$ in 
Definition 2.3 in~\cite{kl5}. 

We do not compute the action of $\tilde{\psi}$ 
on the elementary tensors $e^S_J$ explicitly in this thesis. 
As we will show below, the $e^S_J$ are orthonormal 
w.r.t. $(-,-)_{\mathrm{Lusz}}$. Therefore, it 
is easier to show that $(-,-)_{\mathrm{Lusz}}$ is equal to the 
$\mathbb{Z}[q,q^{-1}]$-bilinear form coming from 
Kuperberg's bracket, which we denote by $(-,-)_{\mathrm{Kup}}$ in this proof, 
than to compare the corresponding $\mathbb{Z}[q,q^{-1}]$-sesquilinear forms 
directly. Just for the record, we remark that 
$\langle-,-\rangle_{\mathrm{Lusz}}$ is not equal to the 
factorwise $q$-Shapovalov form, which is part of the pre-canonical 
structure for Lusztig's canonical basis of $V_S^{\mathbb{Z}}$ (see Theorem 3.10 in~\cite{web3}).

Let us recall the definition of $(-,-)_{\mathrm{Lusz}}$ 
on an irreducible weight $\UZ(\mathfrak{sl}_3)$-module 
$V^{\mathbb{Z}}$ with highest weight vector $v_h$. 
We follow Khovanov and Lauda's normalisation from Proposition 
2.2 in~\cite{kl5}. 
Lusztig's $\mathbb{Z}[q,q^{-1}]$-bilinear form on $V^{\mathbb{Z}}$ 
is uniquely determined by the properties below.
\begin{itemize}
\item $(v_h,v_h)_{\mathrm{Lusz}}=1$.
\item $(ux,y)_{\mathrm{Lusz}}=(x,\overline{\rho}(u)y)_{\mathrm{Lusz}}$. 
\item $(y,x)_{\mathrm{Lusz}}=(x,y)_{\mathrm{Lusz}}$, for any $x,y\in V^{\mathbb{Z}}$ and any 
$u\in \UZ(\mathfrak{sl}_3)$.
\end{itemize} 
Here 
$\overline{\rho}$ is the $\mathbb{Z}[q,q^{-1}]$-linear anti-involution on 
$\UZ(\mathfrak{sl}_3)$ defined by 
\[
\rho(E_i)=q^{-1}K_i^{-1}E_{-i},\;\rho(E_{-i})=q^{-1}K_iE_i,\;\rho(K_i^{\pm 1})=
K_i^{\pm 1}.
\]
Let $(-,-)_{\mathrm{Lusz}}$ also denote the $\mathbb{Z}[q,q^{-1}]$-bilinear inner 
product on $V_S^{\mathbb{Z}}$ obtained by taking 
factorwise the above form on $V_{s_i}^{\mathbb{Z}}$, for $i=1,\ldots,\ell(S)$. 
 
Before we can compute the inner product of the elementary tensors, we first 
have to compute $(-,-)_{\mathrm{Lusz}}$ on $V_+^{\mathbb{Z}}$ and 
$V_-^{\mathbb{Z}}$. Let $e_1^+$ be the highest weight 
vector of $V_+$, of weight $(1,0)$, and define 
\[
e_0^+=E_{-1}(e_1^+)\quad\text{and}\quad e_{-1}^+=E_{-2}(e_0^+).
\]
Note that $e_0^+$ and $e_{-1}^+$ are of weight $(-1,1)$ and $(0,-1)$ 
respectively. Similarly, let $e_1^{-}$ be the highest weight vector of 
$V_-^{\mathbb{Z}}$, of weight $(0,1)$, and define 
\[
e_0^-=E_{-2}(e_1^-)\quad\text{and}\quad e_{-1}^-=E_{-1}(e_0^-).
\]
Note that $e_0^+$ and $e_{-1}^+$ are of weight $(1,-1)$ and $(-1,1)$ 
respectively. Using the rules above, we get 
\[
(e_i^{\pm},e_{j}^{\pm})_{\mathrm{Lusz}}=\delta_{ij}.
\]
On $V_S^{\mathbb{Z}}$, we now get 
\begin{equation}
\label{eq:Lusorthog}
(e^S_{J'},e^S_{J''})_{\mathrm{Lusz}}=\delta_{J',J''},
\end{equation}
for any elementary tensors $e^S_{J'}$ and $e^S_{J''}$.

Note that both $(-,-)_{\mathrm{Lusz}}$ and $(-,-)_{\mathrm{Kup}}$ 
are $\mathbb{Z}[q,q^{-1}]$-bilinear and symmetric. 
Therefore, in order to show that they 
are equal, it suffices to show that we have 
\[
(w^S_J,w^S_J)_{\mathrm{Lusz}}=(w^S_J,w^S_J)_{\mathrm{Kup}},
\]
for any $w^S_J\in B_S$. 

Let $w^S_J\in B_S$ be arbitrary and write  
\[
w^S_J=e^S_J+\sum_{J'<J} c(S,J,J')e^S_{J'},
\]
as in Theorem~\ref{thm:upptriang}. Then, by~\eqref{eq:Lusorthog}, we get 
\begin{equation}
\label{eq:firstcheck}
(w^S_J,w^S_J)_{\mathrm{Lusz}}=1 + \sum_{J'<J}c(S,J,J')^2.
\end{equation}

Finally, let us 
compute $(w^S_J,w^S_J)_{\mathrm{Kup}}$. By~\eqref{eq:sesquilversusbil}, 
we see that 
\[
(w^S_J,w^S_J)_{\mathrm{Kup}}=\langle w^S_J,w^S_J\rangle=q^{\ell(S)}\langle 
(w^S_J)^*w^S_J\rangle_{\mathrm{Kup}}.
\]
The first equality follows from Proposition~\ref{prop:basisbarinvariant}. 
Now consider the way in which the coefficients $c(S,J,J')$ change 
under the symmetry $x\mapsto x^*$, for $x$ any $Y$, cup or cap with flow. 
Comparing the corresponding weights 
in~\eqref{weights} and~\eqref{weights2}, we get 
\[
\mathrm{weight}(x^*)=q^{-(\ell(t(x))-\ell(b(x)))}\mathrm{weight}(x).
\]
where $t(x)$ and $b(x)$ are the top and bottom boundary of $x$. 
Recall also that the canonical flow on $w^S_J$ has weight 0 (see 
Lemma~\ref{lem:canflowzero}). 
It follows that 
\begin{eqnarray*}
\label{eq:secondcheck}
(w^S_J, w^{S}_{\tilde{J}})_{\mathrm{Kup}}&=&
q^{\ell(S)}\langle(w^S_J)^*w^{S}_{\tilde{J}}\rangle_{\mathrm{Kup}}\\
&=&q^{\ell(S)}\left(q^{-\ell(S)}+q^{-\ell(S)}\sum_{J'<J}c(S,J,J')^2\right)\\
&=&1+\sum_{J'<J}c(S,J,J')^2.
\end{eqnarray*}
This finishes the proof that $(-,-)_{\mathrm{Lusz}}=(-,-)_{\mathrm{Kup}}$. 
\end{proof}
\subsection{An isotopy invariant basis}\label{sec-webbase}
In the present section we define a \textit{homogeneous, isotopy invariant basis} of $K_S$. Note that this section is not part of our paper~\cite{mpt} and that it splits into two subsections.

To be more precise, we give a method for obtaining several different homogeneous bases in the first subsection and explain how one can define at least one of them (by an \textit{algorithm}) in an \textit{isotopy invariant} way in the second subsection. Note that the method follows a face removing convention, motivated by the Kuperberg relations~\cite{kup}. The algorithm makes a particular, isotopy invariant choice, that we call \textit{preferred}. But the procedure that we explain works for any \textit{fixed} choice of how to remove faces as explained in Theorem~\ref{thm-basis}.

This isotopy invariant basis is parametrised by flow lines, as we explain later in Theorem~\ref{thm-basis2} and Corollary~\ref{cor-basis}. Moreover, the whole discussion in this section also works for the more general web algebra ${\mathcal W}_S^c$ from Definition~\ref{defn:webalg}. And therefore for Gronik's filtered algebra $G_S$ defined in Definition~\ref{defn-khgo}.
\vskip0.5cm
The question can be explained as follows. Given an algebra $A$ that is only defined by generators and relation, e.g. $K_S$, then it is not obvious what the dimension of $A$ is, how to find a basis of $A$ and how to find a basis of $A$ with ``good'' structure coefficients (for example cellularity as described in Section~\ref{sec-techcell}, e.g. Theorem~\ref{thm-cellular2}).

In this section we answer the second question for $K_S$ and we conjecture that one of the bases we define is a graded cellular basis related to the basis for the cyclotomic KLR-algebra defined by Hu and Mathas~\cite{hm}. Recall that $K_{S}$ and $R(\mu_S-\lambda,\lambda)$ are graded Morita equivalent as we showed in Proposition~\ref{prop:morita}. This, by Theorem~\ref{thm-cellular3} of K\"onig and Xi~\cite{kx1} and the discussion in Section~\ref{sec-techcell}, shows that $K_{(3^k)}$ is a graded cellular algebra.  
\vskip0.5cm
In the whole section let $S$ denote a sign string of length $n$. Recall that if $u,v\in B_S$ are two non-elliptic webs with boundary $S$, then $w=u^*v$ is the closed web obtained by glueing $u^*$ on top of the web $v$. Moreover, recall that $u_f$ denotes a web $u$ with a flow $f$ that extents to $u$. By a slight abuse of notation, we use a similar notation for a closed web $w$ with flow $f$, but in this case the flow is closed and can be split into two flows $f_u,f_v$ that extent to $u$ and $v$ respectively and match at the boundary. Recall that flows have a weight that can be read off locally. In most pictures we use wt as a short hand notation for the weight.
\vskip0.5cm
We need some more terminology before we can start, i.e. in this section it is crucial to distinguish three types of faces.
\begin{itemize}
\item The \textit{external face} of $w$, i.e. the unbounded face around the web $w$.
\item The \textit{faces} of $w$, i.e. all bounded faces of $w$.
\item The \textit{internal faces} of $w$, i.e. all bounded faces of $w=u^*v$ not intersecting the cut line. 
\end{itemize}
\vskip0.5cm
For completeness, we note the following Proposition and a Corollary, i.e. the answer of question one from above. It is important to note that the web algebra has its $q$-degree shifted by $\{n\}$.
\begin{prop}\label{prop-dim}
Given $u,v\in B_S$, then the $q$-graded dimension $\dim_{q}$ of ${}_uK_v$ is given by the Kuperberg bracket, i.e. $\dim_{q}({}_uK_v)=q^n\cdot\langle u^*v\rangle_{\mathrm{Kup}}$.
\end{prop}
\begin{proof}
This was already implicit in Section~\ref{sec-webbasicb} and Section~\ref{sec-webalg}, i.e. combine the notes below Definition~\ref{foamhom}, Remark~\ref{rem:flowsfoams} and Lemma~\ref{lem:webalgaltern}.
\end{proof}
\begin{cor}\label{cor-dim}
Let ${}_uB_v$ be any homogeneous basis of ${}_uK_v$ and let $w=u^*v$. Then there is a bijection of sets
\[
\{w_f\mid\;\text{weight of }f=k\}={}_uF_v^{k}\cong {}_uB_v^{-k+n}=\{b\in {}_uB_v\mid \mathrm{deg}(b)=-k+n\}.
\] 
\end{cor}
\begin{proof}
This follows directly from Proposition~\ref{prop-dim} and the discussion in Section~\ref{sec-webbasica}.
\end{proof}
\subsubsection*{Face removing algorithm}
The definition below gives a procedure how to obtain for a given closed web $w=u^*v$ and a flow $f$ of weight $k$ on it an element in ${}_uK_v$ of degree $-k+n$. By a slight abuse of terminology, we call this procedure an algorithm, although we avoid to make some necessary choices at this point. This inductive algorithm is called the \textit{face removing algorithm}.

It is worth noting that the order in which we apply our local rules matters, since different orders give rise to different foams, but can be any fixed order if not otherwise specified.
\begin{defn}\label{defn-procedure}
\textbf{(Face removing algorithm)} Given a closed basis web $w=u^*v$ with any flow $f$, denoted $w_{f}$, we define a 
foam $f_{w_{f}}\colon \emptyset \to w\in {}_uK_v$ by a set of inductive, local rules 
beginning with the identity foam of $w$. Each rule corresponds to the 
removal of a circle, digon or a square face by glueing an elementary foam to the bottom of the previous foam. This is done in a way that is consistent with the flow until no more rules can be applied.

The elementary foams, which determine the local rules, called circle, digon and square removals, are depicted below. The \textit{circle removals} (for both possible orientations of the web) are
\[
\xy
 (0,0)*{\includegraphics[height=.05\textheight]{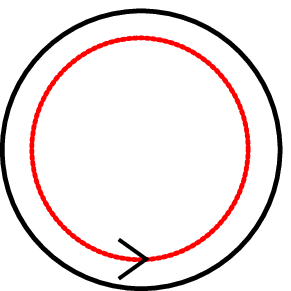}};
 (0,-7)*{\scriptstyle \text{wt}=2};
 (12,0)*{\longmapsto};
 (12,-5)*{\includegraphics[height=.03\textheight]{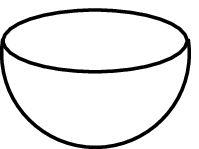}};
 (20,0)*{\emptyset};
 (20,-7)*{\scriptstyle \text{wt}=0};
 \endxy\;\;
 \xy
 (0,0)*{\includegraphics[height=.05\textheight]{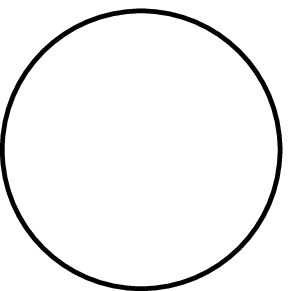}};
 (0,-7)*{\scriptstyle \text{wt}=0};
 (12,0)*{\longmapsto};
 (12,-5)*{\includegraphics[height=.03\textheight]{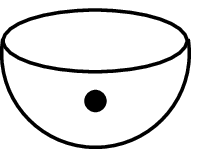}};
 (20,0)*{\emptyset};
 (20,-7)*{\scriptstyle \text{wt}=0};
 \endxy\;\;
\xy
 (0,0)*{\includegraphics[height=.05\textheight]{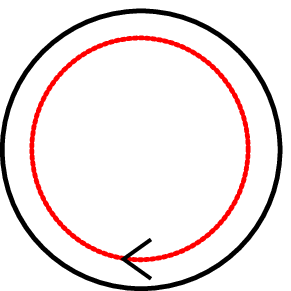}};
 (0,-7)*{\scriptstyle \text{wt}=-2};
 (12,0)*{\longmapsto};
 (12,-5)*{\includegraphics[height=.03\textheight]{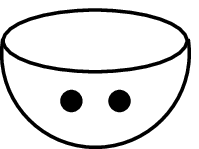}};
 (20,0)*{\emptyset};
 (20,-7)*{\scriptstyle \text{wt}=0};
 \endxy
\]
and the \textit{digon removals} are
\[
\xy
 (0,0)*{\includegraphics[height=.15\textheight]{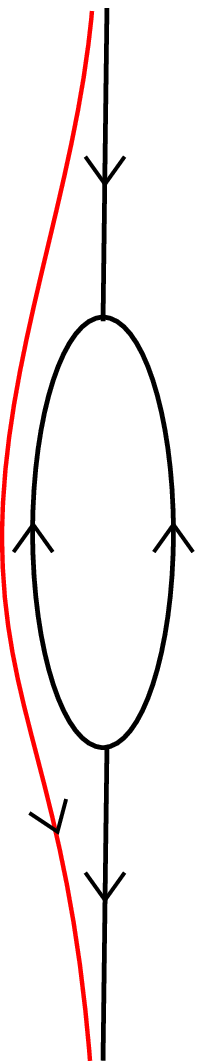}};
 (-6,-10)*{\scriptstyle \text{wt}=1};
 (13,0)*{\longmapsto};
 (13,-7)*{\includegraphics[height=.06\textheight]{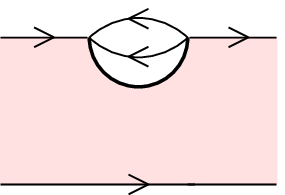}};
 (26,0)*{\includegraphics[height=.15\textheight]{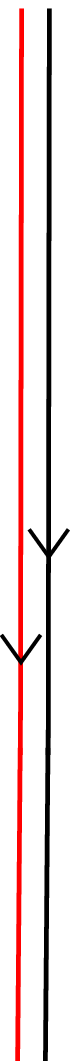}};
 (32,-10)*{\scriptstyle \text{wt}=0};
 (0,-40)*{\includegraphics[height=.15\textheight]{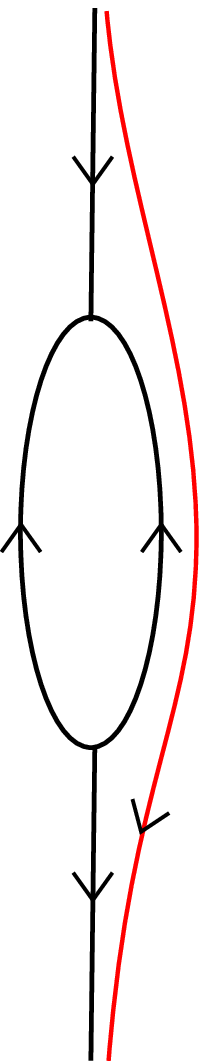}};
 (-6,-50)*{\scriptstyle \text{wt}=-1};
 (13,-40)*{\longmapsto};
 (13,-47)*{\includegraphics[height=.06\textheight]{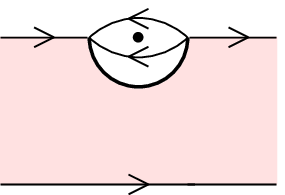}};
 (26,-40)*{\includegraphics[height=.15\textheight]{res/figs/basis/linedown.eps}};
 (32,-50)*{\scriptstyle \text{wt}=0};
 \endxy\;\;
 \xy
 (0,0)*{\includegraphics[height=.15\textheight]{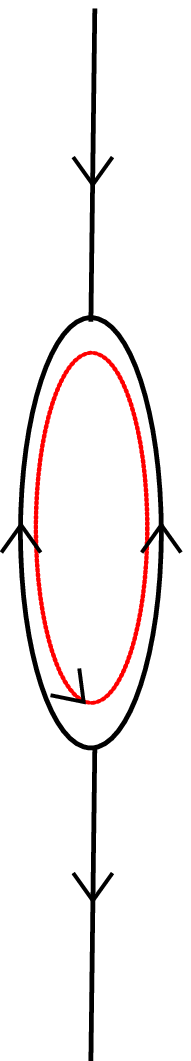}};
 (-6,-10)*{\scriptstyle \text{wt}=1};
 (13,0)*{\longmapsto};
 (13,-7)*{\includegraphics[height=.06\textheight]{res/figs/basis/digonem.eps}};
 (26,0)*{\includegraphics[height=.15\textheight]{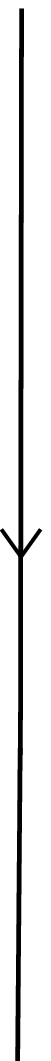}};
 (32,-10)*{\scriptstyle \text{wt}=0};
 (0,-40)*{\includegraphics[height=.15\textheight]{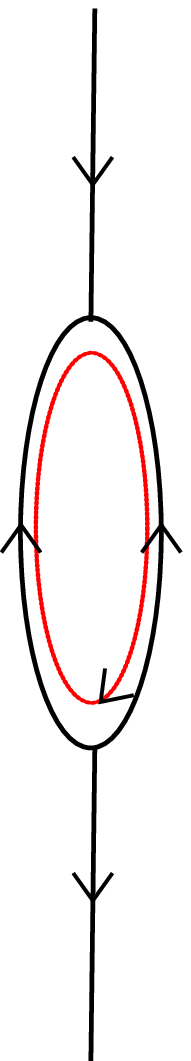}};
 (-6,-50)*{\scriptstyle \text{wt}=-1};
 (13,-40)*{\longmapsto};
 (13,-47)*{\includegraphics[height=.06\textheight]{res/figs/basis/digondot.eps}};
 (26,-40)*{\includegraphics[height=.15\textheight]{res/figs/basis/lineem.eps}};
 (32,-50)*{\scriptstyle \text{wt}=0};
 \endxy\;\;
 \xy
 (0,0)*{\includegraphics[height=.15\textheight]{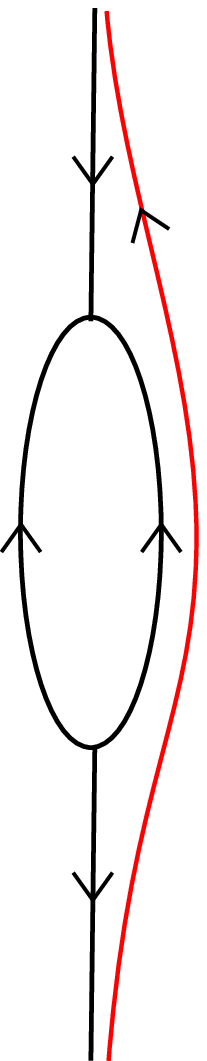}};
 (-6,-10)*{\scriptstyle \text{wt}=1};
 (13,0)*{\longmapsto};
 (13,-7)*{\includegraphics[height=.06\textheight]{res/figs/basis/digonem.eps}};
 (26,0)*{\includegraphics[height=.15\textheight]{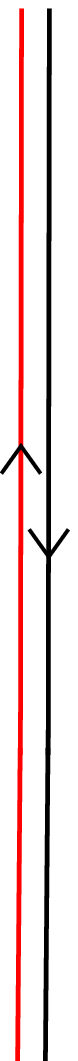}};
 (32,-10)*{\scriptstyle \text{wt}=0};
 (0,-40)*{\includegraphics[height=.15\textheight]{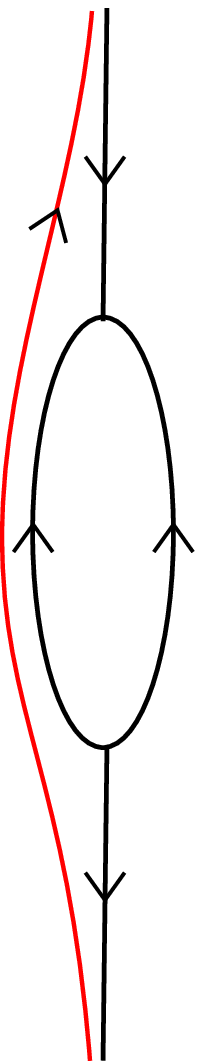}};
 (-6,-50)*{\scriptstyle \text{wt}=-1};
 (13,-40)*{\longmapsto};
 (13,-47)*{\includegraphics[height=.06\textheight]{res/figs/basis/digondot.eps}};
 (26,-40)*{\includegraphics[height=.15\textheight]{res/figs/basis/lineup.eps}};
 (32,-50)*{\scriptstyle \text{wt}=0};
 \endxy
\]
and the \textit{vertical square removals} are
\[
\xy
 (0,0)*{\includegraphics[height=.08\textheight]{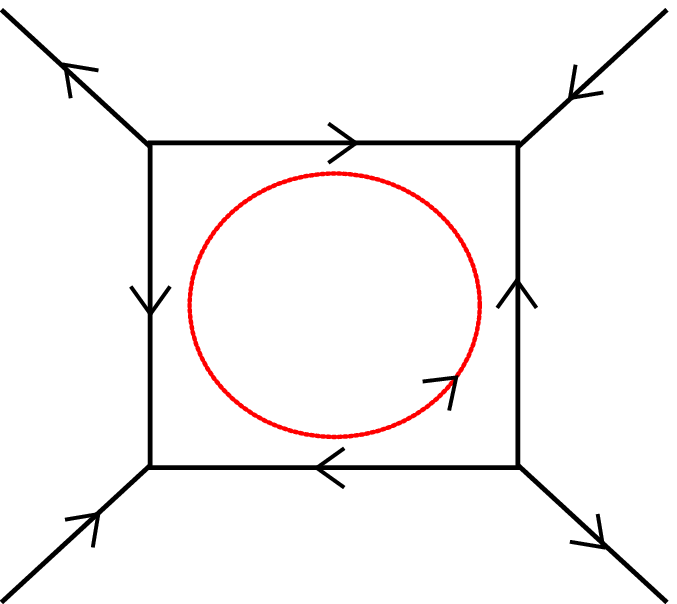}};
 (0,-10)*{\scriptstyle \text{wt}=0};
 (17,0)*{\longmapsto};
 (17,-7.5)*{\includegraphics[height=.055\textheight]{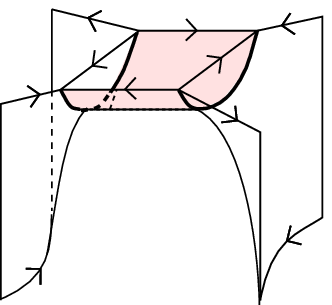}};
 (33,0)*{\includegraphics[height=.08\textheight]{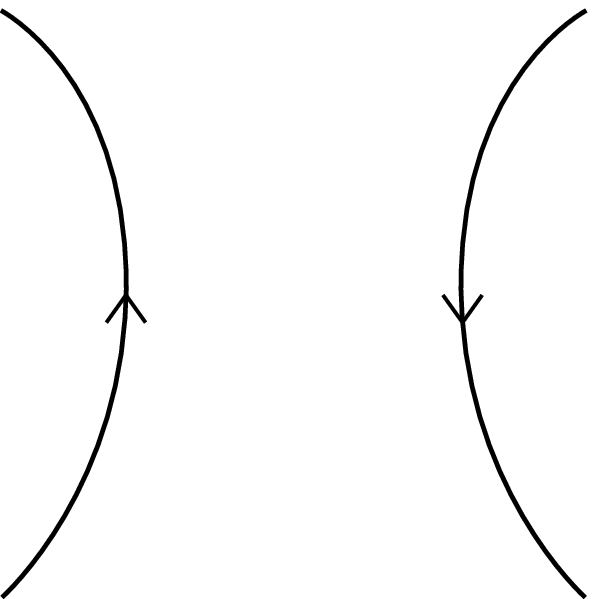}};
 (33,-10)*{\scriptstyle \text{wt}=0};
 (0,-25)*{\includegraphics[height=.08\textheight]{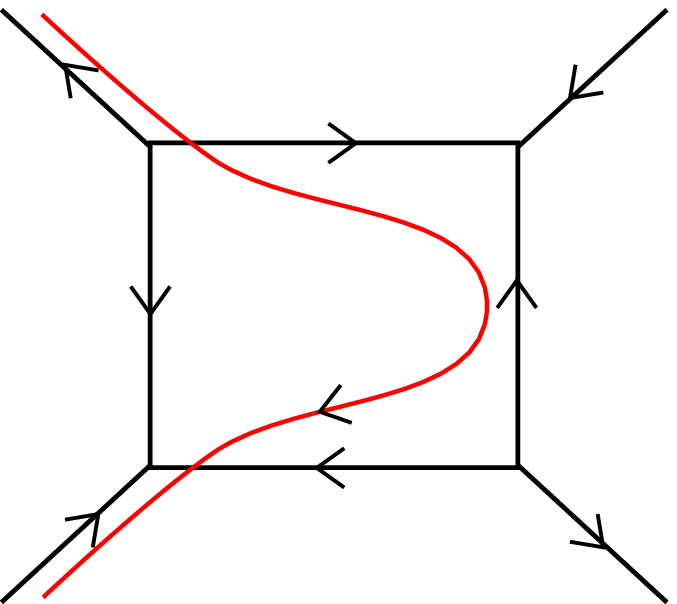}};
 (0,-35)*{\scriptstyle \text{wt}=0};
 (17,-25)*{\longmapsto};
 (17,-32.5)*{\includegraphics[height=.055\textheight]{res/figs/basis/squarevert.eps}};
 (33,-25)*{\includegraphics[height=.08\textheight]{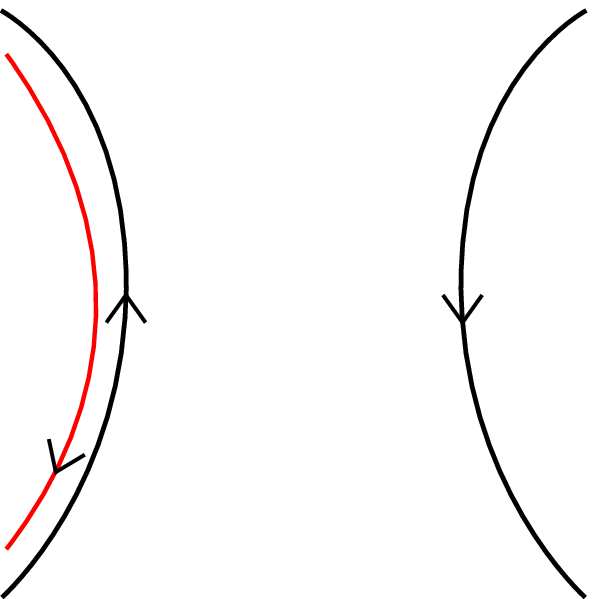}};
 (33,-35)*{\scriptstyle \text{wt}=0};
 (0,-50)*{\includegraphics[height=.08\textheight]{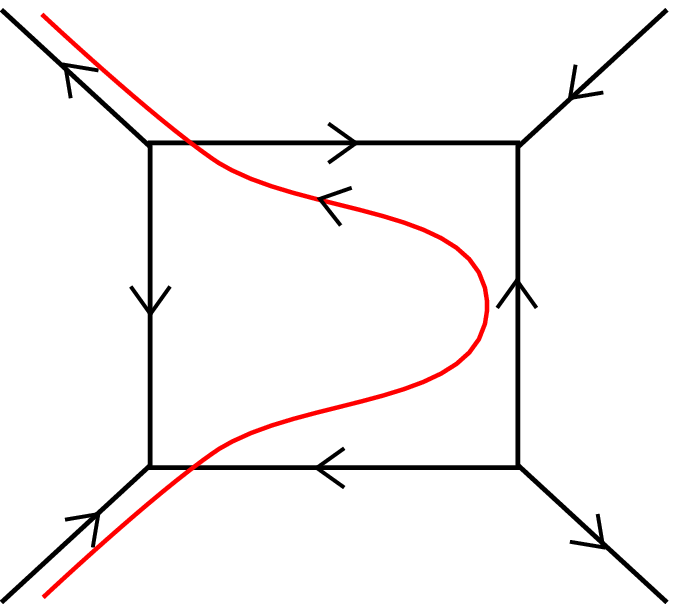}};
 (0,-60)*{\scriptstyle \text{wt}=0};
 (17,-50)*{\longmapsto};
 (17,-57.5)*{\includegraphics[height=.055\textheight]{res/figs/basis/squarevert.eps}};
 (33,-50)*{\includegraphics[height=.08\textheight]{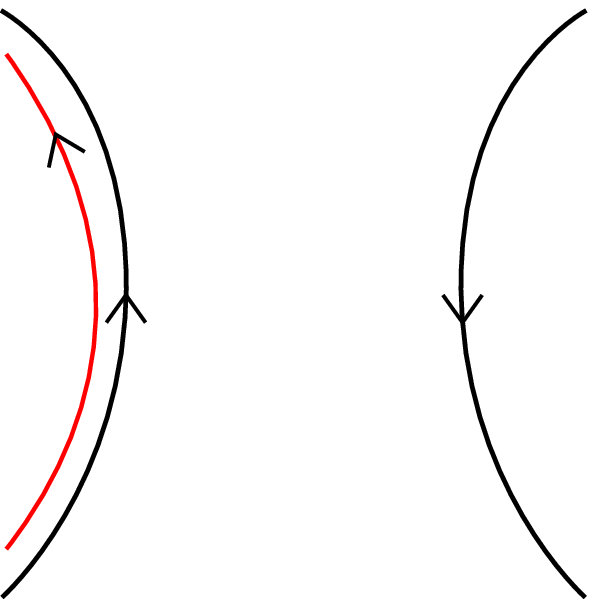}};
 (33,-60)*{\scriptstyle \text{wt}=0};
 \endxy\;\;\;\;
 \xy
 (0,0)*{\includegraphics[height=.08\textheight]{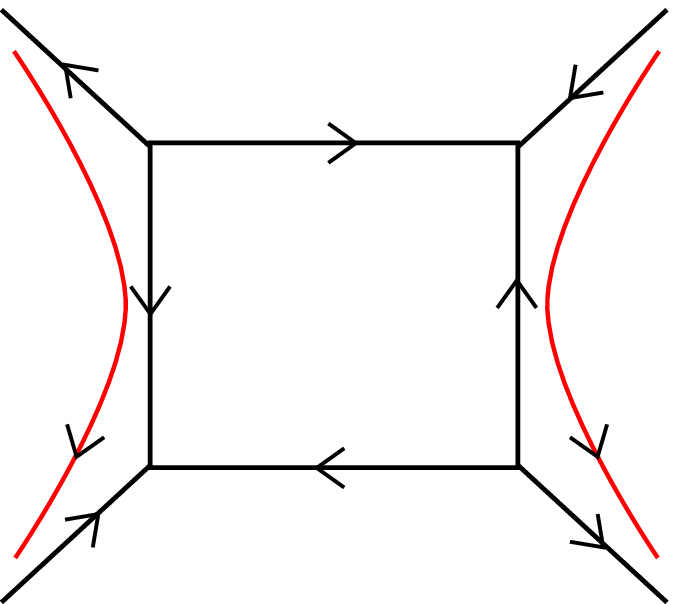}};
 (0,-10)*{\scriptstyle \text{wt}=0};
 (17,0)*{\longmapsto};
 (17,-7.5)*{\includegraphics[height=.055\textheight]{res/figs/basis/squarevert.eps}};
 (33,0)*{\includegraphics[height=.08\textheight]{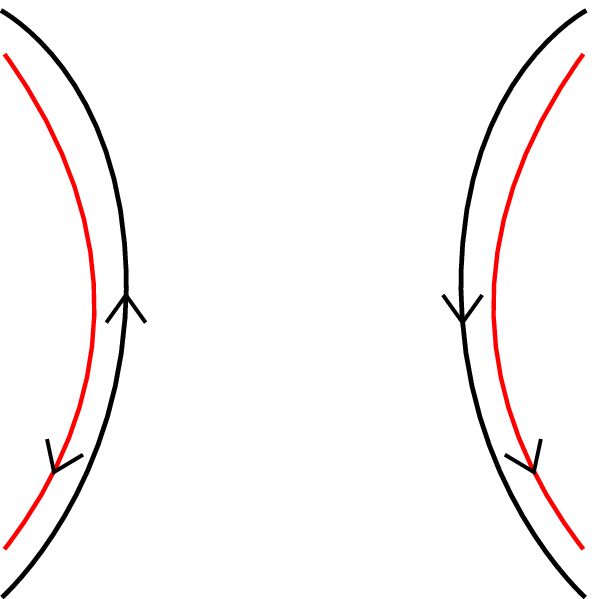}};
 (33,-10)*{\scriptstyle \text{wt}=0};
 (0,-25)*{\includegraphics[height=.08\textheight]{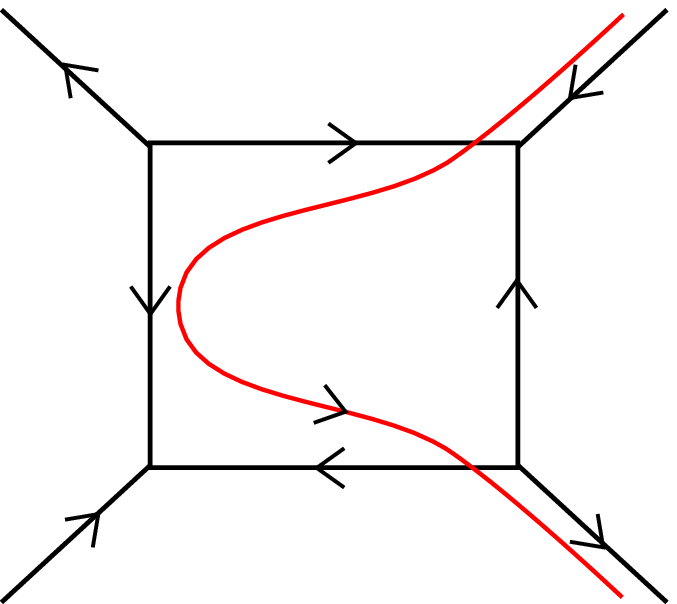}};
 (0,-35)*{\scriptstyle \text{wt}=0};
 (17,-25)*{\longmapsto};
 (17,-32.5)*{\includegraphics[height=.055\textheight]{res/figs/basis/squarevert.eps}};
 (33,-25)*{\includegraphics[height=.08\textheight]{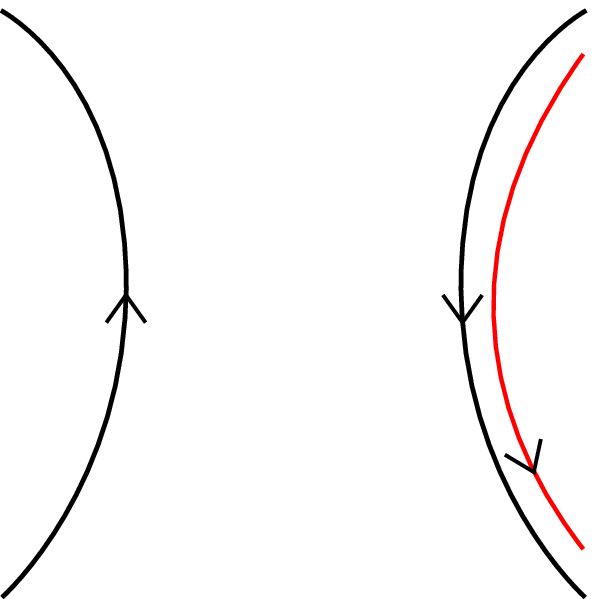}};
 (33,-35)*{\scriptstyle \text{wt}=0};
 (0,-50)*{\includegraphics[height=.08\textheight]{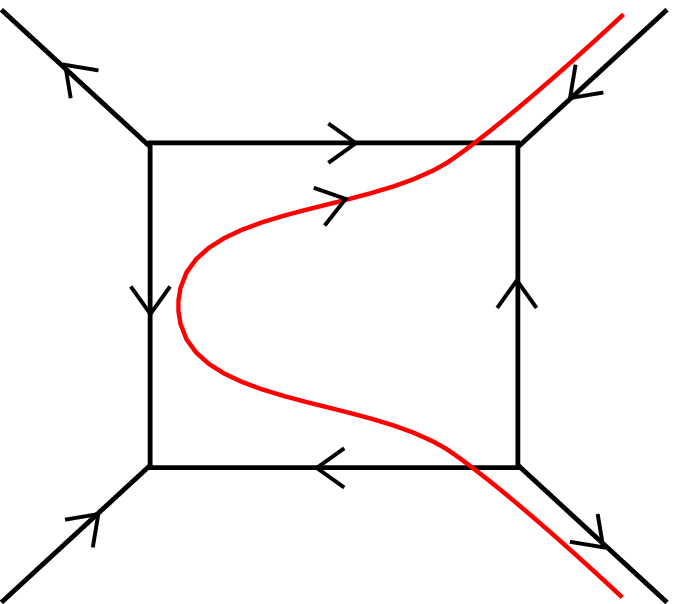}};
 (0,-60)*{\scriptstyle \text{wt}=0};
 (17,-50)*{\longmapsto};
 (17,-57.5)*{\includegraphics[height=.055\textheight]{res/figs/basis/squarevert.eps}};
 (33,-50)*{\includegraphics[height=.08\textheight]{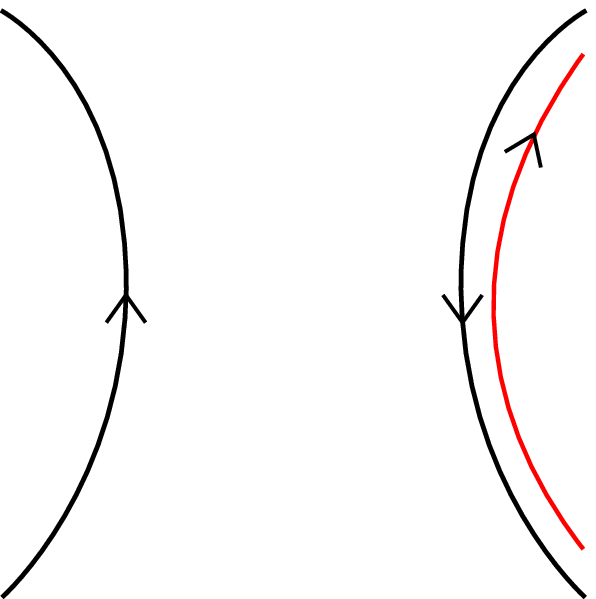}};
 (33,-60)*{\scriptstyle \text{wt}=0};
 \endxy\;\;\;\;
 \xy
 (0,0)*{\includegraphics[height=.08\textheight]{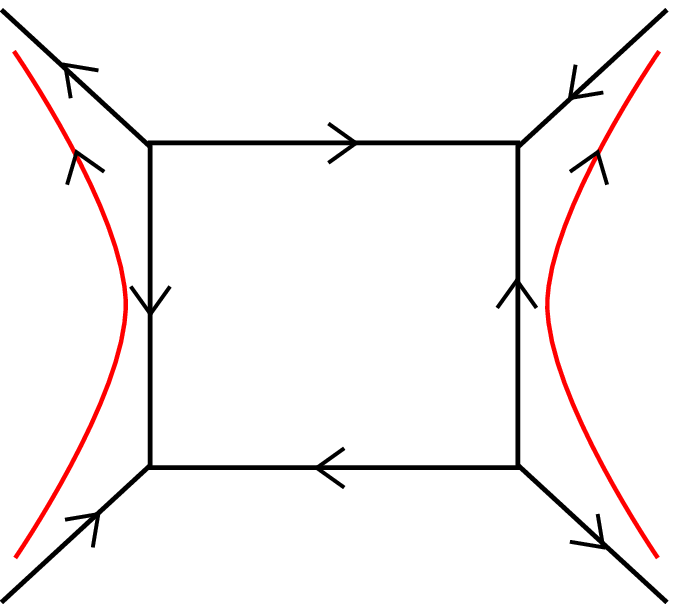}};
 (0,-10)*{\scriptstyle \text{wt}=0};
 (17,0)*{\longmapsto};
 (17,-7.5)*{\includegraphics[height=.055\textheight]{res/figs/basis/squarevert.eps}};
 (33,0)*{\includegraphics[height=.08\textheight]{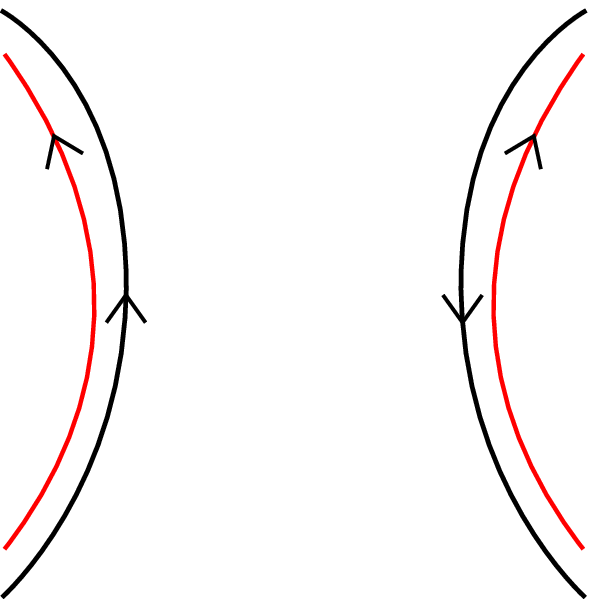}};
 (33,-10)*{\scriptstyle \text{wt}=0};
 (0,-25)*{\includegraphics[height=.08\textheight]{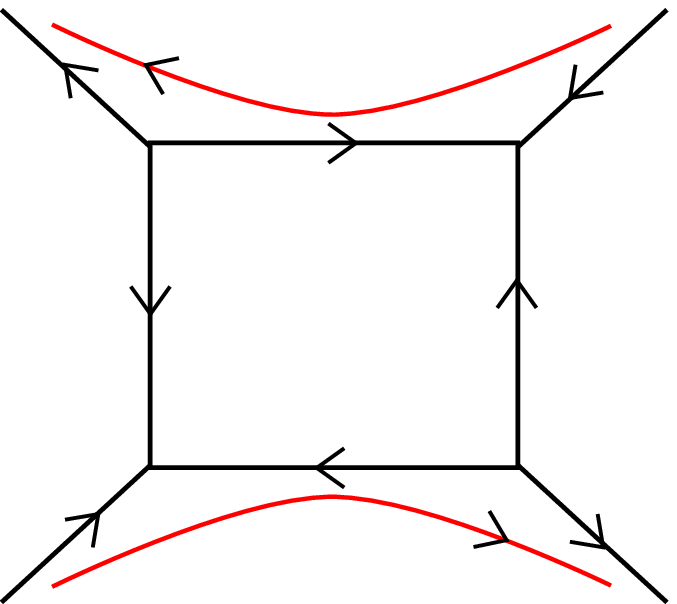}};
 (0,-35)*{\scriptstyle \text{wt}=0};
 (17,-25)*{\longmapsto};
 (17,-32.5)*{\includegraphics[height=.055\textheight]{res/figs/basis/squarevert.eps}};
 (33,-25)*{\includegraphics[height=.08\textheight]{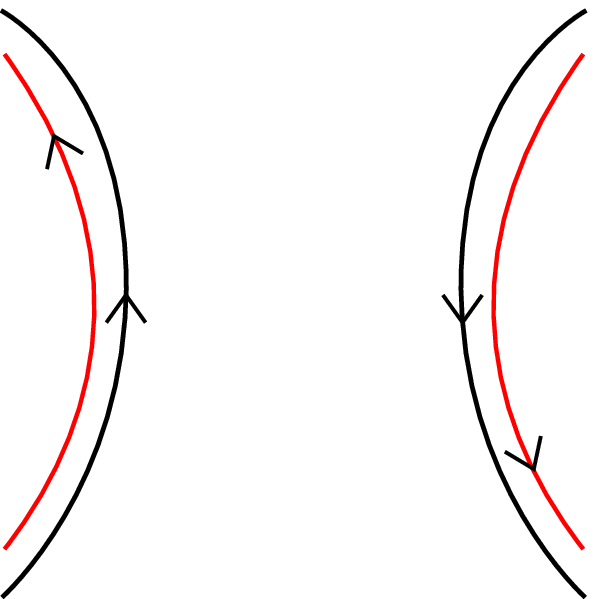}};
 (33,-35)*{\scriptstyle \text{wt}=0};
 (0,-50)*{\includegraphics[height=.08\textheight]{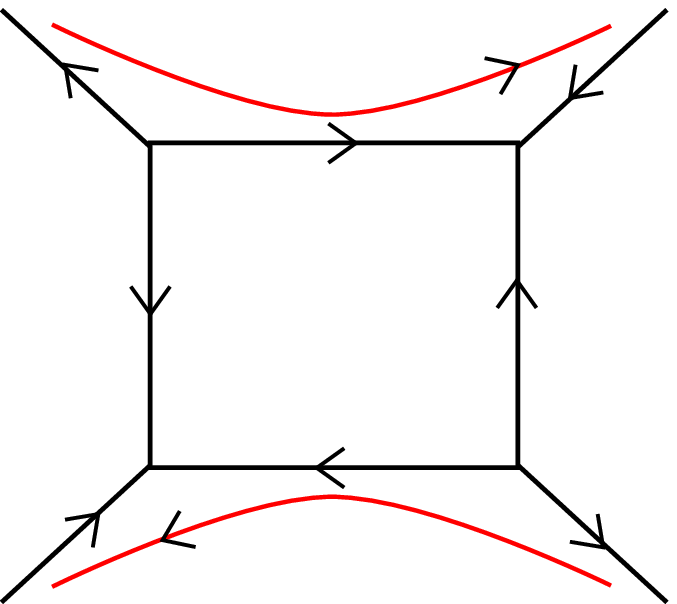}};
 (0,-60)*{\scriptstyle \text{wt}=0};
 (17,-50)*{\longmapsto};
 (17,-57.5)*{\includegraphics[height=.055\textheight]{res/figs/basis/squarevert.eps}};
 (33,-50)*{\includegraphics[height=.08\textheight]{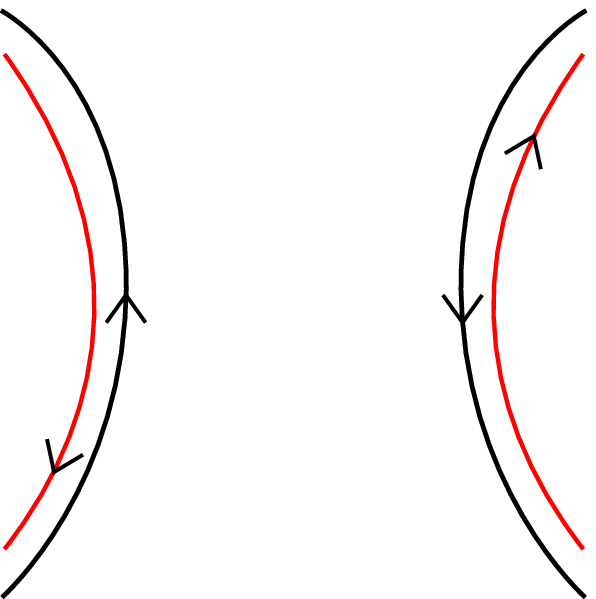}};
 (33,-60)*{\scriptstyle \text{wt}=0};
 \endxy
\]
and the \textit{horizontal square removals} are
\[
\xy
 (0,0)*{\includegraphics[height=.08\textheight]{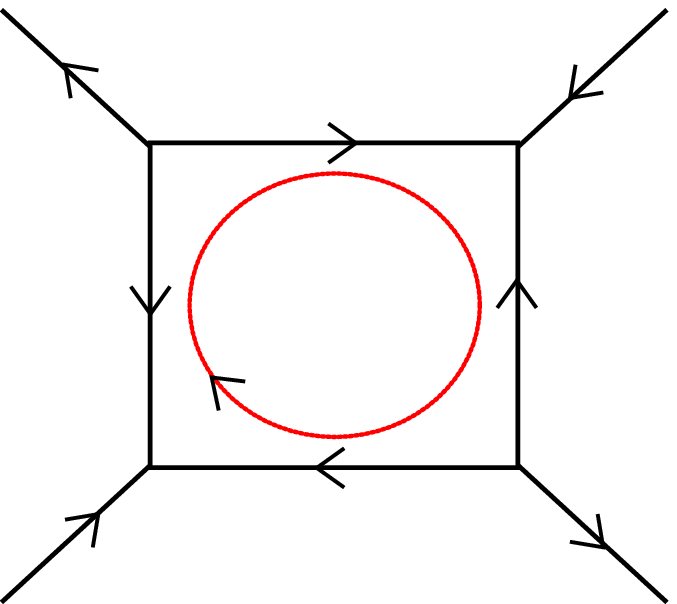}};
 (0,-10)*{\scriptstyle \text{wt}=0};
 (17,0)*{\longmapsto};
 (17,-7.5)*{\includegraphics[height=.055\textheight]{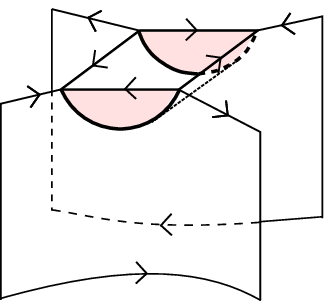}};
 (33,0)*{\reflectbox{\includegraphics[height=.08\textheight,angle=90]{res/figs/basis/vert1.eps}}};
 (33,-10)*{\scriptstyle \text{wt}=0};
 (0,-25)*{\includegraphics[height=.08\textheight]{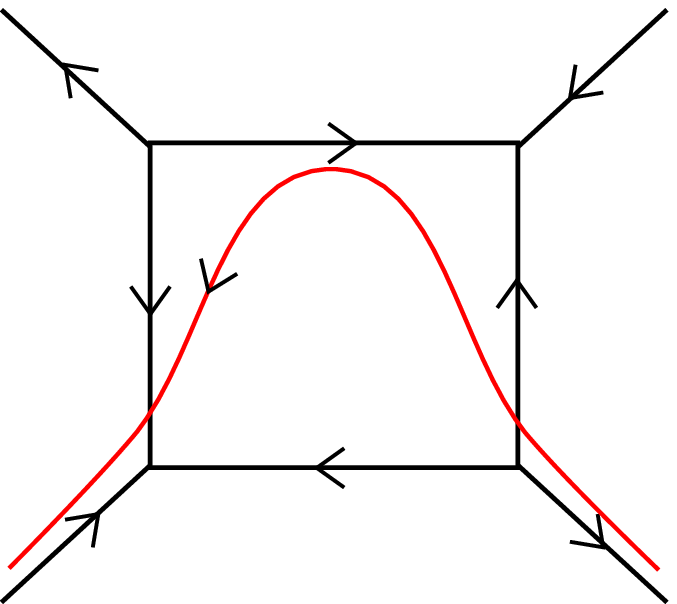}};
 (0,-35)*{\scriptstyle \text{wt}=1};
 (17,-25)*{\longmapsto};
 (17,-32.5)*{\includegraphics[height=.055\textheight]{res/figs/basis/squarehorz.eps}};
 (33,-25)*{\reflectbox{\includegraphics[height=.08\textheight,angle=90]{res/figs/basis/vert2.eps}}};
 (33,-35)*{\scriptstyle \text{wt}=1};
 (0,-50)*{\includegraphics[height=.08\textheight]{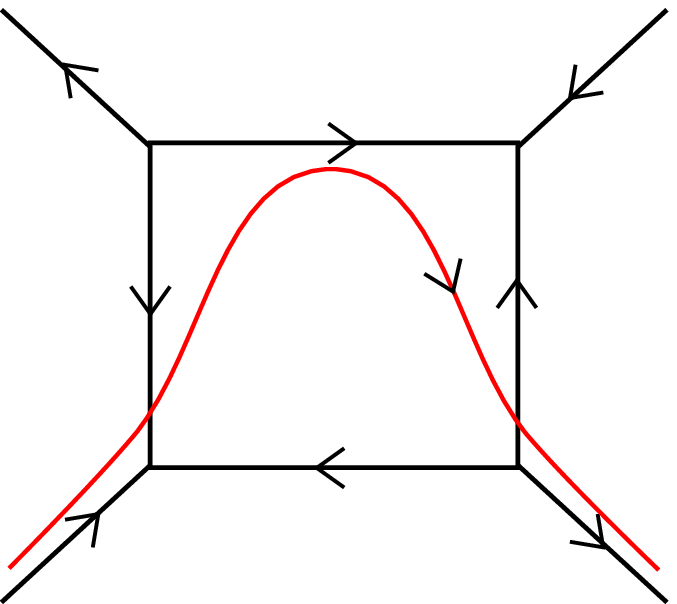}};
 (0,-60)*{\scriptstyle \text{wt}=-1};
 (17,-50)*{\longmapsto};
 (17,-57.5)*{\includegraphics[height=.055\textheight]{res/figs/basis/squarehorz.eps}};
 (33,-50)*{\reflectbox{\includegraphics[height=.08\textheight,angle=90]{res/figs/basis/vert3.eps}}};
 (33,-60)*{\scriptstyle \text{wt}=-1};
 \endxy\;\;\;\;
 \xy
 (0,0)*{\includegraphics[height=.08\textheight]{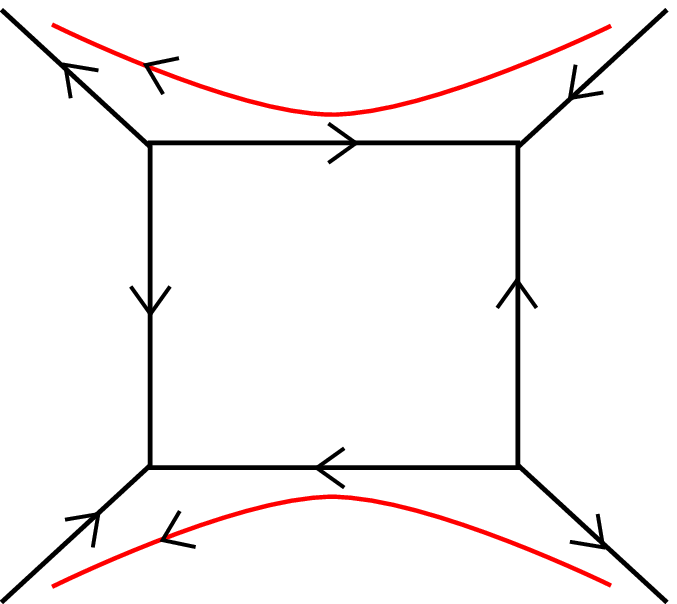}};
 (0,-10)*{\scriptstyle \text{wt}=0};
 (17,0)*{\longmapsto};
 (17,-7.5)*{\includegraphics[height=.055\textheight]{res/figs/basis/squarehorz.eps}};
 (33,0)*{\reflectbox{\includegraphics[height=.08\textheight,angle=90]{res/figs/basis/vert6.eps}}};
 (33,-10)*{\scriptstyle \text{wt}=0};
 (0,-25)*{\includegraphics[height=.08\textheight]{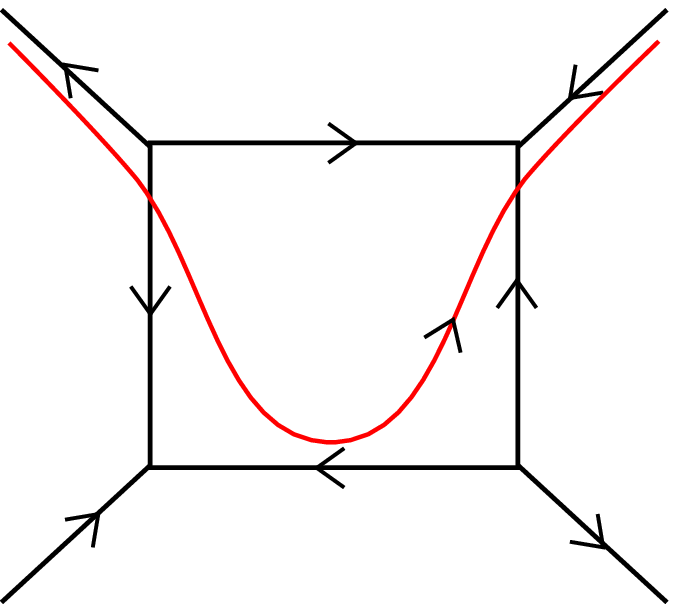}};
 (0,-35)*{\scriptstyle \text{wt}=1};
 (17,-25)*{\longmapsto};
 (17,-32.5)*{\includegraphics[height=.055\textheight]{res/figs/basis/squarehorz.eps}};
 (33,-25)*{\reflectbox{\includegraphics[height=.08\textheight,angle=90]{res/figs/basis/vert4.eps}}};
 (33,-35)*{\scriptstyle \text{wt}=1};
 (0,-50)*{\includegraphics[height=.08\textheight]{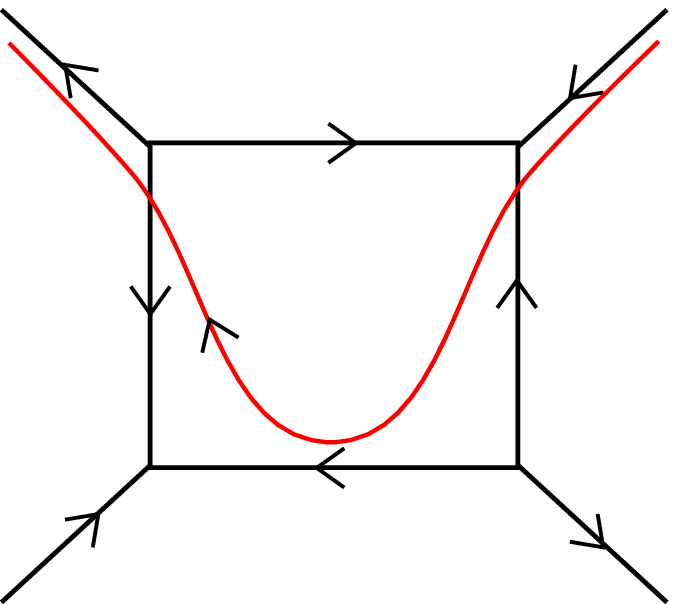}};
 (0,-60)*{\scriptstyle \text{wt}=-1};
 (17,-50)*{\longmapsto};
 (17,-57.5)*{\includegraphics[height=.055\textheight]{res/figs/basis/squarehorz.eps}};
 (33,-50)*{\reflectbox{\includegraphics[height=.08\textheight,angle=90]{res/figs/basis/vert5.eps}}};
 (33,-60)*{\scriptstyle \text{wt}=-1};
 \endxy\;\;\;\;
 \xy
 (0,0)*{\includegraphics[height=.08\textheight]{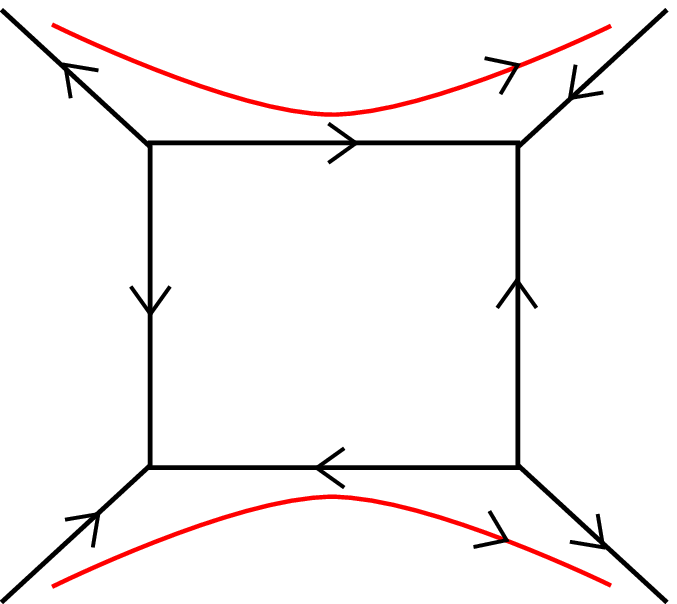}};
 (0,-10)*{\scriptstyle \text{wt}=0};
 (17,0)*{\longmapsto};
 (17,-7.5)*{\includegraphics[height=.055\textheight]{res/figs/basis/squarehorz.eps}};
 (33,0)*{\reflectbox{\includegraphics[height=.08\textheight,angle=90]{res/figs/basis/vert8.eps}}};
 (33,-10)*{\scriptstyle \text{wt}=0};
 (0,-25)*{\includegraphics[height=.08\textheight]{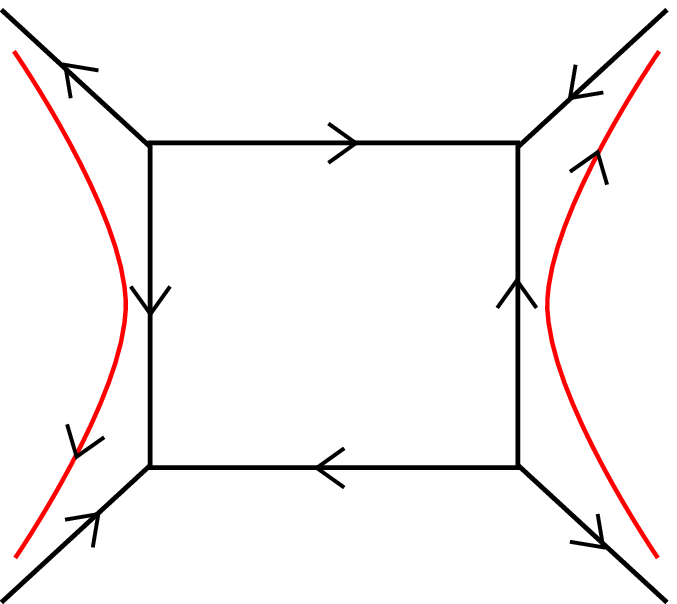}};
 (0,-35)*{\scriptstyle \text{wt}=2};
 (17,-25)*{\longmapsto};
 (17,-32.5)*{\includegraphics[height=.055\textheight]{res/figs/basis/squarehorz.eps}};
 (33,-25)*{\reflectbox{\includegraphics[height=.08\textheight,angle=90]{res/figs/basis/vert7.eps}}};
 (33,-35)*{\scriptstyle \text{wt}=2};
 (0,-50)*{\includegraphics[height=.08\textheight]{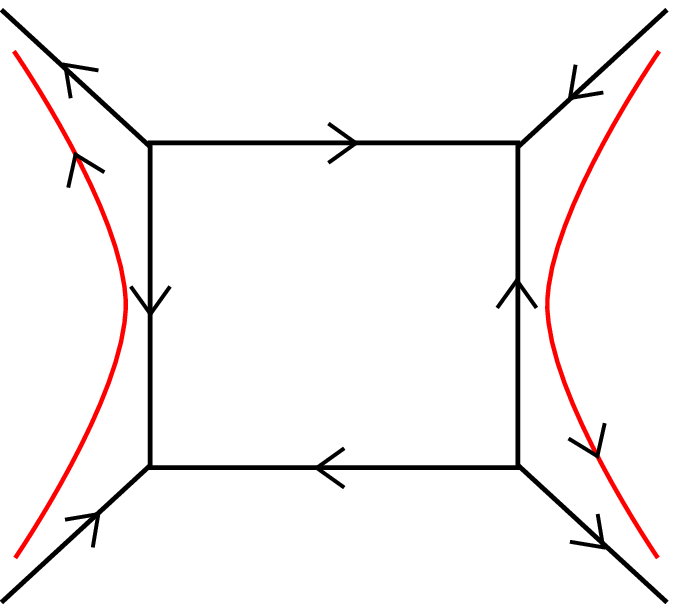}};
 (0,-60)*{\scriptstyle \text{wt}=-2};
 (17,-50)*{\longmapsto};
 (17,-57.5)*{\includegraphics[height=.055\textheight]{res/figs/basis/squarehorz.eps}};
 (33,-50)*{\reflectbox{\includegraphics[height=.08\textheight,angle=90]{res/figs/basis/vert9.eps}}};
 (33,-60)*{\scriptstyle \text{wt}=-2};
 \endxy
\]
It should be noted that none of the faces pictured above has to intersect the cut line, but we still call the two different square removals usually vertical and horizontal.

Furthermore, we call any fixed way in which this procedure is applied a \textit{removal (of the faces)} of the web $w$.
\end{defn}
\begin{prop}\label{prop-resolution}
Let $w=u^*v$ be a closed web. Let $f$ be a flow on $w$ of weight $k$. Moreover, fix an order in which to remove the faces of the web $w$.
\begin{itemize}
\item[(a)] The foam $f_{w_{f}}\colon \emptyset \to w\in {}_uK_v$ obtained from the face removing algorithm~\ref{defn-procedure} has a $q$-degree of $-k+n$. 
\item[(b)] If $u=v$ and $f=c$ is the canonical flow, then $f_{w_{f}}\colon \emptyset \to w\in {}_uK_u$ is the identity foam.
\end{itemize}
\end{prop}
\begin{proof}
(a) The algorithm uses three different elementary foams, i.e. a circle removal, a digon removal and a square removal. It is easy to check that the $q$-degree of these foams are $-2,-1$ and $0$ respectively. Note that the change of weight, as indicated in the figures in Definition~\ref{defn-procedure}, correspond exactly to minus the $q$-degree of the associated foam. For example, the removal of a circle with a counterclockwise flow lowers the weight by $2$ and the associated foam has $q$-degree $-2$. Hence, the resulting foam $f_{w_f}$ will be of the $q$-degree $-k+n$.

(b) First we prove that any removal of $w_{c}$ is a combination of dot-free removal of circle and digon faces and vertical removal of square faces (by which we mean a removal which leaves intact the two sides of the square that lie perpendicular to the cut line) each of which intersects the cut line. That is, faces removed in each step of the removal of $w_{c}$ are as follows (the dashed line represents the cut line in each figure).
\[
\xy
 (0,0)*{\includegraphics[height=.045\textheight]{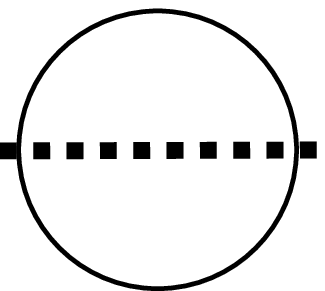}};
\endxy\;\;\;\;
\xy
 (0,0)*{\includegraphics[height=.15\textheight]{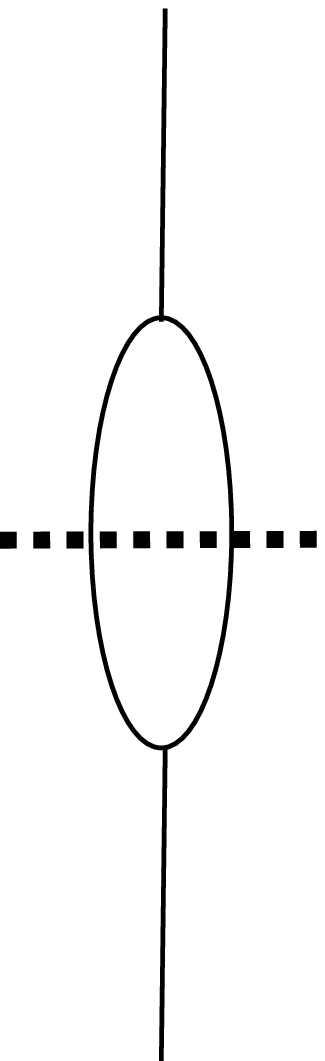}};
\endxy\;\;\;\;
\xy
 (0,0)*{\includegraphics[height=.08\textheight]{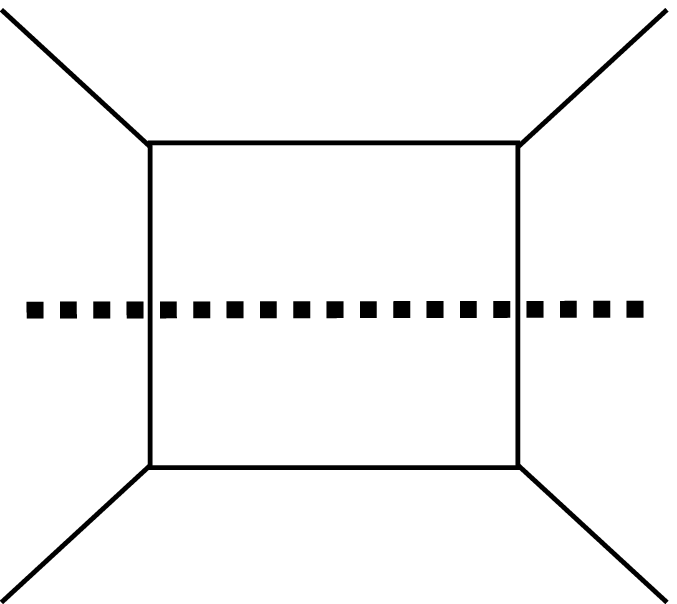}};
\endxy
\]
We proceed by induction on the number of steps in the removal of $w_{c}$, i.e. we show that before each step all faces with four or less edges lie on the cut line and the removal in each step is either a dot-free removal of a circle or digon face or a vertical removal of a square face. 
\vskip0.5cm
Since $u$ is a non-elliptic web with boundary, at the initial step of the removal of $w_{c}$, all faces with four sides or less must lie on the cut line. Hence, the face removed in this step is one which intersects the cut line. In addition, since the flow on $w=u^{*}u$ is canonical, we see that circle faces in $w_{c}$ carry the following flow.
\begin{align}\label{fig-cancir}
   \xy
 (0,0)*{\includegraphics[height=.05\textheight]{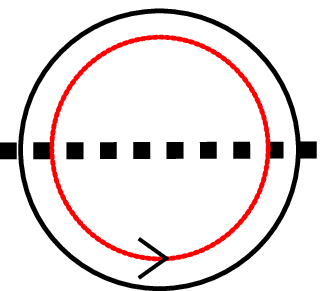}};
\endxy
\end{align}
Moreover, digon faces of $w_{c}$ must carry one of the following three flows.
\begin{align}\label{fig-candigon}
    \xy
 (0,0)*{\includegraphics[height=.15\textheight]{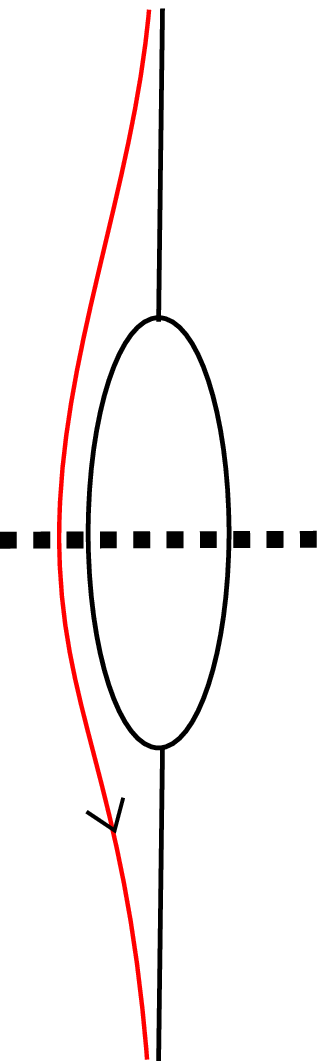}};
\endxy\;\;\;\;\xy
 (0,0)*{\includegraphics[height=.15\textheight]{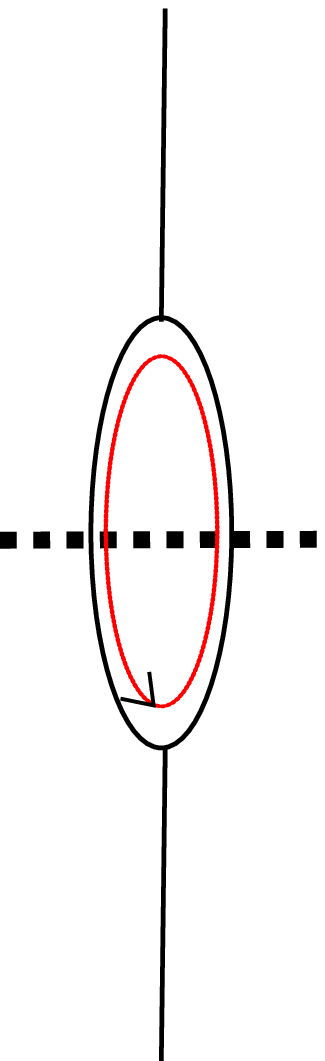}};
\endxy\;\;\;\;\xy
 (0,0)*{\includegraphics[height=.15\textheight]{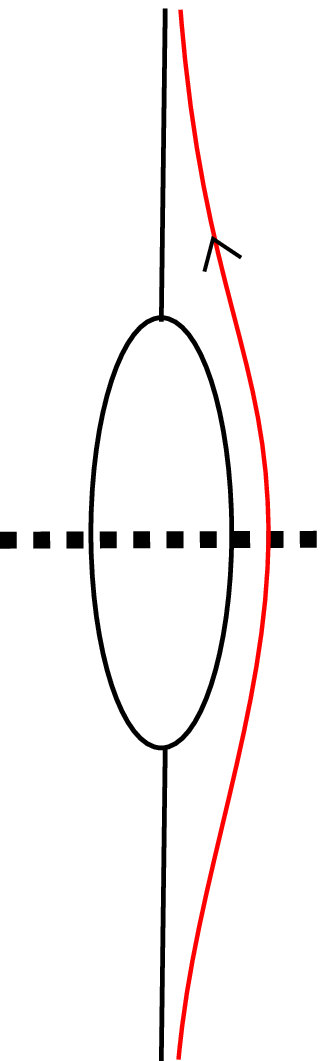}};
\endxy
\end{align}
And finally square faces of $w_{c}$ must carry one of the flows shown below.
\begin{align}\label{fig-cansq}
   \xy
 (0,0)*{\includegraphics[height=.08\textheight]{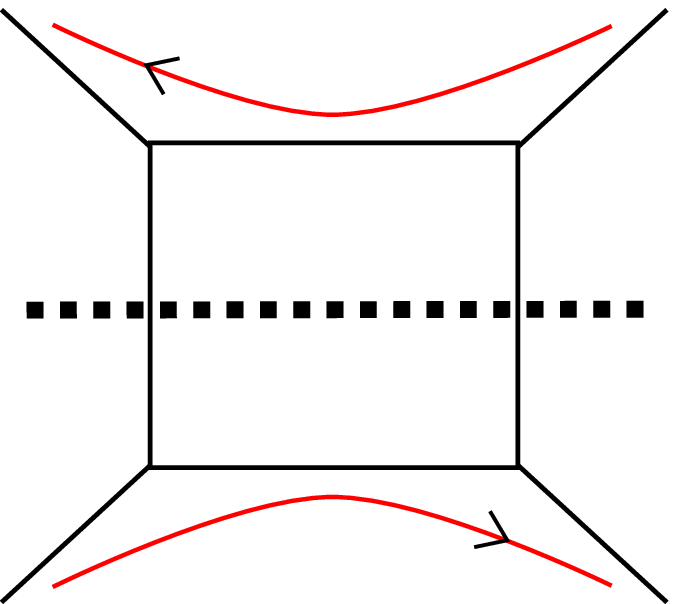}};
\endxy\;\;\;\;\xy
 (0,0)*{\includegraphics[height=.08\textheight]{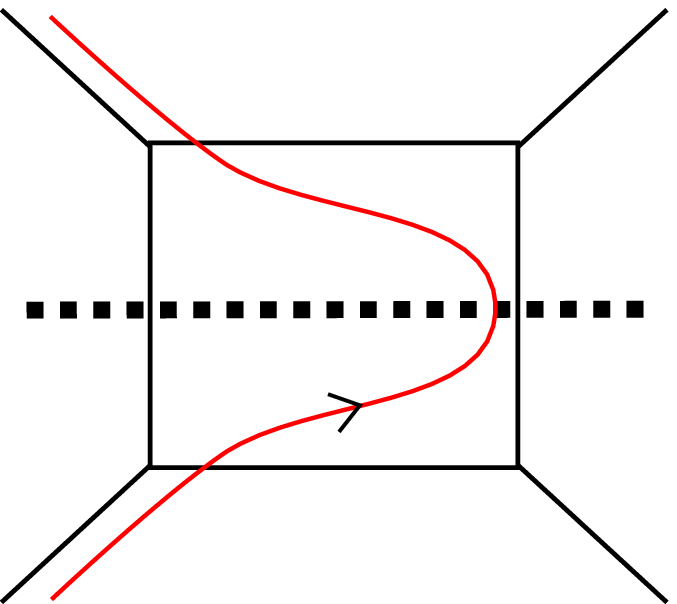}};
\endxy\;\;\;\;\xy
 (0,0)*{\includegraphics[height=.08\textheight]{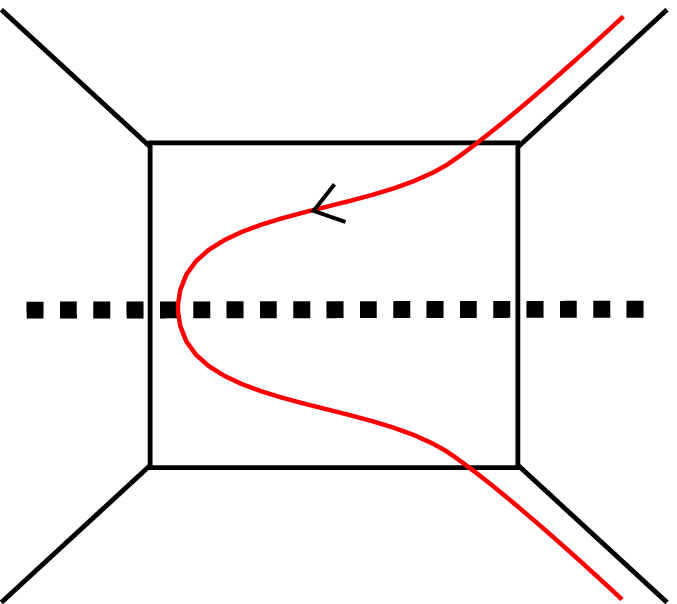}};
\endxy
\end{align}
Thus, at the initial step of the procedure, a removal of a circle and digon face is dot-free and a removal of a square face is vertical (for both orientations of the square).

Assume that in the $n$-th step in the removal, all faces with four sides or less lie on the cut line and the removal is either a dot-free removal of a circle face, a dot-free removal of a digon or vertical removal of a square. In the case that the face removed in the $n$-th step is a circle or a digon, faces with four sides or less in the $(n+1)$-th step must again be on the cut line, since removal of circles or digons on the cut line clearly do not affect the internal faces of $u$ and $u^{*}$. Otherwise, if the face removed in the $n$-th step is a square, we have two possibilities, either the $F_{1} = F_{2}$ (i.e. it is the external face of $w=u^{*}u$ or connects somewhere by crossing the cut line) or $F_{1}$ and $F_{2}$ are different faces of $u$ and $u^{*}$ respectively. The result of removing this square face vertical is pictured below.
\[
\xy
 (0,0)*{\includegraphics[height=.08\textheight]{res/figs/basis/sqcut.eps}};
 (0,8)*{F_1};
 (0,-8)*{F_2};
\endxy
\longmapsto
\xy
 (0,0)*{\includegraphics[height=.08\textheight]{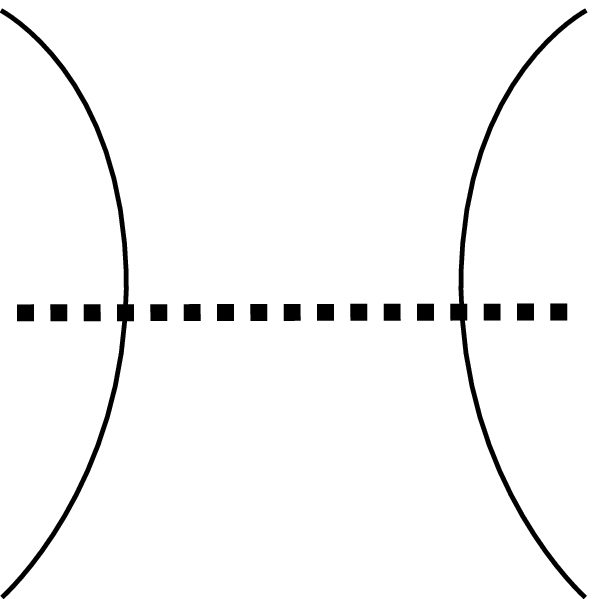}};
 (0,8)*{F};
\endxy
\] 
In the former situation, removal of the square face has no affect on the internal faces of $u$ and $u^{*}$ and thus faces with four sides or less in the $(n+1)$-th step must again be on the cut line. In the latter situation, the removal of the square face results in the creation of a new face on the cut line. Note that no other internal faces of $u$ and $u^{*}$ are affected. Here again, faces with four sides or less in the $(n+1)$-th step must be on the cut line.

Furthermore, beginning with a symmetric web with the canonical flow, then any removal of a circle face, a digon face or the vertical removal of a square face on the cut line results in another web with the canonical flow, due to locality of the removal conventions above. Thus, we see that in the $(n+1)$-th step, the faces on the cut line are again of the form in~\ref{fig-cancir},~\ref{fig-candigon} or~\ref{fig-cansq}. That is, removal of circle and digon faces are dot-free and removal of square faces are vertical.

In the following, we consider the removal of $w_{c}$ as a foam in ${}_uK_u$ and denoted it by $f_{w_{c}}$. Since in the removal of $w_{c}$ consists solely of dot-free removal of digon or circle faces or vertical removal of square faces that intersect the cut line, we have that $f_{w_{c}}$ is composed of the following local foams (corresponding to one part of circle, digon and square ratemoval respectively). Note that we use the alternative description of the foam space (see Lemma~\ref{lem:webalgaltern} for example).
\[
\xy
 (0,0)*{\includegraphics[width=210px]{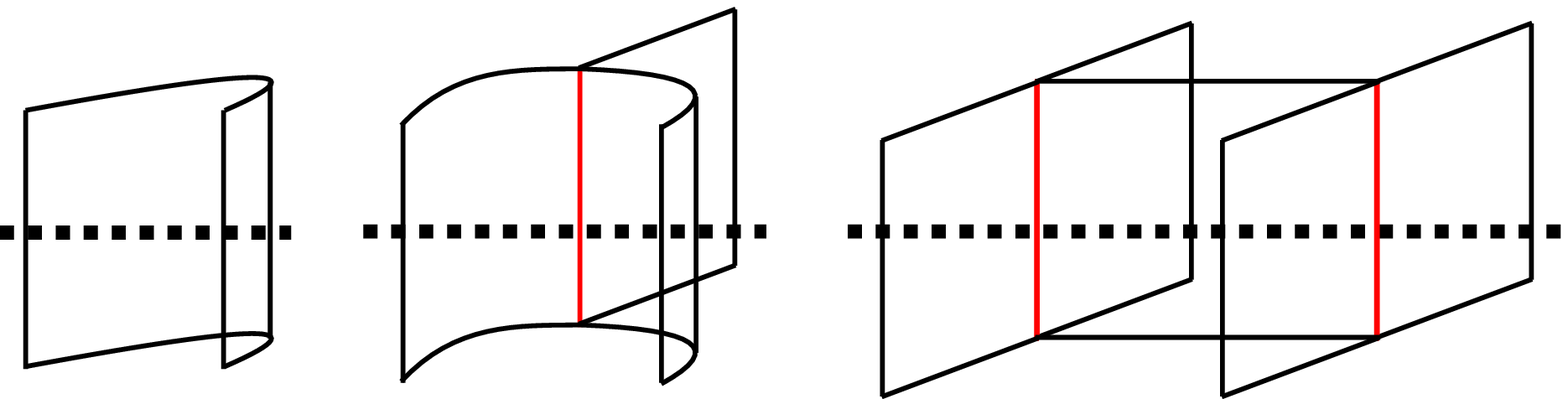}};
\endxy
\]
From this we see that $f_{w_{c}}$ must be the identity foam in ${}_uK_u$.
\end{proof}
\begin{ex}\label{ex-reso1}
We illustrate now that the order of removing faces is important, in contrast to the Kuperberg bracket that is independent of the order of face removals. In particular, if one uses different orders for different flows, then it can not be guaranteed that the resulting foams are linear independent.
\begin{itemize}
\item[(a)] Consider the theta web $w_1$ and the two flows illustrated below.
\[
\xy
 (0,0)*{\includegraphics[width=80px]{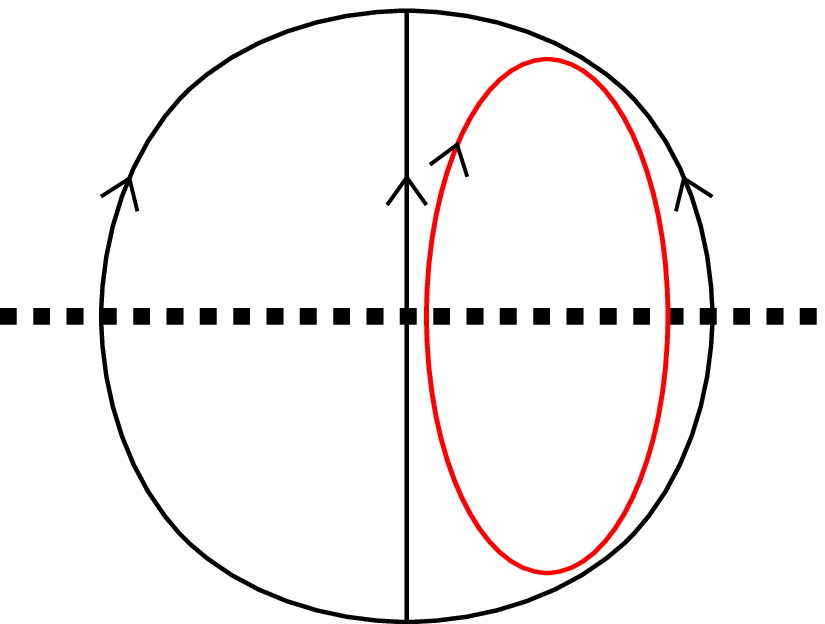}};
\endxy\;\;\;\;
\xy
 (0,0)*{\includegraphics[width=80px]{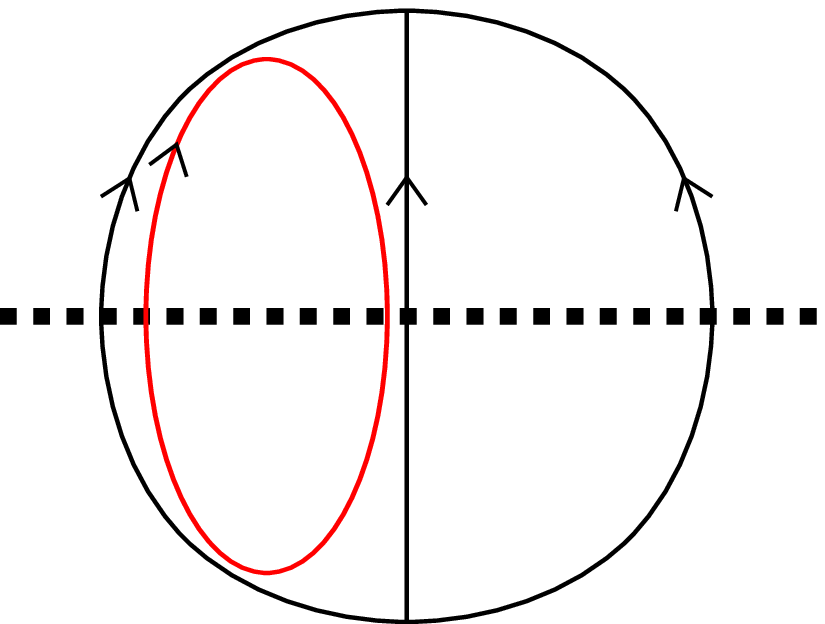}};
\endxy
\]
Then there are two methods to remove the faces, i.e. from right to left or vice versa. Removing the faces from right to left gives the left (right) foam illustrated below for the left (right) flow.
\[
\xy
 (0,0)*{\includegraphics[width=80px]{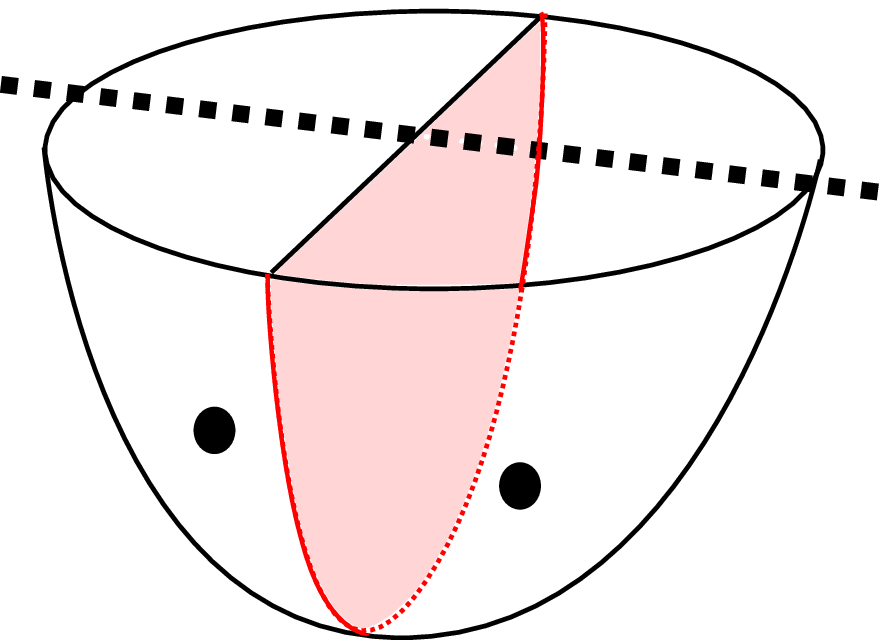}};
\endxy\;\;\;\;
\xy
 (0,0)*{\includegraphics[width=80px]{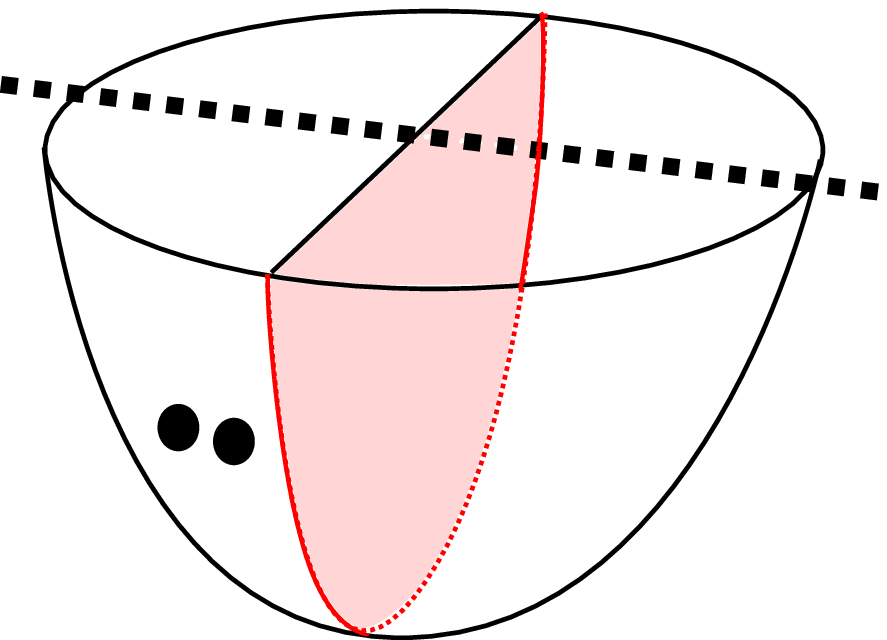}};
\endxy
\]
In contrast, removing faces from left to right gives the following result.
\[
\xy
 (0,0)*{\includegraphics[width=80px]{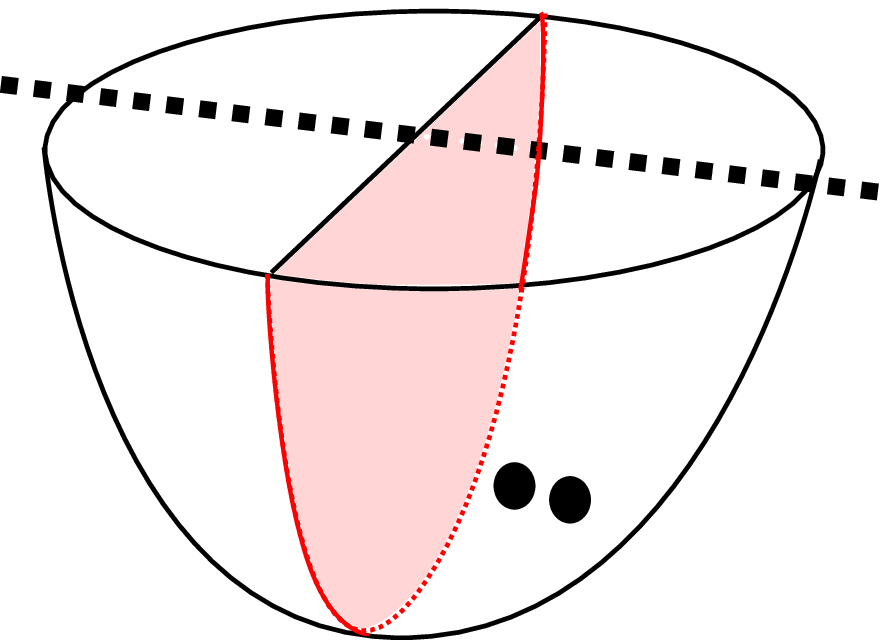}};
\endxy\;\;\;\;
\xy
 (0,0)*{\includegraphics[width=80px]{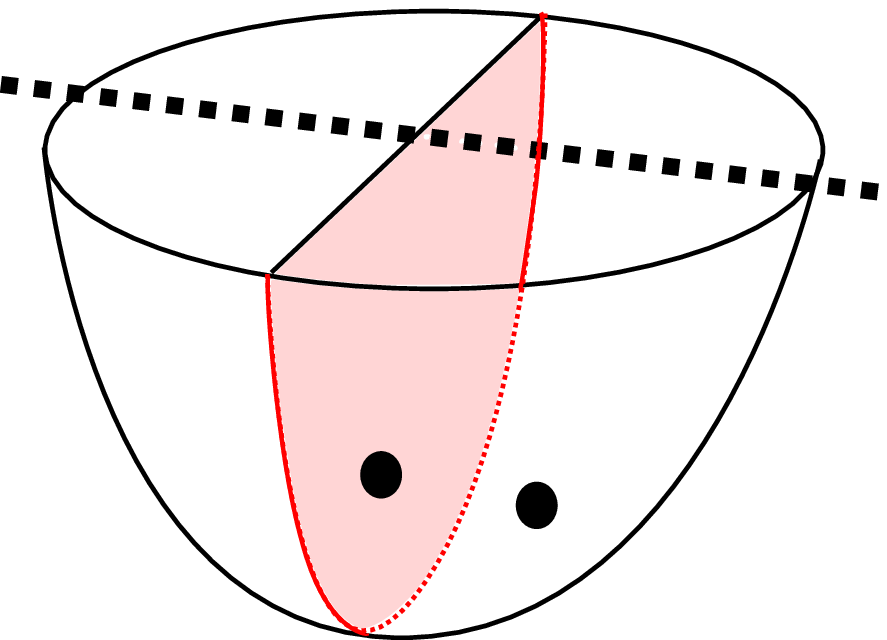}};
\endxy
\]
Note that (b) of Proposition~\ref{prop-resolution} is therefore not trivial because of the possible existence of non-trivial idempotents.
\item[(b)] Another example is the following. Consider the web $w_2$ also called ``square with digon ears'' and the two flows on $w_2$ pictured below.
\[
\xy
 (0,0)*{\includegraphics[width=110px]{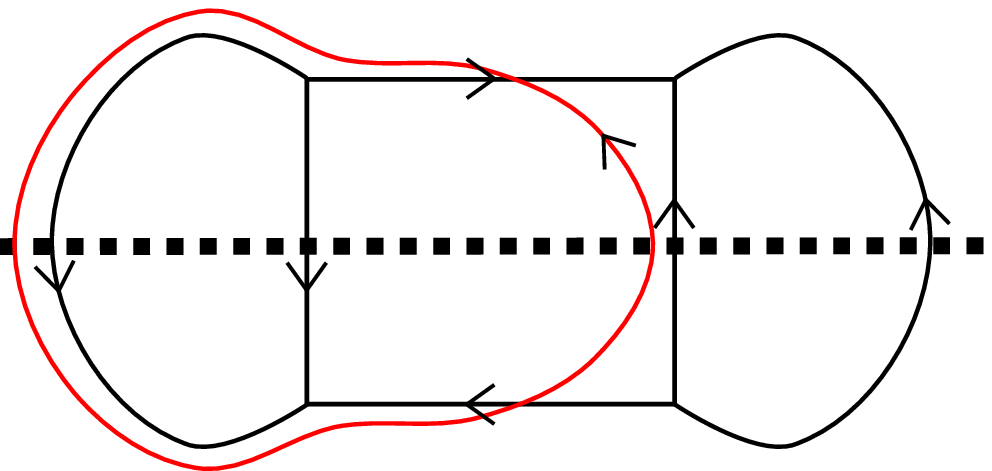}};
\endxy\;\;\;\;
\xy
 (0,0)*{\includegraphics[width=110px]{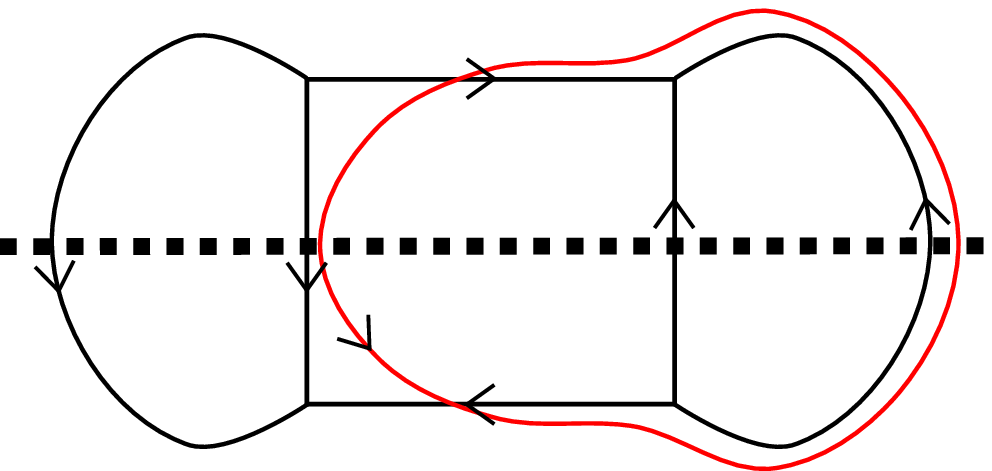}};
\endxy
\]
Removing faces from right to left and from left to right gives the same foam (up to rotation). The resulting foams are pictured below.
\[
\xy
 (0,0)*{\includegraphics[width=110px]{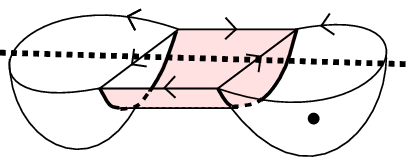}};
\endxy\;\;=\;\;
\xy
 (0,0)*{\includegraphics[width=110px]{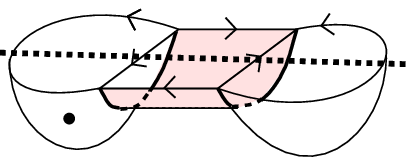}};
\endxy
\]
\end{itemize} 
\end{ex}
For the following theorem assume that one has a \textit{fixed} way how to remove faces in a fixed order for \textit{any} closed web $w$ independent of flow lines. By a slight abuse of notation, we call this a \textit{fixed way to remove webs}. We give one isotopy invariant method to assign to each $w$ such an preferred face removal later in Definition~\ref{defn-resolution}. 
\begin{thm}\label{thm-basis}
The set of flows on a web $w$ parametrises a basis for $\F(w)$, via the face removal algorithm, if one uses the given fixed way to remove webs.
\end{thm}
\begin{proof}
We proceed by induction on the number of faces of $w$. 
In the case that $w$ is a set of circles, it is easy to see 
that the flows on $w$ parametrise a basis for $\F(w)$ by applying 
the cups from the circle removals in Definition~\ref{defn-procedure}. Moreover, one easily shows that the claim is true for the theta web (see Example~\ref{ex-reso1} or Example~\ref{ex-reso2}). Note that the theta web has exactly two ways to remove faces, i.e. from right to left or vice versa. 

Suppose the claim is true for all closed webs with at most $n$ faces. 
Let $w$ be a closed web with $n+1$ faces. Without loss of 
generality, we can assume that $w$ contains no circles or theta webs. 

Suppose the first step of the face removal algorithm is a digon 
removal. The web without the digon is called $w^{\prime}$ and has $n-1$ faces.  
Recall that the order of the faces in $w^{\prime}$ is the same for all flows (because 
we fixed a removal independent of flow lines). 
As we show in the pictures below, given a flow $f$ on $w^{\prime}$, 
there are exactly two flows on $w$ which give rise to $f$. The two possible flows are pictured in the same column and there are three different cases how $f$ can interact with $w^{\prime}$.
\[
\xy
 (0,0)*{\includegraphics[width=90px]{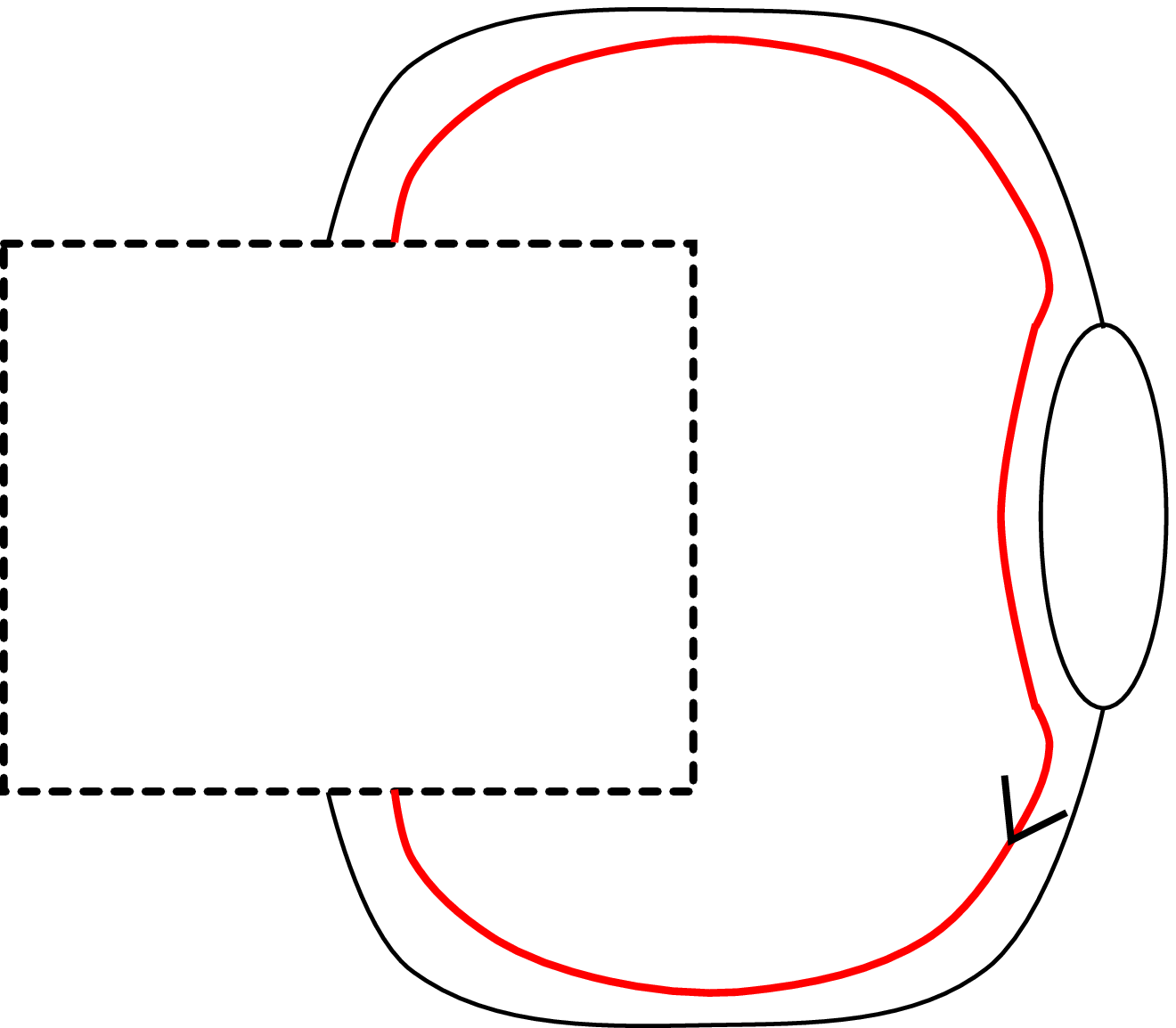}};
 (-6,0)*{w^{\prime}};
 (0,-30)*{\includegraphics[width=90px]{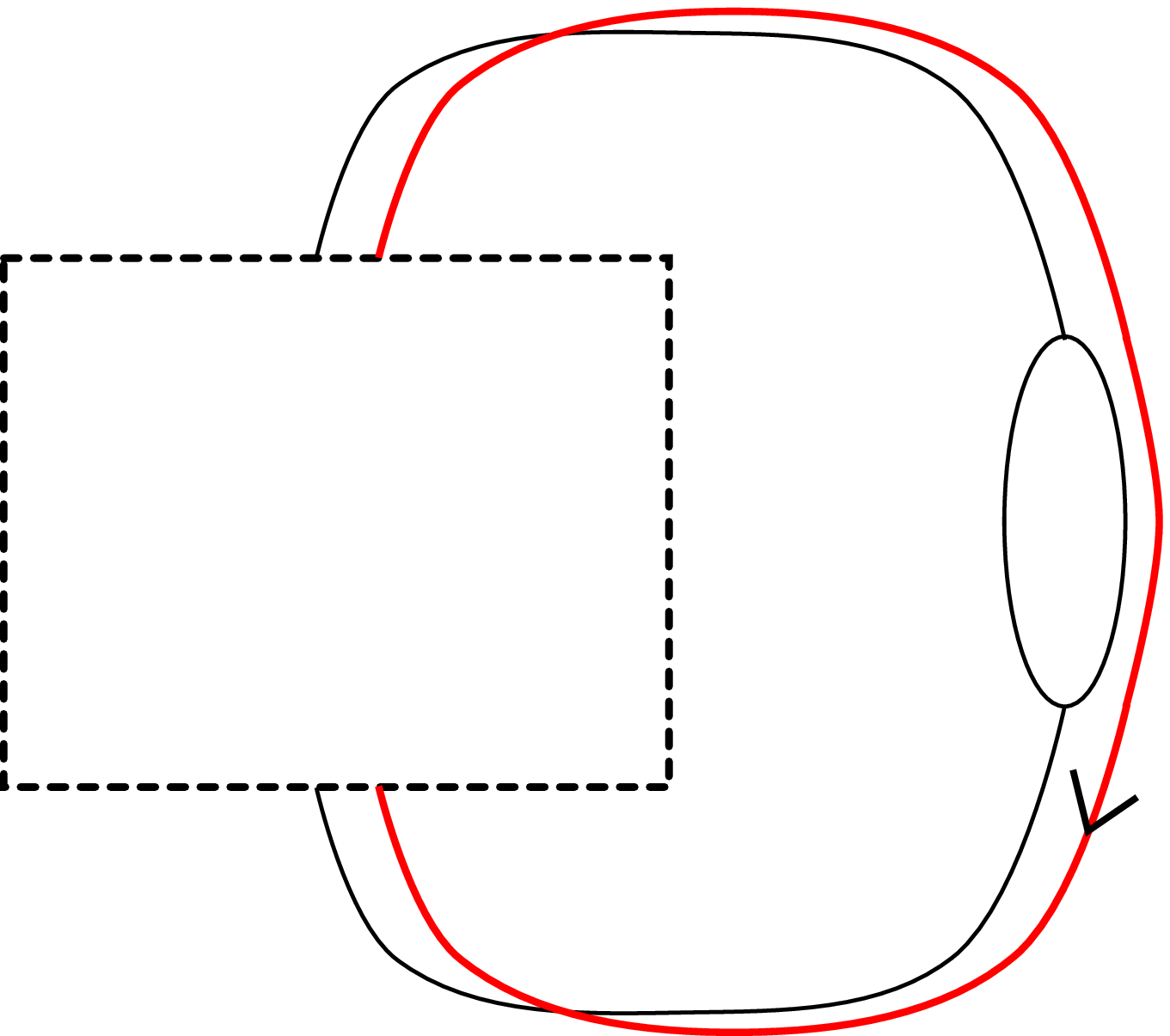}};
 (-6,-30)*{w^{\prime}};
 (-22.5,0)*{\tau_1\colon};
 (-22.5,-30)*{\tau_2\colon};
\endxy\;\;\;\;
\xy
 (0,0)*{\includegraphics[width=90px]{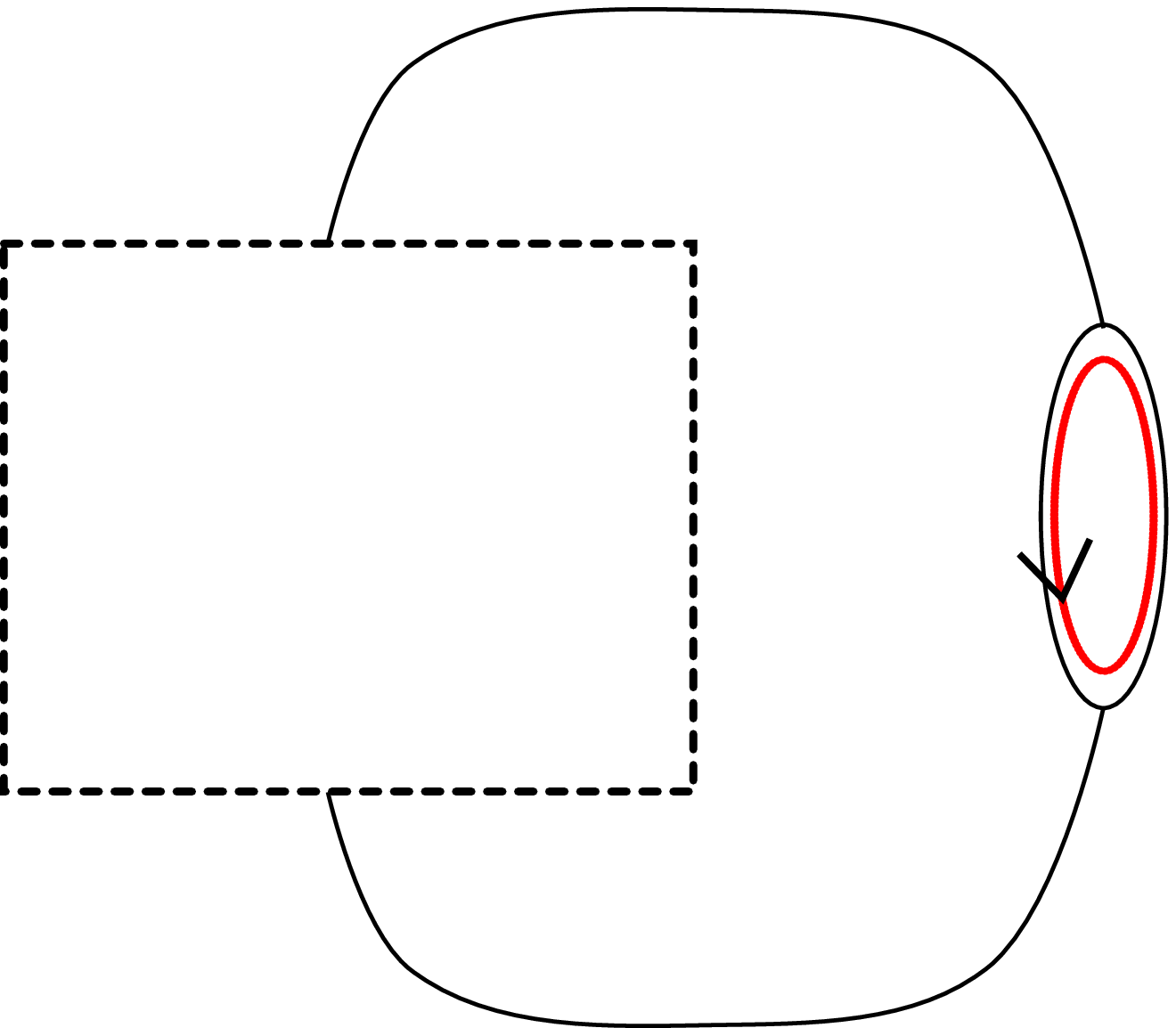}};
 (-6,0)*{w^{\prime}};
 (0,-30)*{\includegraphics[width=90px]{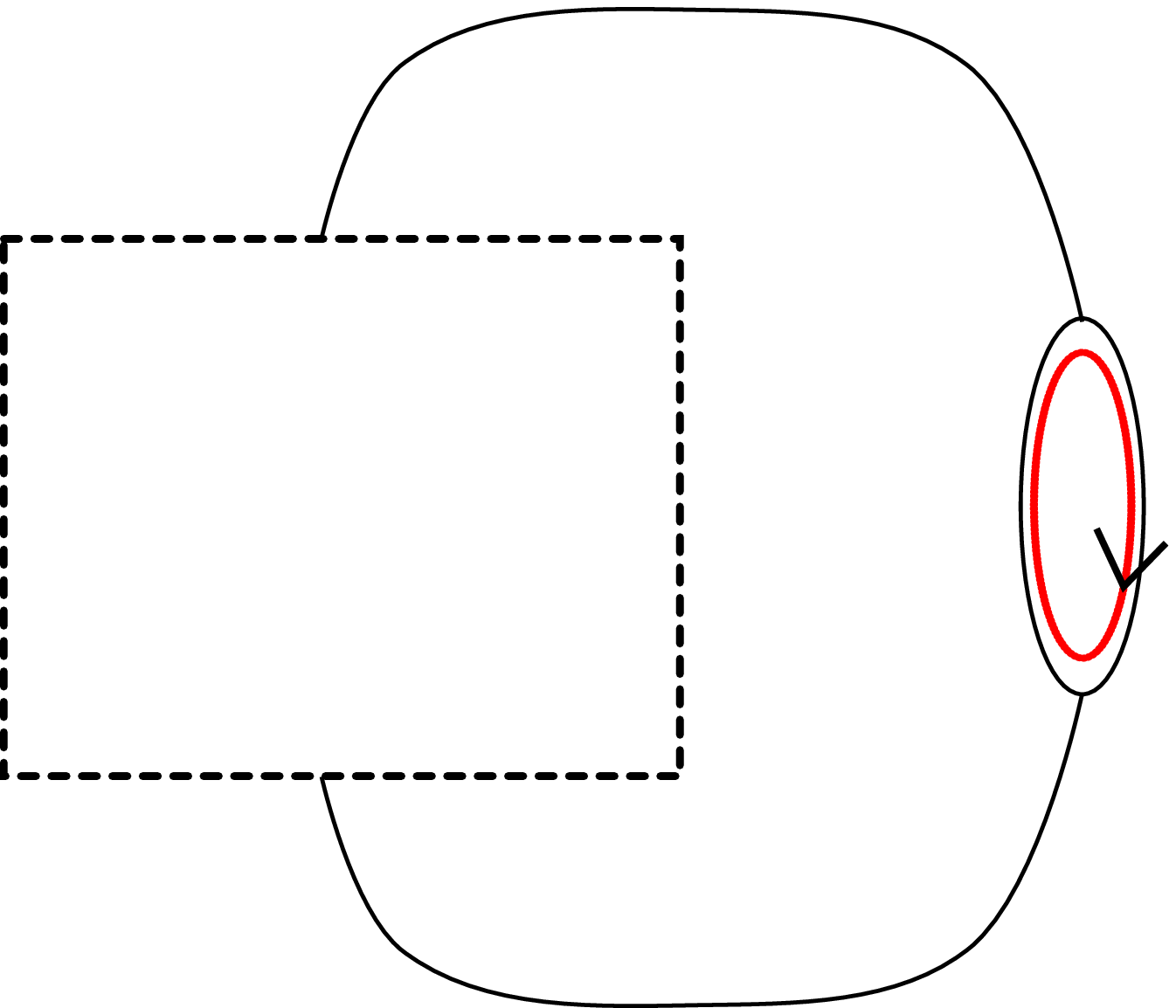}};
 (-6,-30)*{w^{\prime}};
\endxy\;\;\;\;
\xy
 (0,0)*{\includegraphics[width=90px]{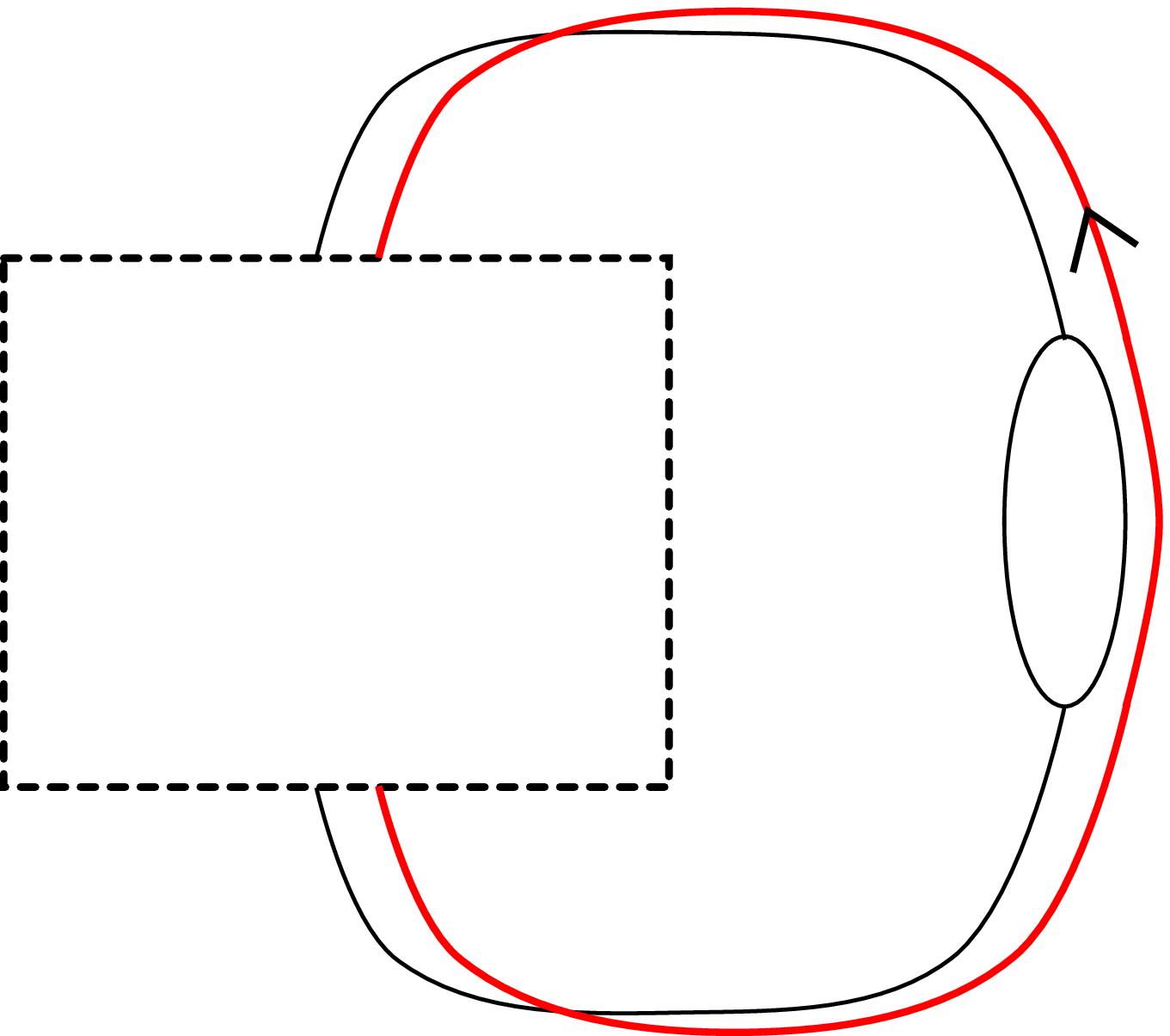}};
 (-6,0)*{w^{\prime}};
 (0,-30)*{\includegraphics[width=90px]{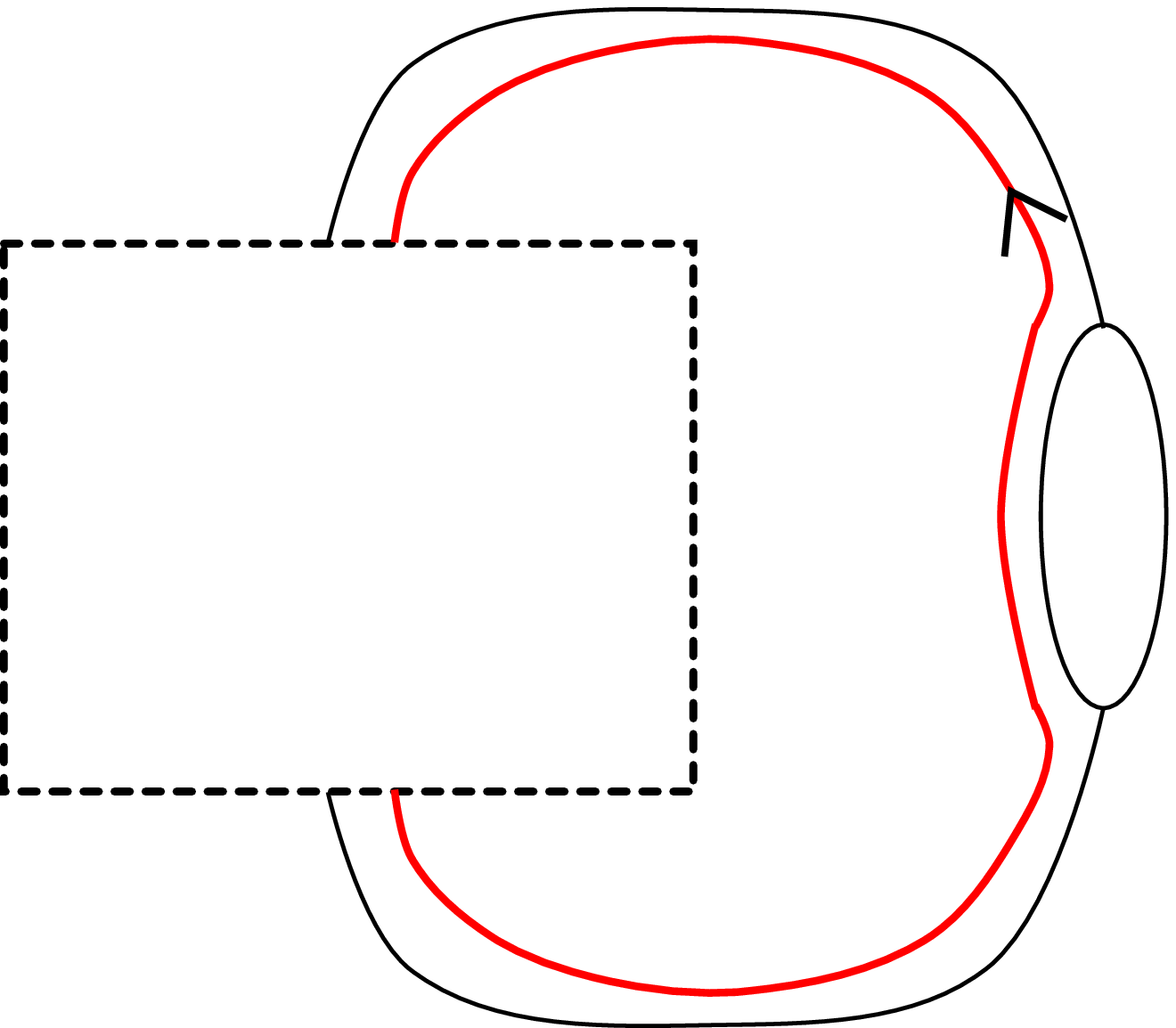}};
 (-6,-30)*{w^{\prime}};
\endxy
\]
\noindent By induction, the flows on $w^{\prime}$ parametrize a basis of $\F(w^{\prime})$ via 
the face removal algorithm. Given a flow on $w^{\prime}$, the 
two compatible flows on $w$ always give rise to two different digon removals shown below, i.e. either without or with a dot.
\[
\tau_{1}=\xy
 (0,2)*{\includegraphics[width=70px]{res/figs/basis/digonem}};
 (14,6.5)*{w};
 (14.5,-6)*{w^{\prime}};
\endxy\;\;\;\;\tau_{2}=\xy
 (0,2)*{\includegraphics[width=70px]{res/figs/basis/digondot}};
 (14,6.5)*{w};
 (14.5,-6)*{w^{\prime}};
\endxy
\]
From Proposition 8 in \cite{kh3} we have 
$\F(w)\cong_{\phi} \F(w^{\prime}) \{1\} \oplus \F(w^{\prime}) \{-1\}$, 
where the isomorphism $\phi$ is given by
\begin{align*}
&\F(w^{\prime})\{-1\} \xrightarrow{\F(\tau_{1})} \F(w)\;\;\text{ and }\;\;\F(w^{\prime})\{1\} \xrightarrow{\F(\tau_{2})} \F(w),
\end{align*}
where $\F(\tau_{1})$ and $\F(\tau_{2})$ are glued on the top of $\F(w^{\prime})$. Hence, we conclude that flows on $w$ parametrise a basis for $\F(w)$ as 
desired.
\vskip0.5cm
Now suppose that the first step of the face removal algorithm is a 
square removal, resulting in two new webs $w^{v},w^{h}$ with possible $n-2$ faces (depending on the position of the external face relatively to the square face).
 
There are two different square 
removals, each corresponding to either a vertical or horizontal square removal in Definition~\ref{defn-procedure}. One checks that every flow on $w^{v}$ corresponds to exactly one on $w$ that is removed vertical and every flow on $w^{h}$ corresponds to exactly one on $w$ that is removed horizontal. As can be seen, all eighteen possible 
flows on $w$ appear exactly once this way. By induction and 
Proposition 9 in \cite{kh3}, i.e. given the two removals
\[
\tau_{1}=\xy
 (0,2)*{\includegraphics[width=70px]{res/figs/basis/squarevert}};
 (14,6.5)*{w};
 (14.5,-6)*{w^{v}};
\endxy\;\;\;\;\tau_{2}=\xy
 (0,2)*{\includegraphics[width=70px]{res/figs/basis/squarehorz}};
 (14,6.5)*{w};
 (14.5,-6)*{w^{h}};
\endxy
\]
one has an isomorphism $\F(w)\cong_{\phi} \F(w^{v}) \oplus \F(w^{h})$, 
where the isomorphism $\phi$ is given by
\begin{align*}
&\F(w^{v}) \xrightarrow{\F(\tau_{1})} \F(w)\;\;\text{ and }\;\;\F(w^{h}) \xrightarrow{\F(\tau_{2})} \F(w),
\end{align*}
where $\F(\tau_{1})$ and $\F(\tau_{2})$ are glued on the top of $\F(w^{v})$ or $\F(w^{h})$. This implies that the flows on $w$ parametrise a basis for $\F(w)$.
\end{proof}
\subsubsection*{Colouring of webs}
We will proceed by giving a particular isotopy invariant method to remove any closed web. We will call this procedure \textit{preferred}. We note that ``isotopy invariance'' is a rather strong requirement. In order to obtain such a basis we need some technical definitions and lemmata, but the main idea is rather simple. We shortly explain it here and the reader may skip the more technical points below on the first reading.
\vskip0.5cm
It should be noted that the face removing algorithm~\ref{defn-procedure}, as explained in Example~\ref{ex-reso1}, \textit{depends} on the order how to remove faces, but only for \textit{neighbouring} faces, as explained in Lemma~\ref{lem-welldefreso} below. Therefore, it suffices to give a face colouring of webs such that neighbouring faces get different colours. Moreover, since a web is a special type of a trivalent graph, it is possible to give a \textit{3-colouring} (we call the colours $1,2,3$) of faces. Hence, we can summarise the rest of the section in three sentences.
\begin{itemize}
\item We give in~\ref{defn-colouring} an isotopy invariant method how to 2-colour edges of a closed web $w$.
\item This 2-colouring of the edges can be used to give a 3-colouring of the faces, as explained in~\ref{defn-resolution}.
\item Remove all available faces with the \textit{lowest} number (colour) in any order using the face removal convention~\ref{defn-procedure}.
\end{itemize}
\vskip0.5cm
We will use the 2-colouring below to fix a way to remove the faces of webs. It is worth noting that this colouring has a very special property, i.e. it \textit{only} depends on the sign string $S$. Therefore, it gives rise to a 2-colouring that extends to \textit{any} possible web with the boundary string $S$. It should be noted, as we explain in Remark~\ref{rem-prefchoice}, that this 2-colouring therefore has somehow ``anti-properties'' of the canonical 2-colouring (the one induced by the canonical flow). See Lemma~\ref{lem:tech}.  

Using Theorem~\ref{thm-basis}, we note that this implies the existence of an isotopy invariant basis, since the face removing algorithm~\ref{defn-procedure} and the procedure~\ref{defn-resolution} below are both isotopy invariant.
\vskip0.5cm
We need some terminology before we can start. Given a web $w=u^*v$ we can split it into its connected components $w^{c_j}$, i.e. we can split a web with $m$ such components as
\[
w=w^{c_1}\amalg\cdots\amalg w^{c_m}.
\]
We call the connected components \textit{nested of type $k\geq 0$} and denote them by $w^{c_j}_k$, using the following inductive definition.
\begin{itemize}
\item The $0$-nested components are the webs incident to the external face $F_{\mathrm{ext}}$.
\item For $k>0$, we call a component $k$-nested, if it is incident to the external face after removing all $k^{\prime}$-nested components for $k^{\prime}\leq k-1$, but not after removing all $k^{\prime}$-nested components for $k^{\prime}< k-1$.
\end{itemize}
The picture below illustrates the definition.
\[
\xy
 (0,0)*{\includegraphics[width=200px]{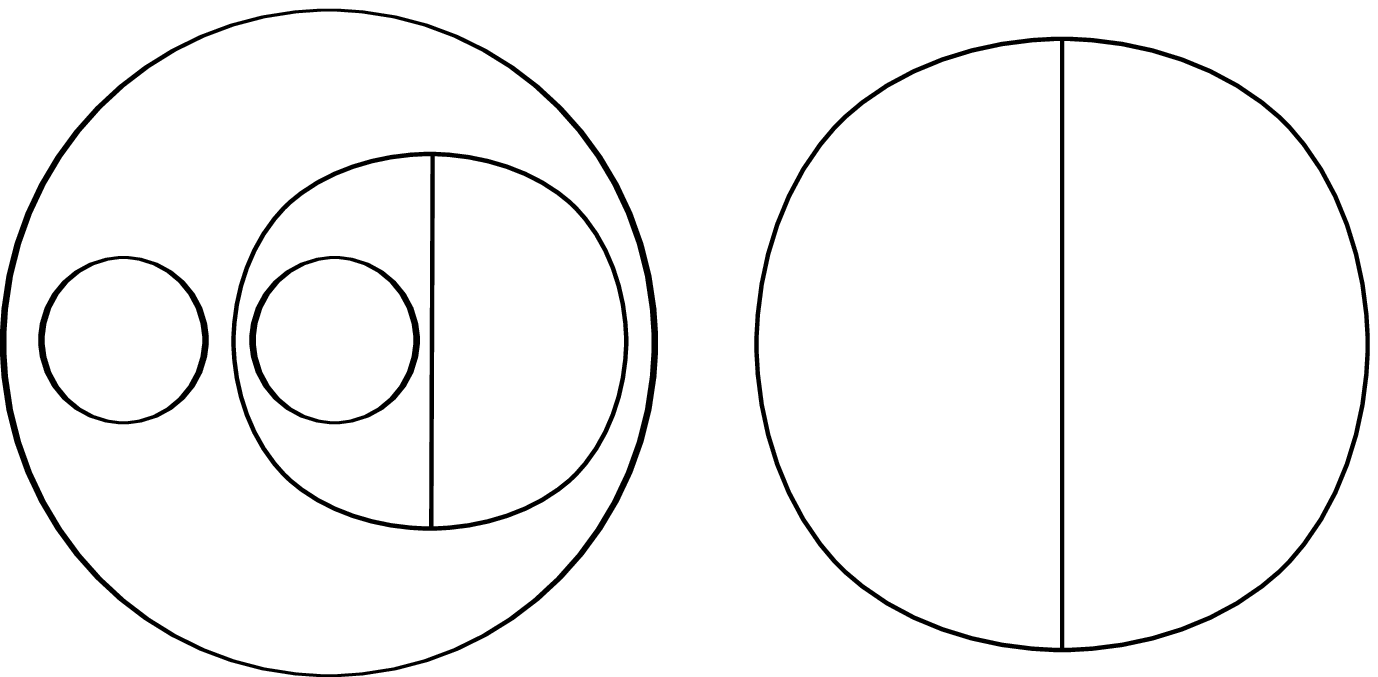}};
 (1,10)*{F_{\mathrm{ext}}};
 (-35,10)*{w_0^{c_1}};
 (-20.5,10)*{w_1^{c_4}};
 (-17.5,-1.5)*{w_2^{c_5}};
 (-28.5,-1.5)*{w_1^{c_3}};
 (36,10)*{w_0^{c_2}};
\endxy
\]
Moreover, given a 2-colouring (say with colours red and green) of the edges of a closed web $w$, we call a face $F$ of $w$ of type $r$, if all the edges of the face are red, and of type $g$ if at least one edge is green. We denote them by $F_r$ or $F_g$ respectively. The following pictures illustrate the notion.
\[
\xy
 (0,0)*{\includegraphics[width=80px]{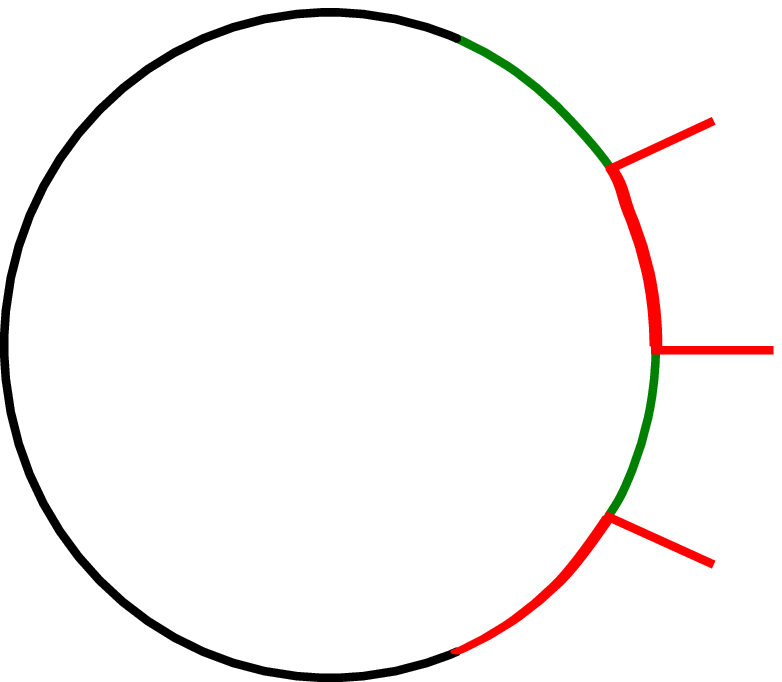}};
 (-1,0)*{F_{g}};
\endxy\;\;\;\;
\xy
 (0,0)*{\includegraphics[width=80px]{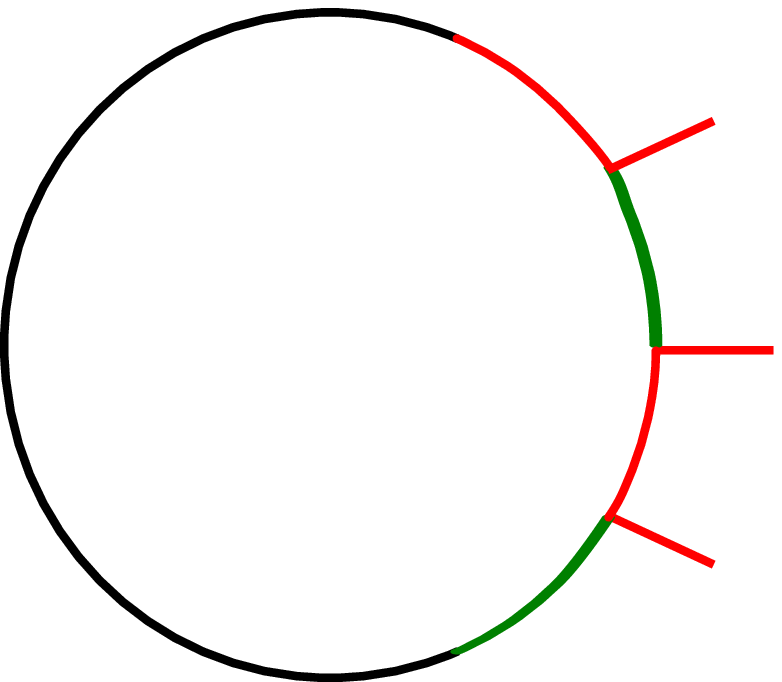}};
 (-1,0)*{F_{g}};
\endxy\;\;\;\;
\xy
 (0,0)*{\includegraphics[width=80px]{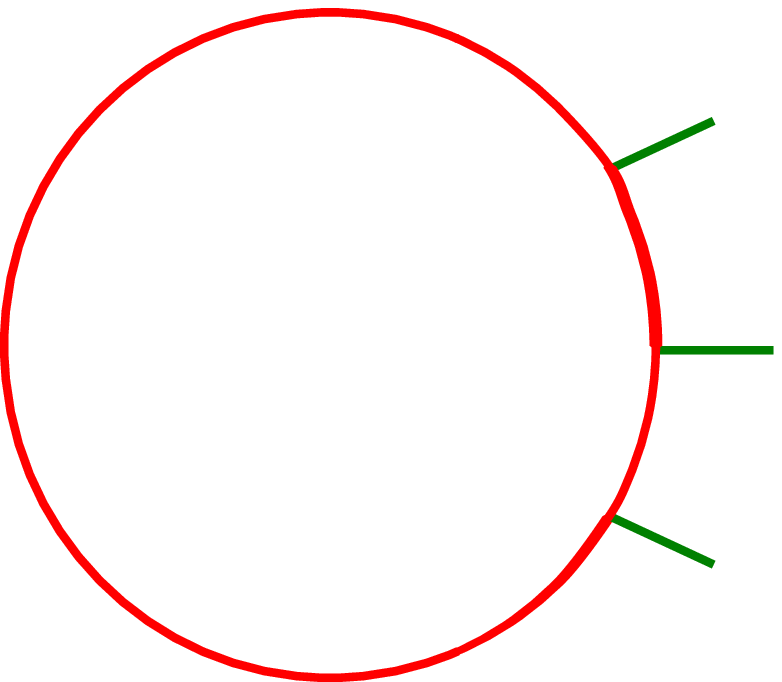}};
 (-1,0)*{F_{r}};
\endxy
\]
It should be noted that this technical distinction between nested or non-nested is necessary as we see later in Lemma~\ref{lem-welldefchoice}. Moreover, the same lemma ensures that the algorithm from Definition~\ref{defn-colouring} does not run into ambiguities. 
\begin{defn}\label{defn-colouring}
\textbf{(Colouring of closed webs)} We give an algorithm to produce a 
2-colouring (say with colours red and green) of the edges of a closed web $w$, 
such that at each vertex two red edges and one green edge meet. We call this colouring \textit{preferred}. It should be noted that it is not a priori clear why this procedure does not run into ambiguities. But this will never happen, see Lemma~\ref{lem-welldefchoice}. The same lemma ensures that the result at the end does not depend on the choices involved.
\vskip0.5cm 
The algorithm works as follows. First colour \textit{all} $0$-nested components $w^{c_j}_0$ by the following procedure.
\begin{enumerate}
\item At the initial stage, colour all edges of the external face of $w^{c_j}_0$ red.
\item At the $i$-th stage we complete the colouring of the edges incident to 
the vertices of $w^{c_j}_0$ with at least one edge already coloured from the 
$(i-1)$-th stage.
\item At a given vertex, if only one edge $e_{i-1}$ had already been coloured 
and it is green, then we colour the two remaining edges $e_i,e^{\prime}_i$ red (if one is already red, then colour the remaining one also red). If two edges $e_{i-1},e^{\prime}_{i-1}$ 
had already been coloured and they are red, then we colour the remaining edge $e_i$ 
green.
\[
\xy
 (0,0)*{\includegraphics[width=95px]{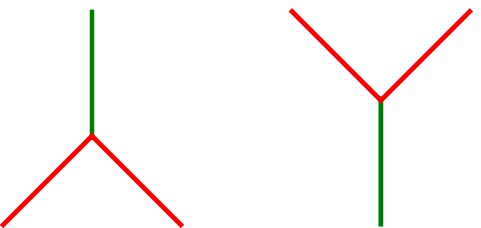}};
 (-5.7,3)*{e^{\phantom{\prime}}_{i-1}};
 (-15.7,-3.1)*{e^{\phantom{\prime}}_{i}};
 (-5,-3.1)*{e^{\prime}_{i}};
 (17,3)*{e^{\prime}_{i-1}};
 (5.3,3)*{e^{\phantom{\prime}}_{i-1}};
 (12.5,-3.1)*{e^{\phantom{\prime}}_{i}};
\endxy
\]
\item At a given vertex, if only one edge $e_{i-1}$ had already been coloured 
and it is is red, then we colour one of the other incident edges in the following fashion.
\begin{itemize}
\item Green for $e_i$, if the type of the corresponding face of $e_i$ is already green.
\item The remaining edges $e^{\prime}_i$ incident to this vertex should be red.
\end{itemize}
Below we have summarised the convention in a picture.
\[
\xy
 (0,0)*{\includegraphics[width=90px]{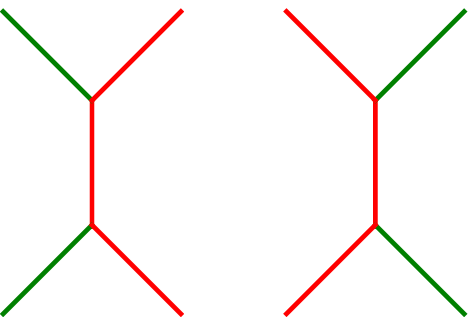}};
 (-5,0)*{e^{\phantom{\prime}}_{i-1}};
 (-4,-6)*{e^{\prime}_{i}};
 (-15.5,-6)*{e^{\phantom{\prime}}_{i}};
 (-15,6)*{e^{\phantom{\prime}}_{i-2}};
 (17,6)*{e^{\phantom{\prime}}_{i-2}};
 (14,0)*{e^{\phantom{\prime}}_{i-1}};
 (4.4,-6)*{e^{\prime}_{i}};
 (15,-6)*{e^{\phantom{\prime}}_{i}};
\endxy
\]
\end{enumerate}
Now, if all $k$-nested components are coloured, then colour exactly one edge of each $k+1$-nested by first choosing a line that cuts though all the components at least once (e.g. the cut line). Then at least one edge of all $k+1$-nested component is neighbouring an already coloured edge. Choose one such edge for all $k+1$-nested component and colour it in the following way and then repeat the procedure from before.
\[
\xy
 (0,0)*{\includegraphics[width=30px]{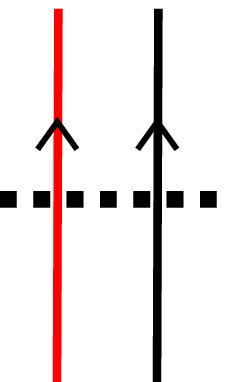}};
\endxy\to\xy
 (0,0)*{\includegraphics[width=30px]{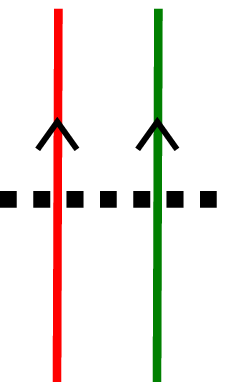}};
\endxy\;\;\;\;
\xy
 (0,0)*{\includegraphics[width=30px]{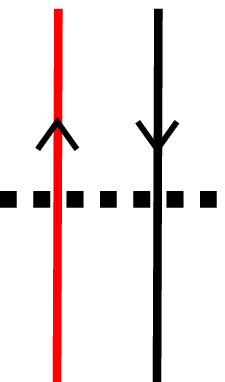}};
\endxy\to\xy
 (0,0)*{\includegraphics[width=30px]{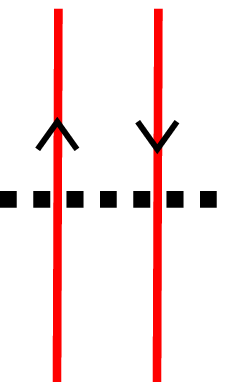}};
\endxy\;\;\;\;
\xy
 (0,0)*{\includegraphics[width=30px]{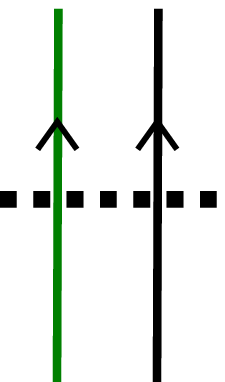}};
\endxy\to\xy
 (0,0)*{\includegraphics[width=30px]{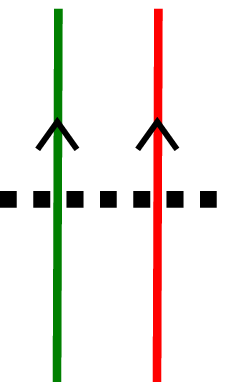}};
\endxy\;\;\;\;
\xy
 (0,0)*{\includegraphics[width=30px]{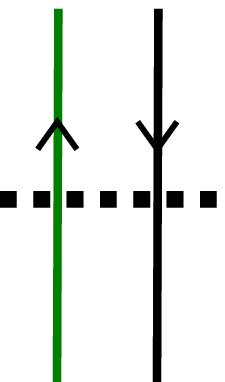}};
\endxy\to\xy
 (0,0)*{\includegraphics[width=30px]{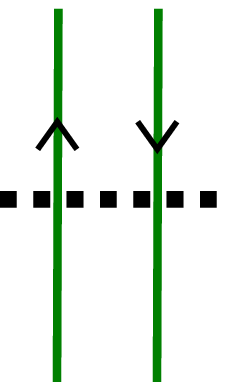}};
\endxy
\]
Or in words, use the same colour, iff the orientation of the line is reversed. The algorithm stops if every edge is coloured.
\end{defn}
\begin{lem}\label{lem-welldefchoice}
The 2-colouring of $w$ above is well-defined and it does not depend on the choices involved.
\end{lem}
\begin{proof}
First we prove the second statement, i.e. that the choices do not matter. In fact this can be easily seen by locally alternating the line that cuts though the components as illustrated below.
\[
\xy
 (0,0)*{\includegraphics[width=80px]{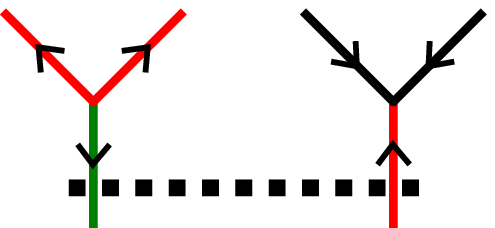}};
 (0,5)*{F_g};
\endxy\;\;\;\;
\xy
 (0,0)*{\includegraphics[width=80px]{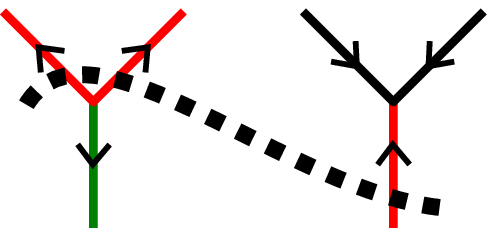}};
 (0,5)*{F_g};
\endxy\;\;\;\;
\xy
 (0,0)*{\includegraphics[width=80px]{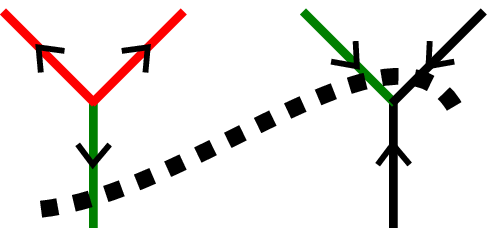}};
 (0,5)*{F_g};
\endxy
\]
Note that all three choices above give the same result since the corresponding face is of green type. We leave it to the reader to check the other possible colourings. Hence, since choosing a different edge can also be seen as rearranging the line, we get the second statement.
\vskip0.5cm
To see that the colouring is well-defined, we proceed by induction on the number $m$ of faces of $w$. It can be easily checked that it is well-defined for a circle or a theta web.

Moreover, one easily checks that every closed $\mathfrak{sl}_3$ web $w$ has at least one circle, digon or square face. Fix one such face $F$ and remove it from the web as explained in the subsection before. The resulting web $w^{\prime}$ is, by induction, colourable without ambiguities. If the face $F$ is a circle, then one easily sees that the colouring of $w^{\prime}$ can be extended without ambiguities to $w$ by applying the procedure in Definition~\ref{defn-colouring}.

If $F$ is either a digon or a square, then one can extend the colouring in the following way. 
\[
\xy
 (-5,0)*{\includegraphics[height=.125\textheight]{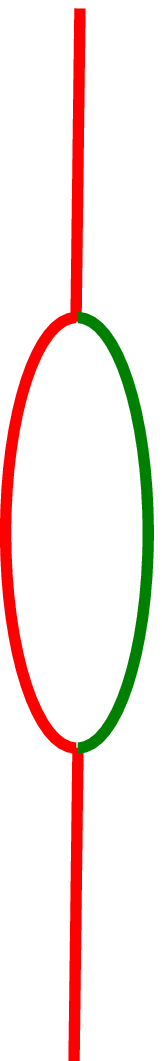}};
 (0,-2.1)*{\mapsto};
 (0,2)*{\mapsfrom};
 (-3.2,7)*{w};
 (6,7.3)*{w^{\prime}};
 (4,0)*{\includegraphics[height=.125\textheight]{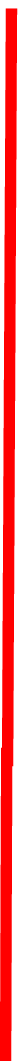}};
 \endxy\;\;\;\;\;\;\;\;
 \xy
 (-5,0)*{\includegraphics[height=.125\textheight,angle=180]{res/figs/basis/dgcolour1a}};
 (0,-2.1)*{\mapsto};
 (0,2)*{\mapsfrom};
 (-3.2,7)*{w};
 (6,7.3)*{w^{\prime}};
 (4,0)*{\includegraphics[height=.125\textheight]{res/figs/basis/dgcolour1b}};
 \endxy\;\;\;\;\;\;\;\;
 \xy
 (-5,0)*{\includegraphics[height=.125\textheight]{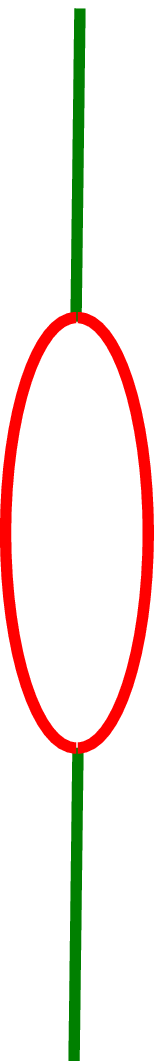}};
 (0,-2.1)*{\mapsto};
 (0,2)*{\mapsfrom};
 (-3.2,7)*{w};
 (6,7.3)*{w^{\prime}};
 (4,0)*{\includegraphics[height=.125\textheight]{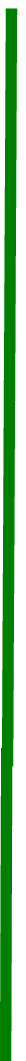}};
 \endxy
\]
in the case that $F$ is a digon and if $F$ is a square, then we use the following.
\[
 \xy
 (-12.5,0)*{\includegraphics[height=.08\textheight]{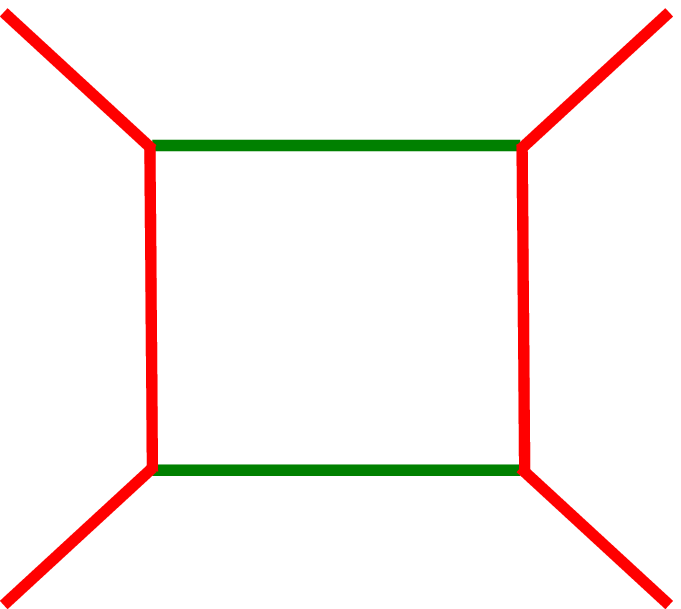}};
 (0,-2.1)*{\mapsto};
 (0,2)*{\mapsfrom};
 (-12.5,7)*{w};
 (12.5,7.3)*{w^{\prime}};
 (12.5,0)*{\includegraphics[height=.08\textheight]{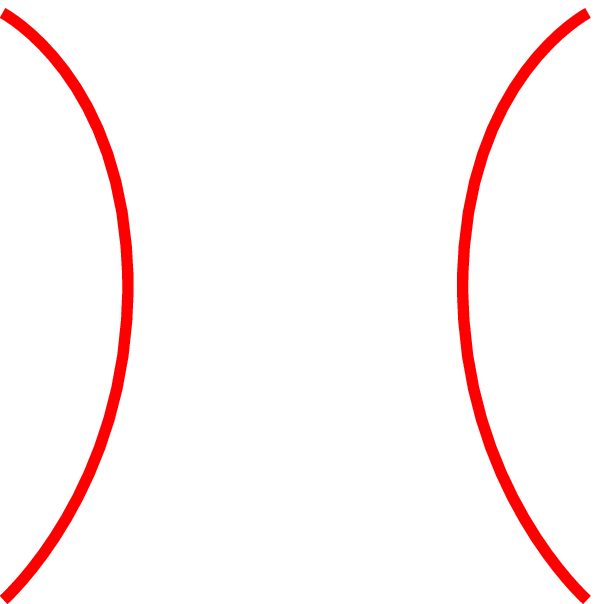}};
 \endxy\;\;\;\;\;\;
 \xy
 (-12.5,0)*{\includegraphics[height=.08\textheight]{res/figs/basis/sqcolour3}};
 (0,-2.1)*{\mapsto};
 (0,2)*{\mapsfrom};
 (-12.5,7)*{w};
 (12.5,7.3)*{w^{\prime}};
 (12.5,0)*{\includegraphics[height=.08\textheight]{res/figs/basis/vertcolour2}};
 \endxy\;\;\;\;\;\;
 \xy
 (-12.5,0)*{\includegraphics[height=.08\textheight]{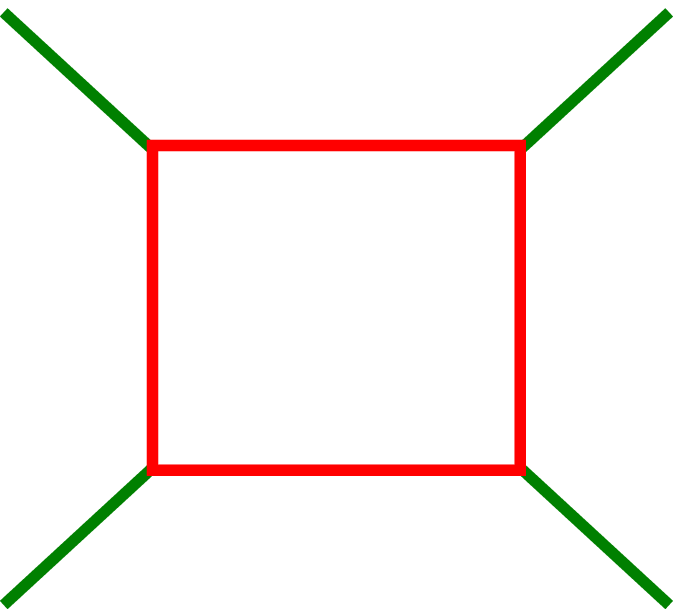}};
 (0,-2.1)*{\mapsto};
 (0,2)*{\mapsfrom};
 (-12.5,7)*{w};
 (12.5,7.3)*{w^{\prime}};
 (12.5,0)*{\includegraphics[height=.08\textheight]{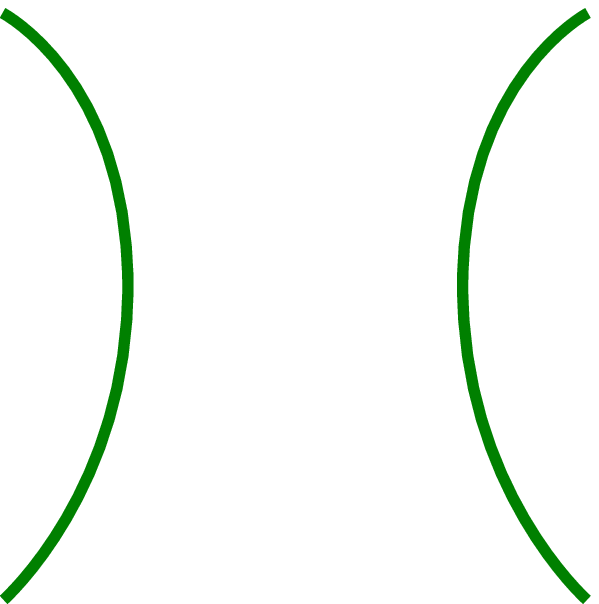}};
 \endxy
\]
It should be noted that the only ambiguous seeming case, i.e. the case where the result of the square removal of $w$ has two alternating coloured edges in $w^{\prime}$, does not occur. This is because the two strings in $w^{\prime}$ have a different orientation.

Hence, if they are part of the same connected component, then they have the same colour, because the colours and the orientations always alternate at vertices. If they are not part of the same connected component, then we have to use the cut procedure explained before to see that they have the same colour. This proves the first statement.
\end{proof}
\begin{rem}\label{rem-prefchoice}
Note that the Definition~\ref{defn-colouring} can also be made for half webs, i.e. webs $u$ with a given sign string $S$ at the boundary. An analogue of Lemma~\ref{lem-welldefchoice} can be shown as above.

Moreover, it is easy to show that, given a fixed sign string $S$ and two webs $u,v$ such that $S=\partial u=\partial v$, then the preferred colourings of $u$ and $v$ match at the boundary and the preferred colourings of $u^*v$ and $v^*u$ are given by glueing the preferred colourings of $u,v$ together. This is already implicit in Lemma~\ref{lem-welldefchoice} because the way how to cut a web does not affect the result, i.e. we can take the cut line for non-elliptic webs.   
\end{rem}
\begin{ex}\label{ex-colour}
An easy example of how the colouring works is shown below.
\[
\xy
 (0,0)*{\includegraphics[width=200px]{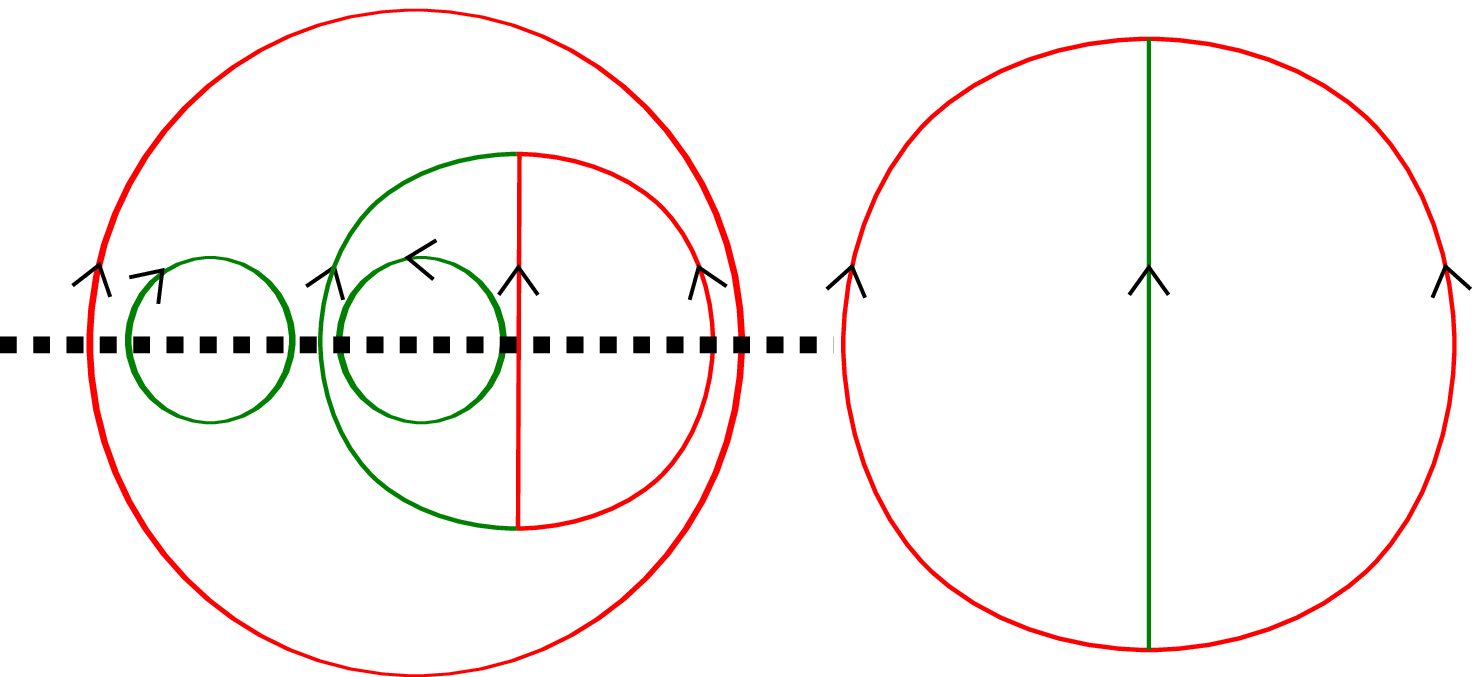}};
\endxy
\]
\end{ex}
\begin{defn}\label{defn-pflow}
\textbf{(Preferred flow on closed webs)} Consider the subgraph of $w$ given by 
the red edges. It is easy to see that this subgraph consists of a disjoint 
set of closed cycles. By orienting these edges so that the flows around the parts of the web are oriented counterclockwise and all other cycles clockwise we obtain a flow on $w$. We call this flow the \textit{preferred flow} and write $w_p$ for $w$ with the preferred flow.
\end{defn}
Note that, by construction, the preferred flow $w_p$ always consists of flows \textit{inside} faces or flows around the web. That is, they will never cross through edges, i.e. the case below will never occur. A good example is the flow pictured in Example~\ref{ex-reso3}. 
\[
\xy
 (0,0)*{\includegraphics[width=36px]{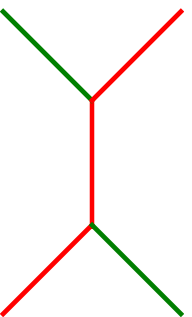}};
\endxy
\]
The following definition of the \textit{preferred face removal} is given for a connected web $w$. For an arbitrary web, the whole process should be repeated for any connected component.
\begin{defn}\label{defn-resolution}
\textbf{(The preferred face removing)} Given a closed web with a flow denoted as $w_{f}$. At each stage of the face removing algorithm let the order of the faces be determined in the following way.

It should be noted that the preferred flow from Definition~\ref{defn-pflow} implies that any face $F$ of $w_p$ is of the following type because every face has an even number of edges and the closed circles of the preferred flow $p$ are always oriented clockwise except the flows around the web. We note that, by abuse of notation, we denote a face around a connected component of the web as $F_{\mathrm{ext}}$, although this face is not the external face for nested parts.
\begin{itemize}
\item Faces $F_1$, such that the preferred flow has a component of $p$ inside $F_1$.
\item Faces $F_2$, such that the preferred flow $p$ has the same orientation as the edges of $F_2$.
\item Faces $F_3$, such that the preferred flow $p$ is against the orientation of the edges of $F_3$. 
\end{itemize}
To summarise see the figure below. A face of type $k$ should be labelled $k$.
\[
\xy
 (0,0)*{\includegraphics[width=100px]{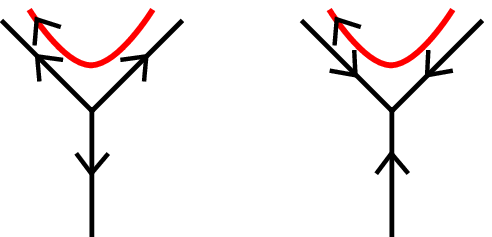}};
 (-10.5,8.5)*{F_1};
 (-7,-1)*{F_3};
 (-14.5,-1)*{F_2};
 (11,8.5)*{F_1};
 (7,-1)*{F_3};
 (14.5,-1)*{F_2};
\endxy\;\;\;\;\xy
 (0,0)*{\includegraphics[width=100px]{res/figs/basis/numbering}};
 (-10.5,8.5)*{F_{\mathrm{ext}}};
 (-7,-1)*{F_3};
 (-14.5,-1)*{F_2};
 (11,8.5)*{F_{\mathrm{ext}}};
 (7,-1)*{F_3};
 (14.5,-1)*{F_2};
\endxy
\]
At each stage of the algorithm we continue to call the web with flow $w_f$. 
We remove $w_{f}$ by the following algorithm.
\begin{enumerate}
\item Remove all circles in $w_{f}$ using the local circle rules of Definition~\ref{defn-resolution}, 
in any order. If there are no faces remaining, the algorithm stops. 
If there are faces remaining and some of them are digons, then proceed to step 2. If no remaining faces are digons, then proceed to step 3.
\item Remove all digons with the smallest label using the local digon rules of Definition~\ref{defn-resolution}, in any order. Go back to step 1. 
\item Remove all square faces with the smallest label using the local square rules of Definition~\ref{defn-resolution}, in any order. Go back to step 1.
\end{enumerate}
We call the above process of obtaining a foam $F_{w_{f}}\in\F(w)$ the 
\textit{preferred face removal} of $w_{f}$ or short \textit{preferred resolution} of $w_{f}$.
\end{defn}
\begin{rem}\label{rem-recursion}
Notice that it is relatively easy to calculate the order of face removing since the preferred colouring of the web $w^{\prime}$ (and therefore the preferred flow) after removing a particular face of $w$ can be computed directly as indicated before in the proof of Lemma~\ref{lem-welldefchoice}.

It is straightforward to check that this recursive procedure gives the same answer as if one calculates the preferred colouring of $w^{\prime}$ as in Definition~\ref{defn-colouring}. 
\end{rem}
\begin{lem}\label{lem-welldefreso}
The preferred resolution of $w_f$ is well-defined. 
\end{lem}
\begin{proof}
We note that the labelling of the faces is in such a way that neighbouring faces never have the same label, i.e. they define a colouring of the faces of $w$ with three colours. Hence, we only need to show that the consecutive resolution of two non-neighbouring faces in two different orders results in isotopic foams.

A simple illustration shows that the consecutive removing of the two faces 
in two different orders yields isotopic foams.
\[
\xy
 (0,0)*{\includegraphics[width=150px]{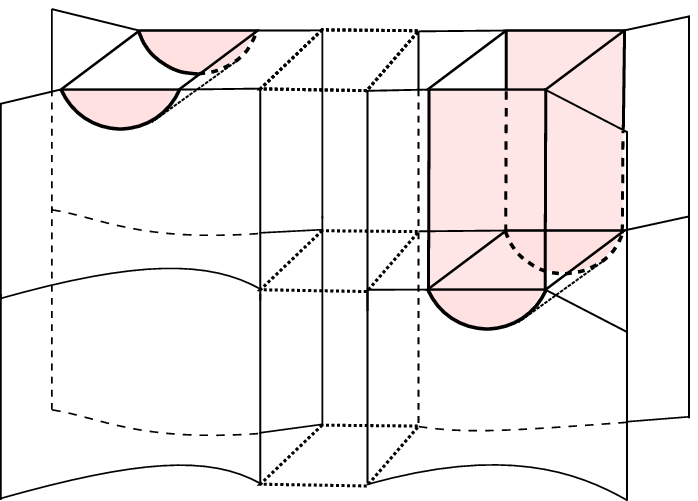}};
\endxy\;=\;
\xy
 (0,0)*{\includegraphics[width=150px]{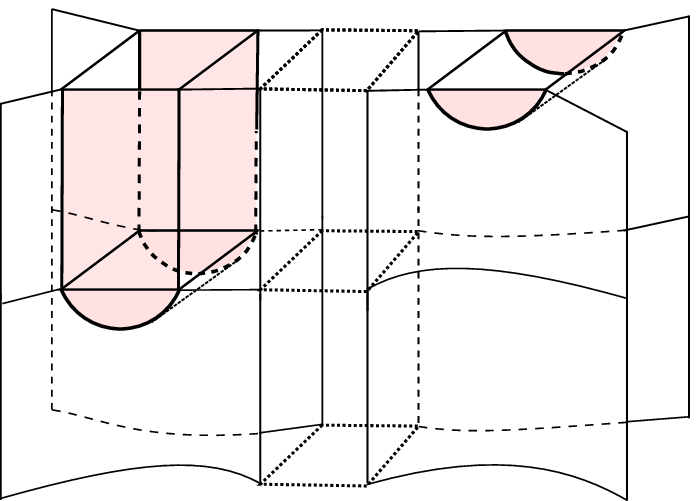}};
\endxy
\]
The above only illustrates the effect of removing both square faces 
in one particular way, but similar arguments demonstrate the claim 
for different removing of the square faces as well as removing of 
digon faces.
\end{proof}
\begin{thm}\label{thm-basis2}
Given a web $w=u^*v$, the face removing algorithm~\ref{defn-procedure} together with the preferred face removal~\ref{defn-resolution} gives an isotopy invariant, homogeneous basis of $\F(w)$ parametrised by flows on $w$. If $w$ is symmetric, then the basis contains the identity given by the canonical flow. 
\end{thm}
\begin{proof}
This is a direct consequence of Theorem~\ref{thm-basis}, Proposition~\ref{prop-resolution} and the fact that the procedures explained in the Definitions~\ref{defn-procedure} and~\ref{defn-resolution} work in an isotopy invariant way.
\end{proof}
From Theorem~\ref{thm-basis2} we get the following corollary since for any fixed sign string $S$ the algebra $K_S$ was defined as
\[
K_S=\bigoplus_{u,v\in B_S}{}_{u}K_{v}.
\]
\begin{cor}\label{cor-basis}
Let $S$ be a fixed sign string. The set of flows on all webs $w$ with $S=\partial w$ parametrises an isotopy invariant, homogeneous basis for $K_S$, via the preferred face removing algorithm.
\end{cor}
\begin{ex}\label{ex-reso2}
Consider the theta web $w$ from Example~\ref{ex-reso1} again. We know that the graded dimension of $\F(w)$ is given by the Kuperberg bracket and is therefore $[2][3]=q^{-3}+2q^{-1}+2q+q^3$. The six possible flows (ordered by weight) on $w$ are illustrated below.
\[
\xy
 (0,0)*{\includegraphics[width=80px]{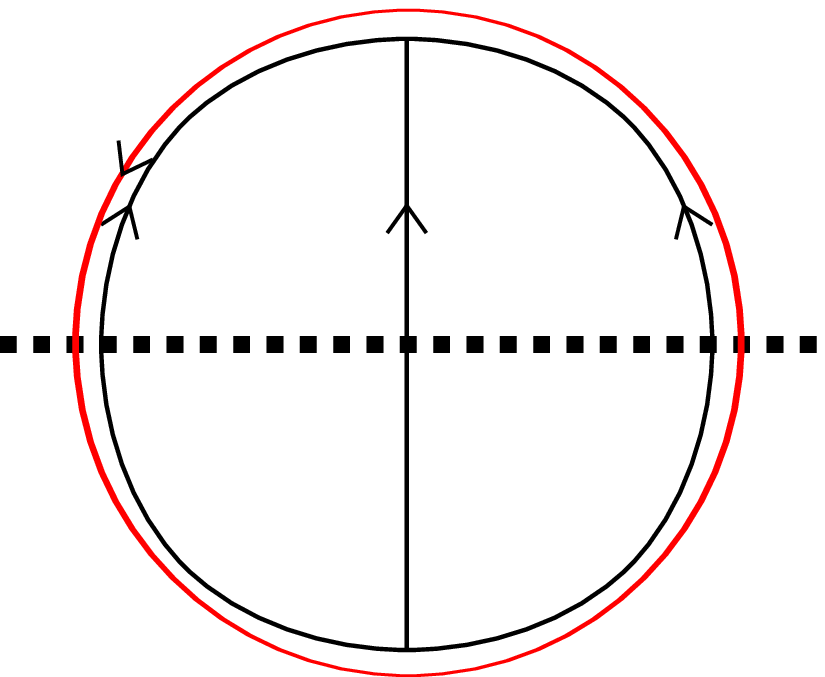}};
 (0,-12.5)*{\scriptstyle \mathrm{wt} = 3};
 (0,-30)*{\includegraphics[width=80px]{res/figs/basis/thetaexd}};
 (0,-42.5)*{\scriptstyle \mathrm{wt} = -1};
\endxy\;\;\;\;
\xy
 (0,0)*{\includegraphics[width=80px]{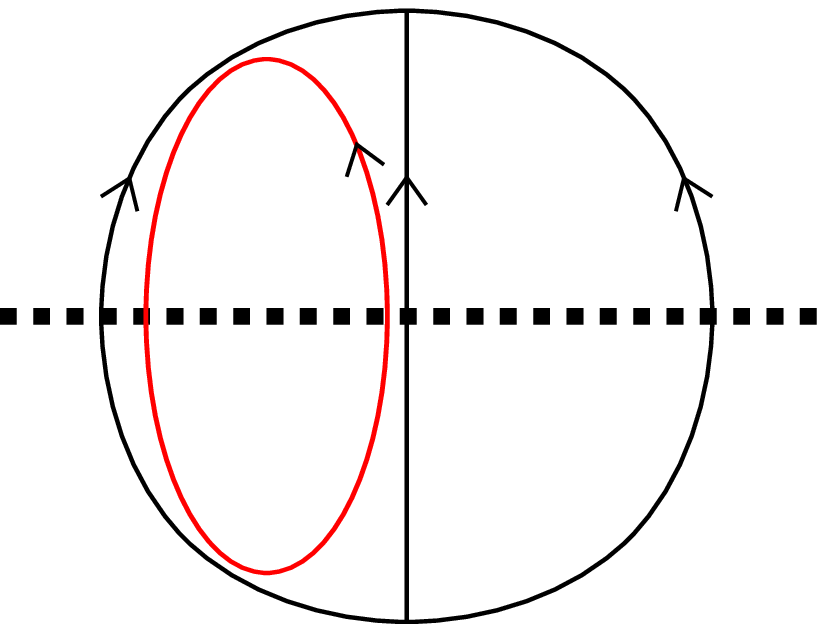}};
 (0,-12.5)*{\scriptstyle \mathrm{wt} = 1};
 (0,-30)*{\includegraphics[width=80px]{res/figs/basis/thetaexe}};
 (0,-42.5)*{\scriptstyle \mathrm{wt} = -1};
\endxy
\;\;\;\;
\xy
 (0,0)*{\includegraphics[width=80px]{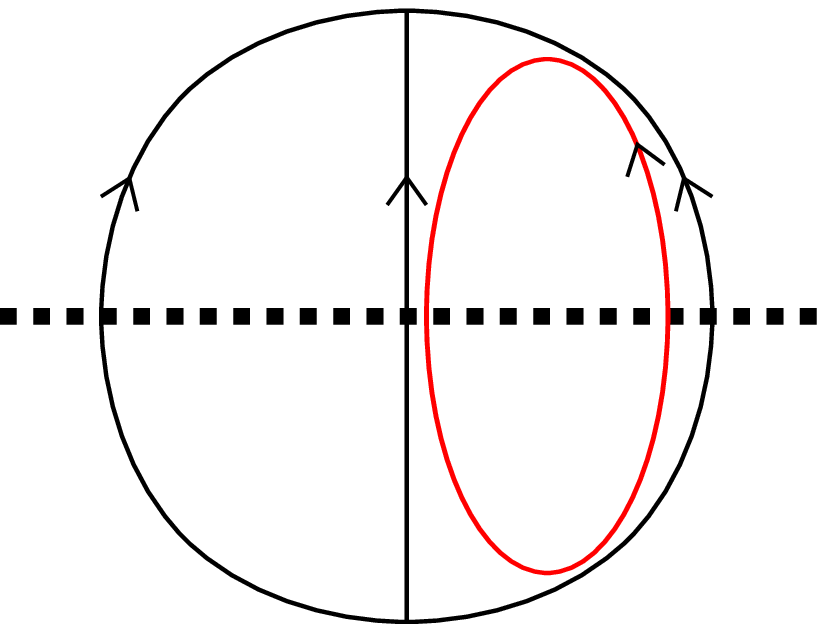}};
 (0,-12.5)*{\scriptstyle \mathrm{wt} = 1};
 (0,-30)*{\includegraphics[width=80px]{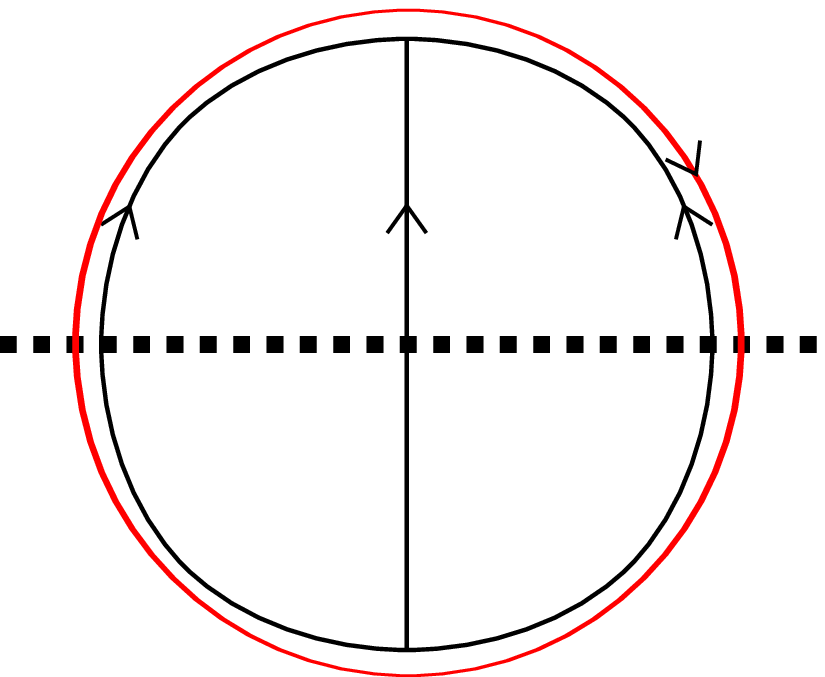}};
 (0,-42.5)*{\scriptstyle \mathrm{wt} = -3};
\endxy
\]
Notice that in this case the canonical flow is also the preferred flow. Hence, the order of removal of the faces is from right to left. This gives the following basis elements (ordered by $q$-degree) for the corresponding flows.
\[
\xy
 (0,0)*{\includegraphics[width=80px]{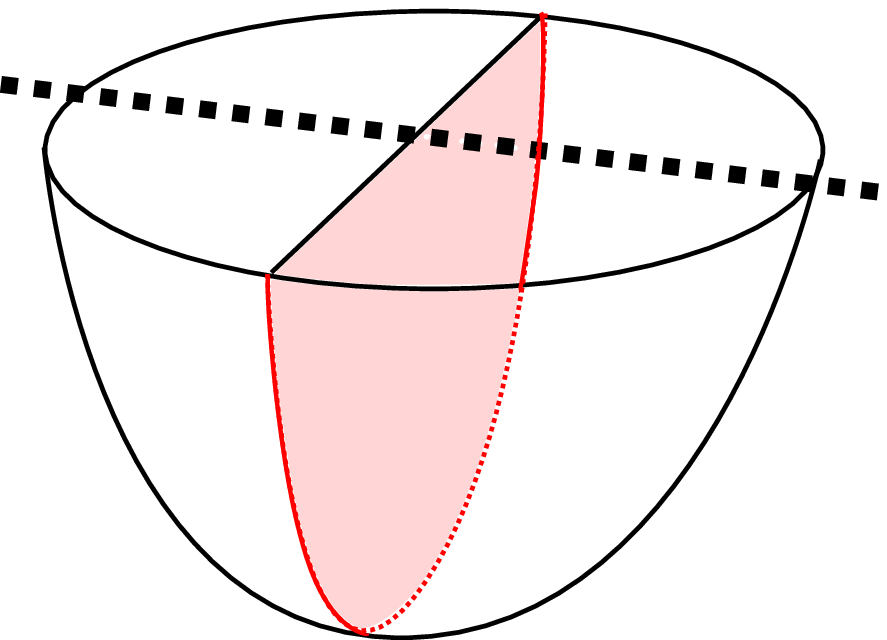}};
 (0,-12.5)*{\scriptstyle q-\deg = 0};
 (0,-30)*{\includegraphics[width=80px]{res/figs/basis/thetad}};
 (0,-42.5)*{\scriptstyle q-\deg = 4};
\endxy\;\;\;\;
\xy
 (0,0)*{\includegraphics[width=80px]{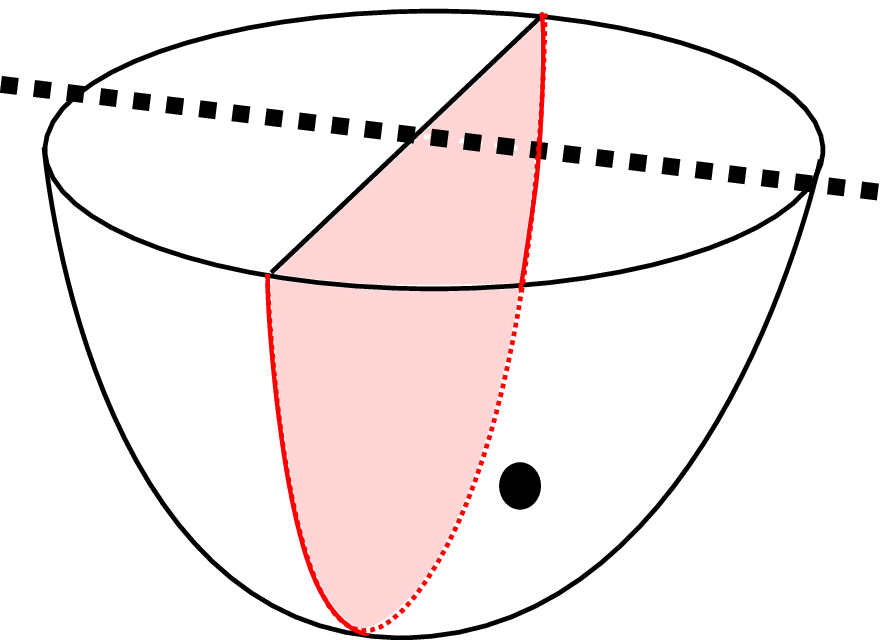}};
 (0,-12.5)*{\scriptstyle q-\deg = 2};
 (0,-30)*{\includegraphics[width=80px]{res/figs/basis/thetae}};
 (0,-42.5)*{\scriptstyle q-\deg = 4};
\endxy
\;\;\;\;
\xy
 (0,0)*{\includegraphics[width=80px]{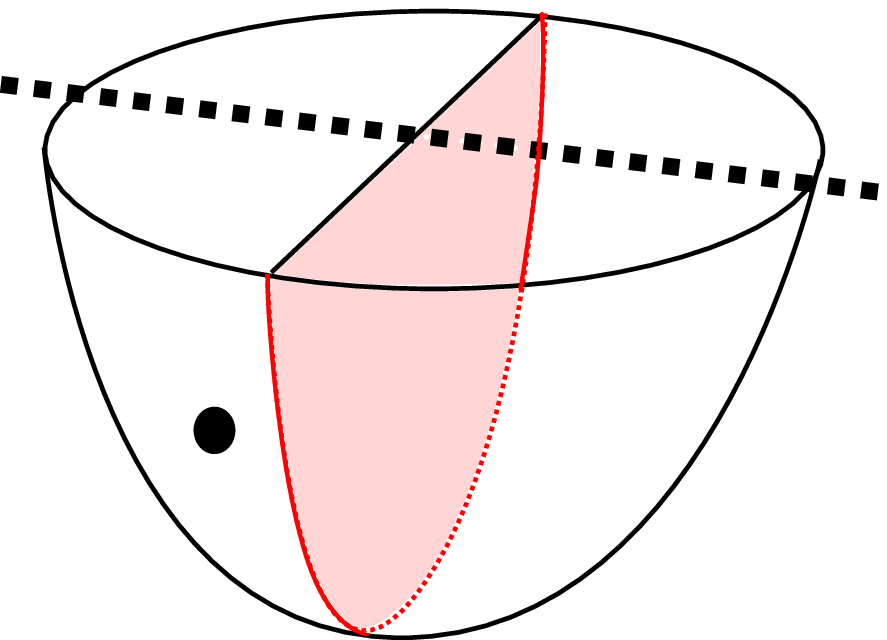}};
 (0,-12.5)*{\scriptstyle q-\deg = 2};
 (0,-30)*{\includegraphics[width=80px]{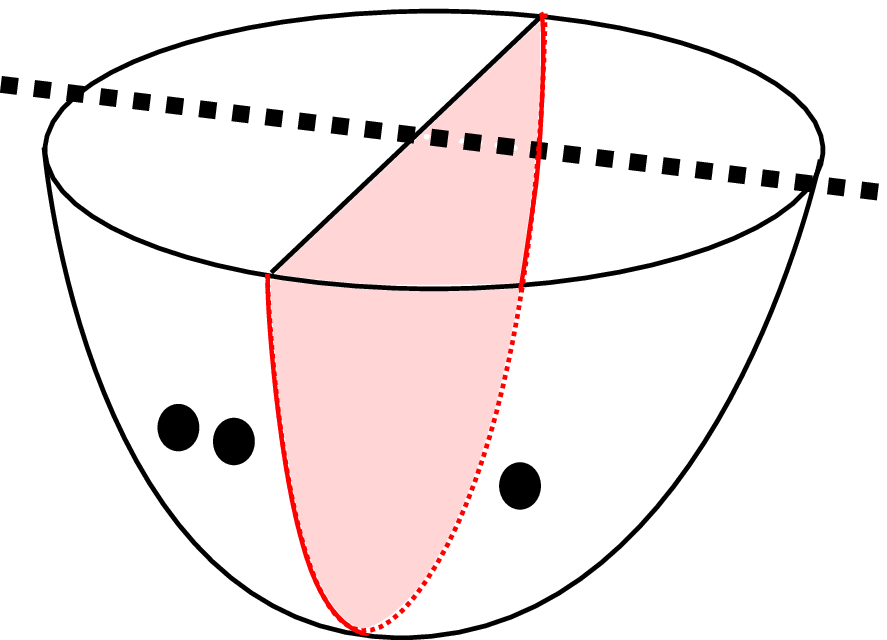}};
 (0,-42.5)*{\scriptstyle q-\deg = 6};
\endxy
\]
\end{ex}
\begin{ex}\label{ex-reso3}
The counterexample of Khovanov and Kuperberg~\cite{kk} for the negative exponent property (compare to Theorem~\ref{thm:upptriang} and the Remark~\ref{rem:counter}) gives rise to another  interesting example, i.e. the corresponding foam will be (up to a scalar factor) a non-trivial idempotent of the web algebra. The preferred flow $p$ from Definition~\ref{defn-pflow} on $w$ pictured below (for both orientations of $w$) has weight $\mathrm{wt}(p)=12$. Thus, the corresponding foam has $q$-degree $0$.
\[
\xy
 (0,0)*{\includegraphics[width=188px]{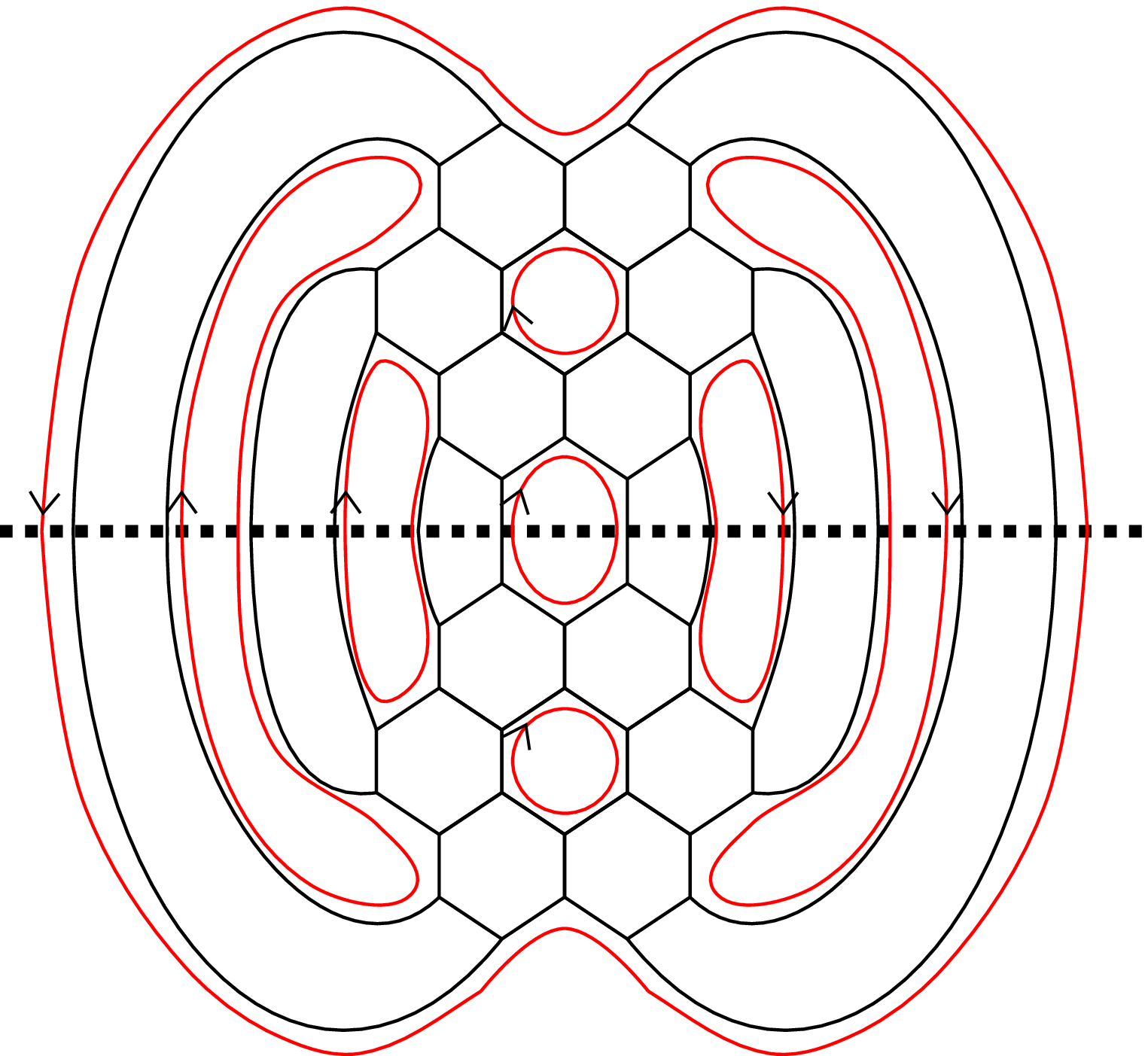}};
 (0,-30)*{\mathrm{wt} = 12};
\endxy
\]
One easily checks that the canonical flow $c$ on this web also has $\mathrm{wt}(c)=12$. Thus, $\F(w)$ has two linear independent foams in $q$-degree zero. Note that (b) of Proposition~\ref{prop-resolution} ensures that the foam obtained from the canonical flow is the identity.
\end{ex}
\subsection{Open issues}\label{sec-webend}
Let us mention some open questions that are hopefully answered in future work. Note that some other open questions were already discussed in Section~\ref{sec-websum}. We will focus here on four questions the author is currently working on, namely the ones listed below.
\begin{itemize}
\item[(a)] We conjecture that there is an ordering how to resolve webs as explained in Section~\ref{sec-webbase} such that the resulting basis is a graded cellular basis. For a lot of constructions involving graded cellular bases one needs a particular basis in an explicit form (compare to the discussion in Section~\ref{sec-techcell}). Hence, it is a future goal to give such a basis for $K_S$. 
\item[(b)] A generalisation of the results on Howe duality from Section~\ref{sec-webhowea}. That is a pictorial version similar to the results in Section~\ref{sec-webhowea}, but for arbitrary representations of $U_q(\mathfrak{sl}_n)$. In order to do so, one would for example consider clasps and clasped web spaces as explained by Kuperberg in~\cite{kup}. Note that this is not known at the moment, even for $n=2$. This would correspond to the coloured versions of the $\mathfrak{sl}_n$ polynomials instead of the uncoloured case.
\item[(c)] Instead of a categorification of the \textit{invariant} tensors, as we have done, one could also try to give a categorification of the \textit{full} tensor product. Note that in the $n=2$ case a categorification is known, e.g. see Chen and Khovanov~\cite{chkh}. It is worth noting that this is related to the question how to construct the quasi-hereditary cover of $K_S$. Such a cover for Khovanov's arc algebra, i.e. the $\mathfrak{sl}_2$ case, was studied by Brundan and Stroppel~\cite{bs1}, Chen and Khovanov~\cite{chkh} and Stroppel~\cite{st2}.
\item[(d)] The Question~\ref{question} asked by Kamnitzer, i.e. how our work is related to the approach from algebraic geometry by Fontaine, Kamnitzer and Kuperberg~\cite{fkk}.
\end{itemize}
\newpage
\section{Technical points}\label{sec-tech}
\subsection{Higher categories}\label{sec-techhigher}
In the present section we recall some definitions and theorems from higher category theory.

Note that we always talk about \textit{strict $(n,n)$-categories} if we say $n$-category. Here we allow $n\in\{0,1,2,\dots,\omega\}$. One can think of a $n$-category as a \textit{$n$-dimensional category}.

We usually drop the ``strict'', that is all categories are assumed to be strict unless otherwise mentioned. Informally, the strict refers to the fact that composition of higher morphism is ``on the nose'' associative and satisfies some identity laws.
\vskip0.5cm
Moreover, a $(n,r)$-category (with $r\leq n$) is a category with only non-trivial $k$-cells for $k\leq n$ of which all $k^{\prime}$-cells for $r<k^{\prime}\leq n$ are invertible. For example a strict $(1,0)$-category is a groupoid, while a strict $(1,1)$-category is a usual category.

Note that $(n,1)$-categories arise in a lot of examples motivated from topology and are sometimes called $n$-categories. But we do not need them in this thesis, so we only refer to Leinster's~\cite{tl} book for a more detailed discussion.

Moreover, to \textit{avoid} all set-theoretical question, which are interesting for their own sake, but not for our purposes, all categories should be \textit{essentially small}, i.e. the skeleton should be small. By a slight abuse of notation, we always take about sets of objects, morphisms etc.
\vskip0.5cm
We start by recalling some ``classical'' notions that we need in the following.
\begin{defn}\label{defn-monoidal}
Let $\mathcal C$ be a category. The category is called \textit{(strict) monoidal} if it is a monoid in the category $\CAT_1$, i.e. the category of categories. That is the category $\mathcal C$ is equipped with a bifunctor $\otimes\colon\mathcal C\times\mathcal C\to\mathcal C$ such that the following is satisfied.
\begin{itemize}
\item[(a)] There exists a unit $1\in\Ob(\mathcal C)$ such that for all $C\in\Ob(\mathcal C)$
\[
1\otimes C=1\otimes C=C.
\]
\item[(b)] For all $C_1,C_2,C_3\in\Ob(\mathcal C)$ the associativity holds, that is
\[
(C_1\otimes C_2)\otimes C_3=C_1\otimes(C_2\otimes C_3).
\]
\end{itemize}
Note that the functoriality implies
\[
(f^{\prime}\circ g^{\prime})\otimes (f\circ g)=(f^{\prime}\otimes f)\circ(g^{\prime}\otimes g)
\]
for suitable morphisms $f,f^{\prime},g,g^{\prime}$.

A \textit{weak monoidal} category is the same as above, but (a) and (b) hold only up to natural isomorphism, denoted $r_{C},\ell_{C}$ and $\alpha_{C_1,C_2,C_3}$, such that the following two diagrams commute.
\[
\xymatrix{
(C_1 \otimes 1) \otimes C_2
\ar[rr]^{\alpha_{C_1,1,C_2}}
\ar[dr]_{r_{C_1} \otimes 1}
&& C_1 \otimes(1 \otimes C_2)
\ar[dl]^{1 \otimes\ell_{C_2} } \\
& C_1 \otimes C_2}
\]
and the so-called pentagram identity
\[
\xy
(0,30)*{(C_1 \otimes C_2) \otimes(C_3 \otimes C_4)}="1";
(50,5)*{C_1 \otimes(C_2 \otimes(C_3 \otimes C_4))}="2";
(30,-20)*{ \quad C_1 \otimes((C_2 \otimes C_3) \otimes C_4)}="3";
(-30,-20)*{(C_1 \otimes(C_2 \otimes C_3)) \otimes C_4}="4";
(-50,5)*{((C_1 \otimes C_2) \otimes C_3) \otimes C_4}="5";
{\ar^{\alpha_{C_1,C_2,C_3 \otimes C_4}} "1";"2"}
{\ar_{1 \otimes\alpha_{C_2,C_3,C_4}} "3";"2"}
{\ar^{\alpha_{C_1,C_2 \otimes C_3,C_4}} "4";"3"}
{\ar_{\alpha_{C_1,C_2,C_3} \otimes 1} "5";"4"}
{\ar^{\alpha_{C_1 \otimes C_2,C_3,C_4}\mbox{\phantom{$x$}}} "5";"1"}
\endxy
\]
A (weak or strict) monoidal category $(\mathcal C,\otimes)$ is called \textit{symmetric} if for all $C_1,C_2\in\Ob(\mathcal C)$ there exists natural isomorphisms $\gamma_{C_1,C_2}\colon C_1\otimes C_2\to C_2\otimes C_1$ and $\gamma_{C_2,C_1}\colon C_2\otimes C_1\to C_1\otimes C_2$ such that $\gamma_{C_2,C_1}\circ \gamma_{C_1,C_2}=\mathrm{id}_{C_1\otimes C_2}$ and $\gamma_{C_1,C_2}\circ \gamma_{C_2,C_1}=\mathrm{id}_{C_2\otimes C_1}$.
\end{defn}
\begin{ex}\label{ex-monoidal}
For a given field $K$ the category $\mathcal C=$\textbf{Vec}${}_K$ together with the usual tensor product is a weak, symmetric monoidal category.
\end{ex}
In general, since one is mostly interested in more than equivalence classes of objects, strict monoidal categories are rare. But a well-known fact, also known as \textit{Mac Lane's coherence theorem}, allows us to ``ignore'' the difference between strict and weak monoidal categories, i.e. by Mac Lane's coherence theorem we have the following.
\begin{thm}\label{thm-maclane}
Every weak monoidal category is equivalent to a strict monoidal category.
\end{thm}
\vskip0.2cm
Note that a monoidal category has already a $2$-dimensional structure, see Example~\ref{ex-2cat}. We are going to recall the notion of a ``$2$-dimensional category'' now.
\begin{defn}\label{defn-2cat}
A $2$-category $\mathfrak C$ is a $1$-category enriched over $\CAT_1$.
\end{defn}
Informally, a $2$-category is a category consisting of the following.
\begin{itemize}
\item Objects, also called $0$-cells. One can imagine them as $0$-dimensional. They are often pictured as $\bullet$. We denote the class of objects usually by $\Ob(\mathfrak C)$.
\item $1$-morphisms, also called $1$-cells. One can imagine them as $1$-dimensional and they have therefore one way of composition. They are often pictured in the following way.
\[
\xy
(-8,0)*+{\bullet}="4";
(8,0)*+{\bullet}="6";
{\ar^{f} "4";"6"};
\endxy
\]
We denote the class of $1$-cells between $C,D\in\Ob(\mathfrak C)$ usually by $\Mor_1(C,D)$ or without the subscript $1$ if the context is clear. Composition is drawn in the following way.
\[
\xy
(-16,0)*+{\bullet}="4";
(0,0)*+{\bullet}="6";
(16,0)*+{\bullet}="8";
{\ar^{f} "4";"6"};
{\ar^{g} "6";"8"};
\endxy
\] 
\item $2$-morphisms, also called $2$-cells. One can imagine them as $2$-dimensional and they have therefore two ways of composition, i.e. a \textit{vertical} $\circ_v$ and a \textit{horizontal} $\circ_h$. The $2$-cells are often pictured in the following way.
\[
\xy
(-8,0)*+{\bullet}="4";
(8,0)*+{\bullet}="6";
{\ar@/^1.65pc/^{f} "4";"6"};
{\ar@/_1.65pc/_{g} "4";"6"};
{\ar@{=>}^<<<{\scriptstyle \alpha} (0,3)*{};(0,-3)*{}} ;
\endxy
\]
The two compositions are usually drawn in the following way. Here the left is the vertical.
\[
\alpha\circ_v\beta=\xy
(-8,0)*+{\bullet}="4";
(8,0)*+{\bullet}="6";
{\ar "4";"6"};
{\ar@/^1.75pc/^{f} "4";"6"};
{\ar@/_1.75pc/_{g} "4";"6"};
{\ar@{=>}^<<{\scriptstyle \alpha} (0,6)*{};(0,1)*{}} ;
{\ar@{=>}^<<{\scriptstyle \beta} (0,-1)*{};(0,-6)*{}} ;
\endxy\;\;\alpha\circ_h\beta=\xy
(-16,0)*+{\bullet}="4";
(0,0)*+{\bullet}="6";
{\ar@/^1.65pc/^{f} "4";"6"};
{\ar@/_1.65pc/_{g} "4";"6"};
{\ar@{=>}^<<<{\scriptstyle \alpha} (-8,3)*{};(-8,-3)*{}} ;
(0,0)*+{\bullet}="4";
(16,0)*+{\bullet}="6";
{\ar@/^1.65pc/^{f^{\prime}} "4";"6"};
{\ar@/_1.65pc/_{g^{\prime}} "4";"6"};
{\ar@{=>}^<<<{\scriptstyle \beta} (8,3)*{};(8,-3)*{}} ;
\endxy 
\]
\item The interchange law, that is
\[
(\alpha\circ_h\gamma)\circ_v(\beta\circ_h\delta)=(\alpha\circ_v\beta)\circ_h(\gamma\circ_v\delta)
\]
or in pictures
\[
\xy
(-16,4)*+{\bullet}="4";
(0,4)*+{\bullet}="6";
{\ar "4";"6"};
{\ar@/^1.75pc/^{f} "4";"6"};
{\ar@{=>}^<<<{\scriptstyle \alpha} (-8,10)*{};(-8,5)*{}};
(0,4)*+{\bullet}="4";
(16,4)*+{\bullet}="6";
{\ar "4";"6"};
{\ar@/^1.75pc/^{f^{\prime}} "4";"6"};
{\ar@{=>}^<<<{\scriptstyle \gamma} (8,10)*{};(8,5)*{}};
(1,0)*{\circ_v};
(-16,-4)*+{\bullet}="4";
(0,-4)*+{\bullet}="6";
{\ar "4";"6"};
{\ar@/_1.75pc/_{g} "4";"6"};
{\ar@{=>}^<<{\scriptstyle \beta} (-8,-5)*{};(-8,-10)*{}};
(0,-4)*+{\bullet}="4";
(16,-4)*+{\bullet}="6";
{\ar "4";"6"};
{\ar@/_1.75pc/_{g^{\prime}} "4";"6"};
{\ar@{=>}^<<{\scriptstyle \delta} (8,-5)*{};(8,-10)*{}};
\endxy=\xy
(-16,0)*+{\bullet}="4";
(0,0)*+{\bullet}="6";
{\ar "4";"6"};
{\ar@/^1.75pc/^{f} "4";"6"};
{\ar@/_1.75pc/_{g} "4";"6"};
{\ar@{=>}^<<<{\scriptstyle \alpha} (-8,6)*{};(-8,1)*{}};
{\ar@{=>}^<<{\scriptstyle \beta} (-8,-1)*{};(-8,-6)*{}};
(0,0)*+{\bullet}="4";
(16,0)*+{\bullet}="6";
{\ar "4";"6"};
{\ar@/^1.75pc/^{f^{\prime}} "4";"6"};
{\ar@/_1.75pc/_{g^{\prime}} "4";"6"};
{\ar@{=>}^<<<{\scriptstyle \gamma} (8,6)*{};(8,1)*{}};
{\ar@{=>}^<<{\scriptstyle \delta} (8,-1)*{};(8,-6)*{}};
\endxy=\xy
(-8,0)*+{\bullet}="4";
(8,0)*+{\bullet}="6";
{\ar "4";"6"};
{\ar@/^1.75pc/^{f} "4";"6"};
{\ar@/_1.75pc/_{g} "4";"6"};
{\ar@{=>}^<<{\scriptstyle \alpha} (0,6)*{};(0,1)*{}} ;
{\ar@{=>}^<<{\scriptstyle \beta} (0,-1)*{};(0,-6)*{}} ;
\endxy\circ_h \xy
(-8,0)*+{\bullet}="4";
(8,0)*+{\bullet}="6";
{\ar "4";"6"};
{\ar@/^1.75pc/^{f} "4";"6"};
{\ar@/_1.75pc/_{g} "4";"6"};
{\ar@{=>}^<<{\scriptstyle \gamma} (0,6)*{};(0,1)*{}} ;
{\ar@{=>}^<<{\scriptstyle \delta} (0,-1)*{};(0,-6)*{}} ;
\endxy.
\]   
\end{itemize}
Again, the notion of a \textit{weak} $2$-category, introduced as a so-called \textit{bicategory} by B\'{e}nabou, is more common. Informally, this is like a strict $2$-category, but some equations should only hold up to natural $2$-isomorphisms. We do not recall the formal definition here and refer e.g. to Leinster~\cite{tl}, since the formal definition is rather complicated and Theorem~\ref{thm-maclane2} below shows that we can ``ignore'' the difference again. A proof of the theorem is also well-known (in the sense that B\'{e}nabou already proved it in the 1960s). For a proof see for example~\cite{tl}. An interesting fact is that if $2<n$, then there is \textit{no} corresponding coherence theorem for $n$-categories!
\begin{thm}\label{thm-maclane2}
Every weak $2$-category is equivalent to a strict $2$-category.
\end{thm}
\vskip0.2cm
\begin{ex}\label{ex-2cat}
There are a lot of examples of (weak and strict) $2$-categories.
\begin{itemize} 
\item[(a)] The category of all categories $\CAT_2$ is a strict $2$-category with categories as $0$-cells, functors as $1$-cells and natural transformations as $2$-cells.
\item[(b)] The category of topological spaces $\TOSP$ can be seen as a weak $2$-category, i.e. the $2$-cells are homotopies of continuous maps. We note that it is a weak $(2,1)$-category, e.g. composition of homotopies is not associative, but all homotopies are invertible. Note that this can be generalised to a weak $(n,1)$-category for $n\in\{1,2,\dots,\omega\}$.
\item[(c)] The category $\BMOD$ is a weak $2$-category. To be more precise.
\begin{itemize}
\item The $0$-cells are unital rings $R,S,T,\dots$.
\item The $1$-cells are bimodules ${}_RM_S, {}_SN_T$ and composition of bimodules is defined by
\[
{}_RM_S\circ_1 {}_SN_T={}_RM_S\otimes_S {}_SN_T.
\]
\item The $2$-cells are bimodule homomorphisms. Vertical and horizontal composition are
\[
f\circ_v g= f\circ g\;\;\text{ and }\;\;f\circ_h g=f\otimes_S g,
\]
where $\circ$ denotes the standard composition of bimodule maps.
\end{itemize}
Note that the higher morphisms are not invertible, i.e. it is a weak $(2,2)$-category.
\item[(d)] An example that is related to our constructions in Section~\ref{sec-vkh} and Section~\ref{sec-web} is the category $\COB_2$, i.e. the category of two dimensional cobordisms. This category is a good example why $2$-categories can be seen as two dimensional. To be more precise.
\begin{itemize}
\item The $0$-cells are disjoint unions of points $x,y,\dots$ in a fixed plane.
\item The $1$-cells $f\colon x\to y$ are one dimensional manifolds with boundary embedded in the plane whose boundary is exactly $x$ at one side and $y$ on the other. Composition is given by glueing along the common boundary.
\item The $2$-cells $\alpha\colon (f\colon x\to y)\Rightarrow (g\colon x\to y)$ are two dimensional manifolds with boundary whose boundary is exactly $f$ at the top and $g$ at the bottom. Such a $2$-cell could look like
\[
 \xy 0;/r.18pc/:
(20,2)*{\bullet}="RU"+(1,3)*{};
(-23,-18)*{x};
(-23,-1)*{x};
(23,-18)*{y};
(23,-1)*{y};
(0,0)*{f};
(0,-21.5)*{g};
(0,-12.5)*{\alpha};
(16,-3)*{\bullet}="RD"+(2.5,2)*{};
(-16,2)*{\bullet}="LU"+(-1,2)*{};
(-20,-3)*{\bullet}="LD"+(-1,2)*{};
"RU";"RD" **\crv{(4,2) & (4,-1)}; ?(.1); ?(.85);
"LD";"LU" **\crv{(-4,-2) & (-4,1)};
?(.08); ?(.85);
(7.5,0)*{}="x1"; (-7.5,0)*{}="x2";
"x1"; "x2" **\crv{(7,-10) & (-7,-10)};
(16,-20)*{\bullet}="RDD"+(2.5,-1)*{};
(-20,-20)*{\bullet}="LDD"+(-1,-3)*{};
(20,-12.5)*{\bullet}="RUD"+(3.5,1)*{};
(-16,-15)*{\bullet}="LUD";
(-16,-2.5)*{}="A"; (16.1,-14.9)*{}="B";
"RD"; "RDD" **\dir{-};
"LD"; "LDD" **\dir{-};
"A"; "LUD" **\dir{.};
"RDD"; "LDD" **\crv{(0,-17)}; ?(.57);
"RU"; "RUD" **\dir{-};
"LU"; "A" **\dir{-};
"B"; "RUD" **\crv{(18,-14.15)};
"B"; "LUD" **\crv{~*=<4pt>{.}(0,-18)}; ?(.4);
\endxy
\]
Composition is given by glueing along the common boundary.
\end{itemize}
Note that this is a $(2,2)$-category, since cobordisms are almost never invertible.
\item[(e)] Another important example we need is that a given (weak or strict) monoidal category $(\mathcal C,\otimes)$ can be seen as a (weak or strict) $2$-category $\mathfrak C$ by ``pushing up the cells''. To be more precise.
\begin{itemize}
\item We add exactly one $0$-cell called $*$.
\item We see the $0$-cells of the category $\mathcal C$ as $1$-cells of $\mathfrak C$. Composition is given by the monoidal product.
\item We see the $1$-cells of the category $\mathcal C$ as $2$-cells of $\mathfrak C$. The vertical composition is the composition in $\mathcal C$, while the horizontal composition is given by the monoidal product.
\end{itemize}
\item[(f)] Another example we need is the following. Recall that a ring $R$ is called \textit{idempoteted}, if there exists a set $E$ of idempotents $E=\{e_i\in R\mid i\in I\}$ such that
\[
R=\bigoplus_{e_i,e_j\in E}e_iRe_j.
\]
Such rings correspond to categories $\mathcal R$ whose $0$-cells are the idempotents $e_i$ and whose $1$-cells between $e_i,e_j$ are the elements of $e_iRe_j$. Note that $e_iRe_j$ is an abelian group, i.e. $\mathcal R$ is pre-additive. An example of such a ring is Beilinson-Lusztig-MacPherson~\cite{blm} ring $\U$ from Section~\ref{sec-webbasicc}. In the same vein, an idempotented category $\mathcal C$ can be seen as a pre-additive $2$-category $\mathfrak C$.
\end{itemize}
\end{ex}
Note that Example~\ref{ex-2cat} (e) shows that Theorem~\ref{thm-maclane2} includes the Theorem~\ref{thm-maclane}.
\subsection{Grothendieck groups of categories}\label{sec-techgrgr}
In this section we recall some well-known facts about the Grothendieck group of an \textit{abelian, triangulated} or \textit{additive} category $\mathcal C$. Moreover, we explain how this notion can be categorified such that it makes sense for $2$-categories. We closely follow Mazorchuk~\cite{maz}. Note that our convention is to write $K_0(\mathcal C), K_0^{\oplus}(\mathcal C)$ or $K_0^{\Delta}(\mathcal C)$ for the \textit{(usual)} Grothendieck group, the \textit{split} Grothendieck group or the \textit{triangulated} Grothendieck group of a corresponding suitable (details below) category $\mathcal C$ respectively.
\vskip0.5cm
The origin of the Grothendieck group lies in abstract algebra, i.e. it is the most natural way to extend a commutative monoid $(M,+,0)$ to an abelian group $(A,+,0)$, e.g. the well-known construction of the integers $(\bZ,+,0)$ from the natural numbers $(\bN,+,0)$. Let us be more precise.
\begin{defn}\label{defn-grgr}
The \textit{Grothendieck group} of a commutative monoid $(M,+,0)$ is a pair $(A,\phi)$ of an abelian group $A$ and a monoid homomorphism $\phi\colon M\to A$ such that the following universal property is true.

Given a monoid homomorphism $\psi\colon M\to A^{\prime}$ to any abelian group $A^{\prime}$, there exists a unique monoid homomorphism $\Phi\colon A\to A^{\prime}$ such that the following diagram commutes.
\[
\begin{xy}
  \xymatrix{
      M \ar[rr]^{\phi} \ar[rd]_{\psi}  &     &  A \ar[dl]^{\exists!\Phi}  \\
                             &  A^{\prime}  &
  }
\end{xy}
\]
Alternatively, if there exists a functor $[\cdot]\colon\MON\to\AGR$ from the category of monoids $\MON$ to the category of abelian groups $\AGR$ that is a left adjoint to the forgetful functor in the other direction, then the \textit{Grothendieck group} of a commutative monoid $(M,+,0)$ is $[M]$. 
\end{defn}
To make sense of the definition we only need to construct the Grothendieck group, since uniqueness up to isomorphisms follows from standard arguments. That the following definition works can be easily checked, as in the construction $[\bN]=\bZ$. For a given monoid $(M,+,0)$ set
\[
A=M\times M/\sim,\;(m_1,n_1)\sim (m_2,n_2)\Leftrightarrow\exists s\in M\text{ such that }m_1+n_2+s=n_1+m_2+s
\]
for the abelian group and
\[
\phi\colon M\to A,\;\phi(m)=(m,0)
\]
as the monoid homomorphism.
\vskip0.5cm
The definition of the Grothendieck group via a functor $[\cdot]$ suggest that the Definition~\ref{defn-grgr} can be generalised to categories with small skeleton and some extra properties. To be more precise, we recall the following classical definition for \textit{abelian} categories.
\begin{defn}\label{defn-grgrcat}
Let $\mathcal A$ be an abelian category with a small and fixed skeleton $\SK(\mathcal A)$. The \textit{Grothendieck group} $K_0(\mathcal A)$ of $\mathcal A$ is defined as the quotient of the free abelian group generated by all $A\in\SK(\mathcal A)$ modulo the relation
\[
A_2=A_1+A_3\Leftrightarrow\exists\text{ an exact sequence }0\to A_1\to A_2\to A_3\to 0.
\]
The elements of $K_0(\mathcal A)$ are denoted by $[A]$ for $A\in\SK(\mathcal A)$.
\end{defn}
It is easy to check that for this construction, given an additive function $\phi\colon\mathcal A\to A^{\prime}$ for an abelian group $A^{\prime}$, there exists a unique group homomorphism $\Phi\colon K_0(\mathcal A)\to A^{\prime}$ such that the following diagram commutes.
\[
\begin{xy}
  \xymatrix{
      \text{\phantom{K-}}\mathcal A \text{\phantom{-}}\ar[rr]^{[\cdot]} \ar[rd]_{\phi}  &     &  K_0(\mathcal A) \ar[dl]^{\exists!\Phi}  \\
                             &  A^{\prime}  &
  }
\end{xy}
\]
Hence, one can say that this construction is the ``most natural'' way to make the category $\mathcal A$ into an abelian group $K_0(\mathcal A)$.
\begin{ex}\label{ex-grothen1}
Let $K$ be an arbitrary field.
\begin{itemize}
\item[(a)] We explain the example $\mathcal C=$\textbf{FinVec}${}_K$ from Section~\ref{sec-intro} in detail now. Since two finite dimensional $K$-vector spaces $V,W$ are isomorphic iff $\dim V=\dim W$, we see that
\[
[V]=[K^{\dim V}]\in K_0(\mathcal C)\;\text{ for all }\;V\in\Ob(\mathcal C).
\]
For $m=m_1+m_2$ we observe that one can construct an injection $\iota\colon K^{m_1}\to K^{m}$ and a projection $p\colon K^{m}\to K^{m_2}$ with $\mathrm{im}(\iota)=\ker(p)$. Therefore, we see that
\[
0\to K^{m_1}\stackrel{\iota}{\rightarrow}K^m\stackrel{p}{\rightarrow}K^{m_2}\to 0\;\Rightarrow\; [K^m]=[K^{m_1}]+[K^{m_2}]\in K_0(\mathcal C).
\]
Hence, the Grothendieck group $K_0(\mathcal C)$ is isomorphic to $\bZ$ and generated by $[K]\in K_0(\mathcal C)$.
\item[(b)] Given a finite dimensional $K$-algebra $A$, we can consider the category of its finite dimensional (left) $A$-modules, denoted by $\AMOD$. Assume that $S_1,\dots, S_k$ form a complete set of pairwise non-isomorphic, simple $A$-modules. Then $K_0(\AMOD)$ is isomorphic to the free abelian group with the set $\{[S_i]\mid i=1,\dots,k\}$ as a basis. 
\end{itemize}
\end{ex}
\vskip0.2cm
Note that, if we only consider an additive category $\mathcal A^{\oplus}$, then it makes no sense to speak about exact sequences in general. Hence, one considers the notion of the \textit{split} Grothendieck group $K_0^{\oplus}(\mathcal A^{\oplus})$ as explained below. Similarly, given a triangulated category $\mathcal T$, it is more convenient to use another notion known as \textit{triangulated} Grothendieck group $K_0^{\Delta}(\mathcal T)$. Notice that every abelian or triangulated category is additive, but the split Grothendieck group can be bigger than $K_0$ or $K_0^{\Delta}$. We recall the following definition.
\begin{defn}\label{defn-grgradd}
Let $\mathcal A_{\oplus}$ be an additive category with a small and fixed skeleton $\SK(\mathcal A_{\oplus})$. The \textit{split Grothendieck group} $K^{\oplus}_0(\mathcal A_{\oplus})$ of $\mathcal A_{\oplus}$ is defined as the quotient of the free abelian group generated by all $A\in\SK(\mathcal A_{\oplus})$ modulo the relation
\[
A_2=A_1+A_3\Leftrightarrow A_2\simeq A_1\oplus A_2.
\]
Let $\mathcal T$ be a triangulated category with a small and fixed skeleton $\SK(\mathcal T)$. Then the \textit{triangulated Grothendieck group} $K^{\Delta}_0(\mathcal T)$ of $\mathcal T$ is defined as the quotient of the free abelian group generated by all $T\in\SK(\mathcal T)$ modulo the relation
\[
T_2=T_1+T_3\Leftrightarrow\exists\text{ a triangle }T_1\to T_2\to T_3\to T_1[1].
\]
The \textit{split Grothendieck group} $K^{\oplus}_0(\mathcal A), K^{\oplus}_0(\mathcal T)$ is defined as before, since every abelian or triangulated category $\mathcal A$ or $\mathcal T$ is additive.
\end{defn}
An interesting example is the following.
\begin{ex}\label{ex-grothen2}
Given any field $K$ and any finite dimensional $K$-algebra $A$, then we can consider $\mathcal C=\AMOD$ or the category of finite dimensional, projective (left) $A$-modules $\mathcal D=\ApMOD$. Notice that $\mathcal C$ is abelian and $\mathcal D$ is additive. Moreover, assume that $P_1,\dots, P_n$ form a complete set of pairwise non-isomorphic indecomposable, projective $A$-modules. Hence, we have the following set of pairwise non-isomorphic simple $A$-modules.
\[
S=\{S_1=P_1/\mathrm{rad}(P_1),\dots,S_n=P_n/\mathrm{rad}(P_n)\}.
\]
As explained above in Example~\ref{ex-grothen1}, $K_0(\AMOD)$ is isomorphic to the free abelian group with the set of all simple $A$-modules as a basis. In the same vein one can see that $K^{\oplus}_0(\ApMOD)$ is isomorphic to the free abelian group with the set of all indecomposable, projective $A$-modules as a basis. Now the obvious embedding functor $\iota\colon \ApMOD\to\AMOD$ gives rise to an injective group homomorphism
\[
[\iota]\colon K^{\oplus}_0(\ApMOD)\to K_0(\AMOD),\;[P]\mapsto [P].
\]
If $A$ has finite global dimension, i.e. the supremum of all projective dimensions of $A$-modules is finite, then $[\iota]$ is an isomorphism giving another basis of $K_0(\AMOD)$. Note that the converse does not apply, e.g. the map $[\iota]$ is also surjective for the $\mathfrak{sl}_2$ analogue of the algebra we define in Section~\ref{sec-webalg}, but the algebra itself has infinite global dimension (see Brundan and Stroppel~\cite{bs1}).
\end{ex}
\begin{rem}\label{rem-grothen}
The Grothendieck group of an abelian, triangulated or additive and \textit{monoidal} category $\mathcal C$ is in fact a ring, i.e. the multiplication is induced by the monoidal product. In this case one calls the corresponding Grothendieck group the (usual, triangulated or split) \textit{Grothendieck ring}.
\end{rem}
\vskip0.2cm
It is worth noting that one motivation to introduce and study Grothendieck groups in the mid 1950s was to give a definition of \textit{generalised Euler characteristic}. To be more precise.
\begin{defn}\label{defn-euler}
Let $\mathcal C$ be an abelian, triangulated or additive category. Denote by $\Kom_b(\mathcal C)$ the category of bounded complexes consisting of
\begin{itemize}
\item The $0$-cells are bounded complexes of the form
\[
\begin{xy}
  \xymatrix{
      0\ar[r] & C_{-k} \ar[r]^/-.3em/{c_{-k}}    &   C_{-k+1} \ar[r] & \cdots\ar[r] &   C_{l-1} \ar[r]^/.3em/{c_{l-1}} &   C_{l} \ar[r] & 0,}
\end{xy}
\]
for some $k,l\in\bN$ and $C_i\in\Ob(\mathcal C)$ and suitable $c_i\in\Arr_1(C_i,C_j)$ such that $c_{i+1}\circ c_i=0$. Note that this makes sense in any abelian, triangulated or additive category.
\item The $1$-cells are maps of complexes, consisting of $1$-cells of $\mathcal C$, together with the standard requirement of commuting squares.
\end{itemize}
Denote with $(C_*,c_*)\in\Ob(\Kom_b(\mathcal C))$ such a bounded complex. Then the \textit{Euler characteristic of $(C_*,c_*)$} is defined by
\[
\chi(C_*)=\sum_{i\in\bZ}(-1)^i[C_i],
\]
with $[C_i]\in K_0(\mathcal C)$, if $\mathcal C$ is abelian, $[C_i]\in K^{\Delta}_0(\mathcal C)$, if $\mathcal C$ is triangulated, or $[C_i]\in K^{\oplus}_0(\mathcal C)$, if $\mathcal C$ is additive. This is well-defined, since the complex is bounded.
\end{defn}
The question that arise is if the generalised Euler characteristic is an invariant under homotopy. Or equivalent, given an additive category $\mathcal C$, then $\Kom_b(\mathcal C)$ is triangulated. Therefore, one can ask is $K_0^{\oplus}(\mathcal C)$ isomorphic to $K_0^{\Delta}(\Kom_b(\mathcal C))$. The answer is yes. We refer to Rose~\cite{rose} for a proof.
\begin{thm}\label{thm-rose}
Let $\mathcal C$ be an additive category and denote the category of bounded complexes by $\mathcal D=\Kom_b(\mathcal C)$. Let $(C_*,c_*),(D_*,d_*)\in\Ob(\mathcal D)$ be homotopy equivalent.
\begin{itemize}
\item[(1)] There exists an isomorphism of groups $K_0^{\oplus}(\mathcal C)\simeq K_0^{\Delta}(\mathcal D)$.
\item[(2)] We have $\chi(C_*)=\chi(D_*)$. 
\end{itemize}
\end{thm}
\vskip0.2cm
Now we are able to \textit{``categorify''} the definitions above. Let $\mathfrak C$ be an abelian, triangulated or additive $2$-category. Then taking some ``Grothendieck group like construction'' should lead to a $1$-category by \textit{identifying structures} on the level of $2$-cells. By a slight abuse of notation, we write $K^{*}_0$ as a short hand notation for the three different cases $K_0,K_0^{\Delta}$ and $K_0^{\oplus}$, since it should be clear which definition should be used. Moreover, we write simply Grothendieck group instead of usual, triangulated or split Grothendieck group. We recall the following definition.
\begin{defn}\label{defn-grgr2cata}
Let $\mathfrak C$ be an abelian, triangulated or additive $2$-category, then the \textit{Grothendieck group} (or Grothendieck \textit{category}) of $\mathfrak C$, denoted $K^{*}_0(\mathfrak C)$, is defined as follows.
\begin{itemize}
\item The $0$-cells of $K_0^{*}(\mathfrak C)$ are exactly the $0$-cells of $\mathfrak C$.
\item Note that for fixed $C,D\in\Ob(\mathfrak C)$ the collection of $1$-cells $\Arr^{\mathfrak C}_1(C,D)$ and the corresponding $2$-cells forms a category under vertical composition $\circ_v$ of $2$-cells. We can take the Grothendieck group of this category!
\item The $1$-cells between two $0$-cells $C,D\in\Ob(\mathfrak C)=\Ob(K_0^{*}(\mathfrak C))$ are precisely the elements of the Grothendieck group mentioned before. For a given $2$-cell $\alpha\colon f\Rightarrow g$, we denote the corresponding $1$-cell by $[\alpha]$, where $f,g\colon C\to D$ are $1$-cells of $\mathfrak C$.
\item Composition of $1$-cells $[\alpha],[\beta]$ should be given by
\[
[\alpha]\circ [\beta]=[\alpha\circ_h \beta].
\] 
\end{itemize}
\end{defn}
Note that the Grothendieck group $K_0^{*}(\mathfrak C)$ of $\mathfrak C$ is enriched over $\AGR$, i.e. it is \textit{pre-additive}. The following example illustrates the notion from Definition~\ref{defn-grgr2cata}.
\begin{ex}\label{ex-grothen3}
The $\bC$-algebra of \textit{dual numbers} is defined by $D=\bC[X]/(X^2)$. We note that Khovanov homology as explained in Section~\ref{sec-intro} uses a vector space $V$ of dimension two which can be realised as the dual numbers (more about this later in Example~\ref{ex-grothen5}). Consider the category $\mathcal C=D$-\textbf{Mod} and denote by $M$ the $D$-bimodule $M=D\otimes_{\bC} D$. Define a $2$-category $\mathfrak C$ with the following data.
\begin{itemize}
\item The $2$-category has only $\mathcal C$ as a $0$-cell.
\item Denote the category of all functors $F\colon\mathcal C\to\mathcal C$ by $\mathcal D$. The collection of $1$-cells and $2$-cells should be the full additive subcategory of $\mathcal D$ of functors isomorphic to direct sums of copies of $\mathrm{id}\colon\mathcal C\to\mathcal C$ or $M\otimes_D-\colon\mathcal C\to\mathcal C$. Since one easily verifies that
\[
M\otimes_D M\simeq M\oplus M,
\]
this collection is closed under composition.
\end{itemize}
Using the isomorphism above, we see that $[\mathrm{id}]$ and $[M\otimes_D-]$ form a basis of the $1$-cells of $K^{\oplus}_0(\mathfrak C)$.
\end{ex}
Note that, given a monoidal category $\mathcal C$, then one can see this as a $2$-category $\mathfrak C$, as explained in Example~\ref{ex-2cat}, then $K_0^{*}(\mathfrak C)$ can be seen as a ring that is isomorphic to the ``classical'' Grothendieck ring of Remark~\ref{rem-grothen}.
\subsection{Grothendieck groups and categorification}\label{sec-techgrgrcat}
In this section we define what we mean by a \textit{Grothendieck group categorification} or \textit{decategorification}. We follow Mazorchuk~\cite{maz}. In the following, as always, all categories are assumed to have a small skeleton. Note that a \textit{Grothendieck decategorification} of a category $\mathcal C$ is either $K_0$, if the category $\mathcal C$ is abelian, $K_0^{\Delta}$, if the category is triangulated, or $K_0^{\oplus}$, if $\mathcal C$ is additive, as defined in Section~\ref{sec-techgrgr}. As before, we write $K^{*}_0$ as a short hand notation for the three different cases $K_0,K_0^{\Delta}$ and $K_0^{\oplus}$ and skip the words usual, triangulated and split.
\vskip0.5cm
If $R$ is a commutative ring with $1\in R$, then we would like to speak about a \textit{$R$-decategorification}. The reader may think of a \textit{categorification} of some algebra over $R$.
\begin{defn}\label{defn-grgrdecat}
Let $\mathcal C$ be an abelian, triangulated or additive category. Then a \textit{(Grothendieck) decategorification} of $\mathcal C$ is the abelian group $K^{*}_0(\mathcal C)$ from Definition~\ref{defn-grgrcat} or Definition~\ref{defn-grgradd}. A \textit{$R$-decategorification} for some commutative, unital ring $R$, is defined by $K^{*}_0(\mathcal C)\otimes_{\bZ} R$.
\end{defn}
One wants to define a \textit{$R$-categorification} of an $R$-module $M$. If a $R$-decategorification is an abelian group, then a $R$-categorification should be an abelian, triangulated or additive category. To be more precise.
\begin{defn}\label{defn-grgrRcat}
A \textit{$R$-categorification} of an $R$-module $M$ is a pair $(\mathcal C,\phi)$ of an abelian, triangulated, or additive category $\mathcal C$ together with a isomorphism
\[
\phi\colon M\to K^{*}_0(\mathcal C)\otimes_{\bZ} R.
\]
We say that the pair $(\mathcal C,\phi)$ is a $R$-\textit{pre}categorification, if we only assume that $\phi$ is a monomorphism.
\end{defn}
\begin{ex}\label{ex-grothen4}
\begin{itemize}
\item[(a)] As always, there is no reason to speak about THE categorification. Only the decategorification is uniquely defined. For example, take $R=\bZ$ and set $M=\bZ^n$. If $A$ is any $K$-algebra, for any field $K$, such that the category $\AMOD$ has exactly $n$ simple modules, then this category is a categorification of $M$ as explained in Example~\ref{ex-grothen1}.
\item[(b)] Consider the $\bC$-algebra of dual number $D$ from Example~\ref{ex-grothen3} again and set $\mathcal C=D-$\textbf{Mod} as before. Then the monomorphism
\[
\phi\colon\bZ\to K_0(\mathcal C),\;1\mapsto [D]
\]
is not surjective since there is a \textit{unique} (up to isomorphisms) simple $D$-module $\bC$ (the $X$ annihilates it) and $D(\not\cong\bC)$ is its projective cover. But tensoring over $\bZ$ with $\bQ$ induces an isomorphism
\[
\phi^{\prime}\colon\bQ\to K_0(\mathcal C)\otimes_{\bZ}\bQ.
\]
Hence, the pair $(\mathcal C,\phi^{\prime})$ is a $\bQ$-categorification of $\bQ=\bQ\otimes_{\bZ}\bZ$.

It is worth noting that in fact in ``most interesting'' cases the ``natural'' basis of the Grothendieck group is given by indecomposable, projective modules, but the example illustrates that the set of indecomposable, projective modules is not a basis in general.
\end{itemize} 
\end{ex}
\vskip0.2cm
One also wants to speak of a \textit{category of categorifications} of a module. Hence, what we need is a $1$-cell between $R$-categorifications. Note that the same also works for \textit{pre}categorifications.
\begin{defn}\label{defn-grgrmor}
Let $M$ be an $R$-module and let $(\mathcal C,\phi)$ and $(\mathcal D,\psi)$ be two $R$-categorifications as in Definition~\ref{defn-grgrRcat}. A morphism of such categorifications is an exact, triangular or additive (for the three cases abelian, triangular or additive) functor $F\colon\mathcal C\to\mathcal D$ such that the following diagram commutes.
\[
\begin{xy}
  \xymatrix{
      K^{*}_0(\mathcal C)\ar[rr]^{[F]}  &     &  K^{*}_0(\mathcal D)  \\
                             &  M \ar[ul]^{\phi}\ar[ur]_{\psi}  &
  }
\end{xy}
\]
\end{defn}
\vskip0.2cm
Now we have (almost) enough terminology to \textit{categorify} $R$-modules using module categories of finite dimensional $K$-algebras. The last technical point we need is to extend the constructions above to the $\bZ$-graded setting, since we need a quantum degree $q$ for our purposes. We give the needed terminology now. Graded always means $\bZ$-graded and the categories in Definition~\ref{defn-grgrgr} should be graded (which implies that its Grothendieck group is a $\bZ[q,q^{-1}]$-module).
\begin{defn}\label{defn-grgrgr}
Let $\iota\colon\bZ[q,q^{-1}]\to R$ a homomorphism with $\iota(1)=1$, where $R$ is a graded, commutative, unital ring. Hence, $R$ can be seen as a (right) $\bZ[q,q^{-1}]$-module. As before, let $\mathcal C$ be an abelian, triangulated or additive category. A \textit{$\iota$-decategorification} of $\mathcal C$ is defined by
\[
K^R_{\iota}(\mathcal C)=K^{*}_0(\mathcal C)\otimes_{\bZ[q,q^{-1}]}R.
\]
In the same vein, given a graded $R$-module $M$, then a \textit{$R$-categorification} is a pair $(\mathcal C,\phi)$ of an abelian, triangulated or additive category $\mathcal C$ with a fixed, free $\bZ$-action on it and an isomorphism $\phi\colon M\to K^R_{\iota}(\mathcal C)$.
\end{defn}
Note that in most examples the homomorphism $\iota$ is the canonical inclusion. In this case we write for simplicity $K_0^{*}$ instead of $K^R_{\iota}$.
\begin{ex}\label{ex-grothen5}
An interesting example is the following. The $\bC$-algebra of dual numbers from Example~\ref{ex-grothen3} has a natural (in the sense that the dual numbers can be obtained as a cohomology ring of certain flag varieties) grading, that is $X$ should be of degree $2$. The reader should also compare this to Khovanov homology explained in Section~\ref{sec-intro}, where the vector space $V$ has a basis $v_+$ and $v_-$ of degrees $1$ and $-1$ respectively. Now set $\mathcal C=D-$\textbf{Mod}${}_{\mathrm{gr}}$, i.e. finite dimensional, graded $D$-modules. Hence, $\bZ[q,q^{-1}]\simeq K_0(\mathcal C)$ as $\bZ[q,q^{-1}]$-modules. Therefore, we see that the graded category $\mathcal C$ is a $(\bZ[q,q^{-1}],\mathrm{id})$-categorification of $\bZ[q,q^{-1}]$.
\end{ex}
\vskip0.2cm
Given a graded, commutative and unital ring $R$, one can speak of \textit{$R$-(de)categorifications} of a suitable $2$-category as in Definition~\ref{defn-grgr2cata}.
\begin{defn}\label{defn-grgrcatb}
Let $R,\mathfrak C$ be as before. The \textit{$R$-decategorification} $K_0^R(\mathfrak C)$ of $\mathfrak C$ is the category
\[
K_0^R(\mathfrak C)=K^{*}_0(\mathfrak C)\otimes_{\bZ}R\;\text{ or }\;K_0^R(\mathfrak C)=K^{*}_0(\mathfrak C)\otimes_{\bZ[q,q^{-1}]}R
\]
(ungraded or graded) and, given a category $\mathcal C$ enriched over $R-\MOD$ (or the graded version), also called \textit{$R$-linear} or \textit{graded $R$-linear}, a \textit{$R$-categorification} of $\mathcal C$ is a pair $(\mathfrak C,\Phi)$, where $\mathfrak C$ denotes a $2$-category as before and $\Phi\colon \mathcal C\to K_0^R(\mathfrak C)$ is an isomorphism.
\end{defn}
\subsection{Higher representation theory}\label{sec-techhigherrep}
In the present section we are going to recall some definitions from \textit{higher representation theory}. Here ``higher'' means \textit{$2$-representation theory}. To be more precise, after recalling some basic notions from $2$-representation theory, we specify from the general case to the case of a \textit{$2$-representation of $\mathfrak{g}$}. Roughly speaking, while a ``classical'' representation of $\mathfrak{g}$ is given by an action on a $K$-vector space, a ``higher'' representation of $\mathfrak{g}$ is given by an action on a $K$-linear category. That is, instead of studying linear maps, one studies linear functors \textit{and} natural transformations between these functors. The latter are \textit{``invisible''} in classical representation theory. We denote by $K$ an arbitrary field in the whole section. Recall that we assume that every $n$-category and every $n$-functor is strict.
\vskip0.5cm
We follow Cautis and Lauda~\cite{cala} in this section. Note that the first systematic study of these higher representations is due to Chuang and Rouquier~\cite{chro}. It is worth noting that Rouquier~\cite{rou} formalises the notion of a $2$-Kac-Moody representation, i.e. he defines $2$-categories $\mathfrak C$ associated to Kac-Moody algebras and then he defines such a $2$-representation as a $2$-functor $F\colon\mathfrak C\to\mathfrak D$ to an appropriated $2$-category $\mathfrak D$.

We start by recalling the basic ideas.
\begin{defn}\label{defn-higherrep}
Let $\mathfrak C,\mathfrak D$ be two $2$-categories. A $2$-representation of $\mathfrak C$ on $\mathfrak D$ is a $2$-functor $F\colon\mathfrak C\to\mathfrak D$.
\end{defn}
The basic question is, given for example a group $G$, what are the \textit{symmetries} the group acts on, i.e. its representation theory. In this spirit, \textit{higher representation theory} should be the study of the \textit{``higher'' symmetries} (e.g. natural transformations). Usually one has to refine Definition~\ref{defn-higherrep}. For instance, the following definition is more suitable for our purposes. First recall the notion of a Cartan datum.
\begin{defn}\label{defn-higherqg1}
A \textit{Cartan datum} for an fixed index set $I$ consists of the following.
\begin{itemize}
\item A \textit{weight lattice} $X$ and two subsets $\alpha,\Lambda\subset X$ called the set of \textit{simple roots} and \textit{fundamental weights} respectively. Here $\alpha,\Lambda$ are indexed by $I$, i.e. we have $\alpha=\{\alpha_i\mid i\in I\}$ and $\Lambda=\{\Lambda_i\mid i\in I\}$.
\item A set $X^{\vee}=\mathrm{hom}_{\bZ}(X,\bZ)$ of \textit{simple co-roots}. Again, this set is indexed by $I$, i.e. we have $X^{\vee}=\{h_i\mid i\in I\}$.
\item A bilinear form $(-,-)\colon X\times X\to\bZ$ and a canonical pairing $\langle-,-\rangle\colon X^{\vee}\times X\to\bZ$ that satisfies the following.
\begin{itemize}
\item For all $i\in I$ and all $\lambda\in X$ we have $0\neq(\alpha_i,\alpha_i)\in 2\bZ$ and
\[
\langle h_i,\lambda\rangle=2\frac{(\alpha_i,\lambda)}{(\alpha_i,\alpha_i)}.
\]
\item For all $i,j\in I$ with $i\neq j$ we have $(\alpha_i,\alpha_j)\leq 0$.
\item For all $i,j\in I$ we have $\langle h_i,\Lambda_j\rangle=\delta_{ij}$.
\end{itemize}
\end{itemize}
We use the two short hand notations $\alpha^{ij}=(\alpha_i,\alpha_j)$ and $\alpha^{\lambda i}=\langle\lambda,\alpha_j\rangle$.
\end{defn}
\begin{defn}\label{defn-strep}
A \textit{strong $2$-representation} of $\mathfrak g$ is a graded, additive, $K$-linear $2$-category $\mathfrak D$ with the following data.
\begin{itemize}
\item The $0$-cells are indexed by the weights $\lambda\in X$.
\item The $2$-category has identity $1$-cells $\mathbbm{1}_{\lambda}$ for all weights. Moreover, for all weights and all $i\in I$ there are $1$-cells $E_i\mathbbm{1}_{\lambda}\colon\lambda\to\lambda+\alpha_i$ and $E_{-i}\mathbbm{1}_{\lambda}\colon\lambda+\alpha_i\to\lambda$ such that the following holds.
\begin{itemize}
\item All $1$-cells $E_i\mathbbm{1}_{\lambda}$ have left and right adjoints. The right adjoints are the $E_{-i}\mathbbm{1}_{\lambda}$.
\item We have $E_{-i}\mathbbm{1}_{\lambda}=E_{i}\mathbbm{1}_{\lambda}\{-\alpha^{ii}-\frac{1}{2}\alpha^{\lambda i}\}$.
\end{itemize}
Other $1$-cells are obtained by composition, sums, grading shifts or adding images of idempotent $2$-cells.
\end{itemize}
These data should satisfy the conditions (1)-(5) below.
\begin{itemize}
\item[(1)] The $0$-cells $\lambda+k\alpha_i$ are isomorphic to the zero object for all $i\in I$ and $k\ll 0$ or $k\gg 0$. This condition is known as \textit{integrability}.
\item[(2)] The space of $2$-cells between any two $1$-cells is finite dimensional. Moreover, for all weights we assume that $\Arr_1(\mathbbm{1}_{\lambda},\mathbbm{1}_{\lambda})$ is one dimensional and $\Arr_1(\mathbbm{1}_{\lambda},\mathbbm{1}_{\lambda}\{k\})$ is zero for negative shifts $k<0$.
\item[(3)] In $\mathfrak{D}$ there exists isomorphisms
\[
E_{i}E_{-i}\mathbbm{1}_{\lambda}\simeq E_{-i}E_{i}\mathbbm{1}_{\lambda}\oplus \mathbbm{1}_{\lambda}\;\text{if}\;\langle h_i,\lambda\rangle>0\;\text{ and }\;E_{-i}E_{i}\mathbbm{1}_{\lambda}\simeq E_{i}E_{-i}\mathbbm{1}_{\lambda}\oplus \mathbbm{1}_{\lambda}\;\text{if}\;\langle h_i,\lambda\rangle<0.
\]
\item[(4)] In $\mathfrak{D}$ there exists isomorphisms for $i,j\in I$ and $i\neq j$
\[
E_{i}E_{-j}\mathbbm{1}_{\lambda}\simeq E_{-j}E_{i}\mathbbm{1}_{\lambda}.
\]
\item[(5)] The $E_{+i}$'s carry an action of the Khovanov-Lauda-Rouquier algebra (see Section~\ref{sec-webbasicc}). Note that this gives an action on the $E_{-i}$'s, too. More details can be found in~\cite{cala}.
\end{itemize}
\end{defn}
We are going to recall parts of the definition of the $2$-categories $\mathcal U(\mathfrak g)$ and its Karoubi envelope $\dot{\mathcal U}(\mathfrak g)$ also called idempotent completion. We are not going to recall the relations in general since we do not need them for general $\mathfrak g$. We are giving the full definition in the cases $\mathfrak g=\mathfrak{gl}_n$ and $\mathfrak g=\mathfrak{sl}_n$ in Section~\ref{sec-webbasicc}. For the full list of relations we refer to~\cite{cala}. Note that, for simplicity, we chose the scalars considered in~\cite{cala} to be $t_{ij}=1=t_{ji}$ and $s^{pq}_{ij}=0$ for all $i,j\in I$ and suppress this notion in the following.
\begin{defn}\label{defn-higherqg2}
Let us fix a Cartan datum as in Definition~\ref{defn-higherqg1}. The $2$-category $\mathcal U(\mathfrak g)$ for this datum is the graded, additive, $K$-linear $2$-category with the following cells.
\begin{itemize}
\item The $0$-cells are the $\lambda\in X$.
\item The $1$-cells (defined for all $i\in I$ and $\lambda\in X$) are formal direct sums of compositions of
\[
\onel,\;\;\mathcal{E}_{+i}\onel={\mathbf 1}_{\lambda+\alpha_i}\mathcal{E}_{+i}={\mathbf 1}_{\lambda+\alpha_i}\mathcal{E}_{+i}\onel\;\text{ and }\;\mathcal{E}_{-i}\onel={\mathbf 1}_{\lambda-\alpha_i}\mathcal{E}_{-i}={\mathbf 1}_{\lambda-\alpha_i}\mathcal{E}_{-i}\onel.
\]
\item The $2$-cells are graded, $K$-vector spaces generated by compositions of diagrams shown below. Here $\{k\}$ denotes a degree shift by $k$. Moreover, we use the two short hand notations $\alpha^{ij}=(\alpha_i,\alpha_j)$ and $\alpha^{\lambda i}=\langle\lambda,\alpha_j\rangle$ again.
\[
\phi_1=\xy
 (0,1)*{\includegraphics[width=09px]{res/figs/section23/upsimple}};
 (1.5,-4)*{\scriptstyle i};
 (3,1)*{\scriptstyle\lambda};
 (-5,1)*{\scriptstyle\lambda+\alpha_i};
 \endxy\,\hspace{8mm}
\,
\phi_2=\xy
 (0,1)*{\includegraphics[width=09px]{res/figs/section23/upsimpledot}};
 (1.5,-4)*{\scriptstyle i};
 (3,1)*{\scriptstyle\lambda};
 (-5,1)*{\scriptstyle\lambda+\alpha_i};
 \endxy\,\hspace{8mm}
\,\phi_3=\xy
 (0,1)*{\includegraphics[width=25px]{res/figs/section23/upcrosscurved}};
 (-5,-3)*{\scriptstyle i};
 (5,-3)*{\scriptstyle j};
 (5.5,0)*{\scriptstyle\lambda};
 \endxy\,\hspace{8mm}\,
 \phi_4=\xy
 (0,0)*{\includegraphics[width=25px]{res/figs/section23/leftcup}};
 (0,-3)*{\scriptstyle i};
 (5,0)*{\scriptstyle\lambda};
 \endxy\hspace{8mm}\,
 \phi_5=\xy
 (0,0)*{\includegraphics[width=25px]{res/figs/section23/rightcup}};
 (0,-3)*{\scriptstyle i};
 (5,0)*{\scriptstyle\lambda};
 \endxy
\]
with $\phi_1=\mathrm{id}_{\mathcal E_{i}\onel}$, $\phi_2\colon\mathcal E_{i}\onel\Rightarrow\mathcal E_{i}\onel\{\alpha^{ii}\}$, $\phi_3\colon\mathcal E_{i}\mathcal E_{j}\onel\Rightarrow\mathcal E_{j}\mathcal E_{i}\onel\{\alpha^{ij}\}$ and cups and caps $\phi_4\colon\mathcal E_{i}\mathcal E_{i}\onel\Rightarrow\onel\{\frac{1}{2}\alpha^{ii}+\alpha^{\lambda i}\}$ and $\phi_5\colon\mathcal E_{i}\mathcal E_{i}\onel\Rightarrow\onel\{\frac{1}{2}\alpha^{ii}-\alpha^{\lambda i}\}$. Moreover, we have diagrams of the form
\[
\psi_1=\xy
 (0,1)*{\includegraphics[width=09px]{res/figs/section23/downsimple}};
 (1.5,-4)*{\scriptstyle i};
 (3,1)*{\scriptstyle\lambda};
 (-5,1)*{\scriptstyle\lambda-\alpha_i};
 \endxy\,\hspace{8mm}
\,
\psi_2=\xy
 (0,1)*{\includegraphics[width=09px]{res/figs/section23/downsimpledot}};
 (1.5,-4)*{\scriptstyle i};
 (3,1)*{\scriptstyle\lambda};
 (-5,1)*{\scriptstyle\lambda-\alpha_i};
 \endxy\,\hspace{8mm}
\,\psi_3=\xy
 (0,1)*{\includegraphics[width=25px]{res/figs/section23/downcrosscurved}};
 (-5,-3)*{\scriptstyle i};
 (5,-3)*{\scriptstyle j};
 (5.5,0)*{\scriptstyle\lambda};
 \endxy\,\hspace{8mm}\,
 \psi_4=\xy
 (0,0)*{\includegraphics[width=25px]{res/figs/section23/leftcap}};
 (0,-3)*{\scriptstyle i};
 (5,0)*{\scriptstyle\lambda};
 \endxy\hspace{8mm}\,
 \psi_5=\xy
 (0,0)*{\includegraphics[width=25px]{res/figs/section23/rightcap}};
 (0,-3)*{\scriptstyle i};
 (5,0)*{\scriptstyle\lambda};
 \endxy
\]
with $\psi_1=\mathrm{id}_{\mathcal E_{-i}\onel}$, $\psi_2\colon\mathcal E_{-i}\onel\Rightarrow\mathcal E_{-i}\onel\{\alpha^{ii}\}$, $\psi_3\colon\mathcal E_{-i}\mathcal E_{-j}\onel\Rightarrow\mathcal E_{-j}\mathcal E_{-i}\onel\{\alpha^{ij}\}$ and cups and caps $\psi_4\colon\mathcal E_{-i}\mathcal E_{i}\onel\Rightarrow\onel\{\frac{1}{2}\alpha^{ii}+\alpha^{\lambda i}\}$ and $\psi_5\colon\mathcal E_{i}\mathcal E_{-i}\onel\Rightarrow\onel\{\frac{1}{2}\alpha^{ii}-\alpha^{\lambda i}\}$.
\item The convention for reading these diagrams is from right to left and bottom to top. The $2$-cells should satisfy several relation which we will not recall here. Details are either in Cautis and Lauda's paper~\cite{cala} or in the special cases $\mathfrak g=\mathfrak{gl}_n$ and $\mathfrak g=\mathfrak{sl}_n$ in Section~\ref{sec-webbasicc}.
\end{itemize}
We denote the \textit{Karoubi envelope} of $\mathcal U(\mathfrak g)$ by $\dot{\mathcal U}(\mathfrak g)$ and of a general category $\mathcal C$ by a slight abuse of notation by $\KAR(\mathcal C)$.
\end{defn}
With these definitions we are able to refine Definition~\ref{defn-higherrep} again. To be more precise.
\begin{defn}\label{defn-higherrep2}
A $2$-representation of $\mathcal U(\mathfrak g)$ in a graded, additive $2$-category $\mathfrak{D}$ is a graded, additive, $K$-linear $2$-functor $F\colon\mathcal U(\mathfrak g)\to\mathfrak{D}$. A $2$-representation of $\dot{\mathcal U}(\mathfrak g)$ in a graded, additive $2$-category $\mathfrak{D}$ is a graded, additive, $K$-linear $2$-functor $F\colon\dot{\mathcal U}(\mathfrak g)\to\mathfrak{D}$.
\end{defn}
It is worth noting that, if the Karoubi envelope $\KAR(\mathfrak{D})$ of $\mathfrak D$ is equivalent to $\mathfrak{D}$, then any $2$-representation of $\mathcal U(\mathfrak g)$ extents uniquely to a $2$-representation of $\dot{\mathcal U}(\mathfrak g)$. Note that Definition~\ref{defn-strep} of $\mathfrak g$ is the same as the definition of an \textit{integrable} $2$-Kac-Moody representation by Rouquier~\cite{rou} except that Rouquier does not require (2) from Definition~\ref{defn-strep}. But we need the notions above so we mention that Rouquier's universality theorem, which we recall in Section~\ref{sec-webbasicc}, is true with or without (2) as Webster remarks in~\cite{web1}. For completeness, Cautis and Lauda proved the following.
\begin{thm}\label{thm-cala}
Any strong $2$-representation of $\mathfrak g$ on $\mathfrak{D}$ in the sense of Definition~\ref{defn-strep} extents to a $2$-representation in the sense of Definition~\ref{defn-higherrep2}.
\end{thm}
\subsection{Cones, strong deformation retracts and homotopy equivalence}\label{sec-techhomalg}
In this section we have collected some well-known facts from homological algebra about (mapping) cones, strong deformation retracts and homotopy equivalences. We need these facts in Section~\ref{sec-vkh}. In particular, we need them in the proof of Theorem~\ref{thm-geoinvarianz}.

In this section let $\mathcal C$ denote any pre-additive category. It should be noted that this includes that the notion ``chain complex'', i.e. $d\circ d=0$, makes sense. We denote the \textit{category of chain complexes of $\mathcal C$} by $\kom(\mathcal C)$ (in contrast to the category of \textit{bounded} chain complexes $\kom_b(\mathcal C)$), i.e. the objects are chain complexes with chain groups in $\Ob(\mathcal C)$ and differentials in $\Mor(\mathcal C)$ and the morphisms are chain maps, i.e. sequences of elements of $\Mor(\mathcal C)$ with the standard requirements. 

We denote chain complexes by $C=(C_i,c_i), D=(D_i,d_i)\in\kom(\mathcal C)$. With a slight abuse of notation, we call elements of $\Mor(\mathcal C)$ simply ``maps''. All appearing indices should be elements of $\bZ$. Moreover, recall the following three definitions.
\begin{defn}\label{defn-chainhomotopy}
Let $C,D$ be two chain complexes with chain groups $C_i,D_i$ and differentials $c_i,d_i$. Let $\varphi,\varphi^{\prime}\colon C\to D$ be two chain maps. Let $h_i\colon C_i\to D_{i-1}$ be a collection of maps as illustrated below.
\begin{equation*}
\xymatrix@+3em{
{\dots} \ar[r]^{c_{i - 2}}
        & C_{i - 1}
                \ar[r]^{c_{i - 1}}
                \ar@<0.5ex>[d]^{\varphi^{\prime}_{i - 1}}
                \ar@<-0.5ex>[d]_{\varphi_{i - 1}}
                \ar[dl]|*+<1ex,1ex>{\scriptstyle h_{i - 1}}
        & C_i
                \ar[r]^{c_i}
                \ar@<0.5ex>[d]^{\varphi^{\prime}_{i}}
                \ar@<-0.5ex>[d]_{\varphi_{i}}
                \ar[dl]|*+<1ex,1ex>{\scriptstyle h_i}
        & C_{i + 1}
                \ar[r]^{c_{i + 1}}
                \ar@<0.5ex>[d]^{\varphi^{\prime}_{i + 1}}
                \ar@<-0.5ex>[d]_{\varphi_{i + 1}}
                \ar[dl]|*+<1ex,1ex>{\scriptstyle h_{i + 1}}
        & {\dots}
                \ar[dl]|*+<1ex,1ex>{\scriptstyle h_{i + 2}}\\
{\dots} \ar[r]^{d_{i - 2}}
        & D_{i - 1} \ar[r]^{d_{i - 1}}
        & D_i \ar[r]^{d_i}
        & D_{i + 1} \ar[r]^{d_{i + 1}}
        & {\dots}
}
\end{equation*}
The two chain maps $\varphi,\varphi^{\prime}\colon C\to D$ are called \textit{chain homotopic}, denoted by $\varphi\sim_h\varphi^{\prime}$, if
\[
\varphi_i-\varphi^{\prime}_i=h_{i+1}\circ c_i+d_{i-1}\circ h_i\text{ for all }i\in\bZ.
\]
Two chain complexes $C,D$ are called \textit{chain homotopic} if there exists two chain maps $\varphi\colon C\to D$ and $\psi\colon D\to C$ such that
\[
\psi_i\circ\varphi_i\sim_h\mathrm{id}_{C}\;\;\text{ and }\;\;\varphi_i\circ\psi_i\sim_h\mathrm{id}_{D}\text{ for all }i\in\bZ.
\]
Such chain maps $\varphi\colon C\to D$ and $\psi\colon C\to D$ are called \textit{homotopy equivalences}. We denote chain homotopic complexes by $C\simeq_h D$.
\end{defn}
\begin{defn}\label{defn-sdr}Let $\psi\colon D\to C$ be a homotopy equivalence. Assume that $\psi$ has a ``homotopy inverse'' $\varphi\colon C\to D$, that is
\[
\psi_i\circ\varphi_i=\mathrm{id}_{C}\;\;\text{ and }\;\;\varphi_i\circ\psi_i\sim_h\mathrm{id}_{D}\text{ for all }i\in\bZ,
\]
then $\psi$ is called a \textit{deformation retraction}. Moreover, if there exists a homotopy $h$ with $h\circ \varphi=0$, then $\psi$ is called a \textit{strong deformation retraction} and $\varphi$ is called an \textit{inclusion into a strong deformation retract}.
\end{defn}
A (mapping) cone of two chain complexes is defined below. Be careful: Some authors use a different sign convention.
\begin{defn}\label{defn-cone}(\textbf{The cone}) Let $C,D$ be two chain complexes with chain groups $C_i,D_i$ and differentials $c_i,d_i$. Let $\varphi\colon C\to D$ be a chain map. The \textit{cone} of $C,D$ \textit{along $\varphi$} is the chain complex $\Gamma(\varphi\colon C\to D)$ with the chain groups and differentials
\[
\Gamma_i=C_i\oplus D_{i-1}\text{ and } \gamma_i=\begin{pmatrix}-c_i & 0\\ \varphi_{i} & d_{i-1}\end{pmatrix},
\]
i.e. if the two chain complexes $C,D$ look like
\[
C,D\colon\cdots\xrightarrow{c_{i-1},d_{i-1}} C_i,D_i\xrightarrow{c_i,d_i}C_{i+1},D_{i+1}\xrightarrow{c_{i+1},d_{i+1}}C_{i+2},D_{i+2}\xrightarrow{c_{i+2},d_{i+2}}\cdots
\]
then the cone along $\varphi$ is generated by direct sums over the diagonal as shown below.
\begin{equation*}
\xymatrix@+3em{
{\dots} \ar[r]^{c_{i - 2}}
        & C_{i - 1}
                \ar[r]^{-c_{i - 1}}
                \ar[d]|{\varphi_{i - 1}}
                \ar@{.}[dl]|*+<1ex,1ex>{\scriptstyle \oplus}
        & C_i
                \ar[r]^{-c_i}
                \ar[d]|{\varphi_{i}}
                \ar@{.}[dl]|*+<1ex,1ex>{\scriptstyle \oplus}
        & C_{i + 1}
                \ar[r]^{-c_{i + 1}}
                \ar[d]|{\varphi_{i + 1}}
                \ar@{.}[dl]|*+<1ex,1ex>{\scriptstyle \oplus}
        & {\dots}
                \ar@{.}[dl]|*+<1ex,1ex>{\scriptstyle \oplus}\\
{\dots} \ar[r]^{d_{i - 2}}
        & D_{i - 1} \ar[r]^{d_{i - 1}}
        & D_i \ar[r]^{d_i}
        & D_{i + 1} \ar[r]^{d_{i + 1}}
        & {\dots}
}
\end{equation*}
\end{defn}
It is easy to check that the cone gives again a chain complex (here one has to use the signs above). The following well-known proposition concludes this section. We need it in order to prove Theorem~\ref{thm-geoinvarianz}.
\begin{prop}\label{prop-sdr}
Let $C,D,E,F$ be four chain complexes and let
\[ 
\begin{xy}
  \xymatrix{
      C  \ar@<2pt>[d]^{f^{\phantom{\prime}}}    &   D\ar@<2pt>[d]^{g^{\prime}}   \\
      E \ar[r]_{\varphi}\ar@<2pt>[u]^{f^{\prime}}             &   F \ar@<2pt>[u]^{g} 
  }
\end{xy}
\]
be a diagram in the category $\kom(\mathcal C)$. Assume that $f$ is an inclusion into a strong deformation retract $f^{\prime}$ and $g$ is a strong deformation retraction with inclusion $g^{\prime}$. Then
\[
\Gamma(\varphi\circ f)\simeq_h\Gamma(\varphi)\simeq_h\Gamma(g\circ \varphi).
\]
\end{prop}
\begin{proof}
In order to maintain readability, we suppress some subscripts in the following.

To show that $\Gamma(\varphi)\simeq_h\Gamma(\varphi\circ f)$ we denote their differentials by $d^{\Gamma(\varphi\circ f)}$ and $d^{\Gamma(\varphi)}$ respectively. Consider the following diagram
\begin{equation*}
\xymatrix@+3em{
{\Gamma(\varphi\circ f):\;\;\dots} \ar[r]^{}
        & C_{i + 1}\oplus F_i
                \ar[r]^{d^{\Gamma(\varphi\circ f)}}
                \ar@<2pt>[d]_{\psi^{f^{\prime}}}
        & C_{i+2}\oplus F_{i+1}
                \ar[r]^{}
                \ar@<2pt>[d]_{\psi^{f^{\prime}}}
        & {\dots}\\
{\Gamma(\varphi):\;\;\dots} \ar[r]^{}
        & E_{i + 1}\oplus F_i \ar@<2pt>[r]^{d^{\Gamma(\varphi)}}\ar@<2pt>[u]_{\phantom{.}\psi^f}
        & E_{i+2}\oplus F_{i+1} \ar[r]^{}\ar@<2pt>[u]_{\phantom{.}\psi^f}\ar@<2pt>[l]^{\psi^h}
        & {\dots},
}
\end{equation*}
where the three maps $\psi^a,\psi^b$ and $h$ are given by
\[
\psi^f=\begin{pmatrix}f & 0\\ 0& \mathrm{id}\end{pmatrix}\;\text{ and }\;\psi^{f^{\prime}}=\begin{pmatrix}f^{\prime} & 0\\ \varphi\circ h^{\prime}& \mathrm{id}\end{pmatrix}\;\text{ and }\;\psi^h=\begin{pmatrix}h & 0\\ 0& 0\end{pmatrix}.
\]
Here the map $h_i\colon E_{i}\to E_{i-1}$ should be the homotopy from the inclusion of $f$ into $f^{\prime}$. One easily checks that the diagram above is commutative and that this setting gives rise to $\Gamma(\varphi)\simeq_h\Gamma(\varphi\circ f)$. The statement $\Gamma(g\circ \varphi)\simeq_h\Gamma(\varphi)$ can be verified analogously.
\end{proof}
\subsection{Cubes and projective complexes}\label{sec-techcube}
In this section we define/recall some facts from homological algebra about cubes and projective complexes. We need them in the Section~\ref{sec-vkhca}, since we can only ensure that our assignment will be a ``projective chain complex''. That means lousily speaking that face are commutative ``up to a sign''.

Let $\mathrm{Cu}_n$ be a standard unit $n$-cube. We can consider this cube as a directed graph by labelling neighbouring vertices $\gamma$ by words $a_{\gamma}$ in $\{0,1\}$ of length $n$ as follows. Choose one vertex and give it the label $0\dots 0$ with $n$-entries. Any of its $n$ neighbours get a different word of length $n$ with exactly one $1$. Continue by changing exactly one entry until every vertex has a label.

For two vertices $\gamma_a,\gamma_b$ that differ by only one entry $k$ one assigns a label for the edges between them by replacing $k$ with a $*$. The edges is oriented from $\gamma_a$ to $\gamma_b$ iff $k=0$ for $\gamma_a$. We denote such an edge by $S\colon\gamma_a\to\gamma_b$.  Recall that $R$ denotes a commutative, unital ring of arbitrary characteristic.
\begin{defn}\label{defn-cube}(\textbf{Cube})
An $n$-cube in a $R$-pre-additive category $\mathcal C$ is a mapping
\[
\mathrm{Cu}_n^{\mathcal C}\colon\mathrm{Cu}_n\to\mathcal C
\]
that associates each vertex $\gamma_a$ with an element $\gamma^{\mathcal C}_a\in\Ob(\mathcal C)$ and each edge $S\colon\gamma_a\to\gamma_b$ with an element $S_{a,b}^{\mathcal C}\in\Mor(\gamma^{\mathcal C}_a,\gamma^{\mathcal C}_b)$.

A morphisms of cubes $\phi\colon\mathrm{Cu}_n^{\mathcal C}\to\mathrm{Cu'}_n^{\mathcal C}$ is a collection of morphisms for all vertices that is
\[
\{S_{a,{a^{\prime}}}^{\mathcal C}\mid \gamma_a,\gamma_{a^{\prime}}\text{ vertices of }\mathrm{Cu}_n^{\mathcal C},\mathrm{Cu'}_n^{\mathcal C}\}.
\]
We denote the category of $n$-cubes in $\mathcal C$ by $\Cb_n(\mathcal C)$ and the category of all cubes in $\mathcal C$ by $\Cb(\mathcal C)$.

It should be noted that a morphisms between two $n$-cubes can be seen as a $n+1$-cube. Moreover, from a $n+1$ cube one can define a morphism of $n$-cubes by fixing a letter $k$ of the words associated to the vertices $\gamma_a$ and fixing the $n$-subcubes $\mathrm{Cu^0}_n^{\mathcal C}$ and $\mathrm{Cu^1}_n^{\mathcal C}$ of $\mathrm{Cu}_{n+1}^{\mathcal C}$ such that the vertices of $\mathrm{Cu^0}_n^{\mathcal C}$ have $k=0$ and the ones of $\mathrm{Cu^1}_n^{\mathcal C}$ have $k=1$. The morphism is the given by all edges of $\mathrm{Cu}_{n+1}^{\mathcal C}$ that change $k=0$ to $k=1$.
\end{defn}
\begin{defn}\label{defn-cubetype}(\textbf{Cube types})
Denote by $R^*$ all units in $R$. The category $\mathcal C_P=\mathcal C/R^*$ is called the \textit{projectivisation} of $\mathcal C$, i.e. morphisms are identified iff they differ only by a unit. A \textit{projectivisation} of a cube $\mathrm{Cu}_n^{\mathcal C_P}$ is given by the composition with the obvious projection.

A face of a cube, denoted by $F$, is given by (we hope that the notation is clear)
\[
\xymatrix{
 & \gamma_{a_{01}}\ar[rd]^{S_{*1}} &\\
 \gamma_{a_{00}}\ar[rd]_{S_{*0}}\ar[ru]^{S_{0*}} &  & \gamma_{a_{11}},\\
 & \gamma_{a_{10}}\ar[ru]_{S_{1*}} &}
\]
or with an extra superscript $\mathcal C$ for a cube in $\mathcal C$. Such a face is said to be of \textit{type $a$, $c$ or $p$} if the following is satisfied.
\begin{itemize}
\item[(Type a)] We have $S^{\mathcal C}_{1*}\circ S^{\mathcal C}_{*0}=-S^{\mathcal C}_{*1}\circ S^{\mathcal C}_{0*}$ (\textit{anticommutative}).
\item[(Type c)] We have $S^{\mathcal C}_{1*}\circ S^{\mathcal C}_{*0}=\phantom{-}S^{\mathcal C}_{*1}\circ S^{\mathcal C}_{0*}$ (\textit{commutative}).
\item[(Type p)] We have $S^{\mathcal C}_{1*}\circ S^{\mathcal C}_{*0}=uS^{\mathcal C}_{*1}\circ S^{\mathcal C}_{0*}$ for $u\in R^*$ (\textit{projective}).
\end{itemize}
Furthermore, a cube $\mathrm{Cu}_n^{\mathcal C}$ is called of \textit{type a}, \textit{type c} or \textit{type p}, if all of its faces are of the corresponding types.

A morphisms between $n$-cubes $\mathrm{Cu}_n^{\mathcal C},\mathrm{Cu'}_n^{\mathcal C}$ is called of \textit{type a}, \textit{type c} or \textit{type p} if the corresponding $n+1$-cube is type a, type c or type p respectively.

Two cubes $\mathrm{Cu}_n^{\mathcal C}$ and $\mathrm{Cu'}_n^{\mathcal C}$ are called \textit{p-equal} if
\[
\mathrm{Cu}_n^{\mathcal C_P}=\mathrm{Cu'}_n^{\mathcal C_P}.
\]
We call two morphisms between $\mathrm{Cu}_n^{\mathcal C}$ and $\mathrm{Cu'}_n^{\mathcal C}$ \textit{p-equal} if the corresponding $n+1$-cubes are p-equal.

Note that morphisms of type c and type p are closed under compositions. Hence, the category $\Cb(\mathcal C)$ has three subcategories, namely the following.
\begin{itemize}
\item[(Type a)] The subcategory $\Cb^a(\mathcal C)$ with cubes of type a and morphisms of type c.
\item[(Type c)] The subcategory $\Cb^c(\mathcal C)$ with cubes of type c and morphisms of type c.
\item[(Type p)] The subcategory $\Cb^p(\mathcal C)$ with cubes of type p and morphisms of type p.
\end{itemize}
\end{defn}
It is worth noting that we can see any cube in $\mathcal C$ as a complex $(C_*,c_*)$ (the reader should be careful that we do not say \textit{chain} complex here) by taking direct sums of vertices with the same number of $1$ in their labels and matrices of the morphisms associated to the edges between neighbouring vertices.
\begin{defn}\label{defn-chainhomotopy2}
Let $(C_*,c_*)$ and $(D_*,d_*)$ be two cubes of type p. We call two morphisms $\varphi,\varphi^{\prime}\colon C\to D$ of type p \textit{p-homotopic}, denoted by $\varphi\sim^P_h\varphi^{\prime}$, if
\[
\varphi_i-\varphi^{\prime}_i=h_{i+1}\circ c_i+u_id_{i-1}\circ h_i\text{ for all }i\in\bZ,
\]
for a backward diagonal $h$ as in Definition~\ref{defn-chainhomotopy} and units $u_i\in R^*$.

Two such cube complexes are called \textit{p-homotopic} if there exists two morphisms of type p $\varphi\colon C\to D$ and $\psi\colon D\to C$ such that
\[
\psi_i\circ\varphi_i\sim^P_h\mathrm{id}_{C}\;\;\text{ and }\;\;\varphi_i\circ\psi_i\sim^P_h\mathrm{id}_{D}\text{ for all }i\in\bZ.
\]
Such morphisms $\varphi\colon C\to D$ and $\psi\colon C\to D$ are called \textit{p-homotopy equivalences}. We denote p-homotopic complexes of type p by $C\simeq^p_h D$.
\end{defn}
\begin{defn}\label{defn-edge}
Let $\mathrm{Cu}_n$ denote an $n$-cube and let us denote the set of edges of $\mathrm{Cu}_n$ by $\mathrm{E}(\mathrm{Cu}_n)$. An \textit{edge assignment} $\epsilon$ of the cube is a map
\[
\epsilon\colon\mathrm{E}(\mathrm{Cu}_n)\to\{+1,-1\}.
\]
Let $\mathcal C$ be a $R$-pre-additive category. Then an edge assignment $\epsilon$ of the cube   $\mathrm{Cu}_n$ is called \textit{negative} (or \textit{positive}), if $\mathrm{Cu}_n^{\mathcal C}$ is a cube of type a (or of type c) after multiplying the morphism $f_e$ of the edge $e\in\mathrm{E}(\mathrm{Cu}_n)$ of the cube $\Cb(\mathcal C)$ with $\epsilon(e)$.
\end{defn}
The following lemma follows immediately from the definition.
\begin{lem}\label{lem-edge}
If $\mathrm{Cu}_n^{\mathcal C}$ is a cube of type a (or of type c), then there is a negative (or positive) edge assignment of $\mathrm{Cu}_n^{\mathcal C_P}$.\qed
\end{lem} 
\subsection{Graded cellular algebras}\label{sec-techcell}
We recall the definition of a \textit{cellular algebra} due to Graham and Lehrer~\cite{grle} in the present section. Note that we need the definition of a \textit{graded cellular algebra} and some facts about these algebras and their representation theory. We follow the paper of Hu and Mathas~\cite{hm}. Moreover, we also need some more facts about cellular algebras, e.g. the notion is (almost) categorical in the sense that it is Morita invariant iff $\mathrm{char}(K)\neq 2$. We follow K\"onig and Xi in~\cite{kx1} in order to recall these facts.
\vskip0.5cm
Note that in the whole section $K$ denotes a field of arbitrary characteristic unless otherwise specified. Moreover, $R$ denotes a commutative, unital, integral domain and our convention for the \textit{weight poset} is that $\rhd$ denotes \textit{strictly bigger}. Furthermore, \textit{graded} always means \textit{$\bZ$-graded}\footnote{Usually if we stress that something is graded we mean with a non-trivial grading.}. A \textit{graded} (left) $R$-module is a module with a direct decomposition $M=\bigoplus_{d\in\bZ}M_d$, where the elements of $M_d$ are \textit{homogeneous of degree $d$} and $M\{i\}$ denotes a degree shift by $i\in\bZ$.

An \textit{graded} (left) $R$-algebra is an algebra that is a graded (left) $R$-module with $A_dA_{d^{\prime}}\subset A_{d+d^{\prime}}$ for all $d,d^{\prime}\in\bZ$. A \textit{graded} (left) $A$-module is a module with $A_dM_{d^{\prime}}\subset M_{d+d^{\prime}}$ for all $d,d^{\prime}\in\bZ$, while the rest is defined in the obvious way. We denote the category of (left) $A$-modules by $\AMOD$ and the category of graded (left) $A$-modules by $\AMOD_{\mathrm{gr}}$.

\begin{defn}\label{defn-cellular}
Suppose $A$ is a graded free algebra over $R$ of finite rank. A \textit{graded cell datum} is an ordered quintuple $(\mathfrak{P},\mathcal T,C,\mathrm{i},\mathrm{deg})$, where $(\mathfrak P,\rhd)$ is the \textit{weight poset}, $\mathcal T(\lambda)$ is a finite set for all $\lambda \in\mathfrak P$, $\mathrm{i}$ is an \textit{involution} of $A$ and $C$ is an injection
\[
C\colon\coprod_{\lambda\in\mathfrak P}\mathcal T(\lambda)\times \mathcal T(\lambda)\to A,\;(s,t)\mapsto c^{\lambda}_{st}.
\]
Moreover, the \textit{degree function} deg is given by
\[
\mathrm{deg}\colon\coprod_{\lambda\in\mathfrak P}\mathcal T(\lambda)\to\bZ.
\] 
The whole data should be such that the $c^{\lambda}_{st}$ form a homogeneous $R$-basis of $A$ with $\mathrm{i}(c^{\lambda}_{st})=c^{\lambda}_{ts}$ and $\mathrm{deg}(c^{\lambda}_{st})=\mathrm{deg}(s)+\mathrm{deg}(t)$ for all $\lambda\in\mathfrak{P}$ and $s,t\in\mathcal T(\lambda)$. Moreover, for all $a\in A$
\[
ac^{\lambda}_{st}=\sum_{u\in \mathcal T(\lambda)}r_{a}(s,u)c^{\lambda}_{ut}\;(\mathrm{mod}\;A^{\rhd\lambda}).
\]
Here $A^{\rhd\lambda}$ is the $R$-submodule of $A$ spanned by the set $\{c^{\mu}_{st}\mid \mu\rhd\lambda\text{ and }s,t\in\mathcal T(\mu)\}$.

An algebra $A$ with such a quintuple is called a \textit{graded cellular algebra} and the $c^{\lambda}_{st}$ are called a \textit{graded cellular basis} of $A$ (with respect to the involution $\mathrm{i}$).
\end{defn}
\begin{ex}\label{ex-cellular1}
Let $A=R[x]/(x^n)$ and $\mathrm{i}=\mathrm{id}$. And let $\mathfrak P=\{0,\dots,n-1\}$ and $\mathcal T(k)=\{1\}$. Then the standard basis $c_{11}^{k}=x^k$ has a very special property, namely that the coefficients for multiplication  only depend on higher powers of $x$ (modulo $x^n$).

Let $A=\mathrm{M}_{n\times n}(R)$, i.e. the set of $n\times n$-matrices over $R$. Set $\mathfrak P=\{*\}$ and $\mathcal T(*)=\{1,\dots,n\}$. The standard basis of $A$, i.e. the $e_{ij}$-matrices, has a very special property, namely that the coefficients for multiplication with a matrix from the right only depend on the $i$-th row and vice versa for multiplication from the left only on the $j$-th column. Moreover, for the involution defined by $\mathrm{i}(M)=M^{t}$, we have $\mathrm{i}(e_{ij})=e_{ji}$.

For an example of a graded cellular algebra, we can take $A=\mathrm{M}_{2\times 2}(R)$ with the same data as above. The degree of the two elements in $\mathcal T(*)=\{1,2\}$ should be $\mathrm{deg}(1)=1$ and $\mathrm{deg}(2)=-1$. One can easily check that this a graded cell datum as in Definition~\ref{defn-cellular}.

In the same spirit, one can see (using the Artin-Wedderburn Theorem) that every semisimple algebra over a algebraically closed field is a cellular algebra. It is worth noting that Gornik's deformation $G_S$ of our web algebra $K_S$ as defined in~\ref{sec-webalg} is by Proposition~\ref{prop:Gsemisimple} semisimple. Hence, we see directly that it is a cellular algebra. Moreover, the example $A=\mathrm{M}_{2\times 2}(R)$ can easily extended to $A=\mathrm{M}_{n\times n}(R)$, showing that it is a \textit{graded} cellular algebra. This implies that every semisimple algebra is in fact a graded cellular algebra, although the grading for $G_S$ is artificial and does not have connections to the $q$-filtration explained in~\ref{sec-webalg}. 
\end{ex}
The idea of Graham and Lehrer was to ``interpolate'' between the two extremes from Example~\ref{ex-cellular1}. One main example for our purposes is the following.
\begin{ex}\label{ex-cellular2}
As Hu and Mathas~\cite{hm} point out, the algebras studied by Brundan and Stroppel~\cite{bs1}, and their quasi-hereditary covers are \textit{graded} cellular algebras in the sense of Definition~\ref{defn-cellular}. Note that these algebras are the $\mathfrak{sl}_2$ analogue of the algebras we defined in Section~\ref{sec-web}. 
\end{ex}
Another main example for us is also given by Hu and Mathas in~\cite{hm}, i.e. they showed the following Theorem.
\begin{thm}\label{thm-cellular}
Suppose that $\mathcal O$ is a commutative integral domain such that $e$ is invertible in $\mathcal O$, $e=0$ or $e$ is a non-zero prime number, and let $R^n_{\Lambda}$ be the cyclotomic Khovanov-Lauda-Rouquier algebra $R^n_{\Lambda}$ over $\mathcal O$. Then $R^n_{\Lambda}$ is a graded cellular algebra with respect to the dominance order and with homogeneous cellular basis explicitly given in~\cite{hm}.

Using Brundan and Kleshchev~\cite{bk1} graded isomorphism, this includes that the corresponding cyclotomic Hecke algebra is a graded cellular algebra if $\mathcal O$ is a field.
\end{thm}
\vskip0.2cm
Let us recall some facts about (graded) cellular algebras.
\begin{defn}\label{defn-cellular2}
Let $A$ be a graded cellular algebra over $R$ with a given graded cell datum $(\mathfrak{P},\mathcal T,C,\mathrm{i},\mathrm{deg})$. Moreover, fix a weight $\lambda\in\mathfrak P$. The graded (left) \textit{$A$-cell module for $\lambda$} denoted $C^{\lambda}$ is
\[
C^{\lambda}=\bigoplus_{d\in\bZ}C^{\lambda}_d,
\]
where $C^{\lambda}_d$ is the free $R$-module with basis
\[
\{c^{\lambda}_s\mid\,s\in\mathcal T(\lambda)\,\text{ and }\,\deg s=d\}.
\]
The action is given by
\[
ac^{\lambda}_{s}=\sum_{u\in \mathcal T(\lambda)}r_{a}(s,u)c^{\lambda}_{u}.
\]
The coefficients should be the ones from Definition~\ref{defn-cellular}. The action of the involution $\mathrm{i}$ is defined in the obvious way. Note that these modules can be seen as a generalisation of the (graded) Specht modules for the symmetric group and the Hecke algebras of type A.

For fixed $\lambda\in\mathfrak P$ we can define
\[
D^{\lambda}=C^{\lambda}/\mathrm{rad}(C^{\lambda}).
\]
\end{defn}
One main point why (graded) cellular algebras are interesting is that they give a lot of information about the categories $\AMOD_{\mathrm{gr}}$ and $\AMOD$, if one knows a particular (graded) cell datum for it, i.e. although the existence of a (graded) cell datum is \textit{not} trivial, as we argue below, an \textit{explicit} (graded) cell datum is preferable. To be more precise, we recall a theorem of Hu and Mathas. The proof can be found in their paper~\cite{hm} in the Section 2.
\begin{thm}\label{thm-cellular2}
Suppose $K$ is a field and that $A$ is a graded cellular algebra over $K$ with graded cell datum $(\mathfrak{P},\mathcal T,C,\mathrm{i},\mathrm{deg})$. Let $\mathfrak P_0=\{\lambda\in\mathfrak P\mid D^{\lambda}\neq 0\}$. Then we have the following.
\begin{itemize}
\item[(a)] The graded $A$-module $D^{\lambda}$ is absolutely irreducible for all $\lambda\in\mathfrak{P}_0$.
\item[(b)] For all $\lambda,\mu\in\mathfrak{P}_0$ we have $D^{\lambda}\simeq D^{\mu}\{i\}$ for some $i\in\bZ$ iff $i=0$ and $\lambda=\mu$.
\item[(c)] A complete set of pairwise non-isomorphic, graded, simple $A$-modules is given by
\[
D=\{D^{\lambda}\{i\}\mid \lambda\in\mathfrak P_0\,\text{ and }\,i\in\bZ\}.
\]
\item[(d)] Let $\overline{\cdot}\colon\AMOD_{\mathrm{gr}}\to\AMOD$ denote the functor that forgets the grading. Then a complete set of pairwise non-isomorphic simple $A$-modules is given by
\[
\overline{D}=\{\overline{D}^{\lambda}\mid \lambda\in\mathfrak P_0\}.
\] 
\end{itemize} 
\end{thm}
\vskip0.2cm
We recall a theorem due to K\"onig and Xi~\cite{kx1}, i.e. that the property ``being cellular'' is Morita invariant over fields of characteristic not $2$. To be more precise.
\begin{thm}\label{thm-cellular3}
Let $K$ be a field with $\mathrm{char}(K)\neq 2$. Moreover, let $A$ be an $K$-algebra that is cellular with respect to an involution $\mathrm{i}$. Let $B$ be an $K$-algebra that is Morita equivalent to $A$. Then $B$ is a cellular algebra with respect to a \textit{suitable} involution $\mathrm{i}^{\prime}$.
\end{thm}
\vskip0.2cm
It should be noted that their proof of Theorem~\ref{thm-cellular3} relies on a ring theoretical and basis free definition of the notion cellular algebra using a \textit{cell filtration of cell ideals}. It is well-known that this definition is equivalent to the original one given by Graham and Lehrer, see~\cite{kx2}. To be more precise, we recall the following.
\begin{defn}\label{defn-cellular3}
Suppose $A$ is a graded, free algebra over $R$ of finite rank and $\mathrm{i}\colon A\to A$ is an involution. A two sided ideal $J\subset A$ that is fixed by the involution $\mathrm{i}$ is called a \textit{cell ideal} iff there exists a left ideal $\Delta\subset J$ that is finitely generated and free as an $R$-module together with an isomorphism of $A$-bimodules $\alpha\colon J\to\Delta\otimes_R\mathrm{i}(\Delta)$ such that the following diagram commutes.
\[
\begin{xy}
  \xymatrix{
      J \ar[r]^/-.95em/{\alpha} \ar[d]_{\mathrm{i}}  &   \Delta\otimes_R\mathrm{i}(\Delta) \ar[d]^{x\otimes y\mapsto \mathrm{i}(y)\otimes \mathrm{i}(x)} \\
      J \ar[r]_/-.95em/{\alpha}         &   \Delta\otimes_R\mathrm{i}(\Delta)   
  }
\end{xy}
\]
The algebra $A$ is called \textit{cellular} with respect to $\mathrm{i}$ if there is a finite chain of two-sided ideals (all fixed by $\mathrm{i}$), i.e. $0=J_0\subset J_1\subset\cdots\subset J_{n-1}\subset J_n=A$, such that $J_k/J_{k-1}$ is a cell ideal of $A/J_{k-1}$ with respect to $\mathrm{i}$ for all $1\leq k\leq n$.
\end{defn}
\begin{thm}\label{thm-cellular4}
An $R$-algebra $A$ is (ungraded) cellular in the sense of Definition~\ref{defn-cellular} iff it is (ungraded) cellular in the sense of Definition~\ref{defn-cellular3}.
\end{thm}
\vskip0.2cm
The Definition~\ref{defn-cellular3} can be copied and applied in the graded setting, too. Moreover, a slight change of the proof of the equivalence given in~\cite{kx2} shows that one can give an equivalent basis free definition of Definition~\ref{defn-cellular} using \textit{cell filtration of graded cell ideals}. Hence, using a result, i.e. Theorem 5.4 and Corollary 5.5, of Gordon~\cite{gordon}, one obtains that \textit{graded} Morita equivalence between two algebras over a field of characteristic not $2$, one a cellular algebra with a non-trivial grading, implies the existence of a non-trivially \textit{graded} cellular basis for the other in the sense of Definition~\ref{defn-cellular}, although neither the bases elements nor the involution or grading are \textit{explicit}.

The assumption that $\mathrm{char}(K)\neq 2$ is necessary as K\"onig and Xi showed and the following example (given in Section 7 in~\cite{kx1}) shows that one has to be careful with the involution.
\begin{ex}\label{ex-cellular3}
Let $K$ be field with $\mathrm{char}(K)\neq 2$ and let $A$ be the $K$-algebra $A=\Mat_{2\times 2}(K)$. Define two involutions $\mathrm{i},\mathrm{i}^{\prime}$ by
\[
\mathrm{i}\colon\begin{pmatrix} a & b \\ c & d \end{pmatrix}\mapsto\begin{pmatrix} d & b \\ c & a \end{pmatrix}\;\;\text{ and }\;\;\mathrm{i}^{\prime}\colon\begin{pmatrix} a & b \\ c & d \end{pmatrix}\mapsto\begin{pmatrix} d & -b \\ -c & a \end{pmatrix}.
\]
One can check that $A$ is cellular with respect to $\mathrm{i}$ if one sets the corresponding cell module to be
\[
\Delta=\left\{\begin{pmatrix} a & b \\ c & d \end{pmatrix}\in\Mat_{2\times 2}(K)\mid a=b,c=d\right\}.
\]
But $A$ is not cellular with respect to $\mathrm{i}^{\prime}$, since $A$ is simple and therefore its own cell ideal which would imply that there exists a cell module $\Delta$ such that
\[
A\simeq \Delta\otimes_K\mathrm{i}^{\prime}(\Delta),
\]
but an easy calculation, using the conditions given in Definition~\ref{defn-cellular3}, shows that this is impossible.
\end{ex}
\subsection{Filtered and graded algebras and modules}\label{sec-appendix}
In this section (note that this was the appendix in~\cite{mpt}), we have collected some basic facts about filtered algebras, 
the associated graded algebras and the idempotents in both. Our main sources 
are~\cite{sj} and~\cite{sr}. 
In this section, everything is defined over an arbitrary commutative, 
associative, unital ring $R$.  

Let $A$ be a finite dimensional, 
associative, unital $R$-algebra together with an increasing filtration of $R$-submodules
\[
\{0\}\subset A_{-p}\subset A_{-p+1}\subset\cdots\subset A_0\subset \cdots 
\subset A_{m-1}\subset A_m=A.
\]
Actually, for any $t\in \mathbb{Z}$ we have a subspace $A_t$, where 
we extend the filtration above by 
\[A_t=
\begin{cases}
\{0\},& \text{if}\quad t<-p,\\
A,&\text{if}\quad p\geq m.
\end{cases}
\]
Note that in the language of~\cite{sj}, such a filtration is \textit{discrete, 
separated, exhaustive and complete}. If $1\in A_0$ and the multiplication 
satisfies $A_iA_j\subseteq A_{i+j}$, we say that $A$ is an associative, unital, 
\textit{filtered algebra}. The \textit{associated graded algebra} is defined by 
\[
E(A)=\bigoplus_{i\in\mathbb{Z}} A_i/A_{i-1},
\] 
and is also associative and unital. Although $A$ and $E(A)$ are isomorphic 
$R$-modules, they are not isomorphic as algebras. 

A finite dimensional, \textit{filtered $A$-module} is a finite dimensional,
unitary $A$-module $M$ with an increasing filtration of $R$-submodules 
\[
\{0\}\subset M_{-q}\subset M_{-q+1}\subset \cdots \subset M_t=M,
\]
such that $A_iM_j\subseteq M_{i+j}$, for all $i,j\in\mathbb{Z}$, after 
extending the finite filtration to a $\mathbb{Z}$-filtration as above. 

We define the $t$-fold \textit{suspension} $M\{t\}$ of $M$, which 
has the same underlying $A$-module structure, but a new filtration defined by 
\[
M\{t\}_r=M_{r+t}.
\]
Given a filtered $A$-module $M$, the \textit{associated graded module} 
is defined by  
\[
E(M)=\bigoplus_{i\in\mathbb{Z}} M_i/M_{i-1}.
\]
An $A$-module map 
$f\colon M\to N$ is said to \textit{preserve the filtrations} if 
$f(M_i)\subseteq N_i$, for all $i\in\mathbb{Z}$. Any such map 
$f\colon M\to N$ induces a grading preserving $E(A)$-module map 
$E(f)\colon E(M)\to E(N)$ in the obvious way. 

This way, we get a functor 
\[
E\colon A\text{-}\mathrm{\textbf{Mod}}_{\mathrm{fl}} \to E(A)\text{-}\mathrm{\textbf{Mod}}_{\mathrm{gr}},
\]
where $A\text{-}\mathrm{\textbf{Mod}}_{\mathrm{fl}}$ is the category of finite dimensional, filtered 
$A$-modules and filtration preserving $A$-module maps and $E(A)\text{-}\mathrm{\textbf{Mod}}_{\mathrm{gr}}$ 
is the category of finite dimensional, graded $E(A)$-modules and grading preserving $E(A)$-module maps. 

Recall that $A\text{-}\mathrm{\textbf{Mod}}_{\mathrm{fl}}$ is not an abelian category, e.g. 
the identity map $M\to M\{1\}$ is a filtration preserving bijective $A$-module 
map, but does not have an inverse in $A\text{-}\mathrm{\textbf{Mod}}_{\mathrm{fl}}$. In order to 
avoid such complications, one can consider a subcategory with fewer morphisms. 
An $A$-module map $f\colon M\to N$ is called \textit{strict} if 
\[
f(M_i)=f(M)\cap N_i
\]
holds, for all $i\in\mathbb{Z}$. 
Let $A\text{-}\mathrm{\textbf{Mod}}_{\mathrm{st}}$ be the subcategory of filtered $A$-modules 
and strict $A$-module homomorphisms. 

\begin{lem}
The restriction of $E$ to $A\text{-}\mathrm{\textbf{Mod}}_{\mathrm{st}}$ is exact. 
\end{lem}

We also need to recall a simple result about bases. 
A basis $\{x_1,\ldots,x_n\}$ of a filtered algebra $A$ is called \textit{homogeneous} if, for each $1\leq j\leq n$, there exists an 
$i\in\mathbb{Z}$ such that $x_j\in A_i\backslash A_{i-1}$. In that case, 
$\{\overline{x}_1,\ldots,\overline{x}_n\}$ defines a homogeneous basis 
of $E(A)$, where $\overline{x_j}\in A_i/A_{i-1}$. In order to avoid cluttering 
our notation, we always write $\overline{x}_j$ and then specify in which 
subquotient we take the equivalence class by saying that it belongs to 
$A_i/A_{i-1}$.

Given a homogeneous basis $\{y_1,\ldots,y_n\}$ of the associated graded $E(A)$, we say that 
a homogeneous basis $\{x_1,\ldots,x_n\}$ of $A$ \textit{lifts} 
$\{y_1,\ldots,y_n\}$ if $\overline{x}_j=y_j\in A_i/A_{i-1}$ holds, for each $1\leq j\leq n$ 
and the corresponding $i\in\mathbb{Z}$. 
The result in the following lemma is well-known. However, we could not find 
a reference in the literature, so we provide a short proof here. 
\begin{lem}
\label{lem:homogbasis}
Let $A$ be a finite dimensional, filtered algebra and 
$\{y_1,\ldots,y_n\}$ be a homogeneous basis of $E(A)$. Then 
there is a homogeneous basis $\{x_1,\ldots,x_n\}$ of $A$ 
which lifts $\{y_1,\ldots,y_n\}$. 
\end{lem}
\begin{proof}
We prove the lemma by induction with respect to the filtration degree $q$. 
Suppose $A_q=0$, for all $q<-p$, and $A_q=A$, for all $q\geq m$. 
Then $E(A_{-p})=A_{-p}$. Since $\{y_1,\ldots,y_n\}$ is a homogeneous basis of $E(A)$, 
a subset of this basis forms a basis of $A_{-p}$. 

For each $-p+1\leq q\leq m$, choose elements in $A_q$ which lift the homogeneous subbasis 
of $E(A_q)$. We claim that the union of the sets of these elements, for all $-p\leq q\leq m$, 
form a homogeneous basis of $A$ which lifts $\{y_1,\ldots,y_n\}$. Call it $\{x_1,\ldots,x_n\}$. 
By definition, the $x_j$ lift the $y_j$, for all $1\leq j\leq n$. It remains to show that 
the $x_j$ are all linearly independent. This is true for $q=-p$, as shown above. 

Let $q>-p$ and suppose that the claim holds for $\{x_1,\ldots,x_{m_{q-1}}\}$, the subset of $\{x_1,\ldots,x_n\}$ 
which belongs to $A_{q-1}$. Let 
\[
\{x_1,\ldots, x_{m_q}\}=\{x_1,\ldots,x_{m_{q-1}}\}\cup \{x_{m_{q-1}+1},\ldots,x_{m_q}\}
\]
be the subset belonging to $A_q$. Suppose that  
\begin{equation}
\label{eqn:linind}
\sum_{j=1}^{m_q}\lambda_jx_j=0,
\end{equation}
with $\lambda_j\in R$. 
Then we have
\[
\sum_{j=1}^{m_q}\lambda_j\overline{x}_j=\sum_{j=m_{q-1}+1}^{m_q}\lambda_j\overline{x}_j=\sum_{j=m_{q-1}+1}^{m_q}\lambda_jy_j=0 \in A_q/A_{q-1}.
\]
By the linear independence of the $y_j$, this implies that $\lambda_j=0$, for all $m_{q-1}+1\leq j\leq m_q$. Thus, the linear 
combination in~\eqref{eqn:linind} becomes 
\[
\sum_{j=1}^{m_{q-1}}\lambda_jx_j=0.
\]
By induction, this implies that $\lambda_j=0$, for all $1\leq j\leq m_{q-1}$. 

This shows that $\lambda_j=0$, for all $1\leq j\leq n$, so the $x_j$ 
are linearly independent. 
\end{proof}
For a proof of the following proposition, see for example 
Proposition 1 in the appendix of~\cite{sr}.
\begin{prop}
\label{prop:Srid}
Let $M$ and $N$ be filtered $A$-modules and $f\colon M\to N$ a filtration 
preserving $A$-linear map. If $E(f)\colon E(M)\to E(N)$ is an isomorphism, 
then $f$ is an isomorphism (and therefore strict too). 
\end{prop} 
The most important fact about filtered, projective modules and their 
associated graded, projective modules, that 
we need in Section~\ref{sec-web}, is Theorem 6 in~\cite{sj}. Note that these projective modules are the projective objects in the category $A\text{-}\mathrm{\textbf{Mod}}_{\mathrm{st}}$.
\begin{thm}[Sj\"{o}din]
\label{thm:sjodin}
Let $P$ be a finite dimensional, graded, projective $E(A)$-module. Then 
there exists a finite dimensional, filtered, projective $A$-module $P'$, 
such that $E(P')=P$. Moreover, if $M$ is a finite dimensional, filtered, 
$A$-module, then any degree preserving $E(A)$-module map 
$P\to E(M)\{t\}$, for some grading shift $t\in\mathbb{Z}$, lifts to a 
filtration preserving $A$-module map $P'\to M\{t\}$. 
\end{thm} 

We also recall the following corollary of Sj\"{o}din (Corollary in~\cite{sj} after Lemma 20).

\begin{cor}
\label{cor:sjodin}
Let $M$ be a finite dimensional, filtered, $A$-module, then any finite or countable set of orthogonal idempotents in
\[
\mathrm{im}(\phi)\subset\mathrm{hom}_{E(A)}(E(M),E(M))
\]
can be lifted to $\mathrm{hom}_{A}(M,M)$, where $\phi$ is the natural transformation
\[
\phi\colon E(\mathrm{hom}_{A}(M,M))\to \mathrm{hom}_{E(A)}(E(M),E(M)).
\]
\end{cor}
\newpage

\end{document}